\documentclass[a4paper,10pt]{book}

\usepackage[french]{babel}
\usepackage[applemac]{inputenc}
\usepackage[T1]{fontenc}

\usepackage{
graphics,
latexsym,
textcomp,
amsfonts,
amssymb,
amsbsy,
amsmath,
amsthm,
geometry
}

\usepackage[all,2cell,cmtip]{xy} \UseAllTwocells 

\usepackage{index}
\makeindex
\newindex{not}{nox}{nod}{Notations}

\SilentMatrices

\makeatletter

\renewcommand\tableofcontents{%
   \if@twocolumn
     \@restonecoltrue\onecolumn
   \else
     \@restonecolfalse
   \fi
   \chapter*{\contentsname
       \@mkboth{%
          \MakeUppercase\contentsname}{\MakeUppercase\contentsname}}%
   \addvspace{-0.3em}%
   \@starttoc{toc}%
   \if@restonecol\twocolumn\fi
 }

\makeatother

\newcommand\EquiQuillen{W_{\infty}^{\Delta}}
\newcommand\EnsSimp{\widehat{\Delta}}
\newcommand\Intervalles{\Delta_{[\kern 1.5pt]}}

\newcommand\Chaines[3]{F_{{#1},{#2}}({#3})}
\newcommand\FoncteurChaines[2]{F_{{#1},{#2}}}

\newcommand\DeuxFoncteurStrict[0]{\deux{}foncteur strict}
\newcommand\DeuxFoncteursStricts[0]{\deux{}foncteurs stricts}
\newcommand\DeuxFoncteurLax[0]{\deux{}foncteur lax}
\newcommand\DeuxFoncteursLax[0]{\deux{}foncteurs lax}
\newcommand\DeuxFoncteurCoLax[0]{\deux{}foncteur colax}
\newcommand\DeuxFoncteursCoLax[0]{\deux{}foncteurs colax}
\newcommand\DeuxTransformationLax[0]{transformation}
\newcommand\DeuxTransformationsLax[0]{transformations}
\newcommand\DeuxTransformationCoLax[0]{optransformation}
\newcommand\DeuxTransformationsCoLax[0]{optransformations}
\newcommand\DeuxTransformationStricte[0]{transformation stricte}
\newcommand\DeuxTransformationsStrictes[0]{transformations strictes}

\newcommand\TrancheLax[3]{{#1}/\negmedspace/_{\negmedspace {\rm{l}}}^{#2}{#3}}

\newcommand\TrancheCoLax[3]{{#1}/\negmedspace/_{\negmedspace {\rm{c}}}^{#2}{#3}}

\newcommand\OpTrancheCoLax[3]{{#3}{\backslash\mspace{-5mu}\backslash}_{\mspace{-.3mu}{\rm{c}}}^{\mspace{-6.mu}#2}{#1}}

\newcommand\OpTrancheLax[3]{{#3}{\backslash\mspace{-5mu}\backslash}_{\mspace{-.3mu}{\rm{l}}}^{\mspace{-6.mu}#2}{#1}}

\newcommand\DeuxFoncTrancheLax[2]{{#1}/\negmedspace/_{\negmedspace {\rm{l}}}{#2}}

\newcommand\DeuxFoncTrancheCoLax[2]{{#1}/\negmedspace/_{\negmedspace {\rm{c}}}{#2}}

\newcommand\DeuxFoncOpTrancheLax[2]{{#2}{\backslash\mspace{-5mu}\backslash}_{\mspace{-.3mu}{\rm{l}}}^{\mspace{-6.mu}}{#1}}

\newcommand\DeuxFoncOpTrancheCoLax[2]{{#2}{\backslash\mspace{-5mu}\backslash}_{\mspace{-.3mu}{\rm{c}}}^{\mspace{-6.mu}}{#1}}

\newcommand\DeuxFoncTrancheLaxCoq[3]{{#1}/\negmedspace/_{\negmedspace{\rm{l}}}^{#2}{#3}}

\newcommand\DeuxFoncTrancheCoLaxCoq[3]{{#1}/\negmedspace/_{\negmedspace{\rm{c}}}^{#2}{#3}}

\newcommand\DeuxFoncOpTrancheLaxCoq[3]{{#3}{\backslash\mspace{-5mu}\backslash}_{\mspace{-.3mu}{\rm{l}}}^{\mspace{-6.mu}#2}{#1}}

\newcommand\DeuxFoncOpTrancheCoLaxCoq[3]{{#3}{\backslash\mspace{-5mu}\backslash}_{\mspace{-.3mu}{\rm{c}}}^{\mspace{-6.mu}#2}{#1}}

\newcommand\Fibre[3]{{#1}^{#2}_{#3}}

\newcommand\FonctComp[4]{c^{#1}_{{#4},{#3},{#2}}}

\newcommand\CompDeuxUn[0]{}

\newcommand\CompDeuxZero[0]{\circ}

\newcommand\TransNatUnit[2]{{#1}_{#2}}

\newcommand\DeuxCellStructComp[3]{{#1}_{{#2},{#3}}}

\newcommand\DeuxCellStructId[2]{{#1}_{{#2}}}

\newcommand\Objets[1]{Ob({#1})}

\newcommand\UnCell[1]{Fl_{1}({#1})}

\newcommand\IdObjet[1]{1_{#1}}

\newcommand\TildeLax[1]{\widetilde{#1}}
\newcommand\BarreLax[1]{\overline{#1}}

\newcommand\FoncBenabou[0]{B}

\newcommand\LaxCanonique[1]{\eta_{{#1}}}
\newcommand\StrictCanonique[1]{\epsilon_{{#1}}}
\newcommand\TransLaxCanonique[0]{\eta}
\newcommand\TransStrictCanonique[0]{\epsilon}

\newcommand\TildeColax[1]{\widetilde{#1}^{c}}
\newcommand\BarreColax[1]{\overline{#1}^{c}}

\newcommand\ColaxCanonique[1]{\eta^{c}_{{#1}}}
\newcommand\StrictColaxCanonique[1]{\epsilon^{c}_{{#1}}}

\newcommand\DeuxInt[1]{\int_{#1}}
\newcommand\DeuxIntOp[1]{\int_{#1}^{op}}
\newcommand\DeuxIntCo[1]{\int_{#1}^{co}}
\newcommand\DeuxIntCoOp[1]{\int_{#1}^{coop}}

\newcommand\DeuxIntUnOp[1]{\int_{#1}^{1}}
\newcommand\DeuxIntOpUnOp[1]{\int_{#1}^{op,1}}
\newcommand\DeuxIntCoUnOp[1]{\int_{#1}^{co,1}}
\newcommand\DeuxIntCoOpUnOp[1]{\int_{#1}^{coop,1}}

\newcommand\DeuxIntDeuxOp[1]{\int_{#1}^{2}}
\newcommand\DeuxIntOpDeuxOp[1]{\int_{#1}^{op,2}}
\newcommand\DeuxIntCoDeuxOp[1]{\int_{#1}^{co, 2}}
\newcommand\DeuxIntCoOpDeuxOp[1]{\int_{#1}^{coop, 2}}

\newcommand\DeuxIntToutOp[1]{\int_{#1}^{1,2}}
\newcommand\DeuxIntOpToutOp[1]{\int_{#1}^{op,1,2}}
\newcommand\DeuxIntCoToutOp[1]{\int_{#1}^{co,1,2}}
\newcommand\DeuxIntCoOpToutOp[1]{\int_{#1}^{coop,1,2}}

\newcommand\trois{$3$\nobreakdash-}
\newcommand\deux{$2$\nobreakdash-}
\newcommand\un{$1$\nobreakdash-}

\newcommand\CatHom[3]{\underline{Hom}_{#1}(#2, #3)}
\newcommand\EnsHom[3]{{Hom}_{#1}(#2, #3)}

\newcommand\Ens{{\mathcal{E}\mspace{-2.mu}\it{ns}}}

\newcommand\Cat{{\mathcal{C}\mspace{-2.mu}\it{at}}}
\newcommand\Top{{\mathcal{T}\mspace{-2.mu}\it{op}}}
\newcommand\DeuxCat{\text{$2$-$\Cat$}}
\newcommand\CAT{{C\mspace{-2.mu}\it{AT}}}
\newcommand\DeuxCatLax{2\hbox{\protect\nobreakdash-}\kern1pt\Cat_{lax}}

\newcommand\DeuxCatDeuxCat{\underline{\DeuxCat}}

\newcommand\TronqueBete[0]{\tau_{\beta}}
\newcommand\TronqueInt[0]{\tau_{\iota}}

\newcommand\ClasseDeuxLocFond{localisateur fondamental de $\DeuxCat$}
\newcommand\ClasseDeuxLocFondLax{localisateur fondamental de $\DeuxCatLax$}
\newcommand\ClassesDeuxLocFond{localisateurs fondamentaux de $\DeuxCat$}
\newcommand\ClassesDeuxLocFondLax{localisateurs fondamentaux de $\DeuxCatLax$}

\newcommand\ClasseUnLocFond{localisateur fondamental de $\Cat$}
\newcommand\ClassesUnLocFond{localisateurs fondamentaux de $\Cat$}

\newcommand\UnLocFond[1]{{#1}}
\newcommand\DeuxLocFond[1]{\mathcal{#1}}
\newcommand\DeuxLocFondLax[1]{\mathcal{#1}}
\newcommand\DeuxLocFondLaxInduit[1]{\mathcal{#1}_{lax}}

\newcommand\UnLocFondMin{\UnLocFond{W}_{\infty}^{1}}
\newcommand\DeuxLocFondMin{\DeuxLocFond{W}_{\infty}^{2}}
\newcommand\DeuxLocFondLaxMin{\DeuxLocFond{W}_{\infty,lax}^{2}}

\newcommand\Localisation[2]{{#2}^{-1}{#1}}

\newcommand\UnNerf{N}
\newcommand\NerfLaxNor{N_{\rm{l,n}}}
\newcommand\NerfLax{N_{\rm{l}}}
\newcommand\NerfCatLaxNor{\underline{N}_{\rm{l,n}}}
\newcommand\NerfCatLax{\underline{N}_{\rm{l}}}
\newcommand\NerfHom{N_{2}}

\newcommand\FonLaxNorCat[2]{\underline{FonLaxNor}({#1}, {#2})}
\newcommand\FonStrictCat[2]{\underline{Fon}({#1}, {#2})}

\newcommand\SupUn{sup}
\newcommand\SupUnObjet[1]{sup_{#1}}

\newcommand\SupLaxNorObjet[1]{sup^{\rm{l,n}}_{#1}}

\newcommand\SupLaxObjet[1]{sup^{\rm{l}}_{#1}}

\newcommand\SupCatLaxNorObjet[1]{\underline{sup}^{\rm{l,n}}_{#1}}

\newcommand\SupHom{\underline{sup}}
\newcommand\SupHomObjet[1]{\underline{sup}_{#1}}

\newcommand\UnCatOp[1]{{#1}^{op}}

\newcommand\DeuxCatUnOp[1]{{#1}^{op}}
\newcommand\DeuxCatDeuxOp[1]{{#1}^{co}}
\newcommand\DeuxCatToutOp[1]{{#1}^{coop}}

\newcommand\DeuxFoncUnOp[1]{{#1}^{op}}
\newcommand\DeuxFoncDeuxOp[1]{{#1}^{co}}
\newcommand\DeuxFoncToutOp[1]{{#1}^{coop}}

\newcommand\DeuxTransUnOp[1]{{#1}^{op}}

\newcommand\DeuxTransDeuxOp[1]{{#1}^{co}}

\newcommand\DeuxTransToutOp[1]{{#1}^{coop}}

\newcommand\DeuxFoncteurTranche[1]{T^{#1}}
\newcommand\DeuxFoncteurTrancheArg[2]{T^{#1}_{#2}}

\newcommand\mathdeuxcat[1]{\mathcal{#1}}
\newcommand\DeuxCatPonct{e}

\newcommand\UnCatPonct{e}

\theoremstyle{plain}
\newtheorem{theo}{Théorème}[section]
\newtheorem{prop}[theo]{Proposition}
\newtheorem{lemme}[theo]{Lemme}
\newtheorem{corollaire}[theo]{Corollaire}
\newtheorem{thmp}{Théorème}[]

\theoremstyle{remark}
\newtheorem{exemple}[theo]{Exemple}
\newtheorem{rem}[theo]{Remarque}

\theoremstyle{definition}
\newtheorem{df}[theo]{Définition}

\newtheorem{paragr}[theo]{} 
\newtheorem{dfp}{Définition}[]

\numberwithin{equation}{theo}

\allowdisplaybreaks

\makeatletter 
\def\s@btitle{\relax} 
\def\subtitle#1{\gdef\s@btitle{#1}} 
\def\@maketitle{%
  \newpage 
  \null 
  \vskip 2em%
  \begin{center}%
  \let \footnote \thanks 
    {\LARGE \@title \par}%
                \if\s@btitle\relax 
                \else\typeout{[subtitle]}%
                        \vskip 1pc 
                        \begin{large}%
                                \textsl{\s@btitle}%
                                \par 
                        \end{large}%
                \fi 
    \vskip 1.5em%
    {\large 
      \lineskip .5em%
      \begin{tabular}[t]{c}%
        \@author 
      \end{tabular}\par}%
    \vskip 1em%
    {\large \@date}%
  \end{center}%
  \par 
  \vskip 1.5em} 
\makeatother

\title{La théorie de l'homotopie des \deux{}catégories}

\author{Jonathan Chiche}

\date{}

\begin{document}
\xyoption{v2}


\thispagestyle{empty}

\begin{center}
\vspace{\stretch{1}}
{\LARGE \textbf{Université Paris Diderot - Paris 7}}\\
{\Large \textbf{École doctorale de sciences mathématiques de Paris Centre}}

\vspace{\stretch{2}}

{\Huge \textsc{Thèse de doctorat}}

\vspace{\stretch{1}}

{\LARGE Discipline : Mathématiques}

\vspace{\stretch{3}}

{\large présentée par}
\vspace{\stretch{1}}

{\LARGE Jonathan CHICHE}

\vspace{\stretch{2}}
\hrule
\vspace{\stretch{1}}
{\LARGE \textbf{La théorie de l'homotopie des \deux{}catégories}}
\vspace{\stretch{1}}
\hrule
\vspace{\stretch{1}}

{\Large 
Directeur de thèse : Georges \textsc{Maltsiniotis}}

\vspace{\stretch{5}}

{\Large Soutenue le 26 septembre 2014 devant le jury composé de :}

\vspace{\stretch{1}}
{\Large
\begin{tabular}{l}
M. Clemens \textsc{Berger} \\
M. Lawrence \textsc{Breen} \\
M. Denis-Charles \textsc{Cisinski} \\
M. Georges \textsc{Maltsiniotis} \\
M. Paul-André \textsc{Melliès} \\
M. Ieke \textsc{Moerdijk} \\
\end{tabular}
}

\vspace{\stretch{2}}

{\Large Rapportée par :}

\vspace{\stretch{1}}
{\Large
\begin{tabular}{l}
M. Clemens \textsc{Berger} \\
M. Steve \textsc{Lack} \\
\end{tabular}
}
\vspace{\stretch{-4}}
\end{center}

\newpage

\tableofcontents

\chapter*{Introduction}

\addcontentsline{toc}{chapter}{Introduction} \markboth{Introduction}{}

\section*{Présentation générale}

Ce travail s'inscrit dans une entreprise de généralisation de la théorie de l'homotopie « de Grothendieck » aux catégories supérieures. 

Les objets de base de la théorie de l'homotopie sont, classiquement, les espaces topologiques ou les CW\nobreakdash-complexes. Les ensembles simpliciaux permettent une approche plus combinatoire. L'équivalence des deux points de vue se précise au moyen de la théorie des catégories de modèles de Quillen : il existe une équivalence de Quillen entre la catégorie des espaces topologiques $\Top$ et celle des ensembles simpliciaux $\EnsSimp$, ces deux catégories se trouvant munies des structures de catégories de modèles dégagées par Quillen dans \cite{QuillenHomotopical}. On appellera les équivalences faibles de ces structures de catégories de modèles les \emph{équivalences faibles topologiques} et les \emph{équivalences faibles simpliciales} respectivement. Une équivalence de Quillen entre ces deux structures est fournie par le couple formé des foncteurs de réalisation géométrique et « ensemble simplicial singulier ». Ce couple de foncteurs adjoints s'obtient, par une construction classique de Kan, à partir du foncteur canonique $\Delta \to \Top$ de la catégorie des simplexes vers celle des espaces topologiques. 

On dispose d'un foncteur canonique de $\Delta$ vers la catégorie $\Cat$ des petites catégories. Le même procédé de Kan fournit une adjonction entre $\EnsSimp$ et $\Cat$, adjonction formée du \emph{foncteur de réalisation catégorique} $c : \EnsSimp \to \Cat$, adjoint à gauche, et du \emph{foncteur nerf} $\UnNerf : \Cat \to \EnsSimp$, adjoint à droite. En définissant les équivalences faibles de $\Cat$ comme les foncteurs dont le nerf est une équivalence faible simpliciale, on pourrait espérer que, de manière analogue à ce qui se produit dans le cas des catégories $\EnsSimp$ et $\Top$, les foncteurs $c$ et $\UnNerf$ induisent des équivalences de catégories entre les catégories localisées de $\EnsSimp$ et $\Cat$ par leurs équivalences faibles, catégories localisées que nous appellerons \emph{catégories homotopiques}. L'examen des propriétés homotopiques du foncteur $c$ montre le caractère illusoire d'un tel espoir, ce foncteur se comportant si mal de ce point de vue qu'il envoie des sphères simpliciales sur des petites catégories contractiles. Les auteurs de \cite{FritschLatch} indiquent que cette observation ainsi que le fait, dégagé par Gabriel et Zisman, que la réalisation catégorique d'un ensemble simplicial ne dépend que de son \deux{}squelette, ont longtemps pu laisser penser que les petites catégories, d'un point de vue homotopique, étaient presque triviales. Des résultats de Quillen, Lee et Latch montrent qu'il n'en est rien. Plus précisément, le foncteur nerf induit une équivalence de catégories entre les catégories homotopiques de $\Cat$ et de $\EnsSimp$. En revanche, comme on l'a déjà mentionné, l'adjoint à gauche $c$ du foncteur nerf se comporte trop mal pour induire un inverse au niveau de ces catégories homotopiques. Un tel inverse est toutefois donné par le foncteur induit par le foncteur \emph{catégorie des simplexes}
$$
\begin{aligned}
i_{\Delta} : \EnsSimp &\to \Cat
\\
X &\mapsto \Delta / X,
\end{aligned}
$$
qui peut se voir comme un cas particulier d'une construction connue sous le nom de \emph{construction de Grothendieck}, associant à tout foncteur $\DeuxCatUnOp{A} \to \Cat$ une catégorie fibrée sur $A$, pour toute petite catégorie $A$. Pour montrer que les foncteurs $i_{\Delta}$ et $\UnNerf$ induisent des équivalences quasi-inverses l'une de l'autre entre les catégories homotopiques de $\EnsSimp$ et de $\Cat$, on utilise entre autres résultats le fait que, pour toute petite catégorie $A$, il existe une équivalence faible 
$$
\SupUn_{A} : \Delta / \UnNerf A \to A
$$
naturelle en $A$. Une démonstration particulièrement simple de ce résultat s'appuie sur le Théorème A de Quillen \cite[p. 93, Théorème A]{QuillenK}. 

Si le couple $(c, \UnNerf)$ n'induit pas un couple d'équivalences quasi-inverses au niveau des ca\-té\-go\-ries homotopiques de $\EnsSimp$ et de $\Cat$, le couple d'endofoncteurs adjoints $(Sd, Ex)$ de $\EnsSimp$ permet de contourner l'obstruction se présentant, comme l'ont démontré Thomason et Fritsch \cite{FritschLatch}, $Sd$ et $Ex$ désignant respectivement les foncteurs de subdivision et d'extension \cite{Kan}. De façon remarquable, le couple $(c Sd^{2}, Ex^{2} \UnNerf)$ forme une équivalence de Quillen entre des structures de catégories de modèles sur $\Cat$ et sur $\EnsSimp$, structures dont les équivalences faibles sont celles que nous considérons \cite{Thomason}, la structure de catégorie de modèles sur $\EnsSimp$ étant celle déjà évoquée. 

Dans \emph{Pursuing Stacks} \cite{Poursuite}, Grothendieck développe une théorie de l'homotopie fondée non pas sur la catégorie des espaces topologiques, non plus que sur celle des ensembles simpliciaux, mais sur la catégorie $\Cat$, « vue avec un œil de géomètre par l'ensemble d'intuitions, étonnamment riche, provenant des topos » \cite{LettreGrothendieckThomason}. Son travail l'amène à dégager la notion de \emph{catégorie test}, petite catégorie dont le topos des préfaisceaux modélise canoniquement les types d'homotopie, notion dont la catégorie des simplexes $\Delta$ constitue le paradigme historique. S'apercevant qu'il n'a, pour étudier la théorie des catégories test, utilisé qu'un petit nombre de propriétés formelles des équivalences faibles de $\Cat$, il définit la notion de \emph{localisateur fondamental} comme une classe de morphismes de $\Cat$ vérifiant ces propriétés, dont la plus importante est le Théorème A de Quillen. À tout localisateur fondamental sont associées diverses notions, non seulement celle de catégorie test et ses variantes, mais également, par exemple, celles de foncteur propre et de foncteur lisse, définies par des propriétés de changement de base analogues à celles des théorèmes de changement de base propre ou lisse en géométrie algébrique. La théorie de l'homotopie de Grothendieck généralise une part fondamentale de la théorie « classique » de l'homotopie simpliciale, dont elle propose une approche conceptuelle remarquablement fructueuse. Elle est dégagée et développée par Grothendieck dans \cite{Poursuite}, présentée de façon plus « bourbachique » par Maltsiniotis dans \cite{THG} et développée plus avant par Cisinski dans sa thèse \cite{TheseCisinski} puis dans \cite{PMTH}. C'est à Cisinski que l'on doit notamment la démonstration de deux conjectures fondamentales de Grothendieck en ce domaine : la \emph{minimalité} du localisateur fondamental « classique » de $\Cat$, c'est-à-dire de la classe des morphismes de $\Cat$ dont le nerf est une équivalence faible simpliciale, et l'existence, \emph{pour essentiellement tout localisateur fondamental}\footnote{La seule hypothèse, anodine, est de nature ensembliste.}, d'une structure de catégorie de modèles sur $\Cat$ dont le localisateur fondamental considéré constitue précisément la classe des équivalences faibles, généralisant le résultat de Thomason. Ces structures sont obtenues par « transfert » à partir de structures de catégories de modèles sur la catégorie des ensembles simpliciaux. On reviendra sur ces deux résultats. Nous démontrons notamment un analogue \deux{}catégorique du premier dans cette thèse. Les résultats de \cite{AraMaltsi} permettent d'en déduire un analogue du second. 

La théorie de l’homotopie des catégories supérieures a fait l'objet d'importants travaux ces dernières années. Dans la catégorie $\DeuxCat$, dont les objets sont les petites \deux{}catégories (strictes) et les morphismes les \DeuxFoncteursStricts{}, se distinguent trois notions possibles d’é\-qui\-va\-lence faible. La plus forte est celle d'équivalence de \deux{}catégories. Une notion d’équivalence faible moins restrictive est celle d'équivalence « de Dwyer-Kan », \deux{}foncteur induisant des équivalences faibles entre les nerfs des catégories de morphismes et une équivalence des catégories obtenues des \deux{}catégories en remplaçant les catégories de morphismes par leur $\pi_{0}$. La catégorie homotopique obtenue en inversant ces équivalences faibles est, au moins conjecturalement, celle des $(\infty, 1)$\nobreakdash-catégories. La définition du foncteur nerf se généralisant aux catégories supérieures \cite{Duskin, StreetAOS}, on peut enfin s’intéresser, et c'est ce que nous ferons, aux équivalences faibles « de Thomason », morphismes de $\DeuxCat$ dont le nerf est une équivalence faible simpliciale.

Dans \cite{BC}, Bullejos et Cegarra démontrent une généralisation aux \DeuxFoncteursStricts{} du Théo\-rème A de Quillen. C'est le point de départ de la généralisation \deux{}catégorique de la notion de localisateur fondamental, et donc d'une théorie de l'homotopie \deux{}catégorique « à la Grothendieck », dont les bases et les résultats fondamentaux se trouvent exposés dans cette thèse une fois le formalisme et les résultats \deux{}catégoriques pertinents dégagés. L'adoption du point de vue de Grothendieck permet notamment de démontrer l'équivalence en un sens remarquablement fort des théories homotopiques de $\Cat$ et de $\DeuxCat$. 

Avant de donner davantage de détails, terminons ce bref aperçu par un retour à la théorie des catégories de modèles de Quillen, bien qu'elle ne se trouve pas abordée directement dans cette étude. Worytkiewicz, Hess, Parent et Tonks affirment construire dans \cite{WHPT} une structure de catégorie de modèles sur $\DeuxCat$ avec les équivalences faibles de Thomason comme équivalences faibles et Quillen\nobreakdash-équivalente à la structure de catégorie de modèles classique sur les ensembles simpliciaux. Maltsiniotis ayant noté que les auteurs ne donnaient aucun argument démontrant que l'adjonction de Quillen considérée constituait une équivalence de Quillen, l'un des objectifs de cette thèse était d'établir cet argument manquant. On en trouvera la démonstration dans le présent travail. Mais, en fait, comme l'ont remarqué plus tard Ara et Maltsiniotis, plusieurs passages cruciaux de la démonstration de l'existence de cette structure de catégorie de modèles sont incorrects dans \cite{WHPT}. Dans \cite{AraMaltsi}, Ara et Maltsiniotis démontrent l'existence de cette adjonction de Quillen et déduisent de nos résultats qu'il s'agit bien d'une équivalence de Quillen. De plus, Ara, dans \cite{Ara}, explique comment les résultats de \cite{AraMaltsi} se combinent aux nôtres, exposés dans \cite{ArticleLocFondMoi}, pour établir, \emph{pour essentiellement tout localisateur fondamental de $\DeuxCat$}, l'existence d'une structure de catégorie de modèles sur $\DeuxCat$ dont le localisateur fondamental en question constitue la classe des équivalences faibles. Le résultat se précise de façon remarquable pour donner une équivalence de Quillen entre cette structure de catégorie de modèles et celle associée sur $\Cat$, l'une étant de plus propre si et seulement si l'autre l'est. Les résultats de Cisinski permettent en outre de montrer que ces structures sont Quillen-équivalentes aux structures de catégories de modèles associées sur la catégorie des ensembles simpliciaux. 

\section*{Présentation des résultats}

Nous présentons dans cette section trois de nos résultats : une version \deux{}catégorique très générale du Théorème A de Quillen, l'équivalence des catégories homotopiques de $\Cat$ et de $\DeuxCat$ et une correspondance bijective entre localisateurs fondamentaux de $\Cat$ et de $\DeuxCat$. L'énoncé de ces résultats s'appuie sur le formalisme introduit dans le premier chapitre, qui contient lui-même des résultats originaux que nous ne mentionnons pas ici.

Pour tout morphisme $u : A \to B$ de $\Cat$ et tout objet $b$ de $B$, nous noterons $A/b$ la catégorie dont les objets sont les couples $(a, p : u(a) \to b)$, avec $a$ un objet de $A$ et $p$ un morphisme de $B$, et dont les morphismes de $(a,p)$ vers $(a',p')$ sont les morphismes $f : a \to a'$ de $A$ tels que $p' u(f) = p$. 

Pour tout diagramme commutatif
$$
\xymatrix{
A 
\ar[rr]^{u}
\ar[dr]_{w}
&&
B
\ar[dl]^{v}
\\
&
C
}
$$
dans $\Cat$ et tout objet $c$ de $C$, on notera $u/c$ le foncteur défini par
$$
\begin{aligned}
A/c &\to B/c
\\
(a, p) &\mapsto (u(a), p)
\\
f &\mapsto u(f).
\end{aligned}
$$

Cette définition permet d'énoncer le cas relatif du Théorème A de Quillen. Nous noterons $\UnLocFondMin$ la classe des morphismes de $\Cat$ dont le nerf est une équivalence faible simpliciale.

\begin{thmp}[Quillen]
Soit
$$
\xymatrix{
A 
\ar[rr]^{u}
\ar[dr]_{w}
&&
B
\ar[dl]^{v}
\\
&
C
}
$$
un diagramme commutatif dans $\Cat$. Supposons que, pour tout objet $c$ de $C$, le foncteur $u/c$ soit dans $\UnLocFondMin$. Alors $u$ est dans $\UnLocFondMin$. 
\end{thmp}


Rappelons maintenant la notion de localisateur fondamental de $\Cat$, due à Grothendieck\footnote{Nous appelons « localisateur fondamental de $\Cat$ » ce que Maltsiniotis appelle « localisateur fondamental » dans \cite{THG}. La terminologie de Grothendieck varie dans ses écrits.} ; c'est une classe de foncteurs entre petites catégories permettant de « faire de l'homotopie » dans $\Cat$. Nous noterons $\UnCatPonct$ la catégorie ponctuelle, n'ayant qu'un seul objet et qu'un seul morphisme. 

\begin{dfp}[Grothendieck]\label{DefUnLocFondIntro}
On appelle \emph{\ClasseUnLocFond{}} une classe $\UnLocFond{W}$ de morphismes de $\Cat$ vérifiant les conditions suivantes.
\begin{itemize}
\item[LA] La classe $\UnLocFond{W}$ est faiblement saturée. Autrement dit, les propriétés suivantes sont vérifiées. 
\begin{itemize}
\item[FS1] Les identités des objets de $\Cat$ sont dans $\UnLocFond{W}$.
\item[FS2] Si deux des trois flèches d'un triangle commutatif de morphismes de $\Cat$ sont dans $\UnLocFond{W}$, alors la troisième l'est aussi.
\item[FS3] Si $i : X \to Y$ et $r : Y \to X$ sont des morphismes de $\Cat$ vérifiant $ri = 1_{X}$ et si $ir$ est dans $\UnLocFond{W}$, alors il en est de même de $r$ (et donc aussi de $i$ en vertu de ce qui précède). 
\end{itemize}
\item[LB] Si $A$ est une petite catégorie admettant un objet final, alors le morphisme canonique $A \to \UnCatPonct$ est dans $\UnLocFond{W} $.
\item[LC] Pour tout diagramme commutatif 
$$
\xymatrix{
A 
\ar[rr]^{u}
\ar[dr]_{w}
&&
B
\ar[dl]^{v}
\\
&
C
}
$$
dans $\Cat$, si, pour tout objet $c$ de $C$, le foncteur $u/c$ est dans $\UnLocFond{W}$, alors $u$ est dans $\UnLocFond{W}$. 
\end{itemize}
\end{dfp} 

On remarque que l'intersection de localisateurs fondamentaux de $\Cat$ en est un. On définit le \emph{localisateur fondamental minimal de $\Cat$} comme l'intersection de tous les localisateurs fondamentaux de $\Cat$. C'est donc un localisateur fondamental de $\Cat$, contenu dans tous les autres. Le théorème \ref{CisinskiGrothendieckIntro} a été conjecturé par Grothendieck et démontré par Cisinski. 

\begin{thmp}[Cisinski]\label{CisinskiGrothendieckIntro}
Le localisateur fondamental minimal de $\Cat$ est $\UnLocFondMin$.
\end{thmp}

Pour tout \DeuxFoncteurCoLax{} (c'est-à-dire ce que d'autres appellent « \deux{}foncteur oplax ») \mbox{$u : \mathdeuxcat{A} \to \mathdeuxcat{B}$} entre petites \deux{}catégories et tout objet $b$ de $\mathdeuxcat{B}$, l'on définit une \deux{}catégorie $\TrancheCoLax{\mathdeuxcat{A}}{u}{b}$ comme suit. Ses objets sont les couples $(a, p : u(a) \to b)$, avec $a$ un objet de $\mathdeuxcat{A}$ et $p$ une \un{}cellule de $\mathdeuxcat{B}$. Les \un{}cellules de $(a, p)$ vers $(a', p')$ sont les couples $(f : a \to a', \alpha : p' u(f) \Rightarrow p)$, avec $f$ une \un{}cellule de $\mathdeuxcat{A}$ et $\alpha$ une \deux{}cellule de $\mathdeuxcat{B}$. Les \deux{}cellules de $(f, \alpha)$ vers $(f', \alpha')$ sont les \deux{}cellules $\beta : f \Rightarrow f'$ dans $\mathdeuxcat{A}$ telles que $\alpha' \CompDeuxUn (p' \CompDeuxZero u(\beta)) = \alpha$. Les diverses unités et compositions sont définies de façon « évidente ». Pour tout diagramme commutatif
$$
\xymatrix{
\mathdeuxcat{A} 
\ar[rr]^{u}
\ar[dr]_{w}
&&\mathdeuxcat{B}
\ar[dl]^{v}
\\
&\mathdeuxcat{C}
}
$$
de \deux{}foncteurs colax entre petites \deux{}catégories, on définit un \DeuxFoncteurCoLax{}
$$
\begin{aligned}
\DeuxFoncTrancheCoLax{u}{c} : \TrancheCoLax{\mathdeuxcat{A}}{w}{c} &\to \TrancheCoLax{\mathdeuxcat{B}}{v}{c}
\\
(a,p)&\mapsto (u(a),p)
\\
(f,\alpha) &\mapsto (u(f), \alpha)
\\
\beta &\mapsto u(\beta). 
\end{aligned}
$$

La catégorie $\DeuxCat$ est une sous-catégorie de la catégorie $\DeuxCatLax$, dont les objets sont les petites \deux{}catégories et dont les morphismes sont les \emph{\DeuxFoncteursLax}. De façon analogue à celle exposée ci-dessus, pour tout morphisme $u : \mathdeuxcat{A} \to \mathdeuxcat{B}$ de $\DeuxCatLax$ et tout objet $b$ de $\mathdeuxcat{B}$, l'on définit une \deux{}catégorie $\TrancheLax{\mathdeuxcat{A}}{u}{b}$ comme suit. Ses objets sont les couples $(a, p : u(a) \to b)$ avec $a$ un objet de $\mathdeuxcat{A}$ et $p$ une \un{}cellule de $\mathdeuxcat{B}$. Les \un{}cellules de $(a,p)$ vers $(a',p')$ sont les couples $(f : a \to a', \alpha : p \Rightarrow p' u(f))$ avec $f$ une \un{}cellule de $\mathdeuxcat{A}$ et $\alpha$ une \deux{}cellule de $\mathdeuxcat{B}$. Les \deux{}cellules de $(f, \alpha)$ vers $(f', \alpha')$ sont les \deux{}cellules $\beta : f \Rightarrow f'$ de $\mathdeuxcat{A}$ telles que $(p' \CompDeuxZero u(\beta)) \alpha = \alpha'$. Les diverses unités et compositions sont définies de façon « évidente ». De plus, pour tout diagramme commutatif
$$
\xymatrix{
\mathdeuxcat{A} 
\ar[rr]^{u}
\ar[dr]_{w}
&&\mathdeuxcat{B}
\ar[dl]^{v}
\\
&\mathdeuxcat{C}
}
$$
dans $\DeuxCatLax$, on définit un \DeuxFoncteurLax{} 
$$
\begin{aligned}
\DeuxFoncTrancheLax{u}{c} : \TrancheLax{\mathdeuxcat{A}}{w}{c} &\to \TrancheLax{\mathdeuxcat{B}}{v}{c}
\\
(a,p)&\mapsto (u(a),p)
\\
(f,\alpha) &\mapsto (u(f), \alpha)
\\
\beta &\mapsto u(\beta). 
\end{aligned}
$$
Rappelons de plus que, de même qu'il existe une notion de transformation naturelle (ou « morphisme de foncteurs ») dans $\Cat$, il existe des notions, duales l'une de l'autre, de \emph{\DeuxTransformationLax} et d'\emph{\DeuxTransformationCoLax} entre \DeuxFoncteursLax{} (de même qu'entre \DeuxFoncteursCoLax{}). Dans la littérature, ces notions se trouvent parfois désignées par les termes « transformation lax » et « transformation oplax ». Nous établirons toutes nos conventions terminologiques dans le premier chapitre. Pour tout diagramme de \DeuxFoncteursLax{}
$$
\xymatrix{
\mathdeuxcat{A} 
\ar[rr]^{u}
\ar[dr]_{w}
&&\mathdeuxcat{B}
\dtwocell<\omit>{<7.3>\sigma}
\ar[dl]^{v}
\\
&\mathdeuxcat{C}
&{}
}
$$
commutatif à l'\DeuxTransformationCoLax{} $\sigma : vu \Rightarrow w$ près seulement, on peut, généralisant la construction ci-dessus, définir un \DeuxFoncteurLax{} 
$$
\DeuxFoncTrancheLaxCoq{u}{\sigma}{c} : \TrancheLax{\mathdeuxcat{A}}{w}{c} \to \TrancheLax{\mathdeuxcat{B}}{v}{c}. 
$$ 

Rappelons enfin que le foncteur nerf se généralise aux \deux{}catégories. Les diverses définitions « raisonnables » possibles sont équivalentes du point de vue homotopique. Le nerf que nous noterons $\NerfLaxNor$ est fonctoriel sur les morphismes de $\DeuxCat$ et définit donc un foncteur \mbox{$\NerfLaxNor : \DeuxCat \to \EnsSimp$} ; celui que nous noterons $\NerfLax$ l'est sur les morphismes de $\DeuxCatLax$ et définit donc un foncteur $\NerfLax : \DeuxCatLax \to \EnsSimp$. Nous noterons $\DeuxLocFondMin$ la classe des morphismes de $\DeuxCat$ dont l'image par le foncteur $\NerfLaxNor$ (ou, de façon équivalente, par le foncteur $\NerfLax$) est une équivalence faible simpliciale. Nous appellerons ces morphismes les \emph{équivalences faibles} de $\DeuxCat$. Nous appellerons de plus \emph{équivalences faibles} les morphismes de $\DeuxCatLax$ dont l'image par le foncteur $\NerfLax$ est une équivalence faible simpliciale. 

\subsection*{Théorèmes A 2-catégoriques}

Il est naturel de se demander si les catégories localisées de $\Cat$ et de $\DeuxCat$ par leurs équivalences faibles sont équivalentes. Nous répondons par l'affirmative à cette question.  

Pour ce faire, on pourrait naïvement souhaiter trouver un remplacement fonctoriel \un{}ca\-té\-go\-rique de toute petite \deux{}catégorie qui lui serait faiblement équivalent, mais un argument simple montre le caractère illusoire d'un tel espoir : certaines petites \deux{}catégories ne sont la source (\emph{resp.} le but) d'\emph{aucune} équivalence faible de $\DeuxCat$ dont le but (\emph{resp.} la source) est une petite \un{}catégorie (proposition \ref{ObstructionGeorges}). Toutefois, l'obstruction disparaît si l'on autorise les zigzags d'équivalences faibles ou si l'on élargit les possibilités aux morphismes qui ne respectent les unités et la composition qu'à \deux{}cellule près — morphismes définis par Bénabou dans son article fondateur de la théorie des bicatégories \cite{BenabouIntro}. Cette dernière remarque illustre l'importance de ces morphismes dans l'étude homotopique des catégories supérieures, importance que l'on peut déjà percevoir dans \cite{BC}, \cite{Cegarra} ou \cite{WHPT}, et c'est l'une des motivations pour une étude homotopique des morphismes lax entre \deux{}catégories, même indépendamment de leur rôle dans la théorie des localisateurs fondamentaux, sur lequel on reviendra.  

Nous avons déjà rappelé que, pour établir l'équivalence des catégories homotopiques de $\EnsSimp$ et $\Cat$, une stratégie possible consiste à utiliser le Théorème A de Quillen pour démontrer que le foncteur $\SupUn_{A} : \Delta / \UnNerf A \to A$ est une équivalence faible pour toute petite catégorie $A$. On a rappelé que le foncteur nerf se généralisait aux \deux{}catégories. On constate toutefois que les divers analogues \deux{}catégoriques possibles du foncteur $\SupUn_{A}$ ci-dessus ne sont pas stricts en général, ce qui confirme l'importance des morphismes lax en théorie de l'homotopie. Par analogie avec le cas de $\Cat$, dans le but d'établir l'équivalence des catégories homotopiques de $\EnsSimp$ et de $\DeuxCat$, il apparaît donc particulièrement pertinent d'essayer de dégager une généralisation du Théorème A de Quillen aux morphismes lax entre \deux{}catégories. Lorsque nous nous sommes renseigné sur l'état de cette question dans la littérature, Cegarra nous a communiqué la thèse de Matias del Hoyo \cite{TheseDelHoyo} et ce dernier, dont certains résultats se trouvent exposés dans l'article \cite{ArticleDelHoyo}, nous a transmis ses notes \cite{NotesDelHoyo}. Les deux premiers textes établissent un analogue du Théorème A de Quillen pour les \deux{}foncteurs lax \emph{normalisés}. Del Hoyo étend ce résultat aux \deux{}foncteurs lax généraux dans \cite{NotesDelHoyo}, mais il ne traite pas le cas relatif. Il déduit son résultat d'une observation portant sur une propriété homotopique d'un adjoint à gauche, construit par Bénabou, de l'inclusion canonique $\DeuxCat \hookrightarrow \DeuxCatLax$. 

Le passage au cas relatif du Théorème A pour les \DeuxFoncteursLax{} ne relève pas de la routine. Nous l'avons dégagé à partir d'une autre propriété de l'adjonction de Bénabou. En prenant connaissance de ce cas relatif, Maltsiniotis a conjecturé une version \deux{}catégorique plus générale encore (et, peut-être, la plus générale possible), sous l'hypothèse d'un triangle non pas commutatif en général, mais commutatif \emph{à \deux{}cellule près} seulement. Les développements conceptuels et les résultats généraux du premier chapitre nous ont permis de la démontrer.  

\begin{thmp}\label{TheoremeALaxTrancheLaxCoqIntro}
Soit
$$
\xymatrix{
\mathdeuxcat{A} 
\ar[rr]^{u}
\ar[dr]_{w}
&&\mathdeuxcat{B}
\dtwocell<\omit>{<7.3>\sigma}
\ar[dl]^{v}
\\
&\mathdeuxcat{C}
&{}
}
$$
un diagramme de \DeuxFoncteursLax{} commutatif à l'\DeuxTransformationCoLax{} $\sigma : vu \Rightarrow w$ près seulement. Supposons que, pour tout objet $c$ de $\mathdeuxcat{C}$, le \DeuxFoncteurLax{} 
$$
\DeuxFoncTrancheLaxCoq{u}{\sigma}{c} : \TrancheLax{\mathdeuxcat{A}}{w}{c} \to \TrancheLax{\mathdeuxcat{B}}{v}{c}
$$ 
soit une équivalence faible. Alors $u$ est une équivalence faible.
\end{thmp}


Ce résultat admet bien entendu des variantes duales. 

Les généralisations du Théorème A permettent de contourner l'obstruction mentionnée au début de cette sous-section. On peut même le faire de deux façons, comme il était possible de le conjecturer : soit par l'intermédiaire d'un zigzag de morphismes stricts, que l'on peut rendre de longueur minimale, c'est-à-dire de longueur $2$ (voir la démonstration du théorème \ref{EqCatLocCatDeuxCat}), soit par l'intermédiaire d'un remplacement fonctoriel \un{}catégorique de toute petite \deux{}catégorie, faiblement équivalent à la \deux{}catégorie de départ dans $\DeuxCatLax$ (proposition \ref{SupLaxW}). Les propriétés de l'adjonction de Bénabou entre $\DeuxCat$ et $\DeuxCatLax$ permettent alors de repasser dans $\DeuxCat$ (voir le théorème \ref{EqCatLocDeuxCatDeuxCatLax}). Ces résultats se généralisent en fait, de façon naturelle, dans le cas d'un localisateur fondamental arbitraire, ce que nous expliquons dans la sous-section suivante.  

\subsection*{Équivalence des catégories homotopiques de $\Cat$ et de $\DeuxCat$}


Pour généraliser la définition de \ClasseUnLocFond{} au contexte \deux{}catégorique, il nous reste à dégager l'analogue pertinent, de ce point de vue — c'est-à-dire du point de vue homotopique —, de la notion d'objet final d'une petite catégorie. C'est la raison d'être de la définition \ref{DefOF2Intro}.

\begin{dfp}\label{DefOF2Intro}
On dira qu'un objet $z$ d'une \deux{}catégorie $\mathdeuxcat{A}$ \emph{admet un objet final} si, pour tout objet $a$ de $\mathdeuxcat{A}$, la catégorie $\CatHom{\mathdeuxcat{A}}{a}{z}$ admet un objet final.
\end{dfp}

À ce stade, la définition \ref{DefDeuxLocFondIntro} ne devrait pas provoquer la surprise. 

\begin{dfp}\label{DefDeuxLocFondIntro}
Un \emph{localisateur fondamental de $\DeuxCat$} est une classe $\DeuxLocFond{W}$ de morphismes de $\DeuxCat$ vérifiant les propriétés suivantes.
\begin{itemize}
\item[LF1] La classe $\DeuxLocFond{W}$ est faiblement saturée.

\item[LF2] Si une petite \deux{}catégorie $\mathdeuxcat{A}$ admet un objet admettant un objet final, alors le morphisme canonique $\mathdeuxcat{A} \to e$ est dans $\DeuxLocFond{W}$.

\item[LF3] Si 
$$
\xymatrix{
\mathdeuxcat{A} 
\ar[rr]^{u}
\ar[dr]_{w}
&&\mathdeuxcat{B}
\ar[dl]^{v}
\\
&\mathdeuxcat{C}
}
$$
désigne un diagramme commutatif dans $\DeuxCat$ et si, pour tout objet $c$ de $\mathcal{C}$, le \DeuxFoncteurStrict{}
$$
\DeuxFoncTrancheCoLax{u}{c} : \TrancheCoLax{\mathdeuxcat{A}}{w}{c} \to \TrancheCoLax{\mathdeuxcat{B}}{v}{c} 
$$
est dans $\DeuxLocFond{W}$, alors $u$ est dans $\DeuxLocFond{W}$.
\end{itemize}
\end{dfp}


Les résultats relatifs à la théorie de l'homotopie de $\DeuxCat$ ne dépendent, conformément à ce que le point de vue de Grothendieck permettait de prédire, que des axiomes des localisateurs fondamentaux de $\DeuxCat$. C'est en particulier le cas de l'équivalence homotopique de $\Cat$ et de $\DeuxCat$. Pour toute catégorie $A$, nous noterons $\UnCell{A}$ la classe des morphismes de $A$.

\begin{thmp}\label{EqCatLocCatDeuxCatIntro}
Pour tout localisateur fondamental $\DeuxLocFond{W}$ de $\DeuxCat$, l'inclusion canonique $\Cat \hookrightarrow \DeuxCat$ induit des équivalences de catégories entre les catégories localisées
$$
\Localisation{\DeuxCat}{\DeuxLocFond{W}} \simeq \Localisation{\Cat}{(\DeuxLocFond{W} \cap \UnCell{\Cat})}.
$$
\end{thmp}

\subsection*{Correspondance entre localisateurs fondamentaux de $\Cat$ et de $\DeuxCat$}

Au début de cette thèse, la définition des localisateurs fondamentaux de $\DeuxCat$ n'était pas encore dégagée. La théorie des localisateurs fondamentaux, pour conjecturale qu'elle fût encore, n'en contenait pas moins quelques principes directeurs devant permettre de vérifier la correction des concepts. L'un d'entre eux requérait l'invariance de l'axiomatique par dualité, question que nous résolvons dans la section \ref{SectionDeuxLocFond}. Un autre stipulait une hypothétique correspondance entre localisateurs fondamentaux de $\Cat$ et de $\DeuxCat$. Autrement dit, les classes de morphismes « permettant de faire de l'homotopie dans $\Cat$ » devaient correspondre à celles « permettant de faire de l'homotopie dans $\DeuxCat$ ». Nous établissons cette correspondance fondamentale dans le troisième chapitre du présent travail. De façon surprenante, la démonstration que nous avons trouvée de cette correspondance passe une fois encore par l'étude des propriétés des morphismes non-stricts, en définissant la notion de \emph{\ClasseDeuxLocFondLax{}} (définition \ref{DefDeuxLocFondLax}). Cette dernière permet d'expliciter la correspondance cherchée, au moyen du nerf $\NerfLaxNor : \DeuxCat \to \EnsSimp$ et du foncteur « catégorie des simplexes » $i_{\Delta} : \EnsSimp \to \Cat$. 

\begin{thmp}\label{IsoUnLocFondDeuxLocFondIntro}
Les applications
$$
\begin{aligned}
\mathcal{P}(\UnCell{\Cat}) &\to \mathcal{P}(\UnCell{\DeuxCat})
\\
\UnLocFond{W} &\mapsto {\NerfLaxNor}^{-1} (i_{\Delta}^{-1} (\UnLocFond{W}))
\end{aligned}
$$
et
$$
\begin{aligned}
\mathcal{P}(\UnCell{\DeuxCat}) &\to \mathcal{P}(\UnCell{\Cat})
\\
\DeuxLocFond{W} &\mapsto \DeuxLocFond{W} \cap \UnCell{\Cat}
\end{aligned}
$$
induisent des isomorphismes inverses l'un de l'autre entre la classe ordonnée par inclusion des \ClassesUnLocFond{} et la classe ordonnée par inclusion des \ClassesDeuxLocFond{}. De plus, ces isomorphismes induisent des équivalences de catégories au niveau des catégories localisées.
\end{thmp}

Pour montrer ce résultat, le plus important de notre travail, nous établissons d'abord une correspondance analogue entre localisateurs fondamentaux de $\DeuxCat$ et ceux de $\DeuxCatLax$ (théorème \ref{IsoDeuxLocFondDeuxLocFondLax}), puis entre ceux de $\Cat$ et ceux de $\DeuxCatLax$ (théorème \ref{IsoUnLocFondDeuxLocFondLax}). Une étape fondamentale consiste à démontrer que la catégorie des simplexes du nerf d'une petite \deux{}catégorie quelconque est faiblement équivalente à cette dernière, généralisant le résultat analogue portant sur le foncteur $\SupUn_{A} : \Delta / \UnNerf A \to A$ pour toute petite catégorie $A$. 

Le théorème \ref{IsoUnLocFondDeuxLocFondIntro} permet de plus de démontrer le caractère minimal du localisateur fondamental $\DeuxLocFondMin$ (théorème \ref{TheoDeuxLocFondMin}), analogue \deux{}catégorique du théorème \ref{CisinskiGrothendieckIntro} de Cisinski. 

La correspondance que nous avons démontrée laisse penser qu'il ferait sens de parler de « localisateur fondamental », indépendamment de toute catégorie « de base », qu'il s'agisse de $\Cat$, $\DeuxCat$, $\DeuxCatLax$ ou toute autre catégorie de structures supérieures à laquelle des travaux ultérieurs pourraient s'intéresser.

\section*{Plan de la thèse}

Après la présente introduction, ce travail s'organise en trois chapitres. 

Le premier présente le formalisme \deux{}catégorique nécessaire à l'établissement de nos principaux résultats. Contrairement à ce que son titre pourrait laisser penser, il contient des résultats originaux tout en comblant par ailleurs certaines carences de la littérature.  

Après un rappel des définitions fondamentales de théorie des \deux{}catégories dans la section \ref{SectionDefFond}, nous démontrons en détail, dans la section \ref{SectionTransHomotopie}, qu'une \deux{}cellule entre morphismes de $\DeuxCatLax$ constitue un cas particulier d'« homotopie lax ». La section \ref{SectionAdjonctionsCatDeuxCat} rappelle l'existence d'adjoints, à gauche comme à droite, de l'inclusion canonique $\Cat \hookrightarrow \DeuxCat$. Les sections \ref{SectionDeuxCategoriesTranches} et \ref{SectionMorphismesInduits} présentent, de façon méthodique, les généralisations au cadre \deux{}catégorique des notions catégoriques bien connues de catégories tranches ou « commas » et des morphismes induits entre icelles. Ces sections sont incluses ici par souci de cohérence et volonté de fournir un traitement systématique de notions qui ne se trouvaient pas encore organisées de la sorte dans la littérature. 

Dans la section \ref{SectionPreadjoints}, nous introduisons, suivant une suggestion de Maltsiniotis, une notion de \emph{préadjoint}, morphisme de $\DeuxCat$ défini par une propriété généralisant une définition possible des foncteurs adjoints dans $\Cat$. Cela permet toutefois d'observer que la simplicité des définitions dans $\Cat$ semble cacher quelques subtilités que seul un travail dans le cadre des catégories supérieures pourrait permettre de comprendre. Cette première définition \deux{}catégorique, bien qu'elle suffise aux besoins de notre recherche, souffre en effet de certaines lacunes. Si nous avons démontré qu'à tout préadjoint se trouvait associé un morphisme « en sens inverse », le couple obtenu ne jouit pas de toutes les propriétés souhaitables. Par conséquent, motivé par les résultats que nous présenterons plus loin, nous introduisons dans la section \ref{SectionAdjonctions}, toujours sur une suggestion de Maltsiniotis, la notion d'\emph{adjonction lax-colax}, formée d'un foncteur lax et d'un foncteur colax ; nous dirons de chacun d'eux qu'il est \emph{adjoint}\footnote{Après le dépôt de ce travail, Steve Lack a suggéré d'examiner les relations qu'entretenaient les notions que nous étudions dans les sections 1.6 et 1.7 avec celle d'\emph{adjonction locale} de Betti-Power \cite{BettiPower} ainsi que la variante de cette dernière étudiée par Verity dans sa thèse \cite{Verity}. Cela nous a permis de constater que ces notions, que nous croyions originales, figuraient déjà dans la littérature, de même que les résultats que nous présentons dans ces deux sections. Le lecteur trouvera des détails dans les sections concernées.}. Nous vérifions que tout adjoint est un préadjoint. C'est toutefois cette dernière notion que nous utilisons pour, dans la section \ref{SectionPrefibrations}, proposer une notion de \emph{préfibration} dans $\DeuxCat$. Précisons dès à présent que ces notions, introduites dans le présent travail, pourraient se trouver amenées à changer de nom lorsqu'elles auront fait l'objet d'études plus poussées. 

Après les développements techniques de la section \ref{SectionMorphismesTranches}, la section \ref{SectionIntegration} développe, dans le cadre \deux{}catégorique, une théorie similaire à celle, que nous nommerons d'\emph{intégration}, de la « construction de Grothendieck » que nous avons déjà mentionnée. L'idée d'une telle généralisation n'est, en soi, pas originale, et nous n'abordons pas la question dans un contexte aussi général que d'autres l'ont fait avant nous. Notre traitement se distingue toutefois, entre autres points, par l'étude systématique des questions de dualité. On explicite notamment qu'à ces dualités près, dualités d'ailleurs assez subtiles, il n'existe qu'une seule procédure d'intégration. De plus, et peut-être surtout, nous dégageons certaines propriétés homotopiques de cette construction. On vérifie notamment qu'elle fournit une adjonction lax-colax et donc, en particulier, une préfibration, la base de cette dernière étant la \deux{}catégorie source du morphisme que l'on intègre. 

Dans la section \ref{SectionCylindres}, nous introduisons deux constructions, parmi d'autres, analogues aux « cylindres » déjà considérés par Bénabou dans son article fondateur \cite{BenabouIntro}. Cela nous permettra, plus loin, d'établir l'invariance par dualité de la notion de localisateur fondamental de $\DeuxCat$. 

Nous consacrons la section \ref{SectionBenabou}, dernière du premier chapitre, à la présentation d'un adjoint à gauche de l'inclusion canonique $\DeuxCat \hookrightarrow \DeuxCatLax$. Cette construction, dégagée par Bénabou dans un cadre plus général, semble malheureusement n'avoir toujours pas fait l'objet d'une publication. Elle nous semble pourtant tenir une place suffisamment importante dans la théorie des catégories supérieures pour qu'il soit nécessaire de ne pas se contenter de faire appel au « folklore » quand le besoin de l'utiliser se fait sentir.

Dans le deuxième chapitre, on rappelle la notion de localisateur fondamental, au sens de Grothendieck, c'est-à-dire de ce que nous appelons \emph{localisateur fondamental de $\Cat$}, et l'on présente une façon d'obtenir, à partir d'une telle classe de morphismes de $\Cat$, une classe de morphismes de $\DeuxCat$ qui vérifiera les axiomes de ce que nous appelons les \emph{localisateurs fondamentaux de $\DeuxCat$}. 

Dans la section \ref{SectionUnLocFond}, on passe en revue quelques éléments de la théorie des localisateurs fondamentaux de $\Cat$. Si nous ne reprenons pas la démonstration de Cisinski du résultat de minimalité déjà évoqué, nous montrons le cas relatif du Théorème A de Quillen ainsi que l'équivalence des catégories homotopiques de $\EnsSimp$ et de $\Cat$. Les détails diffèrent légèrement de ce que le lecteur pourra trouver par ailleurs mais nous ne prétendons bien sûr nullement avoir fait preuve ici d'originalité. 

La section \ref{SectionNerfs} définit différents candidats possibles pour généraliser le foncteur nerf dans un contexte \deux{}catégorique. Si le choix n'est pas unique, il l'est à homotopie près. Pour cette section de rappels, la référence principale est \cite{CCG}. 

La section \ref{SectionSup} définit les analogues du foncteur $\SupUn_{A} : \Delta / \UnNerf A \to A$ pour les divers nerfs introduits dans la section précédente. Aucun de ces morphismes n'est strict en général. 

Dans la section \ref{SectionUnDeux}, nous expliquons une façon d'obtenir un localisateur fondamental de $\DeuxCat$ à partir d'un localisateur fondamental de $\Cat$. Pour cela, nous avons besoin du cas relatif du Théorème A \deux{}catégorique obtenu dans le cas absolu par Bullejos et Cegarra dans \cite{BC}. 

Les définitions et résultats du troisième chapitre sont tous originaux. 

Dans la section \ref{SectionDeuxLocFond}, nous introduisons la notion de localisateur fondamental de $\DeuxCat$ et démontrons les premiers résultats de la théorie. Les développements du premier chapitre permettent notamment une démonstration conceptuelle de l'invariance par dualité de l'axiomatique que nous proposons. 

Dans la section \ref{SectionHomotopieMorphismesLax}, les morphismes lax font leur apparition dans la théorie. Nous étudions notamment des propriétés homotopiques de l'adjonction de Bénabou : son unité et sa coünité sont des équivalences faibles terme à terme, pour tout localisateur fondamental de $\DeuxCat$. Cela généralise un résultat important de del Hoyo, par une méthode différente. 

Nous utilisons ces propriétés de l'adjonction de Bénabou pour démontrer, dans la section \ref{SectionTheoremeA}, une généralisation relative du Théorème A de Quillen aux \DeuxFoncteursLax{}. 

Ce qui précède nous permet de démontrer, dans la section \ref{SectionHoCatHoDeuxCat}, l'équivalence homotopique de $\Cat$ et de $\DeuxCat$, non seulement pour les équivalences faibles de Thomason, mais plus généralement pour tout localisateur fondamental. 

On utilise ensuite le cas relatif du Théorème A pour les \DeuxFoncteursLax{} pour dégager, dans la section \ref{SectionDeuxLocFondLax}, la notion de \emph{localisateur fondamental de $\DeuxCatLax$}. Comme le suggère la terminologie, il s'agit d'une classe de morphismes de $\DeuxCatLax$ vérifiant des propriétés analogues à celles des localisateurs fondamentaux de $\Cat$ ou de $\DeuxCat$. Cette notion nous permettra de « faire le lien » entre celle de localisateur fondamental de $\Cat$ et celle de localisateur fondamental de $\DeuxCat$ dans la section suivante. 

Comme annoncé, nous démontrons, dans la section \ref{SectionCorrespondances}, une correspondance remarquable entre localisateurs fondamentaux de $\Cat$, de $\DeuxCatLax$ et de $\DeuxCat$, cette correspondance étant de plus compatible à l'opération de localisation. Nous en déduisons notamment la minimalité du localisateur fondamental de $\DeuxCat$ formé des équivalences faibles de Thomason, analogue du résultat conjecturé par Grothendieck et démontré par Cisinski en dimension 1. 

La section \ref{SectionTheoremeAGeneral} présente le Théorème A relatif « non-com\-mu\-ta\-tif » pour les \DeuxFoncteursLax{}. 

Nous concluons, dans la section \ref{SectionCritereLocal}, par une autre caractérisation, à l'aide du Théorème B de Quillen celle-là, du localisateur fondamental minimal de $\DeuxCat$. Cela résulte à nouveau de la correspondance déjà mentionnée et de l'analogue pour $\Cat$, démontré par Cisinski, de cette caractérisation locale.

\chapter{Formalisme 2-catégorique}

\section{Définitions fondamentales}\label{SectionDefFond}

Nous rappelons dans cette section quelques définitions classiques. Le lecteur pourra consulter l'article fondateur de Bénabou \cite{BenabouIntro} ainsi que le plus récent fascicule de Leinster \cite{LeinsterBasic}, ces deux références traitant le cas général des bicatégories. Nous ne nous attarderons pas sur ces notions, le lecteur les ayant sans doute déjà rencontrées. Leur définition figure ici pour établir les conventions terminologiques que nous suivrons tout au long de cette étude et rendre la lecture de cette dernière plus aisée. Signalons toutefois dès maintenant que nous appelons « \DeuxTransformationCoLax{} » ce que Leinster appelle « transformation », et qu'il ne mentionne pas ce que nous appelons « \DeuxTransformationLax{} ». Voir aussi la remarque terminologique \ref{TerminologieTrans}.

On notera $\UnCatPonct$\index[not]{e@$\UnCatPonct$ (objet de $\Cat$)} la \emph{catégorie ponctuelle}\index{catégorie ponctuelle}, n'ayant qu'un seul objet, que l'on notera $*$, et qu'un seul morphisme (l'identité de $*$). 

\begin{df}\label{DefDeuxCat}
Une \deux{}catégorie\index{2categorie@\deux{}catégorie} $\mathdeuxcat{A}$ est définie par les données et conditions suivantes : 
\begin{itemize}
\item Une collection d'\emph{objets} $\Objets{\mathdeuxcat{A}}$\index[not]{ObA@$\Objets{\mathdeuxcat{A}}$}.
\item Pour tout couple d'objets $a$ et $a'$ de $\mathdeuxcat{A}$, une catégorie $\CatHom{\mathdeuxcat{A}}{a}{a'}$\index[not]{HomAaa'Souligne@$\CatHom{\mathdeuxcat{A}}{a}{a'}$}, dont les objets sont appelés \emph{\un{}cellules}\index{1cellule@\un{}cellule} de $a$ vers $a'$ et les morphismes \emph{\deux{}cellules}\index{2cellule@\deux{}cellule}. La composition de ces \deux{}cellules sera notée par la simple juxtaposition. 
\item Pour tout objet $a$ de $\mathdeuxcat{A}$, un \emph{foncteur d'unité}\index{foncteur d'unité}
$$
\IdObjet{a}\index[not]{1a@$\IdObjet{a}$} : \DeuxCatPonct \to \CatHom{\mathdeuxcat{A}}{a}{a}.
$$
\item Pour tout triplet $(a,a',a'')$ d'objets de $\mathdeuxcat{A}$, un \emph{foncteur de composition}\index{foncteur de composition}
$$
\begin{aligned}
\FonctComp{\mathdeuxcat{A}}{a}{a'}{a''}\index[not]{c0Aaa'a''@$\FonctComp{\mathdeuxcat{A}}{a}{a'}{a''}$} : \CatHom{\mathdeuxcat{A}}{a'}{a''} \times \CatHom{\mathdeuxcat{A}}{a}{a'} &\to \CatHom{\mathdeuxcat{A}}{a}{a''}
\\
(g,f) &\mapsto gf
\\
(\beta, \alpha) &\mapsto \beta \CompDeuxZero \alpha\index[not]{0betaalpha@$\beta \CompDeuxZero \alpha$}.
\end{aligned}
$$
\end{itemize}
On requiert que soient vérifiées les conditions de cohérence suivantes.
\begin{itemize}
\item 
Pour tout quadruplet $(a,a',a'',a''')$ d'objets de $\mathdeuxcat{A}$, le diagramme
$$
\xymatrix{
\CatHom{\mathdeuxcat{A}}{a''}{a'''} \times \CatHom{\mathdeuxcat{A}}{a'}{a''} \times \CatHom{\mathdeuxcat{A}}{a}{a'}  
\ar[rrr]^{1 \times \FonctComp{\mathdeuxcat{A}}{a}{a'}{a''}} 
\ar[dd]_{\FonctComp{\mathdeuxcat{A}}{a'}{a''}{a'''} \times 1}
&&& \CatHom{\mathdeuxcat{A}}{a''}{a'''} \times \CatHom{\mathdeuxcat{A}}{a}{a''}
\ar[dd]^{\FonctComp{\mathdeuxcat{A}}{a}{a''}{a'''}}
\\
\\
\CatHom{\mathdeuxcat{A}}{a'}{a'''} \times \CatHom{\mathdeuxcat{A}}{a}{a'}
\ar[rrr]_{\FonctComp{\mathdeuxcat{A}}{a}{a'}{a'''}}
&&& \CatHom{\mathdeuxcat{A}}{a}{a'''}
}
$$
est commutatif.

\item Pour tout couple $(a,a')$ d'objets de $\mathdeuxcat{A}$, les diagrammes
$$
\xymatrix{
\CatHom{\mathdeuxcat{A}}{a}{a'} \times \DeuxCatPonct
\ar[d]_{1 \times 1_{a}}
\ar[drr]^{\simeq}
\\
\CatHom{\mathdeuxcat{A}}{a}{a'} \times \CatHom{\mathdeuxcat{A}}{a}{a}
\ar[rr]_{\FonctComp{\mathdeuxcat{A}}{a}{a}{a'}}
&& \CatHom{\mathdeuxcat{A}}{a}{a'}
}
$$
et
$$
\xymatrix{
\DeuxCatPonct \times \CatHom{\mathdeuxcat{A}}{a}{a'}
\ar[d]_{1_{a'} \times 1}
\ar[drr]^{\simeq}
\\
\CatHom{\mathdeuxcat{A}}{a'}{a'} \times \CatHom{\mathdeuxcat{A}}{a}{a'}
\ar[rr]_{\FonctComp{\mathdeuxcat{A}}{a}{a'}{a'}}
&& \CatHom{\mathdeuxcat{A}}{a}{a'}
}
$$
sont commutatifs.
\end{itemize}
\end{df}

\begin{rem}
Une \deux{}catégorie n'est donc rien d'autre qu'une catégorie enrichie en catégories.
\end{rem}

\begin{rem}
On appellera souvent la propriété de fonctorialité des $\FonctComp{\mathdeuxcat{A}}{a}{a'}{a''}$ la « loi d'échange ». 
\end{rem}

\begin{paragr}
On notera par « $\to$ » les \un{}cellules et par « $\Rightarrow$ » les \deux{}cellules. Si $\alpha$ et $\beta$ sont des \deux{}cellules telles que la composée $\beta \CompDeuxUn \alpha$ (\emph{resp.} la composée $\beta \CompDeuxZero \alpha$) fasse sens, on pourra appeler cette composée leur \emph{composée verticale}\index{composée verticale} (\emph{resp.} leur \emph{composée horizontale}\index{composée horizontale}). On confondra souvent, dans les notations, les \un{}cellules avec leur identité. Ainsi, pour toute \un{}cellule $f$ et toute \deux{}cellule $\alpha$ telles que la composée $1_{f} \CompDeuxZero \alpha$ (\emph{resp.} $\alpha \CompDeuxZero 1_{f}$) fasse sens, on pourra noter cette composée $f \CompDeuxZero \alpha$ (\emph{resp.} $\alpha \CompDeuxZero f$).
\end{paragr}

\begin{exemple}
On notera $\DeuxCatPonct{}$\index[not]{ebis@$\DeuxCatPonct$ (objet de $\DeuxCat$)} la \deux{}catégorie n'ayant qu'un seul objet, que l'on notera $*$, une seule \un{}cellule (l'identité $1_{*}$) et une seule \deux{}cellule (l'identité $1_{1_{*}}$). On l'appellera la \deux{}\emph{catégorie ponctuelle}\index{2categorieponctuelle@\deux{}catégorie ponctuelle}. On commettra donc l'abus inoffensif de la noter de la même façon que la catégorie ponctuelle. 
\end{exemple}

\begin{df}\label{dfdeuxfoncteurlax}
Soient $\mathdeuxcat{A}$ et $\mathdeuxcat{B}$ des \deux{}catégories. Un \emph{\DeuxFoncteurLax}\index{2foncteurlax@\DeuxFoncteurLax}
$$u : \mathdeuxcat{A} \to \mathdeuxcat{B}$$
correspond aux données et conditions suivantes.
\begin{itemize}
\item 
Pour tout objet $a$ de $\mathdeuxcat{A}$, un objet $u(a)$ de $\mathdeuxcat{B}$. 

\item
Pour toute \un{}cellule $f : a \to a'$ de $\mathdeuxcat{A}$, une \un{}cellule $u(f) : u(a) \to u(a')$ de $\mathdeuxcat{B}$. 

\item
Pour toute \deux{}cellule $\alpha : f \Rightarrow g$ de $\mathdeuxcat{A}$, une \deux{}cellule $u(\alpha) : u(f) \Rightarrow u(g)$ de $\mathdeuxcat{B}$. 

\item Pour tout couple de \un{}cellules $f$ et $f'$ de $\mathdeuxcat{A}$ telles que la composée $f'f$ fasse sens, une \emph{\deux{}cellule structurale de composition}
$$
u_{f',f}\index[not]{uf'f@$u_{f',f}$} : u(f')u(f) \Rightarrow u(f'f).
$$

\item
Pour tout objet $a$ de $\mathdeuxcat{A}$, une \emph{\deux{}cellule structurale d'unité}
$$
u_{a}\index[not]{ua@$u_{a}$} : 1_{u(a)} \Rightarrow u(1_{a}).
$$
\end{itemize}

On requiert de ces données qu'elles vérifient les propriétés de cohérence suivantes.

\begin{itemize}
\item
Pour toute \un{}cellule $f$ de $\mathdeuxcat{A}$, 
$$
u(1_{f}) = 1_{u(f)}.
$$

\item 
Pour tout couple de \deux{}cellules $\alpha$ et $\alpha'$ de $\mathdeuxcat{A}$ telles que la composée $\alpha' \CompDeuxUn \alpha$ fasse sens, 
$$
u(\alpha' \CompDeuxUn \alpha) = u(\alpha') \CompDeuxUn u(\alpha).
$$

Ces deux dernières conditions expriment la \emph{fonctorialité de $u$ par rapport aux \deux{}cellules}.

\item Pour tout triplet de \un{}cellules $f''$, $f'$ et $f$ de $\mathdeuxcat{A}$ telles que la composée $f'' f' f$ fasse sens, le diagramme 
$$
\xymatrix{
u(f'')u(f')u(f) 
\ar@{=>}[rr]^{u_{f'',f'} \CompDeuxZero u(f)}
\ar@{=>}[dd]_{u(f'') \CompDeuxZero u_{f',f}}
&& u(f''f') u(f) 
\ar@{=>}[dd]^{u_{f''f', f}}
\\
\\
u(f'')u(f'f)
\ar@{=>}[rr]_{u_{f'',f'f}}
&& u(f''f'f)
}
$$
est commutatif. On appellera cette condition la \emph{condition de cocycle}\index{condition de cocycle}. 

\item 
Pour tout couple de \deux{}cellules $\alpha : f \Rightarrow g$ et $\alpha' : f' \Rightarrow g'$ de $\mathdeuxcat{A}$ telles que la composée $\alpha' \CompDeuxZero \alpha$ fasse sens, le diagramme
$$
\xymatrix{
u(f') u(f) 
\ar@{=>}[rr]^{u_{f',f}}
\ar@{=>}[dd]_{u(\alpha') \CompDeuxZero u(\alpha)}
&&
u(f'f)
\ar@{=>}[dd]^{u(\alpha' \CompDeuxZero \alpha)}
\\
\\
u(g')u(g)
\ar@{=>}[rr]_{u_{g',g}}
&&
u(g'g)
}
$$
est commutatif. Cette condition exprime la \emph{naturalité des \deux{}cellules structurales de composition de $u$}. 

\item
Pour toute \un{}cellule $f : a \to a'$ de $\mathdeuxcat{A}$, les diagrammes
$$
\xymatrix{
u(f) \IdObjet{u(a)}
\ar @{=} [rrrrd]
\ar@{=>}[rr]^{u(f) \CompDeuxZero u_{a}}
&& 
u(f) u(\IdObjet{a})
\ar@{=>}[rr]^{u_{f,1_{a}}}
&&
u(f \IdObjet{a})
\ar @{=} [d]
\\
&&&&u(f)
}
$$

et 

$$
\xymatrix{
\IdObjet{u(a')} u(f)
\ar @{=} [rrrrd]
\ar@{=>}[rr]^{u_{a'} \CompDeuxZero u(f)}
&&
u(\IdObjet{a'}) u(f)
\ar@{=>}[rr]^{u_{1_{a'},f}}
&&
u(\IdObjet{a'} f)
\ar @{=} [d]
\\
&&&&
u(f)
}
$$
sont commutatifs. Ces conditions expriment la \emph{naturalité des \deux{}cellules structurales d'unité de $u$}. On les appellera parfois les \emph{contraintes d'unité}. 
\end{itemize}
\end{df}

\begin{paragr}
Si, pour tout couple de \un{}cellules $f$ et $f'$ de $\mathdeuxcat{A}$ telles que la composée $f'f$ fasse sens, la \deux{}cellule $u_{f',f} : u(f') u(f) \Rightarrow u(f'f)$ est un isomorphisme (\emph{resp.} une identité) et si, pour tout objet $a$ de $\mathdeuxcat{A}$, la \deux{}cellule $u_{a} : 1_{u(a)} \Rightarrow u(1_{a})$ est un isomorphisme (\emph{resp.} une identité), on dira que $u$ est un \emph{pseudofoncteur}\index{pseudofoncteur} (\emph{resp.} un \deux{}\emph{foncteur strict}\index{2foncteurstrict@\DeuxFoncteurStrict}). Les pseudofoncteurs non stricts ne présentent qu'un intérêt médiocre du point de vue homotopique\footnote{Ils jouent en revanche un rôle de premier plan en géométrie algébrique ainsi qu'en théorie de la réécriture.}. Dans le cas d'un \DeuxFoncteurStrict{}, la condition de cocycle et les contraintes d'unité sont automatiques. Si, pour tout objet $a$ de $\mathdeuxcat{A}$, la \deux{}cellule $u_{a}$ est une identité (en particulier, $u(1_{a}) = 1_{u(a)}$) et si, pour toute \un{}cellule $f : a \to a'$ de $\mathdeuxcat{A}$, les \deux{}cellules $u_{1_{a'}, f}$ et $u_{f, 1_{a}}$ sont des identités, on dira que le \DeuxFoncteurLax{} $u$ est \emph{normalisé}\index{normalise2foncteurlax@normalisé (\DeuxFoncteurLax{})}. 
\end{paragr}

\begin{df}
Soient $\mathdeuxcat{A}$, $\mathdeuxcat{B}$ et $\mathdeuxcat{C}$ des \deux{}catégories et $u : \mathdeuxcat{A} \to \mathdeuxcat{B}$ et $v : \mathdeuxcat{B} \to \mathdeuxcat{C}$ des \DeuxFoncteursLax. On définit leur composée, notée $vu$, par les données suivantes.
\begin{itemize}
\item Pour tout objet $a$ de $\mathdeuxcat{A}$, 
$$
(vu)(a) = v(u(a)).
$$
\item Pour toute \un{}cellule $f$ de $\mathdeuxcat{A}$, 
$$
(vu)(f) = v(u(f)).
$$
\item Pour toute \deux{}cellule $\alpha$ de $\mathdeuxcat{A}$,  
$$
(vu)(\alpha) = v(u(\alpha)).
$$
\item Pour tout objet $a$ de $\mathdeuxcat{A}$, 
$$
(vu)_{a} = v(u_{a}) \CompDeuxUn{} v_{u(a)}.
$$
\item Pour tout couple de \un{}cellules $f$ et $f'$ de $\mathdeuxcat{A}$ telles que la composée $f'f$ fasse sens, 
$$
(vu)_{f',f} = v(u_{f',f}) \CompDeuxUn{} v_{u(f'), u(f)}.
$$
\end{itemize}
\end{df}

On vérifie sans aucune difficulté l'associativité de cette composition ainsi que l'existence des unités assurant que la définition \ref{DefTralala} fait sens.

De même que le choix d'un univers s'accompagne d'une notion de \emph{petite catégorie}, cela permet évidemment de disposer de la notion analogue de \emph{petite \deux{}catégorie}. De façon générale, nous nous abstiendrons d'évoquer les questions ensemblistes lorsque cela ne sera pas indispensable, en employant néanmoins parfois le terme « petite » pour rappeler que nous considérons surtout des petites \deux{}catégories. En l'absence de ce terme, le contexte devrait toujours rendre clair si les catégories ou \deux{}catégories considérées sont petites ou non. 

\begin{df}\label{DefTralala}
On notera $\DeuxCat{}$\index[not]{2Cat@$\DeuxCat$} (\emph{resp.} $\DeuxCatLax$\index[not]{2CatLax@$\DeuxCatLax$}) la catégorie dont les objets sont les petites \deux{}ca\-té\-go\-ries et dont les morphismes sont les \DeuxFoncteursStricts{} (\emph{resp.} les \DeuxFoncteursLax).
\end{df}

\begin{df}
\label{dfunetdeuxop}
On appellera \emph{\deux{}catégorie \un{}opposée}\index{2categorie1opposee@\deux{}catégorie \un{}opposée} d'une \deux{}catégorie $\mathdeuxcat{A}$ la \deux{}ca\-té\-go\-rie, notée $\DeuxCatUnOp{\mathdeuxcat{A}}$\index[not]{AZop@$\DeuxCatUnOp{\mathdeuxcat{A}}$}, obtenue à partir de $\mathdeuxcat{A}$ en inversant le sens des \un{}cellules. Plus précisément, on pose 
$$
\Objets{\DeuxCatUnOp{\mathdeuxcat{A}}} = \Objets{\mathdeuxcat{A}}
$$ 
et, pour tout couple d'objets $a$ et $a'$ de $\mathdeuxcat{A}$, 
$$
\CatHom{\DeuxCatUnOp{\mathdeuxcat{A}}}{a}{a'} = \CatHom{\mathdeuxcat{A}}{a'}{a},
$$
les diverses compositions et unités s'obtenant à partir de celles de $\mathdeuxcat{A}$ de façon « évidente ».

On appellera \emph{\deux{}catégorie \deux{}opposée}\index{2categorie2opposee@\deux{}catégorie \deux{}opposée} de $\mathdeuxcat{A}$ la \deux{}catégorie, notée $\DeuxCatDeuxOp{\mathdeuxcat{A}}$\index[not]{AZco@$\DeuxCatDeuxOp{\mathdeuxcat{A}}$}, obtenue à partir de $\mathdeuxcat{A}$ en inversant le sens des \deux{}cellules. Plus précisément, on pose 
$$
\Objets{\DeuxCatDeuxOp{\mathdeuxcat{A}}} = \Objets{\mathdeuxcat{A}}
$$
et, pour tout couple d'objets $a$ et $a'$ de $\mathdeuxcat{A}$, 
$$
\CatHom{\DeuxCatDeuxOp{\mathdeuxcat{A}}}{a}{a'} = (\CatHom{\mathdeuxcat{A}}{a}{a'})^{op},
$$
les diverses compositions et unités s'obtenant à partir de celles de $\mathdeuxcat{A}$ de façon « évidente ».

On appellera \emph{\deux{}catégorie opposée}\index{2categorieopposee@\deux{}catégorie opposée} de $\mathdeuxcat{A}$ la \deux{}catégorie, notée $\DeuxCatToutOp{\mathdeuxcat{A}}$\index[not]{AZcoop@$\DeuxCatToutOp{\mathdeuxcat{A}}$}, obtenue à partir de $\mathdeuxcat{A}$ en inversant le sens des \un{}cellules et le sens des \deux{}cellules. Plus précisément, on pose 
$$
\DeuxCatToutOp{\mathdeuxcat{A}} = \DeuxCatUnOp{(\DeuxCatDeuxOp{\mathdeuxcat{A}})} = \DeuxCatDeuxOp{(\DeuxCatUnOp{\mathdeuxcat{A}})}.
$$
On a donc en particulier 
$$
\Objets{\DeuxCatToutOp{\mathdeuxcat{A}}} = \Objets{\mathdeuxcat{A}}
$$
et, pour tout couple d'objets $a$ et $a'$ de $\mathdeuxcat{A}$,
$$
\CatHom{\DeuxCatToutOp{\mathdeuxcat{A}}}{a}{a'} = (\CatHom{\mathdeuxcat{A}}{a'}{a})^{op}.
$$
\end{df}

\begin{df}
Soient $\mathdeuxcat{A}$ et $\mathdeuxcat{B}$ des \deux{}catégories et $u$ un \DeuxFoncteurLax{} de $\mathdeuxcat{A}$ vers $\mathdeuxcat{B}$. On définit le \emph{\DeuxFoncteurLax{} \un{}opposé de} $u$\index{2foncteurlax1oppose@\DeuxFoncteurLax{} \un{}opposé}, noté $\DeuxFoncUnOp{u}$\index[not]{uop@$\DeuxFoncUnOp{u}$}, de source $\DeuxCatUnOp{\mathdeuxcat{A}}$ et de but $\DeuxCatUnOp{\mathdeuxcat{B}}$, par les données suivantes. Pour tout objet $a$ de $\mathdeuxcat{A}$, $\DeuxFoncUnOp{u} (a) = u(a)$. Pour toute \un{}cellule $f$ de $\mathdeuxcat{A}$, $\DeuxFoncUnOp{u} (f) = u(f)$. Pour toute \deux{}cellule $\alpha$ de $\mathdeuxcat{A}$, $\DeuxFoncUnOp{u} (\alpha) = u(\alpha)$. Pour tout couple de \un{}cellules $f$ et $f'$ de $\DeuxCatUnOp{\mathdeuxcat{A}}$ telles que la composée $f'f$ fasse sens (c'est-à-dire telles que la composée $ff'$ fasse sens dans $\mathdeuxcat{A}$), $(\DeuxFoncUnOp{u})_{f',f} = u_{f,f'}$. Pour tout objet $a$ de $\mathdeuxcat{A}$, $(\DeuxFoncUnOp{u})_{a} = u_{a}$.
\end{df}

On vérifie immédiatement que cela définit bien un \DeuxFoncteurLax{} $\DeuxFoncUnOp{u} :\DeuxCatUnOp{\mathdeuxcat{A}} \to \DeuxCatUnOp{\mathdeuxcat{B}}$. 

\begin{rem}
Étant donné des \deux{}catégories $\mathdeuxcat{A}$ et $\mathdeuxcat{B}$, il existe une bijection canonique entre l'ensemble des \DeuxFoncteursLax{} de $\mathdeuxcat{A}$ vers $\mathdeuxcat{B}$ et l'ensemble des \DeuxFoncteursLax{} de $\DeuxCatUnOp{\mathdeuxcat{A}}$ vers $\DeuxCatUnOp{\mathdeuxcat{B}}$, donnée par l'application $u \mapsto \DeuxFoncUnOp{u}$.
\end{rem}

Le lecteur aura peut-être déjà remarqué que la donnée d'un \DeuxFoncteurLax{} $u : \mathdeuxcat{A} \to \mathdeuxcat{B}$ ne permet pas de définir en général un \DeuxFoncteurLax{} de $\DeuxCatDeuxOp{\mathdeuxcat{A}}$ vers $\DeuxCatDeuxOp{\mathdeuxcat{B}}$. La définition \ref{DefFoncteurColax} dégage la notion dont le besoin se fait sentir.

\begin{df}\label{DefFoncteurColax}
Un \deux{}\emph{foncteur colax}\index{2foncteurcolax@\DeuxFoncteurCoLax} de $\mathdeuxcat{A}$ vers $\mathdeuxcat{B}$ est un \DeuxFoncteurLax{} de $\DeuxCatDeuxOp{\mathdeuxcat{A}}$ vers $\DeuxCatDeuxOp{\mathdeuxcat{B}}$. De façon plus explicite, la donnée d'un \DeuxFoncteurCoLax{} $u : \mathdeuxcat{A} \to \mathdeuxcat{B}$ correspond à celle d'un objet $u(a)$ de $\mathdeuxcat{B}$ pour tout objet $a$ de $\mathdeuxcat{A}$, d'une \un{}cellule $u(f) : u(a) \to u(a')$ de $\mathdeuxcat{B}$ pour toute \un{}cellule $f : a \to a'$ dans $\mathdeuxcat{A}$, d'une \deux{}cellule $u(\alpha) : u(f) \Rightarrow u(f')$ pour toute \deux{}cellule $\alpha : f \Rightarrow f'$ dans $\mathdeuxcat{A}$, d'une \deux{}cellule $u_{a} : u(1_{a}) \Rightarrow 1_{u(a)}$ pour tout objet $a$ de $\mathdeuxcat{A}$, d'une \deux{}cellule $u_{f',f} : u(f'f) \Rightarrow u(f') u(f)$ pour tout couple de \un{}cellules $f$ et $f'$ de $\mathdeuxcat{A}$ telles que la composée $f'f$ fasse sens, les conditions de cohérence suivantes étant vérifiées. 
\begin{itemize}
\item
Pour toute \un{}cellule $f$ de $\mathdeuxcat{A}$, 
$$
u(1_{f}) = 1_{u(f)}.
$$ 
\item
Pour tout couple de \deux{}cellules $\alpha$ et $\alpha'$ telles que la composée $\alpha' \CompDeuxUn \alpha$ fasse sens, 
$$
u(\alpha' \CompDeuxUn \alpha) = u(\alpha') \CompDeuxUn u(\alpha).
$$
\item
Pour tout triplet de \un{} cellules $f$, $f'$ et $f''$ de $\mathdeuxcat{A}$ telles que la composée $f'' f' f$ fasse sens, 
$$
(u(f'') \CompDeuxZero u_{f',f}) \CompDeuxUn u_{f'', f'f} = (u_{f'',f'} \CompDeuxZero u(f)) \CompDeuxUn u_{f'' f', f}.
$$
\item
Pour tout couple de \deux{}cellules $\alpha : f \Rightarrow g$ et $\alpha' : f' \Rightarrow g'$ de $\mathdeuxcat{A}$ telles que la composée $\alpha' \CompDeuxZero \alpha$ fasse sens,  
$$
u_{g',g} \CompDeuxUn u(\alpha' \CompDeuxZero \alpha) = (u(\alpha') \CompDeuxZero u(\alpha)) \CompDeuxUn u_{f',f}.
$$
\item
Pour toute \un{}cellule $f : a \to a'$ de $\mathdeuxcat{A}$, 
$$
(u(f) \CompDeuxZero u_{a}) \CompDeuxUn u_{f, 1_{a}} = 1_{u(f)}
$$
et
$$
(u_{a'} \CompDeuxZero u(f)) \CompDeuxUn u_{1_{a'}, f} = 1_{u(f)}.
$$
\end{itemize}
\end{df}

\begin{df}
Soient $\mathdeuxcat{A}$ et $\mathdeuxcat{B}$ des \deux{}catégories et $u$ un \DeuxFoncteurLax{} de $\mathdeuxcat{A}$ vers $\mathdeuxcat{B}$. On définit le \deux{}\emph{foncteur colax} \deux{}\emph{opposé à} $u$\index{2foncteurcolax2oppose@\DeuxFoncteurCoLax{} \deux{}opposé}, noté $\DeuxFoncDeuxOp{u}$\index[not]{uco@$\DeuxFoncDeuxOp{u}$}, de source $\DeuxCatDeuxOp{\mathdeuxcat{A}}$ et de but $\DeuxCatDeuxOp{\mathdeuxcat{B}}$, par les données suivantes. Pour tout objet $a$ de $\mathdeuxcat{A}$, $\DeuxFoncDeuxOp{u} (a) = u(a)$. Pour toute \un{}cellule $f$ de $\mathdeuxcat{A}$, $\DeuxFoncDeuxOp{u} (f) = u(f)$. Pour toute \deux{}cellule $\alpha$ de $\mathdeuxcat{A}$, $\DeuxFoncDeuxOp{u} (\alpha) = u(\alpha)$. Pour tout couple de \un{}cellules $f$ et $f'$ de $\DeuxCatDeuxOp{\mathdeuxcat{A}}$ telles que la composée $f'f$ fasse sens (c'est-à-dire telles que la composée $f'f$ fasse sens dans $\mathdeuxcat{A}$), $(\DeuxFoncDeuxOp{u})_{f',f} = u_{f',f}$. Pour tout objet $a$ de $\mathdeuxcat{A}$, $(\DeuxFoncDeuxOp{u})_{a} = u_{a}$.
\end{df}

\begin{df}
Pour tout \DeuxFoncteurLax{} (\emph{resp.} \DeuxFoncteurCoLax) $u$, on définit le \emph{\deux{}foncteur colax opposé à}\index{2foncteurcolaxoppose@\DeuxFoncteurCoLax{} opposé} $u$ (\emph{resp.} le \emph{\deux{}foncteur lax opposé à}\index{2foncteurlaxoppose@\DeuxFoncteurLax{} opposé} $u$) par 
$$
\DeuxFoncToutOp{u} = \DeuxFoncDeuxOp{(\DeuxFoncUnOp{u})} = \DeuxFoncUnOp{(\DeuxFoncDeuxOp{u})}\index[not]{ucoop@$\DeuxFoncToutOp{u}$}.
$$
\end{df} 

\begin{rem}
De même que l'on pourrait s'abstenir de parler de foncteurs contravariants entre catégories, l'on pourrait ne jamais mentionner le terme « colax » dans l'étude des morphismes entre \deux{}catégories. Nous avons malgré tout décidé de les mentionner pour rendre le propos plus clair. 
\end{rem}

\begin{df}\label{DefDeuxTrans}
Soient $\mathdeuxcat{A}$ et $\mathdeuxcat{B}$ deux \deux{}catégories et $u$ et $v$ deux \DeuxFoncteursLax{} de $\mathdeuxcat{A}$ vers $\mathdeuxcat{B}$. 

Une \emph{\DeuxTransformationLax{}} \index{transformation} $\sigma$ de $u$ vers $v$, notée $\sigma : u \Rightarrow v$\index[not]{0sigma@$\sigma : u \Rightarrow v$}, correspond aux données et conditions suivantes.
\begin{itemize}
\item Pour tout objet $a$ de $\mathdeuxcat{A}$, une \un{}cellule 
$$
\sigma_{a} : u(a) \to v(a)\index[not]{0sigmaa@$\sigma_{a}$}
$$
de $\mathdeuxcat{B}$.
\item Pour toute \un{}cellule $f : a \to a'$ de $\mathdeuxcat{A}$, une \deux{}cellule 
$$
\sigma_{f} : \sigma_{a'} u(f) \Rightarrow v(f) \sigma_{a}\index[not]{0sigmaf@$\sigma_{f}$}
$$
de $\mathdeuxcat{B}$.
\end{itemize}
On requiert de ces données qu'elles satisfassent les conditions de cohérence suivantes.
\begin{itemize}
\item
Pour tout couple de \un{}cellules $f$ et $g$ de $a$ vers $a'$ dans $\mathdeuxcat{A}$, pour toute \deux{}cellule $\alpha : f \Rightarrow g$ de $\mathdeuxcat{A}$, le diagramme
$$
\xymatrix{
\sigma_{a'} u(f)
\ar@{=>}[r]^{\sigma_{f}}
\ar@{=>}[d]_{\sigma_{a'} \CompDeuxZero u(\alpha)}
&
v(f) \sigma_{a}
\ar@{=>}[d]^{v(\alpha) \CompDeuxZero \sigma_{a}}
\\
\sigma_{a'} u(g)
\ar@{=>}[r]_{\sigma_{g}}
&
v(g) \sigma_{a}
}
$$
est commutatif.

\item 
Pour tout triplet d'objets $a$, $a'$ et $a''$ de $\mathdeuxcat{A}$, pour tout couple de \un{}cellules $f : a \to a'$ et $f' : a' \to a''$ de $\mathdeuxcat{A}$, le diagramme
$$
\xymatrix{
\sigma_{a''} u(f') u(f)
\ar @{=>} [rr]^{\sigma_{f'} \CompDeuxZero u(f)}
\ar @{=>} [dd]_{\sigma_{a''} \CompDeuxZero u_{f',f}}
&& v(f') \sigma_{a'} u(f) 
\ar @{=>} [rr]^{v(f') \CompDeuxZero \sigma_{f}}
&& v(f') v(f) \sigma_{a}
\ar @{=>} [dd]^{v_{f',f} \CompDeuxZero \sigma_{a}}
\\
\\
\sigma_{a''} u(f'f)
\ar @{=>} [rrrr]_{\sigma_{f'f}}
&&&& v(f'f) \sigma_{a}
}
$$
est commutatif.

\item Pour tout objet $a$ de $\mathdeuxcat{A}$, le diagramme

$$
\xymatrix{
\sigma_{a} 1_{u(a)}
\ar @{=} [r]
\ar @{=>} [d]_{\sigma_{a} \CompDeuxZero u_{a}}
& 1_{v(a)} \sigma_{a}
\ar @{=>} [d]^{v_{a} \CompDeuxZero \sigma_{a}}
\\
\sigma_{a} u(1_{a})
\ar @{=>} [r]_{\sigma_{1_{a}}}
& v(1_{a}) \sigma_{a}
}
$$
est commutatif. 
\end{itemize}

Une \emph{\DeuxTransformationCoLax{}}\index{optransformation} $\tau$ de $u$ vers $v$, notée $\tau : u \Rightarrow v$, correspond aux données et conditions suivantes.
\begin{itemize}
\item Pour tout objet $a$ de $\mathdeuxcat{A}$, une \un{}cellule 
$$
\tau_{a} : u(a) \to v(a)
$$
de $\mathdeuxcat{B}$.
\item Pour toute \un{}cellule $f : a \to a'$ de $\mathdeuxcat{A}$, une \deux{}cellule 
$$
\tau_{f} : v(f) \tau_{a} \Rightarrow \tau_{a'} u(f) 
$$
de $\mathdeuxcat{B}$.
\end{itemize} 
On requiert de ces données qu'elles satisfassent les conditions de cohérence suivantes.
\begin{itemize}
\item
Pour tout couple de \un{}cellules $f$ et $g$ de $a$ vers $a'$ dans $\mathdeuxcat{A}$, pour toute \deux{}cellule $\alpha : f \Rightarrow g$ de $\mathdeuxcat{A}$, le diagramme
$$
\xymatrix{
v(f) \tau_{a}
\ar@{=>}[r]^{\tau_{f}}
\ar@{=>}[d]_{v(\alpha) \CompDeuxZero \tau_{a}}
&
\tau_{a'} u(f)
\ar@{=>}[d]^{\tau_{a'} \CompDeuxZero u(\alpha)} 
\\
v(g) \tau_{a} 
\ar@{=>}[r]_{\tau_{g}}
&
\tau_{a'} u(g)
}
$$
est commutatif.

\item 
Pour tout triplet d'objets $a$, $a'$ et $a''$ de $\mathdeuxcat{A}$, pour tout couple de \un{}cellules $f : a \to a'$ et $f' : a' \to a''$ de $\mathdeuxcat{A}$, le diagramme
$$
\xymatrix{
v(f') v(f) \tau_{a}
\ar @{=>} [rr]^{v(f') \CompDeuxZero \tau_{f}}
\ar @{=>} [dd]_{\DeuxCellStructComp{v}{f'}{f} \CompDeuxZero \tau_{a}}
&& v(f') \tau_{a'} u(f) 
\ar @{=>} [rr]^{\tau_{f'} \CompDeuxZero u(f)}
&& \tau_{a''} u(f') u(f)
\ar @{=>} [dd]^{\tau_{a''} \CompDeuxZero \DeuxCellStructComp{u}{f'}{f}}
\\
\\
v(f'f) \tau_{a}
\ar @{=>} [rrrr]_{\tau_{f'f}}
&&&& \tau_{a''} u(f'f)
}
$$
est commutatif.

\item Pour tout objet $a$ de $\mathdeuxcat{A}$, le diagramme
$$
\xymatrix{
\tau_{a} 1_{u(a)}
\ar @{=} [r]
\ar @{=>} [d]_{\tau_{a} \CompDeuxZero \DeuxCellStructId{u}{a}}
& 
1_{v(a)} \tau_{a}
\ar @{=>} [d]^{\DeuxCellStructId{v}{a} \CompDeuxZero \tau_{a}}
\\
\tau_{a} u(1_{a})
& 
v(1_{a}) \tau_{a}
\ar @{=>} [l]^{\tau_{1_{a}}}
}
$$
est commutatif. 
\end{itemize}
\end{df}

\begin{rem}\label{RemDeuxTransFoncNor}
Soient $u$ et $v$ deux \DeuxFoncteursLax{} normalisés de source $\mathdeuxcat{A}$ et $\sigma : u \Rightarrow v$ une \DeuxTransformationLax{}. En vertu des axiomes, pour tout objet $a$ de $\mathdeuxcat{A}$, $\sigma_{1_{a}} = 1_{\sigma_{a}}$. 
\end{rem}

\begin{rem}\label{DualDeuxTrans}
Soient $\mathdeuxcat{A}$ et $\mathdeuxcat{B}$ deux \deux{}catégories et $u$ et $v$ deux \DeuxFoncteursLax{} de $\mathdeuxcat{A}$ vers $\mathdeuxcat{B}$. Alors, toute \DeuxTransformationLax{} $\sigma$ de $u$ vers $v$ induit une \DeuxTransformationCoLax{} $\DeuxTransUnOp{\sigma}$ de $\DeuxFoncUnOp{v}$ vers $\DeuxFoncUnOp{u}$, définie par $(\DeuxTransUnOp{\sigma})_{a} = \sigma_{a}$ pour tout objet $a$ de $\mathdeuxcat{A}$ et $(\DeuxTransUnOp{\sigma})_{f} = \sigma_{f}$ pour toute \un{}cellule $f$ de $\mathdeuxcat{A}$, et ce procédé permet d'identifier les \DeuxTransformationsLax{} de $u$ vers $v$ et les \DeuxTransformationsCoLax{} de $\DeuxFoncUnOp{v}$ vers $\DeuxFoncUnOp{u}$.
\end{rem}

\begin{df}\label{DefAbus}
Soient $u$ et $v$ deux \DeuxFoncteursCoLax{} de $\mathdeuxcat{A}$ vers $\mathdeuxcat{B}$. Une \DeuxTransformationLax{} de $u$ vers $v$ est une \DeuxTransformationCoLax{} de $\DeuxFoncDeuxOp{u}$ vers $\DeuxFoncDeuxOp{v}$. Une \DeuxTransformationCoLax{} de $u$ vers $v$ est une \DeuxTransformationLax{} de $\DeuxFoncDeuxOp{u}$ vers $\DeuxFoncDeuxOp{v}$. 
\end{df}

\begin{rem}
La définition \ref{DefAbus} est légèrement abusive ; il faudrait en toute rigueur énoncer des définitions analogues à celles figurant dans l'énoncé de la définition \ref{DefDeuxTrans} avant de procéder à des vérifications permettant de conclure par les identifications que l'on a prises comme définitions. Le lecteur pourra rectifier de lui-même s'il le souhaite. Il est de plus et bien entendu invité à expliciter la définition \ref{DefAbus}.
\end{rem}

\begin{df}
Soient $u, v : \mathdeuxcat{A} \to \mathdeuxcat{B}$ deux \deux{} foncteurs qui sont tous deux lax ou tous deux colax. On dira qu'une \DeuxTransformationLax{} (\emph{resp.} \DeuxTransformationCoLax{}) $\sigma : u \Rightarrow v$ est \emph{relative aux objets}\index{relative aux objets (\DeuxTransformationLax{} ou \DeuxTransformationCoLax{})}
si, pour tout objet $a$ de $\mathdeuxcat{A}$, on a l'égalité $\sigma_{a} = 1_{u(a)} = 1_{v(a)}$. 
\end{df}

\begin{rem}
Pour qu'existe une \DeuxTransformationLax{} relative aux objets ou une \DeuxTransformationCoLax{} relative aux objets entre deux \DeuxFoncteursLax{} ou \DeuxFoncteursCoLax{} parallèles, il faut donc qu'ils coïncident sur les objets.
\end{rem}

\begin{rem}
Étant donné deux \DeuxFoncteursLax{} (\emph{resp.} \DeuxFoncteursCoLax{}) $u$ et $v$, la donnée d'une \DeuxTransformationLax{} relative aux objets de $u$ vers $v$ équivaut à celle d'une \DeuxTransformationCoLax{} relative aux objets de $v$ vers $u$.
\end{rem}

\begin{df}
Étant donné des \deux{}catégories $\mathdeuxcat{A}$ et $\mathdeuxcat{B}$, des \DeuxFoncteursLax{} $u$ et $v$ de $\mathdeuxcat{A}$ vers $\mathdeuxcat{B}$ et une \DeuxTransformationLax{} $\sigma$ de $u$ vers $v$, on dira que $\sigma$ est une \emph{transformation stricte}\index{transformation stricte} si, pour toute \un{}cellule $f$ de $\mathdeuxcat{A}$, la \deux{}cellule $\sigma_{f}$ est une identité.
\end{df}

\begin{rem}
Une \DeuxTransformationStricte{} de $u$ vers $v$ peut donc être considérée comme une \DeuxTransformationLax{} de $u$ vers $v$ aussi bien que comme une \DeuxTransformationCoLax{} de $u$ vers $v$. 
\end{rem}

On va maintenant définir une \deux{}catégorie dont les objets sont les \deux{}catégories. Comme il ne semble pas possible d'obtenir quelque chose de pertinent (du moins dans le cadre de notre étude) en choisissant les \DeuxFoncteursLax{} comme \un{}cellules, on considère le seul cas des \DeuxFoncteursStricts{}, ce qui explique que nous ne nous plaçons pas dans le cas le plus général — que le lecteur pourra rétablir — dans les définitions suivantes. 

\begin{df}\label{IdentiteDeuxFoncteur}
Soit $u : \mathdeuxcat{A} \to \mathdeuxcat{B}$ un \DeuxFoncteurStrict{}. On définit une \DeuxTransformationStricte{} $1_{u} : u \Rightarrow u$ par
$$
(1_{u})_{a} = 1_{u(a)}
$$
pour tout objet $a$ de $\mathdeuxcat{A}$. 
\end{df}

\begin{df}\label{CompositionVerticaleTransformations}
Soient $\mathdeuxcat{A}$ et $\mathdeuxcat{B}$ deux \deux{}catégories, $u$, $v$ et $w$ trois \DeuxFoncteursStricts{} de $\mathdeuxcat{A}$ vers $\mathdeuxcat{B}$, $\sigma : u \Rightarrow v$ et $\tau : v \Rightarrow w$ deux \DeuxTransformationsStrictes{}. On définit la composée $\tau \CompDeuxUn \sigma$ par
$$
(\tau \CompDeuxUn \sigma)_{a} = \tau_{a} \sigma_{a}
$$
pour tout objet $a$ de $\mathdeuxcat{A}$. 
\end{df}

\begin{df}\label{CompositionHorizontaleTransformations}
Soient $\mathdeuxcat{A}$, $\mathdeuxcat{B}$ et $\mathdeuxcat{C}$ trois \deux{}catégories, $u$ et $v$ deux \DeuxFoncteursStricts{} de $\mathdeuxcat{A}$ vers $\mathdeuxcat{B}$, $u'$ et $v'$ deux \DeuxFoncteursStricts{} de $\mathdeuxcat{B}$ vers $\mathdeuxcat{C}$ et $\sigma : u \Rightarrow v$ et $\sigma' : u' \Rightarrow v'$ deux \DeuxTransformationsStrictes{}. En vertu de la naturalité de $\sigma'$, le diagramme
$$
\xymatrix{
u'(u(a))
\ar[r]^{u'(\sigma_{a})}
\ar[d]_{\sigma'_{u(a)}}
&
u'(v(a))
\ar[d]^{\sigma'_{v(a)}}
\\
v'(u(a))
\ar[r]_{v'(\sigma_{a})}
&
v'(v(a))
}
$$
est commutatif pour tout objet $a$ de $\mathdeuxcat{A}$. On définit la composée $\sigma' \CompDeuxZero \sigma$ par
$$
(\sigma' \CompDeuxZero \sigma)_{a} = \sigma'_{v(a)} u'(\sigma_{a}) = v'(\sigma_{a}) \sigma'_{u(a)}
$$
pour tout objet $a$ de $\mathdeuxcat{A}$. 
\end{df}

On vérifie que ces données permettent de définir une \deux{}catégorie dont les objets sont les \deux{}catégories, les \un{}cellules les \DeuxFoncteursStricts{} et les \deux{}cellules les \DeuxTransformationsStrictes{}. 

\begin{df}\label{DefDeuxCatDeuxCat}
On notera $\DeuxCatDeuxCat$\index[not]{2CatSouligne@$\DeuxCatDeuxCat$} la \deux{}catégorie dont les objets sont les \deux{}catégories, les \un{}cellules les \DeuxFoncteursStricts{} et les \deux{}cellules les \DeuxTransformationsStrictes{}. L'identité d'une \un{}cellule se définit comme dans la définition \ref{IdentiteDeuxFoncteur}, la composée verticale et la composée horizontale des \deux{}cellules étant données par les définitions \ref{CompositionVerticaleTransformations} et \ref{CompositionHorizontaleTransformations} respectivement. 
\end{df}

\begin{paragr}\label{DefFoncteursOpCoCoop}
L'assignation $\mathdeuxcat{A} \to \DeuxCatUnOp{\mathdeuxcat{A}}$ permet de définir un \DeuxFoncteurStrict{}
$$
\begin{aligned}
?^{op}\index[not]{?op@$?^{op}$} : \DeuxCatDeuxCat &\to\DeuxCatDeuxOp{\DeuxCatDeuxCat} 
\\
\mathdeuxcat{A} &\mapsto \DeuxCatUnOp{\mathdeuxcat{A}}
\\
(u : \mathdeuxcat{A} \to \mathdeuxcat{B}) &\mapsto (\DeuxFoncUnOp{u} : \DeuxCatUnOp{\mathdeuxcat{A}} \to \DeuxCatUnOp{\mathdeuxcat{B}})
\\
(\alpha : u \Rightarrow v) &\mapsto (\DeuxTransUnOp{\alpha} : \DeuxFoncUnOp{v} \Rightarrow \DeuxFoncUnOp{u}).
\end{aligned}
$$
\end{paragr}

\begin{paragr}\label{DefFoncteursOpCoCoop}
L'assignation $\mathdeuxcat{A} \to \DeuxCatDeuxOp{\mathdeuxcat{A}}$ permet de définir un \DeuxFoncteurStrict{}
$$
\begin{aligned}
?^{co}\index[not]{?co@$?^{co}$} : \DeuxCatDeuxCat &\to \DeuxCatDeuxCat 
\\
\mathdeuxcat{A} &\mapsto \DeuxCatDeuxOp{\mathdeuxcat{A}}
\\
(u : \mathdeuxcat{A} \to \mathdeuxcat{B}) &\mapsto (\DeuxFoncDeuxOp{u} : \DeuxCatDeuxOp{\mathdeuxcat{A}} \to \DeuxCatDeuxOp{\mathdeuxcat{B}})
\\
(\alpha : u \Rightarrow v) &\mapsto (\DeuxTransDeuxOp{\alpha} : \DeuxFoncDeuxOp{u} \Rightarrow \DeuxFoncDeuxOp{v}).
\end{aligned}
$$
\end{paragr}

\begin{paragr}\label{DefFoncteursOpCoCoop}
L'assignation $\mathdeuxcat{A} \to \DeuxCatToutOp{\mathdeuxcat{A}}$ permet de définir un \DeuxFoncteurStrict{}
$$
\begin{aligned}
?^{coop}\index[not]{?coop@$?^{coop}$} : \DeuxCatDeuxCat &\to\DeuxCatDeuxOp{\DeuxCatDeuxCat} 
\\
\mathdeuxcat{A} &\mapsto \DeuxCatToutOp{\mathdeuxcat{A}}
\\
(u : \mathdeuxcat{A} \to \mathdeuxcat{B}) &\mapsto (\DeuxFoncToutOp{u} : \DeuxCatToutOp{\mathdeuxcat{A}} \to \DeuxCatToutOp{\mathdeuxcat{B}})
\\
(\alpha : u \Rightarrow v) &\mapsto (\DeuxTransToutOp{\alpha} : \DeuxFoncToutOp{v} \Rightarrow \DeuxFoncToutOp{u}).
\end{aligned}
$$
\end{paragr}

\begin{df}\label{DefModification}
Soient $\mathdeuxcat{A}$ et $\mathdeuxcat{B}$ des \deux{}catégories, $u$ et $v$ des \DeuxFoncteursLax{} de $\mathdeuxcat{A}$ vers $\mathdeuxcat{B}$ et $\sigma$ et $\tau$ des \DeuxTransformationsLax{} de $u$ vers $v$. Une \emph{modification}\index{modification} $\Gamma : \sigma \Rrightarrow \tau$\index[not]{Gammasigmatau@$\Gamma : \sigma \Rrightarrow \tau$} de $\sigma$ vers $\tau$ correspond à la donnée, pour tout objet $a$ de $\mathdeuxcat{A}$, d'une \deux{}cellule $\Gamma_{a}\index[not]{Gammaa@$\Gamma_{a}$} : \sigma_{a} \Rightarrow \tau_{a}$ dans $\mathdeuxcat{B}$, telle que, pour toute \un{}cellule $f : a \to a'$ de $\mathdeuxcat{A}$, l'égalité
$$
\tau_{f} \CompDeuxUn (\Gamma_{a'} \CompDeuxZero u(f)) = (v(f) \CompDeuxZero \Gamma_{a}) \CompDeuxUn \sigma_{f}
$$
soit vérifiée.
\end{df}

\begin{rem}\label{DeuxCatTroisCat}
Il existe une \emph{\trois{}catégorie} dont les objets sont les \deux{}catégories, les \un{}cellules les \DeuxFoncteursStricts{} entre icelles, les \deux{}cellules les \DeuxTransformationsStrictes{} entre iceux et les \trois{}cellules les modifications entre icelles. Nous n'entrons pas dans les détails. Soulignons toutefois que la notion de modification peut bien sûr se définir entre \DeuxTransformationsCoLax{} et que, dans les deux cas, il est possible de considérer la source et le but de la source et du but d'une modification comme n'étant pas des \DeuxFoncteursLax{} mais des \DeuxFoncteursCoLax{}. Cela permet par exemple, étant donné des \deux{}catégories $\mathdeuxcat{A}$ et $\mathdeuxcat{B}$, de définir une \deux{}catégorie $Colax(\mathdeuxcat{A},\mathdeuxcat{B})$\index[not]{CZolaxAB@$Colax(\mathdeuxcat{A},\mathdeuxcat{B})$} dont les objets sont les \DeuxFoncteursCoLax{} de $\mathdeuxcat{A}$ vers $\mathdeuxcat{B}$, les \un{}cellules les \DeuxTransformationsCoLax{} entre iceux et les \deux{}cellules les modifications entre icelles. Nous la retrouverons plus tard.  
\end{rem}

\section{Transformations et homotopie}\label{SectionTransHomotopie}

Nous consacrons cette section à l'énoncé et à la démonstration détaillée d'un résultat é\-lé\-men\-taire mais fondamental. Il se trouve explicitement utilisé au cours de la démonstration de \cite[proposition 7.1, (ii)]{CCG} et nous ne prétendons faire preuve ici d'aucune originalité.

\begin{df}
Pour tout entier $n \geq 0$, nous noterons $n$\index[not]{[n]@$[n]$} la catégorie associée à l'ensemble $\{ 0, \dots, n \}$ ordonné par l'ordre naturel. 
\end{df}

\begin{lemme}\label{DeuxTransFoncLax}
Soient $u$ et $v$ deux \DeuxFoncteursLax{} de $\mathdeuxcat{A}$ vers $\mathdeuxcat{B}$. S'il existe une \DeuxTransformationLax{} (\emph{resp.} \DeuxTransformationCoLax{}) de $u$ vers $v$, alors il existe un \DeuxFoncteurLax{} $h : [1] \times \mathdeuxcat{A} \to \mathdeuxcat{B}$ tel que le diagramme\footnote{Dans lequel on commet l'abus d'identifier $\mathdeuxcat{A}$ à $\DeuxCatPonct{} \times \mathdeuxcat{A}$.} 
$$
\xymatrix{
&[1] \times \mathdeuxcat{A}
\ar[dd]^{h}
\\
\mathdeuxcat{A}
\ar[ur]^{0 \times 1_{\mathdeuxcat{A}}}
\ar[dr]_{u}
&&\mathdeuxcat{A}
\ar[ul]_{1 \times 1_{\mathdeuxcat{A}}}
\ar[dl]^{v}
\\
&\mathdeuxcat{B}
}
$$
soit commutatif. 
\end{lemme}

\begin{proof}
Supposons qu'il existe une \DeuxTransformationCoLax{} $\alpha : u \Rightarrow v$. Construisons un \deux{}fonc\-teur lax $h : [1] \times \mathdeuxcat{A} \to  \mathdeuxcat{B}$ rendant commutatif le diagramme de l'énoncé. 

Pour tout objet $a$ de $\mathdeuxcat{A}$, on pose $h(0,a) = u(a)$ et $h(1,a) = v(a)$. 

Pour toute \un{}cellule $f : a \to a'$ de $\mathdeuxcat{A}$, on pose $h(0 \to 0, f) = u(f)$, $h(1 \to 1, f) = v(f)$ et $h(0 \to 1, f) = \alpha_{a'} u(f)$.

Pour tout couple de \un{}cellules $f$ et $g$ de $a$ vers $a'$ dans $\mathdeuxcat{A}$, pour toute \deux{}cellule $\gamma : f \Rightarrow g$ de $\mathdeuxcat{A}$, on pose $h(1_{0 \to 0}, \gamma) = u(\gamma)$, $h(1_{1 \to 1}, \gamma) = v(\gamma)$ et $h(1_{0 \to 1}, \gamma) = \alpha_{a'} \CompDeuxZero u(\gamma)$. 

Pour tout objet $a$ de $\mathdeuxcat{A}$, on pose $\DeuxCellStructId{h}{(0,a)} = \DeuxCellStructId{u}{a}$ et $\DeuxCellStructId{h}{(1,a)} = \DeuxCellStructId{v}{a}$. 

Pour tout couple $(f : a \to a', f' : a' \to a'')$ de \un{}cellules de $\mathdeuxcat{A}$, on pose 
$$
\DeuxCellStructComp{h}{(0 \to 0, f')}{(0 \to 0, f)} = \DeuxCellStructComp{u}{f'}{f},
$$ 
$$
\DeuxCellStructComp{h}{(1 \to 1, f')}{(1 \to 1, f)} = \DeuxCellStructComp{v}{f'}{f},
$$
$$
\DeuxCellStructComp{h}{(1 \to 1, f')}{(0 \to 1, f)} = (\alpha_{a''} \CompDeuxZero \DeuxCellStructComp{u}{f'}{f}) \CompDeuxUn (\alpha_{f'} \CompDeuxZero u(f))
$$ 
et 
$$
\DeuxCellStructComp{h}{(0 \to 1, f')}{(0 \to 0, f)} = \alpha_{a''} \CompDeuxZero \DeuxCellStructComp{u}{f'}{f}.
$$ 

Ces données définissent un \DeuxFoncteurLax{} $h : [1] \times \mathdeuxcat{A} \to  \mathdeuxcat{B}$ rendant commutatif le diagramme de l'énoncé. Nous allons vérifier toutes les conditions de cohérence qui ne sont pas évidentes, c'est-à-dire celles qui font intervenir les deux copies $\{0\} \times \mathdeuxcat{A}$ et $\{1\} \times \mathdeuxcat{A}$ de $\mathdeuxcat{A}$ dans $[1] \times \mathdeuxcat{A}$. 

Pour toute \un{}cellule $f : a \to a'$ de $\mathdeuxcat{A}$,
$$
\begin{aligned}
h(1_{(0 \to 1, f)}) &= h(1_{0 \to 1}, 1_{f})
\\
&= \alpha_{a'} \CompDeuxZero u(1_{f})
\\
&= \alpha_{a'} \CompDeuxZero 1_{u(f)}
\\
&= 1_{\alpha_{a'} u(f)}
\\
&= 1_{h(0 \to 1, f)}.
\end{aligned}
$$

Pour tout triplet de \un{}cellules $f$, $f'$ et $f''$ de $a$ vers $a'$ dans $\mathdeuxcat{A}$, pour tout couple de \deux{}cellules $\gamma : f \Rightarrow f'$ et $\delta : f' \Rightarrow f''$ de $\mathdeuxcat{A}$,
$$
\begin{aligned}
h((1_{0 \to 1}, \delta) \CompDeuxUn (1_{0 \to 1}, \gamma)) &= h(1_{0 \to 1}, \delta \CompDeuxUn \gamma)
\\
&= \alpha_{a'} \CompDeuxZero u(\delta \CompDeuxUn \gamma)
\\
&= \alpha_{a'} \CompDeuxZero (u(\delta) \CompDeuxUn u(\gamma))
\\
&= (\alpha_{a'} \CompDeuxZero u(\delta)) \CompDeuxUn (\alpha_{a'} \CompDeuxZero u(\gamma))
\\
&= h(1_{0 \to 1}, \delta) \CompDeuxUn h(1_{0 \to 1}, \gamma).
\end{aligned}
$$

Soient $f, g : a \to a'$ et $f', g' : a' \to a''$ quatre \un{}cellules et $\gamma : f \Rightarrow g$ et $\gamma' : f' \Rightarrow g'$ deux \deux{}cellules de $\mathdeuxcat{A}$. On souhaite vérifier la commutativité du diagramme
$$
\xymatrix{
h(0 \to 1, f') h(0 \to 0, f)
\ar@{=>}[rrr]^{h(1_{0 \to 1}, \gamma') \CompDeuxZero h(1_{0 \to 0}, \gamma)}
\ar@{=>}[dd]_{\DeuxCellStructComp{h}{(0 \to 1, f')}{(0 \to 0, f)}}
&&& h(0 \to 1, g') h(0 \to 0, g)
\ar@{=>}[dd]^{\DeuxCellStructComp{h}{(0 \to 1, g')}{(0 \to 0, g)}}
\\
\\
h(0 \to 1, f'f)
\ar@{=>}[rrr]_{h(1_{0 \to 1}, \gamma' \CompDeuxZero \gamma)}
&&& h(0 \to 1, g'g) 
&.
}
$$
Ce diagramme se récrit
$$
\xymatrix{
\alpha_{a''} u(f') u(f)
\ar@{=>}[rrr]^{\alpha_{a''} \CompDeuxZero u(\gamma') \CompDeuxZero u(\gamma)} 
\ar@{=>}[dd]_{\alpha_{a''} \CompDeuxZero \DeuxCellStructComp{u}{f'}{f}}
&&& \alpha_{a''} u(g') u(g)
\ar@{=>}[dd]^{\alpha_{a''} \CompDeuxZero \DeuxCellStructComp{u}{g'}{g}}
\\
\\
\alpha_{a''} u(f'f)
\ar@{=>}[rrr]_{\alpha_{a''} \CompDeuxZero u(\gamma' \CompDeuxZero \gamma)}
&&& \alpha_{a''} u(g'g)
&.
}
$$
Il est bien commutatif, puisque ce n'est autre que la « post-composition par $\alpha_{a''}$ » d'un diagramme commutatif en vertu de la naturalité des \deux{}cellules structurales de composition de $u$.

Considérons le diagramme
$$
\xymatrix{
h(1 \to 1, f') h(0 \to 1, f)
\ar@{=>}[rrr]^{h(1_{1 \to 1}, \gamma') \CompDeuxZero h(1_{0 \to 1}, \gamma)}
\ar@{=>}[dd]_{\DeuxCellStructComp{h}{(1 \to 1, f')}{(0 \to 1, f)}}
&&& h(1 \to 1, g') h(0 \to 1, g)
\ar@{=>}[dd]^{\DeuxCellStructComp{h}{(1 \to 1, g')}{(0 \to 1, g)}}
\\
\\
h(0 \to 1, f'f)
\ar@{=>}[rrr]_{h(1_{0 \to 1}, \gamma' \CompDeuxZero \gamma)}
&&& h(0 \to 1, g'g) 
&.
}
$$
Il se récrit
$$
\xymatrix{
v(f') \alpha_{a'} u(f)
\ar@{=>}[rrr]^{v(\gamma') \CompDeuxZero \alpha_{a'} \CompDeuxZero u(\gamma)}
\ar@{=>}[dd]_{(\alpha_{a''} \CompDeuxZero \DeuxCellStructComp{u}{f'}{f}) \CompDeuxUn (\alpha_{f'} \CompDeuxZero u(f))}
&&& v(g') \alpha_{a'} u(g)
\ar@{=>}[dd]^{(\alpha_{a''} \CompDeuxZero \DeuxCellStructComp{u}{g'}{g}) \CompDeuxUn (\alpha_{g'} \CompDeuxZero u(g))}
\\
\\
\alpha_{a''} u(f'f)
\ar@{=>}[rrr]_{\alpha_{a''} \CompDeuxZero u(\gamma' \CompDeuxZero \gamma)}
&&& \alpha_{a''} u(g'g) 
&.
}
$$

On a alors
$$
\begin{aligned}
(\alpha_{a''} \CompDeuxZero \DeuxCellStructComp{u}{g'}{g}) \CompDeuxUn (\alpha_{g'} \CompDeuxZero u(g)) \CompDeuxUn (v(\gamma') \CompDeuxZero \alpha_{a'} \CompDeuxZero u(\gamma))
&= (\alpha_{a''} \CompDeuxZero \DeuxCellStructComp{u}{g'}{g}) \CompDeuxUn ((\alpha_{g'} \CompDeuxUn (v(\gamma') \CompDeuxZero \alpha_{a'})) \CompDeuxZero u(\gamma))
\\
&\phantom{=1} \text{(loi d'échange)}
\\
&= (\alpha_{a''} \CompDeuxZero \DeuxCellStructComp{u}{g'}{g}) \CompDeuxUn (((\alpha_{a''} \CompDeuxZero u(\gamma')) \CompDeuxUn \alpha_{f'}) \CompDeuxZero u(\gamma))
\\
&\phantom{=1} \text{(car $\alpha$ est une \DeuxTransformationCoLax{})}
\\
&= (\alpha_{a''} \CompDeuxZero (\DeuxCellStructComp{u}{g'}{g} \CompDeuxUn (u(\gamma') \CompDeuxZero u(\gamma)))) \CompDeuxUn (\alpha_{f'} \CompDeuxZero u(f))
\\
&\phantom{=1} \text{(loi d'échange).}
\end{aligned}
$$

D'autre part,
$$
\begin{aligned}
(\alpha_{a''} \CompDeuxZero u(\gamma' \CompDeuxZero \gamma)) \CompDeuxUn (\alpha_{a''} \CompDeuxZero \DeuxCellStructComp{u}{f'}{f}) \CompDeuxUn (\alpha_{f'} \CompDeuxZero u(f)) &= (\alpha_{a''} \CompDeuxZero (u(\gamma' \CompDeuxZero \gamma) \CompDeuxUn \DeuxCellStructComp{u}{f'}{f})) \CompDeuxUn (\alpha_{f'} \CompDeuxZero u(f))
\\
&\phantom{=1} \text{(loi d'échange)}
\\
&= (\alpha_{a''} \CompDeuxZero (\DeuxCellStructComp{u}{g'}{g} \CompDeuxUn (u(\gamma') \CompDeuxZero u(\gamma)))) \CompDeuxUn (\alpha_{f'} \CompDeuxZero u(f))
\\
&\phantom{=1} \text{(naturalité des \deux{}cellules de composition de $u$).}
\end{aligned}
$$

Le diagramme considéré est donc bien commutatif.

Considérons $f : a \to a'$, $f' : a' \to a''$ et $f'' : a'' \to a'''$ trois \un{}cellules de $\mathdeuxcat{A}$. Considérons le diagramme
$$
\xymatrix{
h(0 \to 1, f'') h(0 \to 0, f') h(0 \to 0, f)
\ar@{=>}[rrr]^{\DeuxCellStructComp{h}{(0 \to 1, f'')}{(0 \to 0, f')} \CompDeuxZero h(0 \to 0, f)}
\ar@{=>}[dd]_{h(0 \to 1, f'') \CompDeuxZero \DeuxCellStructComp{h}{(0 \to 0, f')}{(0 \to 0, f)}}
&&& h(0 \to 1, f''f') h(0 \to 0, f)
\ar@{=>}[dd]^{\DeuxCellStructComp{h}{(0 \to 1, f''f')}{(0 \to 0, f)}}
\\
\\
h(0 \to 1, f'') h(0 \to 0, f'f)
\ar@{=>}[rrr]_{\DeuxCellStructComp{h}{(0 \to 1, f'')}{(0 \to 0, f'f)}}
&&& h(0 \to 1, f''f'f) 
&.
}
$$
Cela se récrit 
$$
\xymatrix{
\alpha_{a'''} u(f'') u(f') u(f)
\ar@{=>}[rrr]^{\alpha_{a'''} \CompDeuxZero \DeuxCellStructComp{u}{f''}{f'} \CompDeuxZero u(f)}
\ar@{=>}[dd]_{\alpha_{a'''} \CompDeuxZero u(f'') \CompDeuxZero \DeuxCellStructComp{u}{f'}{f}}
&&& \alpha_{a'''} u(f''f') u(f)
\ar@{=>}[dd]^{\alpha_{a'''} \CompDeuxZero \DeuxCellStructComp{u}{f''f'}{f}}
\\
\\
\alpha_{a'''} u(f'') u(f'f)
\ar@{=>}[rrr]_{\alpha_{a'''} \CompDeuxZero \DeuxCellStructComp{u}{f''}{f'f}}
&&& \alpha_{a'''} u(f''f'f)
&.
}
$$
Ce diagramme est bien commutatif, puisqu'il s'obtient par « post-composition par $\alpha_{a'''}$ » d'un diagramme qui est commutatif en vertu de la condition de cocycle pour $u$. 

Considérons le diagramme
$$
\xymatrix{
h(1 \to 1, f'') h(0 \to 1, f') h(0 \to 0, f)
\ar@{=>}[rrr]^{\DeuxCellStructComp{h}{(1 \to 1, f'')}{(0 \to 1, f')} \CompDeuxZero h(0 \to 0, f)}
\ar@{=>}[dd]_{h(1 \to 1, f'') \CompDeuxZero \DeuxCellStructComp{h}{(0 \to 1, f')}{(0 \to 0, f)}}
&&& h(0 \to 1, f''f') h(0 \to 0, f)
\ar@{=>}[dd]^{\DeuxCellStructComp{h}{(0 \to 1, f''f')}{(0 \to 0, f)}}
\\
\\
h(1 \to 1, f'') h(0 \to 1, f'f)
\ar@{=>}[rrr]_{\DeuxCellStructComp{h}{(1 \to 1, f'')}{(0 \to 1, f'f)}}
&&& h(0 \to 1, f''f'f) 
&.
}
$$
Cela se récrit
$$
\xymatrix{
v(f'') \alpha_{a''} u(f') u(f)
\ar@{=>}[rrrrr]^{((\alpha_{a'''} \CompDeuxZero \DeuxCellStructComp{u}{f''}{f'}) \CompDeuxUn (\alpha_{f''} \CompDeuxZero u(f'))) \CompDeuxZero u(f)}
\ar@{=>}[dd]_{v(f'') \alpha_{a''} \CompDeuxZero \DeuxCellStructComp{u}{f'}{f}}
&&&&& \alpha_{a'''} u(f''f') u(f)
\ar@{=>}[dd]^{\alpha_{a'''} \CompDeuxZero \DeuxCellStructComp{u}{f''f'}{f}}
\\
\\
v(f'') \alpha_{a''} u(f'f)
\ar@{=>}[rrrrr]_{((\alpha_{a'''} \CompDeuxZero \DeuxCellStructComp{u}{f''}{f'f}) \CompDeuxUn (\alpha_{f''} \CompDeuxZero u(f'f)))}
&&&&& \alpha_{a'''} u(f''f'f) 
&.
}
$$
La loi d'échange permet d'écrire les égalités
$$
\begin{aligned}
(\alpha_{a'''} \CompDeuxZero \DeuxCellStructComp{u}{f''f'}{f}) \CompDeuxUn (((\alpha_{a'''} \CompDeuxZero \DeuxCellStructComp{u}{f''}{f'}) \CompDeuxUn (\alpha_{f''} \CompDeuxZero u(f'))) \CompDeuxZero u(f)) = &(\alpha_{a'''} \CompDeuxZero \DeuxCellStructComp{u}{f''f'}{f}) 
\\
&\CompDeuxUn (\alpha_{a'''} \CompDeuxZero \DeuxCellStructComp{u}{f''}{f'} \CompDeuxZero u(f)) 
\\
&\CompDeuxUn (\alpha_{f''} \CompDeuxZero u(f') u(f))
\end{aligned}
$$
et
$$
\begin{aligned}
((\alpha_{a'''} \CompDeuxZero \DeuxCellStructComp{u}{f''}{f'f}) \CompDeuxUn (\alpha_{f''} \CompDeuxZero u(f'f))) \CompDeuxUn (v(f'') \alpha_{a''} \CompDeuxZero \DeuxCellStructComp{u}{f'}{f}) = &(\alpha_{a'''} \CompDeuxZero \DeuxCellStructComp{u}{f''}{f'f})
\\
&\CompDeuxUn (\alpha_{a'''} u(f'') \CompDeuxZero \DeuxCellStructComp{u}{f'}{f})
\\
&\CompDeuxUn (\alpha_{f''} \CompDeuxZero u(f') u(f)).
\end{aligned}
$$
La commutativité du rectangle ci-dessus résulte donc de l'égalité
$$
\DeuxCellStructComp{u}{f''}{f'f} \CompDeuxUn (u(f'') \CompDeuxZero \DeuxCellStructComp{u}{f'}{f}) = \DeuxCellStructComp{u}{f''f'}{f} \CompDeuxUn (\DeuxCellStructComp{u}{f''}{f'} \CompDeuxZero u(f)),
$$
conséquence de la condition de cocycle pour $u$.

Considérons le diagramme
$$
\xymatrix{
h(1 \to 1, f'') h(1 \to 1, f') h(0 \to 1, f)
\ar@{=>}[rrr]^{\DeuxCellStructComp{h}{(1 \to 1, f'')}{(1 \to 1, f')} \CompDeuxZero h(0 \to 1, f)}
\ar@{=>}[dd]_{h(1 \to 1, f'') \CompDeuxZero \DeuxCellStructComp{h}{(1 \to 1, f')}{(0 \to 1, f)}}
&&& h(1 \to 1, f''f') h(0 \to 1, f)
\ar@{=>}[dd]^{\DeuxCellStructComp{h}{(1 \to 1, f''f')}{(0 \to 1, f)}}
\\
\\
h(1 \to 1, f'') h(0 \to 1, f'f)
\ar@{=>}[rrr]_{\DeuxCellStructComp{h}{(1 \to 1, f'')}{(0 \to 1, f'f)}}
&&& h(0 \to 1, f''f'f) 
&.
}
$$
Cela se récrit
$$
\xymatrix{
v(f'') v(f') \alpha_{a'} u(f)
\ar@{=>}[rrrrr]^{\DeuxCellStructComp{v}{f''}{f'} \CompDeuxZero \alpha_{a'} u(f)}
\ar@{=>}[dd]_{v(f'') \CompDeuxZero ((\alpha_{a''} \CompDeuxZero \DeuxCellStructComp{u}{f'}{f}) \CompDeuxUn (\alpha_{f'} \CompDeuxZero u(f)))}
&&&&& v(f''f') \alpha_{a'} u(f)
\ar@{=>}[dd]^{(\alpha_{a'''} \CompDeuxZero \DeuxCellStructComp{u}{f''f'}{f}) \CompDeuxUn (\alpha_{f''f'} \CompDeuxZero u(f))}
\\
\\
v(f'') \alpha_{a''} u(f'f)
\ar@{=>}[rrrrr]_{(\alpha_{a'''} \CompDeuxZero \DeuxCellStructComp{u}{f''}{f'f}) \CompDeuxUn (\alpha_{f''} \CompDeuxZero u(f'f))}
&&&&& \alpha_{a'''} u(f''f'f) 
&.
}
$$
La loi d'échange permet d'écrire 
$$
(\alpha_{a'''} \CompDeuxZero \DeuxCellStructComp{u}{f''f'}{f}) \CompDeuxUn (\alpha_{f''f'} \CompDeuxZero u(f)) \CompDeuxUn (\DeuxCellStructComp{v}{f''}{f'} \CompDeuxZero \alpha_{a'} u(f)) = (\alpha_{a'''} \CompDeuxZero \DeuxCellStructComp{u}{f''f'}{f}) \CompDeuxUn (\alpha_{f''f'} \CompDeuxZero u(f)) \CompDeuxUn (\DeuxCellStructComp{v}{f''}{f'} \CompDeuxZero \alpha_{a'} u(f)).
$$
D'autre part, on a les égalités suivantes :
$$
\begin{aligned}
&((\alpha_{a'''} \CompDeuxZero \DeuxCellStructComp{u}{f''}{f'f}) \CompDeuxUn (\alpha_{f''} \CompDeuxZero u(f'f))) \CompDeuxUn (v(f'') \CompDeuxZero ((\alpha_{a''} \CompDeuxZero \DeuxCellStructComp{u}{f'}{f}) \CompDeuxUn (\alpha_{f'} \CompDeuxZero u(f)))) 
\\
&= (\alpha_{a'''} \CompDeuxZero \DeuxCellStructComp{u}{f''}{f'f}) 
\CompDeuxUn (\alpha_{a'''} \CompDeuxZero u(f'') \CompDeuxZero \DeuxCellStructComp{u}{f'}{f}) 
\CompDeuxUn (\alpha_{f''} \CompDeuxZero u(f') u(f))
\CompDeuxUn (v(f'') \CompDeuxZero \alpha_{f'} \CompDeuxZero u(f))
\\
&\phantom{=1} \text{(loi d'échange)}
\\
&= (\alpha_{a'''} \CompDeuxZero \DeuxCellStructComp{u}{f''f'}{f}) 
\CompDeuxUn (\alpha_{a'''} \CompDeuxZero \DeuxCellStructComp{u}{f''}{f'} \CompDeuxZero u(f))
\CompDeuxUn (\alpha_{f''} \CompDeuxZero u(f') u(f))
\CompDeuxUn (v(f'') \CompDeuxZero \alpha_{f'} \CompDeuxZero u(f))
\\
&\phantom{=1} \text{(condition de cocycle pour $u$)}
\\
&= (\alpha_{a'''} \CompDeuxZero \DeuxCellStructComp{u}{f''f'}{f}) \CompDeuxUn (\alpha_{f''f'} \CompDeuxZero u(f)) \CompDeuxUn (\DeuxCellStructComp{v}{f''}{f'} \CompDeuxZero \alpha_{a'} u(f))
\\
&\phantom{=1} \text{(car $\alpha$ est une \DeuxTransformationCoLax{}).}
\end{aligned}
$$

Le rectangle ci-dessus est donc bien commutatif.

Soit $f : a \to a'$ une \un{}cellule de $\mathdeuxcat{A}$. Considérons le diagramme
$$
\xymatrix{
h(0 \to 1, f) 1_{h(0,a)}
\ar@{=>}[rrr]^{h(0 \to 1, f) \CompDeuxZero \DeuxCellStructId{h}{(0,a)}}
\ar@{=}[d]
&&&h(0 \to 1, f) h(0 \to 0, 1_{a})
\ar@{=>}[d]^{\DeuxCellStructComp{h}{(0 \to 1, f)}{(0 \to 0, 1_{a})}}
\\
h(0 \to 1, f)
\ar@{=}[rrr]
&&&h((0 \to 1, f) (0 \to 0, 1_{a})) 
&.
}
$$
Cela se récrit 
$$
\xymatrix{
\alpha_{a'} u(f) 1_{u(a)}
\ar@{=>}[rrr]^{\alpha_{a'} u(f) \CompDeuxZero \DeuxCellStructId{u}{a}}
\ar@{=}[d]
&&&\alpha_{a'} u(f) u(1_{a})
\ar@{=>}[d]^{\alpha_{a'} \CompDeuxZero \DeuxCellStructComp{u}{f}{1_{a}}}
\\
\alpha_{a'} u(f)
\ar@{=}[rrr]
&&&\alpha_{a'} u(f 1_{a}) 
&.
}
$$
Ce diagramme est commutatif, puisqu'il s'obtient par « post-composition par $\alpha_{a'}$ » d'un diagramme commutatif du fait des contraintes d'unité pour $u$.

Considérons le diagramme
$$
\xymatrix{
1_{h(1,a')} h(0 \to 1, f)
\ar@{=>}[rrr]^{\DeuxCellStructId{h}{(1,a')} \CompDeuxZero h(0 \to 1, f)}
\ar@{=}[d]
&&&h(1 \to 1, 1_{a'}) h(0 \to 1, f)
\ar@{=>}[d]^{\DeuxCellStructComp{h}{(1 \to 1, 1_{a'})}{(0 \to 1, f)}}
\\
h(0 \to 1, f)
\ar@{=}[rrr]
&&&h((1 \to 1, 1_{a'}) (0 \to 1, f)) 
&.
}
$$
Cela se récrit
$$
\xymatrix{
1_{v(a')} \alpha_{a'} u(f)
\ar@{=>}[rrr]^{\DeuxCellStructId{v}{a'} \CompDeuxZero \alpha_{a'} u(f)}
\ar@{=}[d]
&&&v(1_{a'}) \alpha_{a'} u(f)
\ar@{=>}[d]^{(\alpha_{a'} \CompDeuxZero \DeuxCellStructComp{u}{1_{a'}}{f}) \CompDeuxUn (\alpha_{1_{a'}} \CompDeuxZero u(f))}
\\
\alpha_{a'} u(f)
\ar@{=}[rrr]
&&&\alpha_{a'} u(f) 
&.
}
$$
Les axiomes assurent l'égalité $\alpha_{1_{a'}} \CompDeuxUn (\DeuxCellStructId{v}{a'} \CompDeuxZero \alpha_{a'}) = \alpha_{a'} \CompDeuxZero \DeuxCellStructId{u}{a'}$. La loi d'échange, cette égalité et l'axiomatique des \DeuxFoncteursLax{} permettent alors d'écrire
$$
\begin{aligned}
((\alpha_{a'} \CompDeuxZero \DeuxCellStructComp{u}{1_{a'}}{f}) \CompDeuxUn (\alpha_{1_{a'}} \CompDeuxZero u(f))) \CompDeuxUn (\DeuxCellStructId{v}{a'} \CompDeuxZero \alpha_{a'} \CompDeuxZero u(f)) &= (\alpha_{a'} \CompDeuxZero \DeuxCellStructComp{u}{f}{1_{a'}}) \CompDeuxUn ((\alpha_{1_{a'}} \CompDeuxUn (\DeuxCellStructId{v}{a'} \CompDeuxZero \alpha_{a'})) \CompDeuxZero u(f))
\\
&= (\alpha_{a'} \CompDeuxZero \DeuxCellStructComp{u}{f}{1_{a'}}) \CompDeuxUn (\alpha_{a'} \CompDeuxZero \DeuxCellStructId{u}{a'} \CompDeuxZero u(f))
\\
&= \alpha_{a'} \CompDeuxZero (\DeuxCellStructComp{u}{f}{1_{a'}} \CompDeuxUn (\DeuxCellStructId{u}{a'} \CompDeuxZero u(f)))
\\
&= \alpha_{a'} \CompDeuxZero 1_{u(f)}
\\
&= 1_{\alpha_{a'} \CompDeuxZero u(f)}.
\end{aligned}
$$
Le rectangle ci-dessus est donc bien commutatif. Les autres conditions de cohérence résultent immédiatement du fait que les restrictions de $h$ à $\{0\} \times \mathdeuxcat{A}$ et $\{1\} \times \mathdeuxcat{A}$ sont données par les \DeuxFoncteursLax{} $u$ et $v$ respectivement. On a donc bien défini un \DeuxFoncteurLax{} $h : [1] \times \mathdeuxcat{A} \to \mathdeuxcat{B}$ rendant commutatif le diagramme figurant dans l'énoncé. 

Le cas d'une \DeuxTransformationLax{} de $u$ vers $v$ est conséquence de ce qui précède en vertu d'un argument de dualité. En effet, la donnée d'une telle \DeuxTransformationLax{} correspond à la donnée d'une \DeuxTransformationCoLax{} de $\DeuxFoncUnOp{v}$ vers $\DeuxFoncUnOp{u}$ (voir la remarque \ref{DualDeuxTrans}). Il existe donc un \DeuxFoncteurLax{} $h : [1] \times \DeuxCatUnOp{\mathdeuxcat{A}} \to \DeuxCatUnOp{\mathdeuxcat{B}}$ rendant le diagramme
$$
\xymatrix{
&[1] \times \DeuxCatUnOp{\mathdeuxcat{A}}
\ar[dd]^{h}
\\
\DeuxCatUnOp{\mathdeuxcat{A}}
\ar[ur]^{(\{0\}, 1_{\DeuxCatUnOp{\mathdeuxcat{A}}})}
\ar[dr]_{\DeuxFoncUnOp{v}}
&&\DeuxCatUnOp{\mathdeuxcat{A}}
\ar[ul]_{(\{1\}, 1_{\DeuxCatUnOp{\mathdeuxcat{A}}})}
\ar[dl]^{\DeuxFoncUnOp{u}}
\\
&\DeuxCatUnOp{\mathdeuxcat{B}}
}
$$
commutatif. On en déduit la commutativité du diagramme
$$
\xymatrix{
&\DeuxCatUnOp{[1]} \times \mathdeuxcat{A}
\ar[dd]^{\DeuxFoncUnOp{h}}
\\
\mathdeuxcat{A}
\ar[ur]^{(\{0\}, 1_{\mathdeuxcat{A}})}
\ar[dr]_{v}
&&\mathdeuxcat{A}
\ar[ul]_{(\{1\}, 1_{\mathdeuxcat{A}})}
\ar[dl]^{u}
\\
&\mathdeuxcat{B}
&.
}
$$
On conclut en remarquant que l'application qui envoie $0$ sur $1$ et $1$ sur $0$ induit un isomorphisme de $\DeuxCatUnOp{[1]}$ vers $[1]$.

Plus explicitement, le \DeuxFoncteurLax{} $h : [1] \times \mathdeuxcat{A} \to \mathdeuxcat{B}$ associé à une \DeuxTransformationLax{} $\alpha : u \Rightarrow v$ se caractérise comme suit. La restriction de $h$ à $\{ 0 \} \times \mathdeuxcat{A}$ (\emph{resp.} $\{ 0 \} \times \mathdeuxcat{B}$) est donnée par $u$ (\emph{resp.} $v$). Pour toute \un{}cellule $f : a \to a'$ de $\mathdeuxcat{A}$, 
$$
h (0 \to 1, f) = v(f) \alpha_{a}.
$$ 
Pour toute \deux{}cellule $\gamma$ de $\mathdeuxcat{A}$ dont la source de la source (et donc la source du but) est $a$, 
$$
h(1_{0 \to 1}, \gamma) = v(\gamma) \CompDeuxZero \alpha_{a}.
$$
Pour tout couple de \un{}cellules composables $f : a \to a'$ et $f' : a' \to a''$ dans $\mathdeuxcat{A}$,
$$
h_{(1 \to 1, f'), (0 \to 1, f)} = v_{f',f} \alpha_{a}
$$
et
$$
h_{(0 \to 1, f'), (0 \to 0, f)} = (v_{f',f} \CompDeuxZero \alpha_{a}) \CompDeuxUn (v(f') \CompDeuxZero \alpha_{f}).
$$
\end{proof}

\begin{rem}
On notera que, même si les morphismes $u$ et $v$ de l'énoncé sont stricts, les \DeuxFoncteursLax{} construits selon la méthode exposée au cours de la démonstration ne sont pas stricts en général. En revanche, ils sont normalisés si $u$ et $v$ le sont.
\end{rem}

\begin{rem}\label{TerminologieTrans}
C'est ce résultat qui nous a incité à renoncer aux termes « transformations lax » et « transformations colax » (ou « oplax »), dont l'usage semble toutefois prévaloir, terminologie qui nous semble devoir faire croire — sans que cela soit très clair\footnote{Les propriétés de fonctorialité de l'intégration, que nous mentionnerons plus loin — voir la remarque \ref{FonctorialiteIntegration} —, pourraient toutefois justifier en partie le choix des termes « transformations lax » et « transformations colax ».} — que les « transformations lax » ont plus à voir avec les \DeuxFoncteursLax{} que les « transformations colax », et réciproquement quant à leurs relations aux \DeuxFoncteursCoLax{}, sans compter que les usagers de ce vocabulaire n'ont toujours pas su se mettre d'accord quant à ce qu'il convenait d'appeler « transformation lax » et « transformation oplax » respectivement. Nous tenterons de justifier notre choix en attirant l'attention sur le fait que les \un{}cellules \emph{et} les \deux{}cellules de ce que nous appelons « transformation de $u$ vers $v$ » vont \emph{des} expressions contenant les images par $u$ \emph{vers} les expressions contenant les images par $v$.  
\end{rem}

\section{Deux adjonctions entre $\Cat$ et $\DeuxCat$}\label{SectionAdjonctionsCatDeuxCat}
La catégorie $\Cat$\index[not]{CYat@$\Cat$} des petites catégories s'envoie de façon pleinement fidèle dans la catégorie $\DeuxCat$ par l'inclusion consistant à voir une catégorie comme une \deux{}catégorie dont toutes les \deux{}cellules sont des identités et un foncteur comme un \deux{}foncteur strict dont la source et le but sont les \deux{}catégories obtenues de cette façon à partir de la source et du but du foncteur de départ. On exhibe ici un adjoint à gauche $\TronqueInt\index[not]{0taui@$\TronqueInt$} : \DeuxCat \to \Cat$ et un adjoint à droite $\TronqueBete\index[not]{0taub@$\TronqueBete$} : \DeuxCat \to \Cat$ de cette inclusion $\Cat \hookrightarrow \DeuxCat$.

Pour toute \deux{}catégorie $\mathdeuxcat{A}$, on pose 
$$
\Objets{\TronqueInt{\mathdeuxcat{A}}} = \Objets{\mathdeuxcat{A}}
$$ 
et, pour tout couple d'objets $a$ et $a'$ de $\mathdeuxcat{A}$, 
$$
\CatHom{\TronqueInt{\mathdeuxcat{A}}}{a}{a'} = \pi_{0}(\CatHom{\mathdeuxcat{A}}{a}{a'}).
$$
Autrement dit, on passe de $\mathdeuxcat{A}$ à $\TronqueInt{\mathdeuxcat{A}}$ en identifiant les \un{}cellules reliées par un zigzag de \deux{}cellules dans $\mathdeuxcat{A}$.

Pour toute \deux{}catégorie $\mathdeuxcat{A}$, on pose 
$$
\Objets{\TronqueBete{\mathdeuxcat{A}}} = \Objets{\mathdeuxcat{A}}
$$ 
et, pour tout couple d'objets $a$ et $a'$ de $\mathdeuxcat{A}$, 
$$
\CatHom{\TronqueBete{\mathdeuxcat{A}}}{a}{a'} = \EnsHom{\mathdeuxcat{A}}{a}{a'}\index[not]{HomAaa'@$\EnsHom{\mathdeuxcat{A}}{a}{a'}$},
$$
en notant $\EnsHom{\mathdeuxcat{A}}{a}{a'}$ l'ensemble sous-jacent à la catégorie $\CatHom{\mathdeuxcat{A}}{a}{a'}$. (On commet l'abus de considérer tout ensemble comme une catégorie discrète.) Autrement dit, on passe de $\mathdeuxcat{A}$ à $\TronqueBete{\mathdeuxcat{A}}$ en « oubliant » toutes les \deux{}cellules. 

Pour toute catégorie $A$, on a clairement $\TronqueBete{A} = \TronqueInt{A} = A$ et, pour toute \deux{}catégorie $\mathdeuxcat{A}$, il existe des \DeuxFoncteursStricts{} canoniques $\mathdeuxcat{A} \to \TronqueInt{\mathdeuxcat{A}}$ et $\TronqueBete{\mathdeuxcat{A}} \to \mathdeuxcat{A}$. On vérifie que ces données ainsi que l'identité de $\Cat$ définissent les unités et coünités d'adjonctions $(\TronqueInt, \Cat \hookrightarrow \DeuxCat)$ et $(\Cat \hookrightarrow \DeuxCat, \TronqueBete)$ dont l'inclusion $\Cat \hookrightarrow \DeuxCat$ constitue l'adjoint à droite et l'adjoint à gauche respectivement. En particulier (voir par exemple \cite[p. 83, théorème 2]{CWM}), si $\mathdeuxcat{A}$ est une \deux{}catégorie et $B$ est une catégorie, alors, pour tout \deux{}foncteur (nécessairement strict) $u : \mathdeuxcat{A} \to B$, il existe un diagramme commutatif
$$
\xymatrix{
\mathdeuxcat{A}
\ar[r]
\ar[dr]_{u}
&\TronqueInt{\mathdeuxcat{A}}
\ar[d]
\\
&B
}
$$
dans $\DeuxCat{}$ et, pour tout \DeuxFoncteurStrict{} $v : B \to \mathdeuxcat{A}$, il existe un diagramme commutatif
$$
\xymatrix{
B
\ar[r]
\ar[dr]_{v}
&\TronqueBete{\mathdeuxcat{A}}
\ar[d]
\\
&\mathdeuxcat{A}
}
$$
dans $\DeuxCat$.

\section{2-catégories tranches}\label{SectionDeuxCategoriesTranches}

\begin{paragr}\label{TranchesGray}
Dans \cite[I, 2, 5, p. 29]{Gray}, Gray définit, sous la donnée de \DeuxFoncteursStricts{} $u : \mathdeuxcat{A} \to \mathdeuxcat{C}$ et $v : \mathdeuxcat{B} \to \mathdeuxcat{C}$, la \emph{\deux{}catégorie comma de $u$ par $v$}\index{2categoriecomma@\deux{}catégorie comma}, notée $[u, v]$\index[not]{uv@$[u,v]$}. Un cas particulier de cette notion se trouve utilisé par Kelly dans \cite[section 4.1]{Kelly}. Il est possible d'étendre la construction présentée par Gray dans un cadre non strict en affaiblissant les morphismes $u$ et $v$. Il n'est toutefois pas possible de les prendre tous deux lax en général, une obstruction se présentant si l'on essaie alors de définir la composée des \un{}cellules. En revanche, il est possible de définir une telle \deux{}catégorie comma si $u$ est lax et $v$ colax. En particulier, l'on peut définir la \deux{}catégorie comma d'un \DeuxFoncteurLax{} par un \DeuxFoncteurStrict{}. La définition \ref{DefinitionTranche} et ses variantes duales fournissent des exemples de telles constructions, correspondant au cas dans lequel le morphisme strict a pour source la \deux{}catégorie ponctuelle. Sur une suggestion de Steve Lack, nous donnons la définition du cas général ainsi \hbox{que, plus loin, quelques propriétés de fonctorialité. Suivant la formule con-}\break{}sacrée\hbox{, tout cela se trouve sans doute} « bien connu des experts » ; nous n'en avons toutefois pas trouvé trace dans la littérature et nos demandes de références sont restées vaines.  
\end{paragr}

\begin{paragr}\label{DefComma}
Soient $\mathdeuxcat{A}$, $\mathdeuxcat{B}$ et $\mathdeuxcat{C}$ trois \deux{}catégories, $u : \mathdeuxcat{A} \to \mathdeuxcat{C}$ un \DeuxFoncteurLax{} et $v : \mathdeuxcat{B} \to \mathdeuxcat{C}$ un \DeuxFoncteurCoLax{}. On définit une \deux{}catégorie $[u,v]$ comme suit. 
\begin{itemize}
\item[(Objets)] Les objets de $[u,v]$ sont les triplets $(a, b, r : u(a) \to v(b))$, avec $a$ un objet de $\mathdeuxcat{A}$, $b$ un objet de $\mathdeuxcat{B}$ et $r$ une \un{}cellule de $\mathdeuxcat{C}$. 
\item[(\un{}cellules)] Les \un{}cellules de $(a, b, r)$ vers $(a', b', r')$ sont les triplets $(f : a \to a', g : b \to b', \alpha : v(g) r \Rightarrow r' u(f))$, avec $f$ une \un{}cellule de $\mathdeuxcat{A}$, $g$ une \un{}cellule de $\mathdeuxcat{B}$ et $\alpha$ une \deux{}cellule de $\mathdeuxcat{C}$. 
\item[(\deux{}cellules)] Étant donné deux \un{}cellules $(f, g, \alpha)$ et $(h, k, \beta)$ de $(a, b, r)$ vers $(a', b', r')$ dans $[u,v]$, les \deux{}cellules de $(f,g,\alpha)$ vers $(h,k,\beta)$ sont les couples $(\varphi : f \Rightarrow h, \psi : g \Rightarrow k)$, avec $\varphi$ une \deux{}cellule de $\mathdeuxcat{A}$ et $\psi$ une \deux{}cellule de $\mathdeuxcat{B}$, telles que 
$$
(r' \CompDeuxZero u(\varphi)) \CompDeuxUn \alpha = \beta \CompDeuxUn (v(\psi) \CompDeuxZero r).
$$
\item[(Identité des objets)] L'identité de l'objet $(a,b,r)$ est donnée par $(1_{a}, 1_{b}, (r \CompDeuxZero u_{a}) \CompDeuxUn (v_{b} \CompDeuxZero r))$. 
\item[(Identité des \un{}cellules)] L'identité de la \un{}cellule $(f,g,\alpha)$ est donnée par $(1_{f}, 1_{g})$.
\item[(Composition des \un{}cellules)] Si $(f,g,\alpha)$ est une \un{}cellule de $(a,b,r)$ vers $(a',b',r')$ et $(f',g',\alpha')$ est une \un{}cellule de $(a',b',r')$ vers $(a'',b'',r'')$, leur composée est donnée par la formule
$$
(f',g',\alpha') (f,g,\alpha) = (f'f, g'g, (r'' \CompDeuxZero u_{f',f}) \CompDeuxUn (\alpha' \CompDeuxZero u(f)) \CompDeuxUn (v(g') \CompDeuxZero \alpha) \CompDeuxUn (v_{g',g} \CompDeuxZero r)).
$$
\item[(Composition verticale des \deux{}cellules)] Soient $(f,g,\alpha)$, $(h,k,\beta)$ et $(p,q,\gamma)$ trois \un{}cellules de $(a,b,r)$ vers $(a',b',r')$, $(\varphi, \psi)$ une \deux{}cellule de $(f,g,\alpha)$ vers $(h,k,\beta)$ et $(\mu, \nu)$ une \deux{}cellule de $(h,k,\beta)$ vers $(p,q,\gamma)$. La composée de $(\mu, \nu)$ et $(\varphi, \psi)$ est donnée par la formule
$$
(\mu, \nu) \CompDeuxUn (\varphi, \psi) = (\mu \CompDeuxUn \varphi, \nu \CompDeuxUn \psi).
$$
Le fait que cela définit bien une \deux{}cellule de $(f,g,\alpha)$ vers $(p,q,\gamma)$ résulte des égalités suivantes. 
\begin{align*}
(r' \CompDeuxZero u(\mu \CompDeuxUn \varphi)) \CompDeuxUn \alpha &= (r' \CompDeuxZero (u(\mu) \CompDeuxUn u(\varphi))) \CompDeuxUn \alpha \phantom{blablablablablablablablablablablablablablablablablablablablablablablablablablablablablabla}
\\
&= (r' \CompDeuxZero u(\mu)) \CompDeuxUn (r' \CompDeuxZero u(\varphi)) \CompDeuxUn \alpha
\\
&= (r' \CompDeuxZero u(\mu)) \CompDeuxUn \alpha' \CompDeuxUn (v(\psi) \CompDeuxZero r)
\\
&= \alpha'' \CompDeuxUn (v(\nu) \CompDeuxZero r) \CompDeuxUn (v(\psi) \CompDeuxZero r)
\\
&= \alpha'' \CompDeuxUn ((v(\nu) \CompDeuxUn v(\psi)) \CompDeuxZero r)
\\
&= \alpha'' \CompDeuxUn (v(\nu \CompDeuxUn \psi) \CompDeuxZero r).
\end{align*}
\item[(Composition horizontale des \deux{}cellules)]
Soient $(f,g,\alpha)$ et $(h,k,\beta)$ deux \un{}cellules de $(a,b,r)$ vers $(a',b',r')$, $(f',g',\alpha')$ et $(h',k',\beta')$ deux \un{}cellules de $(a',b',r')$ vers $(a'',b'',r'')$, $(\varphi, \psi)$ une \deux{}cellule de $(f,g,\alpha)$ vers $(h,k,\beta)$ et $(\varphi', \psi')$ une \deux{}cellule de $(f',g',\alpha')$ vers $(h',k',\beta')$. La composée de $(\varphi', \psi')$ et $(\varphi, \psi)$ est donnée par la formule
$$
(\varphi', \psi') \CompDeuxZero (\varphi, \psi) = (\varphi' \CompDeuxZero \varphi, \psi' \CompDeuxZero \psi). 
$$
Le fait que cela définit bien une \deux{}cellule de $(f',g',\alpha') (f,g,\alpha)$ vers $(h',k',\beta') (h,k,\beta)$ résulte de la suite d'égalités suivantes. 
\begin{align*}
&(r'' \CompDeuxZero u(\varphi' \CompDeuxZero \varphi)) \CompDeuxUn (r'' \CompDeuxZero u_{f',f}) \CompDeuxUn (\alpha' \CompDeuxZero u(f)) \CompDeuxUn (v(g') \CompDeuxZero \alpha) \CompDeuxUn (v_{g',g} \CompDeuxZero r) \phantom{blablablablablablablablablablablablablablablablablablablabla}
\\
&= (r'' \CompDeuxZero (u(\varphi' \CompDeuxZero \varphi) \CompDeuxUn u_{f',f})) \CompDeuxUn (\alpha' \CompDeuxZero u(f)) \CompDeuxUn (v(g') \CompDeuxZero \alpha) \CompDeuxUn (v_{g',g} \CompDeuxZero r)
\\
&= (r'' \CompDeuxZero (u_{h',h} \CompDeuxUn (u(\varphi') \CompDeuxZero u(\varphi)))) \CompDeuxUn (\alpha' \CompDeuxZero u(f)) \CompDeuxUn (v(g') \CompDeuxZero \alpha) \CompDeuxUn (v_{g',g} \CompDeuxZero r)
\\
&= (r'' \CompDeuxZero u_{h',h}) \CompDeuxUn (r'' \CompDeuxZero u(\varphi') \CompDeuxZero u(\varphi)) \CompDeuxUn (\alpha' \CompDeuxZero u(f)) \CompDeuxUn (v(g') \CompDeuxZero \alpha) \CompDeuxUn (v_{g',g} \CompDeuxZero r)
\\
&= (r'' \CompDeuxZero u_{h',h}) \CompDeuxUn    (((r'' \CompDeuxZero u(\varphi')) \CompDeuxUn \alpha') \CompDeuxZero u(\varphi))      \CompDeuxUn (v(g') \CompDeuxZero \alpha) \CompDeuxUn (v_{g',g} \CompDeuxZero r)
\\
&= (r'' \CompDeuxZero u_{h',h}) \CompDeuxUn       ((\beta' \CompDeuxUn (v(\psi') \CompDeuxZero r')) \CompDeuxZero u(\varphi))              \CompDeuxUn (v(g') \CompDeuxZero \alpha) \CompDeuxUn (v_{g',g} \CompDeuxZero r)
\\
&= (r'' \CompDeuxZero u_{h',h}) \CompDeuxUn	              (\beta' \CompDeuxZero u(h)) \CompDeuxUn (v(\psi') \CompDeuxZero r' \CompDeuxZero u(\varphi))            \CompDeuxUn (v(g') \CompDeuxZero \alpha) \CompDeuxUn (v_{g',g} \CompDeuxZero r)
\\
&= (r'' \CompDeuxZero u_{h',h}) \CompDeuxUn	 (\beta' \CompDeuxZero u(h)) \CompDeuxUn      (v(k') \CompDeuxZero ((r' \CompDeuxZero u(\varphi)) \CompDeuxUn \alpha)) \CompDeuxUn (v(\psi') \CompDeuxZero v(g)r)                \CompDeuxUn (v_{g',g} \CompDeuxZero r)    
\\
&= (r'' \CompDeuxZero u_{h',h}) \CompDeuxUn	 (\beta' \CompDeuxZero u(h)) \CompDeuxUn (v(k') \CompDeuxZero (\beta \CompDeuxUn (v(\psi) \CompDeuxZero r)))   \CompDeuxUn (v(\psi') \CompDeuxZero v(g)r) \CompDeuxUn (v_{g',g} \CompDeuxZero r) 
\\
&= (r'' \CompDeuxZero u_{h',h}) \CompDeuxUn	 (\beta' \CompDeuxZero u(h)) \CompDeuxUn (v(k') \CompDeuxZero \beta) \CompDeuxUn (v(k') \CompDeuxZero v(\psi) \CompDeuxZero r)  \CompDeuxUn (v(\psi') \CompDeuxZero v(g)r) \CompDeuxUn (v_{g',g} \CompDeuxZero r)  
\\
&= (r'' \CompDeuxZero u_{h',h}) \CompDeuxUn	 (\beta' \CompDeuxZero u(h)) \CompDeuxUn (v(k') \CompDeuxZero \beta) \CompDeuxUn (((v(\psi') \CompDeuxZero v(\psi)) \CompDeuxUn v_{g',g}) \CompDeuxZero r)
\\
&= (r'' \CompDeuxZero u_{h',h}) \CompDeuxUn	 (\beta' \CompDeuxZero u(h)) \CompDeuxUn (v(k') \CompDeuxZero \beta) \CompDeuxUn ((v_{k',k} \CompDeuxUn (v(\psi' \CompDeuxZero \psi))) \CompDeuxZero r)
\\
&= (r'' \CompDeuxZero u_{h',h}) \CompDeuxUn	 (\beta' \CompDeuxZero u(h)) \CompDeuxUn (v(k') \CompDeuxZero \beta) \CompDeuxUn (v_{k',k} \CompDeuxZero r) \CompDeuxUn (v(\psi' \CompDeuxZero \psi) \CompDeuxZero r).
\end{align*}
\end{itemize}
Il reste à vérifier que cela définit bien une \deux{}catégorie $[u,v]$.
\begin{itemize}
\item Pour toute \deux{}cellule $(\varphi, \psi)$ de $(f,g,\alpha)$ vers $(h,k,\beta)$, les égalités 
$$
(\varphi, \psi) \CompDeuxUn (1_{f}, 1_{g}) = (1_{h}, 1_{k}) \CompDeuxUn (\varphi, \psi) = (\varphi, \psi)
$$ 
sont immédiates.  
\item La vérification de l'associativité de la composition verticale des \deux{}cellules est également immédiate. 
\item Pour tout couple de \un{}cellules $(f,g,\alpha)$ et $(f',g',\alpha')$ telles que la composée $(f',g',\alpha') (f,g,\alpha)$ fasse sens, l'égalité
$$
1_{(f',g',\alpha')} \CompDeuxZero 1_{(f,g,\alpha)} = 1_{(f',g',\alpha') (f,g,\alpha)}
$$
est immédiate. 
\item La loi d'échange dans $[u,v]$ est une conséquence immédiate de la loi d'échange dans $\mathdeuxcat{A}$ et $\mathdeuxcat{B}$. 
\item Soient $(f,g,\alpha)$ une \un{}cellule de $(a,b,r)$ vers $(a',b',r')$, $(f',g',\alpha')$ une \un{}cellule de $(a',b',r')$ vers $(a'',b'',r'')$ et $(f'',g'',\alpha'')$ une \un{}cellule de $(a'',b'',r'')$ vers $(a''',b''',r''')$. Alors, l'égalité 
$$
((f'',g'',\alpha'') (f',g',\alpha')) (f,g,\alpha) = (f'',g'',\alpha'') ((f',g',\alpha') (f,g,\alpha))
$$
résulte de la suite d'égalités suivantes. 
\begin{align*}
&(r''' \CompDeuxZero u_{f''f',f}) \CompDeuxUn (r''' \CompDeuxZero u_{f'',f'} \CompDeuxZero u(f)) \CompDeuxUn (\alpha'' \CompDeuxZero u(f') u(f)) \CompDeuxUn (v(g'') \CompDeuxZero \alpha' \CompDeuxZero u(f)) \CompDeuxUn (v_{g'',g'} \CompDeuxZero r'u(f)) \CompDeuxUn
\\
&\phantom{bla} (v(g''g') \CompDeuxZero \alpha) \CompDeuxUn (v_{g''g',g} \CompDeuxZero r) 
\\
&= (r''' \CompDeuxZero u_{f'',f'f}) \CompDeuxUn (r'''u(f'') \CompDeuxZero u_{f',f}) \CompDeuxUn (\alpha'' \CompDeuxZero u(f') u(f)) \CompDeuxUn (v(g'') \CompDeuxZero \alpha' \CompDeuxZero u(f)) \CompDeuxUn (v_{g'',g'} \CompDeuxZero r'u(f)) \CompDeuxUn
\\
&\phantom{bla} (v(g''g') \CompDeuxZero \alpha) \CompDeuxUn (v_{g''g',g} \CompDeuxZero r)
\\
&= (r''' \CompDeuxZero u_{f'',f'f}) \CompDeuxUn (\alpha'' \CompDeuxZero u(f'f)) \CompDeuxUn (v(g'') r'' \CompDeuxZero u_{f',f}) \CompDeuxUn (v(g'') \CompDeuxZero \alpha' \CompDeuxZero u(f)) \CompDeuxUn (v_{g'',g'} \CompDeuxZero r'u(f)) \CompDeuxUn
\\
&\phantom{bla} (v(g''g') \CompDeuxZero \alpha) \CompDeuxUn (v_{g''g',g} \CompDeuxZero r)
\\
&= (r''' \CompDeuxZero u_{f'',f'f}) \CompDeuxUn (\alpha'' \CompDeuxZero u(f'f)) \CompDeuxUn (v(g'') r'' \CompDeuxZero u_{f',f}) \CompDeuxUn (v(g'') \CompDeuxZero \alpha' \CompDeuxZero u(f)) \CompDeuxUn (v(g'') v(g') \CompDeuxZero \alpha) \CompDeuxUn
\\
&\phantom{bla}  (v_{g'',g'} \CompDeuxZero v(g)r) \CompDeuxUn (v_{g''g',g} \CompDeuxZero r) 
\\
&= (r''' \CompDeuxZero u_{f'',f'f}) \CompDeuxUn (\alpha'' \CompDeuxZero u(f'f)) \CompDeuxUn (v(g'') r'' \CompDeuxZero u_{f',f}) \CompDeuxUn (v(g'') \CompDeuxZero \alpha' \CompDeuxZero u(f)) \CompDeuxUn (v(g'') v(g') \CompDeuxZero \alpha) \CompDeuxUn
\\
&\phantom{bla}  (v(g'') \CompDeuxZero v_{g',g} \CompDeuxZero r) \CompDeuxUn (v_{g'',g'g} \CompDeuxZero r). 
\end{align*}
\item L'associativité de la composition horizontale des \deux{}cellules se vérifie immédiatement. 
\item Pour toute \un{}cellule $(f,g,\alpha)$ de $(a,b,r)$ vers $(a',b',r')$, 
\begin{align*}
(f,g,\alpha) 1_{(a,b,r)} &= (f,g,\alpha) (1_{a}, 1_{b}, (r \CompDeuxZero u_{a}) \CompDeuxUn (v_{b} \CompDeuxZero r)) 
\\
&= (f,g, (r' \CompDeuxZero u_{f, 1_{a}}) \CompDeuxUn (\alpha \CompDeuxZero u(1_{a})) \CompDeuxUn (v(g) r \CompDeuxZero u_{a}) \CompDeuxUn (v(g) \CompDeuxZero v_{b} \CompDeuxZero r) \CompDeuxUn (v_{g, 1_{b}} \CompDeuxZero r))
\\
&= (f,g, (r' \CompDeuxZero u_{f, 1_{a}}) \CompDeuxUn (\alpha \CompDeuxZero u(1_{a})) \CompDeuxUn (v(g) r \CompDeuxZero u_{a}))
\\
&= (f,g, (r' \CompDeuxZero u_{f, 1_{a}}) \CompDeuxUn (r'u(f) \CompDeuxZero u_{a}) \CompDeuxUn \alpha) 
\\
&= (f,g, (r' \CompDeuxZero (u_{f, 1_{a}} \CompDeuxUn (u(f) \CompDeuxZero u_{a}))) \CompDeuxUn \alpha)
\\
&= (f,g,\alpha). 
\end{align*}
\item Sous les mêmes données, l'égalité $1_{(a',b',r')} (f,g,\alpha) = (f,g,\alpha)$ se démontre de façon analogue. 
\item Soient $(f,g,\alpha)$ et $(h,k,\beta)$ deux \un{}cellules de $(a,b,r)$ vers $(a',b',r')$ et $(\varphi, \psi)$ une \deux{}cellule de $(f,g,\alpha)$ vers $(h,k,\beta)$. Les égalités 
$$
(\varphi, \psi) \CompDeuxZero (1_{1_{a}}, 1_{1_{b}}) = (\varphi, \psi)
$$
et
$$
(1_{1_{a'}}, 1_{1_{b'}}) \CompDeuxZero (\varphi, \psi) = (\varphi, \psi)
$$
se vérifient immédiatement. 
\end{itemize}
Les vérifications qui précèdent assurent que l'on a bien défini une \deux{}catégorie $[u,v]$. 
\end{paragr}

\begin{paragr}
Soient $\mathdeuxcat{A}$ et $\mathdeuxcat{B}$ des \deux{}catégories, $u : \mathdeuxcat{A} \to \mathdeuxcat{B}$ un \DeuxFoncteurLax{} et $v : \mathdeuxcat{A} \to \mathdeuxcat{B}$ un \DeuxFoncteurCoLax{}. Une \emph{optransformation}\index{optransformation (d'un \deux{}foncteur lax vers un \deux{}foncteur colax)} $\sigma$ de $u$ vers $v$ est donnée par une \un{}cellule $\sigma_{a} : u(a) \to v(a)$ de $\mathdeuxcat{B}$ pour tout objet $a$ de $\mathdeuxcat{A}$ et une \deux{}cellule $\sigma_{f} : v(f) \sigma_{a} \Rightarrow \sigma_{a'} u(f)$ de $\mathdeuxcat{B}$ pour toute \un{}cellule $f : a \to a'$ de $\mathdeuxcat{A}$, ces données vérifiant les conditions suivantes.
\begin{itemize}
\item Pour tout objet $a$ de $\mathdeuxcat{A}$,
$$
(\sigma_{a} \CompDeuxZero u_{a}) \CompDeuxUn (v_{a} \CompDeuxZero \sigma_{a}) = \sigma_{1_{a}}.
$$
\item Pour tout couple de \un{}cellules composables $f : a \to a'$ et $f' : a' \to a''$ de $\mathdeuxcat{A}$, 
$$
(\sigma_{a''} \CompDeuxZero u_{f',f}) \CompDeuxUn (\sigma_{f'} \CompDeuxZero u(f)) \CompDeuxUn (v(f') \CompDeuxZero \sigma_{f}) \CompDeuxUn (v_{f',f} \CompDeuxZero \sigma_{a}) = \sigma_{f'f}. 
$$
\item Pour tout couple de \un{}cellules $f$ et $g$ de $a$ vers $a'$ de $\mathdeuxcat{A}$, pour toute \deux{}cellule $\alpha : f \Rightarrow g$ de $\mathdeuxcat{A}$,
$$
(\sigma_{a'} \CompDeuxZero u(\alpha)) \CompDeuxUn \sigma_{f} = \sigma_{g} \CompDeuxUn (v(\alpha) \CompDeuxZero \sigma_{a}).
$$
\end{itemize}
Sous les données du paragraphe \ref{DefComma}, il existe un diagramme
$$
\xymatrix{
[u,v]
\ar[r]^{q}
\ar[d]_{p}
& \mathdeuxcat{B}
\ar[d]^{v}
\\
\mathdeuxcat{A}
\ar[r]_{u}
& \mathdeuxcat{C}
&,
}
$$
dans lequel $p$ et $q$ désignent les projections canoniques, commutatif à une \DeuxTransformationCoLax{} $\sigma : up \Rightarrow vq$ près, définie par $\sigma_{(a,b,r)} = r$ et $\sigma_{(f,g,\alpha)} = \alpha$ pour tout objet $(a,b,r)$ et toute \un{}cellule $(f,g,\alpha)$ de $[u,v]$.
\end{paragr}

\begin{df}\label{DefinitionTranche}
Soient $u : \mathdeuxcat{A} \to \mathdeuxcat{B}$ un \DeuxFoncteurLax{} et $b$ un objet de $\mathdeuxcat{B}$. La \emph{\deux{}catégorie tranche lax de $\mathdeuxcat{A}$ au-dessus de $b$ relativement à $u$}\index{tranche lax au-dessus d'un objet} est la \deux{}catégorie, que l'on notera $\TrancheLax{\mathdeuxcat{A}}{u}{b}$\index[not]{Aulb@$\TrancheLax{\mathdeuxcat{A}}{u}{b}$}, définie comme la \deux{}catégorie comma $[u,b]$, en considérant $b$ comme un \DeuxFoncteurStrict{} de $\DeuxCatPonct$ vers $\mathdeuxcat{B}$. De façon plus explicite, cette \deux{}catégorie se décrit comme suit.
\begin{itemize}
\item
Les objets de $\TrancheLax{\mathdeuxcat{A}}{u}{b}$ sont les couples $(a, p : u(a) \to b)$, où $a$ est un objet de $\mathdeuxcat{A}$ et $p$ une \un{}cellule de $\mathdeuxcat{B}$.
\item
Si $(a, p : u(a) \to b)$ et $(a', p' : u(a') \to b)$ sont deux objets de $\TrancheLax{\mathdeuxcat{A}}{u}{b}$, les \un{}cellules de $(a,p)$ vers $(a',p')$ dans $\TrancheLax{\mathdeuxcat{A}}{u}{b}$ sont les couples $(f : a \to a', \alpha : p \Rightarrow p' u(f))$, où $f$ est une \un{}cellule de $\mathdeuxcat{A}$ et $\alpha$ est une \deux{}cellule de $\mathdeuxcat{B}$. Le diagramme à conserver en tête est le suivant : 
$$
\xymatrix{
u(a) 
\ar[rr]^{u(f)}
\ar[dr]_{p}
&{}
&u(a')
\ar[dl]^{p'}
\\
&b
\utwocell<\omit>{\alpha}
&.
}
$$
\item 
Si $(f, \alpha)$ et $(f', \alpha')$ sont deux \un{}cellules de $(a,p)$ vers $(a',p')$ dans $\TrancheLax{\mathdeuxcat{A}}{u}{b}$, les \deux{}cellules de $(f, \alpha)$ vers $(f', \alpha')$ dans $\TrancheLax{\mathdeuxcat{A}}{u}{b}$ sont les \deux{}cellules $\beta : f \Rightarrow f'$ dans $\mathdeuxcat{A}$ telles que $(p' \CompDeuxZero u(\beta)) \alpha = \alpha'$.
\item
Si $(f, \alpha) : (a,p) \to (a',p')$ et $(f', \alpha') : (a',p') \to (a'',p'')$ sont deux \un{}cellules dans $\TrancheLax{\mathdeuxcat{A}}{u}{b}$, leur composée est définie par
$$
(f', \alpha') (f, \alpha) = (f'f, (p'' \CompDeuxZero \DeuxCellStructComp{u}{f'}{f}) (\alpha' \CompDeuxZero u(f)) \alpha).
$$
\item
L'identité d'un objet $(a, p : u(a) \to b)$ de $\TrancheLax{\mathdeuxcat{A}}{u}{b}$ est donnée par le couple $(1_{a}, p \CompDeuxZero \DeuxCellStructId{u}{a})$. 
\item
L'identité d'une \un{}cellule $(f, \alpha) : (a,p) \to (a',p')$ de $\TrancheLax{\mathdeuxcat{A}}{u}{b}$ n'est autre que $1_{f}$.
\item
Les \deux{}cellules se composent comme dans $\mathdeuxcat{A}$. 
\end{itemize}
\end{df}

%
%

\begin{df}
Soient $u : \mathdeuxcat{A} \to \mathdeuxcat{B}$ un \DeuxFoncteurLax{} et $b$ un objet de $\mathdeuxcat{B}$. La \emph{\deux{}catégorie optranche lax de $\mathdeuxcat{A}$ au-dessous de $b$ relativement à $u$}\index{optranche lax au-dessous d'un objet} est la \deux{}catégorie $\OpTrancheLax{\mathdeuxcat{A}}{u}{b}$\index[not]{bulA@$\OpTrancheLax{\mathdeuxcat{A}}{u}{b}$} définie par la formule 
$$
\OpTrancheLax{\mathdeuxcat{A}}{u}{b} = \DeuxCatUnOp{(\TrancheLax{\DeuxCatUnOp{\mathdeuxcat{A}}}{\DeuxFoncUnOp{u}}{b})}.
$$
En particulier, les objets, \un{}cellules et \deux{}cellules de $\OpTrancheLax{\mathdeuxcat{A}}{u}{b}$ se décrivent comme suit.
\begin{itemize}
\item
Les objets de $\OpTrancheLax{\mathdeuxcat{A}}{u}{b}$ sont les couples $(a, p : b \to u(a))$, où $a$ est un objet de $\mathdeuxcat{A}$ et $p$ une \un{}cellule de $\mathdeuxcat{B}$.
\item
Si $(a, p : b \to u(a))$ et $(a', p' : b \to u(a'))$ sont deux objets de $\OpTrancheLax{\mathdeuxcat{A}}{u}{b}$, les \un{}cellules de $(a,p)$ vers $(a',p')$ dans $\OpTrancheLax{\mathdeuxcat{A}}{u}{b}$ sont les couples $(f : a \to a', \alpha : p' \Rightarrow u(f) p)$, où $f$ est une \un{}cellule de $\mathdeuxcat{A}$ et $\alpha$ est une \deux{}cellule de $\mathdeuxcat{B}$. 
\item 
Si $(f, \alpha)$ et $(f', \alpha')$ sont deux \un{}cellules de $(a,p)$ vers $(a',p')$ dans $\OpTrancheLax{\mathdeuxcat{A}}{u}{b}$, les \deux{}cellules de $(f, \alpha)$ vers $(f', \alpha')$ dans $\OpTrancheLax{\mathdeuxcat{A}}{u}{b}$ sont les \deux{}cellules $\beta : f \Rightarrow f'$ dans $\mathdeuxcat{A}$ telles que $(u(\beta) \CompDeuxZero p) \CompDeuxUn \alpha = \alpha'$.
\end{itemize}
\end{df}

\begin{df}
Soient $u : \mathdeuxcat{A} \to \mathdeuxcat{B}$ un \DeuxFoncteurCoLax{} et $b$ un objet de $\mathdeuxcat{B}$. La \emph{\deux{}catégorie tranche colax de $\mathdeuxcat{A}$ au-dessus de $b$ relativement à $u$}\index{tranche colax au-dessus d'un objet} est la \deux{}catégorie $\TrancheCoLax{\mathdeuxcat{A}}{u}{b}$\index[not]{Aucb@$\TrancheCoLax{\mathdeuxcat{A}}{u}{b}$} définie par la formule 
$$
\TrancheCoLax{\mathdeuxcat{A}}{u}{b} = \DeuxCatDeuxOp{(\TrancheLax{\DeuxCatDeuxOp{\mathdeuxcat{A}}}{\DeuxFoncDeuxOp{u}}{b})}.
$$
En particulier, les objets, \un{}cellules et \deux{}cellules de $\TrancheCoLax{\mathdeuxcat{A}}{u}{b}$ se décrivent comme suit.

\begin{itemize}
\item
Les objets de $\TrancheCoLax{\mathdeuxcat{A}}{u}{b}$ sont les couples $(a, p : u(a) \to b)$, où $a$ est un objet de $\mathdeuxcat{A}$ et $p$ une \un{}cellule de $\mathdeuxcat{B}$.
\item
Si $(a, p : u(a) \to b)$ et $(a', p' : u(a') \to b)$ sont deux objets de $\TrancheCoLax{\mathdeuxcat{A}}{u}{b}$, les \un{}cellules de $(a,p)$ vers $(a',p')$ dans $\TrancheCoLax{\mathdeuxcat{A}}{u}{b}$ sont les couples $(f : a \to a', \alpha : p' u(f) \Rightarrow p)$, où $f$ est une \un{}cellule de $\mathdeuxcat{A}$ et $\alpha$ est une \deux{}cellule de $\mathdeuxcat{B}$. 
\item 
Si $(f, \alpha)$ et $(f', \alpha')$ sont deux \un{}cellules de $(a,p)$ vers $(a',p')$ dans $\TrancheCoLax{\mathdeuxcat{A}}{u}{b}$, les \deux{}cellules de $(f, \alpha)$ vers $(f', \alpha')$ dans $\TrancheCoLax{\mathdeuxcat{A}}{u}{b}$ sont les \deux{}cellules $\beta : f \Rightarrow f'$ dans $\mathdeuxcat{A}$ telles que 
$\alpha' \CompDeuxUn (p' \CompDeuxZero u(\beta)) = \alpha$. 
\end{itemize}
\end{df}

\begin{df}
Soient $u : \mathdeuxcat{A} \to \mathdeuxcat{B}$ un \DeuxFoncteurCoLax{} et $b$ un objet de $\mathdeuxcat{B}$. La \emph{\deux{}catégorie optranche colax de $\mathdeuxcat{A}$ au-dessous de $b$ relativement à $u$}\index{optranche colax au-dessous d'un objet} est la \deux{}catégorie $\OpTrancheCoLax{\mathdeuxcat{A}}{u}{b}$\index[not]{bucA@$\OpTrancheCoLax{\mathdeuxcat{A}}{u}{b}$} définie par la formule 
$$
\OpTrancheCoLax{\mathdeuxcat{A}}{u}{b} = \DeuxCatToutOp{(\TrancheLax{\DeuxCatToutOp{\mathdeuxcat{A}}}{\DeuxFoncToutOp{u}}{b})}.
$$
En particulier, les objets, \un{}cellules et \deux{}cellules de $\OpTrancheCoLax{\mathdeuxcat{A}}{u}{b}$ se décrivent comme suit.

\begin{itemize}
\item
Les objets de $\OpTrancheCoLax{\mathdeuxcat{A}}{u}{b}$ sont les couples $(a, p : b \to u(a))$, où $a$ est un objet de $\mathdeuxcat{A}$ et $p$ une \un{}cellule de $\mathdeuxcat{B}$.
\item
Si $(a, p : b \to u(a))$ et $(a', p' : b \to u(a'))$ sont deux objets de $\OpTrancheCoLax{\mathdeuxcat{A}}{u}{b}$, les \un{}cellules de $(a,p)$ vers $(a',p')$ dans $\OpTrancheCoLax{\mathdeuxcat{A}}{u}{b}$ sont les couples $(f : a \to a', \alpha : u(f) p \Rightarrow p')$, où $f$ est une \un{}cellule de $\mathdeuxcat{A}$ et $\alpha$ est une \deux{}cellule de $\mathdeuxcat{B}$. 
\item 
Si $(f, \alpha)$ et $(f', \alpha')$ sont deux \un{}cellules de $(a,p)$ vers $(a',p')$ dans $\OpTrancheCoLax{\mathdeuxcat{A}}{u}{b}$, les \deux{}cellules de $(f, \alpha)$ vers $(f', \alpha')$ dans $\OpTrancheCoLax{\mathdeuxcat{A}}{u}{b}$ sont les \deux{}cellules $\beta : f \Rightarrow f'$ dans $\mathdeuxcat{A}$ telles que $\alpha' \CompDeuxUn (u(\beta) \CompDeuxZero p) = \alpha$.
\end{itemize}
\end{df}

Si $a$ est un objet de la \deux{}catégorie $\mathdeuxcat{A}$, on notera $\TrancheLax{\mathdeuxcat{A}}{}{a}$\index[not]{Ala@$\TrancheLax{\mathdeuxcat{A}}{}{a}$} (\emph{resp.} $\OpTrancheLax{\mathdeuxcat{A}}{}{a}$\index[not]{AlaBis@$\OpTrancheLax{\mathdeuxcat{A}}{}{a}$}, \emph{resp.} $\TrancheCoLax{\mathdeuxcat{A}}{}{a}$\index[not]{Aca@$\TrancheCoLax{\mathdeuxcat{A}}{}{a}$}, \emph{resp.} $\OpTrancheCoLax{\mathdeuxcat{A}}{}{a}$\index[not]{AcaBis@$\OpTrancheCoLax{\mathdeuxcat{A}}{}{a}$}) la \deux{}catégorie $\TrancheLax{\mathdeuxcat{A}}{1_{\mathdeuxcat{A}}}{a}$ (\emph{resp.} $\OpTrancheLax{\mathdeuxcat{A}}{1_{\mathdeuxcat{A}}}{a}$, \emph{resp.} $\TrancheCoLax{\mathdeuxcat{A}}{1_{\mathdeuxcat{A}}}{a}$, \emph{resp.} $\OpTrancheCoLax{\mathdeuxcat{A}}{1_{\mathdeuxcat{A}}}{a}$).

\section{Morphismes induits par une transformation}\label{SectionMorphismesInduits}

\begin{paragr}\label{CommaFonctorielle1}
Sous les mêmes hypothèses que celles du paragraphe \ref{DefComma}, soient un \DeuxFoncteurLax{} \hbox{$u' : \mathdeuxcat{A} \to \mathdeuxcat{C}$} ainsi qu'une \DeuxTransformationCoLax{} $\sigma : u \Rightarrow u'$. Cela va nous permettre de construire un \DeuxFoncteurStrict{} $[\sigma, v]\index[not]{0sigmav@$[\sigma, v]$} : [u', v] \to [u, v]$ comme suit. 
\begin{itemize}
\item Pour tout objet $(a,b,r)$ de $[u',v]$, on pose
$$
[\sigma, v] (a,b,r) = (a,b,r\sigma_{a}).
$$
\item Pour toute \un{}cellule $(f,g,\alpha)$ de $(a,b,r)$ vers $(a',b',r')$ dans $[u',v]$, on pose
$$
[\sigma,v] (f,g,\alpha) = (f,g, (r' \CompDeuxZero \sigma_{f}) \CompDeuxUn (\alpha \CompDeuxZero \sigma_{a})). 
$$
\item Soient, dans $[u',v]$, $(f,g,\alpha)$ et $(h,k,\beta)$ des \un{}cellules de $(a,b,r)$ vers $(a',b',r')$ et $(\varphi, \psi)$ une \deux{}cellule de $(f,g,\alpha)$ vers $(h,k,\beta)$. On pose 
$$
[\sigma,v] (\varphi, \psi) = (\varphi, \psi). 
$$
Cela définit bien une \deux{}cellule de $[\sigma,v] (f,g,\alpha)$ vers $[\sigma,v] (h,k,\beta)$ en vertu des égalités suivantes. 
\begin{align*}
(r' \sigma_{a'} \CompDeuxZero u(\varphi)) \CompDeuxUn (r' \CompDeuxZero \sigma_{f}) \CompDeuxUn (\alpha \CompDeuxZero \sigma_{a}) &= (r' \CompDeuxZero ((\sigma_{a'} \CompDeuxZero u(\varphi)) \CompDeuxUn \sigma_{f})) \CompDeuxUn (\alpha \CompDeuxZero \sigma_{a}) 
\\
&= (r' \CompDeuxZero (\sigma_{h} \CompDeuxUn (u'(\varphi) \CompDeuxZero \sigma_{a}))) \CompDeuxUn (\alpha \CompDeuxZero \sigma_{a}) 
\\
&= (r' \CompDeuxZero \sigma_{h}) \CompDeuxUn (r' \CompDeuxZero u'(\varphi) \CompDeuxZero \sigma_{a}) \CompDeuxUn (\alpha \CompDeuxZero \sigma_{a})
\\
&= (r' \CompDeuxZero \sigma_{h}) \CompDeuxUn (((r' \CompDeuxZero u'(\varphi)) \CompDeuxUn \alpha) \CompDeuxZero \sigma_{a})
\\
&= (r' \CompDeuxZero \sigma_{h}) \CompDeuxUn ((\beta \CompDeuxUn (v(\psi) \CompDeuxZero r)) \CompDeuxZero \sigma_{a})
\\
&= (r' \CompDeuxZero \sigma_{h}) \CompDeuxUn (\beta \CompDeuxZero \sigma_{a}) \CompDeuxUn (v(\psi) \CompDeuxZero r \sigma_{a}). 
\end{align*}
\item Pour tout objet $(a,b,r)$ de $[u',v]$, il résulte de l'égalité $\sigma_{1_{a}} \CompDeuxUn (u'_{a} \CompDeuxZero \sigma_{a}) = \sigma_{a} \CompDeuxZero u_{a}$ que 
$$
[\sigma,v] (1_{(a,b,r)}) = 1_{[\sigma,v] (a,b,r)}.
$$
\item Soient $(f,g,\alpha)$ une \un{}cellule de $(a,b,r)$ vers $(a',b',r')$ et $(f',g',\alpha')$ une \un{}cellule de $(a',b',r')$ vers $(a'',b'',r'')$ dans $[u',v]$. L'égalité
$$
[\sigma,v] (f',g',\alpha') [\sigma,v] (f,g,\alpha) = [\sigma,v] ((f',g',\alpha') (f,g,\alpha))
$$ 
résulte des égalités
$$
\sigma_{f'f} (u'_{f',f} \CompDeuxZero \sigma_{a}) = (\sigma_{a''} \CompDeuxZero u_{f',f}) \CompDeuxUn (\sigma_{f'} \CompDeuxZero u(f)) \CompDeuxUn (u'(f') \CompDeuxZero \sigma_{f})
$$
(conséquence du fait que $\sigma$ est une \DeuxTransformationCoLax{}) et 
$$
(r'' u'(f') \CompDeuxZero \sigma_{f}) \CompDeuxUn (\alpha' \CompDeuxZero u'(f) \sigma_{a}) = (\alpha' \CompDeuxZero \sigma_{a'} u(f)) \CompDeuxUn (v(g')r' \CompDeuxZero \sigma_{f})
$$
(conséquence de la loi d'échange).
\item Pour tout couple de \deux{}cellules $(\varphi, \psi)$ et $(\varphi', \psi')$ de $[u',v]$ telles que la composée $(\varphi', \psi') \CompDeuxZero (\varphi, \psi)$ fasse sens, l'égalité
$$
[\sigma,v] ((\varphi', \psi') \CompDeuxZero (\varphi, \psi)) = ([\sigma,v] (\varphi', \psi')) \CompDeuxZero ([\sigma,v] (\varphi, \psi))
$$
est immédiate. 
\item Pour tout couple de \deux{}cellules $(\mu, \nu)$ et $(\varphi, \psi)$ de $[u',v]$ telles que la composée $(\mu,\nu) \CompDeuxUn (\varphi, \psi)$ fasse sens, l'égalité
$$
[\sigma,v] ((\mu, \nu) \CompDeuxUn (\varphi, \psi)) = ([\sigma,v] (\mu, \nu)) \CompDeuxUn ([\sigma,v] (\varphi, \psi))
$$
est immédiate. 
\item Il en est de même, pour toute \un{}cellule $(f,g,\alpha)$ de $[u',v]$, de l'égalité 
$$
[\sigma,v] (1_{(f,g,\alpha)}) = 1_{[\sigma,v] (f,g,\alpha)}.
$$
\end{itemize}
Cela termine les vérifications permettant d'affirmer que l'on a bien défini un \DeuxFoncteurStrict{} $[\sigma,v] : [u',v] \to [u,v]$. 
\end{paragr}

\begin{paragr}\label{CommaFonctorielle2}
Soient des \DeuxFoncteursLax{} composables $u : \mathdeuxcat{A} \to \mathdeuxcat{B}$ et $u' : \mathdeuxcat{B} \to \mathdeuxcat{C}$ et un \DeuxFoncteurCoLax{} $v : \mathdeuxcat{D} \to \mathdeuxcat{C}$. Cela va nous permettre de construire un \DeuxFoncteurLax{} $F : [u'u,v] \to [u',v]$ comme suit.
\begin{itemize}
\item Pour tout objet $(a,d,r)$ de $[u'u,v]$, posons
$$
F(a,d,r) = (u(a),d,r).
$$
\item Pour toute \un{}cellule $(f,g,\alpha)$ de $[u'u,v]$, posons
$$
F(f,g,\alpha) = (u(f),g,\alpha).
$$
\item Pour toute \deux{}cellule $(\varphi, \psi)$ de $[u'u,v]$, posons
$$
F(\varphi, \psi) = (u(\varphi), \psi). 
$$
\item Pour tout objet $(a,d,r)$ de $[u'u,v]$, posons
$$
F_{(a,d,r)} = (u_{a}, 1_{1_{d}}).
$$
\item Pour tout couple de \un{}cellules $(f',g',\alpha')$ et $(f,g,\alpha)$ de $[u'u,v]$ telles que la composée $(f',g',\alpha') (f,g,\alpha)$ fasse sens, posons
$$
F_{(f',g',\alpha'), (f,g,\alpha)} = (u_{f',f}, 1_{g'g}).
$$
\end{itemize}
Il est trivial que ces définitions font sens et que l'on a bien de la sorte construit un \DeuxFoncteurLax{} : les vérifications des conditions portant sur $F$ se ramènent toutes à des vérifications de conditions portant sur $u$ qui sont vraies par hypothèse. On remarquera que $F$ est strict si $u$ l'est, même si $u'$ ne l'est pas.
\end{paragr}

\begin{paragr}\label{DefDeuxFoncInduitFoncLaxOpTrans}
Supposons donné un diagramme de \DeuxFoncteursLax{}
$$
\xymatrix{
\mathdeuxcat{A} 
\ar[rr]^{u}
\ar[dr]_{w}
&&\mathdeuxcat{B}
\ar[dl]^{v}
\dtwocell<\omit>{<7.3>\sigma}
\\
&\mathdeuxcat{C}
&{}
}
$$
commutatif à l'\DeuxTransformationCoLax{} $\sigma : vu \Rightarrow w$ près seulement.

Il résulte des paragraphes \ref{CommaFonctorielle1} et \ref{CommaFonctorielle2} que ces données permettent de définir, pour tout objet $c$ de $\mathdeuxcat{C}$, un \DeuxFoncteurLax{}
$$
\DeuxFoncTrancheLaxCoq{u}{\sigma}{c}\index[not]{u0la@$\DeuxFoncTrancheLaxCoq{u}{\sigma}{c}$} : \TrancheLax{\mathdeuxcat{A}}{w}{c} \to \TrancheLax{\mathdeuxcat{B}}{v}{c}
$$
comme suit.
\begin{itemize}
\item Pour tout objet $(a, p : w(a) \to c)$ de $\TrancheLax{\mathdeuxcat{A}}{w}{c}$, 
$$
(\DeuxFoncTrancheLaxCoq{u}{\sigma}{c}) (a,p) = (u(a), p \sigma_{a} : v(u(a)) \to w(a) \to c).
$$
\item Pour toute \un{}cellule $(f : a \to a', \alpha : p \Rightarrow p'w(f))$ de $(a, p : w(a) \to c)$ vers $(a', p' : w(a') \to c)$ dans $\TrancheLax{\mathdeuxcat{A}}{w}{c}$, 
$$
(\DeuxFoncTrancheLaxCoq{u}{\sigma}{c}) (f, \alpha) = (u(f), (p' \CompDeuxZero \sigma_{f}) \CompDeuxUn (\alpha \CompDeuxZero \sigma_{a}) : p\sigma_{a} \Rightarrow p' w(f) \sigma_{a} \Rightarrow p' \sigma_{a'} v(u(f))).
$$
\item Pour tout couple d'objets $(a,p)$ et $(a',p')$, pour tout couple de \un{}cellules $(f, \alpha)$ et $(f', \alpha')$ de $(a,p)$ vers $(a',p')$ et toute \deux{}cellule $\beta : f \Rightarrow f'$ de $(f, \alpha)$ vers $(f', \alpha')$ dans $\TrancheLax{\mathdeuxcat{A}}{w}{c}$ (on a donc l'égalité $(p' \CompDeuxZero w(\beta)) \CompDeuxUn \alpha = \alpha'$), 
$$
(\DeuxFoncTrancheLaxCoq{u}{\sigma}{c}) (\beta) = u(\beta).
$$

\item
Soient $(f, \alpha)$ une \un{}cellule de $(a,p)$ vers $(a',p')$ et $(f', \alpha')$ une \un{}cellule de $(a',p')$ vers $(a'', p'')$ dans $ \TrancheLax{\mathdeuxcat{A}}{w}{c}$. On pose 
$$
(\DeuxFoncTrancheLaxCoq{u}{\sigma}{c})_{(f', \alpha'), (f, \alpha)} = u_{f', f}.
$$

\item
Pour tout objet $(a,p)$ de $\TrancheLax{\mathdeuxcat{A}}{w}{c}$, on pose
$$
(\DeuxFoncTrancheLaxCoq{u}{\sigma}{c})_{(a,p)} =u_{a}.
$$ 
\end{itemize}

Si le triangle ci-dessus est commutatif, c'est-à-dire que $vu = w$, et si $\sigma = 1_{w}$, on notera ce \DeuxFoncteurLax{} induit $\DeuxFoncTrancheLax{u}{c}$\index[not]{ulc@$\DeuxFoncTrancheLax{u}{c}$} plutôt que $\DeuxFoncTrancheLaxCoq{u}{1_{w}}{c}$. 
\end{paragr}

\begin{rem}
Si, dans la situation ci-dessus, $u$ est un \DeuxFoncteurStrict, alors $\DeuxFoncTrancheLaxCoq{u}{\sigma}{c}$ est un \DeuxFoncteurStrict, même si $v$ ne l'est pas.
\end{rem}

\begin{paragr}\label{DefDeuxFoncInduitFoncLaxTrans}
Soit un diagramme de \DeuxFoncteursLax{}
$$
\xymatrix{
\mathdeuxcat{A} 
\ar[rr]^{u}
\ar[dr]_{w}
&{}
&\mathdeuxcat{B}
\ar[dl]^{v}
\\
&\mathdeuxcat{C}
\utwocell<\omit>{\sigma}
}
$$  
commutatif à la \DeuxTransformationLax{} $\sigma : w \Rightarrow vu$ près seulement. Cela fournit un diagramme de \DeuxFoncteursLax{} 
$$
\xymatrix{
\DeuxCatUnOp{\mathdeuxcat{A}}
\ar[rr]^{\DeuxFoncUnOp{u}}
\ar[dr]_{\DeuxFoncUnOp{w}}
&&\DeuxFoncUnOp{\mathdeuxcat{B}}
\ar[dl]^{\DeuxFoncUnOp{v}}
\dtwocell<\omit>{<9>\DeuxTransUnOp{\sigma}}
\\
&\DeuxFoncUnOp{\mathdeuxcat{C}}
&{}
}
$$  
commutatif à l'\DeuxTransformationCoLax{} $\DeuxTransUnOp{\sigma} : \DeuxFoncUnOp{v} \DeuxFoncUnOp{u} \Rightarrow \DeuxFoncUnOp{w}$ près seulement. En vertu du paragraphe \ref{DefDeuxFoncInduitFoncLaxOpTrans}, on a donc un \DeuxFoncteurLax{} 
$$
\DeuxFoncTrancheLaxCoq{\DeuxFoncUnOp{u}}{\DeuxTransUnOp{\sigma}}{c} :
\TrancheLax{\DeuxCatUnOp{\mathdeuxcat{A}}}{\DeuxFoncUnOp{w}}{c} \to\TrancheLax{\DeuxCatUnOp{\mathdeuxcat{B}}}{\DeuxFoncUnOp{v}}{c},
$$
et donc un \DeuxFoncteurLax{}
$$
\DeuxFoncUnOp{(\DeuxFoncTrancheLaxCoq{\DeuxFoncUnOp{u}}{\DeuxTransUnOp{\sigma}}{c})} :
\DeuxCatUnOp{(\TrancheLax{\DeuxCatUnOp{\mathdeuxcat{A}}}{\DeuxFoncUnOp{w}}{c})} \to \DeuxCatUnOp{(\TrancheLax{\DeuxCatUnOp{\mathdeuxcat{B}}}{\DeuxFoncUnOp{v}}{c})},
$$
c'est-à-dire un \DeuxFoncteurLax{} de $\OpTrancheLax{\mathdeuxcat{A}}{w}{c}$ vers $\OpTrancheLax{\mathdeuxcat{B}}{v}{c}$, \DeuxFoncteurLax{} que l'on notera $\DeuxFoncOpTrancheLaxCoq{u}{\sigma}{c}$\index[not]{c0lu@$\DeuxFoncOpTrancheLaxCoq{u}{\sigma}{c}$}. Autrement dit, on a par définition
$$
\DeuxFoncOpTrancheLaxCoq{u}{\sigma}{c} = \DeuxFoncUnOp{(\DeuxFoncTrancheLaxCoq{\DeuxFoncUnOp{u}}{\DeuxTransUnOp{\sigma}}{c})}.
$$

Si $vu = w$ et $\sigma = 1_{w}$, on notera ce \DeuxFoncteurLax{} induit $\DeuxFoncOpTrancheLax{u}{c}$\index[not]{clu@$\DeuxFoncOpTrancheLax{u}{c}$} plutôt que $\DeuxFoncOpTrancheLaxCoq{u}{1_{w}}{c}$. 
\end{paragr}

\begin{rem}
Si, dans la situation ci-dessus, $u$ est un \DeuxFoncteurStrict, alors $\DeuxFoncOpTrancheLaxCoq{u}{\sigma}{c}$ est un \DeuxFoncteurStrict, même si $v$ ne l'est pas.
\end{rem}

\begin{paragr}\label{DefDeuxFoncInduitFoncCoLaxTrans}
Soit un diagramme de \DeuxFoncteursCoLax{}
$$
\xymatrix{
\mathdeuxcat{A} 
\ar[rr]^{u}
\ar[dr]_{w}
&&\mathdeuxcat{B}
\ar[dl]^{v}
\dtwocell<\omit>{<7.3>{\sigma}}
\\
&\mathdeuxcat{C}
&{}
}
$$  
commutatif à la \DeuxTransformationLax{} $\sigma : vu \Rightarrow w$ près seulement. Cela fournit un diagramme de \DeuxFoncteursLax{} 
$$
\xymatrix{
\DeuxCatDeuxOp{\mathdeuxcat{A}}
\ar[rr]^{\DeuxFoncDeuxOp{u}}
\ar[dr]_{\DeuxFoncDeuxOp{w}}
&&\DeuxFoncDeuxOp{\mathdeuxcat{B}}
\ar[dl]^{\DeuxFoncDeuxOp{v}}
\dtwocell<\omit>{<9>\DeuxTransDeuxOp{\sigma}}
\\
&\DeuxFoncDeuxOp{\mathdeuxcat{C}}
&{}
}
$$  
commutatif à l'\DeuxTransformationCoLax{} $\DeuxTransDeuxOp{\sigma} : \DeuxFoncDeuxOp{v} \DeuxFoncDeuxOp{u} \Rightarrow \DeuxFoncDeuxOp{w}$ près seulement. En vertu du paragraphe \ref{DefDeuxFoncInduitFoncLaxOpTrans}, on a donc un \DeuxFoncteurLax{} 
$$
\DeuxFoncTrancheLaxCoq{\DeuxFoncDeuxOp{u}}{\DeuxTransDeuxOp{\sigma}}{c} :
\TrancheLax{\DeuxCatDeuxOp{\mathdeuxcat{A}}}{\DeuxFoncDeuxOp{w}}{c} \to\TrancheLax{\DeuxCatDeuxOp{\mathdeuxcat{B}}}{\DeuxFoncDeuxOp{v}}{c},
$$
et donc un \DeuxFoncteurCoLax{}
$$
\DeuxFoncDeuxOp{(\DeuxFoncTrancheLaxCoq{\DeuxFoncDeuxOp{u}}{\DeuxTransDeuxOp{\sigma}}{c})} :
\DeuxCatDeuxOp{(\TrancheLax{\DeuxCatDeuxOp{\mathdeuxcat{A}}}{\DeuxFoncDeuxOp{w}}{c})} \to \DeuxCatDeuxOp{(\TrancheLax{\DeuxCatDeuxOp{\mathdeuxcat{B}}}{\DeuxFoncDeuxOp{v}}{c})},
$$
c'est-à-dire un \DeuxFoncteurCoLax{} de $\TrancheCoLax{\mathdeuxcat{A}}{w}{c}$ vers $\TrancheCoLax{\mathdeuxcat{B}}{v}{c}$, \DeuxFoncteurCoLax{} que l'on notera $\DeuxFoncTrancheCoLaxCoq{u}{\sigma}{c}$\index[not]{u0c@$\DeuxFoncTrancheCoLaxCoq{u}{\sigma}{c}$}. Autrement dit, on a par définition
$$
\DeuxFoncTrancheCoLaxCoq{u}{\sigma}{c} = \DeuxFoncDeuxOp{(\DeuxFoncTrancheLaxCoq{\DeuxFoncDeuxOp{u}}{\DeuxTransDeuxOp{\sigma}}{c})}.
$$

Si $vu = w$ et $\sigma = 1_{w}$, on notera ce \DeuxFoncteurCoLax{} induit $\DeuxFoncTrancheCoLax{u}{c}$\index[not]{ucc@$\DeuxFoncTrancheCoLax{u}{c}$} plutôt que $\DeuxFoncTrancheCoLaxCoq{u}{1_{w}}{c}$. 
\end{paragr}

\begin{rem}
Si, dans la situation ci-dessus, $u$ est un \DeuxFoncteurStrict, alors $\DeuxFoncTrancheCoLaxCoq{u}{\sigma}{c}$ est un \DeuxFoncteurStrict, même si $v$ ne l'est pas.
\end{rem}

\begin{paragr}\label{DefDeuxFoncInduitFoncCoLaxOpTrans}
Soit un diagramme de \DeuxFoncteursCoLax{}
$$
\xymatrix{
\mathdeuxcat{A} 
\ar[rr]^{u}
\ar[dr]_{w}
&{}
&\mathdeuxcat{B}
\ar[dl]^{v}
\\
&\mathdeuxcat{C}
\utwocell<\omit>{\sigma}
}
$$  
commutatif à l'\DeuxTransformationCoLax{} $\sigma : w \Rightarrow vu$ près seulement. Cela fournit un diagramme de \DeuxFoncteursLax{} 
$$
\xymatrix{
\DeuxCatToutOp{\mathdeuxcat{A}}
\ar[rr]^{\DeuxFoncToutOp{u}}
\ar[dr]_{\DeuxFoncToutOp{w}}
&&\DeuxFoncToutOp{\mathdeuxcat{B}}
\ar[dl]^{\DeuxFoncToutOp{v}}
\dtwocell<\omit>{<9>\DeuxTransToutOp{\sigma}}
\\
&\DeuxFoncToutOp{\mathdeuxcat{C}}
&{}
}
$$
commutatif à la \DeuxTransformationLax{} $\DeuxTransToutOp{\sigma} : \DeuxFoncToutOp{w} \Rightarrow\DeuxFoncToutOp{v} \DeuxFoncToutOp{u}$ près seulement. En vertu du paragraphe \ref{DefDeuxFoncInduitFoncLaxOpTrans}, on a donc un \DeuxFoncteurLax{} 
$$
\DeuxCatUnOp{(\DeuxFoncTrancheLaxCoq{\DeuxFoncToutOp{u}}{\DeuxTransToutOp{\sigma}}{c})} :
\DeuxCatUnOp{(\TrancheLax{\DeuxCatToutOp{\mathdeuxcat{A}}}{\DeuxFoncToutOp{w}}{c})} \to\DeuxCatUnOp{(\TrancheLax{\DeuxCatToutOp{\mathdeuxcat{B}}}{\DeuxFoncToutOp{v}}{c})},
$$
et donc un \DeuxFoncteurCoLax{}
$$
\DeuxFoncToutOp{(\DeuxFoncTrancheLaxCoq{\DeuxFoncToutOp{u}}{\DeuxTransToutOp{\sigma}}{c})} :
\DeuxCatToutOp{(\TrancheLax{\DeuxCatToutOp{\mathdeuxcat{A}}}{\DeuxFoncToutOp{w}}{c})} \to \DeuxCatToutOp{(\TrancheLax{\DeuxCatToutOp{\mathdeuxcat{B}}}{\DeuxFoncToutOp{v}}{c})},
$$
c'est-à-dire un \DeuxFoncteurCoLax{} de $\OpTrancheCoLax{\mathdeuxcat{A}}{w}{c}$ vers $\OpTrancheCoLax{\mathdeuxcat{B}}{v}{c}$, \DeuxFoncteurCoLax{} que l'on notera $\DeuxFoncOpTrancheCoLaxCoq{u}{\sigma}{c}$\index[not]{c0cu@$\DeuxFoncOpTrancheCoLaxCoq{u}{\sigma}{c}$}. On a par définition
$$
\DeuxFoncOpTrancheCoLaxCoq{u}{\sigma}{c} = \DeuxFoncToutOp{(\DeuxFoncTrancheLaxCoq{\DeuxFoncToutOp{u}}{\DeuxTransToutOp{\sigma}}{c})}.
$$

Si $vu = w$ et $\sigma = 1_{w}$, on notera ce \DeuxFoncteurCoLax{} induit $\DeuxFoncOpTrancheCoLax{u}{c}$\index[not]{ccu@$\DeuxFoncOpTrancheCoLax{u}{c}$} plutôt que $\DeuxFoncOpTrancheCoLaxCoq{u}{1_{w}}{c}$. 
\end{paragr}

\begin{rem}
Si, dans la situation ci-dessus, $u$ est un \DeuxFoncteurStrict, alors $\DeuxFoncOpTrancheCoLaxCoq{u}{\sigma}{c}$ est un \DeuxFoncteurStrict, même si $v$ ne l'est pas.
\end{rem}

\begin{paragr}\label{DefFoncInduits}
Résumons les propriétés de dualité des divers \deux{}foncteurs lax ou colax définis dans cette section. 

Si $u : \mathdeuxcat{A} \to \mathdeuxcat{B}$, $v : \mathdeuxcat{B} \to \mathdeuxcat{C}$ et $w : \mathdeuxcat{A} \to \mathdeuxcat{C}$ sont des \DeuxFoncteursLax{} et si $\sigma$ est une \DeuxTransformationCoLax{} de $vu$ vers $w$, on a défini un \DeuxFoncteurLax{}
$$
\DeuxFoncTrancheLaxCoq{u}{\sigma}{c} : \TrancheLax{\mathdeuxcat{A}}{w}{c} \to \TrancheLax{\mathdeuxcat{B}}{v}{c}
$$
vérifiant les propriétés de dualité suivantes : 
$$
\DeuxFoncTrancheLaxCoq{u}{\sigma}{c} =
\DeuxCatUnOp{(\DeuxFoncOpTrancheLaxCoq{\DeuxFoncUnOp{u}}{\DeuxTransUnOp{\sigma}}{c})} =
\DeuxCatDeuxOp{(\DeuxFoncTrancheCoLaxCoq{\DeuxFoncDeuxOp{u}}{\DeuxTransDeuxOp{\sigma}}{c})} =
 \DeuxCatToutOp{(\DeuxFoncOpTrancheCoLaxCoq{\DeuxFoncToutOp{u}}{\DeuxTransToutOp{\sigma}}{c})}.
$$

Si $u : \mathdeuxcat{A} \to \mathdeuxcat{B}$, $v : \mathdeuxcat{B} \to \mathdeuxcat{C}$ et $w : \mathdeuxcat{A} \to \mathdeuxcat{C}$ sont des \DeuxFoncteursLax{} et si $\sigma$ est une \DeuxTransformationLax{} de $w$ vers $vu$, on a défini un \DeuxFoncteurLax{}
$$
\DeuxFoncOpTrancheLaxCoq{u}{\sigma}{c} : \OpTrancheLax{\mathdeuxcat{A}}{w}{c} \to \OpTrancheLax{\mathdeuxcat{B}}{v}{c}
$$
vérifiant les propriétés de dualité suivantes :
$$
\DeuxFoncOpTrancheLaxCoq{u}{\sigma}{c} = 
\DeuxFoncUnOp{(\DeuxFoncTrancheLaxCoq{\DeuxFoncUnOp{u}}{\DeuxTransUnOp{\sigma}}{c})} = \DeuxFoncDeuxOp{(\DeuxFoncOpTrancheCoLaxCoq{\DeuxFoncDeuxOp{u}}{\DeuxTransDeuxOp{\sigma}}{c})} = \DeuxFoncToutOp{(\DeuxFoncTrancheCoLaxCoq{\DeuxFoncToutOp{u}}{\DeuxTransToutOp{\sigma}}{c})}.
$$

Si $u : \mathdeuxcat{A} \to \mathdeuxcat{B}$, $v : \mathdeuxcat{B} \to \mathdeuxcat{C}$ et $w : \mathdeuxcat{A} \to \mathdeuxcat{C}$ sont des \DeuxFoncteursCoLax{} et si $\sigma$ est une \DeuxTransformationLax{} de $vu$ vers $w$, on a défini un \DeuxFoncteurCoLax{}
$$
\DeuxFoncTrancheCoLaxCoq{u}{\sigma}{c} : \TrancheCoLax{\mathdeuxcat{A}}{w}{c} \to \TrancheCoLax{\mathdeuxcat{B}}{v}{c}
$$
vérifiant les propriétés de dualité suivantes :
$$
\DeuxFoncTrancheCoLaxCoq{u}{\sigma}{c} = 
\DeuxFoncUnOp{(\DeuxFoncOpTrancheCoLaxCoq{\DeuxFoncUnOp{u}}{\DeuxTransUnOp{\sigma}}{c})} =
\DeuxFoncDeuxOp{(\DeuxFoncTrancheLaxCoq{\DeuxFoncDeuxOp{u}}{\DeuxTransDeuxOp{\sigma}}{c})} = 
\DeuxFoncToutOp{(\DeuxFoncOpTrancheLaxCoq{\DeuxFoncToutOp{u}}{\DeuxTransToutOp{\sigma}}{c})}.
$$

Si $u : \mathdeuxcat{A} \to \mathdeuxcat{B}$, $v : \mathdeuxcat{B} \to \mathdeuxcat{C}$ et $w : \mathdeuxcat{A} \to \mathdeuxcat{C}$ sont des \DeuxFoncteursCoLax{} et si $\sigma$ est une \DeuxTransformationCoLax{} de $w$ vers $vu$, on a défini un \DeuxFoncteurCoLax{}
$$
\DeuxFoncOpTrancheCoLaxCoq{u}{\sigma}{c} : \OpTrancheCoLax{\mathdeuxcat{A}}{w}{c} \to \OpTrancheCoLax{\mathdeuxcat{B}}{v}{c}
$$
vérifiant les propriétés de dualité suivantes : 
$$
\DeuxFoncOpTrancheCoLaxCoq{u}{\sigma}{c} = 
\DeuxFoncUnOp{(\DeuxFoncTrancheCoLaxCoq{\DeuxFoncUnOp{u}}{\DeuxTransUnOp{\sigma}}{c})} =
\DeuxFoncDeuxOp{(\DeuxFoncOpTrancheLaxCoq{\DeuxFoncDeuxOp{u}}{\DeuxTransDeuxOp{\sigma}}{c})} =
\DeuxFoncToutOp{(\DeuxFoncTrancheLaxCoq{\DeuxFoncToutOp{u}}{\DeuxTransToutOp{\sigma}}{c})}.
$$
\end{paragr}

\section{Préadjoints}\label{SectionPreadjoints}

\begin{paragr}
Nous présentons dans cette section une généralisation \deux{}catégorique de la notion de foncteur adjoint. Après le dépôt de ce travail, sur une suggestion de Steve Lack, nous avons consulté l'article \cite{BettiPower}, ce qui nous a permis de constater que les notions et résultats de cette section, contrairement à ce que nous croyions, se trouvaient déjà dans la littérature, bien qu'aucune réponse aux requêtes d'ordre bibliographique que nous avons émises lors de la rédaction de ce travail de thèse n'ait attiré notre attention sur \cite{BettiPower}. En guise d'addenda relativement tardif, nous dirigeons à plusieurs reprises le lecteur sur ce dernier texte dans la présente section. Toujours suivant une remarque de Steve Lack, nous adopterons une attitude analogue dans la section \ref{SectionAdjonctions}, pour laquelle il faut citer la thèse de Verity \cite{Verity}, que nous ignorions jusqu'ici. 
\end{paragr}

\begin{paragr}

On rappelle une caractérisation classique des foncteurs adjoints : un morphisme $u : A \to B$ de $\Cat$ est un adjoint à gauche — autrement dit, il admet un adjoint à droite — si et seulement si, pour tout objet $b$ de $B$, la catégorie $A/b$ admet un objet final. De façon duale, le foncteur $u : A \to B$ est un adjoint à droite — autrement dit, il admet un adjoint à gauche —, si et seulement si, pour tout objet $b$ de $B$, la catégorie $b \backslash A$ admet un objet initial. Ces rappels catégoriques nous serviront de guide pour énoncer les définitions qui vont suivre. 
\end{paragr}

\begin{df}\label{DefOF2}
On dira qu'un objet $z$ d'une \deux{}catégorie $\mathdeuxcat{A}$ \emph{admet un objet final}\index{admet un objet final (d'un objet d'une \deux{}catégorie)} si, pour tout objet $a$ de $\mathdeuxcat{A}$, la catégorie $\CatHom{\mathdeuxcat{A}}{a}{z}$ admet un objet final. On dira qu'il \emph{admet un objet initial}\index{admet un objet initial (d'un objet d'une \deux{}catégorie)} s'il admet un objet final dans $\DeuxCatDeuxOp{\mathdeuxcat{A}}$, autrement dit si, pour tout objet $a$ de $\mathdeuxcat{A}$, la catégorie $\CatHom{\mathdeuxcat{A}}{a}{z}$ admet un objet initial. 
\end{df}

\begin{rem}\label{RemarqueJay}
Nous devons à la consultation de \cite{BettiPower} d'avoir appris que la notion d'objet admettant un objet final se trouvait déjà mise en valeur par Jay, dans \cite{Jay}, sous le nom d'\emph{objet final local}, dans un contexte bicatégorique. 
\end{rem}

\begin{exemple}
Un objet de la \deux{}catégorie $\Cat$ admet un objet final (\emph{resp.} initial) en ce sens si et seulement si c'est une catégorie admettant un objet final (\emph{resp.} initial) au sens usuel ; c'est la raison de l'adoption de cette terminologie, suggérée par Jean Bénabou. 
\end{exemple}

\begin{paragr}\label{DefOpAdmet}
Pour des raisons de commodité, l'on dira qu'une \deux{}catégorie $\mathdeuxcat{A}$ \emph{op\nobreakdash-admet}\index{op-admet} (\emph{resp.} \emph{co\nobreakdash-admet}\index{co-admet}, \emph{resp.} \emph{coop\nobreakdash-admet}\index{coop-admet}) une certaine propriété si $\DeuxCatUnOp{\mathdeuxcat{A}}$ (\emph{resp.} $\DeuxCatDeuxOp{\mathdeuxcat{A}}$, \emph{resp.} $\DeuxCatToutOp{\mathdeuxcat{A}}$) vérifie cette propriété. 
\end{paragr}

\begin{exemple}\label{ExemplesOF}
Pour toute \deux{}catégorie $\mathdeuxcat{A}$ et tout objet $a$ de $\mathdeuxcat{A}$, la \deux{}catégorie $\TrancheCoLax{\mathdeuxcat{A}}{}{a}$ admet un objet admettant un objet final. Plus précisément, l'objet $(a, 1_{a})$ admet un objet final : il est tel que, pour tout objet $(a', p : a' \to a)$, le couple $(p, 1_{p})$ définit un objet final de la catégorie $\CatHom{\TrancheCoLax{\mathdeuxcat{A}}{}{a}}{(a',p)}{(a,1_{a})}$. Dualement, la \deux{}catégorie $\TrancheLax{\mathdeuxcat{A}}{}{a}$ admet un objet admettant un objet initial, la \deux{}catégorie $\OpTrancheCoLax{\mathdeuxcat{A}}{}{a}$ op-admet un objet admettant un objet final et la \deux{}catégorie $\OpTrancheLax{\mathdeuxcat{A}}{}{a}$ op-admet un objet admettant un objet initial. 
\end{exemple}

\begin{df}\label{DefMAdjoints}
On dira qu'un \DeuxFoncteurCoLax{} $u : \mathdeuxcat{A} \to \mathdeuxcat{B}$ est un \emph{préadjoint à gauche colax}\index{préadjoint à gauche colax} si, pour tout objet $b$ de $\mathdeuxcat{B}$, la \deux{}catégorie $\TrancheCoLax{\mathdeuxcat{A}}{u}{b}$ admet un objet admettant un objet final.  

On dira qu'un \DeuxFoncteurLax{} $u : \mathdeuxcat{A} \to \mathdeuxcat{B}$ est un \emph{préadjoint à gauche lax}\index{préadjoint à gauche lax} si $\DeuxFoncDeuxOp{u}$ est un préadjoint à gauche colax. Cette condition équivaut à la suivante : pour tout objet $b$ de $\mathdeuxcat{B}$, la \deux{}catégorie $\TrancheLax{\mathdeuxcat{A}}{u}{b}$ admet un objet admettant un objet initial. 

On dira qu'un \DeuxFoncteurCoLax{} $u : \mathdeuxcat{A} \to \mathdeuxcat{B}$ est un \emph{préadjoint à droite colax}\index{préadjoint à droite colax} si $\DeuxFoncUnOp{u}$ est un préadjoint à gauche colax. Cette condition équivaut à la suivante : pour tout objet $b$ de $\mathdeuxcat{B}$, la \deux{}catégorie $\OpTrancheCoLax{\mathdeuxcat{A}}{u}{b}$ op-admet un objet admettant un objet final. 

On dira qu'un \DeuxFoncteurLax{} $u : \mathdeuxcat{A} \to \mathdeuxcat{B}$ est un \emph{préadjoint à droite lax}\index{préadjoint à droite lax} si $\DeuxFoncToutOp{u}$ est un préadjoint à gauche colax. Cette condition équivaut à la suivante : pour tout objet $b$ de $\mathdeuxcat{B}$, la \deux{}catégorie $\OpTrancheLax{\mathdeuxcat{A}}{u}{b}$ op-admet un objet admettant un objet initial.
\end{df}

\begin{rem}
On trouve, à la page 941 de \cite{BettiPower}, l'exacte définition de ce que nous appelons « préadjoint à droite lax », sous le nom de \DeuxFoncteurLax{}  \emph{induisant un adjoint à gauche local}. Comme, suivant les auteurs eux-mêmes, la définition que nous donnons équivaut à \cite[définition 4.1]{BettiPower}, cette dernière définition fournit une caractérisation des morphismes vérifiant la propriété universelle que nous privilégions. Toujours à la lecture de \cite{BettiPower}, nous avons appris que cette propriété universelle avait été dégagée par Bunge comme une généralisation de la notion d'extension de Kan \cite[p. 357]{Bunge}. Soulignons toutefois que \cite[définition 4.1]{BettiPower} n'est pas la définition de ce que les auteurs de \cite{BettiPower} appellent \emph{adjoint local}, notion dont la définition est \cite[définition 3.1]{BettiPower} et qui présente l'avantage d'être stable par composition. 
\end{rem} 

\begin{paragr}\label{RappelEqui1}
On rappelle qu'un foncteur $u : A \to B$ est une \emph{équivalence de catégories} s'il existe un foncteur $v : B \to A$ et des isomorphismes de foncteurs $1_{A} \simeq vu$ et $1_{B} \simeq uv$. En vertu de la théorie classique (voir par exemple \cite[IV. 4. théorème 1, p. 93]{CWM}), cette condition équivaut à celle, \emph{a priori} plus forte, que $u$ admet un adjoint à droite, que nous noterons également $v$, et qu'il existe des isomorphismes de foncteurs $\eta : 1_{A} \Rightarrow vu$ et $\epsilon : uv \Rightarrow 1_{B}$ qui sont respectivement unité et coünité pour le couple de foncteurs adjoints $(u, v)$, c'est-à-dire que $\eta$ et $\epsilon$ vérifient les identités triangulaires (voir ci-dessous). On rappelle de plus que cette notion d'équivalence fait sens dans toute \deux{}catégorie, le cas ci-dessus correspondant à celui de la \deux{}catégorie $\Cat$. Plus précisément, une \un{}cellule $u : a \to a'$ dans une \deux{}catégorie $\mathdeuxcat{A}$ est une \emph{équivalence}\index{equivalence (dans une \deux{}catégorie)@équivalence (dans une \deux{}catégorie)} s'il existe une \un{}cellule $v : a' \to a$ et des \deux{}cellules inversibles $\eta : 1_{a} \Rightarrow vu$ et $\epsilon : uv \Rightarrow 1_{a'}$. On peut vérifier que, comme ci-dessus, cela équivaut à demander l'existence de telles \deux{}cellules inversibles satisfaisant en plus les identités triangulaires, c'est-à-dire telles que les égalités
$$
(\epsilon \CompDeuxZero u) \CompDeuxUn (u \CompDeuxZero \eta) = 1_{u}
$$
et
$$
(v \CompDeuxZero \epsilon) \CompDeuxUn (\eta \CompDeuxZero v) = 1_{v}
$$
soient vérifiées. Une référence possible est \cite[théorème 1.9]{Gurski} (voir aussi \cite[remarque 1.10]{Gurski} pour un résultat plus fort).
\end{paragr}

\begin{paragr}\label{RappelEqui2}
On rappelle enfin qu'un \DeuxFoncteurStrict{} $u : \mathdeuxcat{A} \to \mathdeuxcat{B}$ est une \emph{équivalence de \deux{}catégories}\index{equivalence (de \deux{}catégories)@équivalence (de \deux{}catégories)} si, pour tout couple d'objets $a$ et $a'$ de $\mathdeuxcat{A}$, le foncteur $u_{a,a'} : \CatHom{\mathdeuxcat{A}}{a}{a'} \to \CatHom{\mathdeuxcat{B}}{u(a)}{u(a')}$ est une équivalence de catégories, dont on notera $\overline{u}_{a,a'}$ un quasi-inverse, et si, pour tout objet $b$ de $\mathdeuxcat{B}$, il existe un objet $a_{b}$ de $\mathdeuxcat{A}$ tel que $u(a_{b})$ et $b$ soient équivalents dans $\mathdeuxcat{B}$. De façon plus explicite, cela signifie qu'il existe des \un{}cellules $p_{b} : u(a_{b}) \to b$ et $q_{b} : b \to u(a_{b})$ telles que les \un{}cellules composées $q_{b} p_{b}$ et $p_{b} q_{b}$ soient isomorphes à $1_{u(a_{b})}$ et $1_{b}$ respectivement. Cette dernière condition signifie qu'il existe des \deux{}cellules inversibles $\eta_{b} : 1_{u(a_{b})} \Rightarrow q_{b} p_{b}$ et $\epsilon_{b} : p_{b} q_{b} \Rightarrow 1_{b}$. En vertu des résultats rappelés ci-dessus, on peut supposer en outre que ces \deux{}cellules vérifient les identités triangulaires, c'est-à-dire que les égalités
$$
(\epsilon_{b} \CompDeuxZero p_{b}) \CompDeuxUn (p_{b} \CompDeuxZero \eta_{b}) = 1_{p_{b}}
$$
et
$$
(q_{b} \CompDeuxZero \epsilon_{b}) \CompDeuxUn (\eta_{b} \CompDeuxZero q_{b}) = 1_{q_{b}}
$$
sont vérifiées. La composée horizontale de \deux{}cellules inversibles étant une \deux{}cellule inversible \footnote{On rappelle souvent confondre dans les notations les \un{}cellules avec leur identité. Les identités des \un{}cellules sont évidemment des \deux{}cellules inversibles.}, cela s'écrit aussi
$$
\epsilon_{b} \CompDeuxZero p_{b} = (p_{b} \CompDeuxZero \eta_{b})^{-1} = p_{b} \CompDeuxZero \eta_{b}^{-1}
$$
et
$$
\eta_{b} \CompDeuxZero q_{b} = (q_{b} \CompDeuxZero \epsilon_{b})^{-1} = q_{b} \CompDeuxZero \epsilon_{b}^{-1}.
$$
Dans la suite, nous adopterons des notations consistantes avec celles que nous venons d'utiliser, sans forcément en expliciter le sens, qui devrait être clair.  
\end{paragr}

\begin{prop}\label{EquiPreadjoint}
Une équivalence de \deux{}catégories est un préadjoint à gauche lax (\emph{resp.} un préadjoint à gauche colax, \emph{resp.} un préadjoint à droite lax, \emph{resp.} un préadjoint à droite colax).
\end{prop}

\begin{proof}
Il suffit bien sûr de vérifier l'une des quatre assertions, duales l'une de l'autre. Montrons par exemple qu'une équivalence de \deux{}catégories $u : \mathdeuxcat{A} \to \mathdeuxcat{B}$ est un préadjoint à gauche colax. On utilisera sans explications supplémentaires les notations du paragraphe \ref{RappelEqui2}. Soit $b$ un objet de $\mathdeuxcat{B}$. Il s'agit de vérifier que la \deux{}catégorie $\TrancheCoLax{\mathdeuxcat{A}}{u}{b}$ admet un objet admettant un objet final. 

Le couple $(a_{b}, p_{b} : u(a_{b}) \to b)$ définit un objet de $\TrancheCoLax{\mathdeuxcat{A}}{u}{b}$. On va vérifier qu'il admet un objet final. 

Soit $(a, p : u(a) \to b)$ un objet quelconque de $\TrancheCoLax{\mathdeuxcat{A}}{u}{b}$. Le couple 
$$
(\overline{u}_{a, a_{b}} (q_{b} p), (\epsilon_{b} \CompDeuxZero p) \CompDeuxUn (p_{b} \CompDeuxZero (\epsilon_{a, a_{b}})_{q_{b} p}))
$$
définit une \un{}cellule de $(a,p)$ vers $(a_{b}, p_{b})$ dans $\TrancheCoLax{\mathdeuxcat{A}}{u}{b}$. 

Soit $(s : a \to a_{b}, \sigma : p_{b} u(s) \Rightarrow p)$ une \un{}cellule quelconque de $(a,p)$ vers $(a_{b}, p_{b})$ dans $\TrancheCoLax{\mathdeuxcat{A}}{u}{b}$. La \deux{}cellule 
$$
(\epsilon^{-1}_{a, a_{b}})_{q_{b} p} \CompDeuxUn (q_{b} \CompDeuxZero \sigma) \CompDeuxUn (\eta_{b} \CompDeuxZero u(s))
$$ 
est un morphisme de la catégorie 
$$
\CatHom{\mathdeuxcat{B}}{u(a)}{u(a_{b})}.
$$
Une équivalence de catégories étant pleinement fidèle, il existe une unique \deux{}cellule 
$$
\delta : s \Rightarrow \overline{u}_{a, a_{b}} (q_{b} p)
$$
dans $\mathdeuxcat{A}$ telle que 
$$
u(\delta) = (\epsilon^{-1}_{a, a_{b}})_{q_{b} p} \CompDeuxUn (q_{b} \CompDeuxZero \sigma) \CompDeuxUn (\eta_{b} \CompDeuxZero u(s)).
$$
On veut vérifier l'égalité
$$
(\epsilon_{b} \CompDeuxZero p) \CompDeuxUn (p_{b} \CompDeuxZero (\epsilon_{a, a_{b}})_{q_{b} p}) \CompDeuxUn (p_{b} \CompDeuxZero ((\epsilon^{-1}_{a, a_{b}})_{q_{b} p} \CompDeuxUn (q_{b} \CompDeuxZero \sigma) \CompDeuxUn (\eta_{b} \CompDeuxZero u(s)))) = \sigma.
$$
Cela résulte de la suite d'égalités 
$$
\begin{aligned}
&(\epsilon_{b} \CompDeuxZero p) \CompDeuxUn (p_{b} \CompDeuxZero (\epsilon_{a, a_{b}})_{q_{b} p}) \CompDeuxUn (p_{b} \CompDeuxZero ((\epsilon^{-1}_{a, a_{b}})_{q_{b} p} \CompDeuxUn (q_{b} \CompDeuxZero \sigma) \CompDeuxUn (\eta_{b} \CompDeuxZero u(s)))) 
\\
&= (\epsilon_{b} \CompDeuxZero p) \CompDeuxUn (p_{b} \CompDeuxZero (\epsilon_{a, a_{b}})_{q_{b} p}) \CompDeuxUn (p_{b} \CompDeuxZero (\epsilon^{-1}_{a, a_{b}})_{q_{b} p}) \CompDeuxUn (p_{b} q_{b} \CompDeuxZero \sigma) \CompDeuxUn (p_{b} \CompDeuxZero \eta_{b} \CompDeuxZero u(s))
\\
&= (\epsilon_{b} \CompDeuxZero p) \CompDeuxUn (p_{b} q_{b} \CompDeuxZero \sigma) \CompDeuxUn (p_{b} \CompDeuxZero \eta_{b}^{-1} \CompDeuxZero u(s))
\\
&= (\epsilon_{b} \CompDeuxZero \sigma) \CompDeuxUn (p_{b} \CompDeuxZero \eta_{b} \CompDeuxZero u(s))
\\
&= \sigma \CompDeuxUn (\epsilon_{b} \CompDeuxZero p_{b} \CompDeuxZero \eta_{b} \CompDeuxZero u(s))
\\
&= \sigma \CompDeuxUn (p_{b} \CompDeuxZero \eta_{b}^{-1} \CompDeuxZero \eta_{b} \CompDeuxZero u(s))
\\
&= \sigma \CompDeuxUn (p_{b} u(s))
\\
&= \sigma.
\end{aligned}
$$
Ainsi, $\delta : s \Rightarrow \overline{u}_{a, a_{b}} (q_{b} p)$ définit une \deux{}cellule de $(s, \sigma)$ vers $(\overline{u}_{a, a_{b}} (q_{b} p), (\epsilon_{b} \CompDeuxZero p) \CompDeuxUn (p_{b} \CompDeuxZero (\epsilon_{a, a_{b}})_{q_{b} p}))$. Il reste à vérifier que c'est la seule. Soit donc $\tau : s \Rightarrow \overline{u}_{a, a_{b}} (q_{b} p)$ définissant une \deux{}cellule de $(s, \sigma)$ vers $(\overline{u}_{a, a_{b}} (q_{b} p), (\epsilon_{b} \CompDeuxZero p) \CompDeuxUn (p_{b} \CompDeuxZero (\epsilon_{a, a_{b}})_{q_{b} p}))$, c'est-à-dire telle que
$$
(\epsilon_{b} \CompDeuxZero p) \CompDeuxUn (p_{b} \CompDeuxZero (\epsilon_{a, a_{b}})_{q_{b}p}) \CompDeuxUn (p_{b} \CompDeuxZero u(\tau)) = \sigma.
$$ 
Il s'agit de vérifier l'égalité
$$
u(\tau) = (\epsilon^{-1}_{a, a_{b}})_{q_{b} p} \CompDeuxUn (q_{b} \CompDeuxZero \sigma) \CompDeuxUn (\eta_{b} \CompDeuxZero u(s)),
$$
qui nous permettra de conclure $\tau = \delta$ par définition de $\delta$. Or,
$$
\begin{aligned}
(\epsilon_{b} \CompDeuxZero p) \CompDeuxUn (p_{b} \CompDeuxZero (\epsilon_{a, a_{b}})_{q_{b}p}) \CompDeuxUn (p_{b} \CompDeuxZero u(\tau)) &= \sigma
\\
&= (\epsilon_{b} \CompDeuxZero p) \CompDeuxUn (p_{b} \CompDeuxZero (\epsilon_{a, a_{b}})_{q_{b} p}) \CompDeuxUn (p_{b} \CompDeuxZero ((\epsilon^{-1}_{a, a_{b}})_{q_{b} p} \CompDeuxUn (q_{b} \CompDeuxZero \sigma) \CompDeuxUn (\eta_{b} \CompDeuxZero u(s)))). 
\end{aligned}
$$
Les \deux{}cellules $(\epsilon_{b} \CompDeuxZero p)$ et $(p_{b} \CompDeuxZero (\epsilon_{a, a_{b}})_{q_{b} p})$ sont inversibles puisqu'il s'agit de composées de \deux{}cellules inversibles (voir la fin du paragraphe \ref{RappelEqui2}). Il vient donc
$$
p_{b} \CompDeuxZero u(\tau) = p_{b} \CompDeuxZero ((\epsilon^{-1}_{a, a_{b}})_{q_{b} p} \CompDeuxUn (q_{b} \CompDeuxZero \sigma) \CompDeuxUn (\eta_{b} \CompDeuxZero u(s))),
$$
d'où
$$
q_{b} p_{b} \CompDeuxZero u(\tau) = q_{b} p_{b} \CompDeuxZero ((\epsilon^{-1}_{a, a_{b}})_{q_{b} p} \CompDeuxUn (q_{b} \CompDeuxZero \sigma) \CompDeuxUn (\eta_{b} \CompDeuxZero u(s))).
$$
Comme $q_{b} p_{b}$ est isomorphe à $1_{u(a_{b})}$, il vient
$$
u(\tau) = (\epsilon^{-1}_{a, a_{b}})_{q_{b} p} \CompDeuxUn (q_{b} \CompDeuxZero \sigma) \CompDeuxUn (\eta_{b} \CompDeuxZero u(s)),
$$
CQFD. 
\end{proof}

\begin{prop}\label{PreAdjointsCompo}
La composée de deux préadjoints à gauche colax (\emph{resp.} préadjoints à gauche lax, \emph{resp.} préadjoints à droite colax, \emph{resp.} préadjoints à droite lax) qui sont des \DeuxFoncteursStricts{} est un préadjoint à gauche colax (\emph{resp.} préadjoint à gauche lax, \emph{resp.} préadjoint à droite colax, \emph{resp.} préadjoint à droite lax). 
\end{prop}

\begin{proof} 
Soient $u : \mathdeuxcat{A} \to \mathdeuxcat{B}$ et $v : \mathdeuxcat{B} \to \mathdeuxcat{C}$ deux préadjoints à gauche colax qui sont des \DeuxFoncteursStricts{} et $c$ un objet de $\mathdeuxcat{C}$. Il s'agit de montrer que la \deux{}catégorie 
$
\TrancheCoLax{\mathdeuxcat{A}}{vu}{c}
$
admet un objet admettant un objet final. Par hypothèse, la \deux{}catégorie 
$
\TrancheCoLax{\mathdeuxcat{B}}{v}{c}
$
vérifie cette propriété, c'est-à-dire qu'elle admet un objet 
$$
(b_{F}, q_{b_{F}} : v(b_{F}) \to c) = (b_{F}, q_{b_{F}})
$$ 
tel que, pour tout objet 
$$
(b, s : v(b) \to c)
$$ 
de la \deux{}catégorie
$
\TrancheCoLax{\mathdeuxcat{B}}{v}{c}
$,
il existe un objet final dans la catégorie 
$$
\CatHom{\TrancheCoLax{\mathdeuxcat{B}}{v}{c}}{(b,s)}{(b_{F}, q_{b_{F}})}.
$$
La \deux{}catégorie 
$
\TrancheCoLax{\mathdeuxcat{A}}{u}{b_{F}}
$ 
admet de même un objet 
$$
(a_{F}, q_{a_{F}} : u(a_{F}) \to b_{F}) = (a_{F}, q_{a_{F}})
$$ 
tel que, pour tout objet 
$$
(a, r : u(a) \to b_{F})
$$ 
de 
$
\TrancheCoLax{\mathdeuxcat{A}}{u}{b_{F}}
$, 
il existe un objet final dans la catégorie 
$$
\CatHom{\TrancheCoLax{\mathdeuxcat{A}}{u}{b_{F}}}{(a,r)}{(a_{F}, q_{a_{F}})}.
$$
Le couple 
$$
(a_{F}, q_{b_{F}} v(q_{a_{F}}) : v(u(a_{F})) \to v(b_{F}) \to c) = (a_{F}, q_{b_{F}} v(q_{a_{F}}))
$$ 
définit alors un objet de $\TrancheCoLax{\mathdeuxcat{A}}{vu}{c}$. Soit maintenant 
$$
(a, q : v(u(a)) \to c) = (a, q)
$$ 
un objet de $\TrancheCoLax{\mathdeuxcat{A}}{vu}{c}$. On se propose d'exhiber un objet final de la catégorie 
$$
\CatHom{\TrancheCoLax{\mathdeuxcat{A}}{vu}{c}}{(a,q)}{(a_{F}, q_{b_{F}} v(q_{a_{F}}))}.
$$

Le couple 
$$
(u(a), q : v(u(a)) \to c) = (u(a), q)
$$ 
définit un objet de $\TrancheCoLax{\mathdeuxcat{B}}{v}{c}$. Il existe donc un morphisme 
$$
(g : u(a) \to b_{F}, \alpha : q_{b_{F}} v(g) \Rightarrow q) = (g, \alpha)
$$ 
de $(u(a), q)$ vers $(b_{F}, q_{b_{F}})$ dans $\TrancheCoLax{\mathdeuxcat{B}}{v}{c}$ qui est un objet final de la catégorie 
$$
\CatHom{\TrancheCoLax{\mathdeuxcat{B}}{v}{c}}{(u(a),q)}{(b_{F}, q_{b_{F}})}.
$$ 
Le couple 
$
(a, g : u(a) \to b_{F})
$ 
définit un objet de la \deux{}catégorie 
$
\TrancheCoLax{\mathdeuxcat{A}}{u}{b_{F}}
$. 
Il existe donc un morphisme 
$$
(f : a \to a_{F}, \beta : q_{a_{F}} u(f) \Rightarrow g)
$$ 
de $(a,g)$ vers 
$
(a_{F}, q_{a_{F}})
$ 
dans 
$
\TrancheCoLax{\mathdeuxcat{A}}{u}{b_{F}}
$ 
qui est un objet final de la catégorie 
$$
\CatHom{\TrancheCoLax{\mathdeuxcat{A}}{u}{b_{F}}}{(a,g)}{(a_{F}, q_{a_{F}})}.
$$
On en déduit une \deux{}cellule 
$$
\alpha (q_{b_{F}} \CompDeuxZero v(\beta))
$$ 
de 
$
q_{b_{F}} v(q_{a_{F}}) v(u(f))
$ 
vers $q$ dans $\mathdeuxcat{C}$, et donc une \un{}cellule 
$$
(f, \alpha (q_{b_{F}} \CompDeuxZero v(\beta)))
$$ 
de $(a, q)$ vers $(a_{F}, q_{b_{F}} v(q_{a_{F}}))$ dans $\TrancheCoLax{\mathdeuxcat{A}}{vu}{c}$. On va montrer qu'il s'agit d'un objet final de la catégorie 
$$
\CatHom{\TrancheCoLax{\mathdeuxcat{A}}{vu}{c}}{(a,q)}{(a_{F}, q_{b_{F}} v(q_{a_{F}}))}.
$$ 

Soit donc 
$$
(f' : a \to a_{F}, \gamma : q_{b_{F}} v(q_{a_{F}}) v(u(f')) \Rightarrow q)
$$ 
une \un{}cellule de $(a,q)$ vers $(a_{F}, q_{b_{F}} v(q_{a_{F}}))$ dans $\TrancheCoLax{\mathdeuxcat{A}}{vu}{c}$. Comme $v$ est un \DeuxFoncteurStrict{}\footnote{On pourrait définir une \deux{}cellule pour peu que $v$ soit colax ; c'est plus loin qu'une obstruction se présenterait dans le cas général d'un \DeuxFoncteurCoLax.}, $\gamma$ est une \deux{}cellule de $q_{b_{F}} v(q_{a_{F}} u(f'))$ vers $q$ dans $\mathdeuxcat{C}$. Ainsi, le couple $(q_{a_{F}} u(f'), \gamma)$ définit un morphisme de $(u(a), q)$ vers $(b_{F}, q_{b_{F}})$ dans $\TrancheCoLax{\mathdeuxcat{B}}{v}{c}$. Or, par définition, $(g, \alpha)$ est l'objet final de la catégorie 
$$
\CatHom{\TrancheCoLax{\mathdeuxcat{B}}{v}{c}}{(u(a),q)}{(b_{F}, q_{b_{F}})}.
$$
Par conséquent, il existe une unique \deux{}cellule $\rho$ de $q_{a_{F}} u(f')$ vers $g$ dans $\mathdeuxcat{B}$ telle que 
$$
\alpha (q_{b_{F}} \CompDeuxZero v(\rho)) = \gamma.
$$
Considérons le couple $(f' : a \to a_{F}, \rho : q_{a_{F}} u(f') \to g)$. C'est un objet de la catégorie 
$$
\CatHom{\TrancheCoLax{\mathdeuxcat{A}}{u}{b_{F}}}{(a,g)}{(a_{F}, q_{a_{F}})},
$$ 
dont $(f, \beta)$ est un objet final. Il existe donc une unique \deux{}cellule $\tau$ de $f'$ vers $f$ dans $\mathdeuxcat{A}$ telle que
$$
\beta (q_{a_{F}} \CompDeuxZero u(\tau)) = \rho.
$$
On a alors
$$
\begin{aligned}
\alpha \CompDeuxUn (q_{b_{F}} \CompDeuxZero v(\beta)) \CompDeuxUn (q_{b_{F}} v(q_{a_{F}}) \CompDeuxZero v(u(\tau))) &= \alpha \CompDeuxUn (q_{b_{F}} \CompDeuxZero (v(\beta) \CompDeuxUn (v(q_{a_{F}}) \CompDeuxZero v(u(\tau)))))
\\
&= \alpha \CompDeuxUn (q_{b_{F}} \CompDeuxZero v(\beta \CompDeuxUn (q_{a_{F}} \CompDeuxZero u(\tau))))
\\
&= \alpha \CompDeuxUn (q_{b_{F}} \CompDeuxZero v(\rho))
\\
&= \gamma.
\end{aligned}
$$
(C'est cette dernière égalité qui ne serait pas vérifiée dans le cas général où $v$ est seulement supposé colax, la \deux{}cellule structurale $v_{q_{a_{F}}, u(f')}$ n'étant pas forcément une identité. Le lecteur notera que la démonstration n'utilise pas toutes les hypothèses portant sur $u$ et $v$, que l'on peut affaiblir.) Ainsi, $\tau$ définit une \deux{}cellule de $(f', \gamma)$ vers $(f, \alpha \CompDeuxUn (q_{b_{F}} \CompDeuxZero v(\beta)))$ dans $\TrancheCoLax{\mathdeuxcat{A}}{vu}{c}$. Il ne reste plus qu'à montrer que c'est la seule. Soit donc $\epsilon : f' \Rightarrow f$ une \deux{}cellule de $(f', \gamma)$ vers $(f, \alpha \CompDeuxUn (q_{b_{F}} \CompDeuxZero v(\beta)))$ dans $\TrancheCoLax{\mathdeuxcat{A}}{vu}{c}$. On a donc 
$$
\alpha \CompDeuxUn (q_{b_{F}} \CompDeuxZero v(\beta)) \CompDeuxUn (q_{b_{F}} v(q_{a_{F}}) \CompDeuxZero v(u(\epsilon))) = \gamma,
$$
donc 
$$
\alpha \CompDeuxUn (q_{b_{F}} \CompDeuxZero (v(\beta) \CompDeuxUn (v(q_{a_{F}}) \CompDeuxZero v(u(\epsilon))))) = \gamma,
$$
donc 
$$
 \alpha \CompDeuxUn (q_{b_{F}} \CompDeuxZero v(\beta \CompDeuxUn (q_{a_{F}} \CompDeuxZero u(\epsilon)))) = \gamma.
$$
En vertu de la définition de $\rho$, on en déduit
$$
\beta \CompDeuxUn (q_{a_{F}} \CompDeuxZero u(\epsilon)) = \rho.
$$
En vertu de la définition de $\tau$, on en déduit 
$$
\epsilon = \tau.
$$

Ainsi, la \un{}cellule
$$
(f, \alpha (q_{b_{F}} \CompDeuxZero v(\beta)))
$$ 
est un objet final de la catégorie
$$
\CatHom{\TrancheCoLax{\mathdeuxcat{A}}{vu}{c}}{(a,q)}{(a_{F}, q_{b_{F}} v(q_{a_{F}}))},
$$ 
comme annoncé. Il en résulte que la \deux{}catégorie $\TrancheCoLax{\mathdeuxcat{A}}{vu}{c}$ admet un objet admettant un objet final. CQFD.

Les trois autres assertions s'en déduisent par des arguments de dualité évidents en vertu des définitions. 
\end{proof} 

\begin{rem}\label{CompositionBettiPower}
Le résultat mis en valeur par la proposition \ref{PreAdjointsCompo} se trouve en fait déjà mentionné dans le dernier paragraphe de \cite[page 941]{BettiPower}. 
\end{rem}

\begin{rem}
Une petite \deux{}catégorie $\mathdeuxcat{A}$ admet un objet admettant un objet final si et seulement si le morphisme canonique $\mathdeuxcat{A} \to \DeuxCatPonct$ est un préadjoint à gauche colax. Il résulte donc de la proposition \ref{PreAdjointsCompo} qu'étant donné un morphisme $u : \mathdeuxcat{A} \to \mathdeuxcat{B}$ de $\DeuxCat$, si $\mathdeuxcat{B}$ admet un objet admettant un objet final et si $u$ est un préadjoint à gauche colax, alors $\mathdeuxcat{A}$ admet un objet admettant un objet final. En vertu de la proposition \ref{EquiPreadjoint}, on en déduit également le cas particulier suivant : si $u : \mathdeuxcat{A} \to \mathdeuxcat{B}$ est une équivalence de \deux{}catégories, alors $\mathdeuxcat{A}$ admet un objet admettant un objet final si et seulement si $\mathdeuxcat{B}$ admet un objet admettant un objet final. Ces observations admettent bien entendu trois variantes duales. 
\end{rem}

\begin{paragr}
Soit $u : \mathdeuxcat{A} \to \mathdeuxcat{B}$ un préadjoint à gauche lax. Nous allons décrire un \DeuxFoncteurCoLax{} $v : \mathdeuxcat{B} \to \mathdeuxcat{A}$ ainsi que, si $u$ est strict, une \DeuxTransformationCoLax{} $uv \Rightarrow 1_{\mathdeuxcat{B}}$, tous deux canoniques, ce qui justifie peut-être un peu mieux l'adoption du terme « préadjoint », au-delà de l'analogie portant sur les tranches.

\begin{itemize}
\item[(Image des objets)]
Soit $b$ un objet de $\mathdeuxcat{B}$. Par hypothèse, la \deux{}catégorie $\TrancheLax{\mathdeuxcat{A}}{u}{b}$ admet un objet admettant un objet initial. On note $(v(b), p_{b} : u(v(b)) \to b)$ cet objet, ce qui définit $v(b)$. (On aura remarqué la nécessité de faire un choix, que nous ne signalerons plus au cours des prochaines étapes.) 

\item[(Image des \un{}cellules)]
Soit $g : b \to b'$ une \un{}cellule arbitraire de $\mathdeuxcat{B}$. Cela nous permet de considérer le diagramme
$$
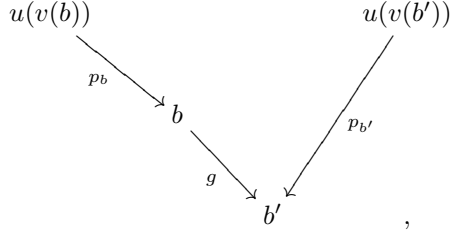
\xymatrix{
u(v(b))
\ar[dr]_{p_{b}}
&&&u(v(b'))
\ar[ddl]^{p_{b'}}
\\
&b
\ar[dr]_{g}
\\
&&b'
&,
}
$$
correspondant à deux objets $(v(b), g p_{b})$ et $(v(b'), p_{b'})$ de $\TrancheLax{\mathdeuxcat{A}}{u}{b'}$. Par définition du couple $(v(b'), p_{b'})$, il existe un objet initial 
$$
(v(g) : v(b) \to v(b'), \alpha_{g} : gp_{b} \Rightarrow p_{b'}u(v(g)))
$$ 
dans la catégorie 
$$
\CatHom{\TrancheLax{\mathdeuxcat{A}}{u}{b'}}{(v(b), g p_{b})}{(v(b'), p_{b'})}.
$$
Cela définit $v(g)$. 

\item[(Image des \deux{}cellules)]
Soient $b$ et $b'$ deux objets de $\mathdeuxcat{B}$, $g$ et $g'$ deux \un{}cellules du premier vers le second, et $\beta : g \Rightarrow g'$ une \deux{}cellule arbitraire. Cela nous fournit une \deux{}cellule
$$
\alpha_{g'} \CompDeuxUn (\beta \CompDeuxZero p_{b}) : gp_{b} \Rightarrow p_{b'}u(v(g'))
$$
et définit donc une \un{}cellule $(v(g'), \alpha_{g'} \CompDeuxUn (\beta \CompDeuxZero p_{b}))$ de $(v(b), gp_{b})$ vers $(v(b'), p_{b'})$ dans la \deux{}ca\-té\-go\-rie $\TrancheLax{\mathdeuxcat{A}}{u}{b'}$. Par construction, il existe donc une et une seule \deux{}cellule $v(\beta) : v(g) \Rightarrow v(g')$ telle que 
\begin{equation}\label{Defvbeta}
(p_{b'} \CompDeuxZero u(v(\beta))) \CompDeuxUn \alpha_{g} = \alpha_{g'} \CompDeuxUn (\beta \CompDeuxZero p_{b}).
\end{equation}
Cela définit $v(\beta)$.

\item[(\deux{}cellules structurales des objets)]
Soit $b$ un objet de $\mathdeuxcat{B}$. On considère le diagramme
\\
\\
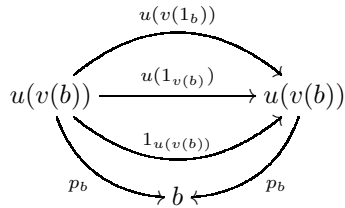
$$
\xymatrix{
u(v(b))
\ar[rr]^{u(1_{v(b)})}
\ar@/^2pc/[rr]^{u(v(1_{b}))}
\ar@/_2pc/[rr]^{1_{u(v(b))}}
\ar@/_1pc/[dr]_{p_{b}}
&&u(v(b))
\ar@/^1pc/[dl]^{p_{b}}
\\
&b
}
$$
et les \deux{}cellules 
$$
\alpha_{1_{b}} : p_{b} \Rightarrow p_{b} u(v(1_{b}))
$$ 
et 
$$
p_{b} \CompDeuxZero u_{v(b)} : p_{b} \Rightarrow p_{b} u(1_{v(b)}).
$$
Il existe alors, par définition, une et une seule \deux{}cellule $v_{b} : v(1_{b}) \Rightarrow 1_{v(b)}$ de $\mathdeuxcat{B}$ telle que
\begin{equation}\label{Defvb}
(p_{b} \CompDeuxZero u(v_{b})) \CompDeuxUn \alpha_{1_{b}} = p_{b} \CompDeuxZero u_{v(b)}.
\end{equation}
Cela définit la \deux{}cellule structurale $v_{b} : v(1_{b}) \Rightarrow 1_{v(b)}$ pour tout objet $b$ de $\mathdeuxcat{B}$. 

\item[(\deux{}cellules structurales de composition)]
Soient $g : b \to b'$ et $g' : b' \to b''$ deux \un{}cellules composables dans $\mathdeuxcat{B}$. On a les \un{}cellules $v(g'g)$ et $v(g') v(g)$ de $v(b)$ vers $v(b'')$ dans $\mathdeuxcat{A}$, le diagramme
\\
\\
$$
\xymatrix{
u(v(b))
\ar@/^2pc/[rrrr]^{u(v(g'g))}
\ar[rrrr]^{u(v(g')v(g))}
\ar@/_1pc/[rrd]^{u(v(g))}
\ar@/_1.5pc/[ddr]_{p_{b}}
&&&&u(v(b''))
\ar@/^1pc/[ddddl]^{p_{b''}}
\\
&&u(v(b'))
\ar[dd]^{p_{b'}}
\ar@/_1pc/[rru]^{u(v(g'))}
\\
&b
\ar@/_0.5pc/[rd]_{g}
\\
&&b'
\ar@/_0.5pc/[rd]_{g'}
\\
&&&b''
}
$$
et des \deux{}cellules 
$$
\alpha_{g} : g p_{b} \Rightarrow p_{b'} u(v(g)),
$$ 
$$
\alpha_{g'} : g' p_{b'} \Rightarrow p_{b''} u(v(g'))
$$ 
et
$$
\alpha_{g'g} : g'g p_{b} \Rightarrow p_{b''} u(v(g'g))
$$
dans $\mathdeuxcat{B}$.
On a donc une \deux{}cellule composée
$$
\xymatrix{
g' g p_{b} 
\ar@{=>}[rr]^{g' \CompDeuxZero \alpha_{g}}
&&
g' p_{b'} u(v(g))
\ar@{=>}[rr]^{\alpha_{g'} \CompDeuxZero u(v(g))}
&&
p_{b''} u(v(g')) u(v(g))
\ar@{=>}[rr]^{p_{b''} \CompDeuxZero u_{v(g'), v(g)}}
&&
p_{b''} u(v(g') v(g)).
}
$$
Par définition du couple $(v(g'g), \alpha_{g'g})$, il existe donc une et une seule \deux{}cellule 
$$
v_{g', g} : v(g'g) \Rightarrow v(g') v(g)
$$ 
telle que 
\begin{equation}\label{Defvg',g}
(p_{b''} \CompDeuxZero u(v_{g', g})) \CompDeuxUn \alpha_{g'g} = (p_{b''} \CompDeuxZero u_{v(g'), v(g)}) \CompDeuxUn (\alpha_{g'} \CompDeuxZero u(v(g))) \CompDeuxUn (g' \CompDeuxZero \alpha_{g}). 
\end{equation}
Cela définit la \deux{}cellule structurale $v_{g', g} : v(g'g) \Rightarrow v(g') v(g)$ pour tout couple de \un{}cellules $g$ et $g'$ de $\mathdeuxcat{B}$ telles que la composée $g'g$ fasse sens. 
\\
\end{itemize}

Vérifions maintenant que cette construction fournit bien un \DeuxFoncteurCoLax{} $v : \mathdeuxcat{B} \to \mathdeuxcat{A}$. 
\\

\begin{itemize}
\item[(Fonctorialité de $v$ relativement aux \deux{}cellules (1))]
Soit $g : b \to b'$ une \un{}cellule quelconque de $\mathdeuxcat{B}$. On veut vérifier l'égalité $v(1_{g}) = 1_{v(g)}$. Pour cela, considérons le diagramme
\\
\\
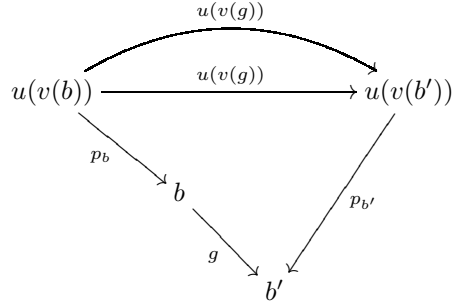
$$
\xymatrix{
u(v(b))
\ar@/^2pc/[rrr]^{u(v(g))}
\ar[rrr]^{u(v(g))}
\ar[dr]_{p_{b}}
&&&
u(v(b'))
\ar[ddl]^{p_{b'}}
\\
&b
\ar[rd]_{g}
\\
&&
b'
}
$$
et la \deux{}cellule $\alpha_{g} : g p_{b} \Rightarrow p_{b'} u(v(g))$. Pour montrer l'égalité 
$$
v(1_{g}) = 1_{v(g)}, 
$$
il suffit de montrer l'égalité
$$
(p_{b'} \CompDeuxZero u(1_{v(g)})) \CompDeuxUn \alpha_{g} = \alpha_{g},
$$
qui est vérifiée puisque, par fonctorialité de $u$ sur les \deux{}cellules, $u(1_{v(g)}) = 1_{u(v(g))}$.

\item[(Fonctorialité de $v$ relativement aux \deux{}cellules (2))]
Soient $b$ et $b'$ deux objets de $\mathdeuxcat{B}$, $g : b \to b'$, $g' : b \to b'$ et $g'' : b \to b'$ trois \un{}cellules de $\mathdeuxcat{B}$  et $\beta : g \Rightarrow g'$ et $\beta' : g' \Rightarrow g''$ deux \deux{}cellules de $\mathdeuxcat{B}$. On veut vérifier l'égalité $v(\beta' \CompDeuxUn \beta) = v(\beta') \CompDeuxUn v(\beta)$. Pour cela, considérons le diagramme
\\
\\
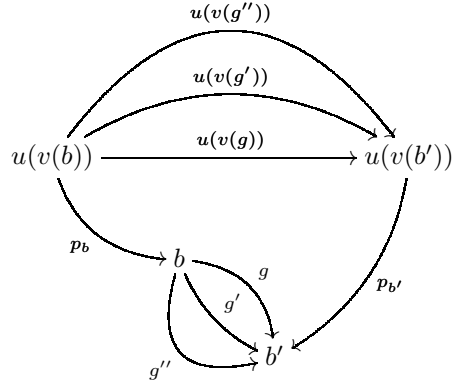
$$
\xymatrix{
u(v(b))
\ar@/^4pc/[rrr]^{u(v(g''))}^{u(v(g''))}
\ar@/^2pc/[rrr]^{u(v(g'))}^{u(v(g'))}
\ar@/^0pc/[rrr]^{u(v(g))}^{u(v(g))}
\ar@/_1pc/[rd]_{p_{b}}_{p_{b}}
&&&
u(v(b'))
\ar@/^1pc/[ddl]^{p_{b'}}^{p_{b'}}
\\
&
b
\ar@/^1pc/[rd]^{g}
\ar@/_0.5pc/[rd]^{g'}
\ar@/_2pc/[rd]_{g''}
\\
&&
b'
}
$$
et les \deux{}cellules 
$$
\alpha_{g} : g p_{b} \Rightarrow p_{b'} u(v(g)),
$$
$$
\alpha_{g'} : g' p_{b} \Rightarrow p_{b'} u(v(g'))
$$
et
$$
\alpha_{g''} : g'' p_{b} \Rightarrow p_{b'} u(v(g'')).
$$
On a, par définition de $v(\beta)$ et $v(\beta')$ respectivement, les égalités 
$$
\alpha_{g'} \CompDeuxUn (\beta \CompDeuxZero p_{b}) = (p_{b'} \CompDeuxZero u(v(\beta))) \CompDeuxUn \alpha_{g}
$$ 
et 
$$
\alpha_{g''} \CompDeuxUn (\beta' \CompDeuxZero p_{b}) = (p_{b'} \CompDeuxZero u(v(\beta'))) \CompDeuxUn \alpha_{g'}.
$$

Pour montrer l'égalité $v(\beta' \CompDeuxUn \beta) = v(\beta') \CompDeuxUn v(\beta)$, il suffit de vérifier l'égalité 
$$
(p_{b'} \CompDeuxZero u(v(\beta') \CompDeuxUn v(\beta))) \CompDeuxUn \alpha_{g} = \alpha_{g''} \CompDeuxUn ((\beta' \CompDeuxUn \beta) \CompDeuxZero p_{b}),
$$ 
qui résulte de la suite d'égalités suivante.  
$$
\begin{aligned}
(p_{b'} \CompDeuxZero u(v(\beta') v(\beta))) \CompDeuxUn \alpha_{g} &= (p_{b'} \CompDeuxZero (u(v(\beta')) u(v(\beta)))) \CompDeuxUn \alpha_{g}
\\
&= (p_{b'} \CompDeuxZero u(v(\beta'))) \CompDeuxUn (p_{b'} \CompDeuxZero u(v(\beta))) \CompDeuxUn \alpha_{g}
\\
&= (p_{b'} \CompDeuxZero u(v(\beta'))) \CompDeuxUn \alpha_{g'} \CompDeuxUn (\beta \CompDeuxZero p_{b})
\\
&= \alpha_{g''} \CompDeuxUn (\beta' \CompDeuxZero p_{b}) \CompDeuxUn (\beta \CompDeuxZero p_{b})
\\
&= \alpha_{g''} \CompDeuxUn ((\beta' \CompDeuxUn \beta) \CompDeuxZero p_{b}).
\end{aligned}
$$

\item[(Naturalité des \deux{}cellules de composition de $v$)]
Soient maintenant $b$, $b'$ et $b''$ trois objets de $\mathdeuxcat{B}$, $g$ et $h$ deux \un{}cellules de $b$ vers $b'$, $g'$ et $h'$ deux \un{}cellules de $b'$ vers $b''$, et $\beta : g \Rightarrow h$ et $\beta' : g' \Rightarrow h'$ deux \deux{}cellules. On veut vérifier la commutativité du diagramme
$$
\xymatrix{
v(g'g)
\ar@{=>}[rr]^{v_{g', g}}
\ar@{=>}[dd]_{v(\beta' \CompDeuxZero \beta)}
&&v(g') v(g)
\ar@{=>}[dd]^{v(\beta') \CompDeuxZero v(\beta)}
\\
\\
v(h'h)
\ar@{=>}[rr]_{v_{h', h}}
&&v(h') v(h)
&.
}
$$
Pour cela, considérons le diagramme
\\
\\
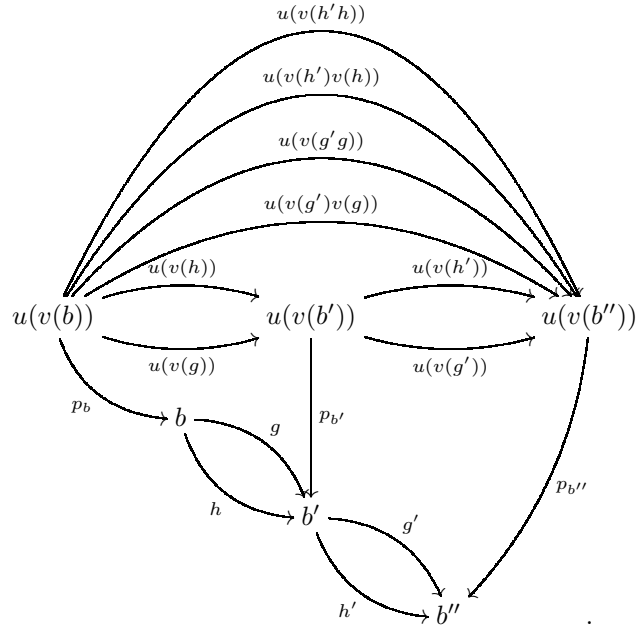
$$
\xymatrix{
u(v(b))
\ar@/^1pc/[rr]^{u(v(h))}
\ar@/_1pc/[rr]_{u(v(g))}
\ar@/_1pc/[dr]_{p_{b}}
\ar@/^3pc/[rrrr]^{u (v(g') v(g))}
\ar@/^5pc/[rrrr]^{u (v (g'g))}
\ar@/^7pc/[rrrr]^{u (v(h') v(h))}
\ar@/^9pc/[rrrr]^{u (v (h'h))}
&&
u(v(b'))
\ar@/^1pc/[rr]^{u(v(h'))}
\ar@/_1pc/[rr]_{u(v(g'))}
\ar[dd]^{p_{b'}}
&&
u(v(b''))
\ar@/^1pc/[dddl]^{p_{b''}}
\\
&
b
\ar@/^1pc/[dr]^{g}
\ar@/_1pc/[dr]_{h}
\\
&&
b'
\ar@/^1pc/[dr]^{g'}
\ar@/_1pc/[dr]_{h'}
\\
&&&
b''
&.
}
$$
On n'a pas fait figurer dans le diagramme ci-dessus les \deux{}cellules $\beta$, $\beta'$, $u(v(\beta))$, $u(v(\beta'))$, 
$$
\alpha_{g} : g p_{b} \Rightarrow p_{b'} u(v(g)),
$$
$$
\alpha_{g'} : g' p_{b'} \Rightarrow p_{b''} u(v(g')),
$$
$$
\alpha_{g'g} : g' g p_{b} \Rightarrow p_{b''} u(v(g'g)),
$$
$$
\alpha_{h} : h p_{b} \Rightarrow p_{b'} u(v(h)),
$$
$$
\alpha_{h'} : h' p_{b'} \Rightarrow p_{b''} u(v(h'))
$$
et
$$
\alpha_{h'h} : h' h p_{b} \Rightarrow p_{b''} u(v(h'h)).
$$
Par définition de $v_{g',g}$, $v_{h',h}$, $v(\beta)$, $v(\beta')$, $v(\beta' \CompDeuxZero \beta)$ et du fait de la naturalité des \deux{}cellules de composition de $u$ respectivement, les diagrammes suivants sont commutatifs : 

\begin{equation}\label{16111}
\xymatrix{
g' g p_{b}
\ar@{=>}[rrr]^{g' \CompDeuxZero \alpha_{g}}
\ar@{=>}[dd]_{\alpha_{g'g}}
&&&
g' p_{b'} u(v(g))
\ar@{=>}[rrr]^{\alpha_{g'} \CompDeuxZero u(v(g))}
&&&
p_{b''} u(v(g')) u(v(g))
\ar@{=>}[dd]^{p_{b''} \CompDeuxZero u_{v(g'), v(g)}}
\\
\\
p_{b''} u(v(g'g))
\ar@{=>}[rrrrrr]_{p_{b''} \CompDeuxZero u(v_{g',g})}
&&&&&&
p_{b''} u(v(g')v(g))
&,
}
\end{equation}

\begin{equation}\label{16112}
\xymatrix{
h' h p_{b}
\ar@{=>}[rrr]^{h' \CompDeuxZero \alpha_{h}}
\ar@{=>}[dd]_{\alpha_{h'h}}
&&&
h' p_{b'} u(v(h))
\ar@{=>}[rrr]^{\alpha_{h'} \CompDeuxZero u(v(h))}
&&&
p_{b''} u(v(h')) u(v(h))
\ar@{=>}[dd]^{p_{b''} \CompDeuxZero u_{v(h'), v(h)}}
\\
\\
p_{b''} u(v(h'h))
\ar@{=>}[rrrrrr]_{p_{b''} \CompDeuxZero u(v_{h',h})}
&&&&&&
p_{b''} u(v(h')v(h))
&,
}
\end{equation}

\begin{equation}\label{16113}
\xymatrix{
g p_{b}
\ar@{=>}[rr]^{\alpha_{g}}
\ar@{=>}[dd]_{\beta \CompDeuxZero p_{b}}
&&
p_{b'} u(v(g))
\ar@{=>}[dd]^{p_{b'} \CompDeuxZero u(v(\beta))}
\\
\\
h p_{b}
\ar@{=>}[rr]_{\alpha_{h}}
&&
p_{b'} u(v(h))
&,
}
\end{equation}

\begin{equation}\label{16114}
\xymatrix{
g' p_{b'}
\ar@{=>}[rr]^{\alpha_{g'}}
\ar@{=>}[dd]_{\beta' \CompDeuxZero p_{b'}}
&&
p_{b''} u(v(g'))
\ar@{=>}[dd]^{p_{b''} \CompDeuxZero u(v(\beta'))}
\\
\\
h' p_{b'}
\ar@{=>}[rr]_{\alpha_{h'}}
&&
p_{b''} u(v(h'))
&,
}
\end{equation}

\begin{equation}\label{16115}
\xymatrix{
g' g p_{b}
\ar@{=>}[rr]^{\alpha_{g'g}}
\ar@{=>}[dd]_{\beta' \CompDeuxZero \beta \CompDeuxZero p_{b}}
&&
p_{b''} u(v(g'g))
\ar@{=>}[dd]^{p_{b''} \CompDeuxZero u(v(\beta' \CompDeuxZero \beta))}
\\
\\
h' h p_{b}
\ar@{=>}[rr]_{\alpha_{h'h}}
&&
p_{b''} u(v(h'h))
&,
}
\end{equation}
et
\begin{equation}\label{16116}
\xymatrix{
u(v(g')) u(v(g))
\ar@{=>}[rr]^{u_{v(g'), v(g)}}
\ar@{=>}[dd]_{u(v(\beta')) \CompDeuxZero u(v(\beta))}
&&
u(v(g')v(g))
\ar@{=>}[dd]^{u(v(\beta') \CompDeuxZero v(\beta))}
\\
\\
u(v(h')) u(v(h))
\ar@{=>}[rr]_{u_{v(h'), v(h)}}
&&
u(v(h')v(h))
&.
}
\end{equation}

Pour montrer l'égalité souhaitée $(v(\beta') \CompDeuxZero v(\beta)) \CompDeuxUn v_{g', g} = v_{h', h} \CompDeuxUn v(\beta' \CompDeuxZero \beta)$, il suffit de vérifier l'égalité
$$
(p_{b''} \CompDeuxZero u ((v(\beta') \CompDeuxZero v(\beta)) \CompDeuxUn v_{g', g}) ) \CompDeuxUn \alpha_{g'g} = (p_{b''} \CompDeuxZero u (v_{h', h} \CompDeuxUn v(\beta' \CompDeuxZero \beta)) ) \CompDeuxUn \alpha_{g'g}.
$$
Cela résulte de la suite d'égalités
\begin{align*}
&(p_{b''} \CompDeuxZero u ((v(\beta') \CompDeuxZero v(\beta)) \CompDeuxUn v_{g', g}) ) \CompDeuxUn \alpha_{g'g} 
\\
&= (p_{b''} \CompDeuxZero u(v(\beta') \CompDeuxZero v(\beta))) \CompDeuxUn (p_{b''} \CompDeuxZero u(v_{g',g})) \CompDeuxUn \alpha_{g'g}
\\
&= (p_{b''} \CompDeuxZero u(v(\beta') \CompDeuxZero v(\beta))) \CompDeuxUn (p_{b''} \CompDeuxZero u_{v(g'), v(g)}) \CompDeuxUn (\alpha_{g'} \CompDeuxZero u(v(g))) \CompDeuxUn (g' \CompDeuxZero \alpha_{g}) \phantom{blabla} (cf. \ref{16111})
\\
&= (p_{b''} \CompDeuxZero (u (v (\beta') \CompDeuxZero v(\beta)) u_{v(g'), v(g)})) \CompDeuxUn (\alpha_{g'} \CompDeuxZero u(v(g))) \CompDeuxUn (g' \CompDeuxZero \alpha_{g})
\\
&= (p_{b''} \CompDeuxZero (u_{v(h'), v(h)} (u(v(\beta')) \CompDeuxZero u(v(\beta))))) \CompDeuxUn (\alpha_{g'} \CompDeuxZero u(v(g))) \CompDeuxUn (g' \CompDeuxZero \alpha_{g}) \phantom{blabla} (cf. \ref{16116})
\\
&= (p_{b''} \CompDeuxZero u_{v(h'), v(h)}) \CompDeuxUn (p_{b''} \CompDeuxZero u(v(\beta')) \CompDeuxZero u(v(\beta))) \CompDeuxUn (\alpha_{g'} \CompDeuxZero u(v(g))) \CompDeuxUn (g' \CompDeuxZero \alpha_{g}) 
\\
&= (p_{b''} \CompDeuxZero u_{v(h'), v(h)}) \CompDeuxUn (((p_{b''} \CompDeuxZero u(v(\beta'))) \alpha_{g'}) \CompDeuxZero u(v(\beta))) \CompDeuxUn (g' \CompDeuxZero \alpha_{g}) \phantom{blabla} \text{(loi d'échange)}
\\
&= (p_{b''} \CompDeuxZero u_{v(h'), v(h)}) \CompDeuxUn ((\alpha_{h'} (\beta' \CompDeuxZero p_{b'})) \CompDeuxZero u(v(\beta))) \CompDeuxUn (g' \CompDeuxZero \alpha_{g}) \phantom{blabla} (cf. \ref{16114})
\\
&= (p_{b''} \CompDeuxZero u_{v(h'), v(h)}) \CompDeuxUn (\alpha_{h'} \CompDeuxZero u(v(h))) \CompDeuxUn (\beta' \CompDeuxZero p_{b'} \CompDeuxZero u(v(\beta))) \CompDeuxUn (g' \CompDeuxZero \alpha_{g})  \phantom{blabla} \text{(loi d'échange)} 
\\
&= (p_{b''} \CompDeuxZero u_{v(h'), v(h)}) \CompDeuxUn (\alpha_{h'} \CompDeuxZero u(v(h))) \CompDeuxUn (\beta' \CompDeuxZero ((p_{b'} \CompDeuxZero u(v(\beta))) \alpha_{g}))  \phantom{blabla} \text{(loi d'échange)}
\\
&= (p_{b''} \CompDeuxZero u_{v(h'), v(h)}) \CompDeuxUn (\alpha_{h'} \CompDeuxZero u(v(h))) \CompDeuxUn (\beta' \CompDeuxZero (\alpha_{h} (\beta \CompDeuxZero p_{b}))) \phantom{blabla} (cf. \ref{16113})
\\
&= (p_{b''} \CompDeuxZero u_{v(h'), v(h)}) \CompDeuxUn (\alpha_{h'} \CompDeuxZero u(v(h))) \CompDeuxUn (h' \CompDeuxZero \alpha_{h}) \CompDeuxUn (\beta' \CompDeuxZero \beta \CompDeuxZero p_{b}) \phantom{blabla} \text{(loi d'échange)}
\\
&= (p_{b''} \CompDeuxZero u(v_{h',h})) \CompDeuxUn \alpha_{h'h} \CompDeuxUn (\beta' \CompDeuxZero \beta \CompDeuxZero p_{b}) \phantom{blabla} (cf. \ref{16112})
\\
&= (p_{b''} \CompDeuxZero u(v_{h',h})) \CompDeuxUn (p_{b''} \CompDeuxZero u(v(\beta' \CompDeuxZero \beta))) \CompDeuxUn \alpha_{g'g} \phantom{blabla} (cf. \ref{16115})
\\
&= (p_{b''} \CompDeuxZero (u(v_{h',h}) u(v(\beta' \CompDeuxZero \beta)))) \CompDeuxUn \alpha_{g'g}
\\
&= (p_{b''} \CompDeuxZero u(v_{h', h} v(\beta' \CompDeuxZero \beta))) \alpha_{g'g}.
\end{align*}

\item[(Condition de cocycle pour $v$)]
On se propose à présent de vérifier la condition de cocycle pour $v$. Soient donc $g : b \to b'$, $g' : b' \to b''$ et $g'' : b'' \to b'''$ trois \un{}cellules composables de $\mathdeuxcat{B}$. On souhaite vérifier la commutativité du diagramme
$$
\xymatrix{
v(g'' g' g) 
\ar@{=>}[rr]^{v_{g'', g'g}}
\ar@{=>}[dd]_{v_{g''g', g}}
&&
v(g'') v(g'g)
\ar@{=>}[dd]^{v(g'') \CompDeuxZero v_{g',g}} 
\\
\\
v(g''g') v(g) 
\ar@{=>}[rr]_{v_{g'',g'} \CompDeuxZero v(g)}
&&
v(g'') v(g') v(g) 
&.
}
$$
Pour cela, considérons le diagramme
$$
\xymatrix{
u(v(b))
\ar[rrr]_{u(v(g))}
\ar@/^2pc/[rrrrrr]^{u( v(g') v(g) )}
\ar@/^4pc/[rrrrrr]^{u(v(g'g))}
\ar@/^6pc/[rrrrrrrrr]^{u(v(g'') v(g') v(g))}
\ar@/^8pc/[rrrrrrrrr]^{u(v(g''g')v(g))}
\ar@/^10pc/[rrrrrrrrr]^{u(v(g'')v(g'g))}
\ar@/^12pc/[rrrrrrrrr]^{u(v(g''g'g))}
\ar[dd]_{p_{b}}
&&&
u(v(b'))
\ar[rrr]_{u(v(g'))}
\ar@/^2pc/[rrrrrr]^{u( v(g'') v(g') )}
\ar@/^4pc/[rrrrrr]^{u(v(g''g'))}
\ar[dd]^{p_{b'}}
&&&
u(v(b''))
\ar[rrr]_{u(v(g''))}
\ar[dd]^{p_{b''}}
&&&
u(v(b'''))
\ar[dd]^{p_{b'''}}
\\
\\
b
\ar[rrr]_{g}
&&&
b'
\ar[rrr]_{g'}
&&&
b''
\ar[rrr]_{g''}
&&&
b'''
&,
}
$$
dans lequel on n'a pas fait figurer les \deux{}cellules 
$$
\alpha_{g} : g p_{b} \Rightarrow p_{b'} u(v(g)),
$$
$$
\alpha_{g'} : g' p_{b'} \Rightarrow p_{b''} u(v(g')),
$$
$$
\alpha_{g''} : g'' p_{b''} \Rightarrow p_{b'''} u(v(g'')),
$$
$$
\alpha_{g'g} : g' g p_{b} \Rightarrow p_{b''} u(v(g'g)),
$$
$$
\alpha_{g''g'} : g'' g' p_{b'} \Rightarrow p_{b'''} u(v(g''g'))
$$
et
$$
\alpha_{g''g'g} : g'' g' g p_{b} \Rightarrow p_{b'''} u(v(g''g'g)).
$$
Par définition de $v_{g',g}$, $v_{g'',g'}$, $v_{g''g',g}$ et $v_{g'',g'g}$, du fait de la relation de cocycle pour $u$ et de la naturalité des \deux{}cellules structurales de composition de $u$ (deux fois) respectivement, les diagrammes suivants sont commutatifs :

\begin{equation}\label{vg',g}
\xymatrix{
g' g p_{b}
\ar@{=>}[rrr]^{g' \CompDeuxZero \alpha_{g}}
\ar@{=>}[dd]_{\alpha_{g'g}}
&&&
g' p_{b'} u(v(g))
\ar@{=>}[rrr]^{\alpha_{g'} \CompDeuxZero u(v(g))}
&&&
p_{b''} u(v(g')) u(v(g))
\ar@{=>}[dd]^{p_{b''} \CompDeuxZero u_{v(g'), v(g)}}
\\
\\
p_{b''} u(v(g'g))
\ar@{=>}[rrrrrr]_{p_{b''} \CompDeuxZero u(v_{g',g})}
&&&&&&
p_{b''} u(v(g') v(g))
&,
}
\end{equation}

\begin{equation}\label{vg'',g'}
\xymatrix{
g'' g' p_{b'}
\ar@{=>}[rrr]^{g'' \CompDeuxZero \alpha_{g'}}
\ar@{=>}[dd]_{\alpha_{g''g'}}
&&&
g'' p_{b''} u(v(g'))
\ar@{=>}[rrr]^{\alpha_{g''} \CompDeuxZero u(v(g'))}
&&&
p_{b'''} u(v(g'')) u(v(g'))
\ar@{=>}[dd]^{p_{b'''} \CompDeuxZero u_{v(g''), v(g')}}
\\
\\
p_{b'''} u(v(g''g'))
\ar@{=>}[rrrrrr]_{p_{b'''} \CompDeuxZero u(v_{g'',g'})}
&&&&&&
p_{b'''} u(v(g'') v(g'))
&,
}
\end{equation}

\begin{equation}\label{vg''g',g}
\xymatrix{
g'' g' g p_{b}
\ar@{=>}[rrr]^{g'' g' \CompDeuxZero \alpha_{g}}
\ar@{=>}[dd]_{\alpha_{g'' g' g}}
&&&
g'' g' p_{b'} u(v(g))
\ar@{=>}[rrr]^{\alpha_{g'' g'} \CompDeuxZero u(v(g))}
&&&
p_{b'''} u(v(g'' g')) u(v(g))
\ar@{=>}[dd]^{p_{b'''} \CompDeuxZero u_{v(g'' g'), v(g)}}
\\
\\
p_{b'''} u(v(g'' g' g))
\ar@{=>}[rrrrrr]_{p_{b'''} \CompDeuxZero u(v_{g'' g', g})}
&&&&&&
p_{b'''} u(v(g'' g') v(g))
&,
}
\end{equation}

\begin{equation}\label{vg'',g'g}
\xymatrix{
g'' g' g p_{b}
\ar@{=>}[rrr]^{g'' \CompDeuxZero \alpha_{g'g}}
\ar@{=>}[dd]_{\alpha_{g''g'g}}
&&&
g'' p_{b''} u(v(g' g))
\ar@{=>}[rrr]^{\alpha_{g''} \CompDeuxZero u(v(g'g))}
&&&
p_{b'''} u(v(g'')) u(v(g'g))
\ar@{=>}[dd]^{p_{b'''} \CompDeuxZero u_{v(g''), v(g'g)}}
\\
\\
p_{b'''} u(v(g''g'g))
\ar@{=>}[rrrrrr]_{p_{b'''} \CompDeuxZero u(v_{g'', g'g})}
&&&&&&
p_{b'''} u(v(g'') v(g'g))
&,
}
\end{equation}

\begin{equation}\label{UCocycle}
\xymatrix{
u(v(g'')) u(v(g')) u(v(g))
\ar@{=>}[rrr]^{u(v(g'')) \CompDeuxZero u_{v(g'), v(g)}}
\ar@{=>}[dd]_{u_{v(g''),v(g')} \CompDeuxZero u(v(g))}
&&&
u(v(g'')) u(v(g') v(g))
\ar@{=>}[dd]^{u_{v(g''), v(g')v(g)}}
\\
\\
u(v(g'') v(g')) u(v(g))
\ar@{=>}[rrr]_{u_{v(g'') v(g'), v(g)}}
&&&
u(v(g'') v(g') v(g))
&,
}
\end{equation}

\begin{equation}\label{UFonctoriel1}
\xymatrix{
u(v(g'')) u(v(g'g))
\ar@{=>}[rrr]^{u_{v(g''), v(g'g)}}
\ar@{=>}[dd]_{u(v(g'')) \CompDeuxZero u(v_{g',g})}
&&&
u(v(g'') v(g'g))
\ar@{=>}[dd]^{u(v(g'') \CompDeuxZero v_{g',g})}
\\
\\
u(v(g'')) u(v(g') v(g))
\ar@{=>}[rrr]_{u_{v(g''), v(g')v(g)}}
&&&
u(v(g'') v(g') v(g))
&,
}
\end{equation}
et
\begin{equation}\label{UFonctoriel2}
\xymatrix{
u(v(g''g'))u(v(g))
\ar@{=>}[rrr]^{u_{v(g''g'),v(g)}}
\ar@{=>}[dd]_{u(v_{g'',g'}) \CompDeuxZero u(v(g))}
&&&
u(v(g''g')v(g))
\ar@{=>}[dd]^{u(v_{g'',g'} \CompDeuxZero v(g))}
\\
\\
u(v(g'')v(g'))u(v(g))
\ar@{=>}[rrr]_{u_{v(g'')v(g'), v(g)}}
&&&
u(v(g'') v(g') v(g))
&.
}
\end{equation}


L'égalité souhaitée est 
$$
(v(g'') \CompDeuxZero v_{g',g}) \CompDeuxUn v_{g'', g'g} = (v_{g'', g'} \CompDeuxZero v(g)) \CompDeuxUn v_{g''g', g}.
$$
Pour la démontrer, il suffit de vérifier l'égalité
$$
(p_{b'''} \CompDeuxZero u ((v(g'') \CompDeuxZero v_{g',g}) v_{g'', g'g})) \alpha_{g''g'g} = (p_{b'''} \CompDeuxZero u ((v_{g'',g'} \CompDeuxZero v(g)) v_{g''g',g})) \alpha_{g''g'g}.
$$
Cette dernière résulte de la suite d'égalités suivantes. 
\begin{align*}
&(p_{b'''} \CompDeuxZero u ((v(g'') \CompDeuxZero v_{g',g}) v_{g'', g'g})) \alpha_{g''g'g} 
\\
&= (p_{b'''} \CompDeuxZero u(v(g'') \CompDeuxZero v_{g',g})) \CompDeuxUn (p_{b'''} \CompDeuxZero u(v_{g'', g'g})) \CompDeuxUn \alpha_{g''g'g} 
\\
&= (p_{b'''} \CompDeuxZero u(v(g'') \CompDeuxZero v_{g',g})) \CompDeuxUn (p_{b'''} \CompDeuxZero u_{v(g''), v(g'g)}) \CompDeuxUn (\alpha_{g''} \CompDeuxZero u(v(g'g))) \CompDeuxUn (g'' \CompDeuxZero \alpha_{g'g}) \phantom{bla} \text{($cf.$ \ref{vg'',g'g})}
\\
&= (p_{b'''} \CompDeuxZero (u(v(g'') \CompDeuxZero v_{g',g}) \CompDeuxUn u_{v(g''), v(g'g)})) \CompDeuxUn (\alpha_{g''} \CompDeuxZero u(v(g'g))) \CompDeuxUn (g'' \CompDeuxZero \alpha_{g'g})
\\
&= (p_{b'''} \CompDeuxZero (u_{v(g''), v(g')v(g)} \CompDeuxUn (u(v(g'')) \CompDeuxZero u(v_{g',g})))) \CompDeuxUn (\alpha_{g''} \CompDeuxZero u(v(g'g))) \CompDeuxUn (g'' \CompDeuxZero \alpha_{g'g}) \phantom{bla} \text{($cf.$ \ref{UFonctoriel1})}
\\
&= (p_{b'''} \CompDeuxZero u_{v(g''), v(g')v(g)}) \CompDeuxUn (p_{b'''} u(v(g'')) \CompDeuxZero u(v_{g',g})) \CompDeuxUn (\alpha_{g''} \CompDeuxZero u(v(g'g))) \CompDeuxUn (g'' \CompDeuxZero \alpha_{g'g})
\\
&= (p_{b'''} \CompDeuxZero u_{v(g''), v(g')v(g)}) \CompDeuxUn (\alpha_{g''} \CompDeuxZero u(v(g') v(g))) \CompDeuxUn (g'' p_{b''} \CompDeuxZero u(v_{g',g})) \CompDeuxUn (g'' \CompDeuxZero \alpha_{g'g}) \phantom{bla} \text{(loi d'échange)}
\\
&= (p_{b'''} \CompDeuxZero u_{v(g''), v(g')v(g)}) \CompDeuxUn (\alpha_{g''} \CompDeuxZero u(v(g') v(g))) \CompDeuxUn (g'' \CompDeuxZero ((p_{b''} \CompDeuxZero u(v_{g',g})) \CompDeuxUn \alpha_{g'g}))
\\
&= (p_{b'''} \CompDeuxZero u_{v(g''), v(g')v(g)}) \CompDeuxUn (\alpha_{g''} \CompDeuxZero u(v(g') v(g))) \CompDeuxUn (g'' \CompDeuxZero ((p_{b''} \CompDeuxZero u_{v(g'), v(g)}) \CompDeuxUn (\alpha_{g'} \CompDeuxZero u(v(g))) \CompDeuxUn (g' \CompDeuxZero \alpha_{g}))) 
\\
&\phantom{=1}\text{($cf.$ \ref{vg',g})}
\\
&= (p_{b'''} \CompDeuxZero u_{v(g''), v(g')v(g)}) \CompDeuxUn (\alpha_{g''} \CompDeuxZero u(v(g') v(g))) \CompDeuxUn (g'' p_{b''} \CompDeuxZero u_{v(g'), v(g)}) \CompDeuxUn (g'' \CompDeuxZero \alpha_{g'} \CompDeuxZero u(v(g))) \CompDeuxUn (g''g' \CompDeuxZero \alpha_{g})
\\
&= (p_{b'''} \CompDeuxZero u_{v(g''), v(g')v(g)}) \CompDeuxUn (p_{b'''} u(v(g'')) \CompDeuxZero u_{v(g'), v(g)}) \CompDeuxUn (\alpha_{g''} \CompDeuxZero u(v(g')) u(v(g))) \CompDeuxUn
\\
&\phantom{=1} (g'' \CompDeuxZero \alpha_{g'} \CompDeuxZero u(v(g))) \CompDeuxUn (g''g' \CompDeuxZero \alpha_{g}) \phantom{bla}  \text{(loi d'échange)}
\\
&= (p_{b'''} \CompDeuxZero (u_{v(g''), v(g')v(g)} \CompDeuxUn (u(v(g'')) \CompDeuxZero u_{v(g'),v(g)}))) \CompDeuxUn (\alpha_{g''} \CompDeuxZero u(v(g')) u(v(g))) \CompDeuxUn 
\\
&\phantom{=1} (g'' \CompDeuxZero \alpha_{g'} \CompDeuxZero u(v(g))) \CompDeuxUn (g''g' \CompDeuxZero \alpha_{g})
\\
&=  (p_{b'''} \CompDeuxZero (u_{v(g'')v(g'), v(g)} \CompDeuxUn (u_{v(g''),v(g')} \CompDeuxZero u(v(g))))) \CompDeuxUn (\alpha_{g''} \CompDeuxZero u(v(g')) u(v(g))) \CompDeuxUn
\\
&\phantom{=1} (g'' \CompDeuxZero \alpha_{g'} \CompDeuxZero u(v(g))) \CompDeuxUn (g''g' \CompDeuxZero \alpha_{g}) \phantom{blabla} \text{($cf.$ \ref{UCocycle})}
\\
&= (p_{b'''} \CompDeuxZero u_{v(g'')v(g'), v(g)}) \CompDeuxUn (p_{b'''} \CompDeuxZero u_{v(g''),v(g')} \CompDeuxZero u(v(g))) \CompDeuxUn (\alpha_{g''} \CompDeuxZero u(v(g')) u(v(g))) \CompDeuxUn 
\\
&\phantom{=1} (g'' \CompDeuxZero \alpha_{g'} \CompDeuxZero u(v(g))) \CompDeuxUn (g''g' \CompDeuxZero \alpha_{g})
\\
&= (p_{b'''} \CompDeuxZero u_{v(g'')v(g'), v(g)}) \CompDeuxUn (((p_{b'''} \CompDeuxZero u_{v(g''), v(g')}) \CompDeuxUn (\alpha_{g''} \CompDeuxZero u(v(g'))) \CompDeuxUn (g'' \CompDeuxZero \alpha_{g'})) \CompDeuxZero u(v(g))) \CompDeuxUn (g''g' \CompDeuxZero \alpha_{g})
\\
&= (p_{b'''} \CompDeuxZero u_{v(g'')v(g'), v(g)}) \CompDeuxUn (((p_{b'''} \CompDeuxZero u(v_{g'',g'})) \CompDeuxUn \alpha_{g''g'}) \CompDeuxZero u(v(g))) \CompDeuxUn (g''g' \CompDeuxZero \alpha_{g}) \phantom{bla} \text{($cf.$ \ref{vg'',g'})}
\\
&= (p_{b'''} \CompDeuxZero u_{v(g'')v(g'), v(g)}) \CompDeuxUn (p_{b'''} \CompDeuxZero u(v_{g'',g'}) \CompDeuxZero u(v(g))) \CompDeuxUn (\alpha_{g''g'} \CompDeuxZero u(v(g))) \CompDeuxUn (g''g' \CompDeuxZero \alpha_{g})
\\
&= (p_{b'''} \CompDeuxZero (u_{v(g'') v(g'), v(g)} \CompDeuxUn (u(v_{g'',g'}) \CompDeuxZero u(v(g))))) (\CompDeuxUn \alpha_{g''g'} \CompDeuxZero u(v(g))) \CompDeuxUn (g''g' \CompDeuxZero \alpha_{g})
\\
&= (p_{b'''} \CompDeuxZero ((u(v_{g'',g'} \CompDeuxZero v(g))) \CompDeuxUn u_{v(g''g'), v(g)})) \CompDeuxUn (\alpha_{g''g'} \CompDeuxZero u(v(g))) \CompDeuxUn (g''g' \CompDeuxZero \alpha_{g}) \phantom{bla} \text{($cf.$ \ref{UFonctoriel2})}
\\
&= (p_{b'''} \CompDeuxZero u(v_{g'',g'} \CompDeuxZero v(g))) \CompDeuxUn (p_{b'''} \CompDeuxZero u_{v(g''g'), v(g)}) \CompDeuxUn (\alpha_{g''g'} \CompDeuxZero u(v(g))) \CompDeuxUn (g''g' \CompDeuxZero \alpha_{g})
\\
&= (p_{b'''} \CompDeuxZero u(v_{g'',g'} \CompDeuxZero v(g))) \CompDeuxUn (p_{b'''} \CompDeuxZero u(v_{g''g',g})) \CompDeuxUn \alpha_{g''g'g} \phantom{bla} \text{($cf.$ \ref{vg''g',g})}
\\
&= (p_{b'''} \CompDeuxZero (u(v_{g'',g'} \CompDeuxZero v(g)) \CompDeuxUn u(v_{g''g',g}))) \CompDeuxUn \alpha_{g''g'g}
\\
&= (p_{b'''} \CompDeuxZero u ((v_{g'',g'} \CompDeuxZero v(g)) v_{g''g',g})) \alpha_{g''g'g}.
\end{align*}

Ainsi, $v$ vérifie la condition de cocycle. 
\\

\item[(Contraintes d'unité (1))]
Soit maintenant $g : b \to b'$ dans $\mathdeuxcat{B}$. On veut vérifier la commutativité du diagramme

$$
\xymatrix{
v(g)
\ar@{=}[rr]
\ar@{=>}[drr]_{v_{g, 1_{b}}}
&&
v(g) 1_{v(b)}
\\
&&
v(g) v(1_{b})
\ar@{=>}[u]_{v(g) \CompDeuxZero v_{b}}
}
$$

Pour l'établir, il suffit de vérifier l'égalité 
$$
(p_{b'} \CompDeuxZero u((v(g) \CompDeuxZero v_{b}) v_{g, 1_{b}})) \alpha_{g} = (p_{b'} \CompDeuxZero u(1_{v(g)})) \alpha_{g}, 
$$
ce qui revient à démontrer l'égalité
$$
(p_{b'} \CompDeuxZero u((v(g) \CompDeuxZero v_{b}) v_{g, 1_{b}})) \alpha_{g} = \alpha_{g}.
$$

Considérons le diagramme
$$
\xymatrix{
u(v(b))
\ar[rr]^{u(v(1_{b}))}
\ar@/^2pc/[rrrr]^{u(v(g)v(1_{b}))}
\ar@/^4pc/[rrrr]^{u(v(g1_{b})) = u(v(g))}
\ar@/_2pc/[rr]^{u(1_{v(b)})}
\ar[dd]_{p_{b}}
&&
u(v(b))
\ar[rr]^{u(v(g))}
\ar[dd]^{p_{b}}
&&
u(v(b'))
\ar[dd]^{p_{b'}}
\\
\\
b
\ar[rr]_{1_{b}}
&&
b
\ar[rr]_{g}
&&
b'
&,
}
$$
et les \deux{}cellules 
$$
\alpha_{1_{b}} : p_{b} \Rightarrow p_{b} u(v(1_{b})),
$$ 
$$
\alpha_{g} : g p_{b} \Rightarrow p_{b'} u(v(g))
$$
et
$$
u(v_{g,1_{b}}) : u(v(g)) \Rightarrow u(v(g)v(1_{b})).
$$
Par définition de $v_{g,1_{b}}$, le diagramme

\begin{equation}\label{vg1b}
\xymatrix{
g p_{b} = g 1_{b} p_{b}
\ar@{=>}[rrrr]^{\alpha_{g} = \alpha_{g1_{b}}}
\ar@{=>}[dd]_{g \CompDeuxZero \alpha_{1_{b}}}
&&&&
p_{b'} u(v(g1_{b})) = p_{b'} u(v(g))
\ar@{=>}[dd]^{p_{b'} \CompDeuxZero u(v_{g,1_{b}})}
\\
\\
g p_{b} u(v(1_{b}))
\ar@{=>}[rr]_{\alpha_{g} \CompDeuxZero u(v(1_{b}))}
&&
p_{b'} u(v(g)) u(v(1_{b}))
\ar@{=>}[rr]_{p_{b'} \CompDeuxZero u_{v(g), v(1_{b})}}
&&
p_{b'} u(v(g)v(1_{b}))
}
\end{equation}
est commutatif, et la définition de $v_{b}$ assure l'égalité 
\begin{equation}\label{vb}
(p_{b} \CompDeuxZero u(v_{b})) \alpha_{1_{b}} = p_{b} \CompDeuxZero u_{v(b)}.
\end{equation}

La naturalité des \deux{}cellules structurales de composition de $u$ implique la commutativité du diagramme

\begin{equation}\label{UFonctoriel3}
\xymatrix{
u(v(g)) u(v(1_{b}))
\ar@{=>}[rrr]^{u_{v(g), v(1_{b})}}
\ar@{=>}[dd]_{u(v(g)) \CompDeuxZero u(v_{b})}
&&&
u(v(g)v(1_{b}))
\ar@{=>}[dd]^{u(v(g) \CompDeuxZero v_{b})}
\\
\\
u(v(g)) u(1_{v(b)})
\ar@{=>}[rrr]_{u_{v(g), 1_{v(b)}}}
&&&
u(v(g)1_{v(b)}) = u(v(g))
&.
}
\end{equation}

Les contraintes d'unité vérifiées par $u$ impliquent la commutativité du diagramme

\begin{equation}\label{UniteU}
\xymatrix{
u(v(g)) = u(v(g)1_{v(b)})
\ar@{=}[rrr]^{1_{u(v(g))}}
&&&
u(v(g)) 1_{u(v(b))}
\ar@{=>}[dd]^{u(v(g)) \CompDeuxZero u_{v(b)}}
\\
\\
&&&
u(v(g)) u(1_{v(b)})
\ar@{=>}[uulll]^{u_{v(g), 1_{v(b)}}}
&.
}
\end{equation}

L'égalité cherchée résulte de la suite d'égalités

$$
\begin{aligned}
&(p_{b'} \CompDeuxZero u((v(g) \CompDeuxZero v_{b}) v_{g, 1_{b}})) \alpha_{g}
\\
&= (p_{b'} \CompDeuxZero u(v(g) \CompDeuxZero v_{b})) \CompDeuxUn (p_{b'} \CompDeuxZero u(v_{g, 1_{b}})) \CompDeuxUn \alpha_{g}
\\
&= (p_{b'} \CompDeuxZero u(v(g) \CompDeuxZero v_{b})) \CompDeuxUn (p_{b'} \CompDeuxZero u_{v(g), v(1_{b})}) \CompDeuxUn (\alpha_{g} \CompDeuxZero u(v(1_{b}))) \CompDeuxUn (g \CompDeuxZero \alpha_{1_{b}}) \phantom{bla} \text{($cf.$ \ref{vg1b})}
\\
&= (p_{b'} \CompDeuxZero ((u(v(g) \CompDeuxZero v_{b})) \CompDeuxUn (u_{v(g), v(1_{b})}))) \CompDeuxUn (\alpha_{g} \CompDeuxZero u(v(1_{b}))) \CompDeuxUn (g \CompDeuxZero \alpha_{1_{b}})
\\
&= (p_{b'} \CompDeuxZero (u_{v(g), 1_{v(b)}} \CompDeuxUn (u(v(g)) \CompDeuxZero u(v_{b})))) \CompDeuxUn (\alpha_{g} \CompDeuxZero u(v(1_{b}))) \CompDeuxUn (g \CompDeuxZero \alpha_{1_{b}}) \phantom{bla} \text{($cf.$ \ref{UFonctoriel3})}
\\
&= (p_{b'} \CompDeuxZero u_{v(g), 1_{v(b)}}) \CompDeuxUn (p_{b'} u(v(g)) \CompDeuxZero u(v_{b})) \CompDeuxUn (\alpha_{g} \CompDeuxZero u(v(1_{b}))) \CompDeuxUn (g \CompDeuxZero \alpha_{1_{b}})
\\
&= (p_{b'} \CompDeuxZero u_{v(g), 1_{v(b)}}) \CompDeuxUn (\alpha_{g} \CompDeuxZero u(1_{v(b)})) \CompDeuxUn (gp_{b} \CompDeuxZero u(v_{b})) \CompDeuxUn (g \CompDeuxZero \alpha_{1_{b}}) \phantom{bla} \text{(loi d'échange)}
\\
&= (p_{b'} \CompDeuxZero u_{v(g), 1_{v(b)}}) \CompDeuxUn (\alpha_{g} \CompDeuxZero u(1_{v(b)})) \CompDeuxUn (g \CompDeuxZero ((p_{b} \CompDeuxZero u(v_{b})) \CompDeuxUn \alpha_{1_{b}}))
\\
&= (p_{b'} \CompDeuxZero u_{v(g), 1_{v(b)}}) \CompDeuxUn (\alpha_{g} \CompDeuxZero u(1_{v(b)})) \CompDeuxUn (gp_{b} \CompDeuxZero u_{v(b)}) \phantom{bla} \text{($cf.$ \ref{vb})}
\\
&= (p_{b'} \CompDeuxZero u_{v(g), 1_{v(b)}}) \CompDeuxUn (p_{b'} u(v(g)) \CompDeuxZero u_{v(b)}) \CompDeuxUn \alpha_{g} \phantom{bla} \text{(loi d'échange)}
\\
&= (p_{b'} \CompDeuxZero (u_{v(g), 1_{v(b)}} \CompDeuxUn (u(v(g)) \CompDeuxZero u_{v(b)}))) \CompDeuxUn \alpha_{g}
\\
&= (p_{b'} \CompDeuxZero 1_{u(v(g))}) \CompDeuxUn \alpha_{g} \phantom{bla} \text{($cf.$ \ref{UniteU})}
\\
&= \alpha_{g}.
\end{aligned}
$$

On a donc bien 
$$
(v(g) \CompDeuxZero v_{b}) v_{g, 1_{b}} = 1_{v(g)}.
$$

\item[(Contraintes d'unité (2))]
La commutativité du diagramme 
$$
\xymatrix{
v(g)
\ar@{=}[rr]
\ar@{=>}[drr]_{v_{1_{b'}, g}}
&&
1_{v(b')} v(g)
\\
&&
v(1_{b'}) v(g)
\ar@{=>}[u]_{v_{b'} \CompDeuxZero v(g)}
}
$$
s'établit de façon similaire à celle que l'on vient de voir. 
\\
\end{itemize}

Ainsi, étant donné un préadjoint à gauche lax $u : \mathdeuxcat{A} \to \mathdeuxcat{B}$, on a bien construit un \DeuxFoncteurCoLax{} $v : \mathdeuxcat{B} \to \mathdeuxcat{A}$. 
\end{paragr}

\begin{rem}\label{LiftBettiPower}
Contrairement à ce que nous pensions, ce résultat se trouve déjà dégagé dans la littérature : à dualité près, c'est \cite[proposition 4.2]{BettiPower}, en tenant compte de \cite[page 941]{BettiPower}.
\end{rem}

\begin{paragr}
Le \DeuxFoncteurCoLax{} $v$ construit ci-dessus ne semble pas devoir être un préadjoint à droite colax en général. Cependant, si $u$ est un \DeuxFoncteurStrict{}, il est possible de définir une \DeuxTransformationCoLax{} $uv \Rightarrow 1_{\mathdeuxcat{B}}$ comme suit. Pour tout objet $b$ de $\mathdeuxcat{B}$, la composante de cette \DeuxTransformationCoLax{} en $b$ est $p_{b}$ et, pour toute \un{}cellule $g : b \to b'$ de $\mathdeuxcat{B}$, la composante de cette \DeuxTransformationCoLax{} en $g$ est $\alpha_{g}$. Il résulte immédiatement des formules \ref{Defvbeta}, \ref{Defvb} et \ref{Defvg',g} que cela définit bien une \DeuxTransformationCoLax{} $uv \Rightarrow 1_{\mathdeuxcat{B}}$. 

La proposition \ref{PropPreadjointsParDeux} résume la situation, en tenant compte des dualités. 
\end{paragr}

\begin{prop}\label{PropPreadjointsParDeux}
Soit $u : \mathdeuxcat{A} \to \mathdeuxcat{B}$ un préadjoint à gauche lax. Les considérations de cette section permettent de construire, de façon canonique, un \DeuxFoncteurCoLax{} $v : \mathdeuxcat{B} \to \mathdeuxcat{A}$. Si $u$ est un \DeuxFoncteurStrict, il existe une \DeuxTransformationCoLax{} canonique $uv \Rightarrow 1_{\mathdeuxcat{B}}$. 

Dualement, pour tout préadjoint à gauche colax $u : \mathdeuxcat{A} \to \mathdeuxcat{B}$, il existe un \DeuxFoncteurLax{} $v : \mathdeuxcat{B} \to \mathdeuxcat{A}$ canonique et, si $u$ est un \DeuxFoncteurStrict{}, il existe une \DeuxTransformationLax{} canonique $uv \Rightarrow 1_{\mathdeuxcat{B}}$. 

Dualement, pour tout préadjoint à droite lax $u : \mathdeuxcat{A} \to \mathdeuxcat{B}$, il existe un \DeuxFoncteurCoLax{} \mbox{$v : \mathdeuxcat{B} \to \mathdeuxcat{A}$} canonique et, si $u$ est un \DeuxFoncteurStrict{}, il existe une \DeuxTransformationLax{} canonique $1_{\mathdeuxcat{B}} \Rightarrow uv$.   

Dualement, pour tout préadjoint à droite colax $u : \mathdeuxcat{A} \to \mathdeuxcat{B}$, il existe un \DeuxFoncteurLax{} \mbox{$v : \mathdeuxcat{B} \to \mathdeuxcat{A}$} canonique et, si $u$ est un \DeuxFoncteurStrict{}, il existe une \DeuxTransformationCoLax{} canonique $1_{\mathdeuxcat{B}} \Rightarrow uv$. 
\end{prop}

\section{Adjonctions}\label{SectionAdjonctions}

\begin{df}\label{DefAdjonctionLaxCoLax}
Une \emph{adjonction lax-colax} correspond aux données et conditions suivantes. \\

\begin{itemize}
\item[(Données 1)]
Des \deux{}catégories $\mathdeuxcat{A}$ et $\mathdeuxcat{B}$, un \DeuxFoncteurLax{} $u : \mathdeuxcat{A} \to \mathdeuxcat{B}$ et un \DeuxFoncteurCoLax{} $v : \mathdeuxcat{B} \to \mathdeuxcat{A}$. \\
\end{itemize}

\begin{itemize}
\item[(Données 2)]
Pour tout objet $b$ de $\mathdeuxcat{B}$, une \un{}cellule $p_{b} : u(v(b)) \to b$ dans $\mathdeuxcat{B}$. \\

\item[(Données 3)]
Pour toute \un{}cellule $g : b \to b'$ de $\mathdeuxcat{B}$, une \deux{}cellule
$$
p_{g} : g p_{b} \Rightarrow p_{b'} u(v(g))
$$
dans $\mathdeuxcat{B}$. 
\end{itemize}

\begin{itemize}
\item[(ALC 1)]
Pour tout couple d'objets $b$ et $b'$ de $\mathdeuxcat{B}$, pour tout couple de \un{}cellules $g$ et $g'$ de $b$ vers $b'$ dans $\mathdeuxcat{B}$, pour toute \deux{}cellule $\beta : g \Rightarrow g'$ dans $\mathdeuxcat{B}$, le diagramme
$$
\xymatrix{
g p_{b}
\ar@{=>}[rr]^{p_{g}}
\ar@{=>}[d]_{\beta \CompDeuxZero p_{b}}
&&
p_{b'} u(v(g))
\ar@{=>}[d]^{p_{b'} \CompDeuxZero u(v(\beta))}
\\
g' p_{b}
\ar@{=>}[rr]_{p_{g'}}
&&
p_{b'} u(v(g'))
}
$$
est commutatif ; autrement dit, l'égalité
$$
(p_{b'} \CompDeuxZero u(v(\beta))) \CompDeuxUn p_{g} = p_{g'} \CompDeuxUn (\beta \CompDeuxZero p_{b})
$$
est vérifiée. \\

\item[(ALC 2)]
Pour tout couple de \un{}cellules composables $g : b \to b'$ et $g' : b' \to b''$ dans $\mathdeuxcat{B}$, le diagramme 
$$
\xymatrix{
g' g p_{b}
\ar@{=>}[rr]^{g' \CompDeuxZero p_{g}}
\ar@{=>}[d]_{p_{g'g}}
&&
g' p_{b'} u(v(g))
\ar@{=>}[rr]^{p_{g'} \CompDeuxZero u(v(g))}
&&
p_{b''} u(v(g')) u(v(g))
\ar@{=>}[d]^{p_{b''} \CompDeuxZero u_{v(g'), v(g)}}
\\
p_{b''} u(v(g'g))
\ar@{=>}[rrrr]_{p_{b''} \CompDeuxZero u(v_{g', g})}
&&&&
p_{b''} u(v(g') v(g))
}
$$
est commutatif ; autrement dit, l'égalité
$$
(p_{b''} \CompDeuxZero u_{v(g'), v(g)}) \CompDeuxUn (p_{g'} \CompDeuxZero u(v(g))) \CompDeuxUn (g' \CompDeuxZero p_{g}) = (p_{b''} \CompDeuxZero u(v_{g', g})) \CompDeuxUn p_{g'g}
$$
est vérifiée. \\

\item[(ALC 3)]
Pour tout objet $b$ de $\mathdeuxcat{B}$, le diagramme
$$
\xymatrix{
p_{b}
\ar@{=>}[rr]^{p_{1_{b}}}
\ar@{=>}[rrd]_{p_{b} \CompDeuxZero u_{v(b)}}
&&
p_{b} u(v(1_{b}))
\ar@{=>}[d]^{p_{b} \CompDeuxZero u(v_{b})}
\\
&&
p_{b} u(1_{v(b)})
}
$$
est commutatif ; autrement dit, l'égalité
$$
(p_{b} \CompDeuxZero u(v_{b})) \CompDeuxUn p_{1_{b}} = p_{b} \CompDeuxZero u_{v(b)}
$$
est vérifiée. \\
\end{itemize}

\begin{itemize}
\item[(Données 4)]
Pour tout objet $a$ de $\mathdeuxcat{A}$, une \un{}cellule $q_{a} : a \to v(u(a))$ dans $\mathdeuxcat{A}$. \\

\item[(Données 5)]
Pour toute \un{}cellule $f : a \to a'$ de $\mathdeuxcat{A}$, une \deux{}cellule
$$
q_{f} : v(u(f)) q_{a} \Rightarrow q_{a'} f
$$
dans $\mathdeuxcat{A}$.
\\
\end{itemize}

\begin{itemize}
\item[(ALC 4)]
Pour tout couple d'objets $a$ et $a'$ de $\mathdeuxcat{A}$, pour tout couple de \un{}cellules $f$ et $f'$ de $a$ vers $a'$ dans $\mathdeuxcat{A}$, pour toute \deux{}cellule $\alpha : f \Rightarrow f'$ dans $\mathdeuxcat{A}$, le diagramme
$$
\xymatrix{
v(u(f)) q_{a}
\ar@{=>}[rr]^{q_{f}}
\ar@{=>}[d]_{v(u(\alpha)) \CompDeuxZero q_{a}}
&&
q_{a'} f
\ar@{=>}[d]^{q_{a'} \CompDeuxZero \alpha}
\\
v(u(f')) q_{a}
\ar@{=>}[rr]_{q_{f'}}
&&
q_{a'} f'
}
$$
est commutatif ; autrement dit, l'égalité
$$
q_{f'} \CompDeuxUn (v(u(\alpha)) \CompDeuxZero q_{a}) = (q_{a'} \CompDeuxZero \alpha) \CompDeuxUn q_{f} 
$$
est vérifiée. \\

\item[(ALC 5)]
Pour tout couple de \un{}cellules composables $f : a \to a'$ et $f' : a' \to a''$ de $\mathdeuxcat{A}$, le diagramme 
$$
\xymatrix{
v (u(f') u(f)) q_{a}
\ar@{=>}[rr]^{v_{u(f'), u(f)} \CompDeuxZero q_{a}}
\ar@{=>}[dd]_{v(u_{f',f}) \CompDeuxZero q_{a}}
&&
v(u(f')) v(u(f)) q_{a}
\ar@{=>}[d]^{v(u(f')) \CompDeuxZero q_{f}}
\\
&&
v(u(f')) q_{a'} f
\ar@{=>}[d]^{q_{f'} \CompDeuxZero f}
\\
v(u(f'f)) q_{a}
\ar@{=>}[rr]_{q_{f'f}}
&&
q_{a''} f' f
}
$$
est commutatif ; autrement dit, l'égalité
$$
(q_{f'} \CompDeuxZero f) \CompDeuxUn (v(u(f')) \CompDeuxZero q_{f}) \CompDeuxUn (v_{u(f'), u(f)} \CompDeuxZero q_{a}) = q_{f'f} \CompDeuxUn (v(u_{f',f}) \CompDeuxZero q_{a})
$$
est vérifiée. \\

\item[(ALC 6)]
Pour tout objet $a$ de $\mathdeuxcat{A}$, le diagramme
$$
\xymatrix{
v(1_{u(a)}) q_{a}
\ar@{=>}[rr]^{v(u_{a}) \CompDeuxZero q_{a}}
\ar@{=>}[rrd]_{v_{u(a)} \CompDeuxZero q_{a}}
&&
v(u(1_{a})) q_{a}
\ar@{=>}[d]^{q_{1_{a}}}
\\
&&
q_{a}
}
$$
est commutatif ; autrement dit, l'égalité
$$
q_{1_{a}} \CompDeuxUn (v(u_{a}) \CompDeuxZero q_{a}) = v_{u(a)} \CompDeuxZero q_{a}
$$
est vérifiée. \\
\end{itemize}

\begin{itemize}
\item[(Données 6)]
Pour tout objet $a$ de $\mathdeuxcat{A}$, une \deux{}cellule
$$
\sigma_{a} : 1_{u(a)} \Rightarrow p_{u(a)} u(q_{a})
$$
dans $\mathdeuxcat{B}$.
\end{itemize}

\begin{itemize}
\item[(ALC 7)]
Pour toute \un{}cellule $f : a \to a'$ de $\mathdeuxcat{A}$, le diagramme
$$
\xymatrix{
u(f)
\ar@{=>}[rrr]^{\sigma_{a'} \CompDeuxZero u(f)}
\ar@{=>}[dd]_{u(f) \CompDeuxZero \sigma_{a}}
&&&
p_{u(a')} u(q_{a'}) u(f)
\ar@{=>}[rrr]^{p_{u(a')} \CompDeuxZero u_{q_{a'}, f}}
&&&
p_{u(a')} u(q_{a'}f)
\\
\\
u(f) p_{u(a)} u(q_{a})
\ar@{=>}[rrr]_{p_{u(f)} \CompDeuxZero u(q_{a})}
&&&
p_{u(a')}
u(v(u(f))) u(q_{a})
\ar@{=>}[rrr]_{p_{u(a')} \CompDeuxZero u_{v(u(f)), q_{a}}}
&&&
p_{u(a')} u(v(u(f)) q_{a})
\ar@{=>}[uu]_{p_{u(a')} \CompDeuxZero u(q_{f})}
}
$$
est commutatif ; autrement dit, l'égalité
$$
(p_{u(a')} \CompDeuxZero u(q_{f})) \CompDeuxUn (p_{u(a')} \CompDeuxZero u_{v(u(f)), q_{a}}) \CompDeuxUn (p_{u(f)} \CompDeuxZero u(q_{a})) \CompDeuxUn (u(f) \CompDeuxZero \sigma_{a}) = (p_{u(a')} \CompDeuxZero u_{q_{a'}, f}) \CompDeuxUn (\sigma_{a'} \CompDeuxZero u(f))
$$
est vérifiée. \\
\end{itemize}

\begin{itemize}
\item[(Données 7)]
Pour tout objet $b$ de $\mathdeuxcat{B}$, une \deux{}cellule
$$
\tau_{b} : v(p_{b}) q_{v(b)} \Rightarrow 1_{v(b)}. \\
$$
dans $\mathdeuxcat{A}$.
\end{itemize}

\begin{itemize}
\item[(ALC 8)]
Pour toute \un{}cellule $g : b \to b'$ de $\mathdeuxcat{B}$, le diagramme
$$
\xymatrix{
v(g p_{b}) q_{v(b)}
\ar@{=>}[rrr]^{v_{g,p_{b}} \CompDeuxZero q_{v(b)}}
\ar@{=>}[dd]_{v(p_{g}) \CompDeuxZero q_{v(b)}}
&&&
v(g) v(p_{b}) q_{v(b)}
\ar@{=>}[rrr]^{v(g) \CompDeuxZero \tau_{b}}
&&&
v(g)
\\
\\
v(p_{b'} u(v(g))) q_{v(b)}
\ar@{=>}[rrr]_{v_{p_{b'}, u(v(g))} \CompDeuxZero q_{v(b)}}
&&&
v(p_{b'}) v(u(v(g))) q_{v(b)}
\ar@{=>}[rrr]_{v(p_{b'}) \CompDeuxZero q_{v(g)}}
&&&
v(p_{b'}) q_{v(b')} v(g)
\ar@{=>}[uu]_{\tau_{b'} \CompDeuxZero v(g)}
}
$$
est commutatif ; autrement dit, l'égalité
$$
(\tau_{b'} \CompDeuxZero v(g)) \CompDeuxUn (v(p_{b'}) \CompDeuxZero q_{v(g)}) \CompDeuxUn (v_{p_{b'}, u(v(g))} \CompDeuxZero q_{v(b)}) \CompDeuxUn (v(p_{g}) \CompDeuxZero q_{v(b)}) = (v(g) \CompDeuxZero \tau_{b}) \CompDeuxUn (v_{g,p_{b}} \CompDeuxZero q_{v(b)})
$$
est vérifiée. \\
\end{itemize}

\begin{itemize}
\item[(ALC 9)]
Pour tout objet $b$ de $\mathdeuxcat{B}$, le diagramme
$$
\xymatrix{
p_{b}
\ar@{=>}[rrr]^{p_{b} \CompDeuxZero \sigma_{v(b)}}
\ar@{=>}[rrrrrrdd]_{p_{b} \CompDeuxZero u_{v(b)}}
&&&
p_{b} p_{u(v(b))} u(q_{v(b)})
\ar@{=>}[rrr]^{p_{p_{b}} \CompDeuxZero u(q_{v(b)})}
&&&
p_{b} u(v(p_{b})) u(q_{v(b)})
\ar@{=>}[d]^{p_{b} \CompDeuxZero u_{v(p_{b}), q_{v(b)}}}
\\
&&&&&&
p_{b} u(v(p_{b}) q_{v(b)})
\ar@{=>}[d]^{p_{b} \CompDeuxZero u(\tau_{b})}
\\
&&&&&&
p_{b} u(1_{v(b)})
}
$$
est commutatif ; autrement dit, l'égalité
$$
(p_{b} \CompDeuxZero u(\tau_{b})) \CompDeuxUn (p_{b} \CompDeuxZero u_{v(p_{b}), q_{v(b)}}) \CompDeuxUn (p_{p_{b}} \CompDeuxZero u(q_{v(b)})) \CompDeuxUn (p_{b} \CompDeuxZero \sigma_{v(b)}) = p_{b} \CompDeuxZero u_{v(b)}
$$
est vérifiée. \\
\end{itemize}

\begin{itemize}
\item[(ALC 10)]
Pour tout objet $a$ de $\mathdeuxcat{A}$, le diagramme
$$
\xymatrix{
v(1_{u(a)}) q_{a}
\ar@{=>}[rrr]^{v(\sigma_{a}) \CompDeuxZero q_{a}}
\ar@{=>}[rrrrrrdd]_{v_{u(a)} \CompDeuxZero q_{a}}
&&&
v(p_{u(a)} u(q_{a})) q_{a}
\ar@{=>}[rrr]^{v_{p_{u(a)}, u(q_{a})} \CompDeuxZero q_{a}}
&&&
v(p_{u(a)}) v(u(q_{a})) q_{a}
\ar@{=>}[d]^{v(p_{u(a)}) \CompDeuxZero q_{q_{a}}}
\\
&&&&&&
v(p_{u(a)}) q_{v(u(a))} q_{a}
\ar@{=>}[d]^{\tau_{u(a)} \CompDeuxZero q_{a}}
\\
&&&&&&
q_{a}
}
$$
est commutatif ; autrement dit, l'égalité
$$
(\tau_{u(a)} \CompDeuxZero q_{a}) \CompDeuxUn (v(p_{u(a)}) \CompDeuxZero q_{q_{a}}) \CompDeuxUn (v_{p_{u(a)}, u(q_{a})} \CompDeuxZero q_{a}) \CompDeuxUn (v(\sigma_{a}) \CompDeuxZero q_{a}) = v_{u(a)} \CompDeuxZero q_{a}
$$
est vérifiée. 
\end{itemize}

On dira que le couple $(u,v)$ constitue une adjonction lax-colax, $u$ étant l'\emph{adjoint à gauche lax de $v$} et $v$ l'\emph{adjoint à droite colax de $u$}. 
\end{df}

\begin{rem}\label{RemarqueVerity}
La notion que nous désignons sous le nom d'adjonction lax-colax, sur une suggestion de Maltsiniotis, se trouve en fait déjà dégagée par Verity dans sa thèse \cite{Verity}, ce que nous ignorions lors de la rédaction du présent travail ; c'est Steve Lack qui nous l'a suggéré. Plus précisément, en gardant à l'esprit la convention prise page 29 de \cite{Verity}, la condition $(ii)$ de \cite[théorème 1.1.6]{Verity} stipule, à dualité près, que le couple de morphismes considérés vérifie les axiomes de notre définition \ref{DefAdjonctionLaxCoLax}. En particulier, du fait de l'équivalence des conditions $(i)$ et $(ii)$ de \cite[théorème 1.1.6]{Verity}, la condition $(i)$ de ce dernier énoncé fournit une autre caractérisation des adjonctions lax-colax, dans l'esprit des définitions de \cite{BettiPower}, qu'elle renforce. 
\end{rem}

%

\begin{prop}\label{AdjointPreadjoint}
Un adjoint à gauche lax est un préadjoint à gauche lax. 
\end{prop}

\begin{proof}
C'est immédiat en vertu de la remarque \ref{RemarqueVerity} ; cependant, sans avoir encore connaissance de \cite{Verity}, nous avons rédigé une démonstration détaillée de ce résultat. Ce dernier mettant en jeu deux notions importantes figurant dans le présent travail, nous conservons cette démonstration en espérant que le lecteur y trouvera quelque utilité. 

Soit $(u : \mathdeuxcat{A} \to \mathdeuxcat{B}, v : \mathdeuxcat{B} \to \mathdeuxcat{A})$ une adjonction lax-colax. Il s'agit de vérifier que, pour tout objet $b$ de $\mathdeuxcat{B}$, la \deux{}catégorie $\TrancheLax{\mathdeuxcat{A}}{u}{b}$ admet un objet admettant un objet initial. Fixons donc un objet $b$ de $\mathdeuxcat{B}$. 

Le couple
$$
(v(b), p_{b} : u(v(b)) \to b)
$$
est un objet de la \deux{}catégorie $\TrancheLax{\mathdeuxcat{A}}{u}{b}$. 

Pour tout objet $(a, r : u(a) \to b)$ de $\TrancheLax{\mathdeuxcat{A}}{u}{b}$, le couple
$$
(v(r) q_{a}, (p_{b} \CompDeuxZero u_{v(r), q_{a}}) \CompDeuxUn (p_{r} \CompDeuxZero u(q_{a})) \CompDeuxUn (r \CompDeuxZero \sigma_{a}))
$$
est une \un{}cellule de $(a,r)$ vers $(v(b), p_{b})$ dans $\TrancheLax{\mathdeuxcat{A}}{u}{b}$. 

Soit $(f : a \to v(b), \gamma : r \Rightarrow p_{b} u(f))$ une \un{}cellule arbitraire de $(a,r)$ vers $(v(b), p_{b})$ dans $\TrancheLax{\mathdeuxcat{A}}{u}{b}$. Montrons que
$$
(\tau_{b} \CompDeuxZero f) \CompDeuxUn (v(p_{b}) \CompDeuxZero q_{f}) \CompDeuxUn (v_{p_{b}, u(f)} \CompDeuxZero q_{a}) \CompDeuxUn (v(\gamma) \CompDeuxZero q_{a})
$$
est une \deux{}cellule de $(v(r) q_{a}, (p_{b} \CompDeuxZero u_{v(r), q_{a}}) \CompDeuxUn (p_{r} \CompDeuxZero u(q_{a})) \CompDeuxUn (r \CompDeuxZero \sigma_{a}))$ vers $(f, \gamma)$ dans $\TrancheLax{\mathdeuxcat{A}}{u}{b}$. Il s'agit de vérifier l'égalité
$$
(p_{b} \CompDeuxZero (u((\tau_{b} \CompDeuxZero f) \CompDeuxUn (v(p_{b}) \CompDeuxZero q_{f}) \CompDeuxUn (v_{p_{b}, u(f)} \CompDeuxZero q_{a}) \CompDeuxUn (v(\gamma) \CompDeuxZero q_{a})))) \CompDeuxUn (p_{b} \CompDeuxZero u_{v(r), q_{a}}) \CompDeuxUn (p_{r} \CompDeuxZero u(q_{a})) \CompDeuxUn (r \CompDeuxZero \sigma_{a}) = \gamma.
$$
Elle résulte des égalités suivantes. 
\begin{align*}
&(p_{b} \CompDeuxZero (u((\tau_{b} \CompDeuxZero f) \CompDeuxUn (v(p_{b}) \CompDeuxZero q_{f}) \CompDeuxUn (v_{p_{b}, u(f)} \CompDeuxZero q_{a}) \CompDeuxUn (v(\gamma) \CompDeuxZero q_{a})))) \CompDeuxUn (p_{b} \CompDeuxZero u_{v(r), q_{a}}) \CompDeuxUn (p_{r} \CompDeuxZero u(q_{a})) \CompDeuxUn (r \CompDeuxZero \sigma_{a}) 
\\
&=(p_{b} \CompDeuxZero u(\tau_{b} \CompDeuxZero f)) \CompDeuxUn (p_{b} \CompDeuxZero u(v(p_{b}) \CompDeuxZero q_{f})) \CompDeuxUn (p_{b} \CompDeuxZero u(v_{p_{b}, u(f)} \CompDeuxZero q_{a})) \CompDeuxUn (p_{b} \CompDeuxZero u(v(\gamma) \CompDeuxZero q_{a})) \CompDeuxUn (p_{b} \CompDeuxZero u_{v(r), q_{a}}) \CompDeuxUn 
\\
&\phantom{=1}(p_{r} \CompDeuxZero u(q_{a})) \CompDeuxUn (r \CompDeuxZero \sigma_{a}) 
\\
&=(p_{b} \CompDeuxZero u(\tau_{b} \CompDeuxZero f)) \CompDeuxUn (p_{b} \CompDeuxZero u(v(p_{b}) \CompDeuxZero q_{f})) \CompDeuxUn (p_{b} \CompDeuxZero u(v_{p_{b}, u(f)} \CompDeuxZero q_{a})) \CompDeuxUn (p_{b} \CompDeuxZero u_{v(p_{b} u(f)), q_{a}}) \CompDeuxUn (p_{b} \CompDeuxZero u(v(\gamma)) \CompDeuxZero u(q_{a})) \CompDeuxUn 
\\
&\phantom{=1}(p_{r} \CompDeuxZero u(q_{a})) \CompDeuxUn (r \CompDeuxZero \sigma_{a}) 
\\
&\phantom{=1}\text{(naturalité des \deux{}cellules structurales de composition de $u$)}
\\
&= (p_{b} \CompDeuxZero u(\tau_{b} \CompDeuxZero f)) \CompDeuxUn (p_{b} \CompDeuxZero u(v(p_{b}) \CompDeuxZero q_{f})) \CompDeuxUn (p_{b} \CompDeuxZero u_{v(p_{b})v(u(f)), q_{a}}) \CompDeuxUn (p_{b} \CompDeuxZero u(v_{p_{b}, u(f)}) \CompDeuxZero u(q_{a})) \CompDeuxUn 
\\
&\phantom{=1}(p_{b} \CompDeuxZero u(v(\gamma)) \CompDeuxZero u(q_{a})) \CompDeuxUn (p_{r} \CompDeuxZero u(q_{a})) \CompDeuxUn (r \CompDeuxZero \sigma_{a}) 
\\
&\phantom{=1}\text{(naturalité des \deux{}cellules structurales de composition de $u$)}
\\
&= (p_{b} \CompDeuxZero u(\tau_{b} \CompDeuxZero f)) \CompDeuxUn (p_{b} \CompDeuxZero u(v(p_{b}) \CompDeuxZero q_{f})) \CompDeuxUn (p_{b} \CompDeuxZero u_{v(p_{b})v(u(f)), q_{a}}) \CompDeuxUn (p_{b} \CompDeuxZero u(v_{p_{b}, u(f)}) \CompDeuxZero u(q_{a})) \CompDeuxUn
\\
&\phantom{=1}(p_{p_{b} u(f)} \CompDeuxZero u(q_{a})) \CompDeuxUn (\gamma \CompDeuxZero p_{u(a)} u(q_{a})) \CompDeuxUn (r \CompDeuxZero \sigma_{a}) 
\\
&\phantom{=1}\text{(ALC 1)}
\\
&= (p_{b} \CompDeuxZero u(\tau_{b} \CompDeuxZero f)) \CompDeuxUn (p_{b} \CompDeuxZero u(v(p_{b}) \CompDeuxZero q_{f})) \CompDeuxUn (p_{b} \CompDeuxZero u_{v(p_{b})v(u(f)), q_{a}}) \CompDeuxUn
(p_{b} \CompDeuxZero u_{v(p_{b}), v(u(f))} \CompDeuxZero u(q_{a})) \CompDeuxUn 
\\
&\phantom{=1}(p_{p_{b}} \CompDeuxZero u(v(u(f))) \CompDeuxZero u(q_{a})) \CompDeuxUn (p_{b} \CompDeuxZero p_{u(f)} \CompDeuxZero u(q_{a}))
\CompDeuxUn (\gamma \CompDeuxZero p_{u(a)} u(q_{a})) \CompDeuxUn (r \CompDeuxZero \sigma_{a})
\\
&\phantom{=1}\text{(ALC 2)}
\\
&= (p_{b} \CompDeuxZero u(\tau_{b} \CompDeuxZero f)) \CompDeuxUn (p_{b} \CompDeuxZero u(v(p_{b}) \CompDeuxZero q_{f})) \CompDeuxUn (p_{b} \CompDeuxZero u_{v(p_{b}), v(u(f))q_{a}}) \CompDeuxUn (p_{b} u(v(p_{b})) \CompDeuxZero u_{v(u(f)), q_{a}}) \CompDeuxUn  
\\
&\phantom{=1}(p_{p_{b}} \CompDeuxZero u(v(u(f))) u(q_{a})) \CompDeuxUn (p_{b} \CompDeuxZero p_{u(f)} \CompDeuxZero u(q_{a}))
\CompDeuxUn (\gamma \CompDeuxZero p_{u(a)} u(q_{a})) \CompDeuxUn (r \CompDeuxZero \sigma_{a})
\\
&\phantom{=1}\text{(condition de cocycle pour $u$)}
\\
&= (p_{b} \CompDeuxZero u(\tau_{b} \CompDeuxZero f)) \CompDeuxUn (p_{b} \CompDeuxZero u(v(p_{b}) \CompDeuxZero q_{f})) \CompDeuxUn (p_{b} \CompDeuxZero u_{v(p_{b}), v(u(f))q_{a}}) \CompDeuxUn (p_{b} u(v(p_{b})) \CompDeuxZero u_{v(u(f)), q_{a}}) \CompDeuxUn  
\\
&\phantom{=1}(p_{p_{b}} \CompDeuxZero u(v(u(f))) u(q_{a})) \CompDeuxUn (p_{b} \CompDeuxZero p_{u(f)} \CompDeuxZero u(q_{a})) \CompDeuxUn (p_{b} u(f) \CompDeuxZero \sigma_{a}) \CompDeuxUn \gamma 
\\ 
&\phantom{=1}\text{(loi d'échange)}
\\
&= (p_{b} \CompDeuxZero u(\tau_{b} \CompDeuxZero f)) \CompDeuxUn (p_{b} \CompDeuxZero u(v(p_{b}) \CompDeuxZero q_{f})) \CompDeuxUn (p_{b} \CompDeuxZero u_{v(p_{b}), v(u(f))q_{a}}) \CompDeuxUn (p_{p_{b}} \CompDeuxZero u_{v(u(f)), q_{a}}) \CompDeuxUn (p_{b} \CompDeuxZero p_{u(f)} \CompDeuxZero u(q_{a})) \CompDeuxUn
\\
&\phantom{=1}(p_{b} u(f) \CompDeuxZero \sigma_{a}) \CompDeuxUn \gamma 
\\
&\phantom{=1}\text{(loi d'échange)}
\\
&= (p_{b} \CompDeuxZero u(\tau_{b} \CompDeuxZero f)) \CompDeuxUn (p_{b} \CompDeuxZero u_{v(p_{b}), q_{v(b)}f}) \CompDeuxUn (p_{b} u(v(p_{b})) \CompDeuxZero u(q_{f})) \CompDeuxUn (p_{p_{b}} \CompDeuxZero u_{v(u(f)), q_{a}}) \CompDeuxUn (p_{b} \CompDeuxZero p_{u(f)} \CompDeuxZero u(q_{a})) \CompDeuxUn
\\
&\phantom{=1}(p_{b} u(f) \CompDeuxZero \sigma_{a}) \CompDeuxUn \gamma
\\
&\phantom{=1}\text{(naturalité des \deux{}cellules de composition de $u$)}
\\
&= (p_{b} \CompDeuxZero u(\tau_{b} \CompDeuxZero f)) \CompDeuxUn (p_{b} \CompDeuxZero u_{v(p_{b}), q_{v(b)}f}) \CompDeuxUn (p_{p_{b}} \CompDeuxZero u(q_{v(b)} f)) \CompDeuxUn (p_{b} p_{u(v(b))} \CompDeuxZero u(q_{f})) \CompDeuxUn (p_{b} p_{u(v(b))} \CompDeuxZero u_{v(u(f)), q_{a}}) \CompDeuxUn
\\
&\phantom{=1}(p_{b} \CompDeuxZero p_{u(f)} \CompDeuxZero u(q_{a})) \CompDeuxUn (p_{b} u(f) \CompDeuxZero \sigma_{a}) \CompDeuxUn \gamma 
\\
&\phantom{=1}\text{(loi d'échange)}
\\
&= (p_{b} \CompDeuxZero u(\tau_{b} \CompDeuxZero f)) \CompDeuxUn (p_{b} \CompDeuxZero u_{v(p_{b}), q_{v(b)}f}) \CompDeuxUn (p_{p_{b}} \CompDeuxZero u(q_{v(b)} f)) \CompDeuxUn (p_{b} p_{u(v(b))} \CompDeuxZero u_{q_{v(b)}, f}) \CompDeuxUn (p_{b} \CompDeuxZero \sigma_{v(b)} \CompDeuxZero u(f))\CompDeuxUn \gamma
\\
&\phantom{=1}\text{(ALC 7)}
\\
&= (p_{b} \CompDeuxZero u(\tau_{b} \CompDeuxZero f)) \CompDeuxUn (p_{b} \CompDeuxZero u_{v(p_{b}), q_{v(b)}f}) \CompDeuxUn (p_{b} u(v(p_{b})) \CompDeuxZero u_{q_{v(b)}, f}) \CompDeuxUn (p_{p_{b}} \CompDeuxZero u(q_{v(b)}) u(f)) \CompDeuxUn (p_{b} \CompDeuxZero \sigma_{v(b)} \CompDeuxZero u(f))\CompDeuxUn \gamma
\\
&\phantom{=1}\text{(loi d'échange)}
\\
&= (p_{b} \CompDeuxZero u(\tau_{b} \CompDeuxZero f)) \CompDeuxUn (p_{b} \CompDeuxZero u_{v(p_{b})q_{v(b)}, f}) \CompDeuxUn (p_{b} \CompDeuxZero u_{v(p_{b}), q_{v(b)}} \CompDeuxZero u(f)) \CompDeuxUn (p_{p_{b}} \CompDeuxZero u(q_{v(b)}) u(f)) \CompDeuxUn (p_{b} \CompDeuxZero \sigma_{v(b)} \CompDeuxZero u(f))\CompDeuxUn \gamma
\\
&\phantom{=1}\text{(condition de cocycle pour $u$)}
\\
&= (p_{b} \CompDeuxZero u_{1_{v(b)}, f}) \CompDeuxUn (p_{b} \CompDeuxZero u(\tau_{b}) \CompDeuxZero u(f)) \CompDeuxUn (p_{b} \CompDeuxZero u_{v(p_{b}), q_{v(b)}} \CompDeuxZero u(f)) \CompDeuxUn (p_{p_{b}} \CompDeuxZero u(q_{v(b)}) u(f)) \CompDeuxUn (p_{b} \CompDeuxZero \sigma_{v(b)} \CompDeuxZero u(f)) \CompDeuxUn \gamma
\\
&\phantom{=1}\text{(naturalité des \deux{}cellules de composition de $u$)}
\\
&=(p_{b} \CompDeuxZero u_{1_{v(b)}, f}) \CompDeuxUn (p_{b} \CompDeuxZero u_{v(b)} \CompDeuxZero u(f)) \CompDeuxUn \gamma
\\
&\phantom{=1}\text{(ALC 9)}
\\
&= (p_{b} \CompDeuxZero (u_{1_{v(b)}, f} \CompDeuxUn (u_{v(b)} \CompDeuxZero u(f)))) \CompDeuxUn \gamma
\\
&= (p_{b} \CompDeuxZero 1_{u(f)}) \CompDeuxUn \gamma 
\\
&\phantom{=1}\text{(axiome d'unité pour $u$)}
\\
&= 1_{p_{b} u(f)} \CompDeuxUn \gamma
\\
&= \gamma.
\end{align*}

Par conséquent, comme annoncé,
$
(\tau_{b} \CompDeuxZero f) \CompDeuxUn (v(p_{b}) \CompDeuxZero q_{f}) \CompDeuxUn (v_{p_{b}, u(f)} \CompDeuxZero q_{a}) \CompDeuxUn (v(\gamma) \CompDeuxZero q_{a})
$
est une \deux{}cellule de $(v(r) q_{a}, (p_{b} \CompDeuxZero u_{v(r), q_{a}}) \CompDeuxUn (p_{r} \CompDeuxZero u(q_{a})) \CompDeuxUn (r \CompDeuxZero \sigma_{a}))$ vers $(f, \gamma)$ dans $\TrancheLax{\mathdeuxcat{A}}{u}{b}$. Il reste à montrer que c'est la seule. Soit donc 
$
\alpha
$
une \deux{}cellule de $(v(r) q_{a}, (p_{b} \CompDeuxZero u_{v(r), q_{a}}) \CompDeuxUn (p_{r} \CompDeuxZero u(q_{a})) \CompDeuxUn (r \CompDeuxZero \sigma_{a}))$ vers $(f, \gamma)$ dans $\TrancheLax{\mathdeuxcat{A}}{u}{b}$. C'est donc une \deux{}cellule de $v(r) q_{a}$ vers $f$ dans $\mathdeuxcat{A}$ telle que l'égalité
$$
(p_{b} \CompDeuxZero u(\alpha)) \CompDeuxUn (p_{b} \CompDeuxZero u_{v(r), q_{a}}) \CompDeuxUn (p_{r} \CompDeuxZero u(q_{a})) \CompDeuxUn (r \CompDeuxZero \sigma_{a}) = \gamma
$$
soit vérifiée. Alors,
\begin{align*}
&(\tau_{b} \CompDeuxZero f) \CompDeuxUn (v(p_{b}) \CompDeuxZero q_{f}) \CompDeuxUn (v_{p_{b}, u(f)} \CompDeuxZero q_{a}) \CompDeuxUn (v(\gamma) \CompDeuxZero q_{a})
\\
&= (\tau_{b} \CompDeuxZero f) \CompDeuxUn (v(p_{b}) \CompDeuxZero q_{f}) \CompDeuxUn (v_{p_{b}, u(f)} \CompDeuxZero q_{a}) \CompDeuxUn (v((p_{b} \CompDeuxZero u(\alpha)) \CompDeuxUn (p_{b} \CompDeuxZero u_{v(r), q_{a}}) \CompDeuxUn (p_{r} \CompDeuxZero u(q_{a})) \CompDeuxUn (r \CompDeuxZero \sigma_{a})) \CompDeuxZero q_{a})
\\
&\phantom{=1}\text{(par hypothèse)}
\\
&= (\tau_{b} \CompDeuxZero f) \CompDeuxUn (v(p_{b}) \CompDeuxZero q_{f}) \CompDeuxUn (v_{p_{b}, u(f)} \CompDeuxZero q_{a}) \CompDeuxUn (v(p_{b} \CompDeuxZero u(\alpha)) \CompDeuxZero q_{a}) \CompDeuxUn (v(p_{b} \CompDeuxZero u_{v(r), q_{a}}) \CompDeuxZero q_{a}) \CompDeuxUn (v(p_{r} \CompDeuxZero u(q_{a})) \CompDeuxZero q_{a}) \CompDeuxUn
\\
&\phantom{=1}(v(r \CompDeuxZero \sigma_{a}) \CompDeuxZero q_{a}) 
\\
&= (\tau_{b} \CompDeuxZero f) \CompDeuxUn (v(p_{b}) \CompDeuxZero q_{f}) \CompDeuxUn (v(p_{b}) \CompDeuxZero v(u(\alpha)) \CompDeuxZero q_{a}) \CompDeuxUn (v_{p_{b}, u(v(r)q_{a})} \CompDeuxZero q_{a}) \CompDeuxUn (v(p_{b} \CompDeuxZero u_{v(r), q_{a}}) \CompDeuxZero q_{a}) \CompDeuxUn
\\
&\phantom{=1}(v(p_{r} \CompDeuxZero u(q_{a})) \CompDeuxZero q_{a}) \CompDeuxUn (v(r \CompDeuxZero \sigma_{a}) \CompDeuxZero q_{a}) 
\\
&\phantom{=1}\text{(naturalité des \deux{}cellules structurales de composition de $v$)}
\\
&= (\tau_{b} \CompDeuxZero f) \CompDeuxUn (v(p_{b}) q_{v(b)} \CompDeuxZero \alpha) \CompDeuxUn (v(p_{b}) \CompDeuxZero q_{v(r)q_{a}}) \CompDeuxUn (v_{p_{b}, u(v(r)q_{a})} \CompDeuxZero q_{a}) \CompDeuxUn (v(p_{b} \CompDeuxZero u_{v(r), q_{a}}) \CompDeuxZero q_{a}) \CompDeuxUn 
\\
&\phantom{=1}(v(p_{r} \CompDeuxZero u(q_{a})) \CompDeuxZero q_{a}) \CompDeuxUn (v(r \CompDeuxZero \sigma_{a}) \CompDeuxZero q_{a}) 
\\
&\phantom{=1}\text{(ALC 4)}
\\
&= (\tau_{b} \CompDeuxZero f) \CompDeuxUn (v(p_{b}) q_{v(b)} \CompDeuxZero \alpha) \CompDeuxUn (v(p_{b}) \CompDeuxZero q_{v(r)q_{a}}) \CompDeuxUn (v(p_{b}) \CompDeuxZero v(u_{v(r), q_{a}}) \CompDeuxZero q_{a}) \CompDeuxUn (v_{p_{b}, u(v(r))u(q_{a})} \CompDeuxZero q_{a}) 
\\
&\phantom{=1}\CompDeuxUn (v(p_{r} \CompDeuxZero u(q_{a})) \CompDeuxZero q_{a}) \CompDeuxUn (v(r \CompDeuxZero \sigma_{a}) \CompDeuxZero q_{a}) 
\\
&\phantom{=1}{\text{(naturalité des \deux{}cellules structurales de composition de $v$)}}
\\
&= (\tau_{b} \CompDeuxZero f) \CompDeuxUn (v(p_{b}) q_{v(b)} \CompDeuxZero \alpha) \CompDeuxUn (v(p_{b}) \CompDeuxZero q_{v(r)} \CompDeuxZero q_{a}) \CompDeuxUn (v(p_{b}) v(u(v(r))) \CompDeuxZero q_{q_{a}}) \CompDeuxUn (v(p_{b}) \CompDeuxZero v_{u(v(r)), u(q_{a})} \CompDeuxZero q_{a})
\\
&\phantom{=1}\CompDeuxUn (v_{p_{b}, u(v(r))u(q_{a})} \CompDeuxZero q_{a}) \CompDeuxUn (v(p_{r} \CompDeuxZero u(q_{a})) \CompDeuxZero q_{a}) \CompDeuxUn (v(r \CompDeuxZero \sigma_{a}) \CompDeuxZero q_{a}) 
\\ 
&\phantom{=1}\text{(ALC 5)}
\\
&= (\tau_{b} \CompDeuxZero f) \CompDeuxUn (v(p_{b}) q_{v(b)} \CompDeuxZero \alpha) \CompDeuxUn (v(p_{b}) \CompDeuxZero q_{v(r)} \CompDeuxZero q_{a}) \CompDeuxUn (v(p_{b}) v(u(v(r))) \CompDeuxZero q_{q_{a}}) \CompDeuxUn (v_{p_{b}, u(v(r))} \CompDeuxZero v(u(q_{a})) q_{a}) 
\\
&\phantom{=1}\CompDeuxUn (v_{p_{b} u(v(r)), u(q_{a})} \CompDeuxZero q_{a}) \CompDeuxUn (v(p_{r} \CompDeuxZero u(q_{a})) \CompDeuxZero q_{a}) \CompDeuxUn (v(r \CompDeuxZero \sigma_{a}) \CompDeuxZero q_{a}) 
\\ 
&\phantom{=1}\text{(condition de cocycle pour $v$)}
\\
&= (\tau_{b} \CompDeuxZero f) \CompDeuxUn (v(p_{b}) q_{v(b)} \CompDeuxZero \alpha) \CompDeuxUn (v(p_{b}) \CompDeuxZero q_{v(r)} \CompDeuxZero q_{a}) \CompDeuxUn (v(p_{b}) v(u(v(r))) \CompDeuxZero q_{q_{a}}) \CompDeuxUn (v_{p_{b}, u(v(r))} \CompDeuxZero v(u(q_{a})) q_{a}) 
\\
&\phantom{=1}\CompDeuxUn (v(p_{r}) \CompDeuxZero v(u(q_{a})) q_{a}) \CompDeuxUn (v_{rp_{u(a)}, u(q_{a})} \CompDeuxZero q_{a}) \CompDeuxUn (v(r \CompDeuxZero \sigma_{a}) \CompDeuxZero q_{a}) 
\\
&\phantom{=1}\text{(naturalité des \deux{}cellules structurales de composition de $v$)}
\\
&= (\tau_{b} \CompDeuxZero f) \CompDeuxUn (v(p_{b}) q_{v(b)} \CompDeuxZero \alpha) \CompDeuxUn (v(p_{b}) \CompDeuxZero q_{v(r)} \CompDeuxZero q_{a}) \CompDeuxUn (v_{p_{b}, u(v(r))} \CompDeuxZero q_{v(u(a))} q_{a}) \CompDeuxUn (v(p_{r}) \CompDeuxZero q_{q_{a}}) \CompDeuxUn (v_{rp_{u(a)}, u(q_{a})} \CompDeuxZero q_{a})
\\
&\phantom{=1}\CompDeuxUn (v(r \CompDeuxZero \sigma_{a}) \CompDeuxZero q_{a}) 
\\
&\phantom{=1}\text{(loi d'échange)}
\\
&= (\tau_{b} \CompDeuxZero f) \CompDeuxUn (v(p_{b}) q_{v(b)} \CompDeuxZero \alpha) \CompDeuxUn (v(p_{b}) \CompDeuxZero q_{v(r)} \CompDeuxZero q_{a}) \CompDeuxUn (v_{p_{b}, u(v(r))} \CompDeuxZero q_{v(u(a))} q_{a}) \CompDeuxUn (v(p_{r}) \CompDeuxZero q_{v(u(a))} q_{a}) 
\\
&\phantom{=1}\CompDeuxUn (v(rp_{u(a)}) \CompDeuxZero q_{q_{a}}) \CompDeuxUn (v_{rp_{u(a)}, u(q_{a})} \CompDeuxZero q_{a}) \CompDeuxUn (v(r \CompDeuxZero \sigma_{a}) \CompDeuxZero q_{a})
\\
&\phantom{=1}\text{(loi d'échange)}
\\
&= \alpha \CompDeuxUn (\tau_{b} \CompDeuxZero v(r) q_{a}) \CompDeuxUn (v(p_{b}) \CompDeuxZero q_{v(r)} \CompDeuxZero q_{a}) \CompDeuxUn (v_{p_{b}, u(v(r))} \CompDeuxZero q_{v(u(a))} q_{a}) \CompDeuxUn (v(p_{r}) \CompDeuxZero q_{v(u(a))} q_{a}) \CompDeuxUn (v(rp_{u(a)}) \CompDeuxZero q_{q_{a}})  
\\
&\phantom{=1}\CompDeuxUn (v_{rp_{u(a)}, u(q_{a})} \CompDeuxZero q_{a}) \CompDeuxUn (v(r \CompDeuxZero \sigma_{a}) \CompDeuxZero q_{a}) 
\\
&\phantom{=1}\text{(loi d'échange)}
\\
&= \alpha \CompDeuxUn (v(r) \CompDeuxZero \tau_{u(a)} \CompDeuxZero q_{a}) \CompDeuxUn (v_{r, p_{u(a)}} \CompDeuxZero q_{v(u(a))} q_{a}) \CompDeuxUn (v(rp_{u(a)}) \CompDeuxZero q_{q_{a}}) \CompDeuxUn (v_{rp_{u(a)}, u(q_{a})} \CompDeuxZero q_{a}) \CompDeuxUn (v(r \CompDeuxZero \sigma_{a}) \CompDeuxZero q_{a})
\\
&\phantom{=1}\text{(ALC 8)}
\\
&= \alpha \CompDeuxUn (v(r) \CompDeuxZero \tau_{u(a)} \CompDeuxZero q_{a}) \CompDeuxUn (v(r) v(p_{u(a)}) \CompDeuxZero q_{q_{a}}) \CompDeuxUn (v_{r, p_{u(a)}} \CompDeuxZero v(u(q_{a})) q_{a}) \CompDeuxUn (v_{rp_{u(a)}, u(q_{a})} \CompDeuxZero q_{a}) \CompDeuxUn (v(r \CompDeuxZero \sigma_{a}) \CompDeuxZero q_{a})
\\
&\phantom{=1}\text{(loi d'échange)}
\\
&= \alpha \CompDeuxUn (v(r) \CompDeuxZero \tau_{u(a)} \CompDeuxZero q_{a}) \CompDeuxUn (v(r) v(p_{u(a)}) \CompDeuxZero q_{q_{a}}) \CompDeuxUn (v(r) \CompDeuxZero v_{p_{u(a)}, u(q_{a})} \CompDeuxZero q_{a}) \CompDeuxUn (v_{r, p_{u(a)}u(q_{a})} \CompDeuxZero q_{a}) \CompDeuxUn (v(r \CompDeuxZero \sigma_{a}) \CompDeuxZero q_{a})
\\
&\phantom{=1}\text{(condition de cocycle pour $v$)}
\\
&= \alpha \CompDeuxUn (v(r) \CompDeuxZero \tau_{u(a)} \CompDeuxZero q_{a}) \CompDeuxUn (v(r) v(p_{u(a)}) \CompDeuxZero q_{q_{a}}) \CompDeuxUn (v(r) \CompDeuxZero v_{p_{u(a)}, u(q_{a})} \CompDeuxZero q_{a}) \CompDeuxUn (v(r) \CompDeuxZero v(\sigma_{a}) \CompDeuxZero q_{a}) \CompDeuxUn (v_{r, 1_{u(a)}} \CompDeuxZero q_{a})
\\
&\phantom{=1}\text{(naturalité des \deux{}cellules structurales de composition de $v$)}
\\
&= \alpha \CompDeuxUn (v(r) \CompDeuxZero v_{u(a)} \CompDeuxZero q_{a}) \CompDeuxUn (v_{r, 1_{u(a)}} \CompDeuxZero q_{a}) 
\\
&\phantom{=1}\text{(ALC 10)}
\\
&= \alpha \CompDeuxUn (((v(r) \CompDeuxZero v_{u(a)}) \CompDeuxUn v_{r, 1_{u(a)}}) \CompDeuxZero q_{a})
\\
&= \alpha \CompDeuxUn (1_{v(r)} \CompDeuxZero q_{a}) 
\\
&\phantom{=1}\text{(condition d'unité pour $v$)}
\\
&= \alpha \CompDeuxUn (1_{v(r)q_{a}}) 
\\
&= \alpha.
\end{align*}
Ainsi, 
$$
\alpha = (\tau_{b} \CompDeuxZero f) \CompDeuxUn (v(p_{b}) \CompDeuxZero q_{f}) \CompDeuxUn (v_{p_{b}, u(f)} \CompDeuxZero q_{a}) \CompDeuxUn (v(\gamma) \CompDeuxZero q_{a}).
$$
Cela termine la démonstration. 
\end{proof}

\begin{df}\label{DefAdjonctionCoLaxLax}
Une \emph{adjonction colax-lax} correspond aux données et conditions suivantes. \\

\begin{itemize}
\item[(Données 1)]
Des \deux{}catégories $\mathdeuxcat{A}$ et $\mathdeuxcat{B}$, un \DeuxFoncteurCoLax{} $u : \mathdeuxcat{A} \to \mathdeuxcat{B}$ et un \DeuxFoncteurLax{} $v : \mathdeuxcat{B} \to \mathdeuxcat{A}$. \\
\end{itemize}

\begin{itemize}
\item[(Données 2)]
Pour tout objet $b$ de $\mathdeuxcat{B}$, une \un{}cellule $p_{b} : u(v(b)) \to b$ dans $\mathdeuxcat{B}$. \\

\item[(Données 3)]
Pour toute \un{}cellule $g : b \to b'$ de $\mathdeuxcat{B}$, une \deux{}cellule
$$
p_{g} : p_{b'} u(v(g)) \Rightarrow g p_{b}
$$
dans $\mathdeuxcat{B}$.
\end{itemize}

\begin{itemize}
\item[(ACL 1)]
Pour tout couple d'objets $b$ et $b'$ de $\mathdeuxcat{B}$, pour tout couple de \un{}cellules $g$ et $g'$ de $b$ vers $b'$ dans $\mathdeuxcat{B}$, pour toute \deux{}cellule $\beta : g \Rightarrow g'$ dans $\mathdeuxcat{B}$, le diagramme
$$
\xymatrix{
g' p_{b}
&&
p_{b'} u(v(g'))
\ar@{=>}[ll]_{p_{g'}}
\\
g p_{b}
\ar@{=>}[u]^{\beta \CompDeuxZero p_{b}}
&&
p_{b'} u(v(g))
\ar@{=>}[ll]^{p_{g}}
\ar@{=>}[u]_{p_{b'} \CompDeuxZero u(v(\beta))}
}
$$
est commutatif. \\

\item[(ACL 2)]
Pour tout couple de \un{}cellules composables $g : b \to b'$ et $g' : b' \to b''$ dans $\mathdeuxcat{B}$, le diagramme 
$$
\xymatrix{
g' g p_{b}
&&
g' p_{b'} u(v(g))
\ar@{=>}[ll]_{g' \CompDeuxZero p_{g}}
&&
p_{b''} u(v(g')) u(v(g))
\ar@{=>}[ll]_{p_{g'} \CompDeuxZero u(v(g))}
\\
p_{b''} u(v(g'g))
\ar@{=>}[u]^{p_{g'g}}
&&&&
p_{b''} u(v(g') v(g))
\ar@{=>}[u]_{p_{b''} \CompDeuxZero u_{v(g'), v(g)}}
\ar@{=>}[llll]^{p_{b''} \CompDeuxZero u(v_{g', g})}
}
$$
est commutatif. \\

\item[(ACL 3)]
Pour tout objet $b$ de $\mathdeuxcat{B}$, le diagramme
$$
\xymatrix{
p_{b}
&&
p_{b} u(v(1_{b}))
\ar@{=>}[ll]_{p_{1_{b}}}
\\
&&
p_{b} u(1_{v(b)})
\ar@{=>}[llu]^{p_{b} \CompDeuxZero u_{v(b)}}
\ar@{=>}[u]_{p_{b} \CompDeuxZero u(v_{b})}
}
$$
est commutatif. \\
\end{itemize}

\begin{itemize}
\item[(Données 4)]
Pour tout objet $a$ de $\mathdeuxcat{A}$, une \un{}cellule $q_{a} : a \to v(u(a))$ dans $\mathdeuxcat{A}$. \\

\item[(Données 5)]
Pour toute \un{}cellule $f : a \to a'$ de $\mathdeuxcat{A}$, une \deux{}cellule
$$
q_{f} : q_{a'} f \Rightarrow v(u(f)) q_{a}
$$
dans $\mathdeuxcat{A}$.
\\
\end{itemize}

\begin{itemize}
\item[(ACL 4)]
Pour tout couple d'objets $a$ et $a'$ de $\mathdeuxcat{A}$, pour tout couple de \un{}cellules $f$ et $f'$ de $a$ vers $a'$ dans $\mathdeuxcat{A}$, pour toute \deux{}cellule $\alpha : f \Rightarrow f'$ dans $\mathdeuxcat{A}$, le diagramme
$$
\xymatrix{
v(u(f')) q_{a}
&&
q_{a'} f'
\ar@{=>}[ll]_{q_{f'}}
\\
v(u(f)) q_{a}
\ar@{=>}[u]^{v(u(\alpha)) \CompDeuxZero q_{a}}
&&
q_{a'} f
\ar@{=>}[u]_{q_{a'} \CompDeuxZero \alpha}
\ar@{=>}[ll]^{q_{f}}
}
$$
est commutatif. \\

\item[(ACL 5)]
Pour tout couple de \un{}cellules composables $f : a \to a'$ et $f' : a' \to a''$ de $\mathdeuxcat{A}$, le diagramme 
$$
\xymatrix{
v (u(f') u(f)) q_{a}
&&
v(u(f')) v(u(f)) q_{a}
\ar@{=>}[ll]_{v_{u(f'), u(f)} \CompDeuxZero q_{a}}
\\
&&
v(u(f')) q_{a'} f
\ar@{=>}[u]_{v(u(f')) \CompDeuxZero q_{f}}
\\
v(u(f'f)) q_{a}
\ar@{=>}[uu]^{v(u_{f',f}) \CompDeuxZero q_{a}}
&&
q_{a''} f' f
\ar@{=>}[u]_{q_{f'} \CompDeuxZero f}
\ar@{=>}[ll]^{q_{f'f}}
}
$$
est commutatif. \\

\item[(ACL 6)]
Pour tout objet $a$ de $\mathdeuxcat{A}$, le diagramme
$$
\xymatrix{
v(1_{u(a)}) q_{a}
&&
v(u(1_{a})) q_{a}
\ar@{=>}[ll]_{v(u_{a}) \CompDeuxZero q_{a}}
\\
&&
q_{a}
\ar@{=>}[u]_{q_{1_{a}}}
\ar@{=>}[llu]^{v_{u(a)} \CompDeuxZero q_{a}}
}
$$
est commutatif. \\
\end{itemize}

\begin{itemize}
\item[(Données 6)]
Pour tout objet $a$ de $\mathdeuxcat{A}$, une \deux{}cellule
$$
\sigma_{a} : p_{u(a)} u(q_{a}) \Rightarrow 1_{u(a)}
$$
dans $\mathdeuxcat{B}$. 
\end{itemize}

\begin{itemize}
\item[(ACL 7)]
Pour toute \un{}cellule $f : a \to a'$ de $\mathdeuxcat{A}$, le diagramme
$$
\xymatrix{
u(f)
&&&
p_{u(a')} u(q_{a'}) u(f)
\ar@{=>}[lll]_{\sigma_{a'} \CompDeuxZero u(f)}
&&&
p_{u(a')} u(q_{a'}f)
\ar@{=>}[lll]_{p_{u(a')} \CompDeuxZero u_{q_{a'}, f}}
\ar@{=>}[dd]^{p_{u(a')} \CompDeuxZero u(q_{f})}
\\
\\
u(f) p_{u(a)} u(q_{a})
\ar@{=>}[uu]^{u(f) \CompDeuxZero \sigma_{a}}
&&&
p_{u(a')} u(v(u(f))) u(q_{a})
\ar@{=>}[lll]^{p_{u(f)} \CompDeuxZero u(q_{a})}
&&&
p_{u(a')} u(v(u(f)) q_{a})
\ar@{=>}[lll]^{p_{u(a')} \CompDeuxZero u_{v(u(f)), q_{a}}}
}
$$
est commutatif. \\
\end{itemize}

\begin{itemize}
\item[(Données 7)]
Pour tout objet $b$ de $\mathdeuxcat{B}$, une \deux{}cellule
$$
\tau_{b} : 1_{v(b)} \Rightarrow  v(p_{b}) q_{v(b)}
$$
dans $\mathdeuxcat{A}$.
\end{itemize}

\begin{itemize}
\item[(ACL 8)]
Pour toute \un{}cellule $g : b \to b'$ de $\mathdeuxcat{B}$, le diagramme
$$
\xymatrix{
v(g p_{b}) q_{v(b)}
&&&
v(g) v(p_{b}) q_{v(b)}
\ar@{=>}[lll]_{v_{g,p_{b}} \CompDeuxZero q_{v(b)}}
&&&
v(g)
\ar@{=>}[lll]_{v(g) \CompDeuxZero \tau_{b}}
\ar@{=>}[dd]^{\tau_{b'} \CompDeuxZero v(g)}
\\
\\
v(p_{b'} u(v(g))) q_{v(b)}
\ar@{=>}[uu]^{v(p_{g}) \CompDeuxZero q_{v(b)}}
&&&
v(p_{b'}) v(u(v(g))) q_{v(b)}
\ar@{=>}[lll]^{v_{p_{b'}, u(v(g))} \CompDeuxZero q_{v(b)}}
&&&
v(p_{b'}) q_{v(b')} v(g)
\ar@{=>}[lll]^{v(p_{b'}) \CompDeuxZero q_{v(g)}}
}
$$
est commutatif. \\
\end{itemize}

\begin{itemize}
\item[(ACL 9)]
Pour tout objet $b$ de $\mathdeuxcat{B}$, le diagramme
$$
\xymatrix{
p_{b}
&&&
p_{b} p_{u(v(b))} u(q_{v(b)})
\ar@{=>}[lll]_{p_{b} \CompDeuxZero \sigma_{v(b)}}
&&&
p_{b} u(v(p_{b})) u(q_{v(b)})
\ar@{=>}[lll]_{p_{p_{b}} \CompDeuxZero u(q_{v(b)})}
\\
&&&&&&
p_{b} u(v(p_{b}) q_{v(b)})
\ar@{=>}[u]_{p_{b} \CompDeuxZero u_{v(p_{b}), q_{v(b)}}}
\\
&&&&&&
p_{b} u(1_{v(b)})
\ar@{=>}[u]_{p_{b} \CompDeuxZero u(\tau_{b})}
\ar@{=>}[lllllluu]^{p_{b} \CompDeuxZero u_{v(b)}}
}
$$
est commutatif. \\
\end{itemize}

\begin{itemize}
\item[(ACL 10)]
Pour tout objet $a$ de $\mathdeuxcat{A}$, le diagramme
$$
\xymatrix{
v(1_{u(a)}) q_{a}
&&&
v(p_{u(a)} u(q_{a})) q_{a}
\ar@{=>}[lll]_{v(\sigma_{a}) \CompDeuxZero q_{a}}
&&&
v(p_{u(a)}) v(u(q_{a})) q_{a}
\ar@{=>}[lll]_{v_{p_{u(a)}, u(q_{a})} \CompDeuxZero q_{a}}
\\
&&&&&&
v(p_{u(a)}) q_{v(u(a))} q_{a}
\ar@{=>}[u]_{v(p_{u(a)}) \CompDeuxZero q_{q_{a}}}
\\
&&&&&&
q_{a}
\ar@{=>}[u]_{\tau_{u(a)} \CompDeuxZero q_{a}}
\ar@{=>}[lllllluu]^{v_{u(a)} \CompDeuxZero q_{a}}
}
$$
est commutatif. 
\end{itemize}

On dira que le couple $(u,v)$ constitue une adjonction colax-lax, $u$ étant l'\emph{adjoint à gauche colax de $v$} et $v$ l'\emph{adjoint à droite lax de $u$}. 
\end{df}

\begin{rem}
Les propositions suivantes sont équivalentes.
\begin{itemize}
\item[$(i)$]
Le couple $(u,v)$ forme une adjonction lax-colax.
\item[$(ii)$]
Le couple $(\DeuxFoncUnOp{v}, \DeuxFoncUnOp{u})$ forme une adjonction colax-lax. 
\item[$(iii)$]
Le couple $(\DeuxFoncDeuxOp{u}, \DeuxFoncDeuxOp{v})$ forme une adjonction colax-lax. 
\item[$(iv)$]
Le couple $(\DeuxFoncToutOp{v}, \DeuxFoncToutOp{u})$ forme une adjonction lax-colax.
\end{itemize}

En particulier, un adjoint à gauche colax est un préadjoint à gauche colax, un adjoint à droite lax est un préadjoint à droite lax et un adjoint à droite colax est un préadjoint à droite colax, ces affirmations résultant immédiatement, par dualité, de la proposition \ref{AdjointPreadjoint}. 
\end{rem}

\section{Préfibrations}\label{SectionPrefibrations}

\begin{paragr}\label{RappelsPrefibrations}
On rappelle que la \emph{fibre} d'un foncteur $u : A \to B$ au-dessus d'un objet $b$ de $B$ est la catégorie $A_{b}$ dont les objets sont les objets $a$ de $A$ vérifiant $u(a) = b$ et dont les morphismes de $a$ vers $a'$ sont les morphismes $f : a \to a'$ de $A$ vérifiant $u(f) = 1_{b}$. 

Il est classique — mais ce n'est pas la définition généralement choisie — qu'un foncteur $u : A \to B$ est une \emph{préfibration} si et seulement si, pour tout objet $b$ de $B$, le foncteur canonique
$$
\begin{aligned}
j_{b}\index[not]{jb@$j_{b}$} : A_{b} &\to b \backslash A
\\
a &\mapsto (a, 1_{b})
\\
f &\mapsto f
\end{aligned}
$$
est un adjoint à gauche. Les \emph{préopfibrations} (parfois appelées « précofibrations ») se caractérisent de façon duale. On gardera ces rappels à l'esprit lors de la lecture des généralisations \deux{}catégoriques de ces notions que nous proposons dans la présente section.
\end{paragr}

\begin{df}\label{DefFibre}
Soient $u : \mathdeuxcat{A} \to \mathdeuxcat{B}$ un \DeuxFoncteurStrict{} et $b$ un objet de $\mathdeuxcat{B}$. On appelle \emph{fibre\index{fibre (au-dessus d'un objet) d'un morphisme de $\DeuxCat${}} de} $u$ \emph{au-dessus de} $b$ la \deux{}catégorie, que l'on notera $\Fibre{\mathdeuxcat{A}}{u}{b}$, dont les objets sont les objets $a$ de $\mathdeuxcat{A}$ tels que $u(a) = b$, dont les \un{}cellules de $a$ vers $a'$ sont les \un{}cellules $f $ de $a$ vers $a'$ dans $\mathdeuxcat{A}$ telles que $u(f) = 1_{b}$, et dont les \deux{}cellules de $f$ vers $f'$ sont les \deux{}cellules $\alpha$ de $f$ vers $f'$ dans $\mathdeuxcat{A}$ telles que $u(\alpha) = 1_{1_{b}}$, les diverses compositions et unités s'obtenant à partir de celles de $\mathdeuxcat{A}$ de façon « évidente ».
\end{df} 

\begin{rem}
Cette définition ne ferait pas sens telle quelle dans le cas d'un \DeuxFoncteurLax{}, qui n'envoie pas nécessairement les identités des objets sur l'identité de leur image. 
\end{rem}

\begin{lemme}\label{FibresCoOp}
Pour tout \DeuxFoncteurStrict{} $u : \mathdeuxcat{A} \to \mathdeuxcat{B}$ et tout objet $b$ de $\mathdeuxcat{B}$, la \deux{}catégorie $\Fibre{(\DeuxCatUnOp{\mathdeuxcat{A}})}{\DeuxFoncUnOp{u}}{b}$ (\emph{resp.} $\Fibre{(\DeuxCatDeuxOp{\mathdeuxcat{A}})}{\DeuxFoncDeuxOp{u}}{b}$, \emph{resp.} $\Fibre{(\DeuxCatToutOp{\mathdeuxcat{A}})}{\DeuxFoncToutOp{u}}{b}$) s'identifie canoniquement à $\DeuxCatUnOp{(\Fibre{\mathdeuxcat{A}}{u}{b})}$ (\emph{resp.} $\DeuxCatDeuxOp{(\Fibre{\mathdeuxcat{A}}{u}{b})}$, \emph{resp.} $\DeuxCatToutOp{(\Fibre{\mathdeuxcat{A}}{u}{b})}$).
\end{lemme}

\begin{proof}
La vérification est immédiate. 
\end{proof}

\begin{df}\label{DefPrefibration}
On dira qu'un \DeuxFoncteurStrict{} $u : \mathdeuxcat{A} \to \mathdeuxcat{B}$ est une \emph{préfibration}\index{préfibration (dans $\DeuxCat$)} si, pour tout objet $b$ de $\mathdeuxcat{B}$, le \DeuxFoncteurStrict{} canonique 
$$
\begin{aligned}
J_{b}\index[not]{Jb@$J_{b}$} : \Fibre{\mathdeuxcat{A}}{u}{b} &\to \OpTrancheCoLax{\mathdeuxcat{A}}{u}{b}
\\
a &\mapsto (a, 1_{b})
\\
f &\mapsto (f, 1_{1_{b}})
\\
\alpha &\mapsto \alpha
\end{aligned}
$$
est un préadjoint à gauche lax.

On dira qu'un  \DeuxFoncteurStrict{} $u : \mathdeuxcat{A} \to \mathdeuxcat{B}$ est une \emph{préopfibration}\index{préopfibration (dans $\DeuxCat$)} si $\DeuxFoncUnOp{u}$ est une préfibration, autrement dit si, pour tout objet $b$ de $\mathdeuxcat{B}$, le morphisme canonique $\Fibre{\mathdeuxcat{A}}{u}{b} \to \TrancheCoLax{\mathdeuxcat{A}}{u}{b}$ est un préadjoint à droite lax. 

On dira qu'un  \DeuxFoncteurStrict{} $u : \mathdeuxcat{A} \to \mathdeuxcat{B}$ est une \emph{précofibration}\index{précofibration (dans $\DeuxCat$)} si $\DeuxFoncDeuxOp{u}$ est une préfibration, autrement dit si, pour tout objet $b$ de $\mathdeuxcat{B}$, le morphisme canonique $\Fibre{\mathdeuxcat{A}}{u}{b} \to \OpTrancheLax{\mathdeuxcat{A}}{u}{b}$ est un préadjoint à gauche colax. 

On dira qu'un  \DeuxFoncteurStrict{} $u : \mathdeuxcat{A} \to \mathdeuxcat{B}$ est une \emph{précoopfibration}\index{précoopfibration (dans $\DeuxCat$)} si $\DeuxFoncToutOp{u}$ est une préfibration, autrement dit si, pour tout objet $b$ de $\mathdeuxcat{B}$, le morphisme canonique $\Fibre{\mathdeuxcat{A}}{u}{b} \to \TrancheLax{\mathdeuxcat{A}}{u}{b}$ est un préadjoint à droite colax.   
\end{df}

\begin{paragr}
On dira qu'une \deux{}catégorie $\mathdeuxcat{A}$ est \emph{préfibrée}\index{préfibrée (\deux{}catégorie)} (\emph{resp.} \emph{préopfibrée}\index{préopfibrée (\deux{}catégorie)}, \emph{resp.} \emph{précofibrée}\index{précofibrée (\deux{}catégorie)}, \emph{resp.} \emph{précoopfibrée}\index{précoopfibrée (\deux{}catégorie)}) sur $\mathdeuxcat{B}$ s'il existe une préfibration (\emph{resp.} une préopfibration, \emph{resp.} une précofibration, \emph{resp.} une précoopfibration) de $\mathdeuxcat{A}$ vers $\mathdeuxcat{B}$. 
\end{paragr}

\begin{exemple}\label{ProjectionPrefibration}
Toute projection est une préfibration, une préopfibration, une précofibration et une précoopfibration, comme un calcul sans aucune difficulté permet de le vérifier. C'est du reste un cas particulier de résultats que nous établirons dans la section \ref{SectionIntegration} (voir la remarque \ref{ProjectionPrefibrationPreuve}).
\end{exemple}

\begin{exemple}\label{ProjectionCommaPrefibration}
Soient $u : \mathdeuxcat{A} \to \mathdeuxcat{C}$ un \DeuxFoncteurLax{} normalisé et $v : \mathdeuxcat{B} \to \mathdeuxcat{C}$ un \DeuxFoncteurCoLax{}. Alors, la projection canonique $[u,v] \to \mathdeuxcat{A}$ est une préfibration ; nous n'utiliserons pas ce fait, que le lecteur pourra vérifier par lui-même. Pour un résultat du même genre, voir la proposition \ref{ProjectionIntegralePrefibration}. 
\end{exemple}

\begin{rem}\label{PrefibrationRelevement}
Soient $u : \mathdeuxcat{A} \to \mathdeuxcat{B}$ un \DeuxFoncteurStrict{}, $b$ un objet de $\mathdeuxcat{B}$, $J_{b} : \Fibre{\mathdeuxcat{A}}{u}{b} \to \OpTrancheCoLax{\mathdeuxcat{A}}{u}{b}$ le \DeuxFoncteurStrict{} canonique défini dans l'énoncé de la définition \ref{DefPrefibration} et $(a, p : b \to u(a))$ un objet de la \deux{}catégorie $\OpTrancheCoLax{\mathdeuxcat{A}}{u}{b}$. Décrivons partiellement la structure de la \deux{}catégorie $\TrancheLax{\Fibre{\mathdeuxcat{A}}{u}{b}}{J_{b}}{(a,p)}$. Ses objets sont les triplets 
$$
(a', (f' : a' \to a, \alpha' : u(f') \Rightarrow p))
$$
où $a'$ est un objet de $\mathdeuxcat{A}$ tel que $u(a') = b$, $f'$ une \un{}cellule de $\mathdeuxcat{A}$ et $\alpha'$ une \deux{}cellule de $\mathdeuxcat{B}$. Étant donné deux objets $(a', (f', \alpha'))$ et $(a'', (f'', \alpha''))$ de $\TrancheLax{\Fibre{\mathdeuxcat{A}}{u}{b}}{J_{b}}{(a,p)}$, les \un{}cellules du premier vers le second sont données par les couples
$$
(g : a' \to a'', \beta : f' \Rightarrow f''g)
$$
où $g$ est une \un{}cellule de $\mathdeuxcat{A}$ telle que $u(g) = 1_{b}$ et $\beta$ une \deux{}cellule de $\mathdeuxcat{A}$ telle que $\alpha'' \CompDeuxUn u(\beta) = \alpha'$. En gardant ces notations, si $(g' : a' \to a'', \beta' : f' \Rightarrow f''g')$ est une \un{}cellule de $(a', (f', \alpha'))$ vers $(a'', (f'', \alpha''))$ (on a donc $u(g') = 1_{b}$ et $\alpha'' \CompDeuxUn u(\beta') = \alpha'$), les \deux{}cellules de $(g, \beta)$ vers $(g', \beta')$ sont données par les \deux{}cellules
$$
\gamma : g \Rightarrow g'
$$
de $\mathdeuxcat{A}$ telles que $u(\gamma) = 1_{1_{b}}$ et $(f'' \CompDeuxZero \gamma) \CompDeuxUn \beta = \beta'$. 

Ainsi, $u : \mathdeuxcat{A} \to \mathdeuxcat{B}$ est une préfibration si et seulement si, pour tout objet $b$ de $\mathdeuxcat{B}$, pour tout objet $a$ de $\mathdeuxcat{A}$, pour toute \un{}cellule $p : b \to u(a)$ de $\mathdeuxcat{B}$, il existe un objet $a''$ de $\mathdeuxcat{A}$ au-dessus de $b$ (c'est-à-dire vérifiant $u(a'') = b$), une \un{}cellule $f'' : a'' \to a$ de $\mathdeuxcat{A}$ et une \deux{}cellule $\alpha'' : u(f'') \Rightarrow p$ de $\mathdeuxcat{B}$ tels que, pour tout tel triplet $(a', f' : a' \to a, \alpha' : u(f') \Rightarrow p)$ tel que $u(a') = b$, il existe un couple $(g : a' \to a'', \beta : f' \Rightarrow f''g)$ tel que $u(g) = 1_{b}$ et $\alpha'' \CompDeuxUn u(\beta) = \alpha'$ tel que, pour tout tel couple $(g' : a' \to a'', \beta' : f' \Rightarrow f''g')$ tel que $u(g') = 1_{b}$ et $\alpha'' \CompDeuxUn u(\beta') = \alpha'$, il existe une unique \deux{}cellule $\gamma : g \Rightarrow g'$ de $\mathdeuxcat{A}$ telle que $u(\gamma) = 1_{1_{b}}$ et $(f'' \CompDeuxZero \gamma) \beta = \beta'$. (En particulier, on a $u(\beta) = u(\beta')$.)
\end{rem}

\begin{rem}\label{RemarquesPrefibration}
Cette notion de préfibration, suggérée par Maltsiniotis, est évidemment compatible avec la notion classique de préfibration dans $\Cat$. La classe des préfibrations de $\Cat$ n'étant pas stable par composition, il n'y a pas lieu d'espérer que la classe des préfibrations de $\DeuxCat$ le soit. Il est probable que l'on puisse\footnote{Peut-être en utilisant la notion d'adjonction lax-colax.} renforcer cette notion pour définir une classe, stable par composition, de \emph{fibrations} dans $\DeuxCat$. Cependant, du point de vue homotopique, la notion de préfibration l'emporte en importance sur celle de fibration. Il est également possible que la théorie des foncteurs préfeuilletants de Bénabou, inédite mais remarquable, se généralise dans le cadre de $\DeuxCat$ pour fournir un cadre conceptuel satisfaisant dans lequel exprimer diverses notions apparentées à celles de préfibration et de fibration. En fait, une notion de fibration dans $\DeuxCat$ se trouve déjà définie depuis longtemps dans la littérature ; voir par exemple \cite[définition 2.3]{Hermida}. Récemment, Buckley a proposé dans \cite{Buckley} de renforcer cette définition. 
\end{rem}

\section{2-foncteurs en tranches}\label{SectionMorphismesTranches}

\begin{paragr}\label{CommaFonctorielle3}
Soient $u : \mathdeuxcat{A} \to \mathdeuxcat{C}$ un \DeuxFoncteurLax{} et $\mathdeuxcat{B}$ une \deux{}catégorie. On se propose de vérifier que ces données induisent\footnote{On néglige les questions de taille, ici comme à d'autres occasions, sans forcément le mentionner, puisque les calculs font sens indépendamment d'hypothèses de petitesse, que nécessitent évidemment certaines affirmations, comme le lecteur le remarquera de lui-même.} un \DeuxFoncteurStrict{}
$$
[u,-] : Colax({\mathdeuxcat{B},\mathdeuxcat{C}}) \to \DeuxCatDeuxCat.
$$

$\bullet$ Pour tout \DeuxFoncteurCoLax{} $v : \mathdeuxcat{B} \to \mathdeuxcat{C}$, on pose 
$$
[u,-] (v) = [u,v].
$$
$\bullet$ Pour tout couple de \DeuxFoncteursCoLax{} $v$ et $w$ de $\mathdeuxcat{B}$ vers $\mathdeuxcat{C}$ et toute \DeuxTransformationCoLax{} $\sigma : v \Rightarrow w$, le \DeuxFoncteurStrict{} $[u,-](\sigma)$, que nous noterons $[u,\sigma]$, de $[u,v]$ vers $[u,w]$, est défini, de façon duale à ce que nous avons déjà fait dans le paragraphe \ref{CommaFonctorielle1}, comme suit.
\begin{itemize}
\item Pour tout objet $(a,b,r)$ de $[u,v]$, 
$$
[u,\sigma] (a,b,r) = (a,b,\sigma_{b}r). 
$$
\item Pour toute \un{}cellule $(f,g,\alpha)$ de $(a,b,r)$ vers $(a',b',r')$ dans $[u,v]$,
$$
[u,\sigma] (f,g,\alpha) = (f,g,(\sigma_{b'} \CompDeuxZero \alpha) \CompDeuxUn (\sigma_{g} \CompDeuxZero r)).
$$
\item Pour toute \deux{}cellule $(\varphi, \psi)$ de $[u,v]$, 
$$
[u,\sigma] (\varphi,\psi) = (\varphi, \psi).
$$
\end{itemize}
$\bullet$ Soient de plus une \DeuxTransformationCoLax{} $\tau : v \Rightarrow w$ et une modification $\Gamma : \sigma \Rrightarrow \tau$. La \DeuxTransformationStricte{} $[u, \Gamma] : [u,\sigma] \Rightarrow [u,\tau]$ est définie par
$$
[u,\Gamma]_{(a,b,r)} = (1_{a}, 1_{b}, (\tau_{b} r \CompDeuxZero u_{a}) \CompDeuxUn (\Gamma_{b} \CompDeuxZero r) \CompDeuxUn (w_{b} \CompDeuxZero \sigma_{b} r))
$$
pour tout objet $(a,b,r)$ de $[u,v]$. Vérifions que cela définit bien une \DeuxTransformationStricte{}. 
\begin{itemize}
\item Soit $(f,g,\alpha) : (a,b,r) \to (a',b',r')$ une \un{}cellule de $[u,v]$. Les égalités
$$
[u,\tau](f,g,\alpha) [u,\Gamma]_{(a,b,r)} = [u,\Gamma]_{(a',b',r')} [u,\sigma](f,g,\alpha) = (f,g, (\tau_{b'} \CompDeuxZero \alpha) \CompDeuxUn (\Gamma_{b'} \CompDeuxZero v(g)r) \CompDeuxUn (\sigma_{g} \CompDeuxZero r))
$$
résultent de la naturalité des \deux{}cellules structurales d'unité de $u$ et $w$, de l'égalité
$$
\tau_{g} \CompDeuxUn (w(g) \CompDeuxZero \Gamma_{b}) = (\Gamma_{b'} \CompDeuxZero v(g)) \CompDeuxUn \sigma_{g}
$$
et de la loi d'échange. 
\item Soient de plus $(h,k,\beta)$ une \un{}cellule de $(a,b,r)$ vers $(a',b',r')$ et $(\varphi, \psi)$ une \deux{}cellule de $(f,g,\alpha)$ vers $(h,k,\beta)$ dans $[u,v]$. Les égalités
$$
[u,\tau](\varphi,\psi) \CompDeuxZero [u,\Gamma]_{(a,b,r)} = (\varphi,\psi) = [u,\Gamma]_{(a',b',r')} \CompDeuxZero [u,\sigma](\varphi,\psi)
$$
sont alors immédiates. 
\end{itemize}
Ce qui précède permet bien d'affirmer que $[u,\Gamma]$ est une \DeuxTransformationStricte{} de $[u,\sigma]$ vers $[u,\tau]$. 

$\bullet$ Soient $v$ et $w$ des \DeuxFoncteursCoLax{} de $\mathdeuxcat{B}$ vers $\mathdeuxcat{C}$ et $\sigma : v \Rightarrow w$ une \DeuxTransformationCoLax{}. La vérification de l'égalité
$$
[u,1_{\sigma}] = 1_{[u,\sigma]}
$$
ne pose aucune difficulté. 

$\bullet$ Soient, sous les mêmes données, $\tau$ et $\mu$ des \DeuxTransformationsCoLax{} de $v$ vers $w$, $\Gamma$ une modification de $\sigma$ vers $\tau$ et $\Lambda$ une modification de $\tau$ vers $\mu$. Alors, pour tout objet $(a,b,r)$ de $[u,v]$, les égalités
$$
[u,\Lambda]_{(a,b,r)} [u,\Gamma]_{(a,b,r)} = (1_{a}, 1_{b}, w_{b} \CompDeuxZero (\Lambda_{b} \CompDeuxUn \Gamma_{b}) \CompDeuxZero r \CompDeuxZero u_{a}) = [u,\Lambda \CompDeuxUn \Gamma]_{(a,b,r)}
$$
résultent de la loi d'échange, de la naturalité des \deux{}cellules structurales d'unité de $u$ et $w$ et de l'égalité $(\Lambda \CompDeuxUn \Gamma)_{b} = \Lambda_{b} \CompDeuxUn \Gamma_{b}$ — cette dernière étant tautologique par définition de la composition verticale des modifications, définition que le lecteur avait sans doute déjà explicitée. Par définition toujours, cela permet d'affirmer
$$
[u,\Lambda] \CompDeuxUn [u,\Gamma] = [u, \Lambda \CompDeuxUn \Gamma].
$$  

$\bullet$ Toujours sous la donnée d'un \DeuxFoncteurCoLax{} $v : \mathdeuxcat{B} \to \mathdeuxcat{C}$, l'égalité
$$
[u,1_{v}] = 1_{[u,v]}
$$
se vérifie facilement. 

$\bullet$ Soient $v$, $w$ et $z$ des \DeuxFoncteursCoLax{} de $\mathdeuxcat{B}$ vers $\mathdeuxcat{C}$ et $\sigma : v \Rightarrow w$ et $\rho : w \Rightarrow z$ des \DeuxTransformationsCoLax{}. L'égalité $[u,\rho] [u,\sigma] = [u,\rho \CompDeuxUn \sigma]$ résulte alors des vérifications suivantes. 
\begin{itemize}
\item Pour tout objet $(a,b,r)$ de $[u,v]$,
$$
[u,\rho] [u,\sigma] (a,b,r) = (a,b,\rho_{b} \sigma_{b} r) = [u,\rho \CompDeuxUn \sigma] (a,b,r).
$$
\item Pour toute \un{}cellule $(f,g,\alpha) : (a,b,r) \to (a',b',r')$ de $[u,v]$,
$$
[u,\rho] [u,\sigma] (f,g,\alpha) = (f,g,(\rho_{b'} \sigma_{b'} \CompDeuxZero \alpha) \CompDeuxUn (\rho_{b'} \CompDeuxZero \sigma_{g} \CompDeuxZero r) \CompDeuxUn (\rho_{g} \CompDeuxZero \sigma_{b}r)) = [u,\rho \CompDeuxUn \sigma] (f,g,\alpha).
$$
\item Pour toute \deux{}cellule $(\varphi, \psi)$ de $[u,v]$, 
$$
[u,\rho] [u,\sigma] (\varphi, \psi) = (\varphi, \psi) = [u,\rho \CompDeuxUn \sigma] (\varphi, \psi).
$$
\end{itemize}
$\bullet$ Soient $v$, $w$ et $z$ des \DeuxFoncteursCoLax{} de $\mathdeuxcat{B}$ vers $\mathdeuxcat{C}$, $\sigma$ et $\mu$ des \DeuxTransformationsCoLax{} de $v$ vers $w$, $\rho$ et $\nu$ des \DeuxTransformationsCoLax{} de $w$ vers $z$, $\Gamma$ une modification de $\sigma$ vers $\mu$ et $\Lambda$ une modification de $\rho$ vers $\nu$. Alors, 
$$
[u, \Lambda \CompDeuxZero \Gamma]_{(a,b,r)} = (1_{a}, 1_{b}, (\nu_{b} \mu_{b} r \CompDeuxZero u_{a}) \CompDeuxUn (\Lambda_{b} \CompDeuxZero \Gamma_{b} \CompDeuxZero r) \CompDeuxUn (z_{b} \CompDeuxZero \rho_{b} \sigma_{b} r)) = ([u,\Lambda] \CompDeuxZero [u,\Gamma])_{(a,b,r)}.
$$
Cela résulte de la définition de la composition horizontale des \deux{}cellules dans $Colax(\mathdeuxcat{B},\mathdeuxcat{C})$ et de quelques instructives et courtes manipulations sans aucune difficulté ; nous laissons au lecteur le soin d'expliciter les détails s'il en ressent le besoin. 

La construction ci-dessus définit donc bien un \DeuxFoncteurStrict{} $[u,-] : Colax({\mathdeuxcat{B},\mathdeuxcat{C}}) \to \DeuxCatDeuxCat$ sous la donnée d'un morphisme $u : \mathdeuxcat{A} \to \mathdeuxcat{C}$ de $\DeuxCatLax$ et d'une petite \deux{}catégorie $\mathdeuxcat{B}$. 
\end{paragr}

\begin{paragr}\label{CommaFonctorielle4}
Soient maintenant $\mathdeuxcat{A}$, $\mathdeuxcat{C}$ et $\mathdeuxcat{D}$ des petites \deux{}catégories, $v$ et $w$ des \DeuxFoncteursLax{} de $\mathdeuxcat{A}$ vers $\mathdeuxcat{C}$ et $\sigma : v \Rightarrow w$ une \DeuxTransformationCoLax{}. En vertu de ce qui précède, $v$ et $w$ permettent de définir des \DeuxFoncteursStricts{} $[v,-]$ et $[w,-]$ de $Colax(\mathdeuxcat{D},\mathdeuxcat{C})$ vers $\DeuxCatDeuxCat$. On peut faire mieux : $\sigma$ induit une \DeuxTransformationStricte{} de $[w,-]$ vers $[v,-]$, dont la composante en un objet $z$ de $Colax(\mathdeuxcat{D},\mathdeuxcat{C})$, c'est-à-dire en un \DeuxFoncteurCoLax{} $z : \mathdeuxcat{D} \to \mathdeuxcat{C}$, n'est autre que le \DeuxFoncteurStrict{} $[\sigma,z] : [w,z] \to [v,z]$ dont on a déjà défini un analogue sous des données similaires dans le paragraphe \ref{CommaFonctorielle1}. Plus précisément, pour tout objet $(a,d,r)$ de $[w,z]$, $[\sigma,z](a,d,r) = (a,d,r\sigma_{a})$, pour toute \un{}cellule $(f,g,\alpha) : (a,d,r) \to (a',d',r')$ de $[w,z]$, $[\sigma,z](f,g,\alpha) = (f,g,(r' \CompDeuxZero \sigma_{f}) \CompDeuxUn (\alpha \CompDeuxZero \sigma_{a}))$ et, pour toute \deux{}cellule $(\varphi,\psi)$ de $[w,z]$, $[\sigma,z] (\varphi,\psi) = (\varphi,\psi)$. Le fait que cela définit bien une \DeuxTransformationStricte{} de $[w,-]$ vers $[v,-]$ résulte des vérifications suivantes. 

$\bullet$ Soient $y$ et $z$ des \DeuxFoncteursCoLax{} de $\mathdeuxcat{D}$ vers $\mathdeuxcat{C}$ et $\tau : y \Rightarrow z$ une \DeuxTransformationCoLax{}. On veut vérifier l'égalité $[\sigma,z] [w,\tau] = [v,\tau] [\sigma,y]$. Cela découle des calculs simples suivants. 
\begin{itemize}
\item Pour tout objet $(a,d,r)$ de $[w,y]$,
$$
[\sigma,z] [w,\tau] (a,d,r) = (a,d,\tau_{d}r\sigma_{a}) = [v,\tau] [\sigma,y] (a,d,r).
$$
\item Pour toute \un{}cellule $(f,g,\alpha) : (a,d,r) \to (a',d',r')$ dans $[w,y]$, 
$$
[\sigma,z] [w,\tau] (f,g,\alpha) = (f,g,(\tau_{d'} r' \CompDeuxZero \sigma_{f}) \CompDeuxUn (\tau_{d'} \CompDeuxZero \alpha \CompDeuxZero \sigma_{a}) \CompDeuxUn (\tau_{g} \CompDeuxZero r \sigma_{a})) = [v,\tau] [\sigma,y] (f,g,\alpha).
$$
\item Pour toute \deux{}cellule $(\varphi,\psi)$ de $[w,y]$, 
$$
[\sigma,z] [w,\tau] (\varphi,\psi) = (\varphi,\psi) = [v,\tau] [\sigma,y] (\varphi,\psi).
$$
\end{itemize} 

$\bullet$ Soient de plus $\mu : y \Rightarrow z$ une \DeuxTransformationCoLax{} et $\Gamma : \tau \Rrightarrow \mu$ une modification. L'égalité $[\sigma,z] \CompDeuxZero [w,\Gamma] = [v,\Gamma] \CompDeuxZero [\sigma,y]$ résulte des égalités, pour tout objet $(a,d,r)$ de $[w,y]$, 
\begin{align*}
([\sigma,z] \CompDeuxZero [w,\Gamma])_{(a,d,r)} &= [\sigma,z] ([w,\Gamma]_{(a,d,r)}) 
\\
&= [\sigma,z] (1_{a}, 1_{d}, (\mu_{d} r \CompDeuxZero w_{a}) \CompDeuxUn (\Gamma_{d} \CompDeuxZero r) \CompDeuxUn (z_{d} \CompDeuxZero \tau_{d} r))
\\
&= (1_{a}, 1_{d}, (\mu_{d} r \CompDeuxZero \sigma_{1_{a}}) \CompDeuxUn (\mu_{d} r \CompDeuxZero w_{a} \CompDeuxZero \sigma_{a}) \CompDeuxUn (\Gamma_{d} \CompDeuxZero r \sigma_{a}) \CompDeuxUn (z_{d} \CompDeuxZero \tau_{d} r \sigma_{a})),
\end{align*}
\begin{align*}
([v,\Gamma] \CompDeuxZero [\sigma,y])_{(a,d,r)} &= [v,\Gamma]_{[\sigma,y](a,d,r)}
\\
&= [v,\Gamma]_{(a,d,r \sigma_{a})}
\\
&= (1_{a}, 1_{d}, (\mu_{d} r \sigma_{a} \CompDeuxZero v_{a}) \CompDeuxUn (\Gamma_{d} \CompDeuxZero r \sigma_{a}) \CompDeuxUn (z_{d} \CompDeuxZero \tau_{d} r \sigma_{a}))
\end{align*}
et
$$
\sigma_{1_{a}} \CompDeuxUn (w_{a} \CompDeuxZero \sigma_{a}) = \sigma_{a} \CompDeuxZero v_{a}.
$$
\end{paragr}

\begin{paragr}\label{CommaFonctorielle5}
Soient maintenant $u : \mathdeuxcat{A} \to \mathdeuxcat{B}$ et $u' : \mathdeuxcat{B} \to \mathdeuxcat{C}$ deux \DeuxFoncteursStricts{} et $\mathdeuxcat{D}$ une \deux{}catégorie. On va voir que ces données induisent une \DeuxTransformationStricte{} du \DeuxFoncteurStrict{} $[u'u,-] : Colax(\mathdeuxcat{D},\mathdeuxcat{C}) \to \DeuxCatDeuxCat$ vers $[u',-] : Colax(\mathdeuxcat{D},\mathdeuxcat{C}) \to \DeuxCatDeuxCat$ dont la composante en le \DeuxFoncteurCoLax{} $v : \mathdeuxcat{D} \to \mathdeuxcat{C}$ est l'analogue, ici strict puisque $u$ est strict, du \DeuxFoncteurLax{} $F : [u'u,v] \to [u',v]$ que nous avons défini dans le paragraphe \ref{CommaFonctorielle2}. Par abus de notation, nous noterons toutes ces composantes $F$ dans les vérifications qui suivent. 

$\bullet$ Soient $v$ et $w$ des \DeuxFoncteursCoLax{} de $\mathdeuxcat{D}$ vers $\mathdeuxcat{C}$ et $\sigma : v \Rightarrow w$ une \DeuxTransformationCoLax{}. L'égalité $[u',\sigma] F = F [u'u,\sigma]$ résulte des vérifications directes suivantes.
\begin{itemize}
\item Pour tout objet $(a,d,r)$ de $[u'u,v]$,
$$
[u',\sigma] F (a,d,r) = (u(a),d,\sigma_{d}r) = F [u'u,\sigma] (a,d,r).
$$
\item Pour toute \un{}cellule $(f,g,\alpha) : (a,d,r) \to (a',d',r')$ de $[u'u,v]$,
$$
[u',\sigma] F (f,g,\alpha) = (u(f),g,(\sigma_{d'} \CompDeuxZero \alpha) \CompDeuxUn (\sigma_{g} \CompDeuxZero r)) = F [u'u,\sigma] (f,g,\alpha).
$$
\item Pour toute \deux{}cellule $(\varphi,\psi)$ de $[u'u,v]$,
$$
[u',\sigma] F (\varphi,\psi) = (u(\varphi),\psi) = F [u'u,\sigma] (\varphi,\psi). 
$$
\end{itemize} 

$\bullet$ Soient, sous les mêmes hypothèses, une \DeuxTransformationCoLax{} $\tau : v \Rightarrow w$ et une modification $\Gamma : \sigma \Rrightarrow \tau$. L'égalité
$$
[u',\Gamma] \CompDeuxZero F = F \CompDeuxZero [u'u,\Gamma]
$$
résulte de ce que, pour tout objet $(a,d,r)$ de $[u'u,v]$,
$$
([u',\Gamma] \CompDeuxZero F)_{(a,d,r)} = (1_{u(a)}, 1_{d}, \Gamma_{d} \CompDeuxZero r) = (F \CompDeuxZero [u'u,\Gamma])_{(a,d,r)}.
$$
\end{paragr} 

\begin{paragr}\label{DefFoncteursTranches}
Soit $u : \mathdeuxcat{A} \to \mathdeuxcat{B}$ un morphisme de $\DeuxCatLax$. En identifiant $\mathdeuxcat{B}$ à la \deux{}catégorie dont les objets sont les \DeuxFoncteursStricts{} de $\DeuxCatPonct$ vers $\mathdeuxcat{B}$, les \un{}cellules les \DeuxTransformationsStrictes{} entre iceux et les \deux{}cellules les modifications entre icelles, ce qui précède permet d'associer à $u$ un \DeuxFoncteurStrict{} $\DeuxFoncteurTranche{u}\index[not]{Tu@$\DeuxFoncteurTranche{u}$} : \mathdeuxcat{B} \to \DeuxCatDeuxCat$ comme suit. Pour des raisons de lisibilité, on pourra noter $\DeuxFoncteurTrancheArg{u}{b}$\index[not]{Tub@$\DeuxFoncteurTrancheArg{u}{b}$}, $\DeuxFoncteurTrancheArg{u}{f}$\index[not]{Tuf@$\DeuxFoncteurTrancheArg{u}{f}$} et $\DeuxFoncteurTrancheArg{u}{\gamma}$\index[not]{Tug0@$\DeuxFoncteurTrancheArg{u}{\gamma}$} plutôt que $\DeuxFoncteurTranche{u}(b)$, $\DeuxFoncteurTranche{u} (f)$ et $\DeuxFoncteurTranche{u} (\gamma)$ respectivement. 

Pour tout objet $b$ de $\mathdeuxcat{B}$, 
$$
\DeuxFoncteurTrancheArg{u}{b} = \TrancheLax{\mathdeuxcat{A}}{u}{b}. 
$$

Pour toute \un{}cellule $f : b \to b'$ de $\mathdeuxcat{B}$, le \DeuxFoncteurStrict{} $\DeuxFoncteurTrancheArg{u}{f} : \DeuxFoncteurTrancheArg{u}{b} \to \DeuxFoncteurTrancheArg{u}{b'}$ est donné par les formules
$$
\begin{aligned}
\TrancheLax{\mathdeuxcat{A}}{u}{b} &\to \TrancheLax{\mathdeuxcat{A}}{u}{b'}
\\
(a, p : u(a) \to b) &\mapsto (a, fp : u(a) \to b')
\\
(g : a \to a', \alpha : p \Rightarrow p' u(g)) &\mapsto (g : a \to a', f \CompDeuxZero \alpha : fp \Rightarrow f p' u(g))
\\
\beta &\mapsto \beta.
\end{aligned}
$$

Soient $f$ et $f'$ deux \un{}cellules de $b$ vers $b'$ et $\gamma : f \Rightarrow f'$ une \deux{}cellule dans $\mathdeuxcat{B}$. La \DeuxTransformationStricte{} $\DeuxFoncteurTrancheArg{u}{\gamma} : \DeuxFoncteurTrancheArg{u}{f} \Rightarrow \DeuxFoncteurTrancheArg{u}{f'}$ est définie par
$$
(\DeuxFoncteurTrancheArg{u}{\gamma})_{(a,p)} = (1_{a}, \gamma \CompDeuxZero p \CompDeuxZero u_{a})
$$
pour tout objet $(a, p)$ de $\TrancheLax{\mathdeuxcat{A}}{u}{b}$. 

\end{paragr}

\begin{paragr}\label{IntegrationTransformation}
Soit
$$
\xymatrix{
\mathdeuxcat{A}  
\ar[rr]^{u} 
\ar[dr] _{w}
&&\mathdeuxcat{B}
\ar[dl]^{v}
\dtwocell<\omit>{<7.5>\sigma}
\\ 
& 
\mathdeuxcat{C}
& 
}
$$
un diagramme de \DeuxFoncteursStricts{} commutatif à l'\DeuxTransformationCoLax{} $\sigma : vu \Rightarrow w$ près seulement. Les considérations qui précèdent nous fournissent des \DeuxFoncteursStricts{} $\DeuxFoncteurTranche{w}$ et $\DeuxFoncteurTranche{v}$ de $\mathdeuxcat{C}$ vers $\DeuxCatDeuxCat$ ainsi qu'une \DeuxTransformationStricte{}, que l'on notera $\DeuxFoncteurTranche{\sigma}$\index[not]{TS0@$\DeuxFoncteurTranche{\sigma}$}, de $\DeuxFoncteurTranche{w}$ vers $\DeuxFoncteurTranche{v}$, définie par 
$$
(\DeuxFoncteurTranche{\sigma})_{c} = \DeuxFoncTrancheLaxCoq{u}{\sigma}{c}
$$
pour tout objet $c$ de $\mathdeuxcat{C}$. On notera $\DeuxFoncteurTrancheArg{\sigma}{c} = (\DeuxFoncteurTranche{\sigma})_{c}$. 

\end{paragr}

\section{Intégration\index{intégration} des 2-foncteurs}\label{SectionIntegration}

\begin{paragr}\label{RappelsIntegration}
Soient $A$ une petite catégorie et $F : A \to \Cat$ un foncteur. On rappelle — voir par exemple \cite[paragraphe 2.2.1]{THG}, dont nous reprenons la présentation — une construction « de Grothendieck »\footnote{Cette construction se trouve souvent appelée « construction de Grothendieck », bien que Grothendieck soit à l'origine d'un nombre assez considérable de constructions. En fait, s'il a notamment montré que cette construction fournissait une catégorie fibrée, la construction proprement dite remonte au moins aux travaux catégoriques de Yoneda.} d'une catégorie $\DeuxInt{A} F$ opfibrée sur $A$. Les objets en sont les couples $(a, x)$ avec $a$ un objet de $A$ et $x$ un objet de $F(a)$. Les morphismes de $(a, x)$ vers $(a', x')$ sont les couples $(f, r)$ avec $f : a \to a'$ un morphisme de $A$ et $r : F(f)(x) \to x'$ un morphisme de $F(a')$. La composition des morphismes est définie par la formule
$$
(f', r') (f, r) = (f'f, r' F(f')(r)).
$$
La projection
$$
\begin{aligned}
\DeuxInt{A} F &\to A
\\
(a,x) &\mapsto a
\\
(f,r) &\mapsto f
\end{aligned}
$$
est une opfibration (donc en particulier une préopfibration) dont la fibre au-dessus de l'objet $a$ de $A$ s'identifie canoniquement à $F(a)$. 

Dans \cite[2.2.6]{THG} — par exemple — se trouve également présentée la construction d'une catégorie fibrée — qui se trouve être, historiquement, celle étudiée par Grothendieck — sur une petite catégorie $A$ pour tout foncteur $F : \DeuxCatUnOp{A} \to \Cat$. Nous n'entrons pas dans les détails puisque nous allons, dans ce qui suit, étudier de façon méthodique les diverses dualités des constructions analogues dans le cadre plus général des \deux{}catégories. 
\end{paragr}

\begin{paragr}\label{ParagrapheIntegration}
Soit $F : \mathdeuxcat{A} \to \DeuxCatDeuxCat$ un \DeuxFoncteurStrict{}. Ainsi, $F(a)$ est une \deux{}catégorie pour tout objet $a$ de $\mathdeuxcat{A}$, $F(f)$ est un \DeuxFoncteurStrict{} pour toute \un{}cellule $f$ de $\mathdeuxcat{A}$, et $F(\gamma)$ est une \DeuxTransformationStricte{} pour toute \deux{}cellule $\gamma$ de $\mathdeuxcat{A}$. Ces données nous permettent de définir une \deux{}catégorie $\DeuxInt{\mathdeuxcat{A}}F$\index[not]{0zAF@$\DeuxInt{\mathdeuxcat{A}}F$} comme suit. 

Les objets de $\DeuxInt{\mathdeuxcat{A}}F$ sont les couples $(a, x)$, avec $a$ un objet de $\mathdeuxcat{A}$ et $x$ un objet de la \deux{}catégorie $F(a)$. 

Les \un{}cellules de $(a, x)$ vers $(a', x')$ dans $\DeuxInt{\mathdeuxcat{A}}F$ sont les couples
$$
(f : a \to a', r : F(f)(x) \to x')
$$
dans lesquels $f$ est une \un{}cellule de $a$ vers $a'$ dans $\mathdeuxcat{A}$ et $r$ une \un{}cellule de $F(f)(x)$ vers $x'$ dans $F(a')$. 

Les \deux{}cellules de $(f : a \to a', r : F(f)(x) \to x')$ vers $(g : a \to a', s : F(g)(x) \to x')$ dans $\DeuxInt{\mathdeuxcat{A}}F$ sont les
$$
(\gamma : f \Rightarrow g, \varphi : r \Rightarrow s (F(\gamma))_{x})
$$
dans lesquels $\gamma$ est une \deux{}cellule de $f$ vers $g$ dans $\mathdeuxcat{A}$ et $\varphi$ une \deux{}cellule de $r$ vers $s (F(\gamma))_{x}$ dans $F(a')$.

L'identité de l'objet $(a,x)$ est donnée par 
$$
1_{(a, x)} = (1_{a}, 1_{x}).
$$

Soient $(a, x)$, $(a', x')$ et $(a'', x'')$ trois objets de $\DeuxInt{\mathdeuxcat{A}}F$, $(f : a \to a', r : F(f)(x) \to x')$ une \un{}cellule de $(a, x)$ vers $(a', x')$, et $(f' : a' \to a'', r' : F(f')(x') \to x'')$ une \un{}cellule de $(a', x')$ vers $(a'', x'')$. La composée de ces \un{}cellules est donnée par
$$
(f', r') (f, r) = (f'f, r' F(f')(r)).
$$ 

Soit $(f : a \to a', r : F(f)(x) \to x')$ une \un{}cellule de $(a, x)$ vers $(a', x')$. Son identité est donnée par
$$
1_{(f, r)} = (1_{f}, 1_{r}).
$$

Soient $(f : a \to a', r : F(f)(x) \to x')$, $(g : a \to a', s : F(g)(x) \to x')$ et $(h : a \to a', t : F(h)(x) \to x')$ trois \un{}cellules de $(a,x)$ vers $(a', x')$, $(\gamma : f \Rightarrow g, \varphi : r \Rightarrow s (F(\gamma))_{x})$ une \deux{}cellule de $(f, r)$ vers $(g, s)$ et $(\delta : g \Rightarrow h, \psi : s \Rightarrow t (F(\delta))_{x})$ une \deux{}cellule de $(g, s)$ vers $(h, t)$. La composée de ces \deux{}cellules est donnée par
$$
(\delta, \psi) \CompDeuxUn (\gamma, \varphi) = (\delta \CompDeuxUn \gamma, (\psi \CompDeuxZero (F(\gamma))_{x}) \CompDeuxUn \varphi).
$$

Soient $(f : a \to a', r : F(f)(x) \to x')$ et $(g : a \to a', s : F(g)(x) \to x')$ deux \un{}cellules de $(a,x)$ vers $(a',x')$, $(f' : a' \to a'', r' : F(f')(x') \to x'')$ et $(g' : a' \to a'', s' : F(g')(x') \to x'')$ deux \un{}cellules de $(a',x')$ vers $(a'',x'')$,  $(\gamma : f \Rightarrow g, \varphi : r \Rightarrow s (F(\gamma))_{x})$ une \deux{}cellule de $(f, r)$ vers $(g, s)$ et  $(\gamma' : f' \Rightarrow g', \varphi' : r' \Rightarrow s' (F(\gamma'))_{x'})$ une \deux{}cellule de $(f', r')$ vers $(g', s')$. Leur composée est donnée par
$$
(\gamma', \varphi') \CompDeuxZero (\gamma, \varphi) = (\gamma' \CompDeuxZero \gamma, \varphi' \CompDeuxZero F(f') (\varphi)).
$$

On peut vérifier sans aucune difficulté que cela définit bien une \deux{}catégorie. La remarque \ref{RemarqueCommaGrothendieck} permet toutefois de faire l'économie des calculs dans ce nouveau cas particulier de « structure comma ».
\end{paragr}

\begin{rem}\label{RemarqueCommaGrothendieck}
La construction présentée en \ref{ParagrapheIntegration} peut se voir comme un cas particulier de \emph{\trois{}catégorie comma}\index{3CategorieComma@\trois{}catégorie comma}, généralisation de la notion de \deux{}catégorie comma. La construction qui nous intéresse n'est pas un cas particulier de \deux{}catégorie comma car elle utilise la structure \trois{}catégorique de $\DeuxCat$. La définition pertinente se trouve chez Gray \cite[I, 2, 7, p. 32]{Gray}. La construction de Grothendieck que nous utilisons est alors la \trois{}catégorie comma $[e,F]$ du morphisme constant de valeur $\DeuxCatPonct$ et du morphisme $F$. Les sources de ces morphismes étant en fait des \deux{}catégories (ou, de façon équivalente, des \trois{}catégories dont toutes les \trois{}cellules sont des identités), la \trois{}catégorie comma $[\DeuxCatPonct{},F] = \DeuxInt{\mathdeuxcat{A}}F$ est en fait une \deux{}catégorie. Nous avons appris ce point de vue par Steve Lack.
\end{rem}

\begin{rem}
Pour d'autres travaux portant sur la « construction de Grothendieck » dans un cadre bicatégorique, on pourra se reporter à \cite{Bakovic} ou \cite{CCH}. 
\end{rem}

\begin{exemple}\label{ProjectionExempleIntegrale}
Pour toutes \deux{}catégories $\mathdeuxcat{A}$ et $\mathdeuxcat{B}$, considérons le \DeuxFoncteurStrict{} constant
$$
\begin{aligned}
C_{\mathdeuxcat{B}} : \mathdeuxcat{A} &\to \DeuxCatDeuxCat{}
\\
a &\mapsto \mathdeuxcat{B}
\\
f &\mapsto 1_{\mathdeuxcat{B}}
\\
\varphi &\mapsto 1_{1_{\mathdeuxcat{B}}}.
\end{aligned}
$$
On vérifie sans aucune difficulté que la \deux{}catégorie $\DeuxInt{\mathdeuxcat{A}} C_{\mathdeuxcat{B}}$ s'identifie canoniquement à $\mathdeuxcat{A} \times \mathdeuxcat{B}$.
\end{exemple}


\begin{paragr}\label{AZERTYUIOP}
Soient $\mathdeuxcat{A}$ une \deux{}catégorie et $F : \mathdeuxcat{A} \to \DeuxCatDeuxCat$ un \DeuxFoncteurStrict{}. Les formules
$$
\begin{aligned}
\DeuxInt{\mathdeuxcat{A}}F &\to \mathdeuxcat{A}
\\
(a,x) &\mapsto a
\\
(f,r) &\mapsto f
\\
(\gamma, \varphi) &\mapsto \gamma
\end{aligned}
$$ 
définissent un \DeuxFoncteurStrict{} $P_{F}\index[not]{PF@$P_{F}$} : \DeuxInt{\mathdeuxcat{A}}F \to \mathdeuxcat{A}$, que l'on appellera parfois « projection canonique de l'intégrale de $F$ »\index{projection canonique (d'une intégrale)}\footnote{Les vérifications de fonctorialité sont immédiates ; c'est du reste un résultat formel provenant du fait que la construction de Grothendieck est un cas particulier de structure comma.}.

\end{paragr}

\begin{prop}\label{ProjectionIntegralePrefibration}
Pour toute \deux{}catégorie $\mathdeuxcat{A}$ et tout \DeuxFoncteurStrict{}  $F : \mathdeuxcat{A} \to \DeuxCatDeuxCat$, la projection canonique
$$
P_{F} : \DeuxInt{\mathdeuxcat{A}}F \to \mathdeuxcat{A}
$$ 
est une précoopfibration.
\end{prop}

\begin{proof}
Décrivons d'abord les objets, \un{}cellules et \deux{}cellules de la \deux{}catégorie $\TrancheLax{\left(\DeuxInt{\mathdeuxcat{A}}F\right)}{P_{F}}{a}$. (On laisse au lecteur le soin d'expliciter les détails si besoin est.)

Les objets en sont les $((a', x'), p : a' \to a)$.

Les \un{}cellules de 
$$
((a', x'), p : a' \to a)
$$ 
vers 
$$
((a'', x''), p' : a'' \to a)
$$ 
sont les 
$$
((f : a' \to a'', r : F(f)(x') \to x''), \sigma : p \Rightarrow p'f).
$$

Les \deux{}cellules de 
$$
((f : a' \to a'', r : F(f)(x') \to x''), \sigma : p \Rightarrow p'f)
$$ 
vers 
$$
((g : a' \to a'', s : F(g)(x') \to x''), \sigma' : p \Rightarrow p'g)
$$ 
sont les 
$$
(\gamma : f \Rightarrow g, \varphi : r \Rightarrow s (F(\gamma))_{x'})
$$ 
tels que 
$$
(p' \CompDeuxZero \gamma) \CompDeuxUn \sigma = \sigma'.
$$

Considérons le \DeuxFoncteurStrict{}
$$
\begin{aligned}
J_{a} : \Fibre{\left(\DeuxInt{\mathdeuxcat{A}}F\right)}{P_{F}}{a} &\to \TrancheLax{\left(\DeuxInt{\mathdeuxcat{A}}F\right)}{P_{F}}{a}
\\
(a, x) &\mapsto ((a, x), 1_{a})
\\
(1_{a}, r) &\mapsto ((1_{a}, r), 1_{1_{a}})
\\
(1_{1_{a}}, \varphi) &\mapsto (1_{1_{a}}, \varphi).
\end{aligned}
$$ 

On va vérifier qu'il s'agit d'un préadjoint à droite colax. Soit $((a', x'), p : a' \to a)$ un objet de $\TrancheLax{\DeuxInt{\mathdeuxcat{A}}F}{P_{F}}{a}$. Ainsi, $a$, $a'$, $x'$ et $p$ sont désormais fixés. Décrivons les objets, \un{}cellules et \deux{}cellules de la \deux{}catégorie 
$$
\OpTrancheCoLax{\Fibre{\left(\DeuxInt{\mathdeuxcat{A}}F\right)}{P_{F}}{a}}{J_{a}}{((a', x'), p)}.
$$
(Comme ci-dessus, on laisse au lecteur le soin d'expliciter les détails en cas de besoin.)

Les objets en sont les 
$$
((a, x), ((q : a' \to a, r : F(q)(x') \to x), \sigma : p \Rightarrow q)).
$$

Les \un{}cellules de 
$$
((a, x), ((q : a' \to a, r : F(q)(x') \to x), \sigma : p \Rightarrow q))
$$ 
vers 
$$
((a, x''), ((q' : a' \to a, r' : F(q')(x') \to x''), \sigma' : p \Rightarrow q'))
$$ 
s'identifient aux 
$$
(s : x \to x'', (\gamma : q \Rightarrow q', \varphi : sr \Rightarrow r' (F(\gamma))_{x'}))
$$ 
tels que 
$$
\gamma \CompDeuxUn \sigma = \sigma'.
$$

Les \deux{}cellules de 
$$
(s : x \to x'', (\gamma : q \Rightarrow q', \varphi : sr \Rightarrow r' (F(\gamma))_{x'}))
$$ 
(vérifiant $\gamma \CompDeuxUn \sigma = \sigma'$) vers 
$$
(t : x \to x'', (\mu : q \Rightarrow q', \psi : tr \Rightarrow r' (F(\mu))_{x'}))
$$ 
(vérifiant $\mu \CompDeuxUn \sigma = \sigma'$) s'identifient aux \deux{}cellules 
$$
\tau : s \Rightarrow t
$$ 
telles que 
$$
\mu = \gamma
$$ 
et 
$$
\psi \CompDeuxUn (\tau \CompDeuxZero r) = \varphi.
$$
(L'égalité $\gamma = \mu$ est donc une condition nécessaire à l'existence d'une telle \deux{}cellule.) 

Dans la \deux{}catégorie 
$$
\OpTrancheCoLax{\Fibre{\left(\DeuxInt{\mathdeuxcat{A}}F\right)}{P_{F}}{a}}{J_{a}}{((a', x'), p)},
$$ 
on distingue l'objet 
$$
((a, F(p)(x')), ((p, 1_{F(p)(x')}), 1_{p})).
$$
Soit
$$
((a, x''), ((q' : a' \to a, r' : F(q')(x') \to x''), \sigma' : p \Rightarrow q'))
$$ 
un objet quelconque de 
$$
\OpTrancheCoLax{\Fibre{\left(\DeuxInt{\mathdeuxcat{A}}F\right)}{P_{F}}{a}}{J_{a}}{((a', x'), p)}.
$$
On distingue la \un{}cellule 
$$
(r' (F(\sigma'))_{x'}, (\sigma', 1_{r' (F(\sigma'))_{x'}}))
$$ 
de 
$
((a, F(p)(x')), ((p, 1_{F(p)(x')}), 1_{p}))
$
vers 
$
((a, x''), ((q', r'), \sigma'))
$. En effet, la condition à vérifier n'est autre que la trivialité $\sigma' \CompDeuxUn 1_{p} = \sigma'$.  

Soit 
$$
(s : F(p)(x') \to x'', (\gamma : p \Rightarrow q', \varphi : s \Rightarrow r' (F(\sigma'))_{x'}))
$$ 
une \un{}cellule quelconque de $((a, F(p)(x')), ((p, 1_{F(p)(x')}), 1_{p}))$ vers $((a, x''), ((q', r'), \sigma'))$. On a donc en particulier $\gamma = \sigma'$. Les \deux{}cellules de $(s, (\gamma, \varphi))$ vers $(r' (F(\sigma'))_{x'}, (\sigma', 1_{r' (F(\sigma'))_{x'}}))$ sont données par les \deux{}cellules $\tau : s \Rightarrow r' (F(\sigma'))_{x'}$ telles que $\gamma = \sigma'$ (égalité déjà vérifiée par hypothèse) et $1_{r' (F(\sigma'))_{x'}} (\tau \CompDeuxZero 1_{F(p)(x')}) = \varphi$, c'est-à-dire $\tau = \varphi$. Il est clair que ces conditions impliquent l'existence et l'unicité d'une telle \deux{}cellule, à savoir $\varphi$. 

Ainsi, la \deux{}catégorie $\OpTrancheCoLax{\Fibre{(\DeuxInt{\mathdeuxcat{A}}F)}{P_{F}}{a}}{J_{a}}{((a', x'), p)}$ op-admet un objet admettant un objet final. Par définition, le résultat suit.
\end{proof}

\begin{rem}\label{ProjectionPrefibrationPreuve}
Ainsi, en vertu de l'exemple \ref{ProjectionExempleIntegrale} et de la proposition \ref{ProjectionIntegralePrefibration}, toute projection est une précoopfibration. Les divers morphismes opposés d'une projection étant des projections, toute projection est donc également une préopfibration, une précofibration et une précoopfibration, ce que nous annoncions déjà dans l'exemple \ref{ProjectionPrefibration}.
\end{rem}

\begin{paragr}
Soient toujours $F : \mathdeuxcat{A} \to \DeuxCatDeuxCat$ un \DeuxFoncteurStrict{} et $a$ un objet de ${\mathdeuxcat{A}}$. On vient de voir que le \DeuxFoncteurStrict{} $J_{a} : \Fibre{\left(\DeuxInt{\mathdeuxcat{A}}F\right)}{P_{F}}{a} \to \TrancheLax{\left(\DeuxInt{\mathdeuxcat{A}}F\right)}{P_{F}}{a}$, dont nous avons rappelé la définition au cours de la démonstration de la proposition \ref{ProjectionIntegralePrefibration}, est un préadjoint à droite colax. En vertu des résultats généraux de la section \ref{SectionPreadjoints}, il existe un \DeuxFoncteurLax{} $K_{a}\index[not]{Ka@$K_{a}$} : \TrancheLax{\left(\DeuxInt{\mathdeuxcat{A}}F\right)}{P_{F}}{a} \to \Fibre{\left(\DeuxInt{\mathdeuxcat{A}}F\right)}{P_{F}}{a}$ ainsi qu'une \DeuxTransformationCoLax{} $1_{\TrancheLax{\left(\DeuxInt{\mathdeuxcat{A}}F\right)}{P_{F}}{a}} \Rightarrow J_{a} K_{a}$. Des calculs quelque peu pénibles mais sans difficulté réelle montrent\footnote{Nous n'utiliserons pas le caractère canonique de $K_{a}$, bien qu'il figure dans l'énoncé de la proposition \ref{ProprietesJaKa}.} que $K_{a}$ est donné par (en conservant des notations consistantes avec celles adoptées précédemment, et donc « évidentes »)
$$
\begin{aligned}
\TrancheLax{\left(\DeuxInt{\mathdeuxcat{A}}F\right)}{P_{F}}{a} &\longrightarrow \Fibre{\left(\DeuxInt{\mathdeuxcat{A}}F\right)}{P_{F}}{a}
\\
((a',x'), p : a' \to a) &\longmapsto (a, F(p)(x'))
\\
((f : a' \to a'', r : F(f)(x') \to x''), \sigma : p \Rightarrow p'f) &\longmapsto (1_{a}, F(p')(r) (F(\sigma))_{x'})
\\
(\gamma, \varphi) &\longmapsto (1_{1_{a}}, F(p')(\varphi) \CompDeuxZero (F(\sigma))_{x'}).
\end{aligned}
$$

Vérifions qu'il s'agit d'un \DeuxFoncteurStrict{}. Pour tout objet $((a',x'), p)$ de $\TrancheLax{\left(\DeuxInt{\mathdeuxcat{A}}F\right)}{P_{F}}{a}$,
$$
\begin{aligned}
K_{a} (1_{((a',x'),p)}) &= K_{a} ((1_{a'}, 1_{x'}), 1_{p})
\\
&= (1_{a}, F(p)(1_{x'}) (F(1_{p}))_{x'})
\\
&= (1_{a}, 1_{F(p)(x')})
\\
&= 1_{(a, F(p)(x'))}
\\
&= 1_{K_{a} ((a',x'), p)}.
\end{aligned}
$$

Pour tout couple de \un{}cellules composables
$$
((f,r), \sigma) : ((a',x'),p) \to ((a'',x''),p')
$$
et
$$
((f',r'), \sigma') : ((a'',x''),p') \to ((a''',x'''),p'')
$$
de $\TrancheLax{\left(\DeuxInt{\mathdeuxcat{A}}F\right)}{P_{F}}{a}$, 
$$
\begin{aligned}
K_{a} (((f',r'), \sigma') ((f,r), \sigma)) &= K_{a} ((f'f, r' F(f')(r)), (\sigma' \CompDeuxZero f) \CompDeuxUn \sigma)
\\
&= (1_{a}, F(p'') (r' F(f')(r)) (F((\sigma' \CompDeuxZero f) \CompDeuxUn \sigma))_{x'})
\\
&= (1_{a}, F(p'')(r') F(p'') F(f') (r) (F(\sigma') \CompDeuxZero F(f))_{x'} (F(\sigma))_{x'})
\\
&= (1_{a}, F(p'')(r') F(p'') F(f') (r) (F(\sigma'))_{F(f)(x')} (F(\sigma))_{x'}).
\end{aligned}
$$ 
D'autre part, 
$$
K_{a} ((f',r'), \sigma') K_{a} ((f,r), \sigma) = (1_{a}, F(p'')(r') (F(\sigma'))_{x''}) (1_{a}, F(p')(r) (F(\sigma))_{x'}).
$$
L'égalité 
$$
K_{a} (((f',r'), \sigma') ((f,r), \sigma)) = K_{a} ((f',r'), \sigma') K_{a} ((f,r), \sigma)
$$
résulte donc du fait que le diagramme
$$
\xymatrix{
F(p') F(f) (x')
\ar[rrr]^{(F(\sigma'))_{F(f)(x')}}
\ar[dd]_{F(p')(r)}
&&&
F(p'') F(f') F(f) (x')
\ar[dd]^{F(p'') F(f')(r)}
\\
\\
F(p') (x'')
\ar[rrr]_{(F(\sigma'))_{x''}}
&&&
F(p'') F(f') (x'')
}
$$
est commutatif par naturalité de $F(\sigma') : F(p') \Rightarrow F(p'') F(f')$. 

Pour toute \un{}cellule
$$
((f,r), \sigma) : ((a',x'), p) \to ((a'',x''), p')
$$
de $\TrancheLax{\left(\DeuxInt{\mathdeuxcat{A}}F\right)}{P_{F}}{a}$, 
$$
\begin{aligned}
K_{a} (1_{((f,r), \sigma)}) &= K_{a} (1_{f}, 1_{r})
\\
&= (1_{1_{a}}, F(p') (1_{r}) \CompDeuxZero (F(\sigma))_{x'})
\\
&= (1_{1_{a}}, 1_{F(p')(r)} \CompDeuxZero (F(\sigma))_{x'})
\\
&= (1_{1_{a}}, 1_{F(p')(r) (F(\sigma))_{x'}})
\\
&= 1_{(1_{a}, F(p')(r) (F(\sigma))_{x'})}
\\
&= 1_{K_{a} ((f,r), \sigma)}.
\end{aligned}
$$

Pour tout triplet de \un{}cellules parallèles $((f,r), \sigma)$, $((g,s), \tau)$ et $((h,t), \mu)$ de $((a',x'), p)$ vers $((a'',x''), p')$, pour tout couple de \deux{}cellules composables
$$
(\gamma, \varphi) : ((f,r), \sigma) \Rightarrow ((g,s), \tau)
$$
et
$$
(\delta, \psi) : ((g,s), \tau) \Rightarrow ((h,t), \mu)
$$
dans $\TrancheLax{\left(\DeuxInt{\mathdeuxcat{A}}F\right)}{P_{F}}{a}$, il suffit de faire un dessin pour vérifier les égalités
$$
\begin{aligned}
K_{a} ((\delta, \psi) \CompDeuxUn (\gamma, \varphi)) &= K_{a} (\delta \CompDeuxUn \gamma, (\psi \CompDeuxZero (F(\gamma))_{x'}) \CompDeuxUn \varphi)
\\
&= (1_{1_{a}}, F(p') ((\psi \CompDeuxZero (F(\gamma))_{x'}) \CompDeuxUn \varphi) \CompDeuxZero (F(\sigma))_{x'})
\\
&= (1_{1_{a}}, ((F(p') (\psi) \CompDeuxZero F(p') ((F(\gamma))_{x'})) F(p')(\varphi)) \CompDeuxZero (F(\sigma))_{x'})
\end{aligned}
$$
et
$$
\begin{aligned}
K_{a} (\delta, \psi) \CompDeuxUn K_{a} (\gamma, \varphi) &= (1_{1_{a}}, F(p') (\psi) \CompDeuxZero (F(\tau))_{x'}) (1_{1_{a}}, F(p') (\varphi) \CompDeuxZero (F(\sigma))_{x'})
\\
&= (1_{1_{a}}, (F(p') (\psi) \CompDeuxZero F(p') ((F(\gamma))_{x'}) (F(\sigma))_{x'}) (F(p')(\varphi) \CompDeuxZero (F(\sigma))_{x'}))
\\
&= (1_{1_{a}}, ((F(p') (\psi) \CompDeuxZero F(p') ((F(\gamma))_{x'})) F(p')(\varphi)) \CompDeuxZero (F(\sigma))_{x'}).
\end{aligned}
$$
Ainsi, 
$$
K_{a} ((\delta, \psi) \CompDeuxUn (\gamma, \varphi)) = K_{a} (\delta, \psi) \CompDeuxUn K_{a} (\gamma, \varphi).
$$

Les formules ci-dessus définissent donc bien un \DeuxFoncteurStrict{} $K_{a}$. On constate que c'est une rétraction de $J_{a}$ : 
$$
K_{a} J_{a} = 1_{\Fibre{\left(\DeuxInt{\mathdeuxcat{A}}F\right)}{P_{F}}{a}}.
$$

De plus, $K_{a}$ \emph{est un préadjoint à gauche lax}. En effet, soit $(a,x)$ un objet de $\Fibre{(\DeuxInt{\mathdeuxcat{A}}F)}{P_{F}}{a}$. Les objets, \un{}cellules et \deux{}cellules de la \deux{}catégorie 
$$
\TrancheLax{\left(\TrancheLax{\left(\DeuxInt{\mathdeuxcat{A}}F\right)}{P_{F}}{a}\right)}{K_{a}}{(a,x)}
$$ 
se décrivent comme suit. Les objets sont les 
$$
(((a',x'), p : a' \to a), (1_{a}, r : F(p)(x') \to x)).
$$
Les \un{}cellules de 
$$
(((a',x'), p : a' \to a), (1_{a}, r : F(p)(x') \to x))
$$ 
vers 
$$
(((a'',x''), p' : a'' \to a), (1_{a}, r' : F(p')(x'') \to x))
$$ 
sont les 
$$
(((f : a' \to a'', s : F(f)(x') \to x''), \sigma : p \Rightarrow p'f), (1_{1_{a}}, \varphi : r \Rightarrow r' F(p')(s) (F(\sigma))_{x'})).
$$
Les \deux{}cellules de 
$$
(((f : a' \to a'', s : F(f)(x') \to x''), \sigma : p \Rightarrow p'f), (1_{1_{a}}, \varphi : r \Rightarrow r' F(p')(s) (F(\sigma))_{x'}))
$$ 
vers 
$$
(((g : a' \to a'', t : F(g)(x') \to x''), \tau : p \Rightarrow p'g), (1_{1_{a}}, \psi : r \Rightarrow r' F(p')(t) (F(\tau))_{x'}))
$$ 
sont les
$$
(\mu : f \Rightarrow g, \nu : s \Rightarrow t (F(\mu))_{x'})
$$
tels que 
$$
(p' \CompDeuxZero \mu) \CompDeuxUn \sigma = \tau
$$
et 
$$
r' \CompDeuxZero F(p')(\nu) \CompDeuxZero (F(\sigma))_{x'} = \psi.
$$
Dans cette \deux{}catégorie, on distingue l'objet 
$$
(((a,x),1_{a}), (1_{a}, 1_{x})).
$$
Soit 
$$
(((a',x'), p : a' \to a), (1_{a}, r : F(p)(x') \to x))
$$ 
un objet quelconque de 
$\TrancheLax{\left(\TrancheLax{\left(\DeuxInt{\mathdeuxcat{A}}F\right)}{P_{F}}{a}\right)}{K_{a}}{(a,x)}$. Alors,  
$$
(((p,r),1_{p}), (1_{1_{a}}, 1_{r}))
$$
définit une \un{}cellule de $(((a',x'), p), (1_{a}, r))$ vers $(((a,x),1_{a}), (1_{a}, 1_{x}))$. Soit 
$$
(((f : a' \to a, s : F(f)(x') \to x), \sigma : p \Rightarrow f), (1_{1_{a}}, \varphi : r \Rightarrow s (F(\sigma))_{x'}))
$$ 
une \un{}cellule quelconque de $(((a',x'), p), (1_{a}, r))$ vers $(((a,x),1_{a}), (1_{a}, 1_{x}))$. Les \deux{}cellules de $$(((p,r),1_{p}), (1_{1_{a}}, 1_{r}))$$ vers $$(((f,s),\sigma),(1_{1_{a}}, \varphi))$$ sont alors les 
$$
(\mu : p \Rightarrow f, \nu : r \Rightarrow s(F(\mu))_{x'})
$$ 
tels que $\mu = \sigma$ et $\nu = \varphi$, ce qui force $(\mu, \nu) = (\sigma, \varphi)$, et l'on vérifie que les égalités requises sont alors satisfaites. Ainsi, la \deux{}catégorie $\TrancheLax{\left(\TrancheLax{\left(\DeuxInt{\mathdeuxcat{A}}F\right)}{P_{F}}{a}\right)}{K_{a}}{(a,x)}$ admet un objet admettant un objet initial. Par définition, $K_{a}$ est donc un préadjoint à gauche lax, ce qu'il fallait vérifier. Résumons.  
\end{paragr}

\begin{prop}\label{ProprietesJaKa}
Pour tout \DeuxFoncteurStrict{} $F : \mathdeuxcat{A} \to \DeuxCatDeuxCat$ et tout objet $a$ de $\mathdeuxcat{A}$, le \DeuxFoncteurStrict{} canonique $J_{a} : \Fibre{\left(\DeuxInt{\mathdeuxcat{A}}F\right)}{P_{F}}{a} \to \TrancheLax{\left(\DeuxInt{\mathdeuxcat{A}}F\right)}{P_{F}}{a}$ (qui est un préadjoint à droite colax) admet une rétraction canonique $K_{a}$ qui est un \DeuxFoncteurStrict{} et un préadjoint à gauche lax. De plus, il existe une \DeuxTransformationCoLax{} canonique $1_{\TrancheLax{\left(\DeuxInt{\mathdeuxcat{A}}F\right)}{P_{F}}{a}} \Rightarrow J_{a} K_{a}$.  
\end{prop} 

On vérifie maintenant une propriété supplémentaire du \DeuxFoncteurStrict{} $J_{a}$ : c'est l'inclusion d'un dual de rétracte par déformation fort dans $\DeuxCatLax{}$. 

\begin{paragr}\label{DefSegment}
On rappelle que, si $M$ est une catégorie admettant un objet final $e_{M}$\index[not]{eM@$e_{M}$}, un \emph{segment}\index{segment} de $M$ est un triplet $(I, \partial_{0}, \partial_{1})$\index[not]{(I01@$(I, \partial_{0}, \partial_{1})$}, où $I$ est un objet de $M$ et $\partial_{0}, \partial_{1} : e_{M} \to I$ sont des morphismes de $M$. 
\end{paragr}

\begin{df}
Soit $(I, \partial_{0}, \partial_{1})$ un segment d'une catégorie $M$ admettant des produits. On dit qu'un morphisme $i : A \to B$ de $M$ est \emph{l'inclusion d'un rétracte par déformation fort}\index{inclusion de rétracte par déformation fort} relativement au segment $(I, \partial_{0}, \partial_{1})$ s'il existe une rétraction $r : B \to A$ de $i$ (c'est-à-dire que $ri = 1_{A}$) et un morphisme (une « homotopie ») $h : I \times B \to B$ tel que $h (\partial_{0}, 1_{B}) = ir$, $h (\partial_{1}, 1_{B}) = 1_{B}$ et $h (1_{I} \times i) = i p_{A}$, $p_{A}$ désignant la projection canonique $I \times A \to A$. 

On dit qu'un morphisme $i : A \to B$ de $X$ est \emph{l'inclusion d'un dual de rétracte par déformation fort}\index{inclusion de dual de rétracte par déformation fort} relativement au segment $(I, \partial_{0}, \partial_{1})$ s'il existe une rétraction $r : B \to A$ de $i$ et un morphisme $h : I \times B \to B$ tel que $h (\partial_{0}, 1_{B}) = 1_{B}$, $h (\partial_{1}, 1_{B}) = ir$ et $h (1_{I} \times i) = i p_{A}$, $p_{A}$ désignant la projection canonique $I \times A \to A$. 
\end{df}

Dans tous les cas que nous examinerons, sauf mention contraire, le segment considéré sera $([1], 0, 1)$. 

\begin{prop}
Soient $\mathdeuxcat{A}$ une \deux{}catégorie, $F : \mathdeuxcat{A} \to \DeuxCatDeuxCat$ un \DeuxFoncteurStrict{} et $a$ un objet de $\mathdeuxcat{A}$. Alors, le \DeuxFoncteurStrict{} canonique
$$
\begin{aligned}
J_{a} : \Fibre{\left(\DeuxInt{\mathdeuxcat{A}}F\right)}{P_{F}}{a} &\to \TrancheLax{\left(\DeuxInt{\mathdeuxcat{A}}F\right)}{P_{F}}{a}
\\
(a,x) &\mapsto ((a,x), 1_{a})
\\
(1_{a}, r) &\mapsto ((1_{a}, r), 1_{1_{a}})
\\
(1_{1_{a}}, \varphi) &\mapsto (1_{1_{a}}, \varphi)
\end{aligned}
$$
est l'inclusion d'un dual de rétracte par déformation fort dans $\DeuxCatLax{}$. 
\end{prop}

\begin{proof}
On sait déjà que le \DeuxFoncteurStrict{} canonique
$$
\begin{aligned}
K_{a} : \TrancheLax{\left(\DeuxInt{\mathdeuxcat{A}}F\right)}{P_{F}}{a} &\to \Fibre{\left(\DeuxInt{\mathdeuxcat{A}}F\right)}{P_{F}}{a}
\\
((a',x'), p : a' \to a) &\mapsto (a, F(p)(x'))
\\
((f : a' \to a'', r : F(f)(x') \to x''), \sigma : p \Rightarrow p'f) &\mapsto (1_{a}, F(p')(r) (F(\sigma))_{x'})
\\
(\gamma, \varphi) &\mapsto (1_{1_{a}}, F(p')(\varphi) (F(\sigma))_{x'})
\end{aligned}
$$
est une rétraction de $J_{a}$. On sait de plus qu'il existe une \DeuxTransformationCoLax{} canonique
$$
\alpha : 1_{\TrancheLax{\left(\DeuxInt{\mathdeuxcat{A}}F\right)}{P_{F}}{a}} \Rightarrow J_{a} K_{a}
$$ 
définie comme suit. Pour tout objet $((a', x'), p)$ de $\TrancheLax{\left(\DeuxInt{\mathdeuxcat{A}}F\right)}{P_{F}}{a}$, 
$$
J_{a} K_{a} ((a',x'), p) = ((a, F(p)(x')), 1_{a})
$$ 
et
$$
\alpha_{((a', x'), p)} = ((p, 1_{F(p)(x')}), 1_{p}).
$$
Pour toute \un{}cellule 
$$
((f : a' \to a'', r : F(f)(x') \to x''), \sigma : p \Rightarrow p'f)
$$
de $((a', x'), p)$ vers $((a'', x''), p')$ dans $\TrancheLax{\left(\DeuxInt{\mathdeuxcat{A}}F\right)}{P_{F}}{a}$, 
$$
J_{a} K_{a} ((f, r), \sigma) = ((1_{a}, F(p')(r) (F(\sigma))_{x'}), 1_{1_{a}})
$$
et
$$
\alpha_{((f, r), \sigma)} = (\sigma, 1_{F(p')(r) (F(\sigma))_{x'}})
$$
est une \deux{}cellule « descendante » dans le diagramme
$$
\xymatrix{
((a',x'), p)
\ar[rrrrr]^{((p, 1_{F(p)(x')}), 1_{p})}
\ar[dd]_{((f, r), \sigma)}
&&&&&
((a, F(p)(x')), 1_{a}) = J_{a} K_{a} ((a',x'), p)
\ar[dd]^{((1_{a}, F(p')(r) (F(\sigma))_{x'}), 1_{1_{a}}) = J_{a} K_{a} ((f, r), \sigma)}
\\
\\
((a'',x''), p')
\ar[rrrrr]_{((p', 1_{F(p')(x'')}), 1_{p'})}
&&&&&
((a, F(p')(x'')), 1_{a}) = J_{a} K_{a} ((a'',x''), p')
&,
}
$$
c'est-à-dire une \deux{}cellule de 
$$
((1_{a}, F(p')(r) (F(\sigma))_{x'}), 1_{1_{a}}) ((p, 1_{F(p)(x')}), 1_{p}) = ((p, F(p')(r) (F(\sigma))_{x'}), 1_{p})
$$
vers
$$
((p', 1_{F(p')(x'')}), 1_{p'}) ((f, r), \sigma) = ((p'f, F(p')(r)), \sigma)
$$
dans $\TrancheLax{\left(\DeuxInt{\mathdeuxcat{A}}F\right)}{P_{F}}{a}$.

En vertu du lemme \ref{DeuxTransFoncLax}, il existe un \DeuxFoncteurLax{} $h : [1] \times (\TrancheLax{\left(\DeuxInt{\mathdeuxcat{A}}F\right)}{P_{F}}{a}) \to \TrancheLax{\left(\DeuxInt{\mathdeuxcat{A}}F\right)}{P_{F}}{a}$ tel que le diagramme
$$
\xymatrix{
&
[1] \times (\TrancheLax{\left(\DeuxInt{\mathdeuxcat{A}}F\right)}{P_{F}}{a})
\ar[dd]^{h}
\\
\TrancheLax{\left(\DeuxInt{\mathdeuxcat{A}}F\right)}{P_{F}}{a}
\ar[ur]^{(\{ 0 \}, 1)}
\ar[dr]_{1}
&&
\TrancheLax{\left(\DeuxInt{\mathdeuxcat{A}}F\right)}{P_{F}}{a}
\ar[ul]_{(\{1\}, 1)}
\ar[dl]^{J_{a} K_{a}}
\\
&
\TrancheLax{\left(\DeuxInt{\mathdeuxcat{A}}F\right)}{P_{F}}{a}
}
$$
soit commutatif. En vertu des formules générales (voir la démonstration du lemme \ref{DeuxTransFoncLax}), $h$ se décrit comme suit. Sa restriction à $\{ 0 \} \times (\TrancheLax{\left(\DeuxInt{\mathdeuxcat{A}}F\right)}{P_{F}}{a})$ (\emph{resp.} $\{ 1 \} \times (\TrancheLax{\left(\DeuxInt{\mathdeuxcat{A}}F\right)}{P_{F}}{a})$) s'identifie à $1_{\TrancheLax{\left(\DeuxInt{\mathdeuxcat{A}}F\right)}{P_{F}}{a}}$ (\emph{resp.} $J_{a} K_{a}$). Pour toute \un{}cellule 
$$
((f : a' \to a'', r : F(f)(x') \to x''), \sigma : p \Rightarrow p'f)
$$ 
de $((a',x'),p)$ vers $((a'',x''), p')$ dans $\TrancheLax{\left(\DeuxInt{\mathdeuxcat{A}}F\right)}{P_{F}}{a}$,
$$
h(0 \to 1, ((f,r), \sigma)) = ((p'f, F(p')(r)), \sigma).
$$
Soit de plus 
$$
((f' : a'' \to a''', r' : F(f')(x'') \to x'''), \sigma' : p' \Rightarrow p''f')$$ 
une \un{}cellule de $((a'',x''),p')$ vers $((a''',x'''), p'')$. On a
$$
h_{(0 \to 1, ((f',r'),\sigma')),(0 \to 0, ((f,r),\sigma))} = 1_{((p''f'f, F(p'')(r') F(p'') F(f') (r)), (\sigma' \CompDeuxZero f) \CompDeuxUn \sigma)}
$$
et
$$
\begin{aligned}
h_{(1 \to 1, ((f',r'),\sigma')), (0 \to 1, ((f,r), \sigma))} &= (\sigma' \CompDeuxZero f, 1_{F(p'')(r') (F(\sigma'))_{x''} F(p')(r)})
\\
&= (\sigma' \CompDeuxZero f, 1_{F(p'')(r') F(p'')F(f')(r) (F(\sigma'))_{F(f)(x')}})
\end{aligned}.
$$
Soient $((f,r), \sigma)$ et $((g,s), \tau)$ deux \un{}cellules de $((a',x'), p)$ vers $((a'',x''), p')$ et $(\gamma, \varphi)$ une \deux{}cellule de $((f,r), \sigma)$ vers $((g,s), \tau)$. Alors, 
$$
h(1_{0 \to 1}, (\gamma, \varphi)) = (p' \CompDeuxZero \gamma, F(p')(\varphi)).
$$
Les égalités $h (\{0\} , 1_{\TrancheLax{\left(\DeuxInt{\mathdeuxcat{A}}F\right)}{P_{F}}{a}}) = 1_{\TrancheLax{\left(\DeuxInt{\mathdeuxcat{A}}F\right)}{P_{F}}{a}}$ et $h (\{1\} , 1_{\TrancheLax{\left(\DeuxInt{\mathdeuxcat{A}}F\right)}{P_{F}}{a}}) = J_{a} K_{a}$ sont évidemment satisfaites. Vérifions l'égalité $h (1_{[1]} \times J_{a}) = J_{a} p_{\Fibre{\left(\DeuxInt{\mathdeuxcat{A}}F\right)}{P_{F}}{a}}$, avec 
$$
p_{\Fibre{\left(\DeuxInt{\mathdeuxcat{A}}F\right)}{P_{F}}{a}} : [1] \times \Fibre{\left(\DeuxInt{\mathdeuxcat{A}}F\right)}{P_{F}}{a} \to \Fibre{\left(\DeuxInt{\mathdeuxcat{A}}F\right)}{P_{F}}{a}
$$ 
la projection canonique. 

Pour tout objet $(a,x)$ de $\Fibre{\left(\DeuxInt{\mathdeuxcat{A}}F\right)}{P_{F}}{a}$, 
$$
\begin{aligned}
h (1_{[1]} \times J_{a}) (0, (a,x)) &= h(0,((a,x),1_{a}))
\\
&= ((a,x), 1_{a}),
\end{aligned}
$$
$$
\begin{aligned}
J_{a} p_{\Fibre{\left(\DeuxInt{\mathdeuxcat{A}}F\right)}{P_{F}}{a}} (0, (a,x)) &= J_{a} (a,x)
\\
&= ((a,x),1_{a}),
\end{aligned}
$$
$$
\begin{aligned}
h (1_{[1]} \times J_{a}) (1,(a,x)) &= h (1,((a,x),1_{a}))
\\
&= J_{a} K_{a} ((a,x),1_{a})
\\
&= ((a, F(1_{a})(x)),1_{a})
\\
&= ((a,x), 1_{a})
\end{aligned}
$$
et
$$
\begin{aligned}
J_{a} p_{\Fibre{\left(\DeuxInt{\mathdeuxcat{A}}F\right)}{P_{F}}{a}} (1, (a,x)) &= J_{a} (a,x)
\\
&= ((a,x), 1_{a}).
\end{aligned}
$$
Pour toute \un{}cellule $(1_{a}, r)$ de $(a,x)$ vers $(a,x')$ dans $\Fibre{\left(\DeuxInt{\mathdeuxcat{A}}F\right)}{P_{F}}{a}$,  
$$
\begin{aligned}
h (1_{[1]} \times J_{a}) (0 \to 0, (1_{a}, r)) &= h(0 \to 0, ((1_{a},r), 1_{1_{a}}))
\\
&= ((1_{a}, r), 1_{1_{a}}),
\end{aligned}
$$
$$
\begin{aligned}
J_{a} p_{\Fibre{\left(\DeuxInt{\mathdeuxcat{A}}F\right)}{P_{F}}{a}} (0 \to 0, (1_{a},r)) &= J_{a} (1_{a}, r)
\\
&= ((1_{a}, r), 1_{1_{a}}),
\end{aligned}
$$
$$
\begin{aligned}
h (1_{[1]} \times J_{a}) (1 \to 1, (1_{a}, r)) &= h(1 \to 1, ((1_{a}, r), 1_{1_{a}})) 
\\
&= J_{a} K_{a} ((1_{a}, r), 1_{1_{a}})
\\
&= ((1_{a}, F(1_{a})(r) (F(1_{1_{a}}))_{x'}), 1_{1_{a}})
\\
&= ((1_{a}, r), 1_{1_{a}}),
\end{aligned}
$$
$$
\begin{aligned}
J_{a} p_{\Fibre{\left(\DeuxInt{\mathdeuxcat{A}}F\right)}{P_{F}}{a}} (1 \to 1, (1_{a}, r)) &= J_{a} (1_{a}, r)
\\
&= ((1_{a}, r), 1_{1_{a}}),
\end{aligned}
$$
$$
\begin{aligned}
h (1_{[1]} \times J_{a}) (0 \to 1, (1_{a}, r)) &= h(0 \to 1, ((1_{a},r), 1_{1_{a}}))
\\
&= ((1_{a} 1_{a}, F(1_{a})(r)), 1_{1_{a}})
\\
&= ((1_{a}, r), 1_{1_{a}})
\end{aligned}
$$
et
$$
\begin{aligned}
J_{a} p_{\Fibre{\left(\DeuxInt{\mathdeuxcat{A}}F\right)}{P_{F}}{a}} (0 \to 1, (1_{a}, r)) &= J_{a} (1_{a}, r)
\\
&= ((1_{a},r), 1_{1_{a}}).
\end{aligned}
$$
Soient $(1_{a}, r) : (a,x) \to (a,x')$ et $(1_{a}, r') : (a,x') \to (a, x'')$ deux \un{}cellules composables de $\Fibre{\left(\DeuxInt{\mathdeuxcat{A}}F\right)}{P_{F}}{a}$. Alors, comme $1_{\TrancheLax{\left(\DeuxInt{\mathdeuxcat{A}}F\right)}{P_{F}}{a}}$ et $J_{a} K_{a}$ sont des \DeuxFoncteursStricts{}, 
$$
(h (1_{[1]} \times J_{a}))_{(0 \to 0, (1_{a},r')), (0 \to 0, (1_{a}, r))}
$$
et 
$$
(h (1_{[1]} \times J_{a}))_{(1 \to 1, (1_{a},r')), (1 \to 1, (1_{a}, r))}
$$
sont des identités. De plus, 
$$
\begin{aligned}
(h (1_{[1]} \times J_{a}))_{(0 \to 1, (1_{a}, r')), (0 \to 0, (1_{a}, r))} &= h_{(1_{[1]} \times J_{a}) (0 \to 1, (1_{a}, r')), (1_{[1]} \times J_{a}) (0 \to 0, (1_{a}, r))}
\\
&= h_{(0 \to 1, ((1_{a}, r'), 1_{1_{a}})), (0 \to 0, ((1_{a}, r), 1_{1_{a}}))}
\end{aligned}
$$
est une identité (voir les formules ci-dessus) et il en est de même de
$$
\begin{aligned}
(h (1_{[1]} \times J_{a}))_{(1 \to 1, (1_{a}, r')), (0 \to 1, (1_{a}, r))} &= h_{(1_{[1]} \times J_{a}) (1 \to 1, (1_{a}, r')), (1_{[1]} \times J_{a}) (0 \to 1, (1_{a}, r))}
\\
&= h_{(1 \to 1, ((1_{a}, r'), 1_{1_{a}})), (0 \to 1, ((1_{a}, r), 1_{1_{a}}))}.
\end{aligned}
$$
Le \deux{}foncteur $h (1_{[1]} \times J_{a})$ est donc strict. Il ne reste plus qu'à vérifier qu'il coïncide avec $J_{a} p_{\Fibre{\left(\DeuxInt{\mathdeuxcat{A}}F\right)}{P_{F}}{a}}$ sur les \deux{}cellules. Soit donc $(1_{a}, r)$ et $(1_{a}, r')$ deux \un{}cellules de $(a,x)$ vers $(a,x')$ et $(1_{1_{a}}, \varphi)$ une \deux{}cellule de $(1_{a}, r)$ vers $(1_{a}, r')$ dans $\Fibre{\left(\DeuxInt{\mathdeuxcat{A}}F\right)}{P_{F}}{a}$. Alors, 
$$
\begin{aligned}
J_{a} p_{\Fibre{\left(\DeuxInt{\mathdeuxcat{A}}F\right)}{P_{F}}{a}} (1_{0 \to 0}, (1_{1_{a}}, \varphi)) &= J_{a} (1_{1_{a}}, \varphi)
\\
&= (1_{1_{a}}, \varphi),
\end{aligned}
$$
$$
\begin{aligned}
(h (1_{[1]} \times J_{a})) (1_{0 \to 0}, (1_{1_{a}}, \varphi)) &= h(1_{0 \to 0}, (1_{1_{a}}, \varphi))
\\
&= (1_{1_{a}}, \varphi),
\end{aligned}
$$
$$
\begin{aligned}
J_{a} p_{\Fibre{\left(\DeuxInt{\mathdeuxcat{A}}F\right)}{P_{F}}{a}} (1_{1 \to 1}, (1_{1_{a}}, \varphi)) &= J_{a} (1_{1_{a}}, \varphi)
\\
&= (1_{1_{a}}, \varphi),
\end{aligned}
$$
$$
\begin{aligned}
(h (1_{[1]} \times J_{a})) (1_{1 \to 1}, (1_{1_{a}}, \varphi)) &= h(1_{1 \to 1}, (1_{1_{a}}, \varphi)) 
\\
&= J_{a} K_{a} (1_{1_{a}}, \varphi)
\\
&= J_{a} (1_{1_{a}}, \varphi)
\\
&= (1_{1_{a}}, \varphi),
\end{aligned}
$$
$$
\begin{aligned}
J_{a} p_{\Fibre{\left(\DeuxInt{\mathdeuxcat{A}}F\right)}{P_{F}}{a}} (1_{0 \to 1}, (1_{1_{a}}, \varphi)) &= J_{a} (1_{1_{a}}, \varphi)
\\
&= (1_{1_{a}}, \varphi)
\end{aligned}
$$
et
$$
\begin{aligned}
(h (1_{[1]} \times J_{a})) (1_{0 \to 1}, (1_{1_{a}}, \varphi)) &= h(1_{0 \to 1}, (1_{1_{a}}, \varphi))
\\
&= (1_{1_{a}}, \varphi).
\end{aligned}
$$

Cela termine les vérifications de l'égalité 
$$
h (1_{[1]} \times J_{a}) = J_{a} p_{\Fibre{\left(\DeuxInt{\mathdeuxcat{A}}F\right)}{P_{F}}{a}}
$$
et donc la vérification de la proposition.
\end{proof}

\begin{paragr}\label{IntegrationAdjonction}
On vérifie maintenant que l'opération d'intégration fournit un exemple d'adjonction lax-colax. Plus précisément : 
\begin{prop}
En conservant les notations employées ci-dessus, le couple de \DeuxFoncteursStricts{} $(K_{a}, J_{a})$ forme une adjonction lax-colax.
\end{prop}

\begin{proof}
Référons-nous à la définition \ref{DefAdjonctionLaxCoLax}. On y procède aux substitutions suivantes : $\TrancheLax{\left(\DeuxInt{\mathdeuxcat{A}}F\right)}{P_{F}}{a}$ pour $\mathdeuxcat{A}$, $\Fibre{\left(\DeuxInt{\mathdeuxcat{A}}F\right)}{P_{F}}{a}$ pour $\mathdeuxcat{B}$, $K_{a}$ pour $u$, $J_{a}$ pour $v$, $1_{1_{\Fibre{\left(\DeuxInt{\mathdeuxcat{A}}F\right)}{P_{F}}{a}}}$ pour $p$, $\alpha$ pour $q$. Comme $J_{a}$ et $K_{a}$ sont stricts, les conditions de cohérence ALC 1 à ALC 3 ne stipulent rien d'autre que le fait que $1_{1_{\Fibre{\left(\DeuxInt{\mathdeuxcat{A}}F\right)}{P_{F}}{a}}}$ est une \DeuxTransformationCoLax{} de $K_{a} J_{a} = 1_{\Fibre{\left(\DeuxInt{\mathdeuxcat{A}}F\right)}{P_{F}}{a}}$ vers $1_{\Fibre{\left(\DeuxInt{\mathdeuxcat{A}}F\right)}{P_{F}}{a}}$, ce qui est trivial, et les conditions de cohérence ALC 4 à ALC 6 ne stipulent rien d'autre que le fait que $\alpha$ est une \DeuxTransformationCoLax{} de $1_{\TrancheLax{\left(\DeuxInt{\mathdeuxcat{A}}F\right)}{P_{F}}{a}}$ vers $J_{a} K_{a}$, ce que l'on sait déjà. Il ne reste donc plus qu'à définir $\sigma$ et $\tau$ et vérifier que les conditions de cohérence ALC 7 à ALC 10 de la définition \ref{DefAdjonctionLaxCoLax} sont satisfaites. 

Pour tout objet $((a',x'), p : a' \to a)$ de $\TrancheLax{\left(\DeuxInt{\mathdeuxcat{A}}F\right)}{P_{F}}{a}$, 
$$
\begin{aligned}
K_{a} (\alpha_{((a',x'), p)}) &= K_{a} ((p, 1_{F(p)(x')}), 1_{p})
\\
&= (1_{a}, F(1_{a}) (1_{F(p)(x')}) (F(1_{p}))_{x'})
\\
&= (1_{a}, 1_{F(p)(x')})
\\
&= 1_{(a, F(p)(x'))}
\\
&= 1_{K_{a}((a',x'), p)}.
\end{aligned}
$$
On pose
$$
\begin{aligned}
\sigma_{((a',x'), p)} &= 1_{1_{K_{a}((a',x'), p)}}
\\
&= 1_{1_{(a, F(p)(x'))}}
\\
&= 1_{(1_{a}, 1_{F(p)(x')})}
\\
&= (1_{1_{a}}, 1_{1_{F(p)(x')}}).
\end{aligned}
$$
La condition de cohérence ALC 7 résulte alors des égalités
$$
\begin{aligned}
K_{a} (\alpha_{((f,r), \sigma)}) &= K_{a} (\sigma, 1_{F(p')(r) (F(\sigma))_{x'}})
\\
&= (1_{1_{a}}, F(1_{a}) (1_{F(p')(r) (F(\sigma))_{x'}}) \CompDeuxZero (F(1_{p}))_{x'})
\\
&= (1_{1_{a}}, 1_{F(p')(r) (F(\sigma))_{x'}})
\\
&= 1_{(1_{a}, F(p')(r) (F(\sigma))_{x'})}
\\
&= 1_{K_{a} ((f,r), \sigma)}.
\end{aligned}
$$

Pour tout objet $(a,x)$ de $\Fibre{\left(\DeuxInt{\mathdeuxcat{A}}F\right)}{P_{F}}{a}$, 
$$
\begin{aligned}
\alpha_{J_{a} (a,x)} &= \alpha_{((a,x), 1_{a})}
\\
&= ((1_{a}, 1_{F(1_{a})(x)}), 1_{1_{a}})
\\
&= ((1_{a}, 1_{x}), 1_{1_{a}})
\\
&= 1_{((a,x), 1_{a})}
\\
&= 1_{J_{a} (a,x)}.
\end{aligned}
$$ 

On pose
$$
\begin{aligned}
\tau_{(a,x)} &= 1_{1_{J_{a} (a,x)}}
\\
&=(1_{1_{a}}, 1_{1_{x}}).
\end{aligned}
$$

La condition de cohérence ALC 8 résulte alors des égalités 
$$
\begin{aligned}
\alpha_{J_{a}(1_{a}, r)} &= \alpha_{((1_{a},r), 1_{1_{a}})}
\\
&= (1_{1_{a}}, 1_{F(1_{a})(r) (F(1_{1_{a}}))_{x}})
\\
&= (1_{1_{a}}, 1_{r})
\\
&= 1_{((1_{a},r), 1_{1_{a}})}
\\
&= 1_{J_{a}(1_{a},r)}. 
\end{aligned}
$$

La condition de cohérence ALC 9, elle, est trivialement vérifiée. 

Il ne reste plus qu'à vérifier la condition de cohérence ALC 10, qui résulte des égalités
$$
\begin{aligned}
\alpha_{\alpha_{((a',x'), p)}} &= \alpha_{((p, 1_{F(p)(x')}), 1_{p})}
\\
&= (1_{p}, 1_{F(1_{a}) (1_{F(p)(x')}) (F(1_{p}))_{x'}})
\\
&= (1_{p}, 1_{1_{F(p)(x')}})
\\
&=1_{((p, 1_{F(p)(x')}), 1_{p})}
\\
&= 1_{\alpha_{((a',x'), p)}}
\\
&= 1_{\alpha_{J_{a} K_{a} ((a',x'),p)} \alpha_{((a',x'), p)}}. 
\end{aligned}
$$
\end{proof}
\end{paragr}

\begin{paragr}\label{IntegrationDeuxFonctorielle}
La donnée de deux \DeuxFoncteursStricts{} $u$ et $v$ d'une \deux{}catégorie $\mathdeuxcat{A}$ vers $\DeuxCatDeuxCat$ et d'une \DeuxTransformationStricte{} $\sigma : u \Rightarrow v$ permet de définir un \DeuxFoncteurStrict{}
$$
\begin{aligned}
\DeuxInt{\mathdeuxcat{A}}{\sigma}\index[not]{0zA0Sigma@$\DeuxInt{\mathdeuxcat{A}}{\sigma}$} : \DeuxInt{\mathdeuxcat{A}}{u} &\to \DeuxInt{\mathdeuxcat{A}}{v}
\\
(a, x) &\mapsto (a, \sigma_{a}(x))
\\
(f, r) &\mapsto (f, \sigma_{a'}(r))
\\
(\gamma, \varphi) &\mapsto (\gamma, \sigma_{a'} (\varphi)).
\end{aligned}
$$
Les vérifications ne posent aucune difficulté ; c'est du reste encore un point que la notion de structure comma permet d'approcher de façon plus conceptuelle.

\end{paragr}

\begin{rem}\label{FonctorialiteIntegration}
Ainsi, la construction d'intégration $\DeuxInt{\mathdeuxcat{A}}$ est fonctorielle sur les \DeuxFoncteursStricts{} $\mathdeuxcat{A} \to \DeuxCatDeuxCat$ et les \DeuxTransformationsStrictes{} entre iceux. En fait, elle est même fonctorielle sur les \DeuxFoncteursCoLax{} $\mathdeuxcat{A} \to \DeuxCatDeuxCat$ et les \DeuxTransformationsCoLax{} entre iceux.  
\end{rem}

\begin{rem}
Soient $\mathdeuxcat{A}$ une \deux{}catégorie, $u$ et $v$ des \DeuxFoncteursStricts{} de $\mathdeuxcat{A}$ vers $\DeuxCatDeuxCat$ et $\sigma$ une \DeuxTransformationStricte{} de $u$ vers $v$. Alors, le diagramme de \DeuxFoncteursStricts{}
$$
\xymatrix{
\DeuxInt{\mathdeuxcat{A}}{u}
\ar[rr]^{\DeuxInt{\mathdeuxcat{A}}{\sigma}}
\ar[dr]_{P_{u}}
&&\DeuxInt{\mathdeuxcat{A}}{v}
\ar[dl]^{P_{v}}
\\
&\mathdeuxcat{A}
}
$$
est commutatif, $P_{u}$ et $P_{v}$ désignant les projections canoniques. On a donc notamment, pour tout objet $a$ de $\mathdeuxcat{A}$, un \DeuxFoncteurStrict{} 
$$
\DeuxFoncTrancheLax{\left(\DeuxInt{\mathdeuxcat{A}}{\sigma}\right)}{a} : \TrancheLax{\left(\DeuxInt{\mathdeuxcat{A}}{u}\right)}{P_{u}}{a} \to \TrancheLax{\left(\DeuxInt{\mathdeuxcat{A}}{v}\right)}{P_{v}}{a}.
$$
\end{rem}

\begin{lemme}\label{KikiDeMontparnasse}
Soient $\mathdeuxcat{A}$ une \deux{}catégorie, $u$ et $v$ des \DeuxFoncteursStricts{} de $\mathdeuxcat{A}$ vers $\DeuxCatDeuxCat$, $\sigma$ une \DeuxTransformationStricte{} de $u$ vers $v$ et $a$ un objet de $\mathdeuxcat{A}$. Il existe alors un diagramme commutatif de \DeuxFoncteursStricts{}
$$
\xymatrix{
u(a)
\ar[r]^{\sigma_{a}}
\ar[d]
&v(a)
\ar[d]
\\
\Fibre{(\DeuxInt{\mathdeuxcat{A}}{u})}{P_{u}}{a}
\ar[r]_{(\DeuxInt{\mathdeuxcat{A}}{\sigma})_{a}}
&\Fibre{(\DeuxInt{\mathdeuxcat{A}}{v})}{P_{v}}{a}
}
$$
dont les flèches verticales sont des isomorphismes.
\end{lemme}

\begin{proof}
Considérons les applications définies par
$$
\begin{aligned}
u(a) &\to \Fibre{\left(\DeuxInt{\mathdeuxcat{A}}{u}\right)}{P_{u}}{a}
\\
x &\mapsto (a,x)
\\
r &\mapsto (1_{a}, r)
\\
\varphi &\mapsto (1_{1_{a}}, \varphi)
\end{aligned}
$$
et
$$
\begin{aligned}
\Fibre{\left(\DeuxInt{\mathdeuxcat{A}}{u}\right)}{P_{u}}{a} &\to u(a)
\\
(a,x) &\mapsto x
\\
(1_{a}, r) &\mapsto r
\\
(1_{1_{a}}, \varphi) &\mapsto \varphi
\end{aligned}
$$
respectivement. Vérifions qu'elles définissent des \DeuxFoncteursStricts{}, que l'on notera $F$ et $G$ respectivement. Pour toute \un{}cellule $r$ de $x$ vers $x'$ dans $u(a)$, $(1_{a}, r)$ est bien une \un{}cellule de $(a,x)$ vers $(a, x')$ dans $\Fibre{\left(\DeuxInt{\mathdeuxcat{A}}{u}\right)}{P_{u}}{a}$ et, pour toute \deux{}cellule $\varphi$ de $r$ vers $s$ dans $u(a)$, $(1_{1_{a}}, \varphi)$ est bien une \deux{}cellule de $(1_{a}, r)$ vers $(1_{a}, s)$ dans $\Fibre{\left(\DeuxInt{\mathdeuxcat{A}}{u}\right)}{P_{u}}{a}$. De plus, pour toute \un{}cellule $r$ de $u(a)$, 
$$
\begin{aligned}
F(1_{r}) &= (1_{1_{a}}, 1_{r})
\\ 
&= 1_{(1_{a}, r)} 
\\
&= 1_{F(r)}.
\end{aligned}
$$

Étant donné $r$, $s$ et $t$ trois \un{}cellules de $x$ vers $x'$, $\gamma$ une \deux{}cellule de $r$ vers $s$ et $\delta$ une \deux{}cellule de $s$ vers $t$ dans $u(a)$, 
$$
\begin{aligned}
F(\delta \gamma) &= (1_{1_{a}}, \delta \gamma) 
\\
&= (1_{1_{a}}, (\delta \CompDeuxZero (u (1_{1_{a}}))_{x}) \gamma) 
\\
&= (1_{1_{a}}, \delta) (1_{1_{a}}, \gamma) 
\\
&= F(\delta) F(\gamma).
\end{aligned}
$$

Étant donné deux \un{}cellules $r$ et $r'$ de $u(a)$ telles que la composée $r'r$ fasse sens, 
$$
\begin{aligned}
F(r'r) &= (1_{a}, r'r) 
\\
&= (1_{a}, r' u(1_{a}) (r)) 
\\
&= (1_{a}, r') (1_{a}, r) 
\\
&= F(r') F(r).
\end{aligned}
$$

Étant donné $r$, $r'$, $s$ et $s'$ quatre \un{}cellules de $u(a)$ telles que les composées $r'r$ et $s's$ fassent sens, $\gamma$ une \deux{}cellule de $r$ vers $s$ et $\gamma'$ une \deux{}cellule de $r'$ vers $s'$, 
$$
\begin{aligned}
F(\gamma' \CompDeuxZero \gamma) &= (1_{1_{a}}, \gamma' \CompDeuxZero \gamma) 
\\
&= (1_{1_{a}}, \gamma' \CompDeuxZero u(1_{a})(\gamma)) 
\\
&= (1_{1_{a}}, \gamma') (1_{1_{a}}, \gamma) 
\\
&= F(\gamma') \CompDeuxZero F(\gamma).
\end{aligned}
$$

Pour tout objet $x$ de $u(a)$, 
$$
\begin{aligned}
F(1_{x}) &= (1_{a}, 1_{x}) 
\\
&= 1_{(a,x)} 
\\
&= 1_{F(x)}.
\end{aligned}
$$

Ainsi, $F$ est bien un \DeuxFoncteurStrict{}.

Pour toute \un{}cellule $(1_{a}, r)$ de $(a,x)$ vers $(a,x')$ dans $\Fibre{(\DeuxInt{\mathdeuxcat{A}}{u})}{P_{u}}{a}$, $r$ est une \un{}cellule de $x$ vers $x'$ dans $u(a)$. Pour toute \deux{}cellule $(1_{1_{a}}, \varphi)$ d'une \un{}cellule $(1_{a}, r)$ vers une \un{}cellule $(1_{a}, s)$ dans $\Fibre{(\DeuxInt{\mathdeuxcat{A}}{u})}{P_{u}}{a}$, $\varphi$ est bien une \deux{}cellule de $r$ vers $s$ dans $u(a)$. Pour toute \un{}cellule $(1_{a}, r)$ de $\Fibre{(\DeuxInt{\mathdeuxcat{A}}{u})}{P_{u}}{a}$,
$$
\begin{aligned}
G(1_{(1_{a}, r)}) &= G(1_{1_{a}}, 1_{r}) 
\\
&= 1_{r} 
\\
&= 1_{G(1_{a}, r)}.
\end{aligned}
$$

Étant donné $(1_{a}, r)$, $(1_{a}, s)$ et $(1_{a}, t)$ trois \un{}cellules de $(a,x)$ vers $(a,x')$, $(1_{1_{a}}, \varphi)$ une \deux{}cellule de $(1_{a}, r)$ vers $(1_{a}, s)$ et $(1_{1_{a}}, \psi)$ une \deux{}cellule de $(1_{a}, s)$ vers $(1_{a}, t)$,
$$
\begin{aligned}
G((1_{1_{a}}, \psi) (1_{1_{a}}, \varphi)) &= G(1_{1_{a}}, (\psi \CompDeuxZero (u (1_{1_{a}}))_{x}) \varphi) 
\\
&= G(1_{1_{a}}, \psi \varphi) 
\\
&= \psi \varphi 
\\
&= G(1_{1_{a}}, \psi) G(1_{1_{a}}, \varphi).
\end{aligned}
$$

Soient $(1_{a}, r)$ et $(1_{a}, r')$ deux \un{}cellules de $\Fibre{(\DeuxInt{\mathdeuxcat{A}}{u})}{P_{u}}{a}$ telles que la composée $(1_{a}, r') (1_{a}, r)$ fasse sens. Alors,
$$
\begin{aligned}
G((1_{a}, r') (1_{a}, r)) &= G(1_{a}, r' u(1_{a}) (r)) 
\\
&= G(1_{a}, r'r) 
\\
&= r'r 
\\
&= G(1_{a}, r') G(1_{a}, r).
\end{aligned}
$$

Soient $(1_{a}, r)$ et $(1_{a}, s)$ deux \un{}cellules de $(a,x)$ vers $(a,x')$, $(1_{a}, r')$ et $(1_{a}, s')$ deux \un{}cellules de $(a,x')$ vers $(a,x'')$, $(1_{1_{a}}, \varphi)$ une \deux{}cellule de $(1_{a}, r)$ vers $(1_{a}, s)$ et $(1_{1_{a}}, \varphi')$ une \deux{}cellule de $(1_{a}, r')$ vers $(1_{a}, s')$. Alors,
$$
\begin{aligned}
G((1_{1_{a}}, \varphi') \CompDeuxZero (1_{1_{a}}, \varphi)) &= G(1_{1_{a}}, \varphi' \CompDeuxZero u(1_{a})(\varphi)) = G(1_{1_{a}}, \varphi' \CompDeuxZero \varphi) 
\\
&= \varphi' \CompDeuxZero \varphi 
\\
&= G(1_{1_{a}}, \varphi') G(1_{1_{a}}, \varphi).
\end{aligned}
$$

Pour tout objet $(a,x)$ de $\Fibre{(\DeuxInt{\mathdeuxcat{A}}{u})}{P_{u}}{a}$,
$$
\begin{aligned}
G(1_{(a,x)}) &= G (1_{a}, 1_{x}) 
\\
&= 1_{x} 
\\
&= 1_{G(a,x)}.
\end{aligned}
$$

Ainsi, $G$ est bien un \DeuxFoncteurStrict{}. Il est clair que $F$ et $G$ sont des isomorphismes inverses l'un de l'autre. On définit de façon analogue un couple d'isomorphismes entre $v(a)$ et $\Fibre{(\DeuxInt{\mathdeuxcat{A}}{v})}{P_{v}}{a}$. La commutativité du diagramme obtenu est alors immédiate. 
\end{proof}

\begin{lemme}\label{PierreDac}
Soient $\mathdeuxcat{A}$ une \deux{}catégorie, $u$ et $v$ des \DeuxFoncteursStricts{} de $\mathdeuxcat{A}$ vers $\DeuxCatDeuxCat$ et $\sigma$ une \DeuxTransformationStricte{} de $u$ vers $v$. Alors, pour tout objet $a$ de $\mathdeuxcat{A}$, le diagramme de \DeuxFoncteursStricts{}\footnote{On a commis le léger abus de noter de la même façon les deux flèches verticales.}
$$
\xymatrix{
\Fibre{(\DeuxInt{\mathdeuxcat{A}}{u})}{P_{u}}{a}
\ar[rr]^{(\DeuxInt{\mathdeuxcat{A}}{\sigma})_{a}}
\ar[d]_{J_{a}}
&&\Fibre{(\DeuxInt{\mathdeuxcat{A}}{v})}{P_{v}}{a}
\ar[d]^{J_{a}}
\\
\TrancheLax{(\DeuxInt{\mathdeuxcat{A}}{u})}{P_{u}}{a}
\ar[rr]_{\DeuxFoncTrancheLax{(\DeuxInt{\mathdeuxcat{A}}{\sigma})}{a}}
&&\TrancheLax{(\DeuxInt{\mathdeuxcat{A}}{v})}{P_{v}}{a}
}
$$
est commutatif. 
\end{lemme}

\begin{proof}
Pour tout objet $(a,x)$ de $\Fibre{(\DeuxInt{\mathdeuxcat{A}}{u})}{P_{u}}{a}$, 
$$
\begin{aligned}
\left(\DeuxFoncTrancheLax{\left(\DeuxInt{\mathdeuxcat{A}}{\sigma}\right)}{a}\right) (J_{a} (a,x)) &= \left(\DeuxFoncTrancheLax{\left(\DeuxInt{\mathdeuxcat{A}}{\sigma}\right)}{a}\right) ((a, x), 1_{a})
\\
&= \left(\left(\DeuxInt{\mathdeuxcat{A}}{\sigma}\right) (a,x), 1_{a}\right)
\\
&= ((a, \sigma_{a} (x)), 1_{a})
\\
&= J_{a} (a, \sigma_{a} (x))
\\
&= J_{a} \left(\left(\DeuxInt{\mathdeuxcat{A}}{\sigma}\right)_{a} (a,x)\right).
\end{aligned}
$$

Pour toute \un{}cellule $(1_{a}, r)$ de $\Fibre{(\DeuxInt{\mathdeuxcat{A}}{u})}{P_{u}}{a}$, 
$$
\begin{aligned}
\left(\DeuxFoncTrancheLax{\left(\DeuxInt{\mathdeuxcat{A}}{\sigma}\right)}{a}\right) (J_{a} (1_{a}, r)) &= \left(\DeuxFoncTrancheLax{\left(\DeuxInt{\mathdeuxcat{A}}{\sigma}\right)}{a}\right) ((1_{a}, r), 1_{1_{a}})
\\
&= \left(\left(\DeuxInt{\mathdeuxcat{A}}{\sigma}\right) (1_{a}, r), 1_{1_{a}}\right)
\\
&= ((1_{a}, \sigma_{a} (r)), 1_{1_{a}})
\\
&= J_{a} (1_{a}, \sigma_{a} (r))
\\
&= J_{a} \left(\left(\DeuxInt{\mathdeuxcat{A}}{\sigma}\right)_{a} (1_{a}, r)\right).
\end{aligned}
$$

Pour toute \deux{}cellule $(1_{1_{a}}, \varphi)$ de $\Fibre{(\DeuxInt{\mathdeuxcat{A}}{u})}{P_{u}}{a}$, 
$$
\begin{aligned}
\left(\DeuxFoncTrancheLax{\left(\DeuxInt{\mathdeuxcat{A}}{\sigma}\right)}{a}\right) (J_{a} (1_{1_{a}}, \varphi)) &= \left(\DeuxFoncTrancheLax{\left(\DeuxInt{\mathdeuxcat{A}}{\sigma}\right)}{a}\right) (1_{1_{a}}, \varphi)
\\
&= \left(\DeuxInt{\mathdeuxcat{A}}{\sigma}\right) (1_{1_{a}}, \varphi)
\\
&= (1_{1_{a}}, \sigma_{a} (\varphi))
\\
&= J_{a} (1_{1_{a}}, \sigma_{a} (\varphi))
\\
&= J_{a} \left(\left(\DeuxInt{\mathdeuxcat{A}}{\sigma}\right)_{a} (1_{1_{a}}, \varphi)\right).
\end{aligned}
$$
Le lemme est donc vérifié.
\end{proof}

Au début de cette section, nous avons associé une \deux{}catégorie précoopfibrée sur $\mathdeuxcat{A}$ à tout \DeuxFoncteurStrict{} $\mathdeuxcat{A} \to \DeuxCatDeuxCat$. Nous introduisons maintenant quelques constructions duales. 

\begin{paragr}
Toujours sous l'hypothèse de la donnée d'un \DeuxFoncteurStrict{} $\mathdeuxcat{A} \to \DeuxCatDeuxCat$, on définit une \deux{}catégorie $\DeuxIntUnOp{\mathdeuxcat{A}} F$\index[not]{0zAF1@$\DeuxIntUnOp{\mathdeuxcat{A}} F$} par la formule
$$
\DeuxIntUnOp{\mathdeuxcat{A}} F = \DeuxCatDeuxOp{\left(\DeuxInt{\DeuxCatDeuxOp{\mathdeuxcat{A}}} \DeuxFoncDeuxOp{(?^{coop} \circ F)}\right)}.
$$
Plus concrètement, cette \deux{}catégorie se décrit comme suit. Les objets sont les couples $(a,x)$, avec $a$ un objet de $\mathdeuxcat{A}$ et $x$ un objet de $F(a)$. Les \un{}cellules de $(a,x)$ vers $(a',x')$ sont les couples $(f : a \to a', r : x' \to F(f)(x))$, avec $f$ une \un{}cellule de $\mathdeuxcat{A}$ et $r$ une \un{}cellule de $F(a')$. Les \deux{}cellules de $(f,r)$ vers $(g,s)$ sont les couples $(\gamma : f \Rightarrow g, \varphi : (F(\gamma))_{x} r \Rightarrow s)$, avec $\gamma$ une \deux{}cellule de $\mathdeuxcat{A}$ et $\varphi$ une \deux{}cellule de $F(a')$. Les diverses unités et compositions sont définies de façon « évidente ». La projection canonique $\DeuxInt{\DeuxCatDeuxOp{\mathdeuxcat{A}}} \DeuxFoncDeuxOp{(?^{coop} \circ F)} \to \DeuxCatDeuxOp{\mathdeuxcat{A}}$ étant une précoopfibration, la projection canonique $\DeuxIntUnOp{\mathdeuxcat{A}} F \to \mathdeuxcat{A}$ est une préopfibration, dont la fibre au-dessus de $a \in \Objets{\mathdeuxcat{A}}$ s'identifie à la \deux{}catégorie $\DeuxCatUnOp{(F(a))}$. Cette construction est en fait fonctorielle sur les \DeuxFoncteursLax{} $\mathdeuxcat{A} \to \DeuxCatDeuxCat$ et les \DeuxTransformationsLax{} entre iceux.
\end{paragr}

\begin{paragr}
Toujours sous l'hypothèse de la donnée d'un \DeuxFoncteurStrict{} $\mathdeuxcat{A} \to \DeuxCatDeuxCat$, on définit une \deux{}catégorie $\DeuxIntDeuxOp{\mathdeuxcat{A}} F$\index[not]{0zAF2@$\DeuxIntDeuxOp{\mathdeuxcat{A}} F$} par la formule
$$
\DeuxIntDeuxOp{\mathdeuxcat{A}} F = \DeuxInt{\mathdeuxcat{A}} (?^{co} \circ F).
$$
Plus concrètement, cette \deux{}catégorie se décrit comme suit. Les objets sont les couples $(a,x)$, avec $a$ un objet de $\mathdeuxcat{A}$ et $x$ un  objet de $F(a)$. Les \un{}cellules de $(a,x)$ vers $(a',x')$ sont les couples $(f : a \to a', r : F(f)(x) \to x')$, avec $f$ une \un{}cellule de $\mathdeuxcat{A}$ et $r$ une \un{}cellule de $F(a')$. Les \deux{}cellules de $(f,r)$ vers $(g,s)$ sont les couples $(\gamma : f \Rightarrow g, \varphi : s (F(\gamma))_{x} \Rightarrow r)$, avec $\gamma$ une \deux{}cellule de $\mathdeuxcat{A}$ et $\varphi$ une \deux{}cellule de $F(a')$. Les diverses unités et compositions sont définies de façon « évidente ». La projection canonique $\DeuxIntDeuxOp{\mathdeuxcat{A}} F \to \mathdeuxcat{A}$ est une précoopfibration, dont la fibre au-dessus de $a \in \Objets{\mathdeuxcat{A}}$ s'identifie à la \deux{}catégorie $\DeuxCatDeuxOp{(F(a))}$. Cette construction est en fait fonctorielle sur les \DeuxFoncteursCoLax{} $\mathdeuxcat{A} \to \DeuxCatDeuxCat$ et les \DeuxTransformationsCoLax{} entre iceux.
\end{paragr}

\begin{paragr}
Toujours sous l'hypothèse de la donnée d'un \DeuxFoncteurStrict{} $\mathdeuxcat{A} \to \DeuxCatDeuxCat$, on définit une \deux{}catégorie $\DeuxIntToutOp{\mathdeuxcat{A}} F$\index[not]{0zAF12@$\DeuxIntToutOp{\mathdeuxcat{A}} F$} par la formule
$$
\DeuxIntToutOp{\mathdeuxcat{A}} F = \DeuxCatDeuxOp{\left(\DeuxInt{\DeuxCatDeuxOp{\mathdeuxcat{A}}} \DeuxFoncDeuxOp{(?^{op} \circ F)}\right)}.
$$
Plus concrètement, cette \deux{}catégorie se décrit comme suit. Les objets sont les couples $(a,x)$, avec $a$ un objet de $\mathdeuxcat{A}$ et $x$ un objet de $F(a)$. Les \un{}cellules de $(a,x)$ vers $(a',x')$ sont les couples $(f : a \to a', r : x' \to F(f)(x))$, avec $f$ une \un{}cellule de $\mathdeuxcat{A}$ et $r$ une \un{}cellule de $F(a')$. Les \deux{}cellules de $(f,r)$ vers $(g,s)$ sont les couples $(\gamma : f \Rightarrow g, \varphi : s \Rightarrow (F(\gamma))_{x} r)$, avec $\gamma$ une \deux{}cellule de $\mathdeuxcat{A}$ et $\varphi$ une \deux{}cellule de $F(a')$. Les diverses unités et compositions sont définies de façon « évidente ». La projection canonique $\DeuxInt{\DeuxCatDeuxOp{\mathdeuxcat{A}}} \DeuxFoncDeuxOp{(?^{op} \circ F)} \to \DeuxCatDeuxOp{\mathdeuxcat{A}}$ étant une précoopfibration, la projection canonique $\DeuxIntToutOp{\mathdeuxcat{A}} F \to \mathdeuxcat{A}$ est une préopfibration, dont la fibre au-dessus de $a \in \Objets{\mathdeuxcat{A}}$ s'identifie à la \deux{}catégorie $\DeuxCatToutOp{(F(a))}$. Cette construction est en fait fonctorielle sur les \DeuxFoncteursLax{} $\mathdeuxcat{A} \to \DeuxCatDeuxCat$ et les \DeuxTransformationsLax{} entre iceux.
\end{paragr}

\begin{paragr}\label{DefDeuxIntOp}
Pour tout \DeuxFoncteurStrict{} $F : \DeuxCatUnOp{\mathdeuxcat{A}} \to \DeuxCatDeuxCat$, on définit une \deux{}catégorie $\DeuxIntOp{\mathdeuxcat{A}} F$\index[not]{0zAFOp@$\DeuxIntOp{\mathdeuxcat{A}} F$} par la formule
$$
\DeuxIntOp{\mathdeuxcat{A}} F = \DeuxCatToutOp{\left( \DeuxInt{\DeuxCatToutOp{\mathdeuxcat{A}}} \DeuxFoncDeuxOp{(?^{coop} \circ F)} \right)}.
$$
Plus concrètement, cette \deux{}catégorie se décrit comme suit. Les objets sont les couples $(a,x)$, avec $a$ un objet de $\mathdeuxcat{A}$ et $x$ un objet de $F(a)$. Les \un{}cellules de $(a,x)$ vers $(a',x')$ sont les couples $(f : a \to a', r : x \to F(f)(x'))$, avec $f$ une \un{}cellule de $\mathdeuxcat{A}$ et $r$ une \un{}cellule de $F(a)$. Les \deux{}cellules de $(f,r)$ vers $(g,s)$ sont les couples $(\gamma : f \Rightarrow g, \varphi : (F(\gamma))_{x'} r \Rightarrow s)$, avec $\gamma$ une \deux{}cellule de $\mathdeuxcat{A}$ et $\varphi$ une \deux{}cellule de $F(a)$. Les diverses unités et compositions sont définies de façon « évidente ». La projection canonique $\DeuxInt{\DeuxCatToutOp{\mathdeuxcat{A}}} \DeuxFoncDeuxOp{(?^{coop} \circ F)} \to \DeuxCatToutOp{\mathdeuxcat{A}}$ étant une précoopfibration, la projection canonique $\DeuxIntOp{\mathdeuxcat{A}} F \to \mathdeuxcat{A}$ est une préfibration, dont la fibre au-dessus de $a \in \Objets{\mathdeuxcat{A}}$ s'identifie à la \deux{}catégorie $F(a)$. Cette construction est en fait fonctorielle sur les \DeuxFoncteursLax{} $\mathdeuxcat{A} \to \DeuxCatDeuxCat$ et les \DeuxTransformationsLax{} entre iceux.
\end{paragr}

\begin{paragr}
Toujours sous l'hypothèse de la donnée d'un \DeuxFoncteurStrict{} $F : \DeuxCatUnOp{\mathdeuxcat{A}} \to \DeuxCatDeuxCat$, on définit une \deux{}catégorie $\DeuxIntOpUnOp{\mathdeuxcat{A}} F$\index[not]{0zAFOp1@$\DeuxIntOpUnOp{\mathdeuxcat{A}} F$} par la formule
$$
\DeuxIntOpUnOp{\mathdeuxcat{A}} F = \DeuxCatUnOp{\left(\DeuxInt{\DeuxCatUnOp{\mathdeuxcat{A}}} F\right)}.
$$
Plus concrètement, cette \deux{}catégorie se décrit comme suit. Les objets sont les couples $(a,x)$, avec $a$ un objet de $\mathdeuxcat{A}$ et $x$ un objet de $F(a)$. Les \un{}cellules de $(a,x)$ vers $(a',x')$ sont les couples $(f : a \to a', r : F(f)(x') \to x)$, avec $f$ une \un{}cellule de $\mathdeuxcat{A}$ et $r$ une \un{}cellule de $F(a)$. Les \deux{}cellules de $(f,r)$ vers $(g,s)$ sont les couples $(\gamma : f \Rightarrow g, \varphi : r \Rightarrow s (F(\gamma))_{x'})$, avec $\gamma$ une \deux{}cellule de $\mathdeuxcat{A}$ et $\varphi$ une \deux{}cellule de $F(a)$. Les diverses unités et compositions sont définies de façon « évidente ». La projection canonique $\DeuxInt{\DeuxCatUnOp{\mathdeuxcat{A}}} F \to \DeuxCatUnOp{\mathdeuxcat{A}}$ étant une précoopfibration, la projection canonique $\DeuxIntOpUnOp{\mathdeuxcat{A}} F \to \mathdeuxcat{A}$ est une précofibration, dont la fibre au-dessus de $a \in \Objets{\mathdeuxcat{A}}$ s'identifie à la \deux{}catégorie $\DeuxCatUnOp{(F(a))}$. Cette construction est en fait fonctorielle sur les \DeuxFoncteursCoLax{} $\mathdeuxcat{A} \to \DeuxCatDeuxCat$ et les \DeuxTransformationsCoLax{} entre iceux.
\end{paragr}

\begin{paragr}
Toujours sous l'hypothèse de la donnée d'un \DeuxFoncteurStrict{} $F : \DeuxCatUnOp{\mathdeuxcat{A}} \to \DeuxCatDeuxCat$, on définit une \deux{}catégorie $\DeuxIntOpDeuxOp{\mathdeuxcat{A}} F$\index[not]{0zAFOp2@$\DeuxIntOpDeuxOp{\mathdeuxcat{A}} F$} par la formule
$$
\DeuxIntOpDeuxOp{\mathdeuxcat{A}} F = \DeuxCatToutOp{\left( \DeuxInt{\DeuxCatToutOp{\mathdeuxcat{A}}} \DeuxFoncDeuxOp{(?^{op} \circ F)} \right)}.
$$
Plus concrètement, cette \deux{}catégorie se décrit comme suit. Les objets sont les couples $(a,x)$, avec $a$ un objet de $\mathdeuxcat{A}$ et $x$ un objet de $F(a)$. Les \un{}cellules de $(a,x)$ vers $(a',x')$ sont les couples $(f : a \to a', r : x \to F(f)(x'))$, avec $f$ une \un{}cellule de $\mathdeuxcat{A}$ et $r$ une \un{}cellule de $F(a)$. Les \deux{}cellules de $(f,r)$ vers $(g,s)$ sont les couples $(\gamma : f \Rightarrow g, \varphi : s \Rightarrow (F(\gamma))_{x'} r)$, avec $\gamma$ une \deux{}cellule de $\mathdeuxcat{A}$ et $\varphi$ une \deux{}cellule de $F(a)$. Les diverses unités et compositions sont définies de façon « évidente ». La projection canonique $\DeuxInt{\DeuxCatToutOp{\mathdeuxcat{A}}} \DeuxFoncDeuxOp{(?^{op} \circ F)} \to \DeuxCatToutOp{\mathdeuxcat{A}}$ étant une précoopfibration, la projection canonique $\DeuxIntOpDeuxOp{\mathdeuxcat{A}} F \to \mathdeuxcat{A}$ est une préfibration, dont la fibre au-dessus de $a \in \Objets{\mathdeuxcat{A}}$ s'identifie à la \deux{}catégorie $\DeuxCatDeuxOp{(F(a))}$. Cette construction est en fait fonctorielle sur les \DeuxFoncteursLax{} $\mathdeuxcat{A} \to \DeuxCatDeuxCat$ et les \DeuxTransformationsLax{} entre iceux.
\end{paragr}

\begin{paragr}
Toujours sous l'hypothèse de la donnée d'un \DeuxFoncteurStrict{} $F : \DeuxCatUnOp{\mathdeuxcat{A}} \to \DeuxCatDeuxCat$, on définit une \deux{}catégorie $\DeuxIntOpToutOp{\mathdeuxcat{A}} F$\index[not]{0zAFOp12@$\DeuxIntOpToutOp{\mathdeuxcat{A}} F$} par la formule
$$
\DeuxIntOpToutOp{\mathdeuxcat{A}} F = 
\DeuxCatUnOp{\left(\DeuxInt{\DeuxCatUnOp{\mathdeuxcat{A}}} (?^{co} \circ F)\right)}.
$$
Plus concrètement, cette \deux{}catégorie se décrit comme suit. Les objets sont les couples $(a,x)$, avec $a$ un objet de $\mathdeuxcat{A}$ et $x$ un objet de $F(a)$. Les \un{}cellules de $(a,x)$ vers $(a',x')$ sont les couples $(f : a \to a', r : F(f)(x') \to x)$, avec $f$ une \un{}cellule de $\mathdeuxcat{A}$ et $r$ une \un{}cellule de $F(a)$. Les \deux{}cellules de $(f,r)$ vers $(g,s)$ sont les couples $(\gamma : f \Rightarrow g, \varphi : s (F(\gamma))_{x'} \Rightarrow r)$, avec $\gamma$ une \deux{}cellule de $\mathdeuxcat{A}$ et $\varphi$ une \deux{}cellule de $F(a)$. Les diverses unités et compositions sont définies de façon « évidente ». La projection canonique $\DeuxInt{\DeuxCatUnOp{\mathdeuxcat{A}}} (?^{co} \circ F) \to \DeuxCatUnOp{\mathdeuxcat{A}}$ étant une précoopfibration, la projection canonique $\DeuxIntOpToutOp{\mathdeuxcat{A}} F \to \mathdeuxcat{A}$ est une précofibration, dont la fibre au-dessus de $a \in \Objets{\mathdeuxcat{A}}$ s'identifie à la \deux{}catégorie $\DeuxCatToutOp{(F(a))}$. Cette construction est en fait fonctorielle sur les \DeuxFoncteursCoLax{} $\mathdeuxcat{A} \to \DeuxCatDeuxCat$ et les \DeuxTransformationsCoLax{} entre iceux.
\end{paragr}

\begin{paragr}\label{DefDeuxIntCo}
Pour tout \DeuxFoncteurStrict{} $F : \DeuxCatDeuxOp{\mathdeuxcat{A}} \to \DeuxCatDeuxCat$, on définit une \deux{}catégorie $\DeuxIntCo{\mathdeuxcat{A}} F$\index[not]{0zAFCo@$\DeuxIntCo{\mathdeuxcat{A}} F$} par la formule
$$
\DeuxIntCo{\mathdeuxcat{A}} F = \DeuxCatDeuxOp{\left(\DeuxInt{\DeuxCatDeuxOp{\mathdeuxcat{A}}} \left(\DeuxFoncDeuxOp{?} \circ F\right)\right)}.
$$ 
Plus concrètement, cette \deux{}catégorie se décrit comme suit. Les objets sont les couples $(a,x)$, avec $a$ un objet de $\mathdeuxcat{A}$ et $x$ un objet de $F(a)$. Les \un{}cellules de $(a,x)$ vers $(a', x')$ sont les couples $(f : a \to a', r : F(f)(x) \to x')$, $f$ et $r$ étant des \un{}cellules de $\mathdeuxcat{A}$ et $F(a')$ respectivement. Les \deux{}cellules de $(f,r)$ vers $(g,s)$ sont les couples $(\gamma : f \Rightarrow g, \varphi : r (F(\gamma))_{x} \Rightarrow s)$, $\gamma$ et $\varphi$ étant des \deux{}cellules dans $\mathdeuxcat{A}$ et $F(a')$ respectivement. Les diverses unités et compositions sont définies de façon « évidente ». La projection canonique $\DeuxInt{\DeuxCatDeuxOp{\mathdeuxcat{A}}} (\DeuxFoncDeuxOp{?} \circ F) \to \DeuxCatDeuxOp{\mathdeuxcat{A}}$ étant une précoopfibration, la projection canonique $\DeuxIntCo{\mathdeuxcat{A}} F \to \mathdeuxcat{A}$ est une préopfibration, dont la fibre au-dessus de $a \in \Objets{\mathdeuxcat{A}}$ s'identifie à la \deux{}catégorie $F(a)$. Cette construction est en fait fonctorielle sur les \DeuxFoncteursCoLax{} $\mathdeuxcat{A} \to \DeuxCatDeuxCat$ et les \DeuxTransformationsCoLax{} entre iceux.
\end{paragr}

\begin{paragr}
Toujours sous l'hypothèse de la donnée d'un \DeuxFoncteurStrict{} $F : \DeuxCatDeuxOp{\mathdeuxcat{A}} \to \DeuxCatDeuxCat$, on définit une \deux{}catégorie $\DeuxIntCoUnOp{\mathdeuxcat{A}} F$\index[not]{0zAFCo1@$\DeuxIntCoUnOp{\mathdeuxcat{A}} F$} par la formule
$$
\DeuxIntCoUnOp{\mathdeuxcat{A}} F = 
\DeuxInt{\mathdeuxcat{A}} \DeuxFoncDeuxOp{(?^{op} \circ F)}.
$$
Plus concrètement, cette \deux{}catégorie se décrit comme suit. Les objets sont les couples $(a,x)$, avec $a$ un objet de $\mathdeuxcat{A}$ et $x$ un objet de $F(a)$. Les \un{}cellules de $(a,x)$ vers $(a', x')$ sont les couples $(f : a \to a', r : x' \to F(f)(x))$, $f$ et $r$ étant des \un{}cellules de $\mathdeuxcat{A}$ et $F(a')$ respectivement. Les \deux{}cellules de $(f,r)$ vers $(g,s)$ sont les couples $(\gamma : f \Rightarrow g, \varphi : (F(\gamma))_{x} r \Rightarrow s)$, $\gamma$ et $\varphi$ étant des \deux{}cellules dans $\mathdeuxcat{A}$ et $F(a')$ respectivement. Les diverses unités et compositions sont définies de façon « évidente ». La projection canonique $\DeuxIntCoUnOp{\mathdeuxcat{A}} F \to \mathdeuxcat{A}$ est une précoopfibration, dont la fibre au-dessus de $a \in \Objets{\mathdeuxcat{A}}$ s'identifie à la \deux{}catégorie $\DeuxCatUnOp{(F(a))}$. Cette construction est en fait fonctorielle sur les \DeuxFoncteursLax{} $\mathdeuxcat{A} \to \DeuxCatDeuxCat$ et les \DeuxTransformationsLax{} entre iceux.
\end{paragr}

\begin{paragr}
Toujours sous l'hypothèse de la donnée d'un \DeuxFoncteurStrict{} $F : \DeuxCatDeuxOp{\mathdeuxcat{A}} \to \DeuxCatDeuxCat$, on définit une \deux{}catégorie $\DeuxIntCoDeuxOp{\mathdeuxcat{A}} F$\index[not]{0zAFCo2@$\DeuxIntCoDeuxOp{\mathdeuxcat{A}} F$} par la formule
$$
\DeuxIntCoDeuxOp{\mathdeuxcat{A}} F = \DeuxCatDeuxOp{\left(\DeuxInt{\DeuxCatDeuxOp{\mathdeuxcat{A}}} F\right)}.
$$
Plus concrètement, cette \deux{}catégorie se décrit comme suit. Les objets sont les couples $(a,x)$, avec $a$ un objet de $\mathdeuxcat{A}$ et $x$ un objet de $F(a)$. Les \un{}cellules de $(a,x)$ vers $(a', x')$ sont les couples $(f : a \to a', r : F(f)(x) \to x')$, $f$ et $r$ étant des \un{}cellules de $\mathdeuxcat{A}$ et $F(a')$ respectivement. Les \deux{}cellules de $(f,r)$ vers $(g,s)$ sont les couples $(\gamma : f \Rightarrow g, \varphi : s \Rightarrow r (F(\gamma))_{x})$, $\gamma$ et $\varphi$ étant des \deux{}cellules dans $\mathdeuxcat{A}$ et $F(a')$ respectivement. Les diverses unités et compositions sont définies de façon « évidente ». La projection canonique $\DeuxInt{\DeuxCatDeuxOp{\mathdeuxcat{A}}} F \to \DeuxCatDeuxOp{\mathdeuxcat{A}}$ étant une précoopfibration, la projection canonique $\DeuxIntCoDeuxOp{\mathdeuxcat{A}} F \to \mathdeuxcat{A}$ est une préopfibration, dont la fibre au-dessus de $a \in \Objets{\mathdeuxcat{A}}$ s'identifie à la \deux{}catégorie $\DeuxCatDeuxOp{(F(a))}$. Cette construction est en fait fonctorielle sur les \DeuxFoncteursCoLax{} $\mathdeuxcat{A} \to \DeuxCatDeuxCat$ et les \DeuxTransformationsCoLax{} entre iceux.
\end{paragr}

\begin{paragr}
Toujours sous l'hypothèse de la donnée d'un \DeuxFoncteurStrict{} $F : \DeuxCatDeuxOp{\mathdeuxcat{A}} \to \DeuxCatDeuxCat$, on définit une \deux{}catégorie $\DeuxIntCoToutOp{\mathdeuxcat{A}} F$\index[not]{0zAFCo12@$\DeuxIntCoToutOp{\mathdeuxcat{A}} F$} par la formule
$$
\DeuxIntCoToutOp{\mathdeuxcat{A}} F = 
\DeuxInt{\mathdeuxcat{A}}\DeuxFoncDeuxOp{(?^{coop} \circ F)}.
$$
Plus concrètement, cette \deux{}catégorie se décrit comme suit. Les objets sont les couples $(a,x)$, avec $a$ un  objet de $\mathdeuxcat{A}$ et $x$ un objet de $F(a)$. Les \un{}cellules de $(a,x)$ vers $(a', x')$ sont les couples $(f : a \to a', r : x' \to F(f)(x))$, $f$ et $r$ étant des \un{}cellules de $\mathdeuxcat{A}$ et $F(a')$ respectivement. Les \deux{}cellules de $(f,r)$ vers $(g,s)$ sont les couples $(\gamma : f \Rightarrow g, \varphi : s \Rightarrow (F(\gamma))_{x} r)$, $\gamma$ et $\varphi$ étant des \deux{}cellules dans $\mathdeuxcat{A}$ et $F(a')$ respectivement. Les diverses unités et compositions sont définies de façon « évidente ». La projection canonique $\DeuxIntCoToutOp{\mathdeuxcat{A}} F \to \mathdeuxcat{A}$ est une précoopfibration, dont la fibre au-dessus de $a \in \Objets{\mathdeuxcat{A}}$ s'identifie à la \deux{}catégorie $\DeuxCatToutOp{(F(a))}$. Cette construction est en fait fonctorielle sur les \DeuxFoncteursLax{} $\mathdeuxcat{A} \to \DeuxCatDeuxCat$ et les \DeuxTransformationsLax{} entre iceux.
\end{paragr}

\begin{paragr}
Pour tout \DeuxFoncteurStrict{} $F : \DeuxCatToutOp{\mathdeuxcat{A}} \to \DeuxCatDeuxCat$, on définit une \deux{}catégorie $\DeuxIntCoOp{\mathdeuxcat{A}} F$\index[not]{0zAFCoOp@$\DeuxIntCoOp{\mathdeuxcat{A}} F$} par la formule
$$
\DeuxIntCoOp{\mathdeuxcat{A}} F = \DeuxCatUnOp{\left(\DeuxInt{\DeuxCatUnOp{\mathdeuxcat{A}}} \DeuxFoncDeuxOp{(?^{op} \circ F)} \right)}.
$$
Plus concrètement, cette \deux{}catégorie se décrit comme suit. Les objets sont les couples $(a,x)$, avec $a$ un objet de $\mathdeuxcat{A}$ et $x$ un objet de $F(a)$. Les \un{}cellules de $(a,x)$ vers $(a', x')$ sont les couples $(f : a \to a', r : x \to F(f)(x'))$, $f$ et $r$ étant des \un{}cellules de $\mathdeuxcat{A}$ et $F(a)$ respectivement. Les \deux{}cellules de $(f,r)$ vers $(g,s)$ sont les couples $(\gamma : f \Rightarrow g, \varphi : r \Rightarrow (F(\gamma))_{x'} s)$, $\gamma$ et $\varphi$ étant des \deux{}cellules dans $\mathdeuxcat{A}$ et $F(a)$ respectivement. Les diverses unités et compositions sont définies de façon « évidente ». La projection canonique $\DeuxInt{\DeuxCatUnOp{\mathdeuxcat{A}}} \DeuxFoncDeuxOp{(?^{op} \circ F)} \to \DeuxCatUnOp{\mathdeuxcat{A}}$ étant une précoopfibration, la projection canonique $\DeuxIntCoOp{\mathdeuxcat{A}} F \to \mathdeuxcat{A}$ est une précofibration, dont la fibre au-dessus de $a \in \Objets{\mathdeuxcat{A}}$ s'identifie à la \deux{}catégorie $F(a)$. Cette construction est en fait fonctorielle sur les \DeuxFoncteursLax{} $\mathdeuxcat{A} \to \DeuxCatDeuxCat$ et les \DeuxTransformationsLax{} entre iceux.
\end{paragr}

\begin{paragr}
Toujours sous l'hypothèse de la donnée d'un \DeuxFoncteurStrict{} $F : \DeuxCatToutOp{\mathdeuxcat{A}} \to \DeuxCatDeuxCat$, on définit une \deux{}catégorie $\DeuxIntCoOpUnOp{\mathdeuxcat{A}} F$\index[not]{0zAFCoOp1@$\DeuxIntCoOpUnOp{\mathdeuxcat{A}} F$} par la formule
$$
\DeuxIntCoOpUnOp{\mathdeuxcat{A}} F = 
\DeuxCatToutOp{\left(\DeuxInt{\DeuxCatToutOp{\mathdeuxcat{A}}} (?^{co} \circ F)\right)}.
$$
Plus concrètement, cette \deux{}catégorie se décrit comme suit. Les objets sont les couples $(a,x)$, avec $a$ un objet de $\mathdeuxcat{A}$ et $x$ un objet de $F(a)$. Les \un{}cellules de $(a,x)$ vers $(a', x')$ sont les couples $(f : a \to a', r : F(f)(x') \to x)$, $f$ et $r$ étant des \un{}cellules de $\mathdeuxcat{A}$ et $F(a)$ respectivement. Les \deux{}cellules de $(f,r)$ vers $(g,s)$ sont les couples $(\gamma : f \Rightarrow g, \varphi : r (F(\gamma))_{x'} \Rightarrow s)$, $\gamma$ et $\varphi$ étant des \deux{}cellules dans $\mathdeuxcat{A}$ et $F(a)$ respectivement. Les diverses unités et compositions sont définies de façon « évidente ». La projection canonique $\DeuxInt{\DeuxCatToutOp{\mathdeuxcat{A}}} (?^{co} \circ F) \to \DeuxCatToutOp{\mathdeuxcat{A}}$ étant une précoopfibration, la projection canonique $\DeuxIntCoOpUnOp{\mathdeuxcat{A}} F \to \mathdeuxcat{A}$ est une préfibration, dont la fibre au-dessus de $a \in \Objets{\mathdeuxcat{A}}$ s'identifie à la \deux{}catégorie $\DeuxCatUnOp{(F(a))}$. Cette construction est en fait fonctorielle sur les \DeuxFoncteursCoLax{} $\mathdeuxcat{A} \to \DeuxCatDeuxCat$ et les \DeuxTransformationsCoLax{} entre iceux.
\end{paragr}

\begin{paragr}
Toujours sous l'hypothèse de la donnée d'un \DeuxFoncteurStrict{} $F : \DeuxCatToutOp{\mathdeuxcat{A}} \to \DeuxCatDeuxCat$, on définit une \deux{}catégorie $\DeuxIntCoOpDeuxOp{\mathdeuxcat{A}} F$\index[not]{0zAFCoOp2@$\DeuxIntCoOpDeuxOp{\mathdeuxcat{A}} F$} par la formule
$$
\DeuxIntCoOpDeuxOp{\mathdeuxcat{A}} F = \DeuxCatUnOp{\left(\DeuxInt{\DeuxCatUnOp{\mathdeuxcat{A}}} \DeuxFoncDeuxOp{(?^{coop} \circ F)} \right)}.
$$
Plus concrètement, cette \deux{}catégorie se décrit comme suit. Les objets sont les couples $(a,x)$, avec $a$ un objet de $\mathdeuxcat{A}$ et $x$ un objet de $F(a)$. Les \un{}cellules de $(a,x)$ vers $(a', x')$ sont les couples $(f : a \to a', r : x \to F(f)(x'))$, $f$ et $r$ étant des \un{}cellules de $\mathdeuxcat{A}$ et $F(a)$ respectivement. Les \deux{}cellules de $(f,r)$ vers $(g,s)$ sont les couples $(\gamma : f \Rightarrow g, \varphi : (F(\gamma))_{x'} s \Rightarrow r)$, $\gamma$ et $\varphi$ étant des \deux{}cellules dans $\mathdeuxcat{A}$ et $F(a)$ respectivement. Les diverses unités et compositions sont définies de façon « évidente ». La projection canonique $\DeuxInt{\DeuxCatUnOp{\mathdeuxcat{A}}} \DeuxFoncDeuxOp{(?^{coop} \circ F)} \to \DeuxCatUnOp{\mathdeuxcat{A}}$ étant une précoopfibration, la projection canonique $\DeuxIntCoOpDeuxOp{\mathdeuxcat{A}} F \to \mathdeuxcat{A}$ est une précofibration, dont la fibre au-dessus de \mbox{$a \in \Objets{\mathdeuxcat{A}}$} s'identifie à la \deux{}catégorie $\DeuxCatDeuxOp{(F(a))}$. Cette construction est en fait fonctorielle sur les \DeuxFoncteursLax{} $\mathdeuxcat{A} \to \DeuxCatDeuxCat$ et les \DeuxTransformationsLax{} entre iceux.
\end{paragr}

\begin{paragr}
Toujours sous l'hypothèse de la donnée d'un \DeuxFoncteurStrict{} $F : \DeuxCatToutOp{\mathdeuxcat{A}} \to \DeuxCatDeuxCat$, on définit une \deux{}catégorie $\DeuxIntCoOpToutOp{\mathdeuxcat{A}} F$\index[not]{0zAFCoOp12@$\DeuxIntCoOpToutOp{\mathdeuxcat{A}} F$} par la formule
$$
\DeuxIntCoOpToutOp{\mathdeuxcat{A}} F = 
\DeuxCatToutOp{\left(\DeuxInt{\DeuxCatToutOp{\mathdeuxcat{A}}} F\right)}.
$$
Plus concrètement, cette \deux{}catégorie se décrit comme suit. Les objets sont les couples $(a,x)$, avec $a$ un objet de $\mathdeuxcat{A}$ et $x$ un objet de $F(\mathdeuxcat{A})$. Les \un{}cellules de $(a,x)$ vers $(a', x')$ sont les couples $(f : a \to a', r : F(f)(x') \to x)$, $f$ et $r$ étant des \un{}cellules de $\mathdeuxcat{A}$ et $F(a)$ respectivement. Les \deux{}cellules de $(f,r)$ vers $(g,s)$ sont les couples $(\gamma : f \Rightarrow g, \varphi : s \Rightarrow r (F(\gamma))_{x'})$, $\gamma$ et $\varphi$ étant des \deux{}cellules dans $\mathdeuxcat{A}$ et $F(a)$ respectivement. Les diverses unités et compositions sont définies de façon « évidente ».  La projection canonique $\DeuxInt{\DeuxCatToutOp{\mathdeuxcat{A}}} F \to \DeuxCatToutOp{\mathdeuxcat{A}}$ étant une précoopfibration, la projection canonique $\DeuxIntCoOpToutOp{\mathdeuxcat{A}} F \to \mathdeuxcat{A}$ est une préfibration, dont la fibre au-dessus de $a \in (\Objets{\mathdeuxcat{A}})$ s'identifie à la \deux{}catégorie $\DeuxCatToutOp{(F(a))}$. Cette construction est en fait fonctorielle sur les \DeuxFoncteursCoLax{} $\mathdeuxcat{A} \to \DeuxCatDeuxCat$ et les \DeuxTransformationsCoLax{} entre iceux.
\end{paragr}

%
%

\begin{rem}\label{RemarqueMutatis}
La proposition \ref{ProprietesJaKa} se transpose bien sûr \emph{mutatis mutandis} aux variantes introduites ci-dessus de la construction d'intégration. Nous explicitons ci-dessous le résultat pour trois d'entre elles.
\end{rem}

\begin{prop}\label{ProprietesJaKaOp}
Pour tout \DeuxFoncteurStrict{} $F : \DeuxCatUnOp{\mathdeuxcat{A}} \to \DeuxCatDeuxCat$ et tout objet $a$ de $\mathdeuxcat{A}$, le \DeuxFoncteurStrict{} canonique $J_{a} : \Fibre{\left(\DeuxIntOp{\mathdeuxcat{A}}F\right)}{P_{F}}{a} \to \OpTrancheCoLax{\left(\DeuxIntOp{\mathdeuxcat{A}}F\right)}{P_{F}}{a}$ (qui est un préadjoint à gauche lax) admet une rétraction canonique $K_{a}$ qui est un \DeuxFoncteurStrict{} et un préadjoint à droite colax. De plus, il existe une \DeuxTransformationCoLax{} canonique $J_{a} K_{a} \Rightarrow 1_{\OpTrancheCoLax{\left(\DeuxIntOp{\mathdeuxcat{A}}F\right)}{P_{F}}{a}}$.  
\end{prop} 

\begin{prop}\label{ProprietesJaKaCo}
Pour tout \DeuxFoncteurStrict{} $F : \DeuxCatDeuxOp{\mathdeuxcat{A}} \to \DeuxCatDeuxCat$ et tout objet $a$ de $\mathdeuxcat{A}$, le \DeuxFoncteurStrict{} canonique $J_{a} : \Fibre{\left(\DeuxIntCo{\mathdeuxcat{A}}F\right)}{P_{F}}{a} \to \TrancheCoLax{\left(\DeuxIntCo{\mathdeuxcat{A}}F\right)}{P_{F}}{a}$ (qui est un préadjoint à droite lax) admet une rétraction canonique $K_{a}$ qui est un \DeuxFoncteurStrict{} et un préadjoint à gauche colax. De plus, il existe une \DeuxTransformationLax{} canonique $1_{\TrancheCoLax{\left(\DeuxIntCo{\mathdeuxcat{A}}F\right)}{P_{F}}{a}} \Rightarrow J_{a} K_{a}$.  
\end{prop} 

\begin{prop}\label{ProprietesJaKaCoOp}
Pour tout \DeuxFoncteurStrict{} $F : \DeuxCatToutOp{\mathdeuxcat{A}} \to \DeuxCatDeuxCat$ et tout objet $a$ de $\mathdeuxcat{A}$, le \DeuxFoncteurStrict{} canonique $J_{a} : \Fibre{\left(\DeuxIntCoOp{\mathdeuxcat{A}}F\right)}{P_{F}}{a} \to \OpTrancheLax{\left(\DeuxIntCoOp{\mathdeuxcat{A}}F\right)}{P_{F}}{a}$ (qui est un préadjoint à gauche colax) admet une rétraction canonique $K_{a}$ qui est un \DeuxFoncteurStrict{} et un préadjoint à droite lax. De plus, il existe une \DeuxTransformationLax{} canonique $J_{a} K_{a} \Rightarrow 1_{\OpTrancheLax{\left(\DeuxIntCoOp{\mathdeuxcat{A}}F\right)}{P_{F}}{a}} $.  
\end{prop} 

\begin{paragr}
Rappelons que les considérations de la section \ref{SectionMorphismesTranches} nous ont permis d'associer, à tout \DeuxFoncteurStrict{}\footnote{Et même à tout \DeuxFoncteurLax{}, mais le cas général ne nous intéressera pas ici.} $w : \mathdeuxcat{A} \to \mathdeuxcat{C}$, un \DeuxFoncteurStrict{} $\DeuxFoncteurTranche{w} : \mathdeuxcat{C} \to \DeuxCatDeuxCat$ dont la définition s'y trouve détaillée. En vertu de la proposition \ref{ProjectionIntegralePrefibration}, la projection canonique $P_{\mathdeuxcat{C}} : \DeuxInt{\mathdeuxcat{C}} \DeuxFoncteurTranche{w} \to \mathdeuxcat{C}$ est une précoopfibration. Explicitons la structure de $\DeuxInt{\mathdeuxcat{C}} \DeuxFoncteurTranche{w}$. Au cours des descriptions suivantes, on omet quelques détails que le lecteur rétablira de lui-même s'il le souhaite. 

Les objets de $\DeuxInt{\mathdeuxcat{C}} \DeuxFoncteurTranche{w}$ sont les 
$$
(c, (a, p : w(a) \to c)).
$$

Les \un{}cellules de $(c, (a, p))$ vers $(c', (a', p'))$ sont les 
$$
(k : c \to c', (f : a \to a', \gamma : kp \Rightarrow p' w(f))).
$$

Les \deux{}cellules de 
$$
(k : c \to c', (f : a \to a', \gamma : kp \Rightarrow p' w(f)))
$$ 
vers 
$$
(l : c \to c', (g : a \to a', \delta : lp \Rightarrow p' w(g)))
$$ 
sont les 
$$
(\epsilon : k \Rightarrow l, \varphi : f \Rightarrow g)
$$ 
tels que  
$$
(p' \CompDeuxZero w(\varphi)) \CompDeuxUn \gamma = \delta \CompDeuxUn (\epsilon \CompDeuxZero p).
$$

L'identité de l'objet $(c, (a, p))$ est donnée par 
$$
1_{(c, (a, p))} = (1_{c}, (1_{a}, 1_{p})).
$$

Étant donné deux \un{}cellules $(k, (f, \gamma))$ et $(k', (f', \gamma'))$ telles que la composée 
$$
(k', (f', \gamma')) (k, (f, \gamma))
$$ 
fasse sens, cette dernière est donnée par la formule 
$$
(k', (f', \gamma')) (k, (f, \gamma)) = (k'k, f'f, (\gamma' \CompDeuxZero w(f)) (k' \CompDeuxZero \gamma)).
$$

L'identité de la \un{}cellule $(k, (f, \gamma))$ est donnée par 
$$
1_{(k, (f, \gamma))} = (1_{k}, 1_{f}).
$$

Étant donné deux \deux{}cellules $(\epsilon, \varphi)$ et $(\lambda, \psi)$ telles que la composée $(\lambda, \psi) (\epsilon, \varphi)$ fasse sens, cette dernière est donnée par la formule
$$
(\lambda, \psi) (\epsilon, \varphi) = (\lambda \epsilon, \psi \varphi).
$$

Étant donné deux \deux{}cellules $(\epsilon, \varphi)$ et $(\epsilon', \varphi')$ telles que la composée $(\epsilon', \varphi') \CompDeuxZero (\epsilon, \varphi)$ fasse sens, cette dernière est donnée par la formule
$$
(\epsilon', \varphi') \CompDeuxZero (\epsilon, \varphi) = (\epsilon' \CompDeuxZero \epsilon, \varphi' \CompDeuxZero \varphi).
$$

Cela termine la description de la \deux{}catégorie $\DeuxInt{\mathdeuxcat{C}} \DeuxFoncteurTranche{w}$. On sait déjà qu'elle est précoopfibrée sur $\mathdeuxcat{C}$. On va voir qu'elle est également préfibrée sur $\mathdeuxcat{A}$. 
\end{paragr}

\begin{paragr}\label{QPrefibration}
La donnée du \DeuxFoncteurStrict{} $w : \mathdeuxcat{A} \to \mathdeuxcat{C}$ permet de définir un \DeuxFoncteurStrict{} 
$$
\begin{aligned}
\DeuxCatUnOp{\mathdeuxcat{A}} &\to \DeuxCatDeuxCat
\\
a &\mapsto \OpTrancheCoLax{\mathdeuxcat{C}}{}{w(a)}.
\end{aligned}
$$
Il existe un isomorphisme canonique\footnote{On détaillera un isomorphisme du même genre au cours de la démonstration du théorème \ref{DeuxLocFondHuitDef}. On peut démontrer des isomorphismes plus généraux en utilisant les structures commas.}
$$
\DeuxInt{\mathdeuxcat{C}} \DeuxFoncteurTranche{w} = \DeuxInt{\mathdeuxcat{C}} \TrancheLax{\mathdeuxcat{A}}{w}{c} \simeq \DeuxIntOp{\mathdeuxcat{A}} (a \mapsto \OpTrancheCoLax{\mathdeuxcat{C}}{}{w(a)}). 
$$

La projection 
$$
\begin{aligned}
Q_{\mathdeuxcat{A}}\index[not]{QA@$Q_{\mathdeuxcat{A}}$} : \DeuxInt{\mathdeuxcat{C}}\DeuxFoncteurTranche{w} &\to \mathdeuxcat{A}
\\
(c, (a, p)) &\mapsto a
\\
(k, (f, \gamma)) &\mapsto f
\\
(\epsilon, \varphi) &\mapsto \varphi
\end{aligned}
$$
s'identifie à la projection canonique $\DeuxIntOp{\mathdeuxcat{A}} (a \mapsto \OpTrancheCoLax{\mathdeuxcat{C}}{}{w(a)}) \to \mathdeuxcat{A}$ ; c'est donc une préfibration, dont la fibre au-dessus de $a \in (\Objets{\mathdeuxcat{A}})$ s'identifie à la \deux{}catégorie $\OpTrancheCoLax{\mathdeuxcat{C}}{}{w(a)}$. 
\end{paragr}

\begin{lemme}\label{Desproges}
Soient
$$
\xymatrix{
\mathdeuxcat{A} 
\ar[rr]^{u}
\ar[dr]_{w}
&&\mathdeuxcat{B}
\ar[dl]^{v}
\dtwocell<\omit>{<7.3>\sigma}
\\
&\mathdeuxcat{C}
&{}
}
$$
un diagramme de \DeuxFoncteursStricts{} commutatif à l'\DeuxTransformationCoLax{} $\sigma : vu \Rightarrow w$ près seulement. Alors, le diagramme de \DeuxFoncteursStricts{}\footnote{Voir le paragraphe \ref{IntegrationTransformation} pour la définition de $\DeuxFoncteurTranche{\sigma}$.}
$$
\xymatrix{
\DeuxInt{\mathdeuxcat{C}}\DeuxFoncteurTranche{w}
\ar[rr]^{\DeuxInt{\mathdeuxcat{C}}\DeuxFoncteurTranche{\sigma}}
\ar[d]_{Q_{\mathdeuxcat{A}}}
&&\DeuxInt{\mathdeuxcat{C}}\DeuxFoncteurTranche{v}
\ar[d]^{Q_{\mathdeuxcat{B}}}
\\
\mathdeuxcat{A}
\ar[rr]_{u}
&&\mathdeuxcat{B}
}
$$
(dans lequel les flèches verticales désignent les projections) est commutatif. 
\end{lemme}

\begin{proof}
Les vérifications, immédiates, ne présentent aucune difficulté : 
$$
\begin{aligned}
u(Q_{\mathdeuxcat{A}} (c, (a,p))) &= u(a)
\\
&= Q_{\mathdeuxcat{B}} (c, (u(a), p \sigma_{a}))
\\
&= Q_{\mathdeuxcat{B}} \left(\DeuxInt{\mathdeuxcat{C}}\DeuxFoncteurTranche{\sigma} (c, (a,p))\right),
\\
u(Q_{\mathdeuxcat{A}} (k, (f, \gamma))) &= u(f)
\\
&= Q_{\mathdeuxcat{B}} (k, u(f), (p' \CompDeuxZero \sigma_{f}) (\gamma \CompDeuxZero \sigma_{a}))
\\
&= Q_{\mathdeuxcat{B}} \left(\DeuxInt{\mathdeuxcat{C}}\DeuxFoncteurTranche{\sigma} (k, (f, \gamma))\right),
\\
u(Q_{\mathdeuxcat{A}} (\epsilon, \varphi)) &= u(\varphi)
\\
&= Q_{\mathdeuxcat{B}} (\epsilon, u(\varphi))
\\
&= Q_{\mathdeuxcat{B}} \left(\DeuxInt{\mathdeuxcat{C}}\DeuxFoncteurTranche{\sigma} (\epsilon, \varphi)\right).
\end{aligned}
$$
\end{proof}

\section{Cylindres tordus}\label{SectionCylindres}

\begin{paragr}\label{DefS1}
Pour toute \deux{}catégorie $\mathdeuxcat{A}$, les applications
$$
a \mapsto \OpTrancheLax{\mathdeuxcat{A}}{}{a}
$$
et
$$
a \mapsto \DeuxCatUnOp{(\TrancheLax{\mathdeuxcat{A}}{}{a})}
$$
définissent des \DeuxFoncteursStricts{} de $\DeuxCatToutOp{\mathdeuxcat{A}}$ vers $\DeuxCatDeuxCat$ et de $\DeuxCatDeuxOp{\mathdeuxcat{A}}$ vers $\DeuxCatDeuxCat$ respectivement\footnote{Cela résulte par dualité du fait que $a \mapsto \TrancheLax{\mathdeuxcat{A}}{}{a}$ définit un \DeuxFoncteurStrict{} de $\mathdeuxcat{A}$ vers $\DeuxCatDeuxCat$, correspondant au cas particulier $u = 1_{\mathdeuxcat{A}}$ dans le paragraphe \ref{DefFoncteursTranches}.}. 

Les formules générales permettent d'expliciter la structure de la \deux{}catégorie $\DeuxIntCo{\DeuxCatUnOp{\mathdeuxcat{A}}} (\OpTrancheLax{\mathdeuxcat{A}}{}{a})$ comme suit. 

Ses objets sont les
$$
(b, (a, k : b \to a)),
$$
$b$ et $a$ étant des objets de $\mathdeuxcat{A}$ et $k$ une \un{}cellule de $\mathdeuxcat{A}$.

Les \un{}cellules de $(b, (a,k))$ vers $(b', (a',k'))$ sont les
$$
(f : b' \to b, (g : a \to a', \alpha : k' \Rightarrow gkf)),
$$
$f$ et $g$ étant des \un{}cellules de $\mathdeuxcat{A}$ et $\alpha$ une \deux{}cellule de $\mathdeuxcat{A}$. 

Les \deux{}cellules de $(f, (g, \alpha)) : (b, (a,k)) \to (b', (a', k'))$ vers $(f', (g', \alpha')) : (b, (a,k)) \to (b', (a', k'))$ sont les
$$
(\varphi : f \Rightarrow f', \gamma : g \Rightarrow g')
$$
vérifiant 
$$
(\gamma \CompDeuxZero k \CompDeuxZero \varphi) \CompDeuxUn \alpha = \alpha',
$$
$\varphi$ et $\gamma$ étant des \deux{}cellules de $\mathdeuxcat{A}$. 

Les diverses unités et compositions sont définies de façon « évidente ». 

Les formules générales permettent d'expliciter la structure de la \deux{}catégorie $\DeuxIntCo{\mathdeuxcat{A}} \DeuxCatUnOp{(\TrancheLax{\mathdeuxcat{A}}{}{a})}$ comme suit. 

Ses objets sont les
$$
(a, (b, k : b \to a)),
$$
$a$ et $b$ étant des objets de $\mathdeuxcat{A}$ et $k$ une \un{}cellule de $\mathdeuxcat{A}$. 

Les \un{}cellules de $(a, (b,k))$ vers $(a', (b',k'))$ sont les
$$
(g : a \to a', (f : b' \to b, \alpha : k' \Rightarrow gkf)),
$$
$g$ et $f$ étant des \un{}cellules de $\mathdeuxcat{A}$ et $\alpha$ une \deux{}cellule de $\mathdeuxcat{A}$. 

Les \deux{}cellules de $(g, (f, \alpha)) : (a, (b,k)) \to (a', (b', k'))$ vers $(g', (f', \alpha')) : (a, (b,k)) \to (a', (b', k'))$ sont les 
$$
(\gamma : g \Rightarrow g', \varphi : f \Rightarrow f')
$$
vérifiant 
$$
(\gamma \CompDeuxZero k \CompDeuxZero \varphi) \CompDeuxUn \alpha = \alpha',
$$
$\varphi$ et $\gamma$ étant des \deux{}cellules de $\mathdeuxcat{A}$. 

Les diverses unités et compositions sont définies de façon « évidente ». 

On note $S_{1}(\mathdeuxcat{A})$\index[not]{S1A@$S_{1}(\mathdeuxcat{A})$} la \deux{}catégorie définie comme suit. Ses objets sont les \un{}cellules $k : b \to a$ de $\mathdeuxcat{A}$. Les \un{}cellules de $k : b \to a$ vers $k' : b' \to a'$ sont les $(f : b' \to b, g : a \to a', \alpha : k' \Rightarrow gkf)$ avec $f$ et $g$ des \un{}cellules de $\mathdeuxcat{A}$ et $\alpha$ une \deux{}cellule de $\mathdeuxcat{A}$. Les \deux{}cellules de $(f, g, \alpha) : k \to k'$ vers $(f', g', \alpha')$ sont données par les couples de \deux{}cellules $(\varphi : f \Rightarrow f', \gamma : g \Rightarrow g')$ de $\mathdeuxcat{A}$ telles que $(\gamma \CompDeuxZero k \CompDeuxZero \varphi) \CompDeuxUn \alpha = \alpha'$. Les diverses unités et compositions sont définies de façon « évidente ». 

Il existe alors des isomorphismes canoniques
$$
\begin{aligned}
S_{1}(\mathdeuxcat{A}) &\to \DeuxIntCo{\DeuxCatUnOp{\mathdeuxcat{A}}} (\OpTrancheLax{\mathdeuxcat{A}}{}{a})
\\
(k : b \to a) &\mapsto (b, (a,k))
\\
(f,g,\alpha) &\mapsto (f,(g,\alpha))
\\
(\varphi, \gamma) &\mapsto (\varphi, \gamma)
\end{aligned}
$$
et
$$
\begin{aligned}
 \DeuxIntCo{\DeuxCatUnOp{\mathdeuxcat{A}}} (\OpTrancheLax{\mathdeuxcat{A}}{}{a}) &\to S_{1}(\mathdeuxcat{A}) 
\\
(b, (a,k)) &\mapsto k 
\\
(f,(g,\alpha)) &\mapsto (f,g,\alpha) 
\\
(\varphi, \gamma) &\mapsto (\varphi, \gamma)
\end{aligned}
$$
inverses l'un de l'autre ainsi que des isomorphismes canoniques
$$
\begin{aligned}
S_{1}(\mathdeuxcat{A}) &\to \DeuxIntCo{\mathdeuxcat{A}} \DeuxCatUnOp{(\TrancheLax{\mathdeuxcat{A}}{}{a})}
\\
(k : b \to a) &\mapsto (a, (b,k))
\\
(f,g,\alpha) &\mapsto (g,(f,\alpha))
\\
(\varphi, \gamma) &\mapsto (\gamma, \varphi)
\end{aligned}
$$
et 
$$
\begin{aligned}
 \DeuxIntCo{\mathdeuxcat{A}} \DeuxCatUnOp{(\TrancheLax{\mathdeuxcat{A}}{}{a})} &\to S_{1}(\mathdeuxcat{A})
\\
(a, (b,k)) &\mapsto k
\\
(g,(f,\alpha)) &\mapsto (f,g,\alpha)
\\
(\gamma, \varphi) &\mapsto (\varphi, \gamma)
\end{aligned}
$$
inverses l'un de l'autre. 

On note $s_{1}^{\mathdeuxcat{A}}$\index[not]{s1A@$s_{1}^{\mathdeuxcat{A}}$} (\emph{resp.} $t_{1}^{\mathdeuxcat{A}}$\index[not]{t1A@$t_{1}^{\mathdeuxcat{A}}$}) la projection canonique $S_{1}(\mathdeuxcat{A}) \to \DeuxCatUnOp{\mathdeuxcat{A}}$ (\emph{resp.} $S_{1}(\mathdeuxcat{A}) \to \mathdeuxcat{A}$). En vertu des isomorphismes ci-dessus et du paragraphe \ref{DefDeuxIntCo}, $s_{1}^{\mathdeuxcat{A}}$ et $t_{1}^{\mathdeuxcat{A}}$ sont des préopfibrations. 

Pour tout \DeuxFoncteurStrict{} $u : \mathdeuxcat{A} \to \mathdeuxcat{B}$, on définit un \DeuxFoncteurStrict{}
$$
\begin{aligned}
S_{1}(u)\index[not]{S1u@$S_{1}(u)$} : S_{1}(\mathdeuxcat{A}) &\to S_{1}(\mathdeuxcat{B}{})
\\
k &\mapsto u(k)
\\
(f,g,\alpha) &\mapsto (u(f), u(g), u(\alpha))
\\
(\varphi, \gamma) &\mapsto (u(\varphi), u(\gamma)).
\end{aligned}
$$

Les vérifications de fonctorialité s'effectuent sans aucune difficulté comme suit. 

Pour tout objet $k : b \to a$ de $S_{1}(\mathdeuxcat{A})$, 
$$
\begin{aligned}
S_{1}(u) (1_{k}) &= S_{1}(u) (1_{b}, 1_{a}, 1_{k})
\\
&= (u(1_{b}), u(1_{a}), u(1_{k}))
\\
&= (1_{u(b)}, 1_{u(a)}, 1_{u(k)})
\\
&= 1_{u(k)}
\\
&= 1_{S_{1}(u)(k)}.
\end{aligned}
$$

Pour tout couple de \un{}cellules $(f,g,\alpha)$ et $(f',g',\alpha')$ de $S_{1}(\mathdeuxcat{A})$ telles que la composée $(f',g',\alpha') (f,g,\alpha)$ fasse sens, 
$$
\begin{aligned}
S_{1}(u) ((f',g',\alpha') (f,g,\alpha)) &= S_{1}(u) (ff', g'g, (g' \CompDeuxZero \alpha \CompDeuxZero f') \CompDeuxUn \alpha')
\\
&= (u(ff'), u(g'g), u((g' \CompDeuxZero \alpha \CompDeuxZero f') \CompDeuxUn \alpha'))
\\
&= (u(f) u(f'), u(g')u(g), (u(g') \CompDeuxZero u(\alpha) \CompDeuxZero u(f')) \CompDeuxUn u(\alpha'))
\\
&= (u(f'), u(g'), u(\alpha')) (u(f), u(g), u(\alpha)) 
\\
&= S_{1}(u) (f', g', \alpha') S_{1}(u) (f, g, \alpha).
\end{aligned}
$$

Pour toute \un{}cellule $(f, g, \alpha)$ de $S_{1}(\mathdeuxcat{A})$, 
$$
\begin{aligned}
S_{1}(u) (1_{(f, g, \alpha)}) &= S_{1}(u) (1_{f}, 1_{g})
\\
&= (u(1_{f}), u(1_{g}))
\\
&= (1_{u(f)}, 1_{u(g)})
\\
&= 1_{(u(f), u(g), u(\alpha))}
\\
&= 1_{S_{1}(u) (f, g, \alpha)}.
\end{aligned}
$$

Pour tout couple de \deux{}cellules $(\varphi', \gamma')$ et $(\varphi, \gamma)$ de $S_{1}(\mathdeuxcat{A})$ telles que la composée $(\varphi', \gamma') \CompDeuxUn (\varphi, \gamma)$ fasse sens, 
$$
\begin{aligned}
S_{1}(u) ((\varphi', \gamma') \CompDeuxUn (\varphi, \gamma)) &= S_{1}(u) (\varphi' \CompDeuxUn \varphi, \gamma' \CompDeuxUn \gamma)
\\
&= (u (\varphi' \CompDeuxUn \varphi), u(\gamma' \CompDeuxUn \gamma))
\\
&= (u(\varphi') \CompDeuxUn u(\varphi), u(\gamma') \CompDeuxUn u(\gamma)) 
\\
&= (u(\varphi'), u(\gamma')) \CompDeuxUn (u(\varphi), u(\gamma))
\\
&= S_{1}(u) (\varphi', \gamma') \CompDeuxUn S_{1}(u) (\varphi, \gamma).
\end{aligned}
$$

Pour tout couple de \deux{}cellules $(\varphi', \gamma')$ et $(\varphi, \gamma)$ de $S_{1}(\mathdeuxcat{A})$ telles que la composée $(\varphi', \gamma') \CompDeuxZero (\varphi, \gamma)$ fasse sens, 
$$
\begin{aligned}
S_{1}(u) ((\varphi', \gamma') \CompDeuxZero (\varphi, \gamma)) &= S_{1}(u) (\varphi \CompDeuxZero \varphi', \gamma' \CompDeuxZero \gamma)
\\
&= (u(\varphi \CompDeuxZero \varphi'), u(\gamma' \CompDeuxZero \gamma))
\\
&= (u(\varphi) \CompDeuxZero u(\varphi'), u(\gamma') \CompDeuxZero u(\gamma))
\\
&= (u(\varphi'), u(\gamma')) \CompDeuxZero (u(\varphi), u(\gamma))
\\
&= (S_{1}(u) (\varphi', \gamma')) \CompDeuxZero (S_{1}(u) (\varphi, \gamma)).
\end{aligned}
$$

On a donc un diagramme commutatif de \DeuxFoncteursStricts{} :
$$
\xymatrix{
\DeuxCatUnOp{\mathdeuxcat{A}}
\ar[d]_{\DeuxFoncUnOp{u}}
&S_{1}(\mathdeuxcat{A})
\ar[l]_{s_{1}^{\mathdeuxcat{A}}}
\ar[d]^{S_{1}(u)}
\ar[r]^{t_{1}^{\mathdeuxcat{A}}}
&\mathdeuxcat{A}
\ar[d]^{u}
\\
\DeuxCatUnOp{\mathdeuxcat{B}}
&S_{1}(\mathdeuxcat{B}{})
\ar[l]^{s_{1}^{\mathdeuxcat{B}{}}}
\ar[r]_{t_{1}^{\mathdeuxcat{B}{}}}
&\mathdeuxcat{B}{}
&.
}
$$
\end{paragr}

\begin{paragr}\label{DefS2}
Pour toute \deux{}catégorie $\mathdeuxcat{A}$, les applications
$$
a \mapsto \DeuxCatUnOp{(\OpTrancheLax{\mathdeuxcat{A}}{}{a})}
$$
et
$$
a \mapsto \DeuxCatDeuxOp{(\TrancheCoLax{\mathdeuxcat{A}}{}{a})}
$$
définissent des \DeuxFoncteursStricts{} de $\DeuxCatUnOp{\mathdeuxcat{A}}$ vers $\DeuxCatDeuxCat$ et de $\DeuxCatDeuxOp{\mathdeuxcat{A}}$ vers $\DeuxCatDeuxCat$ respectivement\footnote{Cela résulte aussi, par dualité, du fait que $a \mapsto \TrancheLax{\mathdeuxcat{A}}{}{a}$ définit un \DeuxFoncteurStrict{} de $\mathdeuxcat{A}$ vers $\DeuxCatDeuxCat$.}. 


Les formules générales permettent d'expliciter la structure de la \deux{}catégorie $\DeuxCatUnOp{\left(\DeuxIntCo{\DeuxCatToutOp{\mathdeuxcat{A}}} 
\DeuxCatUnOp{(\OpTrancheLax{\mathdeuxcat{A}}{}{a})}\right)}$ comme suit. 

Ses objets sont les
$$
(b, (a, k : b \to a)),
$$
$b$ et $a$ étant des objets de $\mathdeuxcat{A}$ et $k$ une \un{}cellule de $\mathdeuxcat{A}$.

Les \un{}cellules de $(b, (a,k))$ vers $(b', (a',k'))$ sont les
$$
(f : b \to b', (g : a \to a', \alpha : k'f \Rightarrow gk)),
$$
$f$ et $g$ étant des \un{}cellules de $\mathdeuxcat{A}$ et $\alpha$ une \deux{}cellule de $\mathdeuxcat{A}$. 

Les \deux{}cellules de $(f, (g, \alpha)) : (b,(a,k)) \to (b',(a',k'))$ vers $(f', (g', \alpha')) : (b,(a,k)) \to (b',(a',k'))$ sont les
$$
(\varphi : f' \Rightarrow f, \gamma : g \Rightarrow g')
$$
vérifiant 
$$
(\gamma \CompDeuxZero k) \CompDeuxUn \alpha \CompDeuxUn (k' \CompDeuxZero \varphi)= \alpha',
$$
$\varphi$ et $\gamma$ étant des \deux{}cellules de $\mathdeuxcat{A}$. 

Les diverses unités et compositions sont définies de façon « évidente ». 

Les formules générales permettent d'expliciter la structure de la \deux{}catégorie $\DeuxIntCo{\mathdeuxcat{A}} \DeuxCatDeuxOp{(\TrancheCoLax{\mathdeuxcat{A}}{}{a})}$ comme suit. 

Ses objets sont les
$$
(a, (b, k : b \to a)),
$$
$a$ et $b$ étant des objets de $\mathdeuxcat{A}$ et $k$ une \un{}cellule de $\mathdeuxcat{A}$.

Les \un{}cellules de $(a, (b,k))$ vers $(a', (b',k'))$ sont les
$$
(g : a \to a', (f : b \to b', \alpha : k'f \Rightarrow gk)),
$$
$g$ et $f$ étant des \un{}cellules de $\mathdeuxcat{A}$ et $\alpha$ une \deux{}cellule de $\mathdeuxcat{A}$. 

Les \deux{}cellules de $(g, (f, \alpha)) : (a,(b,k)) \to (a',(b',k'))$ vers $(g', (f', \alpha')) : (a,(b,k)) \to (a',(b',k'))$ sont les
$$
(\gamma : g \Rightarrow g', \varphi : f' \Rightarrow f)
$$
vérifiant 
$$
(\gamma \CompDeuxZero k) \CompDeuxUn \alpha \CompDeuxUn (k' \CompDeuxZero \varphi)= \alpha',
$$
$\gamma$ et $\varphi$ étant des \deux{}cellules de $\mathdeuxcat{A}$. 

Les diverses unités et compositions sont définies de façon « évidente ». 

On note $S_{2}(\mathdeuxcat{A})$\index[not]{S2A@$S_{2}(\mathdeuxcat{A})$} la \deux{}catégorie définie comme suit. Ses objets sont les \un{}cellules $k : b \to a$ de $\mathdeuxcat{A}$. Les \un{}cellules de $k : b \to a$ vers $k' : b' \to a'$ sont les $(f : b \to b', g : a \to a', \alpha : k' f \Rightarrow g k)$ avec $f$ et $g$ des \un{}cellules de $\mathdeuxcat{A}$ et $\alpha$ une \deux{}cellule de $\mathdeuxcat{A}$. Les \deux{}cellules de $(f, g, \alpha)$ vers $(f', g', \alpha')$ sont données par les couples de \deux{}cellules $(\varphi : f' \Rightarrow f, \gamma : g \Rightarrow g')$ de $\mathdeuxcat{A}$ telles que $(\gamma \CompDeuxZero k) \CompDeuxUn \alpha \CompDeuxUn (k' \CompDeuxZero \varphi) = \alpha'$. Les diverses unités et compositions sont définies de façon « évidente ». 

Il existe alors des isomorphismes canoniques
$$
\begin{aligned}
S_{2}(\mathdeuxcat{A}) &\to \DeuxCatUnOp{\left(\DeuxIntCo{\DeuxCatToutOp{\mathdeuxcat{A}}} 
\DeuxCatUnOp{(\OpTrancheLax{\mathdeuxcat{A}}{}{a})}\right)}
\\
(k : b \to a) &\mapsto (b, (a,k))
\\
(f,g,\alpha) &\mapsto (f,(g,\alpha))
\\
(\varphi, \gamma) &\mapsto (\varphi, \gamma)
\end{aligned}
$$
et
$$
\begin{aligned}
\DeuxCatUnOp{\left(\DeuxIntCo{\DeuxCatToutOp{\mathdeuxcat{A}}} 
\DeuxCatUnOp{(\OpTrancheLax{\mathdeuxcat{A}}{}{a})}\right)} &\to S_{2}(\mathdeuxcat{A}) 
\\
(b, (a,k)) &\mapsto k 
\\
(f,(g,\alpha)) &\mapsto (f,g,\alpha) 
\\
(\varphi, \gamma) &\mapsto (\varphi, \gamma)
\end{aligned}
$$
inverses l'un de l'autre ainsi que des isomorphismes canoniques
$$
\begin{aligned}
S_{2}(\mathdeuxcat{A}) &\to \DeuxIntCo{\mathdeuxcat{A}} \DeuxCatDeuxOp{(\TrancheCoLax{\mathdeuxcat{A}}{}{a})}
\\
(k : b \to a) &\mapsto (a, (b,k))
\\
(f,g,\alpha) &\mapsto (g,(f,\alpha))
\\
(\varphi, \gamma) &\mapsto (\gamma, \varphi)
\end{aligned}
$$
et 
$$
\begin{aligned}
\DeuxIntCo{\mathdeuxcat{A}} \DeuxCatDeuxOp{(\TrancheCoLax{\mathdeuxcat{A}}{}{a})} &\to S_{2}(\mathdeuxcat{A})
\\
(a, (b,k)) &\mapsto k
\\
(g,(f,\alpha)) &\mapsto (f,g,\alpha)
\\
(\gamma, \varphi) &\mapsto (\varphi, \gamma)
\end{aligned}
$$
inverses l'un de l'autre. 

On note $s_{2}^{\mathdeuxcat{A}}$\index[not]{s2A@$s_{2}^{\mathdeuxcat{A}}$} (\emph{resp.} $t_{2}^{\mathdeuxcat{A}}$\index[not]{t2A@$t_{2}^{\mathdeuxcat{A}}$}) la projection canonique $S_{2}(\mathdeuxcat{A}) \to \DeuxCatDeuxOp{\mathdeuxcat{A}}$ (\emph{resp.} $S_{2}(\mathdeuxcat{A}) \to \mathdeuxcat{A}$). En vertu des isomorphismes ci-dessus et du paragraphe \ref{DefDeuxIntCo}, la projection canonique 
$$
\DeuxIntCo{\DeuxCatToutOp{\mathdeuxcat{A}}} 
\DeuxCatUnOp{(\OpTrancheLax{\mathdeuxcat{A}}{}{a})} \to \DeuxCatToutOp{\mathdeuxcat{A}},
$$ 
qui s'identifie à $\DeuxFoncUnOp{(s_{2}^{\mathdeuxcat{A}})}$, est une préopfibration, donc $s_{2}^{\mathdeuxcat{A}}$ est une préfibration. En vertu du même paragraphe \ref{DefDeuxIntCo}, $t_{2}^{\mathdeuxcat{A}}$ est une préopfibration. 

Pour tout \DeuxFoncteurStrict{} $u : \mathdeuxcat{A} \to \mathdeuxcat{B}$, l'on définit un \DeuxFoncteurStrict{}
$$
\begin{aligned}
S_{2}(u)\index[not]{S2u@$S_{2}(u)$} : S_{2}(\mathdeuxcat{A}) &\to S_{2}(\mathdeuxcat{B}{})
\\
k &\mapsto u(k)
\\
(f,g,\alpha) &\mapsto (u(f), u(g), u(\alpha))
\\
(\varphi, \gamma) &\mapsto (u(\varphi), u(\gamma)).
\end{aligned}
$$

Les vérifications de fonctorialité s'effectuent sans aucune difficulté comme suit. 

Pour tout objet $k : b \to a$ de $S_{2}(\mathdeuxcat{A})$, 
$$
\begin{aligned}
S_{2}(u) (1_{k}) &= S_{2}(u) (1_{b}, 1_{a}, 1_{k})
\\
&= (u(1_{b}), u(1_{a}), u(1_{k}))
\\
&= (1_{u(b)}, 1_{u(a)}, 1_{u(k)})
\\
&= 1_{u(k)}
\\
&= 1_{S_{2}(u)(k)}.
\end{aligned}
$$

Pour tout couple de \un{}cellules $(f,g,\alpha)$ et $(f',g',\alpha')$ de $S_{2}(\mathdeuxcat{A})$ telles que la composée $(f',g',\alpha') (f,g,\alpha)$ fasse sens, 
$$
\begin{aligned}
S_{2}(u) ((f',g',\alpha') (f,g,\alpha)) &= S_{2}(u) (f'f, (g' \CompDeuxZero \alpha \CompDeuxZero f') \CompDeuxUn \alpha', g'g)
\\
&= (u(f'f), u((g' \CompDeuxZero \alpha \CompDeuxZero f') \CompDeuxUn \alpha'), u(g'g))
\\
&= (u(f') u(f), (u(g') \CompDeuxZero u(\alpha) \CompDeuxZero u(f')) \CompDeuxUn u(\alpha'), u(g')u(g))
\\
&= (u(f'), u(g'), u(\alpha')) (u(f), u(g), u(\alpha)) 
\\
&= S_{2}(u) (f', g', \alpha') S_{2}(u) (f, g, \alpha).
\end{aligned}
$$

Pour toute \un{}cellule $(f, g, \alpha)$ de $S_{2}(\mathdeuxcat{A})$, 
$$
\begin{aligned}
S_{2}(u) (1_{(f, g, \alpha)}) &= S_{2}(u) (1_{f}, 1_{g})
\\
&= (u(1_{f}), u(1_{g}))
\\
&= (1_{u(f)}, 1_{u(g)})
\\
&= 1_{(u(f), u(g), u(\alpha))}
\\
&= 1_{S_{2}(u) (f, g, \alpha)}.
\end{aligned}
$$

Pour tout couple de \deux{}cellules $(\varphi', \gamma')$ et $(\varphi, \gamma)$ de $S_{2}(\mathdeuxcat{A})$ telles que la composée $(\varphi', \gamma') \CompDeuxUn (\varphi, \gamma)$ fasse sens, 
$$
\begin{aligned}
S_{2}(u) ((\varphi', \gamma') \CompDeuxUn (\varphi, \gamma)) &= S_{2}(u) (\varphi \CompDeuxUn \varphi', \gamma' \CompDeuxUn \gamma)
\\
&= (u (\varphi \CompDeuxUn \varphi'), u(\gamma' \CompDeuxUn \gamma))
\\
&= (u(\varphi) \CompDeuxUn u(\varphi'), u(\gamma') \CompDeuxUn u(\gamma)) 
\\
&= (u(\varphi'), u(\gamma')) \CompDeuxUn (u(\varphi), u(\gamma))
\\
&= S_{2}(u) (\varphi', \gamma') \CompDeuxUn S_{2}(u) (\varphi, \gamma).
\end{aligned}
$$

Pour tout couple de \deux{}cellules $(\varphi', \gamma')$ et $(\varphi, \gamma)$ de $S_{2}(\mathdeuxcat{A})$ telles que la composée $(\varphi', \gamma') \CompDeuxZero (\varphi, \gamma)$ fasse sens, 
$$
\begin{aligned}
S_{2}(u) ((\varphi', \gamma') \CompDeuxZero (\varphi, \gamma)) &= S_{2}(u) (\varphi' \CompDeuxZero \varphi, \gamma' \CompDeuxZero \gamma)
\\
&= (u(\varphi' \CompDeuxZero \varphi), u(\gamma' \CompDeuxZero \gamma))
\\
&= (u(\varphi') \CompDeuxZero u(\varphi), u(\gamma') \CompDeuxZero u(\gamma))
\\
&= (u(\varphi'), u(\gamma')) \CompDeuxZero (u(\varphi), u(\gamma))
\\
&= S_{2}(u) (\varphi', \gamma') \CompDeuxZero S_{2}(u) (\varphi, \gamma).
\end{aligned}
$$

On a donc un diagramme commutatif de \DeuxFoncteursStricts{} : 
$$
\xymatrix{
\DeuxCatDeuxOp{\mathdeuxcat{A}}
\ar[d]_{\DeuxFoncDeuxOp{u}}
&S_{2}(\mathdeuxcat{A})
\ar[l]_{s_{2}^{\mathdeuxcat{A}}}
\ar[d]^{S_{2}(u)}
\ar[r]^{t_{2}^{\mathdeuxcat{A}}}
&\mathdeuxcat{A}
\ar[d]^{u}
\\
\DeuxCatDeuxOp{\mathdeuxcat{B}}
&S_{2}(\mathdeuxcat{B}{})
\ar[l]^{s_{2}^{\mathdeuxcat{B}{}}}
\ar[r]_{t_{2}^{\mathdeuxcat{B}{}}}
&\mathdeuxcat{B}{}
&.
}
$$
\end{paragr}

\section{L'adjonction de Bénabou}\label{SectionBenabou}\index{adjonction de Bénabou}

\begin{paragr}\label{IntroAdjonctionBenabou}
L'objet de cette section est de décrire un adjoint à gauche de l'inclusion $\DeuxCat \hookrightarrow \DeuxCatLax$. En particulier, la construction que nous explicitons fournit, pour toute petite \deux{}catégorie $\mathdeuxcat{A}$, une petite \deux{}catégorie $\TildeLax{\mathdeuxcat{A}}$ et un \DeuxFoncteurLax{} $\LaxCanonique{\mathdeuxcat{A}} : \mathdeuxcat{A} \to \TildeLax{\mathdeuxcat{A}}$ tels que, pour tout morphisme $u : \mathdeuxcat{A} \to \mathdeuxcat{B}$ de $\DeuxCatLax$, il existe un unique morphisme $\BarreLax{u} : \TildeLax{\mathdeuxcat{A}} \to \mathdeuxcat{B}$ de $\DeuxCat$ tel que $\BarreLax{u} \LaxCanonique{\mathdeuxcat{A}} = u$. Cette construction, due à Bénabou dans un cadre bicatégorique plus général, reste malheureusement peu connue malgré son importance. Elle semble avoir fait l'objet de communications orales de Bénabou, mais ce dernier n'a rien publié à ce sujet, ses notes de travail se trouvant du reste aux dernières nouvelles dans des cartons. Dans ces conditions, la première trace écrite à laquelle il est possible de renvoyer le lecteur est le travail de Gray, plus précisément \cite[I, 4. 23]{Gray}. La seconde description explicite parvenue jusqu'à nous se trouve dans les notes \cite{NotesDelHoyo} de del Hoyo, lequel affirme rectifier une inexactitude de Gray. Del Hoyo mentionne la possibilité d'une approche plus conceptuelle de la question grâce aux résultats de \cite{BKP}. Nous devons à Steve Lack les précisions suivantes sur ce dernier point. L'article \cite{BKP} montre que, pour une \deux{}monade vérifiant certaines conditions, les morphismes lax d'algèbres $A \to B$ sont en bijection avec les morphismes stricts d'algèbres $A^{\dagger} \to B$ pour un certain $A^{\dagger}$. L'article \cite{Power} explique comment les \deux{}catégories dont est donné l'ensemble des objets sont les algèbres d'une \deux{}monade. Les morphismes lax pour de telles algèbres sont les \deux{}foncteurs lax agissant comme l'identité sur l'ensemble des objets. Cette \deux{}monade vérifie les conditions d'application de \cite[théorème 3.13]{BKP} et il est possible de montrer que l'algèbre universelle $A^{\dagger}$ satisfait la propriété universelle relativement aux morphismes lax arbitraires. Une autre approche, ne nécessitant pas de jamais supposer fixé l'ensemble des objets, consiste à faire usage de la \deux{}monade de \cite[section 4]{LackPaoli}, une autre référence possible étant \cite[section 6.2]{LackIcons}. Quelle que soit l'approche privilégiée, soulignons la nécessité, dans le cadre de notre travail, d'une description explicite de cet objet universel. 
\end{paragr}

\begin{paragr}
On rappelle que la catégorie des simplexes $\Delta$\index[not]{0Delta@$\Delta$} a pour objets les ensembles $[m] = \{ 0, \dots, m \}$ ordonnés par l'ordre sur les entiers pour $m \geq 0$ (ce qui permet de les considérer comme des petites catégories) et pour morphismes les applications croissantes. On notera $\Delta_{m}$\index[not]{0Deltam@$\Delta_{m}$} l'image de $[m]$ par le foncteur de Yoneda. 
\end{paragr}

\begin{paragr}\label{PreliminairesTilde1}
On notera $\Intervalles$\index[not]{0DeltaIntervalles@$\Intervalles$} la \emph{catégorie des intervalles}\index{catégorie des intervalles}, c'est-à-dire la sous-catégorie de $\Delta$ ayant mêmes objets que $\Delta$ et dont les morphismes sont les applications croissantes respectant les extrémités ; on appellera ces applications les \emph{morphismes d'intervalles}. Autrement dit, pour deux entiers positifs $m$ et $n$, une application $\varphi : [m] \to [n]$ est un morphisme d'intervalles si et seulement si elle est croissante et vérifie $\varphi(0) = 0$ et $\varphi(m) = n$. 
\end{paragr}

\begin{paragr}
Pour tout entier $m \geq 0$, toute \deux{}catégorie $\mathdeuxcat{A}$ et tout \DeuxFoncteurLax{} $x : [m] \to \mathdeuxcat{A}$, on notera $x_{i}$\index[not]{xi@$x_{i}$} l'image de $0 \leq i \leq m$ par $x$, $x_{j,i}$\index[not]{xji@$x_{j,i}$} l'image du morphisme $i \to j$ de $[m]$ par $x$ et $x_{k,j,i}$\index[not]{xkji@$x_{k,j,i}$} la \deux{}cellule structurale $x_{j \to k, i \to j} : x_{k,j} x_{j,i} \Rightarrow x_{k,i}$ pour $0 \leq i \leq j \leq k \leq m$. Afin d'éviter toute confusion avec l'objet $x_{i}$, on notera $(x)_{i}$\index[not]{xistructurale@$(x)_{i}$} la \deux{}cellule structurale d'unité de $x$ associée à l'objet $i$ de $[m]$. On a donc $(x)_{i} : 1_{x_{i}} \Rightarrow x_{i,i}$. Sous ces mêmes données, pour toute \DeuxTransformationLax{} ou \DeuxTransformationCoLax{} $\sigma : x \Rightarrow y$ de $x$ vers un \DeuxFoncteurLax{} $y : [m] \to \mathdeuxcat{A}$, on notera $\sigma_{j,i}$\index[not]{0sigmaji@$\sigma_{j,i}$} la composante de $\sigma$ en le morphisme $i \to j$ de $[m]$. C'est donc une \deux{}cellule de $\sigma_{j} x_{j,i}$ vers $y_{j,i} \sigma_{i}$ si $\sigma$ est une \DeuxTransformationLax{} et de $y_{j,i} \sigma_{i}$ vers $\sigma_{j} x_{j,i}$ si $\sigma$ est une \DeuxTransformationCoLax{}. On notera parfois $\sigma_{i}$ la \deux{}cellule $\sigma_{i,i-1}$.    
\end{paragr}

\begin{df}\label{DefDeuxCellGeneral}
Soient $u$ un \DeuxFoncteurLax{} de source $\mathdeuxcat{A}$, $n \geq 0$ un entier et $x : [n] \to \mathdeuxcat{A}$ un \DeuxFoncteurStrict{}. On définit comme suit une \deux{}cellule $u_{x}$\index[not]{ux@$u_{x}$}, que l'on notera parfois par abus $u_{x_{n,n-1}, \dots, x_{1,0}}$, ce qui ne sous-entend pas $n \geq 1$. Si $n = 0$, $u_{x} = u_{x_{0}} : 1_{u(x_{0})} \Rightarrow u(1_{x_{0}})$. Si $n = 1$, $u_{x} = 1_{u(x_{1,0})}$. Si $n = 2$, $u_{x} = u_{x_{2,1}, x_{1,0}}$. Si $n \geq 3$, 
$$
u_{x} = u_{x_{n,n-1}, \dots, x_{1,0}} = u_{x_{n,n-1}, x_{n-1, n-2} \dots x_{1,0}} (u(x_{n,n-1}) \CompDeuxZero u_{x_{n-1,n-2}, \dots, x_{1,0}}).
$$

Si cela n'entraîne pas d'ambiguïté, cette \deux{}cellule pourra se trouver notée $u_{x_{0} \to \dots \to x_{n}}$\index[not]{ux0blablaxn@$u_{x_{0} \to \dots \to x_{n}}$}, ce qui ne sous-entend pas non plus $n \geq 1$.
\end{df}

\begin{rem}\label{RemarqueCocycleGeneral}
Les \deux{}cellules introduites dans la définition \ref{DefDeuxCellGeneral} correspondent à « une façon d'aller de $u(x_{n,n-1}) \dots u(x_{1,0})$ vers $u(x_{n,n-1} \dots x_{1,0})$ »\footnote{Pour $n=0$, cela fait encore sens, comme nous l'a fait remarquer Dimitri Ara, une composition vide étant une identité.} en privilégiant un certain « bon parenthésage » de la première expression. La \emph{condition de cocycle généralisée}\index{condition de cocycle généralisée}, constamment utilisée dans la littérature relative aux \deux{}catégories, sans qu'aucun énoncé n'y figure généralement, encore moins une dé\-monstra\-tion, affirme que ces parenthésages sont tous équivalents : les \deux{}cellules associées sont égales. Pour un énoncé précis ainsi qu'une démonstration — omettant toutefois certains détails de calculs particulièrement pénibles — d'un résultat de cohérence très général, on renvoie à la thèse de Bénabou \cite[proposition 5.1.4 p. I-47 et théorème 5.2.4 p. I-49]{TheseBenabou}. 
\end{rem}

\begin{rem}\label{NaturaliteStructuraleGeneralisee}
Soient $\mathdeuxcat{A}$ et $\mathdeuxcat{B}$ des \deux{}catégories, $u : \mathdeuxcat{A} \to \mathdeuxcat{B}$ un \DeuxFoncteurLax{}, $m \geq 1$ un entier, $x$ et $y$ des \DeuxFoncteursStricts{} de $[m]$ vers $\mathdeuxcat{A}$ et, pour tout entier $1 \leq i \leq m$, soit $\alpha_{i} : x_{i,i-1} \Rightarrow y_{i,i-1}$ une \deux{}cellule de $\mathdeuxcat{A}$. Alors, l'égalité 
$$
u(\alpha_{m} \CompDeuxZero \dots \CompDeuxZero \alpha_{1}) \CompDeuxUn u_{x} = u_{y} \CompDeuxUn (u(\alpha_{m}) \CompDeuxZero \dots \CompDeuxZero u(\alpha_{1}))
$$
s'établit par récurrence sur $m$, le cas $m=2$ faisant partie de l'axiomatique des \DeuxFoncteursLax{}.
\end{rem}

\begin{rem}\label{NaturalitePlusGenerale}
Soient $\mathdeuxcat{A}$ et $\mathdeuxcat{B}$ des \deux{}catégories, $u : \mathdeuxcat{A} \to \mathdeuxcat{B}$ un \DeuxFoncteurLax{}, $m \geq 0$ et $n \geq 1$ deux entiers, $x : [m] \to \mathdeuxcat{A}$ et $y : [n] \to \mathdeuxcat{A}$ des \DeuxFoncteursStricts{}, $\varphi : [n] \to [m]$ un morphisme d'intervalles tel que $x_{\varphi(i)} = y_{i}$ pour tout entier $0 \leq i \leq n$ et, pour tout entier $1 \leq i \leq n$, soit une \deux{}cellule $\alpha_{i} : x_{\varphi(i), \varphi(i-1)} \Rightarrow y_{i, i-1}$. Alors, l'égalité
$$
u(\alpha_{n} \CompDeuxZero \dots \CompDeuxZero \alpha_{1}) \CompDeuxUn u_{x} = u_{y} \CompDeuxUn (u(\alpha_{n}) \CompDeuxZero \dots \CompDeuxZero u(\alpha_{1})) \CompDeuxUn (u_{x_{\varphi(n-1)} \to \dots \to x_{m}} \CompDeuxZero \dots \CompDeuxZero u_{x_{0} \to \dots \to x_{\varphi(1)}})
$$
s'établit en utilisant la condition de cocycle généralisée et la remarque \ref{NaturaliteStructuraleGeneralisee}, qu'elle généralise. On laisse au lecteur le soin d'examiner ce qui se passe pour $n=0$ (ce qui force $m=0$).
\end{rem}

\begin{paragr}
Pour toute \deux{}catégorie $\mathdeuxcat{A}$ et tout couple d'objets $a$ et $a'$ de $\mathdeuxcat{A}$, on note $\Chaines{a'}{a}{[m]}$ la catégorie dont les objets sont les \DeuxFoncteursStricts{} $s : [m] \to \mathdeuxcat{A}$ tels que $s(0) = a$ et $s(m) = a'$ et dont les morphismes sont les \DeuxTransformationsLax{} relatives aux objets entre tels \DeuxFoncteursStricts{}. Tout morphisme d'intervalles $\varphi : [n] \to [m]$ induit un foncteur
$$
\begin{aligned}
\Chaines{a'}{a}{\varphi} : \Chaines{a'}{a}{[m]} &\to \Chaines{a'}{a}{[n]}
\\
(s : [m] \to \mathdeuxcat{A}) &\mapsto (s \varphi : [n] \to \mathdeuxcat{A})
\\
(\alpha : s \Rightarrow t) &\mapsto (\alpha \CompDeuxZero \varphi : s \varphi \Rightarrow t \varphi).
\end{aligned}
$$
La \DeuxTransformationLax{} relative aux objets $\alpha \CompDeuxZero \varphi$ est définie par
$(\alpha \CompDeuxZero \varphi)_{i} = 1_{s(\varphi(i))} = 1_{t(\varphi(i))}$ pour tout objet $i$ de $[n]$ et $(\alpha \CompDeuxZero \varphi)_{r} = \alpha_{\varphi(r)}$ pour toute \un{}cellule $r$ de $[n]$. Cela permet de définir un foncteur
$$
\begin{aligned}
\FoncteurChaines{a'}{a}\index[not]{Fa'a@$\FoncteurChaines{a'}{a}$} : \Intervalles^{op} &\to \Cat
\\
[m] &\mapsto \Chaines{a'}{a}{[m]}
\\
\varphi &\mapsto \Chaines{a'}{a}{\varphi}.
\end{aligned}
$$
\end{paragr}

\begin{paragr}\label{PreliminairesTilde2}
Soient $A_{1}$, $A_{2}$ et $A$ des catégories, $F_{1} : A_{1} \to \Cat$, $F_{2} : A_{2} \to \Cat$, $F : A \to \Cat$ et $\otimes : A_{1} \times A_{2} \to A$ des foncteurs et $\gamma : \times \circ (F_{1} \times F_{2}) \Rightarrow F \circ \otimes$ un morphisme de foncteurs. Ainsi, pour tout couple d'objets $(a_{1}, a_{2})$ de $A_{1} \times A_{2}$, on a le foncteur 
$$
\gamma_{(a_{1}, a_{2})} : F_{1}(a_{1}) \times F_{2}(a_{2}) \to F(a_{1} \otimes a_{2}).
$$
Ces données permettent de définir un foncteur
$$
\DeuxInt{A_{1} \times A_{2}}  \gamma : \DeuxInt{A_{1} \times A_{2}} \times \circ (F_{1} \times F_{2}) \to \DeuxInt{A_{1} \times A_{2}} F \circ \otimes.
$$
On laisse au lecteur le soin d'expliciter l'isomorphisme canonique évident
$$
\DeuxInt{A_{1}} F_{1} \times \DeuxInt{A_{2}} F_{2} \simeq \DeuxInt{A_{1} \times A_{2}} \times \circ (F_{1} \times F_{2}).
$$
On déduit de ce qui précède un foncteur
$$
\begin{aligned}
\DeuxInt{A_{1}} F_{1} \times \DeuxInt{A_{2}} F_{2} &\to \DeuxInt{A_{1} \times A_{2}} F \circ \otimes
\\
((a_{1}, x_{1}), (a_{2}, x_{2})) &\mapsto ((a_{1}, a_{2}), \gamma_{(a_{1}, a_{2})}(x_{1}, x_{2}))
\\
((f_{1} : a_{1} \to a'_{1}, r_{1}), (f_{2} : a_{2} \to a'_{2}, r_{2})) &\mapsto ((f_{1}, f_{2}), \gamma_{(a'_{1}, a'_{2})}(r_{1}, r_{2})).
\end{aligned}
$$

On a de plus un foncteur
$$
\begin{aligned}
\DeuxInt{A_{1} \times A_{2}} F \circ \otimes &\to \DeuxInt{A}F
\\
((a_{1}, a_{2}), x) &\mapsto ((a_{1} \otimes a_{2}), x)
\\
((f_{1}, f_{2}), r) &\mapsto ((f_{1} \otimes f_{2}), r)
\end{aligned}
$$
qui n'est autre — nous le signalons mais n'utiliserons pas ce fait — que la flèche horizontale du haut du carré cartésien
$$
\xymatrix{
\DeuxInt{A_{1} \times A_{2}} F \circ \otimes
\ar[r]
\ar[d]
&
\DeuxInt{A}F
\ar[d]
\\
A_{1} \times A_{2}
\ar[r]_{\otimes}
&
A
}
$$
dans lequel les flèches verticales désignent les projections canoniques (voir \cite[lemme 3.2.20]{THG}).

On déduit de ce qui précède un foncteur
$$
\begin{aligned}
\DeuxInt{A_{1}} F_{1} \times \DeuxInt{A_{2}} F_{2} &\to \DeuxInt{A}F
\\
((a_{1}, x_{1}), (a_{2}, x_{2})) &\mapsto (a_{1} \otimes a_{2}, \gamma_{(a_{1}, a_{2})} (x_{1}, x_{2}))
\\
((f_{1} : a_{1} \to a'_{1}, r_{1}), (f_{2} : a_{2} \to a'_{2}, r_{2})) &\mapsto (f_{1} \otimes f_{2}, \gamma_{(a'_{1}, a'_{2})} (r_{1}, r_{2})).
\end{aligned}
$$
\end{paragr}

\begin{paragr}
On observe que la catégorie des intervalles $\Intervalles$ est munie d'une structure de catégorie monoïdale stricte en définissant le produit tensoriel des objets par
$$
[n] \otimes [m] = [n + m]
$$
et celui des morphismes $\psi : [n] \to [n']$ et $\varphi : [m] \to [m']$ par
$$
\begin{aligned}
\psi \otimes \varphi : [n + m] &\to [n' + m']
\\
0 \leq k \leq m &\mapsto \varphi(k)
\\
m \leq k \leq n+m &\mapsto \psi(k-m) + m'.
\end{aligned}
$$
La catégorie $\UnCatOp{\Intervalles}$ se trouve donc aussi munie d'une structure de catégorie monoïdale stricte. Soient $a$, $a'$ et $a''$ trois objets d'une \deux{}catégorie $\mathdeuxcat{A}$. On applique les considérations du paragraphe \ref{PreliminairesTilde2} au cas $A_{1} = A_{2} = A = \UnCatOp{\Intervalles}$, $F_{1} = \FoncteurChaines{a''}{a'}$, $F_{2} = \FoncteurChaines{a'}{a}$, $F = \FoncteurChaines{a''}{a}$, la composante du morphisme de foncteurs $\gamma$ en $([n],[m])$ n'étant autre (en omettant quelques détails) que
$$
\begin{aligned}
\gamma^{a'',a',a}_{[n],[m]} : \Chaines{a''}{a'}{[n]} \times \Chaines{a'}{a}{[m]} &\to \Chaines{a''}{a}{[n+m]}
\\
(t,s) &\mapsto (0 \leq i \leq m \mapsto s(i), m \leq i \leq m+n \mapsto t(i-m)).
\end{aligned}
$$
 
Pour tout couple d'objets $(a,a')$ d'une petite \deux{}catégorie $\mathdeuxcat{A}$, on définit une catégorie $\CatHom{\TildeLax{\mathdeuxcat{A}}}{a}{a'}$\index[not]{HomATildeaa'@$\CatHom{\TildeLax{\mathdeuxcat{A}}}{a}{a'}$} par
$$
\CatHom{\TildeLax{\mathdeuxcat{A}}}{a}{a'} = \DeuxInt{\UnCatOp{\Intervalles}}\FoncteurChaines{a'}{a}.
$$
(C'est donc une catégorie opfibrée sur $\UnCatOp{\Intervalles}$.)  
 
Ce qui précède nous permet donc d'obtenir un foncteur de composition
$$
\begin{aligned}
\CatHom{\TildeLax{\mathdeuxcat{A}}}{a'}{a''} \times \CatHom{\TildeLax{\mathdeuxcat{A}}}{a}{a'} &\to \CatHom{\TildeLax{\mathdeuxcat{A}}}{a}{a''}
\\
(([m'], x'), ([m], x)) &\mapsto ([m' + m], (x', x))
\\
((\varphi', \alpha'), (\varphi, \alpha)) &\mapsto (\varphi' \otimes \varphi, (\alpha', \alpha))
\end{aligned}
$$
en notant désormais, pour simplifier, $(x',x)$ le \DeuxFoncteurStrict{} correspondant à la « chaîne formée de $x$ suivi de $x'$ » et de même pour $(\alpha', \alpha)$. 

La contrainte d'associativité résulte de l'associativité du produit tensoriel considéré sur la catégorie $\Intervalles$ et de la commutativité, pour toutes les valeurs possibles des paramètres y figurant, du diagramme
$$
\xymatrix{
\Chaines{a'''}{a''}{[p]} \times \Chaines{a''}{a'}{[n]} \times \Chaines{a'}{a}{[m]}
\ar[rrrr]^{\gamma^{a''',a'',a'}_{([p], [n])} \times 1_{\Chaines{a'}{a}{[m]}}}
\ar[ddd]_{1_{\Chaines{a'''}{a''}{[p]}} \times \gamma^{a'',a',a}_{([n], [m])}}
&&&&
\Chaines{a'''}{a'}{[p] \otimes [n]} \times \Chaines{a'}{a}{[m]}
\ar[dd]^{\gamma^{a''',a',a}_{([p] \otimes [n], [m])}}
\\
\\
&&&&
\Chaines{a'''}{a}{([p] \otimes [n]) \otimes [m]}
\ar@{=}[d]
\\
\Chaines{a'''}{a''}{[p]} \times \Chaines{a''}{a}{[n] \otimes [m]}
\ar[rrrr]_{\gamma^{a''',a'',a}_{([p], [n] \otimes [m])}}
&&&&
\Chaines{a'''}{a}{[p] \otimes ([n] \otimes [m])}
&.
}
$$

De plus, pour tout objet $a$ de $\mathdeuxcat{A}$, l'objet $([0], a : [0] \to \mathdeuxcat{A})$ de $\DeuxInt{\UnCatOp{\Intervalles}} \FoncteurChaines{a}{a} = \CatHom{\TildeLax{\mathdeuxcat{A}}}{a}{a}$ définit un foncteur d'unité
$$
\UnCatPonct \to \CatHom{\TildeLax{\mathdeuxcat{A}}}{a}{a}.
$$
Les contraintes d'unité résultent du fait que $[0]$ est un objet unité de la catégorie monoïdale $\Intervalles$ et de la commutativité, pour toutes les valeurs possibles des paramètres y figurant, des diagrammes
$$
\xymatrix{
\Chaines{a'}{a}{[m]} \times \UnCatPonct
\ar@{=}[rr]
\ar[dd]_{1_{\Chaines{a'}{a}{[m]}} \times ([0], a)}
&&
\Chaines{a'}{a}{[m]} 
\ar@{=}[dd]
\\
\\
\Chaines{a'}{a}{[m]} \times \Chaines{a}{a}{[0]} 
\ar[rr]_{\gamma^{a',a,a}_{([m], [0])}}
&&
\Chaines{a'}{a}{[m] \otimes [0]}
}
$$
et
$$
\xymatrix{
\UnCatPonct \times \Chaines{a'}{a}{[m]}
\ar@{=}[rr]
\ar[dd]_{([0],a') \times 1_{\Chaines{a'}{a}{[m]}}}
&&
\Chaines{a'}{a}{[m]} 
\ar@{=}[dd]
\\
\\
\Chaines{a'}{a'}{[0]} \times \Chaines{a'}{a}{[m]}  
\ar[rr]_{\gamma^{a',a',a}_{([0], [m])}}
&&
\Chaines{a'}{a}{[0] \otimes [m]}
}
$$
dans lesquels les flèches doubles désignent les isomorphismes canoniques. 
\end{paragr}

En vertu de ce qui précède, la définition \ref{DefTildeLax} fait sens. 

\begin{df}\label{DefTildeLax}
Soit $\mathdeuxcat{A}$ une petite \deux{}catégorie. On définit comme suit une petite \deux{}catégorie $\TildeLax{\mathdeuxcat{A}}$\index[not]{AZzz@$\TildeLax{\mathdeuxcat{A}}$}, que l'on appellera parfois la \emph{\deux{}catégorie tilde lax\index{tilde lax (d'une petite \deux{}catégorie)} de $\mathdeuxcat{A}$}, ou plus simplement le \emph{tilde lax de $\mathdeuxcat{A}$}. 
\begin{itemize}
\item
Les objets de $\TildeLax{\mathdeuxcat{A}}$ sont les objets de $\mathdeuxcat{A}$. 
\item
Pour tout couple d'objets $a$ et $a'$ de $\mathdeuxcat{A}$, 
$$
\CatHom{\TildeLax{\mathdeuxcat{A}}}{a}{a'} = \DeuxInt{\UnCatOp{\Intervalles}}\FoncteurChaines{a'}{a}.
$$
\item
Pour tout triplet d'objets $a$, $a'$ et $a''$ de $\mathdeuxcat{A}$, le foncteur de composition
$$
\CatHom{\TildeLax{\mathdeuxcat{A}}}{a'}{a''} \times \CatHom{\TildeLax{\mathdeuxcat{A}}}{a}{a'} \to \CatHom{\TildeLax{\mathdeuxcat{A}}}{a}{a''}.
$$
est défini comme ci-dessus.
\item
Pout tout objet $a$ de $\mathdeuxcat{A}$, le foncteur d'unité 
$$
\UnCatPonct \to \CatHom{\TildeLax{\mathdeuxcat{A}}}{a}{a}
$$
est défini comme ci-dessus par l'objet $([0], a : [0] \to \mathdeuxcat{A})$ de $\CatHom{\TildeLax{\mathdeuxcat{A}}}{a}{a}$.  
\end{itemize} 
\end{df}

\begin{paragr}
Ainsi, les \un{}cellules de $\TildeLax{\mathdeuxcat{A}}$ sont de la forme $([m], x)$, avec $m \geq 0$ un entier et $x : [m] \to \mathdeuxcat{A}$ un \DeuxFoncteurStrict{}. La source (\emph{resp.} le but) de la \un{}cellule $([m], x)$ est $x_{0}$ (\emph{resp.} $x_{m}$). On notera parfois les \un{}cellules de $\TildeLax{\mathdeuxcat{A}}$ sous la forme plus explicite $([m], x_{1,0}, \dots, x_{m,m-1})$, ce qui ne sous-entend pas $m \geq 1$. Les \deux{}cellules de $([m], x)$ vers $([n], y)$ sont de la forme $(\varphi, \alpha)$, avec $\varphi : [n] \to [m]$ un morphisme d'intervalles vérifiant $x_{\varphi(i)} = y_{i}$ pour tout entier $0 \leq i \leq n$ (on doit donc avoir $x_{0} = y_{0}$ et $x_{m} = y_{n}$) et $\alpha : x \varphi \Rightarrow y$ une \DeuxTransformationLax{} relative aux objets. On notera parfois $\alpha$ sous la forme $(\alpha_{1}, \dots, \alpha_{n})$, étant entendu que la \deux{}cellule $\alpha_{i} : (x \varphi)_{i,i-1} \Rightarrow y_{i,i-1}$ est la composante de $\alpha$ en la \un{}cellule $i-1 \to i$ de $[n]$, autrement notée $\alpha_{i,i-1}$. On notera donc parfois la \deux{}cellule $(\varphi, \alpha)$ sous la forme $(\varphi, \alpha_{1}, \dots, \alpha_{n})$, ce qui ne sous-entend pas $n \geq 1$.   
\end{paragr}

\begin{rem}
On rappelle avoir défini le foncteur $\TronqueBete{} : \DeuxCat \to \Cat$ dans la section \ref{SectionAdjonctionsCatDeuxCat}. Pour toute petite \deux{}catégorie $\mathdeuxcat{A}$, la catégorie $\TronqueBete{\TildeLax{\mathdeuxcat{A}}}$ n'est rien d'autre que la catégorie libre engendrée par le graphe sous-jacent à la catégorie $\TronqueBete{\mathdeuxcat{A}}$.
\end{rem}

\begin{df}\label{DefTildeLaxFoncteur}
Pour tout \DeuxFoncteurLax{} $u : \mathdeuxcat{A} \to \mathdeuxcat{B}$, on définit un \DeuxFoncteurStrict{}, le \emph{tilde lax de $u$}\index{tilde lax (d'un \DeuxFoncteurLax{})}, par
$$
\begin{aligned}
\TildeLax{u}\index[not]{uZzz@$\TildeLax{u}$} : \TildeLax{\mathdeuxcat{A}} &\to \TildeLax{\mathdeuxcat{B}} 
\\
a &\mapsto u(a)
\\
([0],a) &\mapsto ([0], u(a))
\\
(([m], x), m \geq 1) &\mapsto ([m], u(x_{1,0}), \dots, u(x_{m,m-1}))
\\
((\varphi, \alpha) : ([m], x) \to ([n], y)) &\mapsto (\varphi, (u(\alpha_{i}) u_{x_{\varphi(i-1)} \to \dots \to x_{\varphi(i)}})_{1 \leq i \leq n}).
\end{aligned}
$$
\end{df}

\begin{paragr}
Donnons tout de même quelques détails assurant de la correction de la définition \ref{DefTildeLaxFoncteur} ; autrement dit, que l'on y définit bien un \DeuxFoncteurStrict{}. 

Pour tout objet $a$ de $\TildeLax{\mathdeuxcat{A}}$,
$$
\begin{aligned}
\TildeLax{u} (1_{a}) &= \TildeLax{u} ([0], a)
\\
&= ([0], u(a))
\\
&= 1_{u(a)}
\\
&= 1_{\TildeLax{u}(a)}.
\end{aligned}
$$

Soient $([m], x : [m] \to \mathdeuxcat{A})$ et $([m'], x' : [m'] \to \mathdeuxcat{A})$ deux \un{}cellules de $\TildeLax{\mathdeuxcat{A}}$ telles que la composée $([m'], x') ([m], x)$ fasse sens. Alors, 
$$
\begin{aligned}
\TildeLax{u} (([m'], x') ([m], x)) &= \TildeLax{u} ([m'+m], (x', x)) 
\\
&= ([m' + m], u(x_{1,0}), \dots, u(x_{m,m-1}), u(x'_{1,0}), \dots, u(x'_{m',m'-1}))
\\
&= ([m'], u(x'_{1,0}), \dots, u(x'_{m',m'-1})) ([m], u(x_{1,0}), \dots, u(x_{m,m-1}))
\\
&= \TildeLax{u}([m'],x') \TildeLax{u}([m],x),
\end{aligned}
$$
en notant toujours $(x',x)$ le \DeuxFoncteurStrict{} correspondant à la « chaîne formée de $x$ suivi de $x'$ ». 

Soient $(\varphi, \alpha)$ et $(\varphi', \alpha')$ des \deux{}cellules de $\TildeLax{\mathdeuxcat{A}}$ telles que la composée $(\varphi', \alpha') \CompDeuxZero (\varphi, \alpha)$ fasse sens. L'égalité
$$
\TildeLax{u} ((\varphi', \alpha') \CompDeuxZero (\varphi, \alpha)) = \TildeLax{u} (\varphi', \alpha') \CompDeuxZero \TildeLax{u} (\varphi, \alpha)
$$
est alors évidente. 

Soit $([m],x)$ une \un{}cellule de $\TildeLax{\mathdeuxcat{A}}$. Alors, 
$$
\begin{aligned}
\TildeLax{u} (1_{([m],x)}) &= \TildeLax{u} (1_{[m]}, (1_{x_{1,0}}, \dots, 1_{x_{m,m-1}}))
\\
&= (1_{[m]}, (u(1_{x_{1,0}}), \dots, u(1_{x_{m,m-1}})))
\\
&= (1_{[m]}, (1_{u(x_{1,0})}, \dots, 1_{u(x_{m,m-1})}))
\\
&= 1_{([m], u(x_{1,0}), \dots, u(x_{m,m-1}))}
\\
&= 1_{\TildeLax{u} ([m], x)}.
\end{aligned}
$$

Enfin, soient 
$$
(\varphi, \alpha) : ([m], x) \Rightarrow ([n], y)
$$
et 
$$
(\psi, \beta) : ([n], y) \Rightarrow ([p], z) 
$$
deux \deux{}cellules de $\TildeLax{\mathdeuxcat{A}}$. Alors, l'égalité
$$
\TildeLax{u} ((\psi, \beta) \CompDeuxUn (\varphi, \alpha)) = \TildeLax{u} (\psi, \beta) \TildeLax{u} (\varphi, \alpha)
$$
résulte de la remarque \ref{NaturalitePlusGenerale}. 
\end{paragr}

\begin{lemme}\label{TildeLaxFoncteur}
L'application
$$
\begin{aligned}
\DeuxCatLax{} &\to \DeuxCat
\\
\mathdeuxcat{A} &\mapsto \TildeLax{\mathdeuxcat{A}}
\\
u &\mapsto \TildeLax{u}
\end{aligned}
$$
définit un foncteur de $\DeuxCatLax$ vers $\DeuxCat$.
\end{lemme}

\begin{proof}
Il s'agit de vérifier que, pour toute petite \deux{}catégorie $\mathdeuxcat{A}$, on a $\TildeLax{1_{\mathdeuxcat{A}}} = 1_{\TildeLax{\mathdeuxcat{A}}}$ et que, pour tout couple de \DeuxFoncteursLax{} $u$ et $v$ tels que la composée $vu$ fasse sens, on a $\TildeLax{vu} = \TildeLax{v} \TildeLax{u}$. Soient donc $a$, $([m], x)$ et $(\varphi : [n] \to [m], \alpha_{1}, \dots, \alpha_{n})$ un objet, une \un{}cellule et une \deux{}cellule de $([m],x)$ vers $([n],y)$ de $\TildeLax{\mathdeuxcat{A}}$ respectivement. Alors,
$$
\begin{aligned}
\TildeLax{1_{\mathdeuxcat{A}}} (a) &= 1_{\mathdeuxcat{A}}(a) 
\\
&=a,
\end{aligned}
$$
$$
\begin{aligned}
\TildeLax{1_{\mathdeuxcat{A}}} ([m], x_{1,0}, \dots, x_{m,m-1}) &= ([m], 1_{\mathdeuxcat{A}}(x_{1,0}), \dots, 1_{\mathdeuxcat{A}}(x_{m,m-1}))
\\
&= ([m], x_{1,0}, \dots, x_{m,m-1})
\end{aligned}
$$
et
$$
\begin{aligned}
\TildeLax{1_{\mathdeuxcat{A}}} (\varphi, \alpha_{1}, \dots, \alpha_{n}) &= (\varphi, 1_{\mathdeuxcat{A}} (\alpha_{1}), \dots, 1_{\mathdeuxcat{A}} (\alpha_{n}))
\\
&= (\varphi, \alpha_{1}, \dots, \alpha_{n}).
\end{aligned}
$$
L'égalité  $\TildeLax{1_{\mathdeuxcat{A}}} = 1_{\TildeLax{\mathdeuxcat{A}}}$ est donc bien vérifiée. En gardant les notations ci-dessus, supposons que la source de $u$ soit $\mathdeuxcat{A}$. Alors, 
$$
\begin{aligned}
\TildeLax{vu} (a) &= vu (a) 
\\
&= v(u(a))
\\
&= v(\TildeLax{u}(a))
\\
&= \TildeLax{v}(\TildeLax{u}(a))
\\
&= \TildeLax{v} \TildeLax{u} (a)
\end{aligned}
$$
et
$$
\begin{aligned}
\TildeLax{vu} ([m], x_{1,0}, \dots, x_{m,m-1}) &= ([m], vu (x_{1,0}), \dots, vu (x_{m,m-1}))
\\
&= ([m], v (u (x_{1,0})), \dots, v (u (x_{m,m-1})))
\\
&= \TildeLax{v} ([m], u(x_{1,0}), \dots, u(x_{m,m-1}))
\\
&= \TildeLax{v} (\TildeLax{u} ([m], x_{1,0}, \dots, x_{m,m-1}))
\\
&= \TildeLax{v} \TildeLax{u} ([m], x_{1,0}, \dots, x_{m,m-1}).
\end{aligned}
$$
De plus, 
$$
\begin{aligned}
\TildeLax{vu} (\varphi, \alpha_{1}, \dots, \alpha_{n}) &= (\varphi, vu (\alpha_{1}) \CompDeuxUn (vu)_{x_{0} \to \dots \to x_{\varphi(1)}}, \dots, vu (\alpha_{n}) \CompDeuxUn (vu)_{x_{\varphi(n-1)} \to \dots \to x_{m}})
\end{aligned}
$$
et
$$
\begin{aligned}
\TildeLax{v} \TildeLax{u} (\varphi, \alpha_{1}, \dots, \alpha_{n}) &= \TildeLax{v} (\varphi, u(\alpha_{1}) \CompDeuxUn u_{x_{0} \to \dots \to x_{\varphi(1)}}, \dots, u(\alpha_{n}) \CompDeuxUn u_{x_{\varphi(n-1)} \to \dots \to x_{m}})
\\
&= (\varphi, v(u(\alpha_{1}) \CompDeuxUn u_{x_{0} \to \dots \to x_{\varphi(1)}}) \CompDeuxUn v_{u(x_{0}) \to \dots \to u(x_{\varphi(1)})}, \dots)
\\
&= (\varphi, v(u(\alpha_{1})) \CompDeuxUn v(u_{x_{0} \to \dots \to x_{\varphi(1)}}) \CompDeuxUn v_{u(x_{0}) \to \dots \to u(x_{\varphi(1)})}, \dots).
\end{aligned}
$$
Le fait que $\TildeLax{vu}$ et $\TildeLax{v} \TildeLax{u}$ coïncident sur les \deux{}cellules résulte donc de l'égalité
$$
(vu)_{x_{\varphi(i-1)} \to \dots \to x_{\varphi(i)}} = v(u_{x_{\varphi(i-1)} \to \dots \to x_{\varphi(i)}}) \CompDeuxUn v_{u(x_{\varphi(i-1)}) \to \dots \to u(x_{\varphi(i)})}
$$
pour tout $1 \leq i \leq n$. 
\end{proof}

\begin{df}
On appellera \emph{foncteur de strictification de Bénabou}\index{foncteur de strictification de Bénabou} le foncteur apparaissant dans l'énoncé du lemme \ref{TildeLaxFoncteur}. On le notera désormais 
$$
\begin{aligned}
B\index[not]{B@$B$} : \DeuxCatLax{} &\to \DeuxCat
\\
\mathdeuxcat{A} &\mapsto \TildeLax{\mathdeuxcat{A}}
\\
u &\mapsto \TildeLax{u}.
\end{aligned}
$$
\end{df}

\begin{paragr}
Pour toute \un{}cellule $f$ de $\mathdeuxcat{A}$, on notera $([1], f)$\index[not]{(1[]f)@$([1], f)$} la \un{}cellule de $\TildeLax{\mathdeuxcat{A}}$ définie par $(0 \to 1) \mapsto f$. Ainsi, on note de la même façon les \un{}cellules de $\mathdeuxcat{A}$ et les \DeuxFoncteursStricts{} de $[1]$ vers $\mathdeuxcat{A}$ qui leur sont canoniquement associés. 
\end{paragr}

\begin{df}\label{DefUniteLax}
Soit $\mathdeuxcat{A}$ une petite \deux{}catégorie. On appellera \emph{\DeuxFoncteurLax{} structural associé à $\mathdeuxcat{A}$} le \DeuxFoncteurLax
$$
\LaxCanonique{\mathdeuxcat{A}}\index[not]{0EtaA@$\LaxCanonique{\mathdeuxcat{A}}$} : \mathdeuxcat{A} \to \TildeLax{\mathdeuxcat{A}}
$$
défini comme suit.

Pour tout objet $a$ de $\mathdeuxcat{A}$, 
$$
\LaxCanonique{\mathdeuxcat{A}}(a) = a.
$$
En particulier, $1_{\LaxCanonique{\mathdeuxcat{A}}(a)} = ([0], a)$.

Pour toute \un{}cellule $f$ de $\mathdeuxcat{A}$,
$$
\LaxCanonique{\mathdeuxcat{A}}(f) = ([1], f).
$$
En particulier, $\LaxCanonique{\mathdeuxcat{A}}(1_{a}) = ([1], 1_{a})$.

Pour toute \deux{}cellule $\alpha$ de $\mathdeuxcat{A}$, 
$$
\LaxCanonique{\mathdeuxcat{A}} (\alpha) = (1_{[1]}, \alpha).
$$
 
Pour tout objet $a$ de $\mathdeuxcat{A}$, la \deux{}cellule structurale 
$$
{(\LaxCanonique{\mathdeuxcat{A}})}_{a} : 1_{\LaxCanonique{\mathdeuxcat{A}}(a)}  \Rightarrow \LaxCanonique{\mathdeuxcat{A}}(1_{a})
$$
est définie par
$$
{(\LaxCanonique{\mathdeuxcat{A}})}_{a} = ([1] \to [0], 1_{1_{a}}).
$$

Pour tout couple de \un{}cellules $f$ et $f'$ de $\mathdeuxcat{A}$ telles que la composée $f'f$ fasse sens, la \deux{}cellule structurale
$$
(\LaxCanonique{\mathdeuxcat{A}})_{f',f} : \LaxCanonique{\mathdeuxcat{A}}(f') \LaxCanonique{\mathdeuxcat{A}}(f) \Rightarrow \LaxCanonique{\mathdeuxcat{A}}(f'f)
$$
est définie par 
$$
(\LaxCanonique{\mathdeuxcat{A}})_{f',f}=([1] \to [2], 1_{f'f}),
$$
le morphisme d'ensembles ordonnés $[1] \to [2]$ présent dans cette expression étant défini de façon unique par la condition qu'il s'agit d'un morphisme d'intervalles. 
\end{df}

\begin{paragr}\label{VerificationLaxCanonique}
Vérifions que cela définit bien un \DeuxFoncteurLax{}. Il s'agit de s'assurer des conditions de cohérence suivantes. 

Pour tout triplet $h$, $g$ et $f$ de \un{}cellules de $\mathdeuxcat{A}$ telles que la composée $hgf$ fasse sens, on doit avoir commutativité du diagramme 
$$
\xymatrix{
\LaxCanonique{\mathdeuxcat{A}}(h) \LaxCanonique{\mathdeuxcat{A}}(g) \LaxCanonique{\mathdeuxcat{A}}(f)
\ar@{=>}[rr]^{{(\LaxCanonique{\mathdeuxcat{A}}})_{h,g} \CompDeuxZero \LaxCanonique{\mathdeuxcat{A}}(f)}
\ar@{=>}[d]_{\LaxCanonique{\mathdeuxcat{A}}(h) \CompDeuxZero {(\LaxCanonique{\mathdeuxcat{A}})}_{g,f}}
&&\LaxCanonique{\mathdeuxcat{A}}(hg) \LaxCanonique{\mathdeuxcat{A}}(f)
\ar@{=>}[d]^{{(\LaxCanonique{\mathdeuxcat{A}})}_{hg,f}}
\\
\LaxCanonique{\mathdeuxcat{A}}(h)\LaxCanonique{\mathdeuxcat{A}}(gf)
\ar@{=>}[rr]_{{(\LaxCanonique{\mathdeuxcat{A}})}_{h,gf}}
&&
\LaxCanonique{\mathdeuxcat{A}}(hgf)
&.
}
$$
On vérifie que les deux chemins menant de $\LaxCanonique{\mathdeuxcat{A}}(h) \LaxCanonique{\mathdeuxcat{A}}(g) \LaxCanonique{\mathdeuxcat{A}}(f)$ à $\LaxCanonique{\mathdeuxcat{A}}(hgf)$ correspondent à la même \deux{}cellule $([1] \to [3], 1_{hgf})$ dans $\TildeLax{\mathdeuxcat{A}}$.

Pour toute \un{}cellule $f : a \to a'$ de $\mathdeuxcat{A}$, on doit avoir commutativité des diagrammes
$$
\xymatrix{
\LaxCanonique{\mathdeuxcat{A}}(f)=\LaxCanonique{\mathdeuxcat{A}}(f)1_{\LaxCanonique{\mathdeuxcat{A}}(a)}
\ar@{=>}[rr]^{\LaxCanonique{\mathdeuxcat{A}}(f) \CompDeuxZero {(\LaxCanonique{\mathdeuxcat{A}})}_{a}}
\ar@{=}[drr]
&& \LaxCanonique{\mathdeuxcat{A}}(f) \LaxCanonique{\mathdeuxcat{A}}(1_{a})
\ar@{=>}[d]^{{(\LaxCanonique{\mathdeuxcat{A}})}_{f,1_{a}}}
\\
&& \LaxCanonique{\mathdeuxcat{A}}(f1_{a})=\LaxCanonique{\mathdeuxcat{A}}(f)
}
$$
et
$$
\xymatrix{
\LaxCanonique{\mathdeuxcat{A}}(f)=1_{\LaxCanonique{\mathdeuxcat{A}}(a')}\LaxCanonique{\mathdeuxcat{A}}(f)
\ar@{=>}[rr]^{{(\LaxCanonique{\mathdeuxcat{A}})}_{a'} \CompDeuxZero \LaxCanonique{\mathdeuxcat{A}}(f)}
\ar@{=}[drr]
&& \LaxCanonique{\mathdeuxcat{A}}(1_{a'}) \LaxCanonique{\mathdeuxcat{A}}(f)
\ar@{=>}[d]^{{(\LaxCanonique{\mathdeuxcat{A}})}_{1_{a'},f}}
\\
&& \LaxCanonique{\mathdeuxcat{A}}(1_{a'}f)=\LaxCanonique{\mathdeuxcat{A}}(f)
}
$$
respectivement. La définition de $\TildeLax{\mathdeuxcat{A}}$ implique bien ces égalités.

Pour tout couple de \deux{}cellules $\alpha : f \Rightarrow f'$ et $\beta : g \Rightarrow g'$ de $\mathdeuxcat{A}$ telles que la composée $\beta \CompDeuxZero \alpha$ fasse sens, on doit avoir commutativité du diagramme
$$
\xymatrix{
\LaxCanonique{\mathdeuxcat{A}}(g) \LaxCanonique{\mathdeuxcat{A}}(f)
\ar@{=>}[rr]^{{(\LaxCanonique{\mathdeuxcat{A}})}_{g,f}}
\ar@{=>}[d]_{\LaxCanonique{\mathdeuxcat{A}}(\beta) \CompDeuxZero \LaxCanonique{\mathdeuxcat{A}}(\alpha)}
&& 
\LaxCanonique{\mathdeuxcat{A}}(gf)
\ar@{=>}[d]^{\LaxCanonique{\mathdeuxcat{A}}(\beta \CompDeuxZero \alpha)}
\\
\LaxCanonique{\mathdeuxcat{A}}(g') \LaxCanonique{\mathdeuxcat{A}}(f')
\ar@{=>}[rr]_{{(\LaxCanonique{\mathdeuxcat{A}})}_{g'f'}}
&& \LaxCanonique{\mathdeuxcat{A}}(g'f')
&,
}
$$
commutativité qui découle également de la définition de $\TildeLax{\mathdeuxcat{A}}$, les deux chemins menant de $\LaxCanonique{\mathdeuxcat{A}}(g) \LaxCanonique{\mathdeuxcat{A}}(f)$ à $\LaxCanonique{\mathdeuxcat{A}}(g'f')$ dans le diagramme correspondant à la même \deux{}cellule $([1] \to [2], \beta \CompDeuxZero \alpha)$.

La flèche $\LaxCanonique{\mathdeuxcat{A}} : \mathdeuxcat{A} \to \TildeLax{\mathdeuxcat{A}}$ est donc bien un \DeuxFoncteurLax{}.
\end{paragr}

\begin{lemme}\label{UniteNaturelle}
Pour tout morphisme $u : \mathdeuxcat{A} \to \mathdeuxcat{B}$ de $\DeuxCatLax{}$, le diagramme
$$
\xymatrix{
\TildeLax{\mathdeuxcat{A}}
\ar[r]^{\TildeLax{u}}
&\TildeLax{\mathdeuxcat{B}}
\\
\mathdeuxcat{A}
\ar[u]^{\LaxCanonique{\mathdeuxcat{A}}}
\ar[r]_{u}
&\mathdeuxcat{B}
\ar[u]_{\LaxCanonique{\mathdeuxcat{B}}}
}
$$
est commutatif.
\end{lemme}

\begin{proof}
Pour tout objet $a$ de $\mathdeuxcat{A}$, 
$$
\begin{aligned}
\TildeLax{u} (\LaxCanonique{\mathdeuxcat{A}} (a)) &= \TildeLax{u} (a)
\\
&= u(a)
\\
&= \LaxCanonique{\mathdeuxcat{B}} (u(a)).
\end{aligned}
$$

Pour toute \un{}cellule $f$ de $\mathdeuxcat{A}$, 
$$
\begin{aligned}
\TildeLax{u} (\LaxCanonique{\mathdeuxcat{A}} (f)) &= \TildeLax{u} ([1], f)
\\
&= ([1], u(f))
\\
&=   \LaxCanonique{\mathdeuxcat{B}} (u(f)).
\end{aligned}
$$

Pour toute \deux{}cellule $\alpha$ de $\mathdeuxcat{A}$,
$$
\begin{aligned}
\TildeLax{u} (\LaxCanonique{\mathdeuxcat{A}} (\alpha)) &= \TildeLax{u} (1_{[1]}, \alpha)
\\
&= (1_{[1]}, u(\alpha))
\\
&= \LaxCanonique{\mathdeuxcat{B}} (u (\alpha)). 
\end{aligned}
$$

Pour tout objet $a$ de $\mathdeuxcat{A}$,  
$$
\begin{aligned}
\TransNatUnit{(\TildeLax{u} \LaxCanonique{\mathdeuxcat{A}})}{a} &= \TildeLax{u} ( \TransNatUnit{(\LaxCanonique{\mathdeuxcat{A}})}{a} ) \CompDeuxUn \TransNatUnit{\TildeLax{u}}{\LaxCanonique{\mathdeuxcat{A}}(a)}
\\ 
&= \TildeLax{u} ( \TransNatUnit{(\LaxCanonique{\mathdeuxcat{A}})}{a} ) \CompDeuxUn 1_{\TildeLax{u} \LaxCanonique{\mathdeuxcat{A}} (a)}
\\
&= \TildeLax{u} ( \TransNatUnit{(\LaxCanonique{\mathdeuxcat{A}})}{a} )
\\
&= \TildeLax{u} ([1] \to [0], 1_{1_{a}})
\\
&= ([1] \to [0], \TransNatUnit{u}{a})
\end{aligned}
$$
et
$$
\begin{aligned}
\TransNatUnit{(\LaxCanonique{\mathdeuxcat{B}} u)}{a} &= \LaxCanonique{\mathdeuxcat{B}} (\TransNatUnit{u}{a}) \CompDeuxUn \TransNatUnit{(\LaxCanonique{\mathdeuxcat{B}})}{u(a)}
\\
&= (1_{[1]}, \TransNatUnit{u}{a}) \CompDeuxUn ([1] \to [0], 1_{1_{u(a)}})
\\
&= ([1] \to [0], \TransNatUnit{u}{a}).
\end{aligned}
$$

Pour tout couple de \un{}cellules $f$ et $f'$ de $\mathdeuxcat{A}$ telles que la composée $f'f$ fasse sens, 
$$
\begin{aligned}
(\TildeLax{u} \LaxCanonique{\mathdeuxcat{A}})_{f',f} &= \TildeLax{u} ((\LaxCanonique{\mathdeuxcat{A}})_{f',f}) \CompDeuxUn \TildeLax{u}_{\LaxCanonique{\mathdeuxcat{A}} (f'), \LaxCanonique{\mathdeuxcat{A}} (f)}
\\
&= \TildeLax{u} ((\LaxCanonique{\mathdeuxcat{A}})_{f',f}) \CompDeuxUn 1_{\TildeLax{u} (\LaxCanonique{\mathdeuxcat{A}} (f')) \TildeLax{u} (\LaxCanonique{\mathdeuxcat{A}} (f))}
\\
&= \TildeLax{u} ((\LaxCanonique{\mathdeuxcat{A}})_{f',f})
\\
&= \TildeLax{u} ([1] \to [2], 1_{f'f})
\\
&= ([1] \to [2], u_{f',f})
\end{aligned}
$$
et
$$
\begin{aligned}
(\LaxCanonique{\mathdeuxcat{B}} u)_{f',f} &= \LaxCanonique{\mathdeuxcat{B}} (u_{f',f}) \CompDeuxUn (\LaxCanonique{\mathdeuxcat{B}})_{u(f'), u(f)}
\\
&= (1_{[1]}, u_{f',f}) \CompDeuxUn ([1] \to [2], 1_{u(f')u(f)})
\\
&= ([1] \to [2], u_{f',f}).
\end{aligned}
$$
\end{proof}

\begin{df}\label{DefCounite}
Soit $\mathdeuxcat{A}$ une petite \deux{}catégorie. On définit un \DeuxFoncteurStrict{} $\StrictCanonique{\mathdeuxcat{A}}\index[not]{0EpsilonA@$\StrictCanonique{\mathdeuxcat{A}}$} : \TildeLax{\mathdeuxcat{A}} \to \mathdeuxcat{A}$, que l'on appellera parfois \deux{}\emph{foncteur strict structural associé à} $\mathdeuxcat{A}$, par
$$
\begin{aligned}
\StrictCanonique{\mathdeuxcat{A}} : \TildeLax{\mathdeuxcat{A}} &\to \mathdeuxcat{A}
\\
a &\mapsto a
\\
([m], x) &\mapsto x_{m,m-1} \dots x_{1,0}
\\
(\varphi, \alpha) &\mapsto \alpha_{n} \CompDeuxZero \dots \CompDeuxZero \alpha_{1}.
\end{aligned}
$$
\end{df}

\begin{paragr}
Vérifions que cela définit un \DeuxFoncteurStrict{}. 

Pour toute \un{}cellule $([m], x)$ de $\TildeLax{\mathdeuxcat{A}}$, 
$$
\begin{aligned}
\StrictCanonique{\mathdeuxcat{A}} (1_{([m], x)}) &= \StrictCanonique{\mathdeuxcat{A}} (1_{[m]}, (1_{x_{1,0}}, \dots, 1_{x_{m, m-1}}))
\\
&= 1_{x_{m, m-1}} \CompDeuxZero \dots \CompDeuxZero 1_{x_{1, 0}}
\\
&= 1_{x_{m,m-1} \dots x_{1,0}}
\\
&= 1_{\StrictCanonique{\mathdeuxcat{A}} ([m], x)}.
\end{aligned}
$$

Pour tout couple de \deux{}cellules $(\varphi, \alpha) : ([m],x) \to ([n],y)$ et $(\psi, \beta) : ([n],y) \to ([p],z)$ de $\TildeLax{\mathdeuxcat{A}}$ telles que la composée $(\psi, \beta) \CompDeuxUn (\varphi, \alpha)$ fasse sens, 
$$
\begin{aligned}
\StrictCanonique{\mathdeuxcat{A}} ((\psi, \beta) \CompDeuxUn (\varphi, \alpha)) &= \StrictCanonique{\mathdeuxcat{A}} (\varphi \psi, (\beta_{p} \CompDeuxUn (\alpha_{\psi(p)} \CompDeuxZero \dots \CompDeuxZero \alpha_{\psi(p-1) + 1})), \dots, (\beta_{1} \CompDeuxUn (\alpha_{\psi(1)} \CompDeuxZero \dots \CompDeuxZero \alpha_{1})))
\\
&= (\beta_{p} (\alpha_{n} \CompDeuxZero \dots \CompDeuxZero \alpha_{\psi(p-1)+1})) \CompDeuxZero \dots \CompDeuxZero (\beta_{1} (\alpha_{\psi(1)} \CompDeuxZero \dots \CompDeuxZero \alpha_{1}))
\\
&= (\beta_{p} \CompDeuxZero \dots \CompDeuxZero \beta_{1}) \CompDeuxUn (\alpha_{n} \CompDeuxZero \dots \CompDeuxZero \alpha_{1})
\\
&= \StrictCanonique{\mathdeuxcat{A}} (\psi, \beta) \StrictCanonique{\mathdeuxcat{A}} (\varphi, \alpha).
\end{aligned}
$$

Soient $([m], x)$ et $([m'], x')$ deux \un{}cellules de $\TildeLax{\mathdeuxcat{A}}$ telles que la composée $([m'], x') ([m], x)$ fasse sens. Alors, 
$$
\begin{aligned}
\StrictCanonique{\mathdeuxcat{A}} (([m'], x') ([m], x)) &= \StrictCanonique{\mathdeuxcat{A}} ([m'+m], (x',x))
\\
&= x'_{m', m'-1} \dots x'_{1,0} x_{m, m-1} \dots x_{1,0}
\\
&= \StrictCanonique{\mathdeuxcat{A}} ([m'], x') \StrictCanonique{\mathdeuxcat{A}} ([m], x).
\end{aligned}
$$ 

Soient $(\varphi, \alpha) = (\varphi, \alpha_{1}, \dots, \alpha_{n})$ et $(\varphi', \alpha') = (\varphi', \alpha'_{1}, \dots, \alpha'_{n'})$ deux \deux{}cellules de $\TildeLax{\mathdeuxcat{A}}$ telles que la composée 
$$
(\varphi', \alpha') \CompDeuxZero (\varphi, \alpha)
$$
fasse sens. Alors, 
$$
\begin{aligned}
\StrictCanonique{\mathdeuxcat{A}} (\varphi', \alpha') \CompDeuxZero \StrictCanonique{\mathdeuxcat{A}} (\varphi, \alpha) &= \alpha'_{n'} \CompDeuxZero \dots \alpha'_{1} \CompDeuxZero \alpha_{n} \CompDeuxZero \dots \CompDeuxZero \alpha_{1}
\\
&= \StrictCanonique{\mathdeuxcat{A}} ((\varphi', \alpha') \CompDeuxZero (\varphi, \alpha)).
\end{aligned}
$$
\end{paragr}

\begin{lemme}\label{CouniteNaturelle}
Pour tout morphisme $u : \mathdeuxcat{A} \to \mathdeuxcat{B}$ de $\DeuxCat$, le diagramme
$$
\xymatrix{
\TildeLax{\mathdeuxcat{A}}
\ar[r]^{\TildeLax{u}}
\ar[d]_{\StrictCanonique{\mathdeuxcat{A}}}
&\TildeLax{\mathdeuxcat{B}}
\ar[d]^{\StrictCanonique{\mathdeuxcat{B}}}
\\
\mathdeuxcat{A}
\ar[r]_{u}
&\mathdeuxcat{B}
}
$$
est commutatif.
\end{lemme}

\begin{proof}
Toutes les flèches apparaissant dans ce diagramme étant des \DeuxFoncteursStricts{}, il suffit de vérifier que $u \StrictCanonique{\mathdeuxcat{A}}$ et $\StrictCanonique{\mathdeuxcat{B}} \TildeLax{u}$ coïncident sur les objets, \un{}cellules et \deux{}cellules. 

Pour tout objet $a$ de $\TildeLax{\mathdeuxcat{A}}$, 
$$
\begin{aligned}
u \StrictCanonique{\mathdeuxcat{A}} (a) &= u(a) 
\\
&= \StrictCanonique{\mathdeuxcat{B}} (u(a)) 
\\
&= \StrictCanonique{\mathdeuxcat{B}} \TildeLax{u} (a).
\end{aligned}
$$

Pour toute \un{}cellule $([m], x)$ de $\TildeLax{\mathdeuxcat{A}}$, 
$$
\begin{aligned}
u \StrictCanonique{\mathdeuxcat{A}} ([m], x) &= u(x_{m,m-1} \dots x_{1,0}) 
\\
&= u(x_{m,m-1}) \dots u(x_{1,0})
\\
&= \StrictCanonique{\mathdeuxcat{B}} ([m], u(x_{1,0}), \dots, u(x_{m,m-1}))
\\
&= \StrictCanonique{\mathdeuxcat{B}} \TildeLax{u} ([m], x_{1,0}, \dots, x_{m,m-1}). 
\end{aligned}
$$

Pour toute \deux{}cellule $(\varphi : [n] \to [m], \alpha_{1}, \dots, \alpha_{n})$ de $\TildeLax{\mathdeuxcat{A}}$,
$$
\begin{aligned}
u \StrictCanonique{\mathdeuxcat{A}} (\varphi : [n] \to [m], \alpha_{1}, \dots, \alpha_{n}) &= u(\alpha_{n} \circ \dots \circ \alpha_{1})
\\
&= u(\alpha_{n}) \circ \dots \circ u(\alpha_{1})
\\
&= \StrictCanonique{\mathdeuxcat{B}} (\varphi, u(\alpha_{1}), \dots, u(\alpha_{n}))
\\
&= \StrictCanonique{\mathdeuxcat{B}} \TildeLax{u} (\varphi, \alpha_{1}, \dots, \alpha_{n}).
\end{aligned}
$$  
\end{proof}

\begin{paragr}
Notons $I\index[not]{I@$I$} : \DeuxCat \to \DeuxCatLax$ l'inclusion canonique. Les lemmes \ref{UniteNaturelle} et \ref{CouniteNaturelle} permettent d'affirmer que l'on a défini des transformations naturelles 
$$
\TransLaxCanonique\index[not]{0Eta@$\TransLaxCanonique$} : 1_{\DeuxCatLax} \Rightarrow IB
$$
et
$$
\TransStrictCanonique\index[not]{0Epsilon@$\TransStrictCanonique$} : BI \Rightarrow 1_{\DeuxCat}.
$$
\end{paragr}

\begin{theo}\label{BIAdjonction}
Le foncteur $B : \DeuxCatLax \to \DeuxCat$ est un adjoint à gauche de l'inclusion $I : \DeuxCat \to \DeuxCatLax$, les transformations naturelles $\TransLaxCanonique$ et $\TransStrictCanonique$ constituant respectivement l'unité et la coünité de l'adjonction $(B,I)$. 
\end{theo}

\begin{proof}
Il suffit de vérifier les identités triangulaires, qui stipulent ici que les diagrammes
$$
\xymatrix{
I
\ar @{=>} [rr]^{\TransLaxCanonique{}I}
\ar @{=} [drr]
&& IBI
\ar @{=>} [d]^{I \TransStrictCanonique}
&\rm{et}
&B
\ar @{=>} [rr]^{B \TransLaxCanonique}
\ar @{=} [drr]
&& BIB
\ar @{=>} [d]^{\TransStrictCanonique B}
\\
&& I
&&&&B
}
$$
sont commutatifs. 

Pour toute petite \deux{}catégorie $\mathdeuxcat{A}$, 
$$
\begin{aligned}
((I \TransStrictCanonique) (\TransLaxCanonique{} I))_{\mathdeuxcat{A}} &= (I \TransStrictCanonique)_{\mathdeuxcat{A}} (\TransLaxCanonique{} I)_{\mathdeuxcat{A}}
\\
&=I(\StrictCanonique{\mathdeuxcat{A}}) \LaxCanonique{I(\mathdeuxcat{A})} 
\\
&= \StrictCanonique{\mathdeuxcat{A}} \LaxCanonique{\mathdeuxcat{A}}
\\
&= 1_{\mathdeuxcat{A}}
\\
&= 1_{I(\mathdeuxcat{A})}
\\
&= (1_{I})_{\mathdeuxcat{A}},
\end{aligned}
$$
ce qui montre la première identité triangulaire. De plus, les égalités
$$
\begin{aligned}
((\TransStrictCanonique{} B)(B \TransLaxCanonique{}))_{\mathdeuxcat{A}} &= (\TransStrictCanonique{}B)_{\mathdeuxcat{A}} (B \TransLaxCanonique{})_{\mathdeuxcat{A}}
\\
&= \StrictCanonique{B(\mathdeuxcat{A})} B(\LaxCanonique{\mathdeuxcat{A}})
\\
&= \StrictCanonique{\TildeLax{\mathdeuxcat{A}}} \TildeLax{\LaxCanonique{\mathdeuxcat{A}}}
\end{aligned}
$$
montrent que la vérification de la seconde identité triangulaire se ramène dans le cas présent à celle de l'égalité
$$
\StrictCanonique{\TildeLax{\mathdeuxcat{A}}} \TildeLax{\LaxCanonique{\mathdeuxcat{A}}} = 1_{\TildeLax{\mathdeuxcat{A}}}.
$$
Les deux termes de cette égalité étant des \DeuxFoncteursStricts{}, il suffit de vérifier qu'ils coïncident sur les objets, \un{}cellules et \deux{}cellules de $\TildeLax{\mathdeuxcat{A}}$. 
Pour tout objet $a$ de $\TildeLax{\mathdeuxcat{A}}$, 
$$
\begin{aligned}
\StrictCanonique{\TildeLax{\mathdeuxcat{A}}} \TildeLax{\LaxCanonique{\mathdeuxcat{A}}} (a) &= \StrictCanonique{\TildeLax{\mathdeuxcat{A}}} (\LaxCanonique{\mathdeuxcat{A}} (a))
\\
&= \StrictCanonique{\TildeLax{\mathdeuxcat{A}}} (a)
\\
&= a.
\end{aligned}
$$
Pour toute \un{}cellule $([m], x)$ de $\TildeLax{\mathdeuxcat{A}}$, 
$$
\begin{aligned}
\StrictCanonique{\TildeLax{\mathdeuxcat{A}}} \TildeLax{\LaxCanonique{\mathdeuxcat{A}}} ([m], x) &= \StrictCanonique{\TildeLax{\mathdeuxcat{A}}} ([m], ([1], x_{1,0}), \dots, ([1], x_{m,m-1}))
\\
&= ([1], x_{m,m-1}) \dots ([1], x_{1,0})
\\
&= ([m], x).
\end{aligned}
$$
Pour toute \deux{}cellule $(\varphi : [n] \to [m], \alpha)$ de $\TildeLax{\mathdeuxcat{A}}$,  
$$
\begin{aligned} 
\StrictCanonique{\TildeLax{\mathdeuxcat{A}}} \TildeLax{\LaxCanonique{\mathdeuxcat{A}}} (\varphi, \alpha) &= \StrictCanonique{\TildeLax{\mathdeuxcat{A}}} (\varphi, ([1] \to [\varphi(1)], \alpha_{1}), \dots, ([1] \to [\varphi(n)], \alpha_{n}))
\\
&= ([1] \to [\varphi(n)], \alpha_{n}) \circ \dots \circ ([1] \to [\varphi(1)], \alpha_{1})
\\
&= (\varphi, \alpha).
\end{aligned} 
$$
\end{proof}

\begin{rem}
Aucun des foncteurs $I : \DeuxCat \hookrightarrow \DeuxCatLax$ et $B : \DeuxCatLax \to \DeuxCat$ n'est plein. En revanche, ils sont tous deux fidèles. Pour $I$, c'est tautologique. Vérifions-le pour $B$. Soient $u$ et $v$ deux \DeuxFoncteursLax{} de $\mathdeuxcat{A}$ vers $\mathdeuxcat{B}$ tels que $\TildeLax{u} = \TildeLax{v}$. 

Pour tout objet $a$ de $\mathdeuxcat{A}$, $\TildeLax{u} (a) = \TildeLax{v} (a)$, c'est-à-dire $u(a) = v(a)$. 

Pour toute \un{}cellule $f$ de $\mathdeuxcat{A}$, $\TildeLax{u} ([1], f) = \TildeLax{v} ([1], f)$, c'est-à-dire $([1], u(f)) = ([1], v(f))$, donc $u(f) = v(f)$. 

Pour toute \deux{}cellule $\alpha$ de $\mathdeuxcat{A}$, $\TildeLax{u} (1_{[1]}, \alpha) = \TildeLax{v} (1_{[1]}, \alpha)$, c'est-à-dire $(1_{[1]}, u(\alpha)) = (1_{[1]}, v(\alpha))$, d'où $u(\alpha) = v(\alpha)$. 

Pour tout couple de \un{}cellules $f$ et $f'$ de $\mathdeuxcat{A}$ telles que la composée $f'f$ fasse sens, le couple $([1] \to [2], 1_{f'f})$ définit une \un{}cellule de $\TildeLax{\mathdeuxcat{A}}$, de source $([2], (f',f))$ et de but $([1], f'f)$, l'objet $([2], (f',f))$ étant défini par $(f',f) (0 \to 1) = f$ et $(f',f) (1 \to 2) = f'$. Alors, $\TildeLax{u} ([1] \to [2], 1_{f'f}) = \TildeLax{v} ([1] \to [2], 1_{f'f})$, c'est-à-dire $([1] \to [2], u_{f',f}) = ([1] \to [2], v_{f',f})$, d'où $u_{f',f} = v_{f',f}$. 

Pour tout objet $a$ de $\mathdeuxcat{A}$, le couple $([1] \to [0], 1_{1_{a}})$ définit une \un{}cellule de $([0], a)$ vers $([1], 1_{a})$ dans $\TildeLax{\mathdeuxcat{A}}$. Alors, $\TildeLax{u} ([1] \to [0], 1_{1_{a}}) = \TildeLax{v} ([1] \to [0], 1_{1_{a}})$, c'est-à-dire $([1] \to [0], u_{a}) = ([1] \to [0], v_{a})$, d'où $u_{a} = v_{a}$. 

Par conséquent, $u = v$. 

Les foncteurs $I : \DeuxCat \hookrightarrow \DeuxCatLax$ et $B : \DeuxCatLax \to \DeuxCat$ sont donc tous deux fidèles. Ainsi, les composantes des transformations naturelles $\TransLaxCanonique{}$ et $\TransStrictCanonique{}$ sont des monomorphismes de $\DeuxCatLax{}$ et des épimorphismes de $\DeuxCat{}$ respectivement (voir par exemple \cite[p. 90, théorème 1]{CWM}). 
\end{rem}

\begin{prop}\label{BijHomBI}
Si $\mathdeuxcat{A}$ et $\mathdeuxcat{B}$ sont deux petites \deux{}catégories, il existe une bijection 
$$
\begin{aligned}
 \EnsHom{\DeuxCatLax}{\mathdeuxcat{A}}{\mathdeuxcat{B}} &\simeq \EnsHom{\DeuxCat}{\TildeLax{\mathdeuxcat{A}}}{\mathdeuxcat{B}} 
\\
u &\mapsto \StrictCanonique{\mathdeuxcat{B}} \TildeLax{u}
\end{aligned}
$$
naturelle en $\mathdeuxcat{A}$ et $\mathdeuxcat{B}$. 
\end{prop}

\begin{proof}
C'est une reformulation du théorème \ref{BIAdjonction}, en vertu de la théorie classique des adjonctions (voir par exemple \cite[p. 83, théorème 2]{CWM}).
\end{proof}

\begin{lemme}\label{LemmeDimitri}
Pour tout morphisme $u : \mathdeuxcat{A} \to \mathdeuxcat{B}$ de $\DeuxCatLax$, le diagramme
$$
\xymatrix{
\TildeLax{\mathdeuxcat{A}}
\ar[r]^{\TildeLax{u}}
\ar[d]_{\StrictCanonique{\mathdeuxcat{A}}}
&\TildeLax{\mathdeuxcat{B}}
\ar[d]^{\StrictCanonique{\mathdeuxcat{B}}}
\\
\mathdeuxcat{A}
\ar[r]_{u}
&\mathdeuxcat{B}
}
$$
est commutatif à une \DeuxTransformationLax{} relative aux objets $\StrictCanonique{\mathdeuxcat{B}} \TildeLax{u} \Rightarrow u \StrictCanonique{\mathdeuxcat{A}}$ près. 
\end{lemme}

\begin{proof}
Pour tout objet $a$ de $\TildeLax{\mathdeuxcat{A}}$, on pose $\sigma_{a} = 1_{u(a)}$ et, pour tout morphisme 
$([m], x)$ de $\TildeLax{\mathdeuxcat{A}}$, on pose $\sigma_{([m], x)} = u_{x}$. Cela définit une \DeuxTransformationLax{} 
$\sigma : \StrictCanonique{\mathdeuxcat{B}} \TildeLax{u} \Rightarrow u \StrictCanonique{\mathdeuxcat{A}}$. 
\end{proof}

\begin{df}\label{DefBarreLax}
Étant donné un \DeuxFoncteurLax{} $u : \mathdeuxcat{A} \to \mathdeuxcat{B}$, on définit son \emph{\deux{}foncteur barre lax}\index{barre lax d'un \DeuxFoncteurLax{}} $\BarreLax{u}\index[not]{uYyy@$\BarreLax{u}$} : \TildeLax{\mathdeuxcat{A}} \to \mathdeuxcat{B}$ par la formule
$$
\BarreLax{u} = \StrictCanonique{\mathdeuxcat{B}} \TildeLax{u}.
$$
Autrement dit, $\BarreLax{u}$ est le \DeuxFoncteurStrict{} associé à $u$ par la bijection naturelle figurant dans l'énoncé de la proposition \ref{BijHomBI}.
\end{df}

\begin{paragr}\label{DefBarreLaxFoncteur}
De façon plus explicite, les formules sont
$$
\begin{aligned}
\BarreLax{u} : \TildeLax{\mathdeuxcat{A}} &\to \mathdeuxcat{B}
\\
a &\mapsto u(a)
\\
([0], a) &\mapsto 1_{u(a)}
\\
(([m], x), m \geq 1) &\mapsto u(x_{m,m-1}) \dots u(x_{1,0})
\\
((\varphi, \alpha) : ([m],x) \to ([n],y)) &\mapsto (u(\alpha_{n}) \CompDeuxZero \dots \CompDeuxZero u(\alpha_{1})) 
\CompDeuxUn (u_{x_{\varphi(n-1)} \to \dots \to x_{m}} \CompDeuxZero \dots \CompDeuxZero u_{x_{0} \to \dots \to x_{\varphi(1)}}).
\end{aligned}
$$
\end{paragr}

\begin{rem}
On utilisera souvent la caractérisation suivante de $\BarreLax{u}$ : c'est l'unique \DeuxFoncteurStrict{} de $\TildeLax{\mathdeuxcat{A}}$ vers $\mathdeuxcat{B}$ rendant le diagramme
$$
\xymatrix{
\TildeLax{\mathdeuxcat{A}}
\ar[rd]
\\
\mathdeuxcat{A}
\ar[u]^{\LaxCanonique{\mathdeuxcat{A}}}
\ar[r]_{u}
&\mathdeuxcat{B}
}
$$
commutatif.
\end{rem}

\begin{paragr}\label{DeuxTransInduite}
Soient $u$ et $v$ deux \DeuxFoncteursLax{} de $\mathdeuxcat{A}$ vers $\mathdeuxcat{B}$ et $\sigma$ une \DeuxTransformationLax{} ou une \DeuxTransformationCoLax{} de $u$ vers $v$. 

Pour tout objet $a$ de $\mathdeuxcat{A}$, posons $\BarreLax{\sigma}_{a} = \sigma_{a}$. 

Pour toute \un{}cellule $([m], x)$ de $a$ vers $a'$ dans $\TildeLax{\mathdeuxcat{A}}$, on définit une \deux{}cellule $\sigma_{([m], x)}$ comme suit. Si $m = 0$, posons $\BarreLax{\sigma}_{([m], x)} = \BarreLax{\sigma}_{([0], a)} = 1_{\sigma_{a}}$. Si $m = 1$, posons $\BarreLax{\sigma}_{([m], x)} = \BarreLax{\sigma}_{([1], x_{1,0})} = \sigma_{x_{1, 0}}$. Si $m \geq 2$, posons
$$
\BarreLax{\sigma}_{([m], x)} = ((v(x_{m, m-1}) \dots v(x_{2, 1})) \CompDeuxZero \sigma_{x_{1, 0}}) \CompDeuxUn (\BarreLax{\sigma}_{([m-1], (x_{m, m-1}, \dots, x_{2, 1}))} \CompDeuxZero u(x_{1, 0}))
$$
si $\sigma$ est une \DeuxTransformationLax{}, et 
$$
\BarreLax{\sigma}_{([m], (x))} = (\sigma_{x_{m, m-1}} \CompDeuxZero (u(x_{m-1, m-2}) \dots u(x_{1,0}))) \CompDeuxZero (v(x_{m, m-1}) \CompDeuxZero \BarreLax{\sigma}_{([m-1], (x_{m-1, m-2}, \dots, x_{1, 0}))})
$$
si $\sigma$ est une \DeuxTransformationCoLax{}.
\end{paragr}

\begin{lemme}\label{BarreLaxDeuxTrans}
Étant donné deux \DeuxFoncteursLax{} parallèles $u$ et $v$ et une \DeuxTransformationLax{} (\emph{resp.} \DeuxTransformationCoLax{}) $\sigma : u \Rightarrow v$, le procédé décrit dans le paragraphe \ref{DeuxTransInduite} définit une \DeuxTransformationLax{} (\emph{resp.} \DeuxTransformationCoLax{}) $\BarreLax{\sigma}\index[not]{0sigmaYyy@$\BarreLax{\sigma}$} : \BarreLax{u} \Rightarrow \BarreLax{v}$.
\end{lemme}

\begin{proof}
Nous nous limitons à quelques indications. Restreignons-nous au cas d'une \DeuxTransformationCoLax{}. Des trois conditions de cohérence à vérifier, à savoir celle « d'unité », celle « de composition des \un{}cellules » et celle « de compatibilité aux \deux{}cellules », seule cette dernière ne devrait pas sembler résulter immédiatement des conditions de cohérence faisant partie des diverses hypothèses. Ces dernières, assaisonnées d'un argument de récurrence, permettent de se ramener au cas d'une \deux{}cellule de $\TildeLax{\mathdeuxcat{A}}$ dont la première composante est l'unique application d'intervalles de $[1]$ vers $[2]$. Il s'agit donc de vérifier l'assertion suivante : étant donné $x_{0}$, $x_{1}$ et $x_{2}$ trois objets de $\mathdeuxcat{A}$, $x_{1,0}$, $x_{2,1}$ et $y_{2,0}$ trois \un{}cellules de $x_{0}$ vers $x_{1}$, de $x_{1}$ vers $x_{2}$ et de $x_{0}$ vers $x_{2}$ respectivement et $\alpha$ une \deux{}cellule de $x_{2,1} x_{1,0}$ vers $y_{2,0}$, le diagramme formé par « l'extérieur » du diagramme
$$
\xymatrix{
v(x_{2,1}) v(x_{1,0}) \sigma_{x_{0}}
\ar@{=>}[rr]^{v(x_{2,1}) \CompDeuxZero \sigma_{x_{1,0}}}
\ar@{=>}[dd]_{v_{x_{2,1}, x_{1,0}} \CompDeuxZero \sigma_{x_{0}}}
&&v(x_{2,1}) \sigma_{x_{1}} u(x_{1,0})
\ar@{=>}[rr]^{\sigma_{x_{2,1}} \CompDeuxZero u(x_{1,0})}
&&\sigma_{x_{2}} u(x_{2,1}) u(x_{1,0})
\ar@{=>}[dd]^{\sigma_{x_{2}} \CompDeuxZero u_{x_{2,1}, x_{1,0}}}
\\
\\
v(x_{2,1} x_{1,0}) \sigma_{x_{0}}
\ar@{=>}[rrrr]^{\sigma_{x_{2,1} x_{1,0}}}
\ar@{=>}[dd]_{v(\alpha) \CompDeuxZero \sigma_{x_{0}}}
&&&&\sigma_{x_{2}} u(x_{2,1} x_{1,0})
\ar@{=>}[dd]^{\sigma_{x_{2}} \CompDeuxZero u(\alpha)}
\\
\\
v(y_{2,0}) \sigma_{x_{0}}
\ar@{=>}[rrrr]_{\sigma_{y_{2,0}}}
&&&&\sigma_{x_{2}} u(y_{2,0})   
}
$$
est commutatif. Or, en vertu de la « naturalité de $\sigma$ par rapport à la composition des \un{}cellules », la partie supérieure de ce diagramme est commutative. La partie inférieure l'est également, en vertu de la « naturalité de $\sigma$ par rapport aux \deux{}cellules ». Le diagramme « extérieur » est donc bien commutatif. 
\end{proof}

\begin{rem}
Réciproquement, la donnée d'une \DeuxTransformationLax{} (\emph{resp.} d'une \DeuxTransformationCoLax{}) de $\BarreLax{u}$ vers $\BarreLax{v}$ permet de définir une \DeuxTransformationLax{} (\emph{resp.} une \DeuxTransformationCoLax{}) de $u$ vers $v$. 
\end{rem}

\begin{rem}
En revanche, et contrairement à ce que l'on pourrait croire, il ne semble pas possible de définir de façon générale une \DeuxTransformationLax{} (\emph{resp.} une \DeuxTransformationCoLax{}) de $\TildeLax{u}$ vers $\TildeLax{v}$ étant donné une \DeuxTransformationLax{} (\emph{resp.} une \DeuxTransformationCoLax{}) de $u$ vers $v$. Le lecteur est invité à s'en convaincre. (La première difficulté réside dans l'absence d'application d'intervalles « canonique » de $[m+1]$ vers $[m+1]$ qui ne soit pas une identité, pour $m \geq 2$.)
\end{rem}

\begin{rem}
Le lemme \ref{BarreLaxDeuxTrans} nous permet de démontrer de façon paresseuse le lemme \ref{DeuxTransFoncLax} en remarquant qu'il suffit de vérifier les conditions de cohérence portant sur le \DeuxFoncteurLax{} $h$ construit au cours de la preuve dans le cas où $u$ et $v$ sont des \DeuxFoncteursStricts{}. Dans le cas général de \DeuxFoncteursLax{} quelconques, on se ramène au cas particulier de \DeuxFoncteursStricts{} par l'argument suivant. Une \DeuxTransformationLax{} (\emph{resp.} \DeuxTransformationCoLax{}) $u \Rightarrow v$ induisant une \DeuxTransformationLax{} (\emph{resp.} \DeuxTransformationCoLax{}) $\BarreLax{u} \Rightarrow \BarreLax{v}$, il existe, en vertu de ce qui précède, un \DeuxFoncteurLax{}
$$
h : [1] \times \TildeLax{\mathdeuxcat{A}} \to \mathdeuxcat{B}
$$
tel que le diagramme
$$
\xymatrix{
&[1] \times \TildeLax{\mathdeuxcat{A}}
\ar[dd]^{h}
\\
\TildeLax{\mathdeuxcat{A}}
\ar[ur]^{0 \times 1_{\TildeLax{\mathdeuxcat{A}}}}
\ar[dr]_{\BarreLax{u}}
&&\TildeLax{\mathdeuxcat{A}}
\ar[ul]_{1 \times 1_{\TildeLax{\mathdeuxcat{A}}}}
\ar[dl]^{\BarreLax{v}}
\\
&\mathdeuxcat{B}
}
$$
soit commutatif. Il en résulte que le diagramme 
$$
\xymatrix{
&&[1] \times \mathdeuxcat{A}
\ar[dd]^{h (1_{[1]} \times \LaxCanonique{\mathdeuxcat{A}})}
\\
\mathdeuxcat{A}
\ar[urr]^{0 \times 1_{\mathdeuxcat{A}}}
\ar[drr]_{u}
&&&&\mathdeuxcat{A}
\ar[ull]_{1 \times 1_{\mathdeuxcat{A}}}
\ar[dll]^{v}
\\
&&\mathdeuxcat{B}
}
$$
est commutatif, ce qui termine la démonstration.
\end{rem}

On mentionne maintenant brièvement les analogues pour les \DeuxFoncteursCoLax{} des constructions précédentes. 

\begin{df}
Soit $\mathdeuxcat{A}$ une petite \deux{}catégorie. On notera $\TildeColax{\mathdeuxcat{A}}$\index[not]{AZzzColax@$\TildeColax{\mathdeuxcat{A}}$} la \deux{}catégorie définie par
$$
\TildeColax{\mathdeuxcat{A}} = \DeuxCatDeuxOp{(\TildeLax{\DeuxCatDeuxOp{\mathdeuxcat{A}}})}.
$$

On notera $\ColaxCanonique{\mathdeuxcat{A}}$\index[not]{0EtaCA@$\ColaxCanonique{\mathdeuxcat{A}}$} le \DeuxFoncteurCoLax{} défini par
$$
\ColaxCanonique{\mathdeuxcat{A}} = \DeuxFoncDeuxOp{(\LaxCanonique{\DeuxCatDeuxOp{\mathdeuxcat{A}}})} : \mathdeuxcat{A} \to \TildeColax{\mathdeuxcat{A}}.
$$

On notera $\StrictColaxCanonique{\mathdeuxcat{A}}$\index[not]{0EpsilonCA@$\StrictColaxCanonique{\mathdeuxcat{A}}$} le \DeuxFoncteurStrict{} défini par 
$$
\StrictColaxCanonique{\mathdeuxcat{A}} = \DeuxFoncDeuxOp{(\StrictCanonique{\DeuxCatDeuxOp{\mathdeuxcat{A}}})} : \TildeColax{\mathdeuxcat{A}} \to\mathdeuxcat{A} .
$$

Pour tout \DeuxFoncteurCoLax{} $u : \mathdeuxcat{A} \to \mathdeuxcat{B}$, on notera $\TildeColax{u}$\index[not]{uZzzColax@$\TildeColax{u}$} le \DeuxFoncteurStrict{} défini par
$$
\TildeColax{u} = \DeuxFoncDeuxOp{(\TildeLax{\DeuxFoncDeuxOp{u}})} : \TildeColax{\mathdeuxcat{A}} \to \TildeColax{\mathdeuxcat{B}}
$$
et l'on notera $\BarreColax{u}$\index[not]{uYyyColax@$\BarreColax{u}$} le \DeuxFoncteurStrict{} défini par
$$
\BarreColax{u} = \DeuxFoncDeuxOp{(\BarreLax{\DeuxFoncDeuxOp{u}})} : \TildeColax{\mathdeuxcat{A}} \to \mathdeuxcat{B}.
$$
On a donc notamment, par définition, l'identité
$$
\StrictColaxCanonique{\mathdeuxcat{B}} \TildeColax{u} = \BarreColax{u}.
$$
\end{df} 

\begin{lemme}
Pour tout \DeuxFoncteurCoLax{} $u : \mathdeuxcat{A} \to \mathdeuxcat{B}$, le diagramme de \DeuxFoncteursCoLax{}
$$
\xymatrix{
\TildeColax{\mathdeuxcat{A}}
\ar[r]^{\TildeColax{u}}
&\TildeColax{\mathdeuxcat{B}}
\\
\mathdeuxcat{A}
\ar[u]^{\ColaxCanonique{\mathdeuxcat{A}}}
\ar[r]_{u}
&\mathdeuxcat{B}
\ar[u]_{\ColaxCanonique{\mathdeuxcat{B}}}
}
$$
est commutatif.
\end{lemme}

\begin{proof}
C'est l'énoncé dual du lemme \ref{UniteNaturelle}.
\end{proof}

\begin{lemme}
Pour tout \DeuxFoncteurStrict{} $u : \mathdeuxcat{A} \to \mathdeuxcat{B}$, le diagramme de \DeuxFoncteursStricts{}
$$
\xymatrix{
\TildeColax{\mathdeuxcat{A}}
\ar[r]^{\TildeColax{u}}
\ar[d]_{\StrictColaxCanonique{\mathdeuxcat{A}}}
&\TildeColax{\mathdeuxcat{B}}
\ar[d]^{\StrictColaxCanonique{\mathdeuxcat{B}}}
\\
\mathdeuxcat{A}
\ar[r]_{u}
&\mathdeuxcat{B}
}
$$
est commutatif.
\end{lemme}

\begin{proof}
C'est l'énoncé dual du lemme \ref{CouniteNaturelle}.
\end{proof}

\begin{rem}
Pour tout \DeuxFoncteurCoLax{} $u : \mathdeuxcat{A} \to \mathdeuxcat{B}$, $\BarreColax{u}$ est l'unique \DeuxFoncteurStrict{} de $\TildeColax{\mathdeuxcat{A}}$ vers $\mathdeuxcat{B}$ rendant le diagramme
$$
\xymatrix{
\TildeColax{\mathdeuxcat{A}}
\ar[dr]
\\
\mathdeuxcat{A}
\ar[u]^{\ColaxCanonique{\mathdeuxcat{A}}}
\ar[r]_{u}
&\mathdeuxcat{B}
}
$$
commutatif. 
\end{rem} 

\chapter{De $\Cat$ à $\DeuxCat$ en passant par $\widehat{\Delta}$}

\section{Localisateurs fondamentaux de $\Cat$}\label{SectionUnLocFond}

\begin{df}\label{DefWTop}
On dit qu'une application continue $f : X \to Y$ entre espaces topologiques est une \emph{équivalence faible topologique}\index{equivalence faible topologique@équivalence faible topologique}, ou plus simplement une \emph{équivalence faible}, si elle induit une bijection au niveau des $\pi_{0}$ et des isomorphismes entre les groupes d'homotopie pour tout choix de point base. Plus précisément, 
$$
\pi_{0}(f) : \pi_{0}(X) \to \pi_{0}(Y)
$$
est une bijection et, pour tout point $x$ de $X$ et tout entier $n \geq 1$,
$$
\pi_{n}(f,x) : \pi_{n}(X,x) \to \pi_{n}(Y,f(x))
$$
est un isomorphisme de groupes.
\end{df}

\begin{paragr}
On rappelle qu'un \emph{ensemble simplicial} est un foncteur de $\DeuxCatUnOp{\Delta}$ vers la catégorie $Ens$\index[not]{Ens@$Ens$} des ensembles. Suivant l'usage, on notera $\EnsSimp$\index[not]{0DeltaChapeau@$\EnsSimp$} la catégorie des ensembles simpliciaux, un morphisme d'ensembles simpliciaux n'étant rien d'autre qu'une transformation naturelle, c'est-à-dire un morphisme de foncteurs. On rappelle l'existence du foncteur \emph{réalisation géométrique} $\EnsSimp \to \Top$, obtenu par extension de Kan, $\Top$\index[not]{Top@$\Top$} désignant la catégorie des espaces topologiques et des applications continues entre iceux. De façon similaire, on appelle \emph{ensembles bisimpliciaux}\index{ensemble bisimplicial} les foncteurs de la catégorie $\DeuxCatUnOp{(\Delta \times \Delta)}$ vers $Ens$, et l'on note $\widehat{\Delta \times \Delta}$\index[not]{0DeltaChapeauDelta@$\widehat{\Delta \times \Delta}$} la catégorie des ensembles bisimpliciaux, dont les morphismes sont les morphismes de tels foncteurs. Pour tout ensemble simplicial $X$, on notera $X_{m}$ l'ensemble $X([m])$ des \emph{$m$\nobreakdash-simplexes de $X$} et, pour tout ensemble bisimplicial $X$, on notera $X_{m,n}$ l'ensemble $X([m],[n])$ des \emph{$(m,n)$\nobreakdash-simplexes de $X$}.
\end{paragr}

\begin{df}\label{DefWEnsSimp}
On dit qu'un morphisme d'ensembles simpliciaux est une \emph{équivalence faible simpliciale}\index{equivalence faible simpliciale@équivalence faible simpliciale}, ou plus simplement une \emph{équivalence faible}, si son image par le foncteur de réalisation géométrique est une équivalence faible topologique. On notera $\EquiQuillen$\index[not]{W0Infini@$\EquiQuillen$} la classe des équivalences faibles simpliciales. 
\end{df}


\begin{lemme}\label{WQuillenStableSomme}
Une petite somme d'équivalences faibles simpliciales est une équivalence faible simpliciale.
\end{lemme}

\begin{proof}
C'est immédiat.   
\end{proof}

\begin{paragr}\label{FoncteurDiagonal}
Le foncteur diagonal 
$$
\begin{aligned}
\delta_{\Delta}\index[not]{0deltaDelta@$\delta_{\Delta}$} : \Delta &\to \Delta \times \Delta
\\
[n] &\mapsto ([n],[n])
\end{aligned}
$$
induit un foncteur
$$
\begin{aligned}
\delta_{\Delta}^{*}\index[not]{0deltaDeltaEtoile@$\delta_{\Delta}^{*}$} : \widehat{\Delta \times \Delta} &\to \EnsSimp
\\
X &\mapsto ([n] \mapsto X_{n,n}).
\end{aligned}
$$
\end{paragr}

La proposition \ref{LemmeBisimplicialDelta} stipule qu'un morphisme d'ensembles bisimpliciaux qui est une é\-qui\-va\-lence faible simpliciale « sur les colonnes » ou « sur les lignes » en est une « sur la diagonale ». 
  
\begin{prop}\label{LemmeBisimplicialDelta}
Soit $f : X \to Y$ un morphisme d'ensembles bisimpliciaux  tel que, pour tout entier $n \geq 0$, le morphisme d'ensembles simpliciaux $f_{n,\bullet} : X_{n,\bullet} \to Y_{n,\bullet}$ (\emph{resp.} $f_{\bullet,n} : X_{\bullet,n} \to Y_{\bullet,n}$) soit une équivalence faible. Alors, $\delta_{\Delta}^{*}(f)$ est une équivalence faible.  
\end{prop}

\begin{proof}
Pour une démonstration de ce résultat folklorique mais non-trivial, le lecteur pourra consulter \cite[p. 94-95]{QuillenK}, \cite[chapitre XII, paragraphe 4.3]{BK} ou \cite[proposition 1.7]{GoerssJardine}. 
\end{proof}

\begin{rem}
On pourra également se reporter à \cite[lemme 3.5]{Illusie}, correspondant au cas particulier que nous utiliserons de la proposition \ref{LemmeBisimplicialDelta}. 
\end{rem}

\begin{df}
On dit qu'un foncteur entre petites catégories est une \emph{équivalence faible catégorique}\index{equivalence faible catégorique@équivalence faible catégorique}, ou plus simplement une \emph{équivalence faible}, si son image par le foncteur nerf\index{nerf} usuel, que l'on notera $\UnNerf$\index[not]{N1@$\UnNerf$}, est une équivalence faible simpliciale. On note $\UnLocFondMin$\index[not]{WInfini1@$\UnLocFondMin$} la classe des équivalences faibles de $\Cat$. Soit, en formule :
$$
\UnLocFondMin = \UnNerf^{-1} (\EquiQuillen).
$$
\end{df}

\begin{df}
On dira qu'une petite catégorie $A$ est \emph{asphérique} si le foncteur canonique $A \to \UnCatPonct$ est une équivalence faible. 
\end{df}

\begin{rem}
Nous suivons la terminologie de Grothendieck. 
\end{rem}

\begin{lemme}\label{ObjetFinalAspherique}
Une petite catégorie admettant un objet initial ou un objet final est asphérique. 
\end{lemme}

\begin{proof}
C'est \cite[p. 84, corollaire 2]{QuillenK}.
\end{proof}

La version « absolue » — c'est-à-dire sans base, correspondant au cas $v = 1_{\mathdeuxcat{B}}$ — du théorème \ref{UnLocFondMinThA} constitue le Théorème A de Quillen. La démonstration de la version relative que nous en donnons s'inspire directement de celle de Quillen. Nous suivons en cela Cisinski. Historiquement, l'origine de cette version relative remonte apparemment à Grothendieck, qui en donne une esquisse de preuve « toposique » dans \cite{Poursuite} (voir aussi \cite[p. 11]{THG}). 

\begin{theo}[Quillen]\label{UnLocFondMinThA}
Soit
$$
\xymatrix{
A 
\ar[rr]^{u}
\ar[dr]_{w}
&&
B
\ar[dl]^{v}
\\
&
C
}
$$
un triangle commutatif dans $\Cat$. Si, pour tout objet $c$ de $C$, le foncteur $u/c\index[not]{uc@$u/c$} : A/c \to B/c$, induit par $u$, est une équivalence faible, alors $u$ est une équivalence faible.
\end{theo}

\begin{proof}
Définissons deux ensembles bisimpliciaux $S_{w}$ et $S_{v}$ par les formules
$$
(S_{w})_{m,n} = \{ (a_{0} \to \dots \to a_{m}, w(a_{m}) \to c_{0} \to \dots \to c_{n}), a_{i} \in \Objets{A}, c_{j} \in \Objets{C}, 0 \leq i \leq m, 0 \leq j \leq n \}
$$
et
$$
(S_{v})_{m,n} = \{ (b_{0} \to \dots \to b_{m}, v(b_{m}) \to c_{0} \to \dots \to c_{n}), b_{i} \in \Objets{B}, c_{j} \in \Objets{C}, 0 \leq i \leq m, 0 \leq j \leq n \}
$$
pour tout couple d'entiers positifs $m$ et $n$. 
Les opérateurs faces et dégénérescences se définissent de façon « évidente ».  

De plus, on considère les ensembles simpliciaux $\UnNerf{A}$ et $\UnNerf{B}$ comme des ensembles bisimpliciaux constants sur les colonnes. Autrement dit, pour tout couple d'entiers positifs $m$ et $n$, on pose $(\UnNerf{A})_{m,n} = (\UnNerf{A})_{m}$ et $(\UnNerf{B})_{m,n} = (\UnNerf{B})_{m}$. On note de même $\UnNerf{(u)} : \UnNerf{A} \to \UnNerf{B}$ le morphisme d'ensembles bisimpliciaux induit par $\UnNerf{(u)}$. On construit alors un diagramme commutatif de morphismes d'ensembles bisimpliciaux
$$
\xymatrix{
\UnNerf{A}
\ar[r]^{\UnNerf{(u)}}
&
\UnNerf{B}
\\
S_{w}
\ar[u]^{\varphi_{w}}
\ar[r]_{U}
&S_{v}
\ar[u]_{\varphi_{v}}
}
$$
en posant 
$$
(\varphi_{w})_{m,n} (a_{0} \to \dots \to a_{m}, w(a_{m}) \to c_{0} \to \dots \to c_{n}) = a_{0} \to \dots \to a_{m}
$$ 
et 
$$
U_{m,n} (a_{0} \to \dots \to a_{m}, w(a_{m}) \to c_{0} \to \dots \to c_{n}) = (u(a_{0}) \to \dots \to u(a_{m}), w(a_{m}) \to c_{0} \to \dots \to c_{n}),
$$
la définition de $\varphi_{v}$ étant analogue à celle de $\varphi_{w}$.

Par construction, $\delta_{\Delta}^{*} (\UnNerf{(u)}) = \UnNerf{(u)}$ (à gauche du symbole d'égalité, $\UnNerf{(u)}$ désigne un morphisme d'ensembles bisimpliciaux, à droite c'est un morphisme d'ensembles simpliciaux ; on ne signalera plus cet abus). Pour démontrer le résultat souhaité, il suffit donc de vérifier que $\delta_{\Delta}^{*} (U)$, $\delta_{\Delta}^{*} (\varphi_{w})$ et $\delta_{\Delta}^{*} (\varphi_{v})$ sont des équivalences faibles simpliciales. Un argument de « $2$ sur $3$ » permettra de conclure. 

En vertu de la proposition \ref{LemmeBisimplicialDelta}, pour montrer que $\delta_{\Delta}^{*} (U)$ est une équivalence faible simpliciale, il suffit de montrer que, pour tout entier $n \geq 0$, le morphisme d'ensembles simpliciaux $U_{\bullet, n} : (S_{w})_{\bullet, n} \to (S_{v})_{\bullet, n}$ en est une. On remarque que la source et le but de ce morphisme s'identifient à
$$
\coprod_{c_{0} \to \dots \to c_{n} \in (\UnNerf{C})_{n}} \UnNerf (A/c_{0})
$$
et
$$
\coprod_{c_{0} \to \dots \to c_{n} \in (\UnNerf{C})_{n}} \UnNerf (B/c_{0})
$$
respectivement, et $U_{\bullet, n}$ s'identifie à 
$$
\coprod_{c_{0} \to \dots \to c_{n} \in (\UnNerf{C})_{n}} \UnNerf{(u/c_{0})}.
$$
En vertu des hypothèses, chaque terme de cette somme est une équivalence faible. Il s'ensuit que $U_{\bullet, n}$ est une équivalence faible (en vertu du lemme \ref{WQuillenStableSomme}). Il en est donc de même de $\delta_{\Delta}^{*} (U)$.

Les raisonnements permettant de montrer que $\delta_{\Delta}^{*} (\varphi_{w})$ et $\delta_{\Delta}^{*} (\varphi_{v})$ sont des équivalences faibles sont évidemment analogues l'un à l'autre. Considérons le cas de $\delta_{\Delta}^{*} (\varphi_{w})$. Pour montrer que c'est une équivalence faible, il suffit, en vertu de la proposition \ref{LemmeBisimplicialDelta}, de montrer que, pour tout entier $m \geq 0$, le morphisme simplicial $(\varphi_{w})_{m, \bullet} : (S_{w})_{m, \bullet} \to (\UnNerf{A})_{m, \bullet}$ est une équivalence faible. On remarque que la source et le but de ce morphisme s'identifient à 
$$
\coprod_{a_{0} \to \dots \to a_{m} \in (\UnNerf{A})_{m}} \UnNerf{(w(a_{m})\backslash C})
$$
et
$$
\coprod_{a_{0} \to \dots \to a_{m} \in (\UnNerf{A})_{m}} \ast\index[not]{*@$\ast$}
$$
respectivement, $\ast$ désignant un ensemble simplicial final. Le morphisme simplicial $(\varphi_{w})_{m, \bullet}$ s'identifie ainsi à
$$
\coprod_{a_{0} \to \dots \to a_{m} \in (\UnNerf{A})_{m}} \UnNerf{(w(a_{m})\backslash C} \to \UnCatPonct).
$$
Pour tout objet $a_{m}$ de $A$, la catégorie $w(a_{m})\backslash C$ admet un objet initial, donc est asphérique (lemme \ref{ObjetFinalAspherique}). L'expression ci-dessus est donc une somme d'équivalences faibles, donc une équi\-valence faible. Ainsi, $(\varphi_{w})_{m, \bullet}$ est une équivalence faible. Il en est donc de même de $\delta_{\Delta}^{*} (\varphi_{w})$. Comme annoncé, le résultat s'ensuit. 
\end{proof}

Pour toute \deux{}catégorie $\mathdeuxcat{A}$, nous noterons $\UnCell{\mathdeuxcat{A}}$\index[not]{Fl1A@$\UnCell{\mathdeuxcat{A}}$} la classe des \un{}cellules de $\mathdeuxcat{A}$. Nous utiliserons la même notation pour la classe des morphismes d'une catégorie, que l'on considérera comme une \deux{}catégorie dont toutes les \deux{}cellules sont des identités. 

\begin{df}\label{DefSaturationFaible}
Soit $C$ une petite catégorie. Une partie $S \subset \UnCell{C}$ est dite \emph{faiblement saturée}\index{faiblement saturée} si elle vérifie les conditions suivantes.
\begin{itemize}
\item[FS1] Les identités des objets de $C$ sont dans $S$.
\item[FS2] Si deux des trois flèches d'un triangle commutatif de morphismes de $C$ sont dans $S$, alors la troisième l'est aussi.
\item[FS3] Si $i : X \to Y$ et $r : Y \to X$ sont des morphismes de $C$ vérifiant $ri = 1_{X}$ et si $ir$ est dans $S$, alors il en est de même de $r$ (et donc aussi de $i$ en vertu de ce qui précède). 
\end{itemize}
\end{df}

Une classe de flèches faiblement saturée d'une catégorie $C$ contient donc en particulier les isomorphismes de $C$.

\begin{df}[Grothendieck]\label{DefUnLocFond}
On appelle \emph{\ClasseUnLocFond{}}\index{localisateur fondamental de $\Cat$} une partie $\UnLocFond{W}$ de $\UnCell{\Cat}$ vérifiant les conditions suivantes.
\begin{itemize}
\item[LA] La partie $\UnLocFond{W}$ de $\UnCell{\Cat}$ est faiblement saturée.
\item[LB] Si $A$ est une petite catégorie admettant un objet final, alors le morphisme canonique $A \to \UnCatPonct$ est dans $\UnLocFond{W} $.
\item[LC] Si
$$
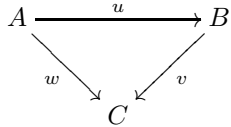
\xymatrix{
A 
\ar[rr]^{u}
\ar[dr]_{w}
&&
B
\ar[dl]^{v}
\\
&
C
}
$$
désigne un triangle commutatif de $\Cat$ et si, pour tout objet $c$ de $C$, le foncteur 
$$
u/c : A/c \to B/c,
$$ 
induit par $u$, est dans $\UnLocFond{W}$, alors $u$ est dans $\UnLocFond{W}$. 
\end{itemize}
\end{df}

\begin{rem}
À l'ajout de l'expression « de $\Cat$ » près (cet ajout se justifiant par le fait que nous introduirons plus loin l'analogue de ce concept pour la catégorie $\DeuxCat$), nous adoptons la terminologie de \cite{THG}, qui n'est pas celle de \cite{Poursuite}, Grothendieck choisissant d'appeler une telle classe un \emph{localisateur fondamental fort}. Il distingue en effet les deux cas d'une classe vérifiant, outre les conditions LA et LB, le cas absolu du Théorème A de Quillen (« localisateur fondamental » pour Grothendieck) ou le cas relatif (« localisateur fondamental fort »). Signalons toutefois que cela reste une question ouverte de savoir s'il existe des classes de morphismes de $\Cat$ vérifiant la première condition mais pas la seconde (sous l'hypothèse qu'elles vérifient les conditions LA et LB). 
\end{rem}

\begin{exemple}\label{UnLocFondMinUnLocFond}
Il est classique que la classe $\UnLocFondMin$ vérifie la condition LA. En vertu du lemme \ref{ObjetFinalAspherique}, elle vérifie la condition LB. Le théorème \ref{UnLocFondMinThA} affirme qu'elle vérifie la condition LC. C'est donc un \ClasseUnLocFond{}. C'en est bien sûr le paradigme justifiant historiquement l'intérêt pour cette notion, comme l'explique l'introduction de \cite{THG}. 
\end{exemple}

\begin{lemme}\label{EquivalenceViaPoint}
Soit $A$ une catégorie admettant un objet final $*$. Pour tout objet $a$ de $A$, on notera $p_{a}$ l'unique morphisme de $a$ vers $*$ dans $A$. Soit $S \subset \UnCell{A}$ vérifiant les propriétés suivantes.
\begin{itemize}
\item[(i)] Si $u \in S$ et $vu \in S$, alors $v \in S$.
\item[(ii)] Si $vu = 1_{*}$ et $uv \in S$, alors $v \in S$ (et donc $u \in S$ en vertu de la propriété précédente).
\end{itemize}
Alors, pour tout diagramme commutatif 
$$
\xymatrix{
a
\ar[rr]^{p_{a}}
\ar[dr]_{s}
&&{*}
\ar[dl]^{q}
\\
&a'
}
$$
dans $A$ tel que $s \in S$, $q$ et $p_{a'}$ sont aussi dans $S$. Si, de plus, $S$ est stable par composition, alors $p_{a}$ est aussi dans $S$. 
\end{lemme}

\begin{proof}
De $s \in S$ et $qp_{a'}s = qp_{a} = s \in S$, on déduit $qp_{a'} \in S$. Comme de plus $p_{a'}q = 1_{*}$, on a bien la première assertion. La seconde se déduit de l'égalité $p_{a} = p_{a'} s$.
\end{proof}

\emph{On suppose maintenant fixé un localisateur fondamental $\UnLocFond{W}$ de $\Cat$, dont on appellera les éléments les \emph{$\UnLocFond{W}$\nobreakdash-équivalences}, ou les \emph{équivalences faibles}\index{equivalence faible (foncteur, pour un localisateur fondamental de $\Cat$)@équivalence faible (pour un localisateur fondamental de $\Cat$)}.} 

\begin{paragr}
Nous nous conformerons aux notations du lemme \ref{EquivalenceViaPoint} en notant $p_{A}\index[not]{pA@$p_{A}$}$, pour toute petite catégorie $A$, le foncteur canonique de $A$ vers $\DeuxCatPonct$.
\end{paragr}

\begin{df}
On dit qu'une petite catégorie $A$ est \emph{$\UnLocFond{W}$\nobreakdash-asphérique}, ou plus simplement \emph{asphérique}\index{asphérique (petite catégorie, pour un localisateur fondamental de $\Cat$)}, si le foncteur canonique $p_{A} : A \to e$ est une équivalence faible.  
\end{df}

\begin{rem}
La condition LB de la définition \ref{DefUnLocFond} stipule donc qu'une petite catégorie admettant un objet final est asphérique. 
\end{rem}

\begin{paragr}
Rappelons qu'un endofoncteur $u : A \to A$ est dit \emph{constant}\index{constant (endofoncteur)} s'il existe un foncteur \mbox{$a : e \to A$} tel que $u = ap_{A}$. Autrement dit, $u$ se factorise par la catégorie ponctuelle. 
\end{paragr}

\begin{lemme}\label{EndoConstantW}
Une petite catégorie admettant un endofoncteur constant qui est une équivalence faible est asphérique. 
\end{lemme}

\begin{proof}
C'est une conséquence immédiate de la condition LA de la définition \ref{DefUnLocFond}. (C'est également un cas particulier du lemme \ref{EquivalenceViaPoint}.)
\end{proof}

\begin{df}
On dit qu'un morphisme $u : A \to B$ de $\Cat$ est \emph{$\UnLocFond{W}$\nobreakdash-asphérique}, ou plus simplement \emph{asphérique}\index{asphérique (foncteur, pour un localisateur fondamental de $\Cat$)}, si, pour tout objet $b$ de $B$, la catégorie $A/b$ est asphérique. 
\end{df}

\begin{rem}
En conservant les mêmes notations, la catégorie $B/b$ est asphérique pour tout objet $b$ de $B$ (elle admet un objet final). En vertu de la saturation faible de $\UnLocFond{W}$, si $u$ est asphérique, le morphisme $u/b : A/b \to B/b$ est donc une équivalence faible pour tout objet $b$ de $B$. Il résulte alors de la condition LC que $u$ est une équivalence faible. 
\end{rem}

\begin{lemme}\label{ProduitAspherique}
Le produit de deux petites catégories asphériques est asphérique. 
\end{lemme}

\begin{proof}
Soient $A$ et $B$ deux petites catégories asphériques. Pour montrer que la catégorie $A \times B$ est asphérique, il suffit de montrer que la projection canonique $A \times B \to B$ est une équivalence faible. Montrons que c'est un foncteur asphérique. Pour tout objet $b$ de $B$, $(A \times B) / b$ s'identifie à $A \times B/b$. Il suffit donc de montrer que la projection canonique $A \times B/b \to A$ est une équivalence faible. Montrons que c'est un foncteur asphérique. Pour tout objet $a$ de $A$, $(A \times B/b) / a$ s'identifie à $A/a \times B/b$, qui admet un objet final, donc est asphérique, ce qui permet de conclure. 
\end{proof}

\begin{lemme}\label{Mistinguett}
Soient $I$ et $A$ deux petites catégories. Si $I$ est asphérique, alors la projection canonique $I \times A \to A$ est asphérique (donc en particulier une équivalence faible). 
\end{lemme}

\begin{proof}
Pour tout objet $a$ de $A$, la catégorie $(I \times A) / a$ s'identifie à $I \times A/a$. Comme $A/a$ admet un objet final, elle est asphérique. Le produit $I \times (A/a)$ est donc asphérique en vertu du lemme \ref{ProduitAspherique}.
\end{proof}

\begin{lemme}\label{UnTransHomotopie}
Soient $A$ et $B$ deux petites catégories et $u$ et $v$ deux foncteurs de $A$ vers $B$. Supposons qu'il existe un morphisme de foncteurs $\sigma : u \Rightarrow v$. Alors, $u$ est une équivalence faible si et seulement si $v$ en est une. 
\end{lemme}

\begin{proof}
Il existe un foncteur $h : [1] \times A \to B$ rendant le diagramme
$$
\xymatrix{
&
[1] \times A
\ar[dd]^{h}
\\
A
\ar[ur]^{0 \times 1_{A}}
\ar[dr]_{u}
&&
A
\ar[ul]_{1 \times 1_{A}}
\ar[dl]^{v}
\\
&
B
}
$$ 
commutatif. Comme la catégorie $[1]$ est asphérique (elle admet un objet final), la projection canonique $[1] \times A \to A$ est une équivalence faible en vertu du lemme \ref{Mistinguett}. Comme c'est une rétraction commune aux deux flèches obliques montantes figurant dans ce diagramme, ces deux flèches sont des équivalences faibles. La conclusion découle de deux arguments consécutifs de « $2$ sur $3$ ».
\end{proof}

\begin{paragr}
En considérant les ensembles comme des catégories dont tous les morphismes sont des identités, on peut voir tout préfaisceau d'ensembles sur une petite catégorie $A$ comme un préfaisceau en petites catégories sur $A$. On note $i_{A}$\index[not]{iA@$i_{A}$} la restriction du foncteur $\DeuxIntOp{A}$ à la catégorie $\widehat{A}$ des préfaisceaux d'ensembles sur $A$. Pour tout préfaisceau d'ensembles $X$ sur $A$, on pourra noter $A / X$ la catégorie $i_{A}X$, dont les objets sont les couples $(a \in \Objets{A}, x \in \Objets{X(a)})$.
En particulier, pour toute petite catégorie $A$, la catégorie $\Delta / \UnNerf{A}$ se décrit comme suit. Ses objets sont les couples $([m], x)$, avec $m \geq 0$ un entier et $x : [m] \to A$ un $m$\nobreakdash-simplexe du nerf de $A$. Les morphismes de $([m], x)$ vers $([n], y)$ sont les applications croissantes $\varphi : [m] \to [n]$ telles que le diagramme
$$
\xymatrix{
[m]
\ar[rr]^{\varphi}
\ar[dr]_{x}
&&
[n]
\ar[dl]^{y}
\\
&
A
}
$$ 
soit commutatif. 
\end{paragr}

\begin{lemme}\label{DecalageObjetFinal}
Pour toute petite catégorie $A$ admettant un objet final, la catégorie $\Delta / \UnNerf{A}$ est asphérique. 
\end{lemme}

\begin{proof}
C'est \cite[lemme 2.2.2]{LFM}, dont nous reprenons la démonstration. Soit $z$ un objet final de $A$. Pour tout $n$\nobreakdash-simplexe ($n \geq 0$) de $\UnNerf{A}$, on définit un $n+1$\nobreakdash-simplexe $Dx$ de $\UnNerf{A}$ par $(Dx)(i) = x(i)$ pour $i \leq n$ et $(Dx)(n+1) = z$. On définit un endofoncteur 
$$
D : \Delta / \UnNerf{A} \to \Delta / \UnNerf{A}
$$
comme suit. Pour tout objet $([m],x)$ de $\Delta / \UnNerf{A}$, 
$
D([m], x) = ([m+1], Dx)
$
et, pour tout morphisme $\varphi : ([m], x) \to ([n], y)$ de $\Delta / \UnNerf{A}$, $D(\varphi)(i) = \varphi(i)$ si $0 \leq i \leq m$ et $D(\varphi)(m+1) = n+1$. Notons $Z : \Delta / \UnNerf{A} \to \Delta / \UnNerf{A}$ l'endofoncteur constant de valeur $([0], z)$. Les inclusions
$$
\begin{aligned}
{}[m] &\to [m+1]
\\
i &\mapsto i 
\end{aligned}
$$
et
$$
\begin{aligned}
{}[0] &\to [m+1]
\\
0 &\mapsto m+1
\end{aligned}
$$
induisent des morphismes de foncteurs $1_{\Delta / \UnNerf{A}} \Rightarrow D$ et $Z \Rightarrow D$. En vertu de deux applications consécutives du lemme \ref{UnTransHomotopie}, $Z$ est une équivalence faible. On conclut grâce au lemme \ref{EndoConstantW}. 
\end{proof}

\begin{paragr}\label{DefUnSup}
Pour toute petite catégorie $A$, on définit un foncteur
$$
\begin{aligned}
\SupUn_{A}\index[not]{Sup1@$\SupUn_{A}$} : \Delta / \UnNerf{A} &\to A
\\
([m], x) &\mapsto x_{m}
\\
\varphi : ([m], x) \to ([n], y) &\mapsto y_{n, \varphi(m)}.
\end{aligned}
$$
L'assignation $A \mapsto \SupUn_{A}$ définit une transformation naturelle de $i_{\Delta} \UnNerf$ vers $1_{\Cat}$. En d'autres termes, pour tout morphisme $u : A \to B$ de $\Cat$, le diagramme
$$
\xymatrix{
\Delta / \UnNerf{A}
\ar[rr]^{\Delta/\UnNerf(u)}
\ar[d]_{\SupUn_{A}}
&&\Delta / \UnNerf{B}
\ar[d]^{\SupUn_{B}}
\\
A
\ar[rr]_{u}
&&B
}
$$
est commutatif.
\end{paragr}

\begin{prop}\label{UnSupAspherique}
Pour toute petite catégorie $A$, le foncteur $\SupUnObjet{A} : \Delta / \UnNerf{A} \to A$ est asphérique (donc en particulier une équivalence faible).
\end{prop}

\begin{proof}
C'est \cite[proposition 2.2.3]{LFM}, attribuée à Grothendieck. Nous reprenons la dé\-mons\-tra\-tion figurant dans \cite{LFM}. On vérifie que, pour tout objet $a$ de $A$, les catégories $\Delta / \UnNerf{(A/a)}$ et $(\Delta / \UnNerf{}A)/a$ sont canoniquement isomorphes. La catégorie $A/a$ admettant un objet final, le lemme \ref{DecalageObjetFinal} permet de conclure. 
\end{proof}

\begin{prop}\label{Sade}
On a l'égalité
$$
\UnLocFond{W} = \UnNerf^{-1} ({i_{\Delta}}^{-1} (W)).
$$
\end{prop}

\begin{proof}
C'est une conséquence immédiate de la proposition \ref{UnSupAspherique}, par un argument de « $2$ sur $3$ ». 
\end{proof}

\begin{paragr}
On ne suppose plus fixé de \ClasseUnLocFond{}. La notion de \ClasseUnLocFond{} est stable par intersection. On définit le \emph{localisateur fondamental minimal de $\Cat$}\index{localisateur fondamental minimal de $\Cat$} comme l'intersection de tous les localisateurs fondamentaux. Le théorème \ref{CisinskiGrothendieck} a été conjecturé par Grothendieck et démontré par Cisinski. 
\end{paragr}

\begin{theo}[Cisinski]\label{CisinskiGrothendieck}
Le localisateur fondamental minimal de $\Cat$ est $\UnLocFondMin$.
\end{theo}

\begin{proof}
C'est \cite[théorème 2.2.11]{LFM}.
\end{proof}

La proposition \ref{MinouDrouet} et le théorème \ref{EquiEnsSimpCat} sont attribués à Quillen par Illusie dans \cite{Illusie}. On pourra consulter \cite[volume 2, chapitre 6, section 3]{Illusie}, et plus précisément \cite[volume 2, chapitre 6, section 3, théorème 3.3]{Illusie} et \cite[volume 2, chapitre 6, section 3, corollaire 3.3.1]{Illusie} pour les résultats que nous reprenons.  

\begin{prop}[Quillen]\label{MinouDrouet}
Il existe un diagramme
$$
\UnNerf i_{\Delta} \Longleftarrow K \Longrightarrow 1_{\EnsSimp}
$$
dont les flèches sont des morphismes d'endofoncteurs de $\EnsSimp$ qui sont des équivalences faibles simpliciales argument par argument.
\end{prop}

\begin{proof}
À tout ensemble simplicial $X$, on associe un ensemble bisimplicial $SX$\index[not]{SX@$SX$} défini par
$$
(SX)_{m,n} = \{ \Delta_{m} \to \Delta_{r_{0}} \to \dots \to \Delta_{r_{n}} \to X, (r_{0}, \dots, r_{n}) \in \mathbb{N}^{n+1} \}
$$
pour tout couple d'entiers positifs $m$ et $n$. Les faces et dégénérescences « verticales » (correspondant à la seconde composante) sont définies par la composition et l'insertion d'identités respectivement. Les faces et dégénérescences « horizontales » (correspondant à la première composante) sont définies par précomposition avec l'image par le foncteur de Yoneda des morphismes faces $[m-1] \to [m]$ et dégénérescences $[m+1] \to [m]$ correspondants dans $\Delta$.  

On notera $X$ et $\UnNerf i_{\Delta} X$ les ensembles bisimpliciaux (constants sur les colonnes et les lignes respectivement) définis par $X_{m,n} = X_{m}$ et $(\UnNerf i_{\Delta} X)_{m,n} = (\UnNerf i_{\Delta} X)_{n}$ respectivement, pour tout couple d'entiers positifs $m$ et $n$. Considérons les morphismes d'ensembles bisimpliciaux
$$
\begin{aligned}
SX &\to \UnNerf i_{\Delta} X
\\
(\Delta_{m} \to \Delta_{r_{0}} \to \dots \to \Delta_{r_{n}} \to X) &\mapsto (\Delta_{r_{0}} \to \dots \to \Delta_{r_{n}} \to X) 
\end{aligned}
$$
et
$$
\begin{aligned}
SX &\to X
\\
(\Delta_{m} \to \Delta_{r_{0}} \to \dots \to \Delta_{r_{n}} \to X) &\mapsto (\Delta_{m} \to X).
\end{aligned}
$$
Posons $KX = \delta_{\Delta}^{*}SX$. Cela nous fournit un diagramme
$$
\xymatrix{
\UnNerf i_{\Delta} X
&
KX
\ar[l]
\ar[r]
&
X
}
$$
dans $\EnsSimp$ dont les deux flèches sont naturelles en $X$ ; il reste à vérifier que ce sont des équivalences faibles et, pour cela, on fait de nouveau appel à la proposition \ref{LemmeBisimplicialDelta}. 

Montrons d'abord que le morphisme d'ensembles simpliciaux
$$
(SX)_{m, \bullet} \to X_{m,\bullet}
$$
est une équivalence faible pour tout entier $m \geq 0$. Il s'identifie à 
\begin{equation}
\coprod_{\varphi : \Delta_{m} \to X} \UnNerf{((\Delta_{m} \backslash \Delta / X)_{\varphi}} \to \UnCatPonct),
\end{equation}
la catégorie $(\Delta_{m} \backslash \Delta / X)_{\varphi}$ ayant pour objets les diagrammes
$$
\xymatrix{
\Delta_{m}
\ar[r]^{\pi}
&
\Delta_{r}
\ar[r]^{\rho}
&
X
}
$$
tels que $\rho \pi = \varphi$ et pour morphismes de 
$$
\xymatrix{
\Delta_{m}
\ar[r]^{\pi}
&
\Delta_{r}
\ar[r]^{\rho}
&
X
}
$$
vers 
$$
\xymatrix{
\Delta_{m}
\ar[r]^{\chi}
&
\Delta_{s}
\ar[r]^{\psi}
&
X
}
$$
les morphismes $\mu : \Delta_{r} \to \Delta_{s}$ de $\Delta$ tels que le diagramme
$$
\xymatrix{
&
\Delta_{r}
\ar[dr]^{\rho}
\ar[dd]^{\mu}
\\
\Delta_{m}
\ar[ur]^{\pi}
\ar[dr]_{\chi}
&&
X
\\
&
\Delta_{s}
\ar[ur]_{\psi}
}
$$
soit commutatif. Cette catégorie (qui ne possède pas d'objet final en général) admet un objet initial défini par le diagramme
$$
\xymatrix{
\Delta_{m}
\ar[r]^{1_{\Delta_{m}}}
&
\Delta_{m}
\ar[r]^{\varphi}
&
X.
}
$$
Par conséquent, le foncteur $(\Delta_{m} \backslash \Delta / X)_{\varphi} \to \UnCatPonct$ est dans $\UnLocFondMin$. Le morphisme d'ensembles simpliciaux $(SX)_{m, \bullet} \to X_{m,\bullet}$ s'identifie donc à une somme d'équivalences faibles ; c'est donc une équivalence faible. En vertu de la proposition \ref{LemmeBisimplicialDelta}, le morphisme d'ensembles simpliciaux 
$$
KX \to \UnNerf i_{\Delta} X
$$ 
est donc une équivalence faible. 

Il suffit maintenant de vérifier que le morphisme d'ensembles simpliciaux
$$
\xymatrix{
(SX)_{\bullet,n} 
\ar[r]
&
(\UnNerf i_{\Delta} X)_{\bullet,n}
}
$$
est une équivalence faible pour tout entier $n \geq 0$. Ce morphisme s'identifie à 
$$
\coprod_{\Delta_{r_{0}} \to \dots \to \Delta_{r_{n}} \to X} (\Delta_{r_{0}} \to \ast).
$$
Pour tout entier $r_{0} \geq 0$, la flèche $\Delta_{r_{0}} \to \ast$ est une équivalence faible. Le morphisme
$$
\xymatrix{
(SX)_{\bullet,n} 
\ar[r]
&
(\UnNerf i_{\Delta} X)_{\bullet,n}
}
$$
s'identifie donc à une somme d'équivalences faibles ; c'est donc une équivalence faible, ce qui termine la démonstration.
\end{proof}

\begin{paragr}
Pour toute catégorie $C$ et toute classe $W_{C}$ de morphismes de $C$, on notera $\Localisation{C}{W_{C}}$\index[not]{WCC@$\Localisation{C}{W_{C}}$} la catégorie localisée de $C$ relativement à $W_{C}$.
\end{paragr}

\begin{theo}[Quillen]\label{EquiEnsSimpCat}
On a l'égalité 
$$
\EquiQuillen = {i_{\Delta}}^{-1}(\UnLocFondMin).
$$
De plus, les foncteurs nerf $\UnNerf : \Cat \to \EnsSimp$ et $i_{\Delta} : \EnsSimp \to \Cat$ induisent des équivalences de catégories quasi-inverses l'une de l'autre
$$
\overline{\UnNerf} : \Localisation{\Cat}{\UnLocFondMin} \to \Localisation{\EnsSimp}{\EquiQuillen}
$$
et
$$
\overline{i_{\Delta}} : \Localisation{\EnsSimp}{\EquiQuillen} \to \Localisation{\Cat}{\UnLocFondMin}.
$$
\end{theo}

\begin{proof}
La première affirmation est une conséquence immédiate de la proposition \ref{MinouDrouet}, par deux arguments de « $2$ sur $3$ » appliqués au diagramme
$$
\xymatrix{
\UnNerf i_{\Delta} X
\ar[d]_{\UnNerf i_{\Delta} (f)}
&&
KX
\ar[ll]_{\sim}
\ar[rr]^{\sim}
\ar[d]^{K(f)}
&&
X
\ar[d]^{f}
\\
\UnNerf i_{\Delta} Y
&&
KY
\ar[ll]^{\sim}
\ar[rr]_{\sim}
&&
Y
}
$$
pour un morphisme arbitraire $f : X \to Y$ d'ensembles simpliciaux. Cela montre au passage que le foncteur $K$ respecte les équivalences faibles. L'existence de foncteurs induits $\overline{\UnNerf} : \Localisation{\Cat}{\UnLocFondMin} \to \Localisation{\EnsSimp}{\EquiQuillen}$, $\overline{i_{\Delta}} : \Localisation{\EnsSimp}{\EquiQuillen} \to \Localisation{\Cat}{\UnLocFondMin}$ et $\overline{K} : \Localisation{\EnsSimp}{\EquiQuillen} \to \Localisation{\EnsSimp}{\EquiQuillen}$ est immédiate. Le zigzag de la proposition \ref{MinouDrouet} fournit un zigzag 
$$
\xymatrix{
\overline{\UnNerf} \overline{i_{\Delta}}
&
\overline{K}
\ar[l]_{\sim}
\ar[r]^{\sim}
&
1_{\Localisation{\EnsSimp}{\EquiQuillen}} 
}
$$
d'isomorphismes naturels, donc un isomorphisme entre $\overline{\UnNerf} \overline{i_{\Delta}}$ et $1_{\Localisation{\EnsSimp}{\EquiQuillen}}$. Il reste à vérifier l'existence d'un isomorphisme naturel entre $1_{\Localisation{\Cat}{\UnLocFondMin}}$ et $\overline{i_{\Delta}} \overline{\UnNerf}$. Pour cela, il suffit de vérifier qu'il existe un morphisme de foncteurs $i_{\Delta} \UnNerf \Rightarrow 1_{\Cat}$ $\UnLocFondMin$\nobreakdash-asphérique argument par argument, c'est-à-dire un foncteur $\UnLocFondMin$\nobreakdash-asphérique $\Delta/\UnNerf{A} \to A$ naturel en $A \in \Objets{\Cat}$. C'est précisément un tel foncteur que fournit la proposition \ref{UnSupAspherique}. 
\end{proof}

\begin{paragr}
Soient $I$ et $J$ deux petites catégories et $F(\bullet, \bullet) : I \times J \to Cat$ un foncteur. On peut considérer les foncteurs 
$$
\begin{aligned}
J &\to \Cat
\\
j &\mapsto \DeuxInt{I}F(\bullet, j)
\end{aligned}
$$
et
$$
\begin{aligned}
I &\to \Cat
\\
i &\mapsto \DeuxInt{J}F(i, \bullet).
\end{aligned}
$$
\end{paragr}

\begin{lemme}\label{Fubini}
Soient $I$ et $J$ deux petites catégories et $F : I \times J \to Cat$ un foncteur. Il existe des isomorphismes canoniques
$$
\DeuxInt{I \times J} F(\bullet, \bullet) \simeq \DeuxInt{I} \left(i \mapsto \DeuxInt{J} F(i, \bullet)\right) \simeq \DeuxInt{J} \left(j \mapsto \DeuxInt{I} F(\bullet, j)\right).
$$
\end{lemme}

\begin{proof}
C'est immédiat.
\end{proof}

\emph{On suppose fixé un localisateur fondamental $\UnLocFond{W}$ de $\Cat$.} 

\begin{lemme}\label{UnAdjointsW}
Un morphisme de $\Cat$ admettant un adjoint à droite (\emph{resp.} à gauche) est une équivalence faible. 
\end{lemme}

\begin{proof}
Soit $u : A \to B$ un morphisme de $\Cat$ admettant un adjoint à droite $v$. De la bijection
$$
Hom_{B}(u(a),b) \simeq Hom_{A}(a,v(b))
$$
fonctorielle en $a \in \Objets{A}$ et $b \in \Objets{B}$, on déduit un isomorphisme de catégories $A/b \simeq A/v(b)$ pour tout objet $b$ de $B$. Comme la catégorie $A/v(b)$ admet un objet final, elle est asphérique. Il en va donc de même de $A/b$. Le foncteur $u$ est donc asphérique, donc une équivalence faible. Le cas d'un morphisme de $\Cat$ admettant un adjoint à droite s'en déduit (du fait de l'existence d'une transformation naturelle $1_{A} \Rightarrow vu$, par exemple), ou se traite de façon analogue. 
\end{proof}

\begin{rem}
Le lemme \ref{UnAdjointsW} n'est autre que \cite[proposition 1.1.9.]{THG}, dont nous reproduisons la démonstration. Comme nous renvoyons souvent le lecteur à \cite{LFM}, précisons que c'est également \cite[corollaire 1.1.10.]{LFM}, mais que cela ne résulte pas de ce qui précède directement ce corollaire dans \cite{LFM}.
\end{rem}

\begin{prop}\label{IntegrationWParArguments}
Soient $A$ une petite catégorie, $F$ et $G$ deux foncteurs de $A$ vers $\Cat$ et \mbox{$\sigma : F \Rightarrow G$} une transformation naturelle. Supposons que, pour tout objet $a$ de $A$, \mbox{$\sigma_{a} : F(a) \to G(a)$} soit une équivalence faible. Alors $\DeuxInt{A}\sigma : \DeuxInt{A}F \to \DeuxInt{A}G$ est une équivalence faible. 
\end{prop}

\begin{proof}
Considérons le diagramme commutatif
$$
\xymatrix{
\DeuxInt{A}F
\ar[rr]^{\DeuxInt{A}\sigma}
\ar[dr]_{P_{F}}
&&
\DeuxInt{A}G
\ar[dl]^{P_{G}}
\\
&
A
&,
}
$$
dans lequel $P_{F}$ et $P_{G}$ désignent les projections canoniques. Pour tout objet $a$ de $A$, on a un diagramme commutatif
$$
\xymatrix{
F(a)
\ar[rr]^{\sigma_{a}}
\ar[d]
&&
G(a)
\ar[d]
\\
(\DeuxInt{A}F)/a
\ar[rr]_{(\DeuxInt{A}\sigma)/a}
&&
(\DeuxInt{A}G)/a
}
$$
dont les flèches verticales admettent un adjoint à droite (les projections $\DeuxInt{A}F \to A$ et $\DeuxInt{A}G \to A$ étant des préopfibrations dont les fibres au-dessus de $a$ s'identifient à $F(a)$ et $G(a)$ respectivement) et sont donc des équivalences faibles en vertu du lemme \ref{UnAdjointsW}. En vertu des hypothèses, un argument de « $2$ sur $3$ » permet d'affirmer que $(\DeuxInt{A}\sigma)/a$ est une équivalence faible, d'où le résultat. 
\end{proof}

\begin{df}
On dit qu'un morphisme de préfaisceaux sur une petite catégorie $A$ est une \emph{équivalence faible}\index{equivalence faible de préfaisceaux@équivalence faible de préfaisceaux} (de préfaisceaux) si son image par le foncteur $i_{A}$ est une équivalence faible (de $\Cat$). On notera $W_{\widehat{A}}$\index[not]{WA@$W_{\widehat{A}}$} la classe des équivalences faibles de préfaisceaux sur $A$. 
\end{df}

\begin{lemme}\label{Queneau}
Soient $A$ et $B$ deux petites catégories et $f(\bullet, \bullet) : X(\bullet, \bullet) \to Y(\bullet, \bullet)$ un morphisme de ${\widehat{A \times B}}$. 

Supposons que, pour tout objet $a$ de $A$, le morphisme $f(a,\bullet) : X(a,\bullet) \to Y(a,\bullet)$ soit dans $\UnLocFond{W}_{\widehat{B}}$. Alors, $f(\bullet, \bullet)$ est dans $\UnLocFond{W}_{\widehat{A \times B}}$. 

Supposons que, pour tout objet $b$ de $B$, le morphisme $f(\bullet,b) : X(\bullet,b) \to Y(\bullet,b)$ soit dans $\UnLocFond{W}_{\widehat{A}}$. Alors, $f(\bullet, \bullet)$ est dans $\UnLocFond{W}_{\widehat{A \times B}}$. 
\end{lemme}

\begin{proof}
Il suffit bien entendu de démontrer la première assertion. Supposons que, pour tout objet $a$ de $A$, le morphisme $f(a, \bullet) : X(a,\bullet) \to Y(a,\bullet)$ soit dans $\UnLocFond{W}_{\widehat{B}}$, c'est-à-dire que le foncteur $i_{B} (f(a, \bullet))$ soit dans $\UnLocFond{W}$. Il s'agit de montrer que $f(\bullet, \bullet)$ est dans $\UnLocFond{W}_{\widehat{A \times B}}$, c'est-à-dire que le morphisme $i_{A \times B}(f(\bullet, \bullet))$ est dans $\UnLocFond{W}$. Un énoncé dual du lemme \ref{Fubini} permet d'identifier
$
i_{A \times B} (f(\bullet, \bullet)) 
$
à
$
i_{A} (i_{B} (f(a,\bullet)))
$. 
L'assignation 
$
a \mapsto i_{B} (f(a,\bullet)) 
$
définissant un morphisme de foncteurs de
$$
\begin{aligned}
A &\to \Cat
\\
a &\mapsto i_{B}X(a,\bullet)
\end{aligned}
$$
vers
$$
\begin{aligned}
A &\to \Cat
\\
a &\mapsto i_{B}Y(a,\bullet),
\end{aligned}
$$
la proposition \ref{IntegrationWParArguments} permet de conclure. 
\end{proof}

\begin{df}\label{DefTotalementAspherique}
On dit qu'une petite catégorie $A$ est \emph{totalement $\UnLocFond{W}$\nobreakdash-asphérique}, ou plus simplement \emph{totalement asphérique}\index{totalement asphérique (petite catégorie)}, si elle est asphérique et si le foncteur diagonal 
$$
\begin{aligned}
A &\to A \times A
\\
a &\mapsto (a,a)
\end{aligned}
$$ 
est asphérique. 
\end{df}

\begin{rem}
Pour qu'une petite catégorie $A$ soit totalement asphérique, il suffit qu'elle soit non vide et que le foncteur diagonal $A \to A \times A$ soit asphérique. Pour une démonstration, ainsi que des conditions équivalentes de totale asphéricité, on renvoie à \cite[proposition 1.6.1]{THG}.
\end{rem}

\begin{prop}\label{CasParticulier1.2.9.THG}
Soient $A$ et $B$ deux catégories asphériques et $u : A \to B$ un morphisme de $\Cat$. Les deux conditions suivantes sont équivalentes : 
\begin{itemize}
\item[(i)]
Le foncteur $u$ est asphérique.
\item[(ii)] 
Pour toute équivalence faible $\varphi$ de $\widehat{B}$, $u^{*}(\varphi)$ est une équivalence faible de $\widehat{A}$. 
\end{itemize}
\end{prop}

\begin{proof}
C'est une partie de l'énoncé de \cite[proposition 1.2.9]{THG}. 
\end{proof}

\begin{prop}\label{LemmeBisimplicial}
Soient $A$ une petite catégorie totalement asphérique et $f$ un morphisme de $\widehat{A \times A}$. Si, pour tout objet $a$ de $A$, le morphisme $f(a,\bullet)$ est dans $\UnLocFond{W}_{\widehat{A}}$, alors $\delta_{A}^{*}(f)$ est dans $\UnLocFond{W}_{\widehat{A}}$.
\end{prop}

\begin{proof}
En vertu des hypothèses et du lemme \ref{Queneau}, $f$ est dans $\UnLocFond{W}_{\widehat{A \times A}}$. Le foncteur diagonal $\delta_{A} : A \to A \times A$ étant asphérique, la conclusion résulte de la proposition \ref{CasParticulier1.2.9.THG}.
\end{proof} 

\begin{prop}\label{DeltaTotalementAspherique}
La catégorie des simplexes $\Delta$ est totalement asphérique. 
\end{prop}

\begin{proof}
Voir la démonstration de \cite[proposition 1.6.13]{THG}.
\end{proof}

\begin{rem}
Insistons sur le fait que les propositions \ref{LemmeBisimplicial} et \ref{DeltaTotalementAspherique} \emph{ne fournissent pas} une démonstration de la proposition \ref{LemmeBisimplicialDelta}. En effet, pour démontrer que l'image d'un morphisme simplicial par $\UnNerf$ est une équivalence faible si et seulement si son image par $i_{\Delta}$ en est une (voir le théorème \ref{EquiEnsSimpCat}), nous avons utilisé la proposition \ref{MinouDrouet}, dont la démonstration repose sur la proposition \ref{LemmeBisimplicialDelta}.
\end{rem}

On conclut cette section par un rappel du « lemme d'homotopie » de \cite{THG}, que nous utiliserons plus loin. Cela nous permet d'introduire quelques notions utiles, dont nous recopions presque mot pour mot les définitions figurant dans \cite{THG}. 

\begin{paragr}
Soient $M$ une petite catégorie et $\UnLocFond{W}$ une partie de $\UnCell{M}$. On dit qu'un morphisme $u : A \to B$ de $M$ est \emph{universellement dans\index{universellement dans} $\UnLocFond{W}$} si les conditions suivantes sont vérifiées.  
\begin{itemize}
\item[$(i)$] Le morphisme $u$ est \emph{quarrable}\index{quarrable} ; autrement dit, pour tout morphisme $B' \to B$, le produit fibré $B' \times_{B} A$ est représentable dans $M$.
\item[$(ii)$] Pour tout carré cartésien
$$
\xymatrix{
A' 
\ar[r]
\ar[d]_{u'}
&
A
\ar[d]^{u}
\\
B'
\ar[r]
&
B
}
$$
dans $M$, le morphisme $u'$ est dans $\UnLocFond{W}$. 
\end{itemize}
\end{paragr}

\begin{paragr}\label{DefHomotopie}
On suppose dans la suite que $M$ admet des produits finis. Soient $\mathbb{I} = (I, \partial_{0}, \partial_{1})$ un segment de $M$ et $f, g : X \to Y$ deux morphismes de $M$. On dit que $f$ est \emph{$\mathbb{I}$\nobreakdash-homotope de façon élémentaire\index{homotope de façon élémentaire (relativement à un segment)} à $g$}, ou encore \emph{élémentairement $\mathbb{I}$\nobreakdash-homotope à $g$}, voire encore \emph{élémentairement homotope\index{elémentairement homotope@élémentairement homotope} à $g$} ou qu'il existe une \emph{homotopie élémentaire\index{homotopie élémentaire (relativement à un segment)} de $u$ vers $v$} en l'absence d'ambiguïté, s'il existe un morphisme $h : I \times X \to Y$ tel que $f = h (\partial_{0}, 1_{X})$ et $g = h (\partial_{1}, 1_{X})$, autrement dit tel que le diagramme
$$
\xymatrix{
&
I \times X
\ar[dd]^{h}
\\
X
\ar[ur]^{\partial_{0} \times 1_{X}}
\ar[dr]_{f}
&&
X
\ar[ul]_{\partial_{1} \times 1_{X}}
\ar[dl]^{g}
\\
&
Y
}
$$
soit commutatif, et on dit que $h$ est une \emph{$\mathbb{I}$\nobreakdash-homotopie}\index{homotopie (relativement à un segment)} (ou, par abus de langage, une $I$\nobreakdash-homotopie) de $f$ vers $g$. 
\end{paragr}

\begin{rem}
Le lemme \ref{DeuxTransFoncLax} stipule donc que, s'il existe une \DeuxTransformationLax{} ou une \DeuxTransformationCoLax{} d'un \DeuxFoncteurLax{} $u$ vers un \DeuxFoncteurLax{} $v$, alors $u$ est élémentairement $([1], \{0\}, \{1\})$\nobreakdash-homotope à $v$ dans $\DeuxCatLax$, catégorie admettant des produits qui sont les mêmes que ceux de la catégorie $\DeuxCat$. 
\end{rem}

\begin{paragr}\label{DefHomotopieBis}
Soit $\mathcal{I}$ un ensemble de segments de $M$. On appelle relation de $\mathcal{I}$\nobreakdash-homotopie\index{homotopie (relativement à un ensemble de segments)} la relation d'équivalence engendrée par la relation « il existe un segment $\mathbb{I}$ appartenant à $\mathcal{I}$ tel que $f$ soit $\mathbb{I}$\nobreakdash-homotope à $g$ de façon élémentaire ». On dit que deux morphismes sont \emph{$\mathcal{I}$\nobreakdash-homotopes} (ou \emph{$\mathbb{I}$\nobreakdash-homotopes} si $\mathcal{I} = \{ \mathbb{I} \}$) s'ils sont en relation par la relation de $\mathcal{I}$\nobreakdash-homotopie. Si $\mathcal{I} = \{ (I, \partial_{0}, \partial_{1}) \}$, on dira parfois, par abus de langage, qu'ils sont $I$\nobreakdash-homotopes. 
\end{paragr}

\begin{paragr}\label{CompatibiliteHomotopie}
On vérifie immédiatement que la relation de $\mathcal{I}$\nobreakdash-homotopie est compatible à la composition et au produit de morphismes. 
\begin{itemize}
\item[$(a)$]
Si $f_{0}, f_{1} : X \to Y$ et $g_{0}, g_{1} : Y \to Z$ sont des morphismes de $M$ tels que $f_{0}$ soit $\mathcal{I}$\nobreakdash-homotope à $f_{1}$ et $g_{0}$ soit $\mathcal{I}$\nobreakdash-homotope à $g_{1}$, alors $g_{0} f_{0}$ est $\mathcal{I}$\nobreakdash-homotope à $g_{1} f_{1}$. 
\item[$(b)$]
Si $f_{0}, f_{1} : X \to Y$ et $f'_{0}, f'_{1} : X' \to Y'$ sont des morphismes de $M$ tels que $f_{0}$ soit $\mathcal{I}$\nobreakdash-homotope à $f_{1}$ et $f'_{0}$ soit $\mathcal{I}$\nobreakdash-homotope à $f'_{1}$, alors $f_{0} \times f'_{0}$ est $\mathcal{I}$\nobreakdash-homotope à $f_{1} \times f'_{1}$. 
\end{itemize}
\end{paragr}

\begin{paragr}
On dit qu'un objet $X$ de $M$ est $\mathcal{I}$\nobreakdash-contractile\index{contractile} (ou $\mathbb{I}$\nobreakdash-contractile si $\mathcal{I} = \{ \mathbb{I} \}$) si l'identité de $X$ est $\mathcal{I}$\nobreakdash-homotope à un endomorphisme \emph{constant} de $X$ (c'est-à-dire se factorisant par l'objet final $e_{M}$ de $M$). Si $\mathcal{I} = \{ (I, \partial_{0}, \partial_{1}) \}$, on parlera parfois, par abus de langage, d'objet $I$\nobreakdash-contractile. 
\end{paragr}

\begin{lemme}[Lemme d'homotopie]\label{LemmeHomotopieTHG}
Soient $M$ une catégorie admettant des produits finis, $\UnLocFond{W}$ une partie faiblement saturée de $\UnCell{M}$ et $\mathcal{I}$ un ensemble de segments de $M$ tel que, pour tout segment $\mathbb{I} = \{ (I, \partial_{0}, \partial_{1}) \}$ appartenant à $\mathcal{I}$, le morphisme canonique $p_{I} : I \to e_{M}$, où $e_{M}$ désigne l'objet final de $M$, soit universellement dans $\UnLocFond{W}$ (autrement dit, pour tout objet $Z$ de $M$, la projection canonique $I \times Z \to Z$ est dans $\UnLocFond{W}$). Si $f, g : X \to Y$ sont deux morphismes $\mathcal{I}$\nobreakdash-homotopes de $M$, alors :
\begin{itemize}
\item[(a)] $f$ est dans $\UnLocFond{W}$ si et seulement si $g$ l'est ;
\item[(b)] $\gamma(f) = \gamma(g)$, où $\gamma : M \to \Localisation{M}{\UnLocFond{W}}$ désigne le foncteur canonique de localisation ;
\item[(c)] si $f$ est un isomorphisme et $g$ un morphisme constant (c'est-à-dire qu'il se factorise par l'objet final $e_{M}$), alors les morphismes canoniques $X \to e_{M}$ et $Y \to e_{M}$ sont universellement dans $\UnLocFond{W}$. 
\end{itemize}
\end{lemme}

\begin{proof}
C'est \cite[lemme 1.4.6]{THG}. Pour démontrer $(a)$ et $(b)$, on peut supposer qu'il existe un segment $\mathbb{I} = (I, \partial_{0}, \partial_{1})$ appartenant à $\mathcal{I}$ et une $\mathbb{I}$\nobreakdash-homotopie $h : I \times X \to Y$ de $f$ vers $g$, de sorte que le diagramme
$$
\xymatrix{
&
I \times X
\ar[dd]^{h}
\\
X
\ar[ur]^{\partial_{0} \times 1_{X}}
\ar[dr]_{f}
&&
X
\ar[ul]_{\partial_{1} \times 1_{X}}
\ar[dl]^{g}
\\
&
Y
}
$$
soit commutatif. En vertu des hypothèses, la projection canonique $p_{2} : I \times X \to X$ est dans $\UnLocFond{W}$. Comme c'est une rétraction commune à $\partial_{0} \times 1_{X}$ et $\partial_{1} \times 1_{X}$, ces deux morphismes sont dans $\UnLocFond{W}$. On en déduit l'assertion $(a)$ par deux arguments consécutifs de « $2$ sur $3$ ». 

De plus, des égalités
$$
p_{2} (\partial_{0} \times 1_{X}) = 1_{X} = p_{2} (\partial_{1} \times 1_{X})
$$
et du fait que $\gamma(p_{2})$ est un isomorphisme (puisque $p_{2}$ est dans $\UnLocFond{W}$ par hypothèse) on déduit, par fonctorialité de $\gamma$,
$$
\gamma(\partial_{0} \times 1_{X}) = (\gamma(p_{2}))^{-1} = \gamma(\partial_{1} \times 1_{X}).
$$
L'assertion $(b)$ en résulte immédiatement. 

Montrons l'assertion $(c)$, dont la preuve semble un peu rapide dans \cite{THG}. Par hypothèse, $g = sp_{X}$, où $p_{X} : X \to e_{M}$ désigne le morphisme canonique et $s : e_{M} \to Y$ est un morphisme de $M$. Il s'agit de montrer que, pour tout objet $T$ de $M$, la projection canonique $X \times T \to T$, qui s'identifie à $p_{X} \times 1_{T}$, est dans $\UnLocFond{W}$. La relation d'équivalence « être homotopes » étant compatible au produit (paragraphe \ref{CompatibiliteHomotopie}), $f \times 1_{T}$ est homotope à $g \times 1_{T} = (s \times 1_{T}) (p_{X} \times 1_{T})$. Comme $f$ est un isomorphisme, il en est de même de $f \times 1_{T}$, qui est donc dans $\UnLocFond{W}$. En vertu de l'assertion $(a)$, le morphisme $(s \times 1_{T}) (p_{X} \times 1_{T})$ est donc dans $\UnLocFond{W}$. Comme $f^{-1}$ est un isomorphisme, $(f^{-1} \times 1_{T})$ est un isomorphisme, donc est dans $\UnLocFond{W}$. De ce qui précède, on déduit que $(f^{-1} \times 1_{T}) (s \times 1_{T}) (p_{X} \times 1_{T})$ est dans $\UnLocFond{W}$. Comme $(p_{X} \times 1_{T}) (f^{-1} \times 1_{T}) (s \times 1_{T}) = 1_{e_{M} \times T}$ est dans $\UnLocFond{W}$, $p_{X} \times 1_{T}$ est dans $\UnLocFond{W}$ en vertu de la saturation faible de $\UnLocFond{W}$, ce qui montre que le morphisme $p_{X} : X \to e_{M}$ est universellement dans $\UnLocFond{W}$. Considérons le diagramme commutatif
$$
\xymatrix{
X \times T 
\ar[rr]^{f \times 1_{T}}
\ar[dr]_{p_{X} \times 1_{T}}
&&
Y \times T
\ar[dl]^{p_{Y} \times 1_{T}}
\\
&
e_{M} \times T
&.
}
$$
On a vu que $f \times 1_{T}$ et $p_{X} \times 1_{T}$ étaient dans $\UnLocFond{W}$. Un argument de « $2$ sur $3$ » permet de conclure que le morphisme $p_{Y} \times 1_{T}$ l'est aussi. Autrement dit, le morphisme canonique $p_{Y} : Y \to e_{M}$ est universellement dans $\UnLocFond{W}$. 
\end{proof}

\begin{prop}\label{1.4.8.THG}
Soient $M$ une catégorie admettant des produits finis, $W$ une partie faiblement saturée de $\UnCell{M}$ et $\mathcal{I}$ un ensemble de segments de $M$ tel que, pour tout segment $\mathbb{I} = (I, \partial_{0}, \partial_{1})$ appartenant à $\mathcal{I}$, le morphisme canonique $p_{I} : I \to e_{M}$ soit universellement dans $W$, $e_{M}$ désignant l'objet final de $M$. Si $X$ est un objet $\mathcal{I}$\nobreakdash-contractile de $M$, alors le morphisme canonique $X \to e_{M}$ est universellement dans $W$. 
\end{prop}

\begin{proof}
C'est une conséquence immédiate de l'assertion $(c)$ du lemme \ref{LemmeHomotopieTHG}.
\end{proof}

\section{Nerfs simpliciaux}\label{SectionNerfs}

\begin{paragr}\label{IntroNerfs}
On définit maintenant divers foncteurs associant à une \deux{}catégorie — que l'on supposera petite — un objet simplicial de $\Cat$. Certains d'entre eux sont à valeurs dans la catégorie des ensembles simpliciaux. On verra qu'ils sont tous reliés par des morphismes induisant des équivalences d'homotopie. Autrement dit, ils sont tous homotopiquement équivalents et, dans l'étude homotopique des \deux{}catégories que nous entreprenons, on pourra indifféremment considérer n'importe lequel de ces nerfs suivant les situations et notre humeur du moment. La référence principale pour cette section est l'article \cite{CCG}, dont nous ne faisons essentiellement que reprendre les résultats. Si l'on ne se limite pas aux travaux de nature homotopique, les définitions de nerfs en dimension supérieure remontent beaucoup plus loin. Sans nullement prétendre à l'exhaustivité, rendue plus délicate encore par le caractère difficilement accessible de certains textes, voici quelques références, en attendant la publication d'une histoire de la théorie des catégories. La définition générale maintenant standard du nerf en dimension quelconque, due à Street, apparaît dans \cite{StreetAOS}, dont l'introduction cite des travaux de Roberts que nous n'avons pu consulter. On peut également se reporter à \cite{StreetHandbook} et, pour les bicatégories, au travail de Duskin \cite{Duskin}. 
\end{paragr}

\begin{df}\label{DefinitionNerfs}
Soient $\mathdeuxcat{A}$ une \deux{}catégorie et $m \geq 0$ un entier. 
\begin{itemize}
\item[$(i)$]
On note 
$$
FonLax([m], \mathdeuxcat{A})\index[not]{FonLaxmA@$FonLax([m], \mathdeuxcat{A})$}
$$
l'ensemble des \DeuxFoncteursLax{} de $[m]$ vers $\mathdeuxcat{A}$ et $\NerfLax{\mathdeuxcat{A}}$ l'ensemble simplicial
$$
\begin{aligned}
\NerfLax{\mathdeuxcat{A}} : \Delta^{op} &\to \Ens
\\
[m] &\mapsto FonLax([m], \mathdeuxcat{A})
\end{aligned}
$$
dont les faces et dégénérescences sont définies de façon « évidente ».

Cela permet de définir un foncteur \emph{nerf lax}\index{nerf lax}
$$
\begin{aligned}
\NerfLax\index[not]{Nl@$\NerfLax$} : \DeuxCatLax &\to \EnsSimp
\\
\mathdeuxcat{A} &\mapsto \NerfLax \mathdeuxcat{A}
\\
u &\mapsto \NerfLax(u).
\end{aligned}
$$


\item[$(ii)$]
On note 
$$
FonLaxNor([m], \mathdeuxcat{A})\index[not]{FonLaxNormA@$FonLaxNor([m], \mathdeuxcat{A})$}
$$
l'ensemble des \DeuxFoncteursLax{} normalisés de $[m]$ vers $\mathdeuxcat{A}$ et $\NerfLaxNor{\mathdeuxcat{A}}$ l'ensemble simplicial 
$$
\begin{aligned}
\NerfLaxNor{\mathdeuxcat{A}} : \Delta^{op} &\to \Ens
\\
[m] &\mapsto FonLaxNor([m], \mathdeuxcat{A})
\end{aligned}
$$
dont les faces et dégénérescences sont définies de façon « évidente ». 

Cela permet de définir un foncteur \emph{nerf lax normalisé}\index{nerf lax normalisé}
$$
\begin{aligned}
\NerfLaxNor\index[not]{Nln@$\NerfLaxNor$} : \DeuxCat &\to \EnsSimp
\\
\mathdeuxcat{A} &\mapsto \NerfLaxNor \mathdeuxcat{A}
\\
u &\mapsto \NerfLaxNor(u).
\end{aligned}
$$


\item[$(iii)$]
On note 
$$
\underline{FonLax}([m], \mathdeuxcat{A})\index[not]{FonLaxmASouligne@$\underline{FonLax}([m], \mathdeuxcat{A})$}
$$
la catégorie dont les objets sont les \DeuxFoncteursLax{} de $[m]$ vers $\mathdeuxcat{A}$ et dont les morphismes sont les \DeuxTransformationsLax{} relatives aux objets entre tels \DeuxFoncteursLax{} et $\NerfCatLax{\mathdeuxcat{A}}$ l'objet simplicial de $\Cat$ 
$$
\begin{aligned}
\NerfCatLax{\mathdeuxcat{A}} : \Delta^{op} &\to \Cat
\\
[m] &\mapsto \underline{FonLax}([m], \mathdeuxcat{A})
\end{aligned}
$$
dont les faces et dégénérescences sont définies de façon « évidente ». 

Cela permet de définir un foncteur \emph{nerf lax catégorique}\index{nerf lax catégorique}
$$
\begin{aligned}
\NerfCatLax\index[not]{NlSouligne@$\NerfCatLax$} : \DeuxCat &\to \CatHom{\CAT}{\Delta^{op}}{\Cat} 
\\
\mathdeuxcat{A} &\mapsto \NerfCatLax \mathdeuxcat{A}
\\
u &\mapsto \NerfCatLax(u)
\end{aligned}
$$
en notant $\CAT$\index[not]{CZAT@$\CAT$} la catégorie des catégories (non nécessairement petites).

\item[$(iv)$]
On note 
$$
\underline{FonLaxNor}([m], \mathdeuxcat{A})\index[not]{FonLaxNormA@$\underline{FonLaxNor}([m], \mathdeuxcat{A})$}
$$
la catégorie dont les objets sont les \DeuxFoncteursLax{} normalisés de $[m]$ vers $\mathdeuxcat{A}$ et dont les morphismes sont les \DeuxTransformationsLax{} relatives aux objets entre tels \DeuxFoncteursLax{} normalisés et $\NerfCatLaxNor{\mathdeuxcat{A}}$ l'objet simplicial de $\Cat$ 
$$
\begin{aligned}
\NerfCatLaxNor{\mathdeuxcat{A}} : \Delta^{op} &\to \Cat
\\
[m] &\mapsto \underline{FonLaxNor}([m], \mathdeuxcat{A})
\end{aligned}
$$
dont les faces et dégénérescences sont définies de façon « évidente ». 

Cela permet de définir un foncteur \emph{nerf lax normalisé catégorique}\index{nerf lax normalisé catégorique}
$$
\begin{aligned}
\NerfCatLaxNor\index[not]{NlnSouligne@$\NerfCatLaxNor$} : \DeuxCat &\to \CatHom{CAT}{\Delta^{op}}{\Cat} 
\\
\mathdeuxcat{A} &\mapsto \NerfCatLaxNor \mathdeuxcat{A}
\\
u &\mapsto \NerfCatLaxNor(u).
\end{aligned}
$$

\item[$(v)$]
On note 
$$
\underline{Fon}([m], \mathdeuxcat{A})\index[not]{FonmASouligne@$\underline{Fon}([m], \mathdeuxcat{A})$}
$$
la catégorie dont les objets sont les \DeuxFoncteursStricts{} de $[m]$ vers $\mathdeuxcat{A}$ et dont les morphismes sont les \DeuxTransformationsLax{} relatives aux objets entre tels \DeuxFoncteursStricts{}. Autrement dit, $\underline{Fon}([m], \mathdeuxcat{A})$ est la catégorie 
$$
\coprod_{\substack{(a_{0}, \dots, a_{m}) \in (\Objets{\mathdeuxcat{A}})^{m+1}}} \CatHom{\mathdeuxcat{A}}{a_{0}}{a_{1}} \times \dots \times \CatHom{\mathdeuxcat{A}}{a_{m-1}}{a_{m}}.
$$
On note 
$\NerfHom{\mathdeuxcat{A}}$ l'objet simplicial de $\Cat$ 
$$
\begin{aligned}
\NerfHom{\mathdeuxcat{A}} : \Delta^{op} &\to \Cat
\\
[m] &\mapsto \underline{Fon}([m], \mathdeuxcat{A})
\end{aligned}
$$
dont les faces et dégénérescences sont définies de façon « évidente ». 

Cela permet de définir un foncteur \emph{nerf strict}\index{nerf strict}
$$
\begin{aligned}
\NerfHom\index[not]{NSouligne@$\NerfHom$} : \DeuxCat &\to \CatHom{CAT}{\Delta^{op}}{\Cat} 
\\
\mathdeuxcat{A} &\mapsto \NerfHom \mathdeuxcat{A}
\\
u &\mapsto \NerfHom(u).
\end{aligned}
$$

\end{itemize}
\end{df}

\begin{rem}\label{NerfLaxNorFonctoriel}
Nous avons considéré le domaine de définition des nerfs $\NerfLaxNor$, $\NerfCatLax$ et $\NerfCatLaxNor$ comme étant $\DeuxCat$. Nous ne les considérerons pas, sauf mention contraire, comme des foncteurs de domaine la catégorie dont les objets sont les \deux{}catégories et dont les morphismes sont les \DeuxFoncteursLax{} normalisés, ce qui serait possible. 
\end{rem}

\begin{rem}\label{NerfHomPasFonctorielSurLax}
Insistons sur le fait que l'on considérera généralement le nerf lax $\NerfLax$ comme un foncteur de $\DeuxCatLax$ vers $\EnsSimp$, et pas simplement de $\DeuxCat$ vers $\EnsSimp$. En revanche, aucun des quatre autres nerfs ci-dessus ne définit de foncteur de source $\DeuxCatLax$. En particulier, le nerf $\NerfHom$ n'est pas fonctoriel sur les \DeuxFoncteursLax{}. Pour un contre-exemple, on pourra considérer l'inclusion
$
{}[1] \to [2],
0 \mapsto 0,
1 \mapsto 2
$.
Si le nerf $\NerfHom$ était fonctoriel sur les \DeuxFoncteursLax{}, on aurait, pour tout \DeuxFoncteurLax{} $u : \mathdeuxcat{A} \to \mathdeuxcat{B}$, un diagramme commutatif
$$
\xymatrix{
(\NerfHom\mathdeuxcat{A})_{2}
\ar[r]
\ar[d]_{(\NerfHom(u))_{2}}
&(\NerfHom\mathdeuxcat{A})_{1}
\ar[d]^{(\NerfHom(u))_{1}}
\\
(\NerfHom\mathdeuxcat{B})_{2}
\ar[r]
&(\NerfHom\mathdeuxcat{B})_{1}
}
$$
dans $\Cat$, dans lequel les flèches horizontales désignent les images des inclusions ci-dessus, ce qui n'est manifestement pas le cas. En l'absence d'ambiguïté, on pourra s'autoriser à commettre l'abus de considérer les nerfs $\NerfLax$ et $\NerfLaxNor$ comme des foncteurs à valeurs dans $\CatHom{CAT}{\Delta^{op}}{\Cat}$ à l'aide de l'inclusion 
$$
\EnsSimp =  \CatHom{CAT}{\Delta^{op}}{Ens} \hookrightarrow \CatHom{CAT}{\Delta^{op}}{\Cat}.
$$
\end{rem}

\begin{lemme}\label{InclusionCatAdjointUn}
Pour toute \deux{}catégorie $\mathdeuxcat{A}$ et tout entier $m \geq 0$, l'inclusion naturelle de catégories
$$
\FonStrictCat{[m]}{\mathdeuxcat{A}}     \hookrightarrow    \FonLaxNorCat{[m]}{\mathdeuxcat{A}}
$$ 
admet un adjoint à droite.
\end{lemme}

\begin{proof}
Voir, par exemple, la démonstration de \cite[proposition 8.4.4]{TheseDelHoyo} ou, pour un résultat plus général (car bicatégorique), celle de \cite[théorème 6.4]{CCG}. 
\end{proof}


\begin{prop}\label{InclusionNerfHomNerfCatLaxNorW}
Pour toute \deux{}catégorie $\mathdeuxcat{A}$, l'inclusion naturelle de catégories
$$
i_{hom}^{\underline{l,n}}\mathdeuxcat{A}\index[not]{ilnhomA@$i_{hom}^{\underline{l,n}}\mathdeuxcat{A}$} : \DeuxIntOp{\Delta} \NerfHom{\mathdeuxcat{A}} \hookrightarrow \DeuxIntOp{\Delta} \NerfCatLaxNor{\mathdeuxcat{A}}
$$
est une équivalence faible pour tout \ClasseUnLocFond{}. 
\end{prop}

\begin{proof}
Pour tout entier $m \geq 0$, notons $\sigma_{m}$ l'inclusion naturelle de catégories
$$
\underline{Fon}([m], \mathdeuxcat{A}) \hookrightarrow \underline{FonLaxNor}([m], \mathdeuxcat{A}).
$$
L'assignation $[m] \mapsto \sigma_{m}$ définit une transformation naturelle $\sigma$ de $\NerfHom{\mathdeuxcat{A}} : \DeuxCatUnOp{\Delta} \to \Cat$ vers $\NerfCatLaxNor{\mathdeuxcat{A}} : \DeuxCatUnOp{\Delta} \to \Cat$ dont les composantes admettent chacune un adjoint à droite, en vertu du lemme \ref{InclusionCatAdjointUn}, et sont donc des équivalences faibles en vertu du lemme \ref{UnAdjointsW}. Pour conclure, il suffit donc d'invoquer la proposition \ref{IntegrationWParArguments}.
\end{proof}

\begin{rem}\label{Nana}
Ainsi, pour tout \DeuxFoncteurStrict{} $u : \mathdeuxcat{A} \to \mathdeuxcat{B}$, il existe un diagramme commutatif
$$
\xymatrix{
\DeuxIntOp{\Delta} \NerfHom{\mathdeuxcat{A}}
\ar[rr]^{\DeuxIntOp{\Delta} \NerfHom(u)}
\ar[d]
&&
\DeuxIntOp{\Delta} \NerfHom{\mathdeuxcat{B}}
\ar[d]
\\
\DeuxIntOp{\Delta} \NerfCatLaxNor{\mathdeuxcat{A}}
\ar[rr]_{\DeuxIntOp{\Delta} \NerfCatLaxNor(u)}
&&
\DeuxIntOp{\Delta} \NerfCatLaxNor{\mathdeuxcat{B}}
}
$$
dont les flèches verticales sont des équivalences faibles pour tout \ClasseUnLocFond{}. En particulier, pour tout \ClasseUnLocFond{}, $\DeuxIntOp{\Delta} \NerfHom(u)$ est une équivalence faible si et seulement si $\DeuxIntOp{\Delta} \NerfCatLaxNor(u)$ en est une. 
\end{rem}

\begin{lemme}\label{InclusionCatAdjointDeux}
Pour toute \deux{}catégorie $\mathdeuxcat{A}$ et tout entier $m \geq 0$, l'inclusion naturelle de catégories
$$
\underline{FonLaxNor}([m], \mathdeuxcat{A}) \hookrightarrow \underline{FonLax}([m], \mathdeuxcat{A})
$$
admet un adjoint à droite.
\end{lemme}

\begin{proof}
Voir la démonstration de \cite[théorème 6.3]{CCG} (relatif au cas plus général des bicatégories).
\end{proof}

\begin{prop}\label{InclusionNerfCatLaxNorNerfCatLaxW}
Pour toute \deux{}catégorie $\mathdeuxcat{A}$, l'inclusion naturelle de catégories
$$
\DeuxIntOp{\Delta} \NerfCatLaxNor{\mathdeuxcat{A}} \hookrightarrow \DeuxIntOp{\Delta} \NerfCatLax{\mathdeuxcat{A}}
$$
est une équivalence faible pour tout \ClasseUnLocFond{}.
\end{prop}

\begin{proof}
Pour tout entier $m \geq 0$, notons $\sigma_{m}$ l'inclusion naturelle de catégories
$$
\underline{FonLaxNor}([m], \mathdeuxcat{A}) \hookrightarrow \underline{FonLax}([m], \mathdeuxcat{A}).
$$
L'assignation $[m] \mapsto \sigma_{m}$ définit une transformation naturelle $\sigma$ de $\NerfCatLaxNor{\mathdeuxcat{A}} : \DeuxCatUnOp{\Delta} \to \Cat$ vers $\NerfCatLax{\mathdeuxcat{A}} : \DeuxCatUnOp{\Delta} \to \Cat$ dont les composantes admettent chacune un adjoint à droite, en vertu du lemme \ref{InclusionCatAdjointDeux}, et sont donc des équivalences faibles en vertu du lemme \ref{UnAdjointsW}. Pour conclure, il suffit donc d'invoquer la proposition \ref{IntegrationWParArguments}.
\end{proof}

\begin{rem}\label{Nene}
Pour tout \DeuxFoncteurLax{} normalisé $u : \mathdeuxcat{A} \to \mathdeuxcat{B}$, il existe donc un diagramme commutatif
$$
\xymatrix{
\DeuxIntOp{\Delta} \NerfCatLaxNor{\mathdeuxcat{A}}
\ar[rr]^{\DeuxIntOp{\Delta} \NerfCatLaxNor(u)}
\ar[d]
&&
\DeuxIntOp{\Delta} \NerfCatLaxNor{\mathdeuxcat{B}}
\ar[d]
\\
\DeuxIntOp{\Delta} \NerfCatLax{\mathdeuxcat{A}}
\ar[rr]_{\DeuxIntOp{\Delta} \NerfCatLax(u)}
&&
\DeuxIntOp{\Delta} \NerfCatLax{\mathdeuxcat{B}}
}
$$
dont les flèches verticales sont des équivalences faibles pour tout \ClasseUnLocFond{}. (Nous avons souligné dans la remarque \ref{NerfLaxNorFonctoriel} la possibilité d'appliquer le foncteur $\NerfCatLaxNor$ aux \DeuxFoncteursLax{} normalisés.) En particulier, pour tout \ClasseUnLocFond{}, $\DeuxIntOp{\Delta} \NerfCatLaxNor(u)$ est une équivalence faible si et seulement si $\DeuxIntOp{\Delta} \NerfCatLax(u)$ en est une. 
\end{rem}

\begin{prop}\label{InclusionNerfEnsNerfCatW}
Pour toute \deux{}catégorie $\mathdeuxcat{A}$, les inclusions naturelles de catégories
$$
i_{l}^{\underline{l}}\mathdeuxcat{A}\index[not]{illA@$i_{l}^{\underline{l}}\mathdeuxcat{A}$} : \DeuxIntOp{\Delta} \NerfLax{\mathdeuxcat{A}} \hookrightarrow \DeuxIntOp{\Delta} \NerfCatLax{\mathdeuxcat{A}}
$$
et
$$
i_{l,n}^{\underline{l,n}}\mathdeuxcat{A}\index[not]{ilnSouligneA@$i_{l,n}^{\underline{l,n}}\mathdeuxcat{A}$} : \DeuxIntOp{\Delta} \NerfLaxNor{\mathdeuxcat{A}} \hookrightarrow \DeuxIntOp{\Delta} \NerfCatLaxNor{\mathdeuxcat{A}}
$$
sont des équivalences faibles pour tout \ClasseUnLocFond{}. 
\end{prop}

\begin{proof}
Nous renvoyons le lecteur à la démonstration de \cite[théorème 6.2]{CCG} (on pourra aussi se reporter à la démonstration de \cite[théorème 8.4.9]{TheseDelHoyo}).
\end{proof}

\begin{rem}\label{RemPanPan}
Pour tout \DeuxFoncteurLax{} normalisé $u : \mathdeuxcat{A} \to \mathdeuxcat{B}$, il existe un diagramme commutatif
$$
\xymatrix{
\DeuxIntOp{\Delta} \NerfLax{\mathdeuxcat{A}}
\ar[rr]
\ar[d]_{\DeuxIntOp{\Delta} \NerfLax{(u)}}
&&
\DeuxIntOp{\Delta} \NerfCatLax{\mathdeuxcat{A}}
\ar[d]^{\DeuxIntOp{\Delta} \NerfCatLax{(u)}}
\\
\DeuxIntOp{\Delta} \NerfLax{\mathdeuxcat{B}}
\ar[rr]
&&
\DeuxIntOp{\Delta} \NerfCatLax{\mathdeuxcat{B}}
}
$$
dont les flèches horizontales sont des équivalences faibles pour tout \ClasseUnLocFond{}. En particulier, pour tout \ClasseUnLocFond{}, $\DeuxIntOp{\Delta} \NerfLax{(u)}$ est une équivalence faible si et seulement si $\DeuxIntOp{\Delta} \NerfCatLax{(u)}$ en est une.
\end{rem}

\begin{rem}\label{RemPanPanPan}
Pour tout \DeuxFoncteurLax{} normalisé $u : \mathdeuxcat{A} \to \mathdeuxcat{B}$, il existe un diagramme commutatif
$$
\xymatrix{
\DeuxIntOp{\Delta} \NerfLaxNor{\mathdeuxcat{A}}
\ar[rr]
\ar[d]_{\DeuxIntOp{\Delta} \NerfLaxNor{(u)}}
&&
\DeuxIntOp{\Delta} \NerfCatLaxNor{\mathdeuxcat{A}}
\ar[d]^{\DeuxIntOp{\Delta} \NerfCatLaxNor{(u)}}
\\
\DeuxIntOp{\Delta} \NerfLaxNor{\mathdeuxcat{B}}
\ar[rr]
&&
\DeuxIntOp{\Delta} \NerfCatLaxNor{\mathdeuxcat{B}}
}
$$
dont les flèches horizontales sont des équivalences faibles pour tout \ClasseUnLocFond{}. En particulier, pour tout \ClasseUnLocFond{}, $\DeuxIntOp{\Delta} \NerfLaxNor{(u)}$ est une équivalence faible si et seulement si $\DeuxIntOp{\Delta} \NerfCatLaxNor{(u)}$ en est une.
\end{rem}

\begin{prop}\label{InclusionNerfLaxNorNerfLaxW}
Pour toute \deux{}catégorie $\mathdeuxcat{A}$ et tout entier $m \geq 0$, l'inclusion naturelle de catégories
$$
i_{l,n}^{l}\mathdeuxcat{A}\index[not]{ilnlA@$i_{l,n}^{l}\mathdeuxcat{A}$} : \DeuxIntOp{\Delta} \NerfLaxNor{\mathdeuxcat{A}} \hookrightarrow \DeuxIntOp{\Delta} \NerfLax{\mathdeuxcat{A}}
$$
est une équivalence faible pour tout \ClasseUnLocFond{}.
\end{prop}

\begin{proof}
C'est une conséquence immédiate des propositions \ref{InclusionNerfCatLaxNorNerfCatLaxW} et \ref{InclusionNerfEnsNerfCatW}, par un argument « de $2$ sur $3$ ».
\end{proof}


\begin{rem}\label{Nini}
Pour tout \DeuxFoncteurLax{} normalisé $u : \mathdeuxcat{A} \to \mathdeuxcat{B}$, il existe donc un diagramme commutatif 
$$
\xymatrix{
\Delta / \NerfLaxNor \mathdeuxcat{A}
\ar[rr]^{\Delta / \NerfLaxNor (u)}
\ar[d]
&& 
\Delta / \NerfLaxNor \mathdeuxcat{B}
\ar[d]
\\
\Delta / \NerfLax \mathdeuxcat{A}
\ar[rr]_{\Delta / \NerfLax (u)}
&& 
\Delta / \NerfLax \mathdeuxcat{B}
}
$$
dont les flèches verticales sont des équivalences faibles pour tout \ClasseUnLocFond{}. En particulier, pour tout \ClasseUnLocFond{}, pour tout \DeuxFoncteurLax{} normalisé $u$, le foncteur $\Delta / \NerfLax{(u)}$ est une équivalence faible si et seulement si le foncteur $\Delta / \NerfLaxNor{(u)}$ en est une (ce que l'on savait en fait déjà). 
\end{rem}

\section{Morphismes extrémistes}\label{SectionSup}

\begin{paragr}
Soit $\mathdeuxcat{A}$ une petite \deux{}catégorie. On définit un \DeuxFoncteurLax{}
$$
\SupLaxObjet{\mathdeuxcat{A}}\index[not]{SuplA@$\SupLaxObjet{\mathdeuxcat{A}}$} : \Delta / \NerfLax{\mathdeuxcat{A}} \to \mathdeuxcat{A}
$$
comme suit. Pour tout objet $([m], x)$ de $\Delta / \NerfLax{\mathdeuxcat{A}}$, 
$$
\SupLaxObjet{\mathdeuxcat{A}}([m], x) = x_{m}.
$$

Pour tout morphisme simplicial $\varphi : [m] \to [n]$ définissant un morphisme de $([m], x)$ vers $([n], y)$ dans $\Delta /\NerfLax{\mathdeuxcat{A}}$, 
$$
\SupLaxObjet{\mathdeuxcat{A}} (\varphi : ([m], x) \to ([n], y)) = y_{n,\varphi(m)}.
$$

Pour tout objet $([m], x)$ de $\Delta / \NerfLax{\mathdeuxcat{A}}$, 
$$
\DeuxCellStructId{\SupLaxObjet{\mathdeuxcat{A}}}{([m], x)} = \DeuxCellStructId{(x)}{m} : 1_{x_{m}} \Rightarrow x(1_{m}).
$$
(On rappelle noter $(x)_{m}$ la \deux{}cellule structurale d'unité du \DeuxFoncteurLax{} $x$ associée à l'objet $[m]$ et ne pas la noter $x_{m}$ afin d'éviter toute confusion avec l'objet $x_{m}$ de $\mathdeuxcat{A}$. C'est une \deux{}cellule de $\mathdeuxcat{A}$.)

Pour tout couple de morphismes composables $\varphi : ([m], x) \to ([n], y)$ et $\psi : ([n], y) \to ([p], z)$ de $\Delta / \NerfLax{\mathdeuxcat{A}}$, 
$$
\DeuxCellStructComp{\SupLaxObjet{\mathdeuxcat{A}}}{\psi}{\varphi} = z_{p, \psi(n), \psi\varphi(m)}.
$$

Vérifions que cela définit bien un \DeuxFoncteurLax. La condition de cocycle s'énonce ainsi : si $([m], x)$, $([n], y)$, $([p], z)$ et $([q], t)$ sont des objets de $\Delta / \NerfLax{\mathdeuxcat{A}}$, $\varphi$ un morphisme de $([m], x)$ vers $([n], y)$, $\psi$ un morphisme de $([n], y)$ vers $([p], z)$ et $\xi$ un morphisme de $([p], z)$ vers $([q], t)$, alors le diagramme
$$
\xymatrix{
\SupLaxObjet{\mathdeuxcat{A}} (\xi) \SupLaxObjet{\mathdeuxcat{A}} (\psi) \SupLaxObjet{\mathdeuxcat{A}} (\varphi)
\ar@{=>}[rrrrr]^{{\SupLaxObjet{\mathdeuxcat{A}}}_{\xi, \psi} \CompDeuxZero \SupLaxObjet{\mathdeuxcat{A}} (\varphi)}
\ar@{=>}[dd]_{{\SupLaxObjet{\mathdeuxcat{A}} (\xi) \CompDeuxZero {\SupLaxObjet{\mathdeuxcat{A}}}_{\psi, \varphi}}}
&&&&&
\SupLaxObjet{\mathdeuxcat{A}} (\xi \psi) \SupLaxObjet{\mathdeuxcat{A}} (\varphi)
\ar@{=>}[dd]^{{\SupLaxObjet{\mathdeuxcat{A}}}_{\xi \psi, \varphi}}
\\
\\
\SupLaxObjet{\mathdeuxcat{A}} (\xi) \SupLaxObjet{\mathdeuxcat{A}} (\psi \varphi)
\ar@{=>}[rrrrr]_{{\SupLaxObjet{\mathdeuxcat{A}}}_{\xi, \psi \varphi}}
&&&&&
\SupLaxObjet{\mathdeuxcat{A}} (\xi \psi \varphi)
}
$$
est commutatif. Ce diagramme est par définition
$$
\xymatrix{
t_{q, \xi(p)} z_{p, \psi(n)} y_{n, \varphi (m)}
\ar@{=>}[rrrrr]^{t_{q, \xi(p), \xi \psi (n)} \CompDeuxZero y_{n, \varphi(m)}}
\ar@{=>}[dd]_{t_{q, \xi (p)} \CompDeuxZero z_{p, \psi(n), \psi \varphi (m)}}
&&&&&
t_{q, \xi \psi (n)} y_{n, \varphi(m)}
\ar@{=>}[dd]^{t_{q, \xi \psi (n), \xi \psi \varphi (m)}}
\\
\\
t_{q, \xi(p)} z_{p, \psi \varphi (m)}
\ar@{=>}[rrrrr]_{t_{q, \xi(p), \xi \psi \varphi (m)}}
&&&&&
t_{q, \xi \psi \varphi (m)}
&,
}
$$
ce qui se récrit
$$
\xymatrix{
t_{q, \xi(p)} t_{\xi (p), \xi \psi (n)} t_{\xi \psi (n), \xi \psi \varphi (m)}
\ar@{=>}[rrrrr]^{t_{q, \xi(p), \xi \psi (n)} \CompDeuxZero t_{\xi \psi (n), \xi \psi \varphi (m)}}
\ar@{=>}[dd]_{t_{q, \xi (p)} \CompDeuxZero t_{\xi(p), \xi \psi (n), \xi \psi \varphi (m)}}
&&&&&
t_{q, \xi \psi (n)} t_{\xi \psi (n), \xi \psi \varphi (m)}
\ar@{=>}[dd]^{t_{q, \xi \psi (n), \xi \psi \varphi (m)}}
\\
\\
t_{q, \xi(p)} t_{\xi(p), \xi \psi \varphi (m)}
\ar@{=>}[rrrrr]_{t_{q, \xi(p), \xi \psi \varphi (m)}}
&&&&&
t_{q, \xi \psi \varphi (m)}
&.
}
$$
Il résulte de la condition de cocycle, vérifiée par le \DeuxFoncteurLax{} $t$, que ce diagramme est bien commutatif. 

Considérons un morphisme $\varphi$ de $([m], x)$ vers $([n], y)$, $1_{[m]}$ vu comme morphisme de $([m],x)$ vers $([m],x)$ (dont c'est l'identité) et $1_{[n]}$ vu comme morphisme de $([n],y)$ vers $([n],y)$ (dont c'est l'identité). Vérifions l'égalité
$$
{\SupLaxObjet{\mathdeuxcat{A}}}_{\varphi, 1_{[m]}} (\SupLaxObjet{\mathdeuxcat{A}} (\varphi) \CompDeuxZero {\SupLaxObjet{\mathdeuxcat{A}}}_{([m], x)}) = 1_{\SupLaxObjet{\mathdeuxcat{A}} (\varphi)}.
$$
En vertu de l'hypothèse $x = y \varphi$, si l'on note $(y)_{\varphi(m)}$ la \deux{}cellule structurale d'unité $1_{y_{\varphi(m)}} \Rightarrow y(1_{\varphi(m)})$ de $y$, cette égalité devient
$$
y_{n, \varphi(m), \varphi(m)} (y_{n, \varphi(m)} \CompDeuxZero (y)_{\varphi(m)}) = 1_{y_{n, \varphi(m)}}.
$$
Cette dernière égalité résulte du fait que $y$ est un \DeuxFoncteurLax{}.

Vérifions l'égalité
$$
{\SupLaxObjet{\mathdeuxcat{A}}}_{(1_{[n]}, \varphi)} ({\SupLaxObjet{\mathdeuxcat{A}}}_{([n], y)} \CompDeuxZero \SupLaxObjet{\mathdeuxcat{A}} (\varphi)) = 1_{\SupLaxObjet{\mathdeuxcat{A}} (\varphi)}.
$$
En notant $(y)_{n}$ la \deux{}cellule structurale $1_{y_{n}} \Rightarrow y(1_{n})$ de $y$, cette égalité s'écrit
$$
y_{n, n, \varphi(m)} ((y)_{n} \CompDeuxZero y_{n, \varphi(m)}) = 1_{y_{n, \varphi(m)}}.
$$
Cette dernière égalité résulte du fait que $y$ est un \DeuxFoncteurLax{}. 

Les autres conditions de cohérence portant sur $\SupLaxObjet{\mathdeuxcat{A}}$ sont automatiquement vérifiées du fait que $\Delta / \NerfLax{\mathdeuxcat{A}}$ est une catégorie. Nous avons donc vérifié que $\SupLaxObjet{\mathdeuxcat{A}}$ est un \DeuxFoncteurLax{}.
\end{paragr}

\begin{paragr}
Soit $\mathdeuxcat{A}$ une petite \deux{}catégorie. On définit un \DeuxFoncteurLax{} normalisé
$$
\SupLaxNorObjet{\mathdeuxcat{A}}\index[not]{SuplnA@$\SupLaxNorObjet{\mathdeuxcat{A}}$} : \Delta / \NerfLaxNor{\mathdeuxcat{A}} \to \mathdeuxcat{A}
$$
par la condition de commutativité du diagramme
$$
\xymatrix{
\DeuxIntOp{\Delta}\NerfLax{\mathdeuxcat{A}}
\ar[rd]_{\SupLaxObjet{\mathdeuxcat{A}}}
&&
\DeuxIntOp{\Delta}\NerfLaxNor{\mathdeuxcat{A}}
\ar[ll]_{i_{l,n}^{l}\mathdeuxcat{A}}
\ar[ld]^{\SupLaxNorObjet{\mathdeuxcat{A}}}
\\
&
\mathdeuxcat{A}
&.
}
$$

Plus explicitement, pour tout objet $([m], x)$ de $\Delta / \NerfLaxNor{\mathdeuxcat{A}}$,  
$$
\SupLaxNorObjet{\mathdeuxcat{A}}([m], x) = x_{m}.
$$

Pour tout morphisme $\varphi : ([m], x) \to ([n], y)$ de $\Delta /\NerfLaxNor{\mathdeuxcat{A}}$, 
$$
\SupLaxNorObjet{\mathdeuxcat{A}} (\varphi) = y_{n,\varphi(m)}.
$$

Pour tout objet $([m], x)$ de $\Delta / \NerfLaxNor{\mathdeuxcat{A}}$, 
$$
\DeuxCellStructId{\SupLaxNorObjet{\mathdeuxcat{A}}}{([m], x)} = 1_{1_{x_{m}}}.
$$

Pour tout couple de morphismes composables $\varphi : ([m], x) \to ([n], y)$ et $\psi : ([n], y) \to ([p], z)$ de $\Delta / \NerfLaxNor{\mathdeuxcat{A}}$, 
$$
\DeuxCellStructComp{\SupLaxNorObjet{\mathdeuxcat{A}}}{\psi}{\varphi} = z_{p, \psi(n), \psi(\varphi(m))}.
$$
\end{paragr}

\begin{paragr}
Soit $\mathdeuxcat{A}$ une petite \deux{}catégorie. On définit un \DeuxFoncteurLax{} normalisé
$$
\SupCatLaxNorObjet{\mathdeuxcat{A}}\index[not]{SuplnSouligneA@$\SupCatLaxNorObjet{\mathdeuxcat{A}}$} : \DeuxIntOp{\Delta} \NerfCatLaxNor{\mathdeuxcat{A}} \to \mathdeuxcat{A}
$$
comme suit. 

Pour tout objet $([m], x)$ de $\DeuxIntOp{\Delta} \NerfCatLaxNor{\mathdeuxcat{A}}$, 
$$
\SupCatLaxNorObjet{\mathdeuxcat{A}}([m], x) = x_{m}.
$$

Pour tout morphisme $(\varphi, \alpha) : ([m], x) \to ([n], y)$ de $\DeuxIntOp{\Delta} \NerfCatLaxNor{\mathdeuxcat{A}}$, 
$$
\SupCatLaxNorObjet{\mathdeuxcat{A}} (\varphi, \alpha) = y_{n,\varphi(m)}.
$$

Pour tout objet $([m], x)$ de $\DeuxIntOp{\Delta} \NerfCatLaxNor{\mathdeuxcat{A}}$
$$
\DeuxCellStructId{\SupCatLaxNorObjet{\mathdeuxcat{A}}}{([m],x)} = \DeuxCellStructId{(x)}{m} = 1_{1_{x_{m}}}. 
$$

Pour tout couple de morphismes composables $(\varphi, \alpha) : ([m], x) \to ([n], y)$ et $(\psi, \beta) : ([n], y) \to ([p], z)$ de $\DeuxIntOp{\Delta} \NerfCatLaxNor{\mathdeuxcat{A}}$, 
$$
\DeuxCellStructComp{\SupCatLaxNorObjet{\mathdeuxcat{A}}}{(\psi, \beta)}{(\varphi, \alpha)} = z_{p, \psi(n), \psi\varphi(m)} \CompDeuxUn (z_{p, \psi(n)} \CompDeuxZero \beta_{n, \varphi(m)}).
$$

Vérifions que cela définit bien un \DeuxFoncteurLax{}. 

Soit 
$$
\xymatrix{
([m], x)
\ar[rr]^{(\varphi, \alpha)}
&&
([n], y)
\ar[rr]^{(\psi, \beta)}
&&
([p], z)
\ar[rr]^{(\xi, \gamma)}
&&
([q], t)
}
$$
dans $\DeuxIntOp{\Delta} \NerfCatLaxNor{\mathdeuxcat{A}}$. La condition de cocycle résulte des égalités suivantes : 
$$
\begin{aligned}
&t_{q, \xi \psi (n), \xi \psi \varphi (m)} \CompDeuxUn (t_{q, \xi \psi (n)} \CompDeuxZero (\gamma_{\psi (n), \psi \varphi (m)} \CompDeuxUn \beta_{n, \varphi(m)})) \CompDeuxUn (t_{q, \xi (p), \xi \psi (n)} \CompDeuxZero y_{n, \varphi (m)}) \CompDeuxUn (t_{q, \xi (p)} \CompDeuxZero \gamma_{p, \psi (n)} \CompDeuxZero y_{n, \varphi (m)}) 
\\
&= t_{q, \xi \psi (n), \xi \psi \varphi (m)} \CompDeuxUn (t_{q, \xi (p), \xi \psi (n)} \CompDeuxZero t_{\xi \psi (n), \xi \psi \varphi (m)}) \CompDeuxUn (t_{q, \xi (p)} \CompDeuxZero \gamma_{p, \psi (n)} \CompDeuxZero \gamma_{\psi (n), \psi \varphi (m)}) \CompDeuxUn (t_{q, \xi (p)} z_{p, \psi (n)} \CompDeuxZero \beta_{n, \varphi (m)})
\\
&= t_{q, \xi (p), \xi \psi \varphi (m)} \CompDeuxUn (t_{q, \xi (p)} \CompDeuxZero t_{\xi (p), \xi \psi (n), \xi \psi \varphi (m)}) \CompDeuxUn (t_{q, \xi (p)} \CompDeuxZero \gamma_{p, \psi (n)} \CompDeuxZero \gamma_{\psi (n), \psi \varphi (m)}) \CompDeuxUn (t_{q, \xi (p)} z_{p, \psi (n)} \CompDeuxZero \beta_{n, \varphi (m)})
\\
&= t_{q, \xi (p), \xi \psi \varphi (m)} \CompDeuxUn (t_{q, \xi (p)} \CompDeuxZero \gamma_{p, \psi \varphi (m)}) \CompDeuxUn (t_{q, \xi (p)} \CompDeuxZero z_{p, \psi (n), \psi \varphi (m)}) \CompDeuxUn (t_{q, \xi (p)} z_{p, \psi (n)} \CompDeuxZero \beta_{n, \varphi (m)}).
\end{aligned}
$$

Les autres conditions de cohérence sont automatiquement vérifiées puisque $\DeuxIntOp{\Delta} \NerfCatLaxNor{\mathdeuxcat{A}}$ est une catégorie et que les \DeuxFoncteursLax{} considérés sont normalisés (voir la remarque \ref{RemDeuxTransFoncNor}). On a donc bien défini un \DeuxFoncteurLax{} normalisé
$$
\SupCatLaxNorObjet{\mathdeuxcat{A}} : \DeuxIntOp{\Delta} \NerfCatLaxNor{\mathdeuxcat{A}} \to \mathdeuxcat{A}.
$$
\end{paragr}

\begin{paragr}
Soit $\mathdeuxcat{A}$ une petite \deux{}catégorie. On définit un \DeuxFoncteurLax{} normalisé
$$
\SupHomObjet{\mathdeuxcat{A}}\index[not]{SupSouligneA@$\SupHomObjet{\mathdeuxcat{A}}$} : \DeuxIntOp{\Delta} \NerfHom{\mathdeuxcat{A}} \to \mathdeuxcat{A}
$$
par la condition de commutativité du diagramme
$$
\xymatrix{
\DeuxIntOp{\Delta}\NerfCatLaxNor{\mathdeuxcat{A}}
\ar[dr]_{\SupCatLaxNorObjet{\mathdeuxcat{A}}}
&&
\DeuxIntOp{\Delta}\NerfHom{\mathdeuxcat{A}}
\ar[ll]_{i_{hom}^{\underline{l,n}}\mathdeuxcat{A}}
\ar[ld]^{\SupHomObjet{\mathdeuxcat{A}}}
\\
&
\mathdeuxcat{A}
&.
}
$$

Plus explicitement, pour tout objet $([m], x)$ de $\DeuxIntOp{\Delta} \NerfHom{\mathdeuxcat{A}}$, 
$$
\SupHomObjet{\mathdeuxcat{A}} ([m], x) = x_{m}.
$$

Pour tout morphisme $(\varphi, \alpha) : ([m], x) \to ([n], y))$ de $\DeuxIntOp{\Delta} \NerfHom{\mathdeuxcat{A}}$, 
$$
\SupHomObjet{\mathdeuxcat{A}} (\varphi, \alpha) = y_{n, \varphi(m)}.
$$

Pour tout objet $([m], x)$ de $\DeuxIntOp{\Delta} \NerfHom{\mathdeuxcat{A}}$, 
$$
\DeuxCellStructId{\SupHomObjet{\mathdeuxcat{A}}}{([m], x)} = 1_{1_{x_{m}}}. 
$$

Pour tout couple de morphismes composables $(\varphi, \alpha) : ([m], x) \to ([n], y)$ et $(\psi, \beta) : ([n], y) \to ([p], z)$ de $\DeuxIntOp{\Delta} \NerfHom{\mathdeuxcat{A}}$, 
$$
\DeuxCellStructComp{\SupHomObjet{\mathdeuxcat{A}}}{(\psi, \beta)}{(\varphi, \alpha)} = z_{p, \psi(n)} \CompDeuxZero \beta_{n, \varphi(m)}.
$$
\end{paragr}

\begin{lemme}\label{DiagrammeSups}
Pour toute petite \deux{}catégorie $\mathdeuxcat{A}$, le diagramme
$$
\xymatrix{
\DeuxIntOp{\Delta}\NerfLax{\mathdeuxcat{A}}
\ar@/_1.5pc/[rrrdd]_{\SupLaxObjet{\mathdeuxcat{A}}}
&&
\DeuxIntOp{\Delta}\NerfLaxNor{\mathdeuxcat{A}}
\ar[ll]_{i_{l,n}^{l}\mathdeuxcat{A}}
\ar@/_0.5pc/[rdd]_{\SupLaxNorObjet{\mathdeuxcat{A}}}
\ar[rr]^{i_{l,n}^{\underline{l,n}}\mathdeuxcat{A}}
&&
\DeuxIntOp{\Delta}\NerfCatLaxNor{\mathdeuxcat{A}}
\ar@/^0.5pc/[ldd]^{\SupCatLaxNorObjet{\mathdeuxcat{A}}}
&&
\DeuxIntOp{\Delta}\NerfHom{\mathdeuxcat{A}}
\ar[ll]_{i_{hom}^{\underline{l,n}}\mathdeuxcat{A}}
\ar@/^1.5pc/[llldd]^{\SupHomObjet{\mathdeuxcat{A}}}
\\
\\
&&&
\mathdeuxcat{A}
}
$$
est commutatif.
\end{lemme}

\begin{proof}
En vertu des définitions, il suffit de vérifier l'égalité 
$
\SupLaxNorObjet{\mathdeuxcat{A}} = \SupCatLaxNorObjet{\mathdeuxcat{A}} \phantom{a} i_{l,n}^{\underline{l,n}}\mathdeuxcat{A}
$, ce qui ne pose aucune difficulté. 
\end{proof}

%

\section{De $1$ à $2$}\label{SectionUnDeux}
À toute classe $\UnLocFond{S}$ de morphismes de $\Cat$, on peut associer une classe $\NerfLaxNor^{-1} (i_{\Delta}^{-1} (\UnLocFond{S}))$ de morphismes de $\DeuxCat$. On étudie dans cette section les premières propriétés des classes de \DeuxFoncteursStricts{} obtenues par un tel procédé à partir d'un \ClasseUnLocFond{}, avant d'en entreprendre une étude plus axiomatique dans le chapitre suivant. 
\\

\emph{On suppose fixé un localisateur fondamental $\UnLocFond{W}$ de $\Cat$.}  

\begin{lemme}\label{Blabla}
Pour tout \DeuxFoncteurStrict{} $u$, les conditions suivantes sont équivalentes :
\begin{itemize}
\item[(i)] 
Le foncteur $\Delta / \NerfLaxNor{(u)} = i_{\Delta} (\NerfLaxNor{(u)}) = \DeuxIntOp{\Delta} \NerfLaxNor{(u)}$ est une équivalence faible.
\item[(ii)]
Le foncteur $\Delta / \NerfLax{(u)} = i_{\Delta} (\NerfLax{(u)}) = \DeuxIntOp{\Delta} \NerfLax{(u)}$ est une équivalence faible.
\item[(iii)]
Le foncteur $\DeuxIntOp{\Delta} \NerfCatLaxNor{(u)}$ est une équivalence faible.
\item[(iv)]
Le foncteur $\DeuxIntOp{\Delta} \NerfCatLax{(u)}$ est une équivalence faible.
\item[(v)] 
Le foncteur $\DeuxIntOp{\Delta} \NerfHom{(u)}$ est une équivalence faible. 
\end{itemize}
\end{lemme}

\begin{proof}
Il suffit de mettre bout à bout les remarques \ref{Nana}, \ref{RemPanPan}, \ref{RemPanPanPan} et \ref{Nini}.
\end{proof}

\begin{df}\label{DefDeuxW}
On notera $\DeuxLocFond{W}$ la classe des \DeuxFoncteursStricts{} vérifiant les cinq conditions équivalentes de l'énoncé du lemme \ref{Blabla}. On appellera les éléments de $\DeuxLocFond{W}$ des $\DeuxLocFond{W}$\nobreakdash-\emph{équivalences faibles}, ou plus simplement des \emph{équivalences faibles}\index{equivalence faible (\DeuxFoncteurStrict{}, pour un localisateur fondamental de $\Cat$)@équivalence faible (\DeuxFoncteurStrict{}, pour un localisateur fondamental de $\Cat$)}. 
\end{df}

\begin{df}
On notera $\DeuxLocFondLaxInduit{W}$\index[not]{Wlax@$\DeuxLocFondLaxInduit{W}$} la classe des \DeuxFoncteursLax{} dont l'image par le foncteur $i_{\Delta} \NerfLax$ est dans $\UnLocFond{W}$. On appellera les éléments de $\DeuxLocFondLaxInduit{W}$ des $\DeuxLocFond{W}$\nobreakdash-\emph{équivalences faibles lax}, ou plus simplement des \emph{équivalences faibles lax}, voire encore plus simplement des \emph{équivalences faibles}.  
\end{df}

\begin{rem}
En vertu du lemme \ref{Blabla}, un \DeuxFoncteurStrict{} est dans $\DeuxLocFond{W}$ si et seulement s'il est dans $\DeuxLocFondLaxInduit{W}$. 
\end{rem} 

\begin{lemme}\label{DeuxLocFondSat}
Les classes $\DeuxLocFond{W}$ et $\DeuxLocFondLaxInduit{W}$ sont faiblement saturées.
\end{lemme}

\begin{proof}
Cela résulte immédiatement, par fonctorialité, de la saturation faible de la classe $\UnLocFond{W}$.
\end{proof}

\begin{lemme}\label{Zizi}
Pour toute petite \deux{}catégorie $\mathdeuxcat{A}$, la projection canonique $[1] \times \mathdeuxcat{A} \to \mathdeuxcat{A}$ est une équivalence faible.
\end{lemme}

\begin{proof}
La commutativité du nerf lax $\NerfLax$ et du produit permet d'identifier le morphisme d'ensembles simpliciaux 
$$
\NerfLax{([1] \times \mathdeuxcat{A})} \to \NerfLax{\mathdeuxcat{A}}
$$
à la projection simpliciale
$$
\Delta_{1} \times \NerfLax{\mathdeuxcat{A}} \to \NerfLax{\mathdeuxcat{A}},
$$
dont l'image par le foncteur $i_{\Delta}$ est bien dans $\UnLocFond{W}$ en vertu de \cite[lemme 2.2.6]{LFM}.
\end{proof}

\begin{df}
Deux morphismes de $\DeuxCatLax$ seront dits \emph{homotopes} s'ils sont $([1],0,1)$\nobreakdash-ho\-mo\-topes\footnote{Voir le paragraphe \ref{DefHomotopieBis} pour la définition.}. 
\end{df}

\begin{lemme}\label{HomotopieW}
Si deux morphismes de $\DeuxCatLax$ sont homotopes, alors l'un est une équivalence faible si et seulement si l'autre en est une.
\end{lemme}

\begin{proof}
En vertu du lemme \ref{Zizi}, c'est une conséquence du lemme \ref{LemmeHomotopieTHG}. 
\end{proof}

\begin{df}
On dira qu'une petite \deux{}catégorie $\mathdeuxcat{A}$ est $\DeuxLocFond{W}$\nobreakdash-\emph{asphérique}, ou plus simplement \emph{asphérique}\index{asphérique (petite \deux{}catégorie, pour un localisateur fondamental de $\Cat$)}, si le morphisme canonique $\mathdeuxcat{A} \to \DeuxCatPonct$ est une équivalence faible.
\end{df}

\begin{exemple}[Bourbaki]\label{PointAsph}
La \deux{}catégorie ponctuelle $\DeuxCatPonct$ est asphérique.
\end{exemple}

\begin{lemme}\label{Rutebeuf}
Une petite \deux{}catégorie $\mathdeuxcat{A}$ est $\DeuxLocFond{W}$\nobreakdash-asphérique si et seulement si la catégorie $\Delta/\NerfLaxNor\mathdeuxcat{A}$ est $\UnLocFond{W}$\nobreakdash-asphérique (c'est-à-dire si et seulement si le foncteur $\Delta/\NerfLaxNor\mathdeuxcat{A} \to \UnCatPonct$ est dans $\UnLocFond{W}$). 
\end{lemme}

\begin{proof}
C'est immédiat, la catégorie $\Delta$ étant $\UnLocFond{W}$\nobreakdash-asphérique (elle admet un objet final) et la classe $\UnLocFond{W}$ vérifiant la propriété de $2$ sur $3$. 
\end{proof}

\begin{lemme}\label{OFAspherique1}
Une petite \deux{}catégorie admettant un objet admettant un objet final (\emph{resp.} admettant un objet admettant un objet initial, \emph{resp.} op-admettant un objet admettant un objet final, \emph{resp.} op-admettant un objet admettant un objet initial) est $\DeuxLocFond{W}$\nobreakdash-asphérique.
\end{lemme}

\begin{proof}
Soit $\mathdeuxcat{A}$ une petite \deux{}catégorie admettant un objet $z$ tel que, pour tout objet $a$ de $\mathdeuxcat{A}$, la catégorie $\CatHom{\mathdeuxcat{A}}{a}{z}$ admette un objet final. Pour montrer le résultat désiré, il suffit, en vertu du lemme \ref{Rutebeuf}, de vérifier que la catégorie $\Delta/\NerfLaxNor \mathdeuxcat{A}$ est asphérique. 

Pour tout objet $a$ de $\mathdeuxcat{A}$, on notera $p_{a} : a \to z$ l'objet final de $\CatHom{\mathdeuxcat{A}}{a}{z}$ et, pour toute \un{}cellule $f : a \to z$, on notera $\varphi_{f} : f \Rightarrow p_{a}$ l'unique \deux{}cellule de $f$ vers $p_{a}$ dans $\mathdeuxcat{A}$. Pour tout \DeuxFoncteurLax{} normalisé $x : [m] \to \mathdeuxcat{A}$, on définit un \DeuxFoncteurLax{} normalisé $D(x) : [m+1] \to \mathdeuxcat{A}$ comme suit. Pour tout objet $i$ de $[m]$, $D(x)_{i} = x_{i}$ ; de plus, $D(x)_{m+1} = z$. Pour tout couple d'entiers $0 \leq i \leq j \leq m$, $D(x)_{j,i} = x_{j,i}$ ; de plus, $D(x)_{m+1, i} = p_{x_{i}}$ pour tout objet $i$ de $[m]$. Pour tout triplet d'entiers $0 \leq i \leq j \leq k \leq m$, $D(x)_{k,j,i} = x_{k,j,i}$ ; de plus, $D(x)_{m+1, j, i} = \varphi_{p_{x_{j}} x_{j,i}}$ pour tout couple d'entiers $0 \leq i \leq j \leq m$. Pour tout morphisme simplicial $\psi : [m] \to [n]$, on définit un morphisme simplicial $D(\psi) : [m+1] \to [n+1]$ par $D(\psi)(i) = \psi(i)$ si $i \leq m$ et $D(\psi)(m+1) = n+1$. Cela permet de définir un endofoncteur
$$
\begin{aligned}
D : \Delta/\NerfLaxNor{\mathdeuxcat{A}} &\to \Delta/\NerfLaxNor{\mathdeuxcat{A}} 
\\
([m], x : [m] \to \mathdeuxcat{A}) &\mapsto ([m+1], D(x) : [m+1] \to \mathdeuxcat{A})
\\
\psi &\mapsto D(\psi).
\end{aligned}
$$
Considérons en outre l'endofoncteur constant
$$
\begin{aligned}
Z : \Delta/\NerfLaxNor{\mathdeuxcat{A}} &\to \Delta/\NerfLaxNor{\mathdeuxcat{A}} 
\\
([m], x : [m] \to \mathdeuxcat{A}) &\mapsto ([0], z)
\\
\psi &\mapsto 1_{[0]}.
\end{aligned}
$$ 
Pour tout objet $([m], x)$ de $\Delta/\NerfLaxNor{\mathdeuxcat{A}}$, posons 
$$
\begin{aligned}
\iota_{([m], x)} : [m] &\to [m+1]
\\
i &\mapsto i
\end{aligned}
$$
et 
$$
\begin{aligned}
\omega_{([m], x)} : [0] &\to [m+1]
\\
0 &\mapsto m+1.
\end{aligned}
$$
Cela définit des morphismes de foncteurs $\iota : 1_{\Delta/\NerfLaxNor{\mathdeuxcat{A}}} \Rightarrow D$ et $\omega : Z \Rightarrow D$. Il en résulte que $Z$ est une équivalence faible. Comme c'est un endofoncteur constant de $\Delta/\NerfLaxNor{\mathdeuxcat{A}}$, cette catégorie est asphérique (en vertu du lemme \ref{EndoConstantW}). Les trois autres cas s'en déduisent par dualité ou se démontrent de façon analogue.
\end{proof}

\begin{rem}\label{OFContractile}
Voici une autre démonstration du lemme \ref{OFAspherique1}. En conservant les notations utilisées ci-dessus, l'on définit un \deux{}endofoncteur constant $Z$ de $\mathdeuxcat{A}$ par 
$$
\begin{aligned}
Z : \mathdeuxcat{A} &\to \mathdeuxcat{A}
\\
a &\mapsto z
\\
f &\mapsto 1_{z}
\\
\alpha &\mapsto 1_{1_{z}}.
\end{aligned}
$$
On pose alors $\sigma_{a} = p_{a}$ pour tout objet $a$ de $\mathdeuxcat{A}$ et $\sigma_{f} = \varphi_{p_{a'} f}$ pour toute \un{}cellule $f : a \to a'$ de $\mathdeuxcat{A}$. Cela définit une \DeuxTransformationLax{} $\sigma : 1_{\mathdeuxcat{A}} \Rightarrow Z$. En vertu du lemme \ref{DeuxTransFoncLax}, il existe une homotopie élémentaire de $1_{\mathdeuxcat{A}}$ vers $Z$. Le lemme \ref{HomotopieW} permet donc d'affirmer que $Z$ est une équivalence faible. On conclut grâce au lemme \ref{EquivalenceViaPoint}. 
\end{rem}

\begin{paragr}\label{TrancheSimplexe}
Suivant \cite{Cegarra}, nous définissons maintenant la notion de \deux{}catégorie tranche au-dessus d'un simplexe\index{tranche au-dessus d'un simplexe (2-catégorie)} (les définitions duales permettant de remplacer « au-dessus » par « au-dessous »), généralisant celle de \deux{}catégorie au-dessus (ou, dualement, au-dessous) d'un objet, correspondant au cas des $0$\nobreakdash-simplexes. Nous n'en présentons pas la version la plus générale. Les calculs ne sont pas détaillés, puisque l'on se borne à reprendre les résultats de \cite{Cegarra}. Soit donc $u : \mathdeuxcat{A} \to \mathdeuxcat{B}$ un morphisme de $\DeuxCat$ et ${b}$ un $m$\nobreakdash-simplexe de $\NerfLaxNor{\mathdeuxcat{B}}$. On définit une \deux{}catégorie $\TrancheCoLax{\mathdeuxcat{A}}{u}{{b}}$\index[not]{AcuB@$\TrancheCoLax{\mathdeuxcat{A}}{u}{{b}}$ ($b$ simplexe)} comme suit. Ses objets sont les couples $(a, {x})$ avec $a$ un objet de $\mathdeuxcat{A}$ et ${x}$ un $m+1$\nobreakdash-simplexe de $\NerfLaxNor{\mathdeuxcat{B}}$ tel que ${x}_{0} = u(a)$ et $d_{0} {x} = {b}$. (L'expression $d_{0} {x}$ désigne la « zéroième » face de ${x}$.) Une \un{}cellule de $(a, {x})$ vers $(a', {x'})$ est un couple $(f, {y})$ avec $f : a \to a'$ une \un{}cellule de $\mathdeuxcat{A}$ et ${y}$ un $m+2$\nobreakdash-simplexe de $\mathdeuxcat{B}$ tel que $y_{1,0} = u(f)$, $d_{0} {y} = {x'}$ et $d_{1} {y} = {x}$. Étant donné les mêmes objets $(a, {x})$ et $(a', {x'})$ ainsi que deux \un{}cellules $(f, {y})$ et $(f', {y'})$ du premier vers le second, les \deux{}cellules de la première vers la seconde sont les \deux{}cellules $\alpha : f \Rightarrow f'$ dans $\mathdeuxcat{A}$ telles que, pour tout $0 \leq i \leq m$, 
$$
y'_{i+2, 1, 0} (y_{i+2, 1} \CompDeuxZero u(\alpha)) = y_{i+2, 1, 0}.
$$
Les diverses unités et compositions sont définies de façon « évidente ». 

On construit maintenant un couple de \DeuxFoncteursStricts{} qui sont des équivalences faibles entre $\TrancheCoLax{\mathdeuxcat{A}}{u}{{b}}$ et $\TrancheCoLax{\mathdeuxcat{A}}{u}{b_{0}}$, suivant toujours en cela \cite{Cegarra}. Notons $R$ le \DeuxFoncteurStrict{} défini par 
$$
\begin{aligned}
\TrancheCoLax{\mathdeuxcat{A}}{u}{{b}} &\to \TrancheCoLax{\mathdeuxcat{A}}{u}{b_{0}}
\\
(a, {x}) &\mapsto (a, x_{1,0})
\\
(f, {y}) &\mapsto (f, y_{2,1,0})
\\
\gamma &\mapsto \gamma.
\end{aligned}
$$
Il admet une section $I : \TrancheCoLax{\mathdeuxcat{A}}{u}{b_{0}} \to \TrancheCoLax{\mathdeuxcat{A}}{u}{{b}}$ définie comme suit. À tout objet $(a, p)$ de $\TrancheCoLax{\mathdeuxcat{A}}{u}{b_{0}}$, $I$ associe le couple $(a, {x})$ défini par $d_{0}{x} = {b}$, ${x}_{0} = u(a)$, ${x}_{1,0} = p$, ${x}_{i+1, 0} = {b}_{i, 0} p$ et ${x}_{j+1, i+1, 0} = {b}_{j, i, 0} p$. À toute \un{}cellule $(f, \alpha)$ de $(a, p)$ vers $(a', p')$, $I$ associe le couple $(f, {y})$ défini par ${y}_{i+2, 1, 0} = {b}_{i, 0} \CompDeuxZero \alpha$. Enfin, pour toute \deux{}cellule $\gamma$ de $\TrancheCoLax{\mathdeuxcat{A}}{u}{b_{0}}$, on pose $I(\gamma) = \gamma$. L'égalité $RI = 1_{\TrancheCoLax{\mathdeuxcat{A}}{u}{b_{0}}}$ est alors évidente. De plus, on construit une \DeuxTransformationStricte{} $\sigma : 1_{\TrancheCoLax{\mathdeuxcat{A}}{u}{{b}}} \Rightarrow IR$ en posant $\sigma_{(a, x)} = (1_{a}, \tilde{x})$ avec $\tilde{x}_{i+2, 1, 0} = x_{i+1, 1, 0}$. En vertu des lemmes \ref{DeuxTransFoncLax} et \ref{HomotopieW}, $IR$ est donc une équivalence faible. On en déduit que $I$ et $R$ sont des équivalences faibles.
\end{paragr}

\begin{paragr}\label{Zigouigoui}
Soient maintenant 
$$
\xymatrix{
\mathdeuxcat{A} 
\ar[rr]^{u}
\ar[dr]_{w}
&&\mathdeuxcat{B}
\ar[dl]^{v}
\\
&\mathdeuxcat{C}
}
$$
un triangle commutatif de \DeuxFoncteursStricts{} et ${c}$ un simplexe de $\NerfLaxNor{\mathdeuxcat{C}}$. Ces données permettent de définir un \DeuxFoncteurStrict{}
$$
\DeuxFoncTrancheCoLax{u}{{c}}\index[not]{ucB@$\DeuxFoncTrancheCoLax{u}{{c}}$ ($c$ simplexe)} : \TrancheCoLax{\mathdeuxcat{A}}{w}{{c}} \to \TrancheCoLax{\mathdeuxcat{B}}{v}{{c}}
$$
de façon « évidente ».
On vérifie alors la commutativité du diagramme
$$
\xymatrix{
\TrancheCoLax{\mathdeuxcat{A}}{w}{{c}}
\ar[r]^{\DeuxFoncTrancheCoLax{u}{{c}}}
\ar[d]
&\TrancheCoLax{\mathdeuxcat{B}}{v}{{c}}
\ar[d]
\\
\TrancheCoLax{\mathdeuxcat{A}}{w}{c_{0}}
\ar[r]_{\DeuxFoncTrancheCoLax{u}{c_{0}}}
&\TrancheCoLax{\mathdeuxcat{B}}{v}{c_{0}}
&,
}
$$
dans lequel les flèches verticales désignent les \DeuxFoncteursStricts{} canoniques précédemment notés $R$ et sont donc en particulier des équivalences faibles. On en déduit le lemme \ref{WSurSimplexe}. 
\end{paragr}

\begin{lemme}\label{WSurSimplexe}
En conservant les notations ci-dessus, $\DeuxFoncTrancheCoLax{u}{{c}}$ est une équivalence faible si et seulement si $\DeuxFoncTrancheCoLax{u}{c_{0}}$ en est une. 
\end{lemme}

\begin{paragr}
Pour un $m+n+1$-simplexe $c$ de $\NerfLaxNor{\mathdeuxcat{C}}$, nous désignerons ci-dessous par $d_{m+1}^{n+1}({c})$ la « $(m+1)^{e}$ face de $c$ itérée $n+1$ fois », qu'il faudrait en toute rigueur noter par exemple $d_{m+1}^{m+1} \dots d_{m+1}^{m+n+1} (c)$. C'est donc un $m$-simplexe de $\NerfLaxNor{\mathdeuxcat{C}}$. 

En conservant les données du paragraphe \ref{Zigouigoui}, on définit un ensemble bisimplicial $S_{w}$ par la formule
$$
(S_{w})_{m,n} = \{ ({a} \in (\NerfLaxNor{\mathdeuxcat{A}})_{m}, {c} \in (\NerfLaxNor{\mathdeuxcat{C}})_{m+n+1}), d_{m+1}^{n+1}({c}) = (\NerfLaxNor{(w)}) ({a})  \}.
$$
Autrement dit, on considère les couples $({a}, {c})$ tels que « l'image de ${a}$ par $w$ soit le début de ${c}$ ». On définit de même un ensemble bisimplicial $S_{v}$ par la formule
$$
(S_{v})_{m,n} = \{ ({b} \in (\NerfLaxNor{\mathdeuxcat{B}})_{m}, {c} \in (\NerfLaxNor{\mathdeuxcat{C}})_{m+n+1}), d_{m+1}^{n+1}({c}) = (\NerfLaxNor{(v)}) ({b})  \}.
$$
De plus, on considère $\NerfLaxNor{\mathdeuxcat{A}}$ et $\NerfLaxNor{\mathdeuxcat{B}}$ comme des ensembles bisimpliciaux constants sur les colonnes. Autrement dit, pour tout couple d'entiers positifs $m$ et $n$, on pose $(\NerfLaxNor{\mathdeuxcat{A}})_{m,n} = (\NerfLaxNor{\mathdeuxcat{A}})_{m}$ et $(\NerfLaxNor{\mathdeuxcat{B}})_{m,n} = (\NerfLaxNor{\mathdeuxcat{B}})_{m}$. On note de même $\NerfLaxNor{(u)} : \NerfLaxNor{\mathdeuxcat{A}} \to \NerfLaxNor{\mathdeuxcat{B}}$ le morphisme d'ensembles bisimpliciaux induit par $\NerfLaxNor{(u)}$. On construit alors un diagramme commutatif de morphismes d'ensembles bisimpliciaux
$$
\xymatrix{
\NerfLaxNor{\mathdeuxcat{A}}
\ar[rr]^{\NerfLaxNor{(u)}}
&&
\NerfLaxNor{\mathdeuxcat{B}}
\\
S_{w}
\ar[u]^{\varphi_{w}}
\ar[rr]_{U}
&&S_{v}
\ar[u]_{\varphi_{v}}
}
$$
comme suit. Pour tout couple d'entiers positifs $m$ et $n$, pour tout objet $({a}, {c})$ de $(S_{w})_{m,n}$, on pose $(\varphi_{w})_{m,n} ({a}, {c}) = {a}$ et $U_{m,n} ({a}, {c}) = ((\NerfLaxNor{(u)}) ({a}), {c})$. La définition de $\varphi_{v}$ est analogue à celle de $\varphi_{w}$.


Plaçons-nous désormais sous l'hypothèse que, pour tout \emph{objet} $c$ de $\mathdeuxcat{C}$, le \DeuxFoncteurStrict{} $\DeuxFoncTrancheCoLax{u}{c}$ est une équivalence faible. On se propose de montrer qu'il en est de même de $u$. 

Par construction, $\delta_{\Delta}^{*} (\NerfLaxNor{(u)}) = \NerfLaxNor{(u)}$. Pour démontrer que $u$ est une équivalence faible, il suffit donc de vérifier que $\delta_{\Delta}^{*} (U)$, $\delta_{\Delta}^{*} (\varphi_{w})$ et $\delta_{\Delta}^{*} (\varphi_{v})$ sont des équivalences faibles simpliciales. 

En vertu des propositions \ref{LemmeBisimplicial} et \ref{DeltaTotalementAspherique}, pour montrer que $\delta_{\Delta}^{*} (U)$ est une équivalence faible simpliciale, il suffit de montrer que, pour tout entier $n \geq 0$, le morphisme d'ensembles simpliciaux $U_{\bullet, n} : (S_{w})_{\bullet, n} \to (S_{v})_{\bullet, n}$ en est une. On remarque que la source et le but de ce morphisme s'identifient à
$$
\coprod_{{c} \in (\NerfLaxNor{\mathdeuxcat{C}})_{n}} \NerfLaxNor (\TrancheCoLax{\mathdeuxcat{A}}{w}{{c}})
$$
et
$$
\coprod_{{c} \in (\NerfLaxNor{\mathdeuxcat{C}})_{n}} \NerfLaxNor (\TrancheCoLax{\mathdeuxcat{B}}{v}{{c}})
$$
respectivement et $U_{\bullet, n}$ s'identifie à 
$$
\coprod_{{c} \in (\NerfLaxNor{\mathdeuxcat{C}})_{n}} \NerfLaxNor (\DeuxFoncTrancheCoLax{u}{{c}}).
$$
En vertu des hypothèses et du lemme \ref{WSurSimplexe}, chaque terme de cette somme est une équivalence faible. Il s'ensuit que $U_{\bullet, n}$ est une équivalence faible. Il en est donc de même de $\delta_{\Delta}^{*} (U)$.

Les raisonnements permettant de montrer que $\delta_{\Delta}^{*} (\varphi_{w})$ et $\delta_{\Delta}^{*} (\varphi_{v})$ sont des équivalences faibles sont évidemment analogues l'un à l'autre. Considérons le cas de $\delta_{\Delta}^{*} (\varphi_{w})$. Pour montrer que c'est une équivalence faible, il suffit, en vertu des propositions \ref{LemmeBisimplicial} et \ref{DeltaTotalementAspherique}, de montrer que, pour tout entier $m \geq 0$, le morphisme simplicial $(\varphi_{w})_{m, \bullet} : (S_{w})_{m, \bullet} \to (\NerfLaxNor{\mathdeuxcat{A}})_{m, \bullet}$ est une équivalence faible. On remarque que la source et le but de ce morphisme s'identifient à 
$$
\coprod_{{a} \in (\NerfLaxNor{\mathdeuxcat{A}})_{m}} \NerfLaxNor (\OpTrancheCoLax{\mathdeuxcat{C}}{}{(\NerfLaxNor(w)(a))})
$$
et
$$
\coprod_{{a} \in (\NerfLaxNor{\mathdeuxcat{A}})_{m}} \ast
$$
respectivement, $\ast$ désignant un ensemble simplicial final. Le morphisme simplicial $(\varphi_{w})_{m, \bullet}$ s'identifie ainsi à
$$
\coprod_{{a} \in (\NerfLaxNor{\mathdeuxcat{A}})_{m}} \NerfLaxNor (\OpTrancheCoLax{\mathdeuxcat{C}}{}{(\NerfLaxNor(w)(a))} \to \DeuxCatPonct).
$$
Comme la \deux{}catégorie $\OpTrancheCoLax{\mathdeuxcat{C}}{}{(\NerfLaxNor(w)(a))}$ est asphérique en vertu du lemme \ref{WSurSimplexe}, de l'exemple \ref{ExemplesOF} et du lemme \ref{OFAspherique1}, l'expression ci-dessus est une somme d'équivalences faibles, donc une équivalence faible. Ainsi, $(\varphi_{w})_{m, \bullet}$ est une équivalence faible. Il en est donc de même de $\delta_{\Delta}^{*} (\varphi_{w})$. Comme annoncé, il s'ensuit que $\NerfLaxNor{(u)}$ est une équivalence faible. On a donc démontré le théorème \ref{ThABC}, dont la version absolue constitue l'un des résultats principaux de \cite{BC}.
\end{paragr}

\begin{theo}\label{ThABC}
Soit
$$
\xymatrix{
\mathdeuxcat{A} 
\ar[rr]^{u}
\ar[dr]_{w}
&&\mathdeuxcat{B}
\ar[dl]^{v}
\\
&\mathdeuxcat{C}
}
$$
un triangle commutatif de \deux{}foncteurs stricts. Supposons que, pour tout objet $c$ de $\mathdeuxcat{C}$, le \DeuxFoncteurStrict{}
$$
\DeuxFoncTrancheCoLax{u}{c} : \TrancheCoLax{\mathdeuxcat{A}}{w}{c} \to \TrancheCoLax{\mathdeuxcat{B}}{v}{c}
$$
soit une équivalence faible. Alors, $u$ est une équivalence faible.
\end{theo}

\begin{rem}
Il est bien entendu possible de remplacer le \DeuxFoncteurStrict{} $\DeuxFoncTrancheCoLax{u}{c}$ par $\DeuxFoncTrancheLax{u}{c}$, $\DeuxFoncOpTrancheCoLax{u}{c}$ ou $\DeuxFoncOpTrancheLax{u}{c}$ dans l'énoncé du théorème \ref{ThABC}. On démontrera ces variantes plus loin, dans un cadre plus conceptuel.
\end{rem}

\chapter{Localisateurs fondamentaux de $\DeuxCat$}

\section{Premières définitions et propriétés}\label{SectionDeuxLocFond}

\begin{df}\label{DefDeuxLocFond}
Un \emph{localisateur fondamental de $\DeuxCat$}\index{localisateur fondamental de $\DeuxCat$} est une partie $\DeuxLocFond{W} \subset \UnCell{\DeuxCat}$ vérifiant les propriétés suivantes.
\begin{itemize}
\item[LF1] La classe $\DeuxLocFond{W}$ est faiblement saturée.

\item[LF2] Si une petite \deux{}catégorie $\mathdeuxcat{A}$ admet un objet admettant un objet final, alors le morphisme canonique $\mathdeuxcat{A} \to e$ est dans $\DeuxLocFond{W}$.

\item[LF3] Si 
$$
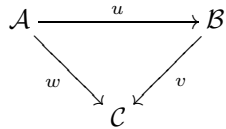
\xymatrix{
\mathdeuxcat{A} 
\ar[rr]^{u}
\ar[dr]_{w}
&&\mathdeuxcat{B}
\ar[dl]^{v}
\\
&\mathdeuxcat{C}
}
$$
désigne un triangle commutatif de \DeuxFoncteursStricts{} et si, pour tout objet $c$ de $\mathcal{C}$, le \DeuxFoncteurStrict{}
$$
\DeuxFoncTrancheCoLax{u}{c} : \TrancheCoLax{\mathdeuxcat{A}}{w}{c} \to \TrancheCoLax{\mathdeuxcat{B}}{v}{c} 
$$
est dans $\DeuxLocFond{W}$, alors $u$ est dans $\DeuxLocFond{W}$.
\end{itemize}
\end{df}

\begin{exemple}\label{Levet}
Si $\UnLocFond{W}$ est un \ClasseUnLocFond{}, posons 
$$
\begin{aligned}
\DeuxLocFond{W} &= \NerfLaxNor^{-1} (i_{\Delta}^{-1} (W))
\\
&= \NerfLax^{-1} (i_{\Delta}^{-1} (W)) \cap \UnCell{\DeuxCat},
\end{aligned}
$$ 
la seconde égalité résultant de l'équivalence des conditions $(i)$ et $(ii)$ du lemme \ref{Blabla}. Cette classe est un \ClasseDeuxLocFond{} en vertu du lemme \ref{DeuxLocFondSat}, du lemme \ref{OFAspherique1} et du théorème \ref{ThABC}. On montrera plus loin que tous les \ClassesDeuxLocFond{} s'obtiennent ainsi. 
\end{exemple}

On commettra sans plus le signaler l'abus anodin de considérer $\Cat$ comme une sous-catégorie (pleine) de $\DeuxCat$.

\begin{rem}\label{Proust}
Une petite catégorie admet un objet final si et seulement si, vue comme \deux{}catégorie, elle admet un objet admettant un objet final. De plus, si le triangle commutatif de la condition LF3 est un triangle de $\Cat$, le morphisme $\DeuxFoncTrancheCoLax{u}{c}$ n'est rien d'autre que $u/c$. Par conséquent, si $\DeuxLocFond{W}$ est un \ClasseDeuxLocFond{}, alors $\DeuxLocFond{W} \cap \UnCell{\Cat}$ est un \ClasseUnLocFond{}. On montrera plus loin que tous les \ClassesUnLocFond{} s'obtiennent ainsi. 
\end{rem}

\emph{On suppose fixé un localisateur fondamental $\DeuxLocFond{W}$ de $\DeuxCat$.}  

\begin{df}
On appellera les éléments de $\DeuxLocFond{W}$ des $\DeuxLocFond{W}\emph{-équivalences}$, \deux{}\emph{équivalences faibles} ou, plus simplement, si cela n'introduit aucune ambiguïté, \emph{équivalences faibles}\index{equivalence faible (\DeuxFoncteurStrict{}, pour un localisateur fondamental de $\DeuxCat$)@équivalence faible (\DeuxFoncteurStrict{}, pour un localisateur fondamental de $\DeuxCat$)}. On dira qu'une petite \deux{}catégorie $\mathdeuxcat{A}$ est \emph{$\DeuxLocFond{W}$-asphérique}, ou plus simplement \emph{asphérique}\index{asphérique (petite \deux{}catégorie, pour un localisateur fondamental de $\DeuxCat$)} en l'absence d'ambiguïté, si le morphisme $\mathdeuxcat{A} \to e$ est dans $\DeuxLocFond{W}$. 
\end{df}

\begin{rem}\label{SiEntreAsphAlorsAsph}
En vertu de la condition LF1, tout \DeuxFoncteurStrict{} entre \deux{}catégories asphériques est une équivalence faible.
\end{rem}

\begin{exemple}\label{TrancheCoLaxAspherique}
Pour toute petite \deux{}catégorie $\mathdeuxcat{A}$ et tout objet $a$ de $\mathdeuxcat{A}$, la \deux{}catégorie $\TrancheCoLax{\mathdeuxcat{A}}{}{a}$ est asphérique, en vertu de l'axiome LF2, puisque elle admet un objet admettant un objet final (exemple \ref{ExemplesOF}). En vertu de l'axiome LF3, en y posant $v = 1_{\mathdeuxcat{B}}$, on en déduit qu'étant donné un \DeuxFoncteurStrict{} $u : \mathdeuxcat{A} \to \mathdeuxcat{B}$, si, pour tout objet $b$ de $\mathdeuxcat{B}$, la \deux{}catégorie $\TrancheCoLax{\mathdeuxcat{A}}{u}{b}$ est asphérique, alors $u$ est une équivalence faible. De ce qui précède, on déduit qu'un préadjoint à gauche colax est une équivalence faible.
\end{exemple}

\begin{lemme}\label{LemmeDeGeorges}
Soient $\mathdeuxcat{A}$ une petite \deux{}catégorie, $z$ un objet de
$\mathdeuxcat{A}$ et $q : \DeuxCatPonct \to \mathdeuxcat{A}$ le
morphisme de $\DeuxCat$ défini par l'objet $z$. Alors, pour
tout objet $a$ de $\mathdeuxcat{A}$, $\TrancheCoLax{\DeuxCatPonct}{q}{a}$
est la \deux{}catégorie associée à la catégorie $\DeuxCatUnOp{(\CatHom{\mathdeuxcat{A}}{z}{a})}$.
\end{lemme}

\begin{proof}
C'est immédiat.
\end{proof}

\begin{lemme}\label{OFAspherique2}
Une petite \deux{}catégorie op-admettant un objet admettant un objet initial est asphé\-rique. 
\end{lemme}

\begin{proof}
Soient $\mathdeuxcat{A}$ une \deux{}catégorie op-admettant un objet admettant un objet initial et $z$ un objet de $\mathdeuxcat{A}$ tel que, pour tout objet $a$ de $\mathdeuxcat{A}$, la catégorie $\CatHom{\mathdeuxcat{A}}{z}{a}$ admette un objet initial. Considérons les \DeuxFoncteursStricts{} $p_{\mathdeuxcat{A}} : \mathdeuxcat{A} \to \DeuxCatPonct$ et $q_{\mathdeuxcat{A}} : \DeuxCatPonct \to \mathdeuxcat{A}$, le second étant défini par $q_\mathdeuxcat{A}(*)= z$. En vertu du lemme \ref{LemmeDeGeorges}, pour tout objet $a$ de $\mathdeuxcat{A}$, la \deux{}catégorie $\TrancheCoLax{\DeuxCatPonct}{q_\mathdeuxcat{A}}{a}$ admet un objet admettant un objet final (elle admet également un objet admettant un objet initial). On en déduit que la \deux{}catégorie $\TrancheCoLax{\DeuxCatPonct}{q_\mathdeuxcat{A}}{a}$ est asphérique pour tout objet $a$ de $\mathdeuxcat{A}$. Par conséquent, $q_{\mathdeuxcat{A}}$ une équivalence faible. Comme $p_{\mathdeuxcat{A}} q_{\mathdeuxcat{A}} = 1_{\DeuxCatPonct}$, il résulte de la saturation faible de $\DeuxLocFond{W}$ que $p_{\mathdeuxcat{A}}$ est une équivalence faible. Par définition, la \deux{}catégorie $\mathdeuxcat{A}$ est donc asphérique.
\end{proof}

\begin{corollaire}\label{OpTrancheLaxAspherique}
Pour toute petite \deux{}catégorie $\mathdeuxcat{A}$ et tout objet $a$ de $\mathdeuxcat{A}$, la \deux{}catégorie $\OpTrancheLax{\mathdeuxcat{A}}{}{a}$ est asphérique. 
\end{corollaire} 

\begin{proof}
La \deux{}catégorie $\OpTrancheLax{\mathdeuxcat{A}}{}{a}$ op-admettant un objet admettant un objet initial (voir l'exemple \ref{ExemplesOF}), cela résulte du lemme \ref{OFAspherique2}.
\end{proof}

\begin{lemme}\label{OFAspherique2Bis}
Une petite \deux{}catégorie op-admettant un objet admettant un objet final est asphé\-rique. 
\end{lemme}

\begin{proof}
Soit $\mathdeuxcat{A}$ une \deux{}catégorie op-admettant un objet admettant un objet final et soit $z$ un objet de $\mathdeuxcat{A}$ tel que, pour tout objet $a$ de $\mathdeuxcat{A}$, la catégorie $\CatHom{\mathdeuxcat{A}}{z}{a}$ admette un objet final. Considérons les \DeuxFoncteursStricts{} $p_{\mathdeuxcat{A}} : \mathdeuxcat{A} \to \DeuxCatPonct$ et $q_{\mathdeuxcat{A}} : \DeuxCatPonct \to \mathdeuxcat{A}$, le second étant défini par $q_\mathdeuxcat{A}(*)= z$. En vertu du lemme \ref{LemmeDeGeorges}, pour tout objet $a$ de $\mathdeuxcat{A}$, la \deux{}catégorie $\TrancheCoLax{\DeuxCatPonct}{q_\mathdeuxcat{A}}{a}$ op-admet un objet admettant un objet initial (elle op-admet également un objet admettant un objet final). On en déduit, en vertu du lemme \ref{OFAspherique2}, que la \deux{}catégorie $\TrancheCoLax{\DeuxCatPonct}{q_\mathdeuxcat{A}}{a}$ est asphérique pour tout objet $a$ de $\mathdeuxcat{A}$. Par conséquent, $q_{\mathdeuxcat{A}}$ est une équivalence faible. La saturation faible de $\DeuxLocFond{W}$ permet d'en conclure que la \deux{}catégorie $\mathdeuxcat{A}$ est asphérique. 
\end{proof}

\begin{corollaire}\label{OpTrancheCoLaxAspherique}
Pour toute petite \deux{}catégorie $\mathdeuxcat{A}$ et tout objet $a$ de $\mathdeuxcat{A}$, la \deux{}catégorie $\OpTrancheCoLax{\mathdeuxcat{A}}{}{a}$ est asphérique. 
\end{corollaire} 

\begin{proof}
La \deux{}catégorie $\OpTrancheCoLax{\mathdeuxcat{A}}{}{a}$ op-admettant un objet admettant un objet final (voir l'exemple \ref{ExemplesOF}), cela résulte du lemme \ref{OFAspherique2Bis}.
\end{proof}

\begin{prop}\label{PrefibrationIntegraleW}
Soient $\mathdeuxcat{A}$ une \deux{}catégorie et $F : \DeuxCatDeuxOp{\mathdeuxcat{A}} \to \DeuxCatDeuxCat$ un \DeuxFoncteurStrict{} tel que, pour tout objet $a$ de $\mathdeuxcat{A}$, la \deux{}catégorie $F(a)$ soit asphérique. Alors, la projection canonique
$$
P_{F} : \DeuxIntCo{\mathdeuxcat{A}} F \to \mathdeuxcat{A}
$$
est une équivalence faible. 
\end{prop}

\begin{proof}
La fibre $\Fibre{(\DeuxIntCo{\mathdeuxcat{A}} F)}{P_{F}}{a}$ de $P_{F}$ au-dessus de $a$ s'identifiant à $F(a)$, elle est asphérique en vertu des hypothèses. En vertu de la proposition \ref{ProprietesJaKaCo}, il existe un \DeuxFoncteurStrict{} 
$$
K_{a} : \TrancheCoLax{\left(\DeuxIntCo{\mathdeuxcat{A}} F\right)}{P_{F}}{a} \to \Fibre{\left(\DeuxIntCo{\mathdeuxcat{A}} F\right)}{P_{F}}{a}
$$  
qui est un préadjoint à gauche colax, donc une équivalence faible (exemple \ref{TrancheCoLaxAspherique}). Par conséquent, la \deux{}catégorie $\TrancheCoLax{(\DeuxIntCo{\mathdeuxcat{A}} F)}{P_{F}}{a}$ est asphérique. Il en résulte que $P_{F}$ est une équivalence faible (voir l'exemple \ref{TrancheCoLaxAspherique}).
\end{proof}

\begin{paragr}
Soit $u : \mathdeuxcat{A} \to \mathdeuxcat{B}$ un \DeuxFoncteurStrict{}. On rappelle avoir construit, dans le paragraphe \ref{DefS1}, un diagramme commutatif 
$$
\xymatrix{
\DeuxCatUnOp{\mathdeuxcat{A}}
\ar[d]_{\DeuxFoncUnOp{u}}
&S_{1}(\mathdeuxcat{A})
\ar[l]_{s_{1}^{\mathdeuxcat{A}}}
\ar[d]^{S_{1}(u)}
\ar[r]^{t_{1}^{\mathdeuxcat{A}}}
&\mathdeuxcat{A}
\ar[d]^{u}
\\
\DeuxCatUnOp{\mathdeuxcat{B}}
&S_{1}(\mathdeuxcat{B}{})
\ar[l]^{s_{1}^{\mathdeuxcat{B}{}}}
\ar[r]_{t_{1}^{\mathdeuxcat{B}{}}}
&\mathdeuxcat{B}{}
}
$$
avec 
$$
S_{1}(\mathdeuxcat{A}) = 
\DeuxIntCo{\DeuxCatUnOp{\mathdeuxcat{A}}} (\OpTrancheLax{\mathdeuxcat{A}}{}{a}) =
\DeuxIntCo{\mathdeuxcat{A}} \DeuxCatUnOp{(\TrancheLax{\mathdeuxcat{A}}{}{a})} 
$$
et
$$
S_{1}(\mathdeuxcat{B}) = 
\DeuxIntCo{\DeuxCatUnOp{\mathdeuxcat{B}}} (\OpTrancheLax{\mathdeuxcat{B}}{}{b}) =
\DeuxIntCo{\mathdeuxcat{B}} \DeuxCatUnOp{(\TrancheLax{\mathdeuxcat{B}}{}{b})}.
$$
(Les symboles d'égalité désignent des isomorphismes tout ce qu'il y a de canoniques.) Les flèches horizontales désignent les projections canoniques. Les \deux{}catégories $\OpTrancheLax{\mathdeuxcat{A}}{}{a}$ et $\OpTrancheLax{\mathdeuxcat{B}}{}{b}$ sont asphériques en vertu du corollaire \ref{OpTrancheLaxAspherique}. Les \deux{}catégories $\DeuxCatUnOp{\left(\TrancheLax{\mathdeuxcat{A}}{}{a}\right)}$ et $\DeuxCatUnOp{\left(\TrancheLax{\mathdeuxcat{B}}{}{b}\right)}$ sont asphériques en vertu de l'exemple \ref{ExemplesOF} et du lemme \ref{OFAspherique2}. Ainsi, en vertu de la proposition \ref{PrefibrationIntegraleW}, les quatre flèches horizontales du diagramme ci-dessus sont des équivalences faibles. 
\end{paragr}

\begin{prop}\label{DeuxFoncUnOpW}
Un \DeuxFoncteurStrict{} $u$ est une équivalence faible si et seulement si $\DeuxFoncUnOp{u}$ est une équivalence faible.
\end{prop}

\begin{proof}
Cela résulte de ce qui précède par deux arguments de « 2 sur 3 » consécutifs. 
\end{proof}

\begin{corollaire}\label{DeuxCatUnOpAsph}
Une petite \deux{}catégorie $\mathdeuxcat{A}$ est asphérique si et seulement si $\DeuxCatUnOp{\mathdeuxcat{A}}$ est asphérique.
\end{corollaire}

\begin{proof}
Il suffit d'appliquer la proposition \ref{DeuxFoncUnOpW} au morphisme canonique $\mathdeuxcat{A} \to \DeuxCatPonct$. 
\end{proof}

\begin{corollaire}\label{OFAspherique3}
Une petite \deux{}catégorie admettant un objet admettant un objet initial est asphérique.
\end{corollaire}

\begin{proof}
Si une petite \deux{}catégorie $\mathdeuxcat{A}$ admet un objet admettant un objet initial, alors $\DeuxCatUnOp{\mathdeuxcat{A}}$ op-admet un objet admettant un objet initial, donc est asphérique (lemme \ref{OFAspherique2}). On conclut grâce au corollaire \ref{DeuxCatUnOpAsph}.
\end{proof}

\begin{corollaire}\label{TrancheLaxAspherique}
Pour toute petite \deux{}catégorie $\mathdeuxcat{A}$ et tout objet $a$ de $\mathdeuxcat{A}$, la \deux{}catégorie $\TrancheLax{\mathdeuxcat{A}}{}{a}$ est asphérique. 
\end{corollaire}

\begin{proof}
En vertu de l'exemple \ref{ExemplesOF}, cela résulte du corollaire \ref{OFAspherique3}.
\end{proof}

\begin{paragr}
Soit $u : \mathdeuxcat{A} \to \mathdeuxcat{B}$ un \DeuxFoncteurStrict{}. On rappelle avoir construit, dans le paragraphe \ref{DefS2}, un diagramme commutatif 
$$
\xymatrix{
\DeuxCatDeuxOp{\mathdeuxcat{A}}
\ar[d]_{\DeuxFoncDeuxOp{u}}
&S_{2}(\mathdeuxcat{A})
\ar[l]_{s_{2}^{\mathdeuxcat{A}}}
\ar[d]^{S_{2}(u)}
\ar[r]^{t_{2}^{\mathdeuxcat{A}}}
&\mathdeuxcat{A}
\ar[d]^{u}
\\
\DeuxCatDeuxOp{\mathdeuxcat{B}}
&S_{2}(\mathdeuxcat{B}{})
\ar[l]^{s_{2}^{\mathdeuxcat{B}{}}}
\ar[r]_{t_{2}^{\mathdeuxcat{B}{}}}
&\mathdeuxcat{B}{}
}
$$
avec 
$$
S_{2}(\mathdeuxcat{A}) =
\DeuxCatUnOp{\left(\DeuxIntCo{\DeuxCatToutOp{\mathdeuxcat{A}}} 
\DeuxCatUnOp{(\OpTrancheLax{\mathdeuxcat{A}}{}{a})}\right)}
=
\DeuxIntCo{\mathdeuxcat{A}} \DeuxCatDeuxOp{(\TrancheCoLax{\mathdeuxcat{A}}{}{a})}  
$$
et
$$
S_{2}(\mathdeuxcat{B}) =
\DeuxCatUnOp{\left(\DeuxIntCo{\DeuxCatToutOp{\mathdeuxcat{B}}} 
\DeuxCatUnOp{(\OpTrancheLax{\mathdeuxcat{B}}{}{b})}\right)}
=
\DeuxIntCo{\mathdeuxcat{B}} \DeuxCatDeuxOp{(\TrancheCoLax{\mathdeuxcat{B}}{}{b})}.
$$
(Les symboles d'égalité désignent des isomorphismes tout ce qu'il y a de canoniques.) Les flèches horizontales désignent les projections canoniques. En vertu de la proposition \ref{PrefibrationIntegraleW} et du corollaire \ref{OFAspherique3}, $t_{2}^{\mathdeuxcat{A}}$ et $t_{2}^{\mathdeuxcat{B}}$ sont des équivalences faibles. Toujours en vertu de la proposition \ref{PrefibrationIntegraleW} et du corollaire \ref{OFAspherique3}, les projections canoniques
$$
\DeuxIntCo{\DeuxCatToutOp{\mathdeuxcat{A}}} 
\DeuxCatUnOp{(\OpTrancheLax{\mathdeuxcat{A}}{}{a})} \to \DeuxCatToutOp{\mathdeuxcat{A}}
$$ 
et 
$$
\DeuxIntCo{\DeuxCatToutOp{\mathdeuxcat{B}}} 
\DeuxCatUnOp{(\OpTrancheLax{\mathdeuxcat{B}}{}{b})} \to \DeuxCatToutOp{\mathdeuxcat{B}}
$$ 
sont des équivalences faibles. Comme elles s'identifient à $\DeuxFoncUnOp{\left(s_{2}^{\mathdeuxcat{A}}\right)}$ et $\DeuxFoncUnOp{\left(s_{2}^{\mathdeuxcat{B}}\right)}$, $s_{2}^{\mathdeuxcat{A}}$ et $s_{2}^{\mathdeuxcat{B}}$ sont des équivalences faibles en vertu de la proposition \ref{DeuxFoncUnOpW}. \end{paragr}

\begin{prop}\label{DeuxFoncDeuxOpW}
Un \DeuxFoncteurStrict{} $u$ est une équivalence faible si et seulement si $\DeuxFoncDeuxOp{u}$ en est une.
\end{prop}

\begin{proof}
Cela résulte de ce qui précède par deux arguments de « 2 sur 3 » consécutifs. 
\end{proof}

\begin{corollaire}\label{DeuxCatDeuxOpAsph}
Une petite \deux{}catégorie $\mathdeuxcat{A}$ est asphérique si et seulement si $\DeuxCatDeuxOp{\mathdeuxcat{A}}$ est asphérique.
\end{corollaire}

\begin{proof}
Il suffit d'appliquer la proposition \ref{DeuxFoncDeuxOpW} au morphisme canonique $\mathdeuxcat{A} \to \DeuxCatPonct$.
\end{proof}

\begin{prop}\label{DeuxFoncToutOpW}
Un \DeuxFoncteurStrict{} $u$ est une équivalence faible si et seulement si $\DeuxFoncToutOp{u}$ en est une.
\end{prop}

\begin{proof}
Cela résulte des propositions \ref{DeuxFoncUnOpW} et \ref{DeuxFoncDeuxOpW}.
\end{proof}

\begin{corollaire}\label{DeuxCatToutOpAsph}
Une \deux{}catégorie $\mathdeuxcat{A}$ est asphérique si et seulement si $\DeuxCatToutOp{\mathdeuxcat{A}}$ est asphérique.
\end{corollaire}

\begin{proof}
Il suffit d'appliquer la proposition \ref{DeuxFoncToutOpW} au morphisme canonique $\mathdeuxcat{A} \to \DeuxCatPonct$.
\end{proof}

\begin{prop}\label{TheoremeAOpTrancheCoLax}
Soit 
$$
\xymatrix{
\mathdeuxcat{A} 
\ar[rr]^{u}
\ar[dr]_{w}
&&\mathdeuxcat{B}
\ar[dl]^{v}
\\
&\mathdeuxcat{C}
}
$$
un diagramme commutatif dans $\DeuxCat$. Supposons que, pour tout objet $c$ de $\mathdeuxcat{C}$, le morphisme
$$
\DeuxFoncOpTrancheCoLax{u}{c} : \OpTrancheCoLax{\mathdeuxcat{A}}{w}{c} \to \OpTrancheCoLax{\mathdeuxcat{B}}{v}{c} 
$$
soit une équivalence faible. Alors, $u$ est une équivalence faible.
\end{prop}

\begin{proof}
En vertu des hypothèses et par définition, le \DeuxFoncteurStrict{} $\DeuxFoncUnOp{(\DeuxFoncTrancheCoLax{\DeuxFoncUnOp{u}}{c})}$ est une équivalence faible pour tout objet $c$ de $\mathdeuxcat{C}$. Par conséquent, en vertu de la proposition \ref{DeuxFoncUnOpW}, le \DeuxFoncteurStrict{} $\DeuxFoncTrancheCoLax{\DeuxFoncUnOp{u}}{c}$ est une équivalence faible pour tout objet $c$ de $\mathdeuxcat{C}$. En vertu de la condition LF$3$, le \DeuxFoncteurStrict{} $\DeuxFoncUnOp{u}$ est donc une équivalence faible. Une nouvelle invocation de la proposition \ref{DeuxFoncUnOpW} permet de conclure.
\end{proof} 

\begin{prop}\label{TheoremeATrancheLax}
Soit 
$$
\xymatrix{
\mathdeuxcat{A} 
\ar[rr]^{u}
\ar[dr]_{w}
&&\mathdeuxcat{B}
\ar[dl]^{v}
\\
&\mathdeuxcat{C}
}
$$
un diagramme commutatif dans $\DeuxCat$. Supposons que, pour tout objet $c$ de $\mathdeuxcat{C}$, le morphisme
$$
\DeuxFoncTrancheLax{u}{c} : \TrancheLax{\mathdeuxcat{A}}{w}{c} \to \TrancheLax{\mathdeuxcat{B}}{v}{c} 
$$
soit une équivalence faible. Alors, $u$ est une équivalence faible.
\end{prop}

\begin{proof}
En vertu des hypothèses et par définition, le \DeuxFoncteurStrict{} $\DeuxFoncDeuxOp{(\DeuxFoncTrancheCoLax{\DeuxFoncDeuxOp{u}}{c})}$ est une équivalence faible pour tout objet $c$ de $\mathdeuxcat{C}$. Par conséquent, en vertu de la proposition \ref{DeuxFoncDeuxOpW}, le \DeuxFoncteurStrict{} $\DeuxFoncTrancheCoLax{\DeuxFoncDeuxOp{u}}{c}$ est une équivalence faible pour tout objet $c$ de $\mathdeuxcat{C}$. En vertu de la condition LF$3$, le \DeuxFoncteurStrict{} $\DeuxFoncDeuxOp{u}$ est donc une équivalence faible. Une nouvelle invocation de la proposition \ref{DeuxFoncDeuxOpW} permet de conclure.
\end{proof} 

\begin{prop}\label{TheoremeAOpTrancheLax}
Soit 
$$
\xymatrix{
\mathdeuxcat{A} 
\ar[rr]^{u}
\ar[dr]_{w}
&&\mathdeuxcat{B}
\ar[dl]^{v}
\\
&\mathdeuxcat{C}
}
$$
un diagramme commutatif dans $\DeuxCat$. Supposons que, pour tout objet $c$ de $\mathdeuxcat{C}$, le morphisme
$$
\DeuxFoncOpTrancheLax{u}{c} : \OpTrancheLax{\mathdeuxcat{A}}{w}{c} \to \OpTrancheLax{\mathdeuxcat{B}}{v}{c} 
$$
soit une équivalence faible. Alors, $u$ est une équivalence faible.
\end{prop}

\begin{proof}
En vertu des hypothèses et par définition, le \DeuxFoncteurStrict{} $\DeuxFoncToutOp{(\DeuxFoncTrancheCoLax{\DeuxFoncToutOp{u}}{c})}$ est une équivalence faible pour tout objet $c$ de $\mathdeuxcat{C}$. Par conséquent, en vertu de la proposition \ref{DeuxFoncToutOpW}, le \DeuxFoncteurStrict{} $\DeuxFoncTrancheCoLax{\DeuxFoncToutOp{u}}{c}$ est une équivalence faible pour tout objet $c$ de $\mathdeuxcat{C}$. En vertu de la condition LF3, le \DeuxFoncteurStrict{} $\DeuxFoncToutOp{u}$ est donc une équivalence faible. Une nouvelle invocation de la proposition \ref{DeuxFoncToutOpW} permet de conclure.
\end{proof} 

\begin{df}\label{Zonzon}
Étant donné un diagramme commutatif 
$$
\xymatrix{
\mathdeuxcat{A} 
\ar[rr]^{u}
\ar[dr]_{w}
&&\mathdeuxcat{B}
\ar[dl]^{v}
\\
&\mathdeuxcat{C}
}
$$
de \DeuxFoncteursStricts{}, on dira que $u$ est \emph{lax-asphérique au-dessus de\index{lax-asphérique au-dessus de} $\mathdeuxcat{C}$} (\emph{resp.} \emph{lax-opasphérique au-dessus de\index{lax-opasphérique au-dessus de} $\mathdeuxcat{C}$}, \emph{resp.} \emph{colax-asphérique au-dessus de\index{colax-asphérique au-dessus de} $\mathdeuxcat{C}$}, \emph{resp.} \emph{colax-opasphérique au-dessus de\index{colax-opasphérique au-dessus de} $\mathdeuxcat{C}$}) si, pour tout objet $c$ de $\mathdeuxcat{C}$, le \DeuxFoncteurStrict{} $\DeuxFoncTrancheLax{u}{c}$ (\emph{resp.} $\DeuxFoncOpTrancheLax{u}{c}$, \emph{resp.} $\DeuxFoncTrancheCoLax{u}{}{c}$, \emph{resp.} $\DeuxFoncOpTrancheCoLax{u}{c}$) est une équivalence faible. Si $v = 1_{\mathdeuxcat{B}}$, on dira simplement que $u$ est \emph{lax-asphérique}\index{lax-asphérique} (\emph{resp.} \emph{lax-opasphérique}\index{lax-opasphérique}, \emph{resp.} \emph{colax-asphérique}\index{colax-asphérique}, \emph{resp.} \emph{colax-opasphérique}\index{colax-opasphérique}).
\end{df}

\begin{rem}\label{Zozo}
En vertu des résultats ci-dessus, sous les données de la définition \ref{Zonzon}, le \DeuxFoncteurStrict{} $u$ est une équivalence faible pour peu qu'il soit lax-asphérique, lax-opasphérique, colax-asphérique ou colax-opasphérique au-dessus de $\mathdeuxcat{C}$. En faisant $v = 1_{\mathdeuxcat{B}}$, on obtient le cas particulier suivant : pour tout \DeuxFoncteurStrict{} $u : \mathdeuxcat{A} \to \mathdeuxcat{B}$, si, pour tout objet $b$ de $\mathdeuxcat{B}$, la \deux{}catégorie $\TrancheLax{\mathdeuxcat{A}}{u}{b}$ (\emph{resp.} $\OpTrancheLax{\mathdeuxcat{A}}{u}{b}$, \emph{resp.} $\TrancheCoLax{\mathdeuxcat{A}}{u}{b}$, \emph{resp.} $\OpTrancheCoLax{\mathdeuxcat{A}}{u}{b}$) est asphérique, alors $u$ est une équivalence faible (voir aussi l'exemple \ref{TrancheCoLaxAspherique}).
\end{rem}

\begin{corollaire}\label{PreadjointsW}
Un préadjoint à gauche lax (\emph{resp.} un préadjoint à droite lax, \emph{resp.} un préadjoint à droite colax) est une équivalence faible.
\end{corollaire}

\begin{proof}
C'est une conséquence de la remarque \ref{Zozo}. 
\end{proof}

\begin{prop}\label{PrefibrationFibresAspheriquesW}
Une préfibration à fibres asphériques est une équivalence faible.
\end{prop}

\begin{proof}
Soit $u : \mathdeuxcat{A} \to \mathdeuxcat{B}$ une préfibration à fibres asphériques. Par définition, pour tout objet $b$ de $\mathdeuxcat{B}$, le \DeuxFoncteurStrict{} canonique $J_{b} : \Fibre{\mathdeuxcat{A}}{u}{b} \to \OpTrancheCoLax{\mathdeuxcat{A}}{u}{b}$ est un préadjoint à gauche lax. En vertu du corollaire \ref{PreadjointsW}, c'est donc une équivalence faible. Par conséquent, $\OpTrancheCoLax{\mathdeuxcat{A}}{u}{b}$ est asphérique. On conclut en invoquant la remarque \ref{Zozo}. 
\end{proof}

\begin{prop}\label{PreopfibrationFibresAspheriquesW}
Une préopfibration à fibres asphériques est une équivalences faible.
\end{prop}

\begin{proof}
Soit $u : \mathdeuxcat{A} \to \mathdeuxcat{B}$ une préopfibration à fibres asphériques. Alors $\DeuxFoncUnOp{u} : \DeuxCatUnOp{\mathdeuxcat{A}} \to \DeuxCatUnOp{\mathdeuxcat{B}}$ est une préfibration à fibres asphériques en vertu des définitions, du lemme \ref{FibresCoOp} et du corollaire \ref{DeuxCatUnOpAsph}. On conclut grâce aux propositions \ref{PrefibrationFibresAspheriquesW} et \ref{DeuxFoncUnOpW}.
\end{proof}

\begin{prop}\label{PrecofibrationFibresAspheriquesW}
Une précofibration à fibres asphériques est une équivalences faible.
\end{prop}

\begin{proof}
Soit $u : \mathdeuxcat{A} \to \mathdeuxcat{B}$ une précofibration à fibres asphériques. Alors $\DeuxFoncDeuxOp{u} : \DeuxCatDeuxOp{\mathdeuxcat{A}} \to \DeuxCatDeuxOp{\mathdeuxcat{B}}$ est une préfibration à fibres asphériques en vertu des définitions, du lemme \ref{FibresCoOp} et du corollaire \ref{DeuxCatDeuxOpAsph}. On conclut grâce aux propositions \ref{PrefibrationFibresAspheriquesW} et \ref{DeuxFoncDeuxOpW}.
\end{proof}

\begin{prop}\label{PrecoopfibrationFibresAspheriquesW}
Une précoopfibration à fibres asphériques est une équivalences faible.
\end{prop}

\begin{proof}
Soit $u : \mathdeuxcat{A} \to \mathdeuxcat{B}$ une précoopfibration à fibres asphériques. Alors $\DeuxFoncToutOp{u} : \DeuxCatToutOp{\mathdeuxcat{A}} \to \DeuxCatToutOp{\mathdeuxcat{B}}$ est une préfibration à fibres asphériques en vertu des définitions, du lemme \ref{FibresCoOp} et du corollaire \ref{DeuxCatToutOpAsph}. On conclut grâce aux propositions \ref{PrefibrationFibresAspheriquesW} et \ref{DeuxFoncToutOpW}. 
\end{proof}

\begin{prop}\label{2.1.10.THG}
Soit
$$
\xymatrix{
\mathdeuxcat{A} 
\ar[rr]^{u}
\ar[dr]_{w}
&&\mathdeuxcat{B}
\ar[dl]^{v}
\\
&\mathdeuxcat{C}
}
$$
un diagramme commutatif dans $\DeuxCat$. On suppose que, pour tout objet $c$ de $\mathdeuxcat{C}$, le morphisme $u_{c} : \Fibre{\mathdeuxcat{A}}{w}{c} \to \Fibre{\mathdeuxcat{B}}{v}{c}$, induit entre les fibres de $w$ et $v$ au-dessus de $c$, est une équivalence faible. 
\begin{itemize}
\item[(a)]
Si $v$ et $w$ sont des préfibrations, alors $u$ est colax-opasphérique au-dessus de $\mathdeuxcat{C}$. 
\item[(b)]
Si $v$ et $w$ sont des préopfibrations, alors $u$ est colax-asphérique au-dessus de $\mathdeuxcat{C}$. 
\item[(c)]
Si $v$ et $w$ sont des précofibrations, alors $u$ est lax-opasphérique au-dessus de $\mathdeuxcat{C}$.
\item[(d)]
Si $v$ et $w$ sont des précoopfibrations, alors $u$ est lax-asphérique au-dessus de $\mathdeuxcat{C}$. 
\end{itemize}
En particulier, dans n'importe lequel de ces quatre cas, $u$ est une équivalence faible.
\end{prop}

\begin{proof}
Plaçons-nous dans le premier cas, les trois autres s'en déduisant par dualité. Pour tout objet $c$ de $\mathdeuxcat{C}$, on a un carré commutatif
$$
\xymatrix{
\Fibre{\mathdeuxcat{A}}{w}{c}
\ar[r]^{u_{c}}
\ar[d]_{J_{c}}
&
\Fibre{\mathdeuxcat{B}}{v}{c}
\ar[d]^{J_{c}}
\\
\OpTrancheCoLax{\mathdeuxcat{A}}{w}{c}
\ar[r]_{\DeuxFoncOpTrancheCoLax{u}{c}}
&
\OpTrancheCoLax{\mathdeuxcat{B}}{v}{c}
}
$$
dont les flèches verticales sont des préadjoints à gauche lax, donc des équivalences faibles (corollaire \ref{PreadjointsW}). Comme, par hypothèse, $u_{c}$ est une équivalence faible, il en est de même de $\DeuxFoncOpTrancheCoLax{u}{c}$.
\end{proof}

\emph{On ne suppose plus fixé de \ClasseDeuxLocFond{}.}

\begin{theo}\label{DeuxLocFondHuitDef}
Soit $\DeuxLocFond{W}$ une partie de $\UnCell{\DeuxCat}$. Les conditions suivantes sont équivalentes :
\begin{itemize}
\item[(i)] 
La classe $\DeuxLocFond{W}$ est un \ClasseDeuxLocFond{}.
\item[(ii)]
Les conditions suivantes sont vérifiées.  
\begin{itemize}
\item[LF1$'$] 
La partie $\DeuxLocFond{W}$ de $\UnCell{\DeuxCat}$ est faiblement saturée. 
\item[LF2$'$] 
Si une petite \deux{}catégorie $\mathdeuxcat{A}$ admet un objet admettant un objet initial, alors le morphisme canonique $\mathdeuxcat{A} \to \DeuxCatPonct$ est dans $\DeuxLocFond{W}$. 
\item[LF3$'$] Si 
$$
\xymatrix{
\mathdeuxcat{A} 
\ar[rr]^{u}
\ar[dr]_{w}
&&\mathdeuxcat{B}
\ar[dl]^{v}
\\
&\mathdeuxcat{C}
}
$$
désigne un triangle commutatif dans $\DeuxCat$ et si, pour tout objet $c$ de $\mathdeuxcat{C}$, le \deux{}foncteur strict 
$$
\DeuxFoncTrancheLax{u}{c} : \TrancheLax{\mathdeuxcat{A}}{w}{c} \to \TrancheLax{\mathdeuxcat{B}}{v}{c} 
$$
est dans $\DeuxLocFond{W}$, alors $u$ est dans $\DeuxLocFond{W}$.
\end{itemize}
\item[(iii)]
Les conditions suivantes sont vérifiées.
\begin{itemize}
\item[LF1$''$] 
La partie $\DeuxLocFond{W}$ de $\UnCell{\DeuxCat}$ est faiblement saturée. 
\item[LF2$''$] 
Si une petite \deux{}catégorie $\mathdeuxcat{A}$ op-admet un objet admettant un objet final, alors le morphisme canonique $\mathdeuxcat{A} \to \DeuxCatPonct$ est dans $\DeuxLocFond{W}$. 
\item[LF3$''$] Si 
$$
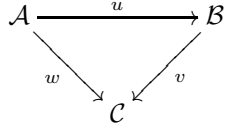
\xymatrix{
\mathdeuxcat{A} 
\ar[rr]^{u}
\ar[dr]_{w}
&&\mathdeuxcat{B}
\ar[dl]^{v}
\\
&\mathdeuxcat{C}
}
$$
désigne un triangle commutatif dans $\DeuxCat$ et si, pour tout objet $c$ de $\mathdeuxcat{C}$, le \deux{}foncteur strict
$$
\DeuxFoncOpTrancheCoLax{u}{c} : \OpTrancheCoLax{\mathdeuxcat{A}}{w}{c} \to \OpTrancheCoLax{\mathdeuxcat{B}}{v}{c} 
$$
est dans $\DeuxLocFond{W}$, alors $u$ est dans $\DeuxLocFond{W}$.
\end{itemize}
\item[(iv)]
Les conditions suivantes sont vérifiées.
\begin{itemize}
\item[LF1$'''$] 
La partie $\DeuxLocFond{W}$ de $\UnCell{\DeuxCat}$ est faiblement saturée. 
\item[LF2$'''$] 
Si une petite \deux{}catégorie $\mathdeuxcat{A}$ op-admet un objet admettant un objet initial, alors le morphisme canonique $\mathdeuxcat{A} \to \DeuxCatPonct$ est dans $\DeuxLocFond{W}$. 
\item[LF3$'''$] Si 
$$
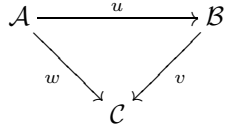
\xymatrix{
\mathdeuxcat{A} 
\ar[rr]^{u}
\ar[dr]_{w}
&&\mathdeuxcat{B}
\ar[dl]^{v}
\\
&\mathdeuxcat{C}
}
$$
désigne un triangle commutatif dans $\DeuxCat$ et si, pour tout objet $c$ de $\mathdeuxcat{C}$, le \deux{}foncteur strict
$$
\DeuxFoncOpTrancheLax{u}{c} : \OpTrancheLax{\mathdeuxcat{A}}{w}{c} \to \OpTrancheLax{\mathdeuxcat{B}}{v}{c} 
$$
est dans $\DeuxLocFond{W}$, alors $u$ est dans $\DeuxLocFond{W}$.
\end{itemize}
\item[(v)]
Les conditions suivantes sont vérifiées.
\begin{itemize}
\item[LF$\alpha$] 
La partie $\DeuxLocFond{W}$ de $\UnCell{\DeuxCat}$ est faiblement saturée. 
\item[LF$\beta$] 
Le morphisme canonique $[1] \to \DeuxCatPonct$ est dans $\DeuxLocFond{W}$. 
\item[LF$\gamma$] Si 
$$
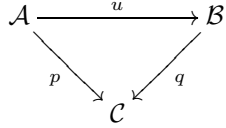
\xymatrix{
\mathdeuxcat{A} 
\ar[rr]^{u}
\ar[dr]_{p}
&&\mathdeuxcat{B}
\ar[dl]^{q}
\\
&\mathdeuxcat{C}
}
$$
désigne un triangle commutatif dans $\DeuxCat$, si $p$ et $q$ sont des précoopfibrations et si, pour tout objet $c$ de $\mathdeuxcat{C}$, le \DeuxFoncteurStrict{} induit entre les fibres
$$
\Fibre{u}{}{c} : \Fibre{\mathdeuxcat{A}}{p}{c} \to \Fibre{\mathdeuxcat{B}}{q}{c} 
$$
est dans $\DeuxLocFond{W}$, alors $u$ est dans $\DeuxLocFond{W}$.
\end{itemize}
\item[(vi)]
Les conditions suivantes sont vérifiées.
\begin{itemize}
\item[LF$\alpha'$] 
La partie $\DeuxLocFond{W}$ de $\UnCell{\DeuxCat}$ est faiblement saturée. 
\item[LF$\beta'$] 
Le morphisme canonique $[1] \to \DeuxCatPonct$ est dans $\DeuxLocFond{W}$. 
\item[LF$\gamma'$] Si 
$$
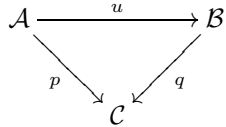
\xymatrix{
\mathdeuxcat{A} 
\ar[rr]^{u}
\ar[dr]_{p}
&&\mathdeuxcat{B}
\ar[dl]^{q}
\\
&\mathdeuxcat{C}
}
$$
désigne un triangle commutatif dans $\DeuxCat$, si $p$ et $q$ sont des précofibrations et si, pour tout objet $c$ de $\mathdeuxcat{C}$, le \DeuxFoncteurStrict{} induit entre les fibres
$$
\Fibre{u}{}{c} : \Fibre{\mathdeuxcat{A}}{p}{c} \to \Fibre{\mathdeuxcat{B}}{q}{c} 
$$
est dans $\DeuxLocFond{W}$, alors $u$ est dans $\DeuxLocFond{W}$.
\end{itemize}
\item[(vii)]
Les conditions suivantes sont vérifiées.
\begin{itemize}
\item[LF$\alpha''$] 
La partie $\DeuxLocFond{W}$ de $\UnCell{\DeuxCat}$ est faiblement saturée. 
\item[LF$\beta''$] 
Le morphisme canonique $[1] \to \DeuxCatPonct$ est dans $\DeuxLocFond{W}$. 
\item[LF$\gamma''$] Si 
$$
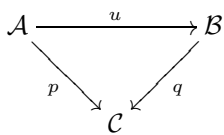
\xymatrix{
\mathdeuxcat{A} 
\ar[rr]^{u}
\ar[dr]_{p}
&&\mathdeuxcat{B}
\ar[dl]^{q}
\\
&\mathdeuxcat{C}
}
$$
désigne un triangle commutatif dans $\DeuxCat$, si $p$ et $q$ sont des préopfibrations et si, pour tout objet $c$ de $\mathdeuxcat{C}$, le \DeuxFoncteurStrict{} induit entre les fibres
$$
\Fibre{u}{}{c} : \Fibre{\mathdeuxcat{A}}{p}{c} \to \Fibre{\mathdeuxcat{B}}{q}{c} 
$$
est dans $\DeuxLocFond{W}$, alors $u$ est dans $\DeuxLocFond{W}$.
\end{itemize}
\item[(viii)]
Les conditions suivantes sont vérifiées.
\begin{itemize}
\item[LF$\alpha'''$] 
La partie $\DeuxLocFond{W}$ de $\UnCell{\DeuxCat}$ est faiblement saturée. 
\item[LF$\beta'''$] 
Le morphisme canonique $[1] \to \DeuxCatPonct$ est dans $\DeuxLocFond{W}$. 
\item[LF$\gamma'''$] Si 
$$
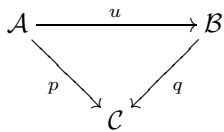
\xymatrix{
\mathdeuxcat{A} 
\ar[rr]^{u}
\ar[dr]_{p}
&&\mathdeuxcat{B}
\ar[dl]^{q}
\\
&\mathdeuxcat{C}
}
$$
désigne un triangle commutatif dans $\DeuxCat$, si $p$ et $q$ sont des préfibrations et si, pour tout objet $c$ de $\mathdeuxcat{C}$, le \DeuxFoncteurStrict{} induit entre les fibres
$$
\Fibre{u}{}{c} : \Fibre{\mathdeuxcat{A}}{p}{c} \to \Fibre{\mathdeuxcat{B}}{q}{c} 
$$
est dans $\DeuxLocFond{W}$, alors $u$ est dans $\DeuxLocFond{W}$.
\end{itemize}
\end{itemize}
\end{theo}

\begin{proof}
L'implication $(i) \Rightarrow (v)$ résulte de la proposition \ref{2.1.10.THG} et de l'asphéricité de la catégorie $[1]$. 

Montrons l'implication $(v) \Rightarrow (iii)$. 

Soit $\mathdeuxcat{A}$ une petite \deux{}catégorie. Les flèches du triangle commutatif
$$
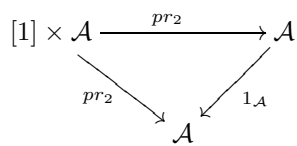
\xymatrix{
[1] \times \mathdeuxcat{A} 
\ar[rr]^{pr_{2}}
\ar[dr]_{pr_{2}}
&&\mathdeuxcat{A}
\ar[dl]^{1_{\mathdeuxcat{A}}}
\\
&\mathdeuxcat{A}
}
$$
sont des précoopfibrations\footnote{Voir la remarque \ref{ProjectionPrefibrationPreuve}.} et, pour tout objet $a$ de $\mathdeuxcat{A}$, le \DeuxFoncteurStrict{} induit entre les fibres au-dessus de $a$ s'identifie au morphisme canonique $[1] \to \DeuxCatPonct$, qui est dans $\DeuxLocFond{W}$ par la condition LF$\beta$. Ainsi, pour toute petite \deux{}catégorie $\mathdeuxcat{A}$, la projection canonique $[1] \times \mathdeuxcat{A} \to \mathdeuxcat{A}$ est dans $\DeuxLocFond{W}$ par la condition LF$\gamma$. 

Supposons que $\mathdeuxcat{A}$ op-admette un objet admettant un objet final. En vertu d'un argument similaire à celui que nous avons présenté dans la remarque \ref{OFContractile}, $\mathdeuxcat{A}$ admet un endomorphisme constant homotope à $1_{\mathdeuxcat{A}}$ ; elle est donc contractile. En vertu de ce qui précède et de la proposition \ref{1.4.8.THG}, le morphisme canonique $\mathdeuxcat{A} \to \DeuxCatPonct$ est dans $\DeuxLocFond{W}$. La condition LF$2''$ est donc vérifiée.

Pour tout \DeuxFoncteurStrict{} $u : \mathdeuxcat{A} \to \mathdeuxcat{B}$, on peut considérer les \DeuxFoncteursStricts{}
$$
\begin{aligned}
\DeuxCatUnOp{\mathdeuxcat{B}} &\to \DeuxCatDeuxCat
\\
b &\mapsto \OpTrancheCoLax{\mathdeuxcat{A}}{u}{b}
\end{aligned}
$$
et 
$$
\begin{aligned}
\mathdeuxcat{A} &\to \DeuxCatDeuxCat
\\
a &\mapsto \DeuxCatUnOp{\left(\TrancheCoLax{\mathdeuxcat{B}}{}{u(a)}\right)}.
\end{aligned}
$$

Les formules générales permettent d'expliciter la structure de la \deux{}catégorie $\DeuxInt{\DeuxCatUnOp{\mathdeuxcat{B}}} (\OpTrancheCoLax{\mathdeuxcat{A}}{u}{b})$ comme suit. 

Ses objets sont les
$$
(b, (a, k : b \to u(a))),
$$
$b$ étant un objet de $\mathdeuxcat{B}$, $a$ un objet de $\mathdeuxcat{A}$ et $k$ une \un{}cellule de $\mathdeuxcat{B}$.

Les \un{}cellules de $(b, (a,k))$ vers $(b', (a',k'))$ sont les
$$
(f : b' \to b, (g : a \to a', \alpha : u(g)kf \Rightarrow k')),
$$
$f$ étant une \un{}cellule de $\mathdeuxcat{B}$, $g$ une \un{}cellule de $\mathdeuxcat{A}$ et $\alpha$ une \deux{}cellule de $\mathdeuxcat{B}$. 

Les \deux{}cellules de $(f, (g, \alpha))$ vers $(f', (g', \alpha'))$ sont les
$$
(\varphi : f \Rightarrow f', \gamma : g \Rightarrow g')
$$
vérifiant 
$$
\alpha' \CompDeuxUn (u(\gamma) \CompDeuxZero k \CompDeuxZero \varphi) = \alpha,
$$
$\varphi$ étant une \deux{}cellule de $\mathdeuxcat{B}$ et $\gamma$ une \deux{}cellule de $\mathdeuxcat{A}$. 

Les diverses unités et compositions sont définies de façon « évidente ». 

Les formules générales permettent d'expliciter la structure de la \deux{}catégorie $\DeuxInt{\mathdeuxcat{A}} \DeuxCatUnOp{(\TrancheCoLax{\mathdeuxcat{B}}{}{u(a)})}$ comme suit. 

Ses objets sont les
$$
(a, (b, k : b \to u(a))),
$$
$a$ étant un objet de $\mathdeuxcat{A}$, $b$ un objet de $\mathdeuxcat{B}$ et $k$ une \un{}cellule de $\mathdeuxcat{B}$.

Les \un{}cellules de $(a, (b,k))$ vers $(a', (b',k'))$ sont les
$$
(g : a \to a', (f : b' \to b, \alpha : u(g)kf \Rightarrow k')),
$$
$g$ étant une \un{}cellule de $\mathdeuxcat{A}$, $f$ une \un{}cellule de $\mathdeuxcat{B}$ et $\alpha$ une \deux{}cellule de $\mathdeuxcat{B}$. 

Les \deux{}cellules de $(g, (f, \alpha))$ vers $(g', (f', \alpha'))$ sont les
$$
(\gamma : g \Rightarrow g', \varphi : f \Rightarrow f')
$$
vérifiant 
$$
\alpha' \CompDeuxUn (u(\gamma) \CompDeuxZero k \CompDeuxZero \varphi) = \alpha,
$$
$\gamma$ étant une \deux{}cellule de $\mathdeuxcat{A}$ et $\varphi$ une \deux{}cellule de $\mathdeuxcat{B}$. 

Les diverses unités et compositions sont définies de façon « évidente ». 

Considérons la \deux{}catégorie $S(u)$\index[not]{Su@$S(u)$} définie comme suit. Ses objets sont les triplets $(b, a , k : b \to u(a))$ avec $b$ un objet de $\mathdeuxcat{B}$, $a$ un objet de $\mathdeuxcat{A}$ et $k$ une \un{}cellule de $\mathdeuxcat{B}$. Les \un{}cellules de $(b,a,k)$ vers $(b',a',k')$ sont les triplets $(f : b' \to b, g : a \to a', \alpha : u(g) k f \Rightarrow k')$, $f$ étant une \un{}cellule de $\mathdeuxcat{B}$, $g$ une \un{}cellule de $\mathdeuxcat{A}$ et $\alpha$ une \deux{}cellule de $\mathdeuxcat{B}$. Les \deux{}cellules de $(f,g,\alpha)$ vers $(f',g',\alpha')$ sont les couples $(\varphi : f \Rightarrow f', \gamma : g \Rightarrow g')$ avec $\varphi$ une \deux{}cellule de $\mathdeuxcat{B}$ et $\gamma$ une \deux{}cellule de $\mathdeuxcat{A}$ telles que $\alpha' \CompDeuxUn (u(\gamma) \CompDeuxZero k \CompDeuxZero \varphi) = \alpha$. 

Il existe alors des isomorphismes canoniques
$$
\begin{aligned}
S(u) &\to \DeuxInt{\DeuxCatUnOp{\mathdeuxcat{B}}} \OpTrancheCoLax{\mathdeuxcat{A}}{u}{b}
\\
(b, a, k) &\mapsto (b, (a,k))
\\
(f,g,\alpha) &\mapsto (f,(g,\alpha))
\\
(\varphi, \gamma) &\mapsto (\varphi, \gamma)
\end{aligned}
$$
et
$$
\begin{aligned}
\DeuxInt{\DeuxCatUnOp{\mathdeuxcat{B}}} \OpTrancheCoLax{\mathdeuxcat{A}}{u}{b} &\to S(u) 
\\
(b, (a,k)) &\mapsto (b, a, k) 
\\
(f,(g,\alpha)) &\mapsto (f,g,\alpha) 
\\
(\varphi, \gamma) &\mapsto (\varphi, \gamma)
\end{aligned}
$$
inverses l'un de l'autre, ainsi que des isomorphismes canoniques 
$$
\begin{aligned}
S(u) &\to \DeuxInt{\mathdeuxcat{A}} \DeuxCatUnOp{\left(\TrancheCoLax{\mathdeuxcat{B}}{}{u(a)}\right)}
\\
(b, a, k) &\mapsto (a, (b,k))
\\
(f,g,\alpha) &\mapsto (g,(f,\alpha))
\\
(\varphi, \gamma) &\mapsto (\gamma, \varphi)
\end{aligned}
$$
et 
$$
\begin{aligned}
\DeuxInt{\mathdeuxcat{A}} \DeuxCatUnOp{\left(\TrancheCoLax{\mathdeuxcat{B}}{}{u(a)}\right)} &\to S(u)
\\
(a, (b,k)) &\mapsto (b, a, k)
\\
(g,(f,\alpha)) &\mapsto (f,g,\alpha)
\\
(\gamma, \varphi) &\mapsto (\varphi, \gamma)
\end{aligned}
$$
inverses l'un de l'autre. 

Les projections canoniques
$$
\begin{aligned}
s_{u}\index[not]{su@$s_{u}$} : S(u) &\to \DeuxCatUnOp{\mathdeuxcat{B}}
\\
(b,a,k) &\mapsto b
\\
(f,g,\alpha) &\mapsto f
\\
(\varphi, \gamma) &\mapsto \varphi
\end{aligned}
$$
et
$$
\begin{aligned}
t_{u}\index[not]{tu@$t_{u}$} : S(u) &\to \mathdeuxcat{A}
\\
(b,a,k) &\mapsto a
\\
(f,g,\alpha) &\mapsto g
\\
(\varphi, \gamma) &\mapsto \gamma
\end{aligned}
$$
sont donc des précoopfibrations. De plus, pour tout objet $a$ de $\mathdeuxcat{A}$, la \deux{}catégorie $\DeuxCatUnOp{\left(\TrancheCoLax{\mathdeuxcat{B}}{}{u(a)}\right)}$ op-admet un objet admettant un objet final. En vertu de ce qui précède, le morphisme canonique $\DeuxCatUnOp{\left(\TrancheCoLax{\mathdeuxcat{B}}{}{u(a)}\right)} \to \DeuxCatPonct$ est donc dans $\DeuxLocFond{W}$. Or, c'est à ce morphisme canonique que s'identifie le \DeuxFoncteurStrict{} induit entre les fibres au-dessus de $a$ par le diagramme
$$
\xymatrix{
S(u)
\ar[rr]^{t_{u}}
\ar[dr]_{t_{u}}
&&\mathdeuxcat{A}
\ar[dl]^{1_{\mathdeuxcat{A}}}
\\
&\mathdeuxcat{A}
&.
}
$$
Comme $t_{u}$ et $1_{\mathdeuxcat{A}}$ sont des précoopfibrations, $t_{u}$ est dans $\DeuxLocFond{W}$ en vertu de la condition $LF\gamma$. 

Soit 
$$
\xymatrix{
\mathdeuxcat{A}
\ar[r]^{u}
\ar[d]_{v}
&
\mathdeuxcat{B}
\ar[d]^{w}
\\
\mathdeuxcat{A'}
\ar[r]_{u'}
&
\mathdeuxcat{B'}
}
$$
un carré commutatif dans $\DeuxCat$. On définit un \DeuxFoncteurStrict{}
$$
\begin{aligned}
S(v,w)\index[not]{Svw@$S(v,w)$} : S(u) &\to S(u')
\\
(b,a,k) &\mapsto (w(b), v(a), w(k))
\\
(f, g, \alpha) &\mapsto (w(f), v(g), w(\alpha))
\\
(\varphi, \gamma) &\mapsto (w(\varphi), v(\gamma)).
\end{aligned}
$$
Cela fournit un diagramme commutatif
$$
\xymatrix{
\DeuxCatUnOp{\mathdeuxcat{B}}
\ar[d]_{\DeuxFoncUnOp{w}}
&
S(u)
\ar[l]_{s_{u}}
\ar[r]^{t_{u}}
\ar[d]^{S(v,w)}
&
\mathdeuxcat{A}
\ar[d]^{v}
\\
\DeuxCatUnOp{\mathdeuxcat{B'}}
&
S(u')
\ar[l]^{s_{u'}}
\ar[r]_{t_{u'}}
&
\mathdeuxcat{A'}
}
$$
dans lequel $t_{u}$ et $t_{u'}$ sont dans $\DeuxLocFond{W}$ en vertu de ce qui précède. 

Soit
$$
\xymatrix{
\mathdeuxcat{A}
\ar[rr]^{u}
\ar[dr]_{w}
&&\mathdeuxcat{B}
\ar[dl]^{v}
\\
&\mathdeuxcat{C}
}
$$
un diagramme commutatif dans $\DeuxCat$ tel que, pour tout objet $c$ de $\mathdeuxcat{C}$, le \DeuxFoncteurStrict{} \mbox{$\DeuxFoncOpTrancheCoLax{u}{c} : \OpTrancheCoLax{\mathdeuxcat{A}}{w}{c} \to \OpTrancheCoLax{\mathdeuxcat{B}}{v}{c}$} soit dans $\DeuxLocFond{W}$. En vertu de ce qui précède, cela fournit un diagramme commutatif 
$$
\xymatrix{
\DeuxCatUnOp{\mathdeuxcat{C}}
\ar[d]_{1_{\DeuxCatUnOp{\mathdeuxcat{C}}}}
&
S(w)
\ar[l]_{s_{w}}
\ar[r]^{t_{w}}
\ar[d]^{S(u,1_{\mathdeuxcat{C}})}
&
\mathdeuxcat{A}
\ar[d]^{u}
\\
\DeuxCatUnOp{\mathdeuxcat{C}}
&
S(v)
\ar[l]^{s_{v}}
\ar[r]_{t_{v}}
&
\mathdeuxcat{B}
&.
}
$$
En particulier, on a un diagramme commutatif
$$
\xymatrix{
S(w)
\ar[rr]^{S(u, 1_{\mathdeuxcat{C}})}
\ar[dr]_{s_{w}}
&&
S(v)
\ar[dl]^{s_{v}}
\\
&
\DeuxCatUnOp{\mathdeuxcat{C}}
}
$$
dans lequel $s_{w}$ et $s_{v}$ sont des précoopfibrations. Pour tout objet $c$ de $\mathdeuxcat{C}$, le \DeuxFoncteurStrict{} induit par ce diagramme entre les fibres au-dessus de $c$ s'identifie à $\DeuxFoncOpTrancheCoLax{u}{c}$, qui est dans $\DeuxLocFond{W}$ par hypothèse. Par conséquent, $S(u, 1_{\mathdeuxcat{C}})$ est dans $\DeuxLocFond{W}$ en vertu de la condition LF$\gamma$. Comme $t_{w}$ et $t_{v}$ sont dans $\DeuxLocFond{W}$, qui est faiblement saturée, $u$ est dans $\DeuxLocFond{W}$, ce qui termine la démonstration de la condition LF$3''$ et donc celle de l'implication $(v) \Rightarrow (iii)$.

Montrons l'implication $(iii) \Rightarrow (i)$. Notons $\DeuxFoncUnOp{\DeuxLocFond{W}}$\index[not]{Wop@$\DeuxFoncUnOp{\DeuxLocFond{W}}$} la partie de $\UnCell{\DeuxCat}$ définie par
$$
\DeuxFoncUnOp{\DeuxLocFond{W}} = \{ u \in \UnCell{\DeuxCat} \vert \DeuxFoncUnOp{u} \in \DeuxLocFond{W} \}.
$$
Vérifions que la classe $\DeuxFoncUnOp{\DeuxLocFond{W}}$ constitue un \ClasseDeuxLocFond{}. La saturation faible de $\DeuxFoncUnOp{\DeuxLocFond{W}}$ est immédiate. Soit $\mathdeuxcat{A}$ une petite \deux{}catégorie admettant un objet admettant un objet final. Alors, $\DeuxCatUnOp{\mathdeuxcat{A}}$ op-admet un objet admettant un objet final, donc la flèche $\DeuxCatUnOp{\mathdeuxcat{A}} \to \DeuxCatPonct$ est dans $\DeuxLocFond{W}$, c'est-à-dire que la flèche $\mathdeuxcat{A} \to \DeuxCatPonct$ est dans $\DeuxFoncUnOp{\DeuxLocFond{W}}$. Soit 
$$
\xymatrix{
\mathdeuxcat{A}
\ar[rr]^{u}
\ar[dr]_{w}
&&\mathdeuxcat{B}
\ar[dl]^{v}
\\
&\mathdeuxcat{C}
}
$$
un diagramme commutatif dans $\DeuxCat$ tel que, pour tout objet $c$ de $\mathdeuxcat{C}$, le \DeuxFoncteurStrict{} $\DeuxFoncTrancheCoLax{u}{c}$ soit dans $\DeuxFoncUnOp{\DeuxLocFond{W}}$, c'est-à-dire tel que le \DeuxFoncteurStrict{} $\DeuxFoncUnOp{(\DeuxFoncTrancheCoLax{u}{c})}$ soit dans $\DeuxLocFond{W}$, c'est-à-dire tel que le \DeuxFoncteurStrict{} $\DeuxFoncUnOp{(\DeuxFoncTrancheCoLax{\DeuxFoncUnOp{(\DeuxFoncUnOp{u})}}{c})}$ soit dans $\DeuxLocFond{W}$. Pour tout objet $c$ de $\mathdeuxcat{C}$, le \DeuxFoncteurStrict{} $\DeuxFoncOpTrancheCoLax{\DeuxFoncUnOp{u}}{c}$ est donc dans $\DeuxLocFond{W}$. La condition LF3$'$ implique que $\DeuxFoncUnOp{u}$ est dans $\DeuxLocFond{W}$, c'est-à-dire que $u$ est dans $\DeuxFoncUnOp{\DeuxLocFond{W}}$, qui est donc bien un \ClasseDeuxLocFond{}. Par conséquent, en vertu de la proposition \ref{DeuxFoncUnOpW}, $\DeuxFoncUnOp{\DeuxLocFond{W}} = \DeuxLocFond{W}$. La classe $\DeuxLocFond{W}$ est donc un localisateur fondamental de $\DeuxCat$, ce qui démontre l'implication $(iii) \Rightarrow (i)$. 

Les autres implications se démontrent de façon analogue ou se déduisent de ce qui précède par des arguments de dualité.   

\end{proof}

\begin{rem}
Le lecteur à la page n'aura pas manqué de remarquer que le théorème \ref{DeuxLocFondHuitDef} s'inspire de \cite[théorème 2.1.11]{THG}. Pour en obtenir un exact analogue, toutefois, il aurait fallu disposer de notions de fibration et de morphisme cartésien dans $\DeuxCat$. Le développement actuel de la théorie des fibrations dans $\DeuxCat$ ne permet malheureusement pas, comme nous l'avons déjà remarqué (remarque \ref{RemarquesPrefibration}), de considérer ces notions comme définitivement dégagées. Lorsqu'elles seront fermement établies, il devrait être facile de renforcer le théorème \ref{DeuxLocFondHuitDef} ci-dessus pour en faire un analogue \deux{}catégorique plus précis de \cite[théorème 2.1.11]{THG}.  
\end{rem}

\begin{paragr}\label{IntegrationCartesien}
Soient $\mathdeuxcat{A}$ et $\mathdeuxcat{B}$ des \deux{}catégories et $u : \mathdeuxcat{A} \to \mathdeuxcat{B}$ et $\mathdeuxcat{B} \to \DeuxCatDeuxCat$ des \DeuxFoncteursStricts{}. Ces données permettent de définir un \DeuxFoncteurStrict{}
$$
\begin{aligned}
U : \DeuxInt{\mathdeuxcat{A}} Fu &\to \DeuxInt{\mathdeuxcat{B}} F
\\
(a,x) &\mapsto (u(a), x)
\\
(f,r) &\mapsto (u(f), r)
\\
(\gamma, \varphi) &\mapsto (u(\gamma), \varphi).
\end{aligned}
$$
Il résulte d'une simple vérification que le carré 
$$
\xymatrix{
\DeuxInt{\mathdeuxcat{A}} Fu
\ar[r]^{U}
\ar[d]_{P_{Fu}}
&
\DeuxInt{\mathdeuxcat{B}} F
\ar[d]^{P_{F}}
\\
\mathdeuxcat{A}
\ar[r]_{u}
&
\mathdeuxcat{B}
}
$$
est cartésien, les flèches verticales désignant les projections canoniques. 
\end{paragr}

\begin{lemme}\label{CarreTrancheCartesien}
Pour tout \DeuxFoncteurStrict{} $u : \mathdeuxcat{A} \to \mathdeuxcat{B}$, pour tout objet $b$ de $\mathdeuxcat{B}$, les carrés
$$
\xymatrix{
\TrancheLax{\mathdeuxcat{A}}{u}{b}
\ar[r]
\ar[d]_{\DeuxFoncTrancheLax{u}{b}}
&
\mathdeuxcat{A}
\ar[d]^{u}
&
\OpTrancheLax{\mathdeuxcat{A}}{u}{b}
\ar[r]
\ar[d]_{\DeuxFoncOpTrancheLax{u}{b}}
&
\mathdeuxcat{A}
\ar[d]^{u}
&
\TrancheCoLax{\mathdeuxcat{A}}{u}{b}
\ar[r]
\ar[d]_{\DeuxFoncTrancheCoLax{u}{b}}
&
\mathdeuxcat{A}
\ar[d]^{u}
&
\OpTrancheCoLax{\mathdeuxcat{A}}{u}{b}
\ar[r]
\ar[d]_{\DeuxFoncOpTrancheCoLax{u}{b}}
&
\mathdeuxcat{A}
\ar[d]^{u}
\\
\TrancheLax{\mathdeuxcat{B}}{}{b}
\ar[r]
&
\mathdeuxcat{B}
&
\OpTrancheLax{\mathdeuxcat{B}}{}{b}
\ar[r]
&
\mathdeuxcat{B}
&
\TrancheCoLax{\mathdeuxcat{B}}{}{b}
\ar[r]
&
\mathdeuxcat{B}
&
\OpTrancheCoLax{\mathdeuxcat{B}}{}{b}
\ar[r]
&
\mathdeuxcat{B}
&,
}
$$
dans lesquels les flèches horizontales désignent les projections canoniques, sont cartésiens dans $\DeuxCat$.
\end{lemme}

\begin{proof}
Les quatre assertions, duales l'une de l'autre, sont conséquences du caractère cartésien du diagramme figurant dans le paragraphe \ref{IntegrationCartesien}, puisque les diverses \deux{}catégories tranches et \deux{}foncteurs tranches ne sont que des cas particuliers d'intégrales et de \deux{}foncteurs induits suivant le procédé décrit dans le paragraphe \ref{IntegrationCartesien}. 
\end{proof}

\emph{On suppose fixé un localisateur fondamental $\DeuxLocFond{W}$ de $\DeuxCat$.}

\begin{prop}\label{UniverselColaxAspherique}
Un morphisme de $\DeuxCat$ universellement dans $\DeuxLocFond{W}$ est lax-asphérique, lax-opasphérique, colax-asphérique et colax-opasphérique. 
\end{prop}

\begin{proof}
C'est une conséquence immédiate du lemme \ref{CarreTrancheCartesien}.
\end{proof}

\begin{prop}\label{1.1.4.THG}
Soit $\mathdeuxcat{A}$ une petite \deux{}catégorie. Les conditions suivantes sont équivalentes.
\begin{itemize}
\item[(a)] Le morphisme canonique $\mathdeuxcat{A} \to e$ est une équivalence faible (autrement dit, $\mathdeuxcat{A}$ est asphé\-rique).
\item[(b)] Le morphisme canonique $\mathdeuxcat{A} \to e$ est lax-asphérique.
\item[(b$'$)] Le morphisme canonique $\mathdeuxcat{A} \to e$ est lax-opasphérique.
\item[(b$''$)] Le morphisme canonique $\mathdeuxcat{A} \to e$ est colax-asphérique.
\item[(b$'''$)] Le morphisme canonique $\mathdeuxcat{A} \to e$ est colax-opasphérique.
\item[(c)] Le morphisme canonique $\mathdeuxcat{A} \to e$ est universellement dans $\DeuxLocFond{W}$.
\end{itemize}
\end{prop}

\begin{proof}
L'équivalence de $(a)$, $(b)$, $(b')$, $(b'')$ et $(b''')$ est immédiate. L'implication \mbox{$(c) \Rightarrow (b)$} résulte (tout comme ses trois variantes duales) de la proposition \ref{UniverselColaxAspherique}. Montrons que $(a)$ implique $(c)$. Il s'agit de montrer que, pour toute petite \deux{}catégorie $\mathdeuxcat{B}{}$, la projection canonique \mbox{$p_{2} : \mathdeuxcat{A} \times \mathdeuxcat{B}{} \to \mathdeuxcat{B}{}$} est une équivalence faible. Il suffit pour cela de montrer que, pour tout objet $b$ de $\mathdeuxcat{B}{}$, la \deux{}catégorie $\TrancheCoLax{(\mathdeuxcat{A} \times \mathdeuxcat{B})}{p_{2}}{b}$ est asphérique. Or, on a un isomorphisme canonique de \deux{}catégories $\TrancheCoLax{(\mathdeuxcat{A} \times \mathdeuxcat{B})}{p_{2}}{b} \simeq \mathdeuxcat{A} \times (\TrancheCoLax{\mathdeuxcat{B}}{}{b})$ et, comme $\mathdeuxcat{A}$ est asphérique par hypothèse, il suffit de montrer que la première projection $q_{1} : \mathdeuxcat{A} \times (\TrancheCoLax{\mathdeuxcat{B}}{}{b}) \to \mathdeuxcat{A}$ est une équivalence faible. Il suffit pour cela de montrer que, pour tout objet $a$ de $\mathdeuxcat{A}$, la \deux{}catégorie $\TrancheCoLax{(\mathdeuxcat{A} \times (\TrancheCoLax{\mathdeuxcat{B}}{}{b}))}{q_{1}}{a}$ est asphérique. Comme on a un isomorphisme canonique $\TrancheCoLax{(\mathdeuxcat{A} \times (\TrancheCoLax{\mathdeuxcat{B}}{}{b}))}{q_{1}}{a}  \simeq (\TrancheCoLax{\mathdeuxcat{A}}{}{a})  \times (\TrancheCoLax{\mathdeuxcat{B}}{}{b})$, cela résulte de la condition LF2, puisque la \deux{}catégorie $(\TrancheCoLax{\mathdeuxcat{A}}{}{a})  \times (\TrancheCoLax{\mathdeuxcat{B}}{}{b})$ admet un objet admettant un objet final. 
\end{proof}

\begin{corollaire}\label{ProduitAspheriques}
Un produit fini de petites \deux{}catégories asphériques est as\-phé\-rique.
\end{corollaire}

\begin{proof}
Soient $\mathdeuxcat{A}$ et $\mathdeuxcat{B}{}$ des petites \deux{}catégories asphériques. En vertu de la proposition \ref{1.1.4.THG}, la projection canonique $\mathdeuxcat{A} \times \mathdeuxcat{B}{} \to \mathdeuxcat{B}{}$ est une équivalence faible. Par conséquent, en vertu d'un argument de $2$ sur $3$, le morphisme canonique $\mathdeuxcat{A} \times \mathdeuxcat{B}{} \to e$ est une équivalence faible, ce qui prouve le résultat. 
\end{proof}

\begin{corollaire}\label{ProduitMorphismesColaxAspheriques}
Un produit fini de \DeuxFoncteursStricts{} lax-asphériques (\emph{resp.} lax-op\-asphé\-riques, \emph{resp.} colax-asphériques, \emph{resp.} colax-opasphériques) est lax-asphé\-rique (\emph{resp.} lax-op\-asphé\-rique, \emph{resp.} colax-asphérique, \emph{resp.} co\-lax-\-op\-asphé\-rique).
\end{corollaire}

\begin{proof}
Les quatre cas s'obtiennent évidemment les uns des autres par un argument de dualité. Il suffit de montrer que, si $u : \mathdeuxcat{A} \to \mathdeuxcat{B}{}$ et $u' : \mathcal{A'} \to \mathcal{B'}$ sont des \DeuxFoncteursStricts{} lax-asphé\-riques, alors leur produit $u \times u' : \mathdeuxcat{A} \times \mathcal{A'} \to \mathdeuxcat{B}{} \times \mathcal{B'}$ est lax-asphérique. Pour cela, on remarque que, pour tout objet $(b,b')$ de $\mathdeuxcat{B}{} \times \mathcal{B'}$, la \deux{}catégorie $\TrancheLax{(\mathdeuxcat{A} \times \mathcal{A'})}{u \times u'}{(b,b')}$ est canoniquement isomorphe à $\TrancheLax{\mathdeuxcat{A}}{u}{b} \times \TrancheLax{\mathdeuxcat{A'}}{u'}{b'}$ et, par hypothèse, $\TrancheLax{\mathdeuxcat{A}}{u}{b}$ et $\TrancheLax{\mathdeuxcat{A'}}{u'}{b'}$ sont asphériques. L'assertion résulte donc du corollaire \ref{ProduitAspheriques}.
\end{proof}

\begin{prop}\label{Proposition2.1.3THG}
Un produit fini d'équivalences faibles est une équivalence faible.
\end{prop}

\begin{proof}
La démonstration est analogue à celle de \cite[proposition 2.1.3]{THG}, que nous reprenons mot pour mot. Il suffit de montrer que le produit de deux équivalences faibles $u_{1} : \mathdeuxcat{A}_{1} \to \mathdeuxcat{B}_{1}$ et $u_{2} : \mathdeuxcat{A}_{2} \to \mathdeuxcat{B}_{2}$ est une équivalence faible. Comme $u_{1} \times u_{2} = (1_{\mathdeuxcat{B}_{1}} \times u_{2}) (u_{1} \times 1_{\mathdeuxcat{A}_{2}})$ ,
il suffit de montrer que, pour toute équivalence faible $u : \mathdeuxcat{A} \to \mathdeuxcat{B}$ et toute petite \deux{}catégorie $\mathdeuxcat{C}$, le morphisme $u \times 1_{\mathdeuxcat{C}}$ est une équivalence faible. On a un triangle commutatif
$$
\xymatrix{
\mathdeuxcat{A} \times \mathdeuxcat{C}
\ar[rr]^{u \times 1_{\mathdeuxcat{C}}}
\ar[dr]_{p}
&&
\mathdeuxcat{B} \times \mathdeuxcat{C}
\ar[dl]^{q}
\\
&
\mathdeuxcat{C}
}
$$
dans lequel $p$ et $q$ désignent les projections canoniques. En vertu de la condition LF3, il suffit de montrer que, pour tout objet $c$ de $\mathdeuxcat{C}$, le morphisme
$$
\xymatrix{
\TrancheCoLax{(\mathdeuxcat{A} \times \mathdeuxcat{C})}{p}{c} \simeq \mathdeuxcat{A} \times (\TrancheCoLax{\mathdeuxcat{C}}{}{c}) 
\ar[rrrr]^{\DeuxFoncTrancheCoLax{(u \times 1_{\mathdeuxcat{C}})}{}{c} \phantom{a} \simeq \phantom{a} u \times (1_{\TrancheCoLax{\mathdeuxcat{C}}{}{c}})}
&&&&
\TrancheCoLax{\mathdeuxcat{B} \times (\mathdeuxcat{C}}{}{c}) \simeq \TrancheCoLax{(\mathdeuxcat{B} \times \mathdeuxcat{C})}{q}{c}
}
$$
est une équivalence faible. Or, on a un carré commutatif
$$
\xymatrix{
\mathdeuxcat{A} \times (\TrancheCoLax{\mathdeuxcat{C}}{}{c})
\ar[rr]^{u \times (1_{\TrancheCoLax{\mathdeuxcat{C}}{}{c}})}
\ar[d]
&&
\mathdeuxcat{B} \times (\TrancheCoLax{\mathdeuxcat{C}}{}{c})
\ar[d]
\\
\mathdeuxcat{A}
\ar[rr]_{u}
&&
\mathdeuxcat{B}
}
$$
dans lequel les flèches verticales désignent les projections canoniques, qui sont des équivalences faibles en vertu de la proposition \ref{1.1.4.THG} puisque $\TrancheCoLax{\mathdeuxcat{C}}{}{c}$ est asphérique. Comme $u$ est une équivalence faible, il en est de même de $u \times (1_{\TrancheCoLax{\mathdeuxcat{C}}{}{c}})$, ce qui achève la démonstration.
\end{proof}

\begin{prop}\label{Proposition2.1.4THG}
Une petite somme arbitraire d'équivalences faibles est une équivalence faible.
\end{prop}

\begin{proof}
La démonstration est analogue à celle de \cite[proposition 2.1.4]{THG}, que nous reprenons mot pour mot. Soit $u_{i} : \mathdeuxcat{A}_{i} \to \mathdeuxcat{B}_{i}$, $i \in I$, une petite famille d'équivalences faibles. Notons aussi $I$ la catégorie discrète correspondant à l’ensemble $I$. On a un triangle commutatif
$$
\xymatrix{
\mathdeuxcat{A} = \coprod \mathdeuxcat{A}_{i}
\ar[rr]^{u = \coprod u_{i}}
\ar[dr]_{\pi_{\mathdeuxcat{A}}}
&&
\mathdeuxcat{B} = \coprod \mathdeuxcat{B}_{i}
\ar[dl]^{\pi_{\mathdeuxcat{B}}}
\\
&
I
}
$$
et l'on remarque que, pour tout objet $i$ de $I$, le morphisme $\DeuxFoncTrancheCoLax{u}{i} : \TrancheCoLax{\mathdeuxcat{A}}{\pi_{\mathdeuxcat{A}}}{i} \to \TrancheCoLax{\mathdeuxcat{B}}{\pi_{\mathdeuxcat{B}}}{i}$ s’identifie à $u_{i} : \mathdeuxcat{A}_{i} \to \mathdeuxcat{B}_{i}$. En vertu de la condition LF3, $u$ est donc une équivalence faible, ce qui démontre la proposition.
\end{proof}

\begin{lemme}\label{Lemme1.3.6.THG}
Soient $M$ une catégorie, $W$ une classe de flèches de $M$ et $\gamma_{M} : M \to \Localisation{M}{W}$ le foncteur de localisation. Si $e_{M}$ est un objet final de $M$, alors son image par le foncteur $\gamma_{M}$ est un objet final de $\Localisation{M}{W}$.
\end{lemme}

\begin{proof}
C'est \cite[lemme 1.3.6]{THG}, dont nous donnons une démonstration. 

Par commodité, l'on notera de la même façon les objets de $M$ et leur image par le foncteur $\gamma_{M}$. Soit $x$ un objet quelconque de $M$. Notons $p_{x}$ l'unique morphisme de $x$ vers $e_{M}$ dans $M$. Considérons un morphisme quelconque de $x$ vers $e_{M}$ dans $\Localisation{M}{W}$, c'est-à-dire une classe d'équivalence de chaînes constituées de morphismes de $M$ et d'inverses formels d'éléments de $W$ modulo la relation d'équivalence bien connue, que nous noterons $\sim$. Soit $f$ un représentant de cette classe d'équivalence.  

S'il est de longueur $1$, soit c'est $p_{x}$, auquel cas il n'y a rien à faire, soit c'est un inverse formel $w^{-1} : x \to e_{M}$ d'un élément $w : e_{M} \to x$ de $W$. Dans ce cas, comme $p_{x} w = 1_{e_{M}}$, on a les équivalences $w^{-1} \sim 1_{e_{M}} w^{-1} \sim p_{x} w w^{-1} \sim p_{x}$. 

Supposons qu'il soit de longueur $2$. En vertu de ce qui précède, on peut se ramener aux deux seuls cas suivants. S'il s'écrit 
$$
\xymatrix{
x
\ar[r]^{g}
&
x'
\ar[r]^{p_{x'}}
&
e_{M}
&,
}
$$
où $p_{x'}$ est l'unique morphisme de $x'$ vers $e_{M}$ dans $M$ et $g$ est un morphisme de $M$, alors $p_{x'} g = p_{x}$, donc $f \sim p_{x}$. S'il s'écrit 
$$
\xymatrix{
x
\ar[r]^{w^{-1}}
&
x'
\ar[r]^{p_{x'}}
&
e_{M}
&,
}
$$
avec $w : x' \to x$ un élément de $W$, alors $p_{x} w = p_{x'}$, donc $p_{x'} w^{-1} \sim p_{x} w w^{-1} \sim p_{x}$, donc $f \sim p_{x}$. 

Le résultat s'en déduit par récurrence. 
\end{proof}

\begin{prop}\label{Proposition2.1.8.THG}
Soient $M$ une catégorie admettant des produits (\emph{resp.} des sommes) binaires et $W$ une partie de $\UnCell{M}$ contenant les identités et stable par produit (\emph{resp.} par somme) de deux flèches. Alors, la catégorie $\Localisation{M}{W}$ admet des produits (\emph{resp.} des sommes) binaires et le foncteur de localisation $M \to \Localisation{M}{W}$ commute à ces produits (\emph{resp.} ces sommes). 
\end{prop}

\begin{proof}
C'est \cite[proposition 2.1.8]{THG}.
\end{proof}

\begin{prop}
La catégorie $\Localisation{\DeuxCat}{\DeuxLocFond{W}}$ admet des produits finis et des sommes finies et le foncteur de localisation $\DeuxCat \to \Localisation{\DeuxCat}{\DeuxLocFond{W}}$ y commute.
\end{prop}

\begin{proof}
Cela résulte des propositions \ref{Proposition2.1.8.THG}, \ref{Proposition2.1.3THG}, \ref{Proposition2.1.4THG}, du lemme \ref{Lemme1.3.6.THG} et de son dual. 
\end{proof}

On termine cette section en montrant la stabilité par composition des morphismes asphériques de $\DeuxCat$. 

\begin{lemme}\label{LemmeHomotopieStricte}
Soient $u, v : \mathdeuxcat{A} \to \mathdeuxcat{B}$ deux morphismes parallèles de $\DeuxCat$. S'il existe une \DeuxTransformationStricte{} de $u$ vers $v$, alors $u$ est une équivalence faible si et seulement si $v$ en est une.
\end{lemme}

\begin{proof}
En vertu du lemme \ref{DeuxTransFoncLax}, il existe un diagramme commutatif
$$
\xymatrix{
&[1] \times \mathdeuxcat{A}
\ar[dd]^{h}
\\
\mathdeuxcat{A}
\ar[ur]^{0 \times 1_{\mathdeuxcat{A}}}
\ar[dr]_{u}
&&\mathdeuxcat{A}
\ar[ul]_{1 \times 1_{\mathdeuxcat{A}}}
\ar[dl]^{v}
\\
&\mathdeuxcat{B}
}
$$ 
dans $\DeuxCatLax$. L'examen des formules définissant $h$ (voir la démonstration du lemme \ref{DeuxTransFoncLax}) permet de vérifier qu'il s'agit, dans ce cas particulier, d'un \DeuxFoncteurStrict{}. Les flèches du diagramme ci-dessus sont donc toutes des morphismes de $\DeuxCat$. Le résultat s'en déduit par un argument de 2 sur 3.\end{proof}

\begin{rem}
Le lemme \ref{LemmeHomotopieLax} généralisera le lemme \ref{LemmeHomotopieStricte}.
\end{rem}

\begin{lemme}\label{1.1.7.livre}
Pour tout triangle commutatif
$$
\xymatrix{
\mathdeuxcat{A}
\ar[rr]^{u}
\ar[dr]_{w}
&&
\mathdeuxcat{B}
\ar[dl]^{v}
\\
&
\mathdeuxcat{C}
}
$$
de \DeuxFoncteursStricts{}, $u$ est lax-asphérique si et seulement si, pour tout objet $c$ de $\mathdeuxcat{C}$, le \DeuxFoncteurStrict{} $\DeuxFoncTrancheLax{u}{c} : \TrancheLax{\mathdeuxcat{A}}{w}{c} \to \TrancheLax{\mathdeuxcat{B}}{v}{c}$ est lax-asphérique. 
\end{lemme}

\begin{proof}
Soit $(b, p : v(b) \to c)$ un objet de $\TrancheLax{\mathdeuxcat{B}}{v}{c}$. On va exhiber un couple de \DeuxFoncteursStricts{} 
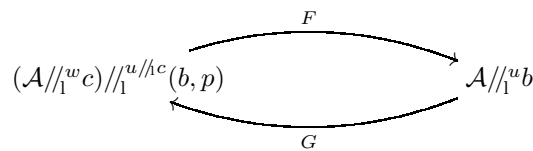
$$
\xymatrix{
\TrancheLax{(\TrancheLax{\mathdeuxcat{A}}{w}{c})}{\DeuxFoncTrancheLax{u}{c}}{(b,p)} 
\ar @/^1.5pc/ [rrr]^{F} 
&&&
\TrancheLax{\mathdeuxcat{A}}{u}{b}  
\ar @/^1.5pc/ [lll]^{G}}
$$
tels que 
$$
FG = 1_{\TrancheLax{\mathdeuxcat{A}}{u}{b}}
$$
et une \DeuxTransformationStricte{} 
$$
\sigma : 1_{\TrancheLax{(\TrancheLax{\mathdeuxcat{A}}{w}{c})}{\DeuxFoncTrancheLax{u}{c}}{(b,p)}} \Rightarrow GF.
$$
Il en résultera que $F$ et $G$ sont des équivalences faibles, donc en particulier que la \deux{}catégorie $\TrancheLax{(\TrancheLax{\mathdeuxcat{A}}{w}{c})}{\DeuxFoncTrancheLax{u}{c}}{(b,p)}$ est asphérique si et seulement si la \deux{}catégorie $\TrancheLax{\mathdeuxcat{A}}{u}{b}$ l'est, d'où le résultat annoncé. 

Explicitons partiellement la structure de la \deux{}catégorie 
$
\TrancheLax{(\TrancheLax{\mathdeuxcat{A}}{w}{c})}{\DeuxFoncTrancheLax{u}{c}}{(b,p)}
$.
\begin{itemize}
\item
Ses objets sont les quadruplets (que l'on écrira sous la forme d'un couple de couples) $((a,q),(r,\delta))$, où $a$ est un objet de $\mathdeuxcat{A}$, $q : w(a) \to c$ est une \un{}cellule de $\mathdeuxcat{C}$, $r : u(a) \to b$ est une \un{}cellule de $\mathdeuxcat{B}$ et $\delta : q \Rightarrow pv(r)$ est une \deux{}cellule de $\mathdeuxcat{C}$. 
\item 
Les \un{}cellules de $((a,q),(r,\delta))$ vers $((a',q'),(r',\delta'))$ dans $\TrancheLax{(\TrancheLax{\mathdeuxcat{A}}{w}{c})}{\DeuxFoncTrancheLax{u}{c}}{(b,p)}$ sont les triplets $((f,\alpha), \epsilon)$, où $f : a \to a'$ est une \un{}cellule de $\mathdeuxcat{A}$, $\alpha : q \Rightarrow q'w(f)$ est une \deux{}cellule de $\mathdeuxcat{C}$ et $\epsilon : r \Rightarrow r'u(f)$ est une \deux{}cellule de $\mathdeuxcat{B}$, ce triplet satisfaisant l'égalité
$$
(p \CompDeuxZero v(\epsilon)) \delta = (\delta' \CompDeuxZero w(f)) \alpha.
$$
\item 
Les \deux{}cellules de $((f, \alpha), \epsilon) : ((a,q),(r,\delta)) \to ((a',q'),(r',\delta'))$ vers $((g, \beta), \eta) : ((a,q),(r,\delta)) \to ((a',q'),(r',\delta'))$ dans $\TrancheLax{(\TrancheLax{\mathdeuxcat{A}}{w}{c})}{\DeuxFoncTrancheLax{u}{c}}{(b,p)}$ sont les \deux{}cellules $\varphi : f \Rightarrow g$ dans $\mathdeuxcat{A}$ telles que soient vérifiées les égalités
$$
(q' \CompDeuxZero w(\varphi)) \alpha = \beta
$$
et
$$
(r' \CompDeuxZero u(\varphi)) \epsilon = \eta.
$$
\item
L'identité de l'objet $((a,q), (r,\delta))$ est donnée par $((1_{a}, 1_{q},), 1_{r})$. 
\item
La composition des \un{}cellules est donnée par la formule générale
$$
((f', \alpha'), \epsilon') ((f, \alpha), \epsilon) = ((f'f, (\alpha' \CompDeuxZero w(f)) \CompDeuxUn \alpha), (\epsilon' \CompDeuxZero u(f)) \CompDeuxUn \epsilon).
$$
\end{itemize}

Définissons un \DeuxFoncteurStrict{} 
$$
\begin{aligned}
F : \TrancheLax{(\TrancheLax{\mathdeuxcat{A}}{w}{c})}{\DeuxFoncTrancheLax{u}{c}}{(b,p)} &\longrightarrow\TrancheLax{\mathdeuxcat{A}}{u}{b}
\\
((a,q),(r,\delta)) &\mapsto (a,r)
\\
((f, \alpha), \epsilon) &\mapsto (f, \epsilon)
\\
\varphi &\mapsto \varphi.
\end{aligned}
$$

Cela définit bien un \DeuxFoncteurStrict{} en vertu des vérifications suivantes.
\begin{itemize}
\item
Pour tout objet $((a, q), (r, \delta))$ de $\TrancheLax{(\TrancheLax{\mathdeuxcat{A}}{w}{c})}{\DeuxFoncTrancheLax{u}{c}}{(b,p)}$, 
$$
\begin{aligned}
F(1_{((a, q), (r, \delta))}) &= (1_{a}, 1_{r})
\\
&= 1_{(a, r)}
\\
&= 1_{F((a, q), (r, \delta))}.
\end{aligned}
$$
\item
Pour tout couple de \un{}cellules $((f, \alpha), \epsilon)$ et $((f', \alpha'), \epsilon')$ de $\TrancheLax{(\TrancheLax{\mathdeuxcat{A}}{w}{c})}{\DeuxFoncTrancheLax{u}{c}}{(b,p)}$ telles que la composée $((f', \alpha'), \epsilon') ((f, \alpha), \epsilon)$ fasse sens, 
$$
\begin{aligned}
F (((f', \alpha'), \epsilon') ((f, \alpha), \epsilon)) &= (f'f, (\epsilon' \CompDeuxZero u(f)) \epsilon)
\\
&= (f', \epsilon') (f, \epsilon)
\\
&= F((f', \alpha'), \epsilon') F((f, \alpha), \epsilon).
\end{aligned}
$$
\item
Pour tout couple de \deux{}cellules $\varphi$ et $\varphi'$ de $\TrancheLax{(\TrancheLax{\mathdeuxcat{A}}{w}{c})}{\DeuxFoncTrancheLax{u}{c}}{(b,p)}$ telles que la composée $\varphi' \CompDeuxZero \varphi$ fasse sens, l'égalité
$$
F (\varphi' \CompDeuxZero \varphi) = F(\varphi') \CompDeuxZero F(\varphi)
$$
est évidente. 
\item
Pour tout couple de \deux{}cellules $\varphi$ et $\psi$ telles que la composée $\psi \CompDeuxUn \varphi$ fasse sens, l'égalité
$$
F (\psi \CompDeuxUn \varphi) = F(\psi) \CompDeuxUn F(\varphi)
$$
est évidente. 
\item
Pour toute \un{}cellule $((f, \alpha), \epsilon)$, l'égalité
$$
F(1_{((f, \alpha), \epsilon)}) = 1_{F((f, \alpha), \epsilon)}
$$
est évidente. 
\end{itemize}

Définissons maintenant un \DeuxFoncteurStrict{} 
$$
\begin{aligned}
G : \TrancheLax{\mathdeuxcat{A}}{u}{b} &\to \TrancheLax{(\TrancheLax{\mathdeuxcat{A}}{w}{c})}{\DeuxFoncTrancheLax{u}{c}}{(b,p)}
\\
(a,r) &\mapsto ((a, pv(r)), (r, 1_{pv(r)}))
\\
(f, \epsilon) &\mapsto ((f, p \CompDeuxZero v(\epsilon)), \epsilon)
\\
\varphi &\mapsto \varphi.
\end{aligned}
$$

Cela définit bien un \DeuxFoncteurStrict{} en vertu des vérifications suivantes. 

\begin{itemize}
\item
Il est immédiat que $G$ envoie bien toute \deux{}cellule de $\TrancheLax{\mathdeuxcat{A}}{u}{b}$ sur une \deux{}cellule de $\TrancheLax{(\TrancheLax{\mathdeuxcat{A}}{w}{c})}{\DeuxFoncTrancheLax{u}{c}}{(b,p)}$ de source et de but l'image par $G$ de la source et du but, respectivement, de la \deux{}cellule de départ. 
\item
Pour tout objet $(a,r)$ de $\TrancheLax{\mathdeuxcat{A}}{u}{b}$, 
$$
\begin{aligned}
G(1_{(a,r)}) &= G(1_{a}, 1_{r})
\\
&= ((1_{a}, 1_{pv(r)}), 1_{r})
\\
&= 1_{((a, pv(r)), (r, 1_{pv(r)}))}
\\
&= 1_{G(a,r)}.
\end{aligned}
$$
\item
Pour tout couple de \un{}cellules $(f, \epsilon) : (a,r) \to (a', r')$ et $(f', \epsilon') : (a', r') \to (a'', r'')$,
$$
\begin{aligned}
G((f', \epsilon') (f, \epsilon)) &= G(f'f, (\epsilon' \CompDeuxZero u(f)) \CompDeuxUn \epsilon)
\\
&= ((f'f, p \CompDeuxZero v((\epsilon' \CompDeuxZero u(f)) \epsilon)), (\epsilon' \CompDeuxZero u(f)) \CompDeuxUn \epsilon)
\\
&= ((f'f, (p \CompDeuxZero v(\epsilon') \CompDeuxZero v(u(f))) \CompDeuxUn (p \CompDeuxZero v(\epsilon))), (\epsilon' \CompDeuxZero u(f)) \CompDeuxUn \epsilon)
\\
&= ((f'f, (p \CompDeuxZero v(\epsilon') \CompDeuxZero w(f)) \CompDeuxUn (p \CompDeuxZero v(\epsilon))), (\epsilon' \CompDeuxZero u(f)) \CompDeuxUn \epsilon)
\\
&= ((f', p \CompDeuxZero v(\epsilon')), \epsilon') ((f, p \CompDeuxZero v(\epsilon)), \epsilon)
\\
&= G(f',\epsilon') G(f,\epsilon).
\end{aligned}
$$ 
\item
Pour tout couple de \deux{}cellules $\varphi$ et $\varphi'$ telles que la composée $\varphi' \CompDeuxZero \varphi$ fasse sens, l'égalité
$$
G (\varphi' \CompDeuxZero \varphi) = G(\varphi') \CompDeuxZero G(\varphi)
$$
est évidente. 
\item
Pour tout couple de \deux{}cellules $\varphi$ et $\psi$ telles que la composée $\psi \CompDeuxUn \varphi$ fasse sens, l'égalité
$$
G (\psi \CompDeuxUn \varphi) = G (\psi) \CompDeuxUn G (\varphi)
$$
est évidente. 
\item
Pour toute \un{}cellule $(f, \epsilon)$, l'égalité
$$
G(1_{(f, \epsilon)}) = 1_{G(f, \epsilon)}
$$
est évidente. 
\end{itemize}

L'identité 
$$
FG = 1_{\TrancheLax{\mathdeuxcat{A}}{u}{b}}
$$
est immédiate.

Pour tout objet $((a,q), (r,\delta))$ de $\TrancheLax{(\TrancheLax{\mathdeuxcat{A}}{w}{c})}{\DeuxFoncTrancheLax{u}{c}}{(b,p)}$, 
$$
GF((a,q),(r,\delta)) = ((a, pv(r)), (r, 1_{pv(r)})).
$$
Le triplet $((1_{a}, \delta), 1_{r})$ définit une \un{}cellule
$$
\sigma_{((a,q), (r,\delta))} : ((a,q), (r,\delta)) \to GF((a,q), (r,\delta)).
$$

Soit $((f,\alpha), \epsilon)$ une \un{}cellule de $((a,q), (r,\delta))$ vers $((a',q'), (r',\delta'))$ dans $\TrancheLax{(\TrancheLax{\mathdeuxcat{A}}{w}{c})}{\DeuxFoncTrancheLax{u}{c}}{(b,p)}$. On a donc en particulier l'égalité 
$$
(p \CompDeuxZero v(\epsilon)) \CompDeuxUn \delta = (\delta' \CompDeuxZero w(f)) \CompDeuxUn \alpha.
$$
De plus, 
$$
GF ((f, \alpha), \epsilon) = ((f, p \CompDeuxZero v(\epsilon)), \epsilon).
$$
Considérons le diagramme
$$
\xymatrix{
((a,q), (r,\delta))
\ar[rrr]^{((f,\alpha), \epsilon)}
\ar[dd]_{((1_{a}, \delta), 1_{r})}
&&&((a',q'), (r',\delta'))
\ar[dd]^{((1_{a'}, \delta'), 1_{r'})}
\\
\\
((a, pv(r)), (r, 1_{pv(r)}))
\ar[rrr]_{((f, p \CompDeuxZero v(\epsilon)), \epsilon)}
&&&
((a', pv(r')), (r', 1_{pv(r')}))
&.
}
$$
Il est commutatif en vertu de l'égalité déjà soulignée $(p \CompDeuxZero v(\epsilon)) \CompDeuxUn \delta = (\delta' \CompDeuxZero w(f)) \CompDeuxUn \alpha$ et des égalités
$$
\begin{aligned}
((1_{a'}, \delta'), 1_{r'}) ((f,\alpha), \epsilon) &= ((f, (\delta' \CompDeuxZero w(f)) \CompDeuxUn \alpha), \epsilon)
\end{aligned}
$$
et
$$
\begin{aligned}
((f, p \CompDeuxZero v(\epsilon)), \epsilon) ((1_{a}, \delta), 1_{r}) &= ((f, (p \CompDeuxZero v(\epsilon) \CompDeuxZero w(1_{a})) \CompDeuxUn \delta), (\epsilon \CompDeuxZero u(1_{a})) \CompDeuxUn 1_{r})
\\
&= ((f, (p \CompDeuxZero v(\epsilon)) \CompDeuxUn \delta), \epsilon).
\end{aligned}
$$

De plus, pour toute \deux{}cellule $\varphi$ de $((f, \alpha), \epsilon)$ vers $((g, \beta), \eta)$ dans $\TrancheLax{(\TrancheLax{\mathdeuxcat{A}}{w}{c})}{\DeuxFoncTrancheLax{u}{c}}{(b,p)}$, l'égalité $1_{1_{a'}} \CompDeuxZero \varphi = \varphi \CompDeuxZero 1_{1_{a}}$ se récrit
$\sigma_{((a',q'), (r',\delta'))} \CompDeuxZero \varphi = GF(\varphi) \CompDeuxZero \sigma_{((a, q), (r, \delta))}$. On a donc bien défini une \DeuxTransformationStricte{} 
$$
\sigma : 1_{\TrancheLax{(\TrancheLax{\mathdeuxcat{A}}{w}{c})}{\DeuxFoncTrancheLax{u}{c}}{(b,p)}} \Rightarrow GF.
$$

En vertu du lemme \ref{LemmeHomotopieStricte}, comme annoncé, $F$ et $G$ sont donc des équivalences faibles. En particulier, la \deux{}catégorie $\TrancheLax{(\TrancheLax{\mathdeuxcat{A}}{w}{c})}{\DeuxFoncTrancheLax{u}{c}}{(b,p)}$ est asphérique si et seulement si la \deux{}catégorie $\TrancheLax{\mathdeuxcat{A}}{u}{b}$ l'est. Ainsi, $\DeuxFoncTrancheLax{u}{c}$ est lax-asphérique si et seulement si $u$ l'est. 
\end{proof}

\begin{rem}
On peut en outre vérifier les « identités triangulaires »
$
F \CompDeuxZero \sigma = 1_{F}
$
et
$
\sigma \CompDeuxZero G = 1_{G}
$
dans $\DeuxCatDeuxCat$.
\end{rem}

\begin{rem}
L'énoncé du lemme \ref{1.1.7.livre} semble rester valide en n'imposant pas à $u$ d'être strict\footnote{Voir la définition \ref{DefLaxAspherique}.}, mais nous n'avons pas vérifié suffisamment les détails pour énoncer ici le résultat le plus général. 
\end{rem}

\begin{rem}
Le lemme \ref{1.1.7.livre} admet bien entendu trois variantes duales. Plus précisément, sous les mêmes hypothèses, $u$ est lax-opasphérique (\emph{resp.} co\-lax-asphérique, \emph{resp.} co\-lax-op\-asphé\-rique) si et seulement si, pour tout objet $c$ de $\mathdeuxcat{C}$, le \DeuxFoncteurStrict{} $\DeuxFoncOpTrancheLax{u}{c} : \OpTrancheLax{\mathdeuxcat{A}}{w}{c} \to \OpTrancheLax{\mathdeuxcat{B}}{v}{c}$ (\emph{resp.} $\DeuxFoncTrancheCoLax{u}{c} : \TrancheCoLax{\mathdeuxcat{A}}{w}{c} \to \TrancheCoLax{\mathdeuxcat{B}}{v}{c}$, \emph{resp.} $\DeuxFoncOpTrancheCoLax{u}{c} : \OpTrancheCoLax{\mathdeuxcat{A}}{w}{c} \to \OpTrancheCoLax{\mathdeuxcat{B}}{v}{c}$) est lax-opasphérique (\emph{resp.} co\-lax-asphérique, \emph{resp.} co\-lax-op\-asphérique). Donnons par exemple la démonstration du premier cas. Le triangle de \DeuxFoncteursStricts{}
$$
\xymatrix{
\DeuxCatUnOp{\mathdeuxcat{A}}
\ar[rr]^{\DeuxFoncUnOp{u}}
\ar[dr]_{\DeuxFoncUnOp{w}}
&&
\DeuxFoncUnOp{\mathdeuxcat{B}}
\ar[dl]^{\DeuxFoncUnOp{v}}
\\
&
\DeuxFoncUnOp{\mathdeuxcat{C}}
}
$$
étant commutatif, le lemme \ref{1.1.7.livre} assure que $\DeuxFoncUnOp{u}$ est lax-asphérique si et seulement si $\DeuxFoncTrancheLax{\DeuxFoncUnOp{u}}{c}$ est lax-asphérique pour tout objet $c$ de $\mathdeuxcat{C}$. Ainsi, $u$ est lax-opasphérique si et seulement si $\DeuxFoncUnOp{(\DeuxFoncTrancheLax{\DeuxFoncUnOp{u}}{c})}$ l'est pour tout objet $c$ de $\mathdeuxcat{C}$, c'est-à-dire si et seulement si $\DeuxFoncOpTrancheLax{u}{c}$ l'est pour tout objet $c$ de $\mathdeuxcat{C}$. 
\end{rem}

\begin{prop}\label{1.1.8.livre}
Soit $u : \mathdeuxcat{A} \to \mathdeuxcat{B}$ et $v : \mathdeuxcat{B} \to \mathdeuxcat{C}$ un couple de \DeuxFoncteursStricts{} composables. Si $u$ est lax-asphérique (\emph{resp.} lax-op\-asphé\-rique, \emph{resp.} co\-lax-asphé\-rique, \emph{resp.} co\-lax-op\-asphé\-rique), alors $vu$ est lax-asphérique (\emph{resp.} lax-op\-asphé\-rique, \emph{resp.} co\-lax-asphé\-rique, \emph{resp.} co\-lax-op\-asphé\-rique) si et seulement si $v$ l'est.
\end{prop}

\begin{proof} 
Supposons $u$ lax-asphérique. En vertu du lemme \ref{1.1.7.livre}, pour tout objet $c$ de $\mathdeuxcat{C}$, $\DeuxFoncTrancheLax{u}{c} : \TrancheLax{\mathdeuxcat{A}}{vu}{c} \to \TrancheLax{\mathdeuxcat{B}}{v}{c}$ est lax-asphérique, donc en particulier une équivalence faible. Par conséquent, $\TrancheLax{\mathdeuxcat{B}}{v}{c}$ est asphérique si et seulement si $\TrancheLax{\mathdeuxcat{A}}{vu}{c}$ est asphérique. Le résultat s'ensuit. Les trois autres assertions s'en déduisent facilement par un argument de dualité. Par exemple, si $u$ est lax-opasphérique, alors $\DeuxFoncUnOp{u}$ est lax-asphérique. En vertu de ce que l'on vient de voir, $\DeuxFoncUnOp{v} \DeuxFoncUnOp{u}$ est lax-asphérique si et seulement si $\DeuxFoncUnOp{v}$ l'est, donc $\DeuxFoncUnOp{(vu)}$ est lax-asphérique si et seulement si $\DeuxFoncUnOp{v}$ l'est, donc $vu$ est lax-opasphérique si et seulement si $v$ l'est.
\end{proof} 

\section{Homotopie des morphismes lax}\label{SectionHomotopieMorphismesLax}

\begin{prop}\label{LemmeCouniteColaxAspherique}
Pour toute petite \deux{}catégorie $\mathdeuxcat{A}$, pour tout objet $a$ de $\mathdeuxcat{A}$, la \deux{}catégorie $\TrancheCoLax{\TildeLax{\mathdeuxcat{A}}}{\StrictCanonique{\mathdeuxcat{A}}}{a}$ admet un objet admettant un objet final.
\end{prop}

\begin{proof}
Décrivons partiellement la structure de cette \deux{}catégorie\footnote{Conformément à nos habitudes, nous n'explicitons pas tout.} :
\begin{itemize}
\item Les objets de $\TrancheCoLax{\TildeLax{\mathdeuxcat{A}}}{\StrictCanonique{\mathdeuxcat{A}}}{a}$ sont les couples 
$$
(a', p : a' \to a) 
$$
où $a'$ est un objet de $\mathdeuxcat{A}$ et $p$ est une 1-cellule de $\mathdeuxcat{A}$.
\item Si $(a',p)$ et $(a'',q)$ sont deux objets de $\TrancheCoLax{\TildeLax{\mathdeuxcat{A}}}{\StrictCanonique{\mathdeuxcat{A}}}{a}$, les \un{}cellules de $(a',p)$ vers $(a'',q)$ dans $\TrancheCoLax{\TildeLax{\mathdeuxcat{A}}}{\StrictCanonique{\mathdeuxcat{A}}}{a}$ sont les triplets
$$
([m], x, \alpha : q x_{m,m-1} \dots x_{1,0} \Rightarrow p)
$$
tels que $([m], x)$ soit une \un{}cellule de $a'$ vers $a''$ dans $\TildeLax{\mathdeuxcat{A}}$ (en particulier, $x_{0} = a'$ et $x_{m} = a''$).
\item Si $(a',p)$ et $(a'',q)$ sont deux objets de $\TrancheCoLax{\TildeLax{\mathdeuxcat{A}}}{\StrictCanonique{\mathdeuxcat{A}}}{a}$, et $([m], x, \alpha)$ et $([n], y, \beta)$ deux \un{}cellules de $(a',p)$ vers $(a'',q)$ dans $\TrancheCoLax{\TildeLax{\mathdeuxcat{A}}}{\StrictCanonique{\mathdeuxcat{A}}}{a}$, les \deux{}cellules de $([m], x, \alpha)$ vers $([n], y, \beta)$ dans $\TrancheCoLax{\TildeLax{\mathdeuxcat{A}}}{\StrictCanonique{\mathdeuxcat{A}}}{a}$ sont les \deux{}cellules $(\varphi : [n] \to [m], \delta : x \varphi \Rightarrow y)$ de $([m],x)$ vers $([n],y)$ dans $\TildeLax{\mathdeuxcat{A}}$ ($\varphi$ est donc un morphisme d'intervalles) telles que 
$$
\beta (1_{q} \CompDeuxZero \delta_{n,n-1} \CompDeuxZero \dots \CompDeuxZero \delta_{1,0}) = \alpha.
$$
\end{itemize}

La \deux{}catégorie $\TrancheCoLax{\TildeLax{\mathdeuxcat{A}}}{\StrictCanonique{\mathdeuxcat{A}}}{a}$ possède l'objet $(a, 1_{a})$. Si $(a', p : a' \to a)$ est un objet arbitraire de $\TrancheCoLax{\TildeLax{\mathdeuxcat{A}}}{\StrictCanonique{\mathdeuxcat{A}}}{a}$, $([1], p, 1_{p})$
définit un objet  de
$\CatHom{\TrancheCoLax{\TildeLax{\mathdeuxcat{A}}}{\StrictCanonique{\mathdeuxcat{A}}}{a}}{(a',p)} {(a, 1_{a})}$. Soit $([m], x, \alpha)$ un objet de $\CatHom{\TrancheCoLax{\TildeLax{\mathdeuxcat{A}}}{\StrictCanonique{\mathdeuxcat{A}}}{a}}{(a',p)} {(a, 1_{a})}$. Les morphismes de $([m], x, \alpha)$ vers $([1], p, 1_{p})$ dans la catégorie $\CatHom{\TrancheCoLax{\TildeLax{\mathdeuxcat{A}}}{\StrictCanonique{\mathdeuxcat{A}}}{a}}{(a',p)} {(a, 1_{a})}$ s'identifient aux couples
$$
(\varphi : [1] \to [m], \delta : x_{m,m-1} \dots x_{1,0} \Rightarrow p)
$$
avec $\varphi$ un morphisme d'intervalles et $\delta$ une \deux{}cellule de $\mathdeuxcat{A}$ tels que $1_{p} \CompDeuxUn (1_{1_{a}} \CompDeuxZero \delta) = \alpha$, c'est-à-dire $\delta = \alpha$. Cela impose l'existence et l'unicité d'un tel morphisme de $([m], x, \alpha)$ vers $([1], p, 1_{p})$ dans $\CatHom{\TrancheCoLax{\TildeLax{\mathdeuxcat{A}}}{\StrictCanonique{\mathdeuxcat{A}}}{a}}{(a',p)} {(a, 1_{a})}$. La \deux{}catégorie $\TrancheCoLax{\TildeLax{\mathdeuxcat{A}}}{\StrictCanonique{\mathdeuxcat{A}}}{a}$ admet donc un objet admettant un objet final, ce qu'il fallait vérifier. 
\end{proof}

\begin{prop}\label{CouniteColaxAspherique}
Pour toute petite \deux{}catégorie $\mathdeuxcat{A}$, le \DeuxFoncteurStrict{} $\StrictCanonique{\mathdeuxcat{A}} : \TildeLax{\mathdeuxcat{A}} \to \mathdeuxcat{A}$ est colax-asphérique (donc en particulier une équivalence faible).
\end{prop}

\begin{proof}
C'est une conséquence directe de la proposition \ref{LemmeCouniteColaxAspherique}.
\end{proof}

\begin{lemme}\label{BarreTildeW}
Soit $u$ un \deux{}foncteur strict. Les propositions suivantes sont équivalentes. 
\begin{itemize}
\item[(i)]
Le \DeuxFoncteurStrict{} $u$ est une équivalence faible ;
\item[(ii)]
Le \DeuxFoncteurStrict{} $\BarreLax{u}$ est une équivalence faible ; 
\item[(iii)]
Le \DeuxFoncteurStrict{} $\TildeLax{u}$ est une équivalence faible.
\end{itemize}
\end{lemme}

\begin{proof}
En vertu de la définition \ref{DefBarreLax}, c'est une conséquence immédiate du lemme \ref{CouniteNaturelle}, de la saturation faible de $\DeuxLocFond{W}$ et de la proposition \ref{CouniteColaxAspherique}.
\end{proof}

\begin{df}
On dira qu'un \DeuxFoncteurLax{} $u$ est une $\DeuxLocFond{W}$\emph{-équivalence faible lax} ou, plus simplement, une \emph{équivalence faible lax}\index{equivalence faible lax (pour un localisateur fondamental de $\DeuxCat$)@équivalence faible lax (pour un localisateur fondamental de $\DeuxCat$)}, voire, en l'absence d'ambiguïté, une \emph{équivalence faible}, si le \DeuxFoncteurStrict{} $\TildeLax{u}$ est dans $\DeuxLocFond{W}$. On notera $\DeuxLocFondLaxInduit{W}$ la classe des $\DeuxLocFond{W}$\nobreakdash-équivalences faibles lax. Autrement dit, $\DeuxLocFondLaxInduit{W} = \FoncBenabou^{-1} (\DeuxLocFond{W})$. 
\end{df}

\begin{rem}\label{EquiStricteEquiLax}
En vertu du lemme \ref{BarreTildeW}, un \deux{}foncteur strict est une équivalence faible si et seulement si c'est une équivalence faible lax.
\end{rem}

\begin{paragr}
Il paraît naturel d'utiliser la notion de nerf lax pour définir une notion d'équivalence faible lax. Plus précisément, on pourrait se trouver tenté de dire qu'un \DeuxFoncteurLax{} $u$ est une équivalence faible lax si $\Delta/\NerfLax{(u)}$ est dans $\DeuxLocFond{W} \cap \UnCell{\Cat}$. 
Il semble donc que l'on dispose de deux définitions de la classe des équivalences faibles lax, à savoir $\NerfLax^{-1} (i_{\Delta}^{-1} (\DeuxLocFond{W} \cap \UnCell{\Cat}))$ et $\FoncBenabou^{-1}({\DeuxLocFond{W}})$. On verra plus loin qu'elles sont équivalentes. La proposition \ref{EquiDefEquiLax} dissipe les premières inquiétudes à ce sujet.
\end{paragr}

\emph{On ne suppose plus fixé de \ClasseDeuxLocFond{}.} 

\begin{prop}\label{EquiDefEquiLax}
Pour tout localisateur fondamental $\UnLocFond{W}$ de $\Cat$, pour tout \DeuxFoncteurLax{} u, $\Delta/\NerfLax(\TildeLax{u})$ est dans $\UnLocFond{W}$ si et seulement si $\Delta/\NerfLax(u)$ l'est.
\end{prop}

\begin{proof}
Soient $\UnLocFond{W}$ un \ClasseUnLocFond{} et $u : \mathdeuxcat{A} \to \mathdeuxcat{B}$ un  \DeuxFoncteurLax{}. Les \DeuxFoncteursStricts{} $\StrictCanonique{\mathdeuxcat{A}}$ et $\StrictCanonique{\mathdeuxcat{B}}$ sont des équivalences faibles pour tout \ClasseDeuxLocFond{} (proposition \ref{CouniteColaxAspherique}). Comme $\NerfLaxNor^{-1} (i_{\Delta}^{-1} (W))$ est un \ClasseDeuxLocFond{}, $\Delta/\NerfLaxNor(\StrictCanonique{\mathdeuxcat{A}})$ et $\Delta/\NerfLaxNor(\StrictCanonique{\mathdeuxcat{B}})$ sont dans $\UnLocFond{W}$. Il en est donc de même de $\Delta/\NerfLax(\StrictCanonique{\mathdeuxcat{A}})$ et $\Delta/\NerfLax(\StrictCanonique{\mathdeuxcat{B}})$ (en vertu de l'implication $(i) \Rightarrow (ii)$ du lemme \ref{Blabla}). La saturation faible de $\UnLocFond{W}$ permet d'en déduire que les sections $\Delta/\NerfLax(\LaxCanonique{\mathdeuxcat{A}})$ et $\Delta/\NerfLax(\LaxCanonique{\mathdeuxcat{B}})$ sont dans $\UnLocFond{W}$. On conclut par un argument de 2 sur 3 après avoir appliqué le foncteur $i_{\Delta} \NerfLax$ au diagramme commutatif
$$
\xymatrix{
\TildeLax{\mathdeuxcat{A}}
\ar[r]^{\TildeLax{u}}
&\TildeLax{\mathdeuxcat{B}}
\\
\mathdeuxcat{A}
\ar[u]^{\LaxCanonique{\mathdeuxcat{A}}}
\ar[r]_{u}
&\mathdeuxcat{B}
\ar[u]_{\LaxCanonique{\mathdeuxcat{B}}}
&.
}
$$
\end{proof}

\emph{On suppose fixé un localisateur fondamental $\DeuxLocFond{W}$ de $\DeuxCat$.}

\begin{lemme}\label{SatFaibleEquiLax}
La classe des équivalences faibles lax est faiblement saturée.
\end{lemme}

\begin{proof}
Cela résulte par fonctorialité de la définition de la notion d'équivalence faible lax et de la saturation faible de $\DeuxLocFond{W}$. 
\end{proof}

\begin{prop}\label{UniteW}
Pour toute petite \deux{}catégorie $\mathdeuxcat{A}$, le \DeuxFoncteurLax{} $\LaxCanonique{\mathdeuxcat{A}} : \mathdeuxcat{A} \to \TildeLax{\mathdeuxcat{A}}$ est une équivalence faible.
\end{prop}

\begin{proof}
En vertu de la proposition \ref{CouniteColaxAspherique} et de la remarque \ref{EquiStricteEquiLax}, cela résulte de l'égalité $\StrictCanonique{\mathdeuxcat{A}} \LaxCanonique{\mathdeuxcat{A}} = 1_{\mathdeuxcat{A}}$.
\end{proof}

\begin{lemme}\label{EquiLaxBarreLax}
Un \DeuxFoncteurLax{} $u : \mathdeuxcat{A} \to \mathdeuxcat{B}$ est une équivalence faible lax si et seulement si $\BarreLax{u}$ est une équivalence faible.
\end{lemme}

\begin{proof}
Cela résulte de l'égalité $\BarreLax{u} \LaxCanonique{\mathdeuxcat{A}} = u$ et de la proposition \ref{UniteW}. 
\end{proof}


\begin{theo}\label{EqCatLocDeuxCatDeuxCatLax}
L'inclusion 
$$
I : \DeuxCat \hookrightarrow \DeuxCatLax
$$ 
et le foncteur de strictification de Bénabou 
$$
B : \DeuxCatLax \to \DeuxCat
$$ 
induisent des équivalences de catégories 
$$
\Localisation{\DeuxCat}{\DeuxLocFond{W}} \to \Localisation{\DeuxCatLax}{\DeuxLocFondLaxInduit{W}}
$$ 
et 
$$
\Localisation{\DeuxCatLax}{\DeuxLocFondLaxInduit{W}} \to \Localisation{\DeuxCat}{\DeuxLocFond{W}}
$$ 
quasi-inverses l'une de l'autre.
\end{theo}

\begin{proof}
Les foncteurs $I$ et $B$ respectant les équivalences faibles, ils induisent bien des foncteurs entre les catégories localisées de l'énoncé. Le reste est conséquence du fait que les composantes des transformations naturelles $\TransLaxCanonique$ et $\TransStrictCanonique$ sont dans $\DeuxLocFondLaxInduit{W}$ et $\DeuxLocFond{W}$ respectivement.
\end{proof}

\begin{lemme}\label{LemmeHomotopieLax}(Lemme d'homotopie lax) 
Soit $\mathcal{I}$ un ensemble de segments de $\DeuxCatLax$ tel que, pour tout segment $\mathbb{I} = (I, \partial_{0}, \partial_{1})$ appartenant à $\mathcal{I}$, le morphisme canonique $I \to \DeuxCatPonct{}$ soit dans $\DeuxLocFond{W}$ (donc universellement dans $\DeuxLocFond{W}$). Si $F, G : \mathdeuxcat{A} \to \mathdeuxcat{B}{}$ sont deux \DeuxFoncteursLax{} $\mathcal{I}$\nobreakdash-homotopes, alors :
\begin{itemize}
\item[(a)] On a l'égalité $\gamma_{lax}(F) = \gamma_{lax}(G)$, où $\gamma_{lax} : \DeuxCatLax{} \to \Localisation{\DeuxCatLax}{\DeuxLocFondLaxInduit{W}}$ désigne le foncteur canonique de localisation.
\item[(a$'$)] Si $F$ et $G$ sont des \deux{}foncteurs stricts, alors $\gamma(F) = \gamma(G)$, où $\gamma : \DeuxCat \to \Localisation{\DeuxCat}{\DeuxLocFond{W}}$ désigne le foncteur canonique de localisation.
\item[(b)] Le \DeuxFoncteurLax{} $F$ est dans $\DeuxLocFondLaxInduit{W}$ si et seulement si $G$ l'est.
\item[(b$'$)] Si $F$ et $G$ sont des \deux{}foncteurs stricts, alors $F$ est dans $\DeuxLocFond{W}$ si et seulement si $G$ l'est.
\item[(c)] Si $F$ est un isomorphisme et $G$ un morphisme constant, alors les morphismes canoniques $\mathdeuxcat{A} \to \DeuxCatPonct{}$ et $\mathdeuxcat{B} \to \DeuxCatPonct{}$ sont universellement dans $\DeuxLocFond{W}$. En particulier, $\mathdeuxcat{A}$ et $\mathdeuxcat{B}$ sont $\DeuxLocFond{W}$\nobreakdash-asphériques.
\end{itemize}
\end{lemme}

\begin{proof}
Les assertions \emph{(a)} et \emph{(b)} résultent du lemme \ref{LemmeHomotopieTHG}. On en déduit les assertions \emph{(a$'$)} et \emph{(b$'$)} en vertu du théorème \ref{EqCatLocDeuxCatDeuxCatLax} (puisqu'une équivalence de catégories est un foncteur fidèle) et du lemme \ref{BarreTildeW} respectivement. L'assertion $(c)$ résulte également du lemme \ref{LemmeHomotopieTHG}. 
\end{proof}

\begin{lemme}\label{LemmeDualiteTildeUnOp}
Soient $\mathdeuxcat{A}$ une petite \deux{}catégorie et $u$ et $v$ des \DeuxFoncteursStricts{} de source $\DeuxCatUnOp{(\TildeLax{\mathdeuxcat{A}})}$. Alors, $u = v$ si et seulement si $u \DeuxFoncUnOp{(\LaxCanonique{\mathdeuxcat{A}})} = v \DeuxFoncUnOp{(\LaxCanonique{\mathdeuxcat{A}})}$.
\end{lemme}

\begin{proof}
L'égalité $u \DeuxFoncUnOp{(\LaxCanonique{\mathdeuxcat{A}})} = v \DeuxFoncUnOp{(\LaxCanonique{\mathdeuxcat{A}})}$ équivaut à $\DeuxFoncUnOp{u} \LaxCanonique{\mathdeuxcat{A}} = \DeuxFoncUnOp{v} \LaxCanonique{\mathdeuxcat{A}}$, donc à $\DeuxFoncUnOp{u} = \DeuxFoncUnOp{v}$, donc à $u = v$. 
\end{proof}

\begin{lemme}\label{DualiteTildeUnOp}
Pour toute petite \deux{}catégorie $\mathdeuxcat{A}$, il existe un isomorphisme canonique (qui est un \DeuxFoncteurStrict)
$$
j_{\mathdeuxcat{A}} : \DeuxCatUnOp{(\TildeLax{\mathdeuxcat{A}})} \to \TildeLax{(\DeuxCatUnOp{\mathdeuxcat{A}})}
$$
tel que le diagramme
$$
\xymatrix{
\DeuxCatUnOp{(\TildeLax{\mathdeuxcat{A}})}
\ar[rr]^{j_{\mathdeuxcat{A}}}
&&\TildeLax{(\DeuxCatUnOp{\mathdeuxcat{A}})}
\\
&\DeuxCatUnOp{\mathdeuxcat{A}}
\ar[ul]^{\DeuxFoncUnOp{(\LaxCanonique{\mathdeuxcat{A}})}}
\ar[ur]_{\LaxCanonique{(\DeuxCatUnOp{\mathdeuxcat{A}})}}
}
$$
soit commutatif.
\end{lemme}

\begin{proof}
C'est immédiat.  
\end{proof}

\begin{rem}\label{RemStrictCanoniqueUnOp}
On en déduit l'égalité $\StrictCanonique{(\DeuxCatUnOp{\mathdeuxcat{A}})} j_{\mathdeuxcat{A}} = \DeuxFoncUnOp{(\StrictCanonique{\mathdeuxcat{A}})}$. En effet, on a les égalités
$$
\StrictCanonique{(\DeuxCatUnOp{\mathdeuxcat{A}})} j_{\mathdeuxcat{A}} \DeuxFoncUnOp{(\LaxCanonique{\mathdeuxcat{A}})}  = \StrictCanonique{(\DeuxCatUnOp{\mathdeuxcat{A}})} \LaxCanonique{(\DeuxCatUnOp{\mathdeuxcat{A}})} = 1_{\DeuxCatUnOp{\mathdeuxcat{A}}}
$$
et
$$
\DeuxFoncUnOp{(\StrictCanonique{\mathdeuxcat{A}})} \DeuxFoncUnOp{(\LaxCanonique{\mathdeuxcat{A}})} = \DeuxFoncUnOp{(\StrictCanonique{\mathdeuxcat{A}} \LaxCanonique{\mathdeuxcat{A}})} = \DeuxFoncUnOp{(1_{\mathdeuxcat{A}})} = 1_{\DeuxCatUnOp{\mathdeuxcat{A}}}
$$
L'identité annoncée résulte donc du lemme \ref{LemmeDualiteTildeUnOp}.
\end{rem}

\begin{lemme}\label{LemmeDeuxFoncLaxUnOpW}
En conservant les notations du lemme \ref{DualiteTildeUnOp}, pour toute petite \deux{}catégorie $\mathdeuxcat{A}$ et tout \DeuxFoncteurLax{} $u$ de source $\mathdeuxcat{A}$, 
$$
\BarreLax{(\DeuxFoncUnOp{u})} j_{\mathdeuxcat{A}} = \DeuxFoncUnOp{(\BarreLax{u})}
$$
\end{lemme}

\begin{proof}
En vertu du lemme \ref{LemmeDualiteTildeUnOp}, cette égalité équivaut à $\BarreLax{(\DeuxFoncUnOp{u})} j_{\mathdeuxcat{A}} \DeuxFoncUnOp{(\LaxCanonique{\mathdeuxcat{A}})}  = \DeuxFoncUnOp{(\BarreLax{u})} \DeuxFoncUnOp{(\LaxCanonique{\mathdeuxcat{A}})}$, ce qui se récrit $\BarreLax{(\DeuxFoncUnOp{u})} \LaxCanonique{(\DeuxCatUnOp{\mathdeuxcat{A}})} = \DeuxFoncUnOp{(\BarreLax{u} \LaxCanonique{\mathdeuxcat{A}})}$, soit $\BarreLax{(\DeuxFoncUnOp{u})} \LaxCanonique{(\DeuxCatUnOp{\mathdeuxcat{A}})} = \DeuxFoncUnOp{u}$, ce qui est vrai par définition.
\end{proof}

\begin{prop}\label{DeuxFoncLaxUnOpW}
Un \DeuxFoncteurLax{} $u$ est une équivalence faible si et seulement si $\DeuxFoncUnOp{u}$ en est une. 
\end{prop}

\begin{proof}
Notons $\mathdeuxcat{A}$ la source de $u$. En conservant les notations employées ci-dessus, $j_{\mathdeuxcat{A}}$ est un isomorphisme, donc une équivalence faible et donc, en vertu de la saturation faible de $\DeuxLocFond{W}$ et du lemme \ref{LemmeDeuxFoncLaxUnOpW}, $\BarreLax{(\DeuxFoncUnOp{u})}$ est une équivalence faible si et seulement si $\DeuxFoncUnOp{(\BarreLax{u})}$ en est une. Or, par définition, $\BarreLax{(\DeuxFoncUnOp{u})}$ est une équivalence faible si et seulement si $\DeuxFoncUnOp{u}$ en est une et, en vertu de la proposition \ref{DeuxFoncUnOpW}, $\DeuxFoncUnOp{(\BarreLax{u})}$ est une équivalence faible si et seulement si $\BarreLax{u}$ en est une, c'est-à-dire, par définition, si et seulement si $u$ en est une ; d'où le résultat annoncé.
\end{proof}

\begin{df}\label{DefEquiColax}
On dira qu'un \DeuxFoncteurCoLax{} $u : \mathdeuxcat{A} \to \mathdeuxcat{B}$ est une \emph{$\DeuxLocFond{W}$\nobreakdash-équivalence faible colax}\index{equivalence faible colax (pour un localisateur fondamental de $\DeuxCat$)@équivalence faible colax (pour un localisateur fondamental de $\DeuxCat$)}, ou plus simplement une \emph{équivalence faible}, si le \DeuxFoncteurLax{} $\DeuxFoncDeuxOp{u}$ est une équivalence faible lax. 
\end{df}

\begin{lemme}
Un \DeuxFoncteurStrict{} $u$ est une équivalence faible colax si et seulement s'il est dans $\DeuxLocFond{W}$. 
\end{lemme}

\begin{proof}
Par définition, $u$ est une équivalence faible colax si et seulement si le \DeuxFoncteurStrict{} $\DeuxFoncDeuxOp{u}$ est une équivalence faible lax, ce qui est le cas si et seulement si $\DeuxFoncDeuxOp{u}$ est dans $\DeuxLocFond{W}$ (voir la remarque \ref{EquiStricteEquiLax}), ce qui est le cas si et seulement si $u$ est dans $\DeuxLocFond{W}$ (en vertu de la proposition \ref{DeuxFoncDeuxOpW}). 
\end{proof}

\begin{lemme}\label{SatFaibleEquiColax}
La classe des équivalences faibles colax est faiblement saturée.
\end{lemme}

\begin{proof}
C'est immédiat en vertu de la définition \ref{DefEquiColax}, la classe des équivalences faibles lax étant faiblement saturée (voir le lemme \ref{SatFaibleEquiLax}). 
\end{proof} 

\begin{prop}\label{DeuxFoncColaxUnOpW}
Un \DeuxFoncteurCoLax{} $u$ est une équivalence faible si et seulement si $\DeuxFoncUnOp{u}$ en est une.
\end{prop}

\begin{proof}
Par définition, $\DeuxFoncUnOp{u}$ est une équivalence faible colax si et seulement si $\DeuxFoncToutOp{u}$ est une équivalence faible lax, ce qui, en vertu de la proposition \ref{DeuxFoncLaxUnOpW}, est le cas si et seulement si $\DeuxFoncDeuxOp{u}$ est une équivalence faible lax, ce qui est le cas, par définition, si et seulement si $u$ est une équivalence faible colax. 
\end{proof}

\begin{lemme}
Un \DeuxFoncteurCoLax{} $u$ est une équivalence faible si et seulement si $\BarreColax{u}$ est dans $\DeuxLocFond{W}$. 
\end{lemme}

\begin{proof}
Par définition, $u$ est une équivalence faible colax si et seulement si le \DeuxFoncteurLax{} $\DeuxFoncDeuxOp{u}$ est une équivalence faible lax, ce qui est le cas si et seulement si le \DeuxFoncteurStrict{} $\BarreLax{\DeuxFoncDeuxOp{u}}$ est dans $\DeuxLocFond{W}$, ce qui est le cas si et seulement si $\DeuxFoncDeuxOp{(\BarreLax{\DeuxFoncDeuxOp{u}})}$ est dans $\DeuxLocFond{W}$ (en vertu de la proposition \ref{DeuxFoncDeuxOpW}), autrement dit si et seulement si $\BarreColax{u}$ est dans $\DeuxLocFond{W}$. 
\end{proof}

\begin{df}\label{DefLaxAspherique}
Soient 
$$
\xymatrix{
\mathdeuxcat{A} 
\ar[rr]^{u}
\ar[dr]_{w}
&&\mathdeuxcat{B}
\ar[dl]^{v}
\\
&\mathdeuxcat{C}
}
$$
un diagramme de \DeuxFoncteursLax{} (\emph{resp.} de \DeuxFoncteursLax{}, \emph{resp.} de \DeuxFoncteursCoLax{}, \emph{resp.} de \DeuxFoncteursCoLax{}), non nécessairement commutatif, ainsi qu'une \DeuxTransformationCoLax{} (\emph{resp.} une \DeuxTransformationLax{}, \emph{resp.} une \DeuxTransformationLax{}, \emph{resp.} une \DeuxTransformationCoLax{}) $\sigma$ de $vu$ vers $w$ (\emph{resp.} de $w$ vers $vu$, \emph{resp.} de $vu$ vers $w$, \emph{resp.} de $w$ vers $vu$). Nous utilisons les notations de la section \ref{SectionMorphismesInduits}. On dira que $u$ est \emph{lax-asphérique au-dessus de}\index{lax-asphérique au-dessus de, relativement à} $\mathdeuxcat{C}$ \emph{relativement à} $\sigma$ (\emph{resp.} \emph{lax-opasphérique au-dessus de}\index{lax-opasphérique au-dessus de, relativement à} $\mathdeuxcat{C}$ \emph{relativement à} $\sigma$, \emph{resp.} \emph{colax-asphérique au-dessus de}\index{colax-asphérique au-dessus de, relativement à} $\mathdeuxcat{C}$ \emph{relativement à} $\sigma$, \emph{resp.} \emph{colax-opasphérique au-dessus de}\index{colax-opasphérique au-dessus de, relativement à} $\mathdeuxcat{C}$ \emph{relativement à} $\sigma$) si le \DeuxFoncteurLax{} $\DeuxFoncTrancheLaxCoq{u}{\sigma}{c} : \TrancheLax{\mathdeuxcat{A}}{w}{c} \to \TrancheLax{\mathdeuxcat{B}}{v}{c}$ (\emph{resp.} le \DeuxFoncteurLax{} $\DeuxFoncOpTrancheLaxCoq{u}{\sigma}{c} : \OpTrancheLax{\mathdeuxcat{A}}{w}{c} \to \OpTrancheLax{\mathdeuxcat{B}}{v}{c}$, \emph{resp.} le \DeuxFoncteurCoLax{} $\DeuxFoncTrancheCoLaxCoq{u}{\sigma}{c} : \TrancheCoLax{\mathdeuxcat{A}}{w}{c} \to \TrancheCoLax{\mathdeuxcat{B}}{v}{c}$, \emph{resp.} le \DeuxFoncteurCoLax{} $\DeuxFoncOpTrancheCoLaxCoq{u}{\sigma}{c} : \OpTrancheCoLax{\mathdeuxcat{A}}{w}{c} \to \OpTrancheCoLax{\mathdeuxcat{B}}{v}{c}$) est une équivalence faible lax (\emph{resp.} lax, \emph{resp.} colax, \emph{resp.} colax). Conformément à la terminologie déjà introduite, si $\sigma$ est une identité, on omettra les termes « relativement à $\sigma$ » dans les expressions ci-dessus. Si $\sigma$  et $v$ sont des identités, on omettra les termes « au-dessus de $\mathdeuxcat{C}$ ».
\end{df}

\begin{rem}
Un \DeuxFoncteurLax{} (\emph{resp.} \DeuxFoncteurLax{}, \emph{resp.} \DeuxFoncteurCoLax{}, \emph{resp.} \DeuxFoncteurCoLax{}) $u : \mathdeuxcat{A} \to \mathdeuxcat{B}$ est lax-asphérique (\emph{resp.} lax-opasphérique, \emph{resp.} colax-asphérique, \emph{resp.} colax-opasphérique) si et seulement si, pour tout objet $b$ de $\mathdeuxcat{B}$, la \deux{}catégorie $\TrancheLax{\mathdeuxcat{A}}{u}{b}$ (\emph{resp.} $\OpTrancheLax{\mathdeuxcat{A}}{u}{b}$, \emph{resp.} $\TrancheCoLax{\mathdeuxcat{A}}{u}{b}$, \emph{resp.} $\OpTrancheCoLax{\mathdeuxcat{A}}{u}{b}$) est asphérique. 
\end{rem}

\begin{lemme}\label{DualiteTrancheLax}
Soient 
$$
\xymatrix{
\mathdeuxcat{A} 
\ar[rr]^{u}
\ar[dr]_{w}
&{}
&\mathdeuxcat{B}
\ar[dl]^{v}
\\
&\mathdeuxcat{C}
\utwocell<\omit>{\sigma}
}
$$
un diagramme de \DeuxFoncteursLax{} commutatif à la \DeuxTransformationLax{} $\sigma : w \Rightarrow vu$ près seulement et $c$ un objet de $\mathdeuxcat{C}$. Alors, le \DeuxFoncteurLax{} $\DeuxFoncOpTrancheLaxCoq{u}{\sigma}{c}$ est une équivalence faible si et seulement si le \DeuxFoncteurLax{} $\DeuxFoncTrancheLaxCoq{\DeuxFoncUnOp{u}}{\DeuxTransUnOp{\sigma}}{c}$ en est une. 
\end{lemme}

\begin{proof}
Le morphisme $\DeuxFoncOpTrancheLaxCoq{u}{\sigma}{c}$ n'est autre que $\DeuxFoncUnOp{(\DeuxFoncTrancheLaxCoq{(\DeuxFoncUnOp{u})}{\DeuxTransUnOp{\sigma}}{c})}$. En vertu de la proposition \ref{DeuxFoncLaxUnOpW}, c'est une équivalence faible si et seulement si $\DeuxFoncTrancheLaxCoq{\DeuxFoncUnOp{u}}{\DeuxTransUnOp{\sigma}}{c}$ est une équivalence faible. 
\end{proof}

\begin{prop}\label{AAA}
Soit
$$
\xymatrix{
\mathdeuxcat{A} 
\ar[rr]^{u}
\ar[dr]_{w}
&{}
&\mathdeuxcat{B}
\ar[dl]^{v}
\\
&\mathdeuxcat{C}
\utwocell<\omit>{\sigma}
}
$$
un diagramme de \DeuxFoncteursLax{} commutatif à la \DeuxTransformationLax{} $\sigma : w \Rightarrow vu$ près seulement. Alors, $u$ est lax-opasphérique au-dessus de $\mathdeuxcat{C}$ relativement à $\sigma$ si et seulement si $\DeuxFoncUnOp{u}$ est lax-asphérique au-dessus de $\DeuxCatUnOp{\mathdeuxcat{C}}$ relativement à $\DeuxTransUnOp{\sigma}$.
\end{prop}

\begin{proof}
C'est une conséquence immédiate du lemme \ref{DualiteTrancheLax}.
\end{proof}

\begin{rem}
Le lemme \ref{DualiteTrancheLax} et la proposition \ref{AAA} admettent bien entendu moult variantes duales. Le lecteur pourra s'exercer à les expliciter. 
\end{rem}

\begin{lemme}\label{LemmeIsoInduit}
Soient $j : \mathdeuxcat{A} \to \mathdeuxcat{B}$ un isomorphisme de \deux{}catégories, $u : \mathdeuxcat{B} \to \mathdeuxcat{C}$ un \DeuxFoncteurLax{} (\emph{resp.} un \DeuxFoncteurLax{}, \emph{resp.} un \DeuxFoncteurCoLax{}, \emph{resp.} un \DeuxFoncteurCoLax{}) et $c$ un objet de $\mathdeuxcat{C}$. Alors $\DeuxFoncTrancheLax{j}{c} : \TrancheLax{\mathdeuxcat{A}}{uj}{c} \to \TrancheLax{\mathdeuxcat{B}}{u}{c}$ (\emph{resp.} $\DeuxFoncOpTrancheLax{j}{c} : \OpTrancheLax{\mathdeuxcat{A}}{uj}{c} \to \OpTrancheLax{\mathdeuxcat{B}}{u}{c}$, \emph{resp.} $\DeuxFoncTrancheCoLax{j}{c} : \TrancheCoLax{\mathdeuxcat{A}}{uj}{c} \to \TrancheCoLax{\mathdeuxcat{B}}{u}{c}$, \emph{resp.} $\DeuxFoncOpTrancheCoLax{j}{c} : \OpTrancheCoLax{\mathdeuxcat{A}}{uj}{c} \to \OpTrancheCoLax{\mathdeuxcat{B}}{u}{c}$) est un isomorphisme.
\end{lemme}

\begin{proof}
La vérification est laissée au lecteur. 
\end{proof}

\begin{lemme}\label{LemmeCompoIsoAsph}
Soient $j : \mathdeuxcat{A} \to \mathdeuxcat{B}$ un isomorphisme de \deux{}catégories et $u : \mathdeuxcat{B} \to \mathdeuxcat{C}$ un \DeuxFoncteurLax{} (\emph{resp.} un \DeuxFoncteurLax{}, \emph{resp.} un \DeuxFoncteurCoLax{}, \emph{resp.} un \DeuxFoncteurCoLax{}). Alors, $u$ est lax-asphérique (\emph{resp.} lax-opasphérique, \emph{resp.} colax-asphérique, \emph{resp.} colax-opasphérique) si et seulement si $uj$ l'est.
\end{lemme}

\begin{proof}
C'est une conséquence du lemme \ref{LemmeIsoInduit} et du fait qu'un isomorphisme est une équivalence faible. 
\end{proof}

\begin{prop}\label{CouniteColaxCoaspherique}
Pour toute \deux{}catégorie $\mathdeuxcat{A}$, le \DeuxFoncteurStrict{} $\StrictCanonique{\mathdeuxcat{A}} : \TildeLax{\mathdeuxcat{A}} \to \mathdeuxcat{A}$ est colax-opasphérique.
\end{prop}
 
\begin{proof}
On sait déjà (voir la proposition \ref{CouniteColaxAspherique}) que $\StrictCanonique{(\DeuxCatUnOp{\mathdeuxcat{A}})}$ est colax-asphérique. En vertu de l'égalité $\StrictCanonique{(\DeuxCatUnOp{\mathdeuxcat{A}})} j_{\mathdeuxcat{A}} = \DeuxFoncUnOp{(\StrictCanonique{\mathdeuxcat{A}})}$ (voir la remarque \ref{RemStrictCanoniqueUnOp}) et du lemme \ref{LemmeCompoIsoAsph}, on en déduit que $\DeuxFoncUnOp{(\StrictCanonique{\mathdeuxcat{A}})}$ est colax-asphérique, d'où la conclusion en vertu d'un énoncé dual de celui de la proposition \ref{AAA}.
\end{proof}

\begin{rem}
Plus précisément, pour tout objet $a$ de $\mathdeuxcat{A}$, la \deux{}catégorie $\OpTrancheCoLax{\TildeLax{\mathdeuxcat{A}}}{\StrictCanonique{\mathdeuxcat{A}}}{a}$ op-admet un objet admettant un objet final.
\end{rem}

\begin{prop}\label{CouniteTildeColaxLaxAspheriqueLaxCoaspherique}
Pour toute \deux{}catégorie $\mathdeuxcat{A}$, le \DeuxFoncteurStrict{} $\StrictColaxCanonique{\mathdeuxcat{A}} : \TildeColax{\mathdeuxcat{A}} \to \mathdeuxcat{A}$ est lax-asphérique et lax-opasphérique. 
\end{prop}

\begin{proof}
Le \DeuxFoncteurStrict{} $\DeuxFoncDeuxOp{(\StrictColaxCanonique{\mathdeuxcat{A}})}$ n'est autre que $\StrictCanonique{(\DeuxCatDeuxOp{\mathdeuxcat{A}})}$, qui est colax-asphérique et colax-opasphérique en vertu des propositions \ref{CouniteColaxAspherique} et \ref{CouniteColaxCoaspherique}. La conclusion découle alors d'un énoncé dual de celui de la proposition \ref{AAA}.
\end{proof}
 
\section{Un Théorème A}\label{SectionTheoremeA}

\begin{prop}\label{StrictInduitAspherique}
Soient $u : \mathdeuxcat{A} \to \mathdeuxcat{B}$ un morphisme de $\DeuxCatLax$, $b$ un objet de $\mathdeuxcat{B}$ et $\DeuxFoncTrancheLaxCoq{\StrictCanonique{\mathdeuxcat{A}}}{\sigma}{b} : \TrancheLax{\TildeLax{\mathdeuxcat{A}}}{\BarreLax{u}}{b} \to  \TrancheLax{\mathdeuxcat{A}}{u}{b}$ le \DeuxFoncteurStrict{} induit par le diagramme\footnote{Voir le lemme \ref{LemmeDimitri} pour la définition de $\sigma$.}
$$
\xymatrix{
\TildeLax{\mathdeuxcat{A}}
\ar[dr]^{\BarreLax{u}}
\ar[d]_{\StrictCanonique{\mathdeuxcat{A}}}
\drtwocell<\omit>{<2>\sigma}
\\
\mathdeuxcat{A}
\ar[r]_{u}
&\mathdeuxcat{B}
&.
}
$$
Alors, pour tout objet $(a, p : u(a) \to b)$ de $\TrancheLax{\mathdeuxcat{A}}{u}{b}$, la \deux{}catégorie
$$
\TrancheCoLax{(\TrancheLax{\TildeLax{\mathdeuxcat{A}}}{\BarreLax{u}}{b})}{\DeuxFoncTrancheLaxCoq{\StrictCanonique{\mathdeuxcat{A}}}{\sigma}{b}}{(a,p)}
$$
admet un objet admettant un objet final.  
\end{prop}

\begin{proof}
Le \DeuxFoncteurStrict{} $\DeuxFoncTrancheLaxCoq{\StrictCanonique{\mathdeuxcat{A}}}{\sigma}{b}$ est défini comme suit :
$$
\begin{aligned}
\TrancheLax{\TildeLax{\mathdeuxcat{A}}}{\BarreLax{u}}{b} &\to \TrancheLax{\mathdeuxcat{A}}{u}{b}
\\
(a,p) &\mapsto (a,p)
\\
(([m], x, \alpha) : (a,p) \to (a',p')) &\mapsto (x_{m,m-1} \dots x_{1,0}, (p' \CompDeuxZero u_{x}) \CompDeuxUn \alpha)
\\
(\varphi, (\gamma_{1}, \dots, \gamma_{n})) &\mapsto \gamma_{n} \CompDeuxZero \dots \CompDeuxZero \gamma_{1}.
\end{aligned}
$$ 
Explicitons partiellement la structure de la \deux{}catégorie $
\TrancheCoLax{(\TrancheLax{\TildeLax{\mathdeuxcat{A}}}{\BarreLax{u}}{b})}{\DeuxFoncTrancheLaxCoq{\StrictCanonique{\mathdeuxcat{A}}}{\sigma}{b}}{(a,p)}
$.
\begin{itemize}
\item Les objets en sont les quadruplets $(a', p', f, \alpha)$, où $a'$ est un objet de $\mathdeuxcat{A}$, $p' : u(a') \to b$ une \un{}cellule de $\mathdeuxcat{B}$, $f : a' \to a$ une \un{}cellule de $\mathdeuxcat{A}$ et $\alpha : p' \Rightarrow p u(f)$ une \deux{}cellule de $\mathdeuxcat{B}$.
\item Les \un{}cellules de $(a',p', f, \alpha)$ vers $(a'', p'', f', \alpha')$ sont données par les $([m], x, \beta, \gamma)$, avec $([m],x)$ une \un{}cellule de $a'$ vers $a''$ dans $\TildeLax{\mathdeuxcat{A}}$, $\beta : p' \Rightarrow p'' u(x_{m,m-1}) \dots u(x_{1,0})$ une \deux{}cellule de $\mathdeuxcat{B}$ et $\gamma : f' x_{m,m-1} \dots x_{1,0} \Rightarrow f$ une \deux{}cellule dans $\mathdeuxcat{A}$, quadruplets satisfaisant l'égalité
$$
(p \CompDeuxZero (u (\gamma) \CompDeuxUn \DeuxCellStructComp{u}{f'}{x})) \CompDeuxUn (\alpha' \CompDeuxZero u(x_{m,m-1} \dots x_{1,0})) \CompDeuxUn (p'' \CompDeuxZero u_{x}) \CompDeuxUn \beta = \alpha.
$$
(On a noté $\DeuxCellStructComp{u}{f'}{x}$ la \deux{}cellule structurale de composition $u_{f', x_{m,m-1} \dots x_{1,0}}$ de $u$, de source $u(f') u (x_{m,m-1} \dots x_{1,0})$ et de but $u(f' x_{m,m-1} \dots x_{1,0})$.) 
\item Si $([m], x, \beta, \gamma)$ et $([n], y, \sigma, \tau)$ sont deux \un{}cellules de $(a',p', f, \alpha)$ vers $(a'', p'', f', \alpha')$, les \deux{}cellules de $([m], x, \beta, \gamma)$ vers $([n], y, \sigma, \tau)$ sont données par les \deux{}cellules $(\varphi, \rho)$ de $([m],x)$ vers $([n],y)$ dans $\TildeLax{\mathdeuxcat{A}}$ satisfaisant 
$$
(p'' \CompDeuxZero (  (u(\rho_{n,n-1}) \CompDeuxZero \dots \CompDeuxZero u(\rho_{1,0})) \CompDeuxUn (u_{x_{\varphi(n-1)} \to \dots \to x_{m}} \CompDeuxZero \dots \CompDeuxZero u_{x_{0} \to \dots \to x_{\varphi(1)}})  )   ) \CompDeuxUn \beta = \sigma
$$
et
$$
\tau \CompDeuxUn (f' \CompDeuxZero \rho_{n} \CompDeuxZero \dots \CompDeuxZero \rho_{1}) = \gamma.
$$
\end{itemize}

Dans cette \deux{}catégorie $\TrancheCoLax{(\TrancheLax{\TildeLax{\mathdeuxcat{A}}}{\BarreLax{u}}{b})}{\DeuxFoncTrancheLaxCoq{\StrictCanonique{\mathdeuxcat{A}}}{\sigma}{b}}{(a,p)}$ se distingue l'objet $(a, p, 1_{a}, p \CompDeuxZero \DeuxCellStructId{u}{a})$. 

Soit $(a'', p'' : u(a'') \to b, f' : a'' \to a, \alpha : p'' \Rightarrow p u(f'))$ un objet quelconque de la \deux{}catégorie $\TrancheCoLax{(\TrancheLax{\TildeLax{\mathdeuxcat{A}}}{\BarreLax{u}}{b})}{\DeuxFoncTrancheLaxCoq{\StrictCanonique{\mathdeuxcat{A}}}{\sigma}{b}}{(a,p)}$. Le quadruplet $([1], f', \alpha, 1_{f'})$ définit alors une \un{}cellule de $(a'', p'', f', \alpha)$ vers $(a, p, 1_{a}, p \CompDeuxZero \DeuxCellStructId{u}{a})$. En effet, la condition de commutativité à vérifier se simplifie ici en l'égalité
$$
(p \CompDeuxZero (\DeuxCellStructComp{u}{1_{a}}{f'}  \CompDeuxUn (\DeuxCellStructId{u}{a} \CompDeuxZero u(f')))) \CompDeuxUn \alpha = \alpha,
$$
qui est vérifiée en vertu de l'égalité
$$
\DeuxCellStructComp{u}{1_{a}}{f'}  \CompDeuxUn (  \DeuxCellStructId{u}{a} \CompDeuxZero u(f')  ) = 1_{u(f')}.
$$

Supposons donné une \un{}cellule 
$$
([m], x, \alpha' : p'' \Rightarrow p u(x_{m,m-1}) \dots u(x_{1,0}), \gamma : x_{m,m-1}  \dots x_{1,0} \Rightarrow f')
$$
de $(a'', p'', f', \alpha)$ vers $(a, p, 1_{a}, p \CompDeuxZero \DeuxCellStructId{u}{a})$ dans 
$
\TrancheCoLax{(\TrancheLax{\TildeLax{\mathdeuxcat{A}}}{\BarreLax{u}}{b})}{\DeuxFoncTrancheLaxCoq{\StrictCanonique{\mathdeuxcat{A}}}{\sigma}{b}}{(a,p)}
$. La définition des \un{}cellules de cette catégorie assure l'égalité
$$
(p \CompDeuxZero u(\gamma)) \CompDeuxUn (p \CompDeuxZero \DeuxCellStructComp{u}{1_{a}}{x}) \CompDeuxUn (p \CompDeuxZero \DeuxCellStructId{u}{a} \CompDeuxZero u(x_{m,m-1} \dots x_{1,0})) \CompDeuxUn (p \CompDeuxZero u_{x}) \CompDeuxUn \alpha' = \alpha,
$$
qui se récrit
$$
(p \CompDeuxZero (u(\gamma) \CompDeuxUn    
\DeuxCellStructComp{u}{1_{a}}{x} \CompDeuxUn    
(\DeuxCellStructId{u}{a} \CompDeuxZero u(x_{m,m-1} \dots x_{1,0}))   \CompDeuxUn  
u_{x})) 
\CompDeuxUn \alpha' = \alpha.
$$
La composée
$$
\DeuxCellStructComp{u}{1_{a}}{x} \CompDeuxUn    
(\DeuxCellStructId{u}{a} \CompDeuxZero u(x_{m,m-1} \dots x_{1,0})) 
$$
est une identité. Ainsi, l'égalité 
$$
(p \CompDeuxZero (u(\gamma) \CompDeuxUn u_{x})) \CompDeuxUn \alpha' = \alpha
$$
fait partie des hypothèses. 

Une \deux{}cellule de $([m], x, \alpha', \gamma)$ vers $([1], f', \alpha, 1_{f'})$ dans 
$
\TrancheCoLax{(\TrancheLax{\TildeLax{\mathdeuxcat{A}}}{\BarreLax{u}}{b})}{\DeuxFoncTrancheLaxCoq{\StrictCanonique{\mathdeuxcat{A}}}{\sigma}{b}}{(a,p)}
$
correspond à un couple $(\varphi, \rho)$, où $\varphi : [1] \to [m]$ est un morphisme d'intervalles et $\rho$ une \deux{}cellule de $x_{m,m-1} \dots x_{1,0}$ vers $f'$ dans $\mathdeuxcat{A}$, telles que
$$
(p \CompDeuxZero (u(\rho) \CompDeuxUn u_{x})) \CompDeuxUn \alpha' = \alpha
$$
et
$$
1_{f'} \CompDeuxUn (1_{a} \CompDeuxZero \rho) = \gamma.
$$
La seconde égalité force $\rho = \gamma$ et, par hypothèse, la première égalité est vérifiée si l'on fait $\rho = \gamma$. Cela termine la démonstration.
\end{proof}

\begin{prop}\label{StrictInduitEquiFaible}
Soient $u : \mathdeuxcat{A} \to \mathdeuxcat{B}$ un \DeuxFoncteurLax{} et $b$ un objet de $\mathdeuxcat{B}$. Alors, le \DeuxFoncteurStrict{}
$\DeuxFoncTrancheLaxCoq{\StrictCanonique{\mathdeuxcat{A}}}{\sigma}{b} : \TrancheLax{\TildeLax{\mathdeuxcat{A}}}{\BarreLax{u}}{b} \to  \TrancheLax{\mathdeuxcat{A}}{u}{b}$ (voir l'énoncé de la proposition \ref{StrictInduitAspherique}) est colax-asphérique (donc en particulier une équivalence faible).
\end{prop}

\begin{proof}
C'est une conséquence immédiate de la proposition \ref{StrictInduitAspherique}.
\end{proof}


\begin{prop}\label{LaxInduitEquiFaible}
Soient $u : \mathdeuxcat{A} \to \mathdeuxcat{B}$ un \DeuxFoncteurLax{} et $b$ un objet de $\mathdeuxcat{B}$. Alors, le  \DeuxFoncteurLax{}
$\DeuxFoncTrancheLax{\LaxCanonique{\mathdeuxcat{A}}}{b} : \TrancheLax{\mathdeuxcat{A}}{u}{b} \to \TrancheLax{\TildeLax{\mathdeuxcat{A}}}{\BarreLax{u}}{b}$ induit par le diagramme commutatif
$$
\xymatrix{
\TildeLax{A}
\ar[dr]^{\BarreLax{u}}
\\
\mathdeuxcat{A}
\ar[u]^{\LaxCanonique{\mathdeuxcat{A}}}
\ar[r]_{u}
&\mathdeuxcat{B}
}
$$
est une équivalence faible. 
\end{prop}

\begin{proof}
On remarque que le \DeuxFoncteurStrict{} $\DeuxFoncTrancheLaxCoq{\StrictCanonique{\mathdeuxcat{A}}}{\sigma}{b}$ (explicité par exemple au cours de l'énoncé et de la démonstration de la proposition \ref{StrictInduitAspherique}) est une rétraction de $\DeuxFoncTrancheLax{\LaxCanonique{\mathdeuxcat{A}}}{b}$. En vertu de la proposition \ref{StrictInduitEquiFaible}, de la remarque \ref{EquiStricteEquiLax} et du lemme \ref{SatFaibleEquiLax}, on en déduit  le résultat.
\end{proof}

\begin{rem}
Les propositions \ref{StrictInduitAspherique} et \ref{LaxInduitEquiFaible} généralisent les propositions \ref{CouniteColaxAspherique} et \ref{UniteW} sans que leur démonstration en dépende. On aurait donc pu s'abstenir d'énoncer les cas particuliers. Les calculs s'avérant toutefois sensiblement plus pénibles dans le cas général, on a privilégié la solution consistant à détailler les deux démonstrations.
\end{rem}

\begin{rem}
Soit $u : \mathdeuxcat{A} \to \mathdeuxcat{B}$ un \DeuxFoncteurLax{}. En conservant des notations analogues à celles des propositions \ref{StrictInduitAspherique} et \ref{LaxInduitEquiFaible}, on peut considérer le diagramme d'équivalences faibles
$$
\xymatrix{
\TildeLax{\TrancheLax{\mathdeuxcat{A}}{u}{b}}
\ar @/_1pc/ [rr]_{\StrictCanonique{\TrancheLax{\mathdeuxcat{A}}{u}{b}}}
&&\TrancheLax{\mathdeuxcat{A}}{u}{b}
\ar @/_1pc/ [ll]_{\LaxCanonique{\TrancheLax{\mathdeuxcat{A}}{u}{b}}}
\ar @/^1pc/ [rr]^{\DeuxFoncTrancheLax{\LaxCanonique{\mathdeuxcat{A}}}{b}}
&&\TrancheLax{\TildeLax{\mathdeuxcat{A}}}{\BarreLax{u}}{b}
\ar @/^1pc/ [ll]^{\DeuxFoncTrancheLaxCoq{\StrictCanonique{\mathdeuxcat{A}}}{\sigma}{b}}
}.
$$
Cela fournit un couple $(\DeuxFoncTrancheLax{\LaxCanonique{\mathdeuxcat{A}}}{b} \phantom{a} \StrictCanonique{\TrancheLax{\mathdeuxcat{A}}{u}{b}}, \LaxCanonique{\TrancheLax{\mathdeuxcat{A}}{u}{b}} \phantom{a} \DeuxFoncTrancheLaxCoq{\StrictCanonique{\mathdeuxcat{A}}}{\sigma}{b})$ d'équivalences faibles entre les \deux{}catégories $\TildeLax{\TrancheLax{\mathdeuxcat{A}}{u}{b}}$ et $\TrancheLax{\TildeLax{\mathdeuxcat{A}}}{\BarreLax{u}}{b}$. 
\end{rem}

\begin{lemme}\label{BarreLaxUV}
Soient $u$ et $v$ deux \DeuxFoncteursLax{} tels que la composée $vu$ fasse sens. Alors, $\BarreLax{v} \TildeLax{u} = \BarreLax{vu}$.
\end{lemme}

\begin{proof}
Comme $\BarreLax{v} \TildeLax{u}$ est un \DeuxFoncteurStrict{}, cela résulte de la suite d'égalités $\BarreLax{v} \TildeLax{u} \LaxCanonique{\mathdeuxcat{A}} = \BarreLax{v} \LaxCanonique{\mathdeuxcat{B}} u = vu$.
\end{proof}

Ainsi, à tout diagramme commutatif de \DeuxFoncteursLax{}
$$
\xymatrix{
\mathdeuxcat{A} 
\ar[rr]^{u}
\ar[dr]_{w}
&&\mathdeuxcat{B}
\ar[dl]^{v}
\\
&\mathdeuxcat{C}
}
$$
est associé un diagramme commutatif de \DeuxFoncteursStricts{}
$$
\xymatrix{
\TildeLax{\mathdeuxcat{A}}
\ar[rr]^{\TildeLax{u}}
\ar[dr]_{\BarreLax{w}}
&&\TildeLax{\mathdeuxcat{B}}
\ar[dl]^{\BarreLax{v}}
\\
&\mathdeuxcat{C}
}
$$
et donc, pour tout objet $c$ de $\mathdeuxcat{C}$, un \DeuxFoncteurStrict{} 
$$
\DeuxFoncTrancheLax{\TildeLax{u}}{c} : \TrancheLax{\TildeLax{\mathdeuxcat{A}}}{\BarreLax{w}}{c} \to \TrancheLax{\TildeLax{\mathdeuxcat{B}}}{\BarreLax{v}}{c}.
$$

\begin{lemme}\label{EncoreUnCarre}
Soit
$$
\xymatrix{
\mathdeuxcat{A} 
\ar[rr]^{u}
\ar[dr]_{w}
&&\mathdeuxcat{B}
\ar[dl]^{v}
\\
&\mathdeuxcat{C}
}
$$
un diagramme commutatif de \DeuxFoncteursLax{}. Alors, pour tout objet $c$ de $\mathdeuxcat{C}$, le diagramme
$$
\xymatrix{
\TrancheLax{\TildeLax{\mathdeuxcat{A}}}{\BarreLax{w}}{c}
\ar[r]^{\DeuxFoncTrancheLax{\TildeLax{u}}{c}}
&\TrancheLax{\TildeLax{\mathdeuxcat{B}}}{\BarreLax{v}}{c}
\\
\TrancheLax{\mathdeuxcat{A}}{w}{c}
\ar[u]^{\DeuxFoncTrancheLax{\LaxCanonique{\mathdeuxcat{A}}}{c}}
\ar[r]_{\DeuxFoncTrancheLax{u}{c}}
&\TrancheLax{\mathdeuxcat{B}}{v}{c}
\ar[u]_{\DeuxFoncTrancheLax{\LaxCanonique{\mathdeuxcat{B}}}{c}}
}
$$
est commutatif.
\end{lemme}

\begin{proof}
Cela résulte de calculs ne posant aucune difficulté. On détaillera la démonstration d'un cas plus général de ce résultat — et indépendant de la démonstration de celui-ci — plus loin (voir le lemme \ref{CarreCommutatifThALax}). 
\end{proof}

\begin{theo}\label{TheoremeALaxTrancheLax}
Soit
$$
\xymatrix{
\mathdeuxcat{A} 
\ar[rr]^{u}
\ar[dr]_{w}
&&\mathdeuxcat{B}
\ar[dl]^{v}
\\
&\mathdeuxcat{C}
}
$$
un diagramme commutatif de \DeuxFoncteursLax{}. Supposons que, pour tout objet $c$ de $\mathdeuxcat{C}$, le \DeuxFoncteurLax{} $\DeuxFoncTrancheLax{u}{c} : \TrancheLax{\mathdeuxcat{A}}{w}{c} \to \TrancheLax{\mathdeuxcat{B}}{v}{c}$ soit une équivalence faible. Alors, $u$ est une équivalence faible.
\end{theo}

\begin{proof}
Supposant donné un diagramme commutatif de \DeuxFoncteursLax{} tel que celui de l'énoncé, on a vu qu'on pouvait lui associer, pour tout objet $c$ de $\mathdeuxcat{C}$, un \DeuxFoncteurStrict{} $\DeuxFoncTrancheLax{\TildeLax{u}}{c} : \TrancheLax{\TildeLax{\mathdeuxcat{A}}}{\BarreLax{w}}{c} \to \TrancheLax{\TildeLax{\mathdeuxcat{B}}}{\BarreLax{v}}{c}$ induit par le diagramme commutatif de \DeuxFoncteursStricts{}
$$
\xymatrix{
\TildeLax{\mathdeuxcat{A}}
\ar[rr]^{\TildeLax{u}}
\ar[dr]_{\BarreLax{w}}
&&\TildeLax{\mathdeuxcat{B}}
\ar[dl]^{\BarreLax{v}}
\\
&\mathdeuxcat{C}
&.
}
$$
Vérifions que $\DeuxFoncTrancheLax{u}{c}$ est une équivalence faible si et seulement si $\DeuxFoncTrancheLax{\TildeLax{u}}{c}$ en est une. Pour cela, on considère le diagramme commutatif
$$
\xymatrix{
\TrancheLax{\TildeLax{\mathdeuxcat{A}}}{\BarreLax{w}}{c}
\ar[r]^{\DeuxFoncTrancheLax{\TildeLax{u}}{c}}
&\TrancheLax{\TildeLax{\mathdeuxcat{B}}}{\BarreLax{v}}{c}
\\
\TrancheLax{\mathdeuxcat{A}}{w}{c}
\ar[u]^{\DeuxFoncTrancheLax{\LaxCanonique{\mathdeuxcat{A}}}{c}}
\ar[r]_{\DeuxFoncTrancheLax{u}{c}}
&\TrancheLax{\mathdeuxcat{B}}{v}{c}
\ar[u]_{\DeuxFoncTrancheLax{\LaxCanonique{\mathdeuxcat{B}}}{c}}
}
$$
de l'énoncé du lemme \ref{EncoreUnCarre}. En vertu de la proposition  \ref{LaxInduitEquiFaible}, $\DeuxFoncTrancheLax{\LaxCanonique{\mathdeuxcat{A}}}{c}$ et $\DeuxFoncTrancheLax{\LaxCanonique{\mathdeuxcat{B}}}{c}$ sont des équivalences faibles. On conclut en invoquant la remarque \ref{EquiStricteEquiLax} et le lemme \ref{SatFaibleEquiLax}. 
\end{proof}

\begin{rem}
Dans ses notes \cite{NotesDelHoyo}, qu'il nous avait communiquées lorsque nous cherchions à dégager un analogue du Théorème A de Quillen pour les \DeuxFoncteursLax{}, del Hoyo invoque le fait qu'une équivalence de Dwyer-Kan\footnote{Les équivalences de Dwyer-Kan sont définies de façon standard pour ce que l'on appelle maintenant généralement les « catégories simpliciales », c'est-à-dire les catégories enrichies en ensembles simpliciaux. La définition fait sens dans le cas des \deux{}catégories que l'on considère comme des « catégories simpliciales » en remplaçant les catégories de morphismes par leur nerf.} bijective sur les objets est une équivalence faible pour établir une version absolue du Théorème A de Quillen pour les \DeuxFoncteursLax{} dans le cas du localisateur fondamental de $\DeuxCat$ formé des équivalences de Thomason « classiques » (morphismes dont le nerf est une équivalence faible simpliciale). Nous utilisons pour notre part un critère d'asphéricité et donnons le cas relatif. Nous généralisons plus loin ce résultat au cas d'un diagramme commutatif à \deux{}cellule près seulement. 
\end{rem}

\begin{theo}\label{TheoremeALaxOpTrancheLax}
Soit
$$
\xymatrix{
\mathdeuxcat{A} 
\ar[rr]^{u}
\ar[dr]_{w}
&&\mathdeuxcat{B}
\ar[dl]^{v}
\\
&\mathdeuxcat{C}
}
$$
un diagramme commutatif de \DeuxFoncteursLax{}. Supposons que, pour tout objet $c$ de $\mathdeuxcat{C}$, le \DeuxFoncteurLax{} $\DeuxFoncOpTrancheLax{u}{c} : \OpTrancheLax{\mathdeuxcat{A}}{w}{c} \to \OpTrancheLax{\mathdeuxcat{B}}{v}{c}$ soit une équivalence faible. Alors, $u$ est une équivalence faible.
\end{theo}

\begin{proof}
Par hypothèse, le \DeuxFoncteurLax{} $\DeuxFoncUnOp{(\DeuxFoncTrancheLax{\DeuxFoncUnOp{u}}{c})}$ est une équivalence faible pour tout objet $c$ de $\mathdeuxcat{C}$. En vertu de la proposition \ref{DeuxFoncLaxUnOpW}, $\DeuxFoncTrancheLax{\DeuxFoncUnOp{u}}{c}$ est une équivalence faible pour tout objet $c$ de $\mathdeuxcat{C}$. En vertu du théorème \ref{TheoremeALaxTrancheLax}, $\DeuxFoncUnOp{u}$ est une équivalence faible. Une nouvelle invocation de la proposition \ref{DeuxFoncLaxUnOpW} permet de conclure. 
\end{proof}

\section{$\rm{Ho}(\Cat) \simeq \rm{Ho}(\DeuxCat)$}\label{SectionHoCatHoDeuxCat}

\begin{paragr}
L'objet de cette section est d'établir l'équivalence des catégories homotopiques de $\Cat$ et $\DeuxCat$ relativement aux équivalences faibles que nous considérons. Une approche naïve de la question, motivée par la considération des phénomènes connus en dimension $1$, consisterait à tenter de vérifier que, pour toute petite \deux{}catégorie $\mathdeuxcat{A}$, il existe une équivalence faible, pour un certain foncteur nerf $N$, donnée par un foncteur du type
$$
\DeuxIntOp{\Delta} N \mathdeuxcat{A} \to \mathdeuxcat{A},
$$
qui jouerait un rôle analogue à celui du foncteur 
$$
sup_{A} : \DeuxIntOp{\Delta} \UnNerf A \to A
$$
défini pour toute petite catégorie $A$ (voir le paragraphe \ref{DefUnSup} et la proposition \ref{UnSupAspherique}). On a vu dans la section \ref{SectionSup} qu'il était possible de définir fonctoriellement de tels morphismes, mais il s'agit de \DeuxFoncteursLax{}, ce qui nous fait sortir de la catégorie $\DeuxCat$. On pourrait être tenté d'appliquer la variante du Théorème A établie pour les morphismes lax (théorème \ref{TheoremeALaxTrancheLax}) afin de montrer que ces \DeuxFoncteursLax{} sont des équivalences faibles, puis utiliser l'équivalence déjà démontrée entre les catégories homotopiques de $\DeuxCat$ et $\DeuxCatLax$ (voir le théorème \ref{EqCatLocDeuxCatDeuxCatLax}). Bien que cela ne soit pas la façon dont nous procéderons dans cette section, nous verrons plus loin qu'il est bien possible de parvenir à démontrer très simplement le résultat souhaité par un raisonnement de ce genre. La démonstration que nous présentons dans cette section utilise l'existence d'un remplacement fonctoriel de toute petite \deux{}catégorie par une petite catégorie qui lui est reliée par un zigzag de longueur $2$ d'équivalences faibles de $\DeuxCat$. La proposition \ref{ObstructionGeorges} montre qu'il serait illusoire d'espérer mieux. 
\end{paragr}

\begin{paragr}
On notera $\DeuxLocFond{W}_{2}^{2}$ la classe des morphismes de $\DeuxCat$ dont la réalisation topologique du nerf induit une bijection entre les $\pi_{0}$ et des isomorphismes entre les $\pi_{1}$ et $\pi_{2}$ respectifs pour tout choix de point base.
\end{paragr}

\begin{prop}\label{ObstructionGeorges}
Soient $G$ un groupe abélien et $\mathdeuxcat{G}$ la \deux{}catégorie ayant un seul objet, une seule \un{}cellule et dont le monoïde des \deux{}cellules est donné par le monoïde sous-jacent au groupe $\mathdeuxcat{G}$. Supposons qu'\emph{au moins une} des deux propriétés suivantes soit vérifiée.
\begin{itemize}
\item[(i)]
Il existe un élément de $\DeuxLocFond{W}_{2}^{2}$ de source une catégorie et de but $\mathdeuxcat{G}$.
\item[(ii)]
Il existe un élément de $\DeuxLocFond{W}_{2}^{2}$ de source  $\mathdeuxcat{G}$ et de but une catégorie.
\end{itemize}
Alors $G$ est le groupe trivial.
\end{prop}

\begin{proof}
Supposons qu'il existe une petite catégorie $A$ et un morphisme $s : A \to \mathdeuxcat{G}$ qui soit dans $\DeuxLocFond{W}_{2}^{2}$. Le morphisme $s$ se factorise alors par $\TronqueBete{\mathdeuxcat{G}} = \DeuxCatPonct$ (voir la section \ref{SectionAdjonctionsCatDeuxCat}) et il existe ainsi un diagramme commutatif
$$
\xymatrix{
A
\ar[rr]^{}
\ar[dr]_{s}
&&\DeuxCatPonct
\ar[dl]^{}
\\
&\mathdeuxcat{G}
&.
}
$$
En vertu du lemme \ref{EquivalenceViaPoint}, le morphisme $\DeuxCatPonct \to \mathdeuxcat{G}$ est dans $\DeuxLocFond{W}_{2}^{2}$. En particulier, il induit un isomorphisme entre $\pi_{2}(\DeuxCatPonct) = \{*\}$ et $\pi_{2}(\mathdeuxcat{G}) = G$ (on ne note pas le point base puisqu'il est unique), d'où le résultat.

Supposons, de façon analogue, qu'il existe une petite catégorie $A$ et un morphisme $s : \mathdeuxcat{G} \to A$ qui soit dans $\DeuxLocFond{W}_{2}^{2}$. Le morphisme $s$ se factorise alors par $\TronqueInt{\mathdeuxcat{G}} = \DeuxCatPonct$ (voir la section \ref{SectionAdjonctionsCatDeuxCat}) et il existe ainsi un diagramme commutatif
$$
\xymatrix{
\mathdeuxcat{G}
\ar[rr]^{}
\ar[dr]_{s}
&&\DeuxCatPonct
\ar[dl]^{}
\\
&A
&.
}
$$
En vertu du lemme \ref{EquivalenceViaPoint}, le morphisme $\mathdeuxcat{G} \to \DeuxCatPonct$ est dans $\DeuxLocFond{W}_{2}^{2}$. En particulier, il induit un isomorphisme entre $\pi_{2}(\mathdeuxcat{G}) = G$ et $\pi_{2}(\DeuxCatPonct) = \{*\}$ (on ne note pas le point base puisqu'il est unique), d'où le résultat.
\end{proof}

La proposition \ref{SupHomAspherique} provient de \cite{TheseDelHoyo} (résultat repris dans \cite{ArticleDelHoyo}), aux contextes et notations près. On comparera avec la proposition \ref{UnSupAspherique}. On suppose toujours fixé un localisateur fondamental $\DeuxLocFond{W}$ de $\DeuxCat$.

\begin{prop}[Del Hoyo]\label{SupHomAspherique}
Pour toute \deux{}catégorie $\mathdeuxcat{A}$, le \DeuxFoncteurLax{} normalisé
$$
\SupHom_{\mathdeuxcat{A}} : \DeuxIntOp{\Delta} \NerfHom \mathdeuxcat{A} \to \mathdeuxcat{A}
$$
est lax\nobreakdash-asphérique (donc en particulier une équivalence faible). 
\end{prop}

\begin{proof}
Soit $a$ un objet quelconque de $\mathdeuxcat{A}$. La catégorie $\TrancheLax{(\DeuxIntOp{\Delta} \NerfHom \mathdeuxcat{A})}{\SupHom_{\mathdeuxcat{A}}}{a}$ se décrit comme suit. Ses objets sont donnés par les $([m], x, p : x_{m} \to a)$, en conservant des notations déjà utilisées. Les morphismes de $([m], x, p)$ vers $([n], y, q)$ sont donnés par les 
$$
(\varphi : [m] \to [n], \alpha, \gamma : p \Rightarrow q y_{n,n-1} \dots y_{\varphi(m)+1, \varphi(m)}),
$$
dans lesquels $\varphi$ est un morphisme de $\Delta$, $\alpha$ une transformation relative aux objets de $x$ vers $y \varphi$ et $\gamma$ une \deux{}cellule de $\mathdeuxcat{A}$. Cette catégorie admet un endofoncteur constant 
$$
\begin{aligned}
C_{a} : \TrancheLax{\left(\DeuxIntOp{\Delta} \NerfHom \mathdeuxcat{A}\right)}{\SupHom_{\mathdeuxcat{A}}}{a} &\to \TrancheLax{\left(\DeuxIntOp{\Delta} \NerfHom \mathdeuxcat{A}\right)}{\SupHom_{\mathdeuxcat{A}}}{a}
\\
([m], x, p) &\mapsto ([0], a, 1_{a})
\\
(\varphi, \alpha, \gamma) &\mapsto (1_{[0]}, 1_{a}, 1_{1_{a}})
\end{aligned}
$$
(on a noté $1_{a}$ l'unique transformation relative aux objets de $a$ vers $a$, c'est-à-dire du \DeuxFoncteurStrict{} $[0] \to \mathdeuxcat{A}$ de valeur $a$ vers lui-même)
ainsi qu'un endofoncteur
$$
\begin{aligned}
D : \TrancheLax{\left(\DeuxIntOp{\Delta} \NerfHom \mathdeuxcat{A}\right)}{\SupHom_{\mathdeuxcat{A}}}{a} &\to \TrancheLax{\left(\DeuxIntOp{\Delta} \NerfHom \mathdeuxcat{A}\right)}{\SupHom_{\mathdeuxcat{A}}}{a}
\\
([m], x, p) &\mapsto ([m+1], D_{p}(x), 1_{a})
\\
(\varphi, \alpha, \gamma) &\mapsto (D(\varphi), \alpha, \gamma, 1_{1_{a}}).
\end{aligned}
$$
Dans cette expression, $D_{p}(x)$ désigne le \DeuxFoncteurStrict{} de $[m+1]$ vers $\mathdeuxcat{A}$ défini par $(D_{p}(x))_{i,i-1} = x_{i,i-1}$ si $1 \leq i \leq m$ et $(D_{p}(x))_{m+1,m} = p$ et $D(\varphi) : [m+1] \to [n+1]$ est défini par $D(\varphi)(i) = \varphi(i)$ si $0 \leq i \leq m$ et $D(\varphi)(m+1) = n+1$. De plus, on a noté « $\alpha, \gamma$ » la transformation relative aux objets correspondant à $\alpha$ « suivie de » $\gamma$. 

Vérifions la fonctorialité de $D$. Pour tout objet $([m], x, p)$ de la catégorie $\TrancheLax{\left(\DeuxIntOp{\Delta} \NerfHom \mathdeuxcat{A}\right)}{\SupHom_{\mathdeuxcat{A}}}{a}$, 
$$
\begin{aligned}
D(1_{([m], x, p)}) &= D(1_{[m]}, 1_{x}, 1_{p})
\\
&= (1_{[m+1]}, 1_{D_{p}(x)}, 1_{1_{a}})
\end{aligned}
$$
et
$$
\begin{aligned}
1_{D([m], x, p)} &= 1_{([m+1], D_{p}(x), 1_{a})}
\\
&= (1_{[m+1]}, 1_{D_{p}(x)}, 1_{1_{a}}),
\end{aligned}
$$
donc $D(1_{([m], x, p)}) = 1_{D([m], x, p)}$. Soient maintenant $(\varphi, \alpha, \gamma) : ([m], x, p) \to ([n], y, q)$ et $(\psi, \beta, \delta) : ([n], y, q) \to ([r], z, s)$ deux morphismes composables de $\TrancheLax{\left(\DeuxIntOp{\Delta} \NerfHom \mathdeuxcat{A}\right)}{\SupHom_{\mathdeuxcat{A}}}{a}$. Notons $\alpha_{i',i}$ la « composante de $\alpha$ en le morphisme $i \to i'$ » pour $0 \leq i \leq i' \leq m$ ; c'est une \deux{}cellule de $x_{i',i}$ vers $y_{\varphi(i'), \varphi(i)}$ dans $\mathdeuxcat{A}$. De façon analogue, on notera $\beta_{j',j}$ la composante de $\beta$ en le morphisme $j \to j'$ pour $0 \leq j \leq j' \leq n$. Rappelons qu'en vertu des axiomes des transformations (qui se simplifient dans le cas, qui nous occupe ici, de transformations relatives aux objets entre morphismes stricts), les composantes $\alpha_{1, 0}$, \dots, $\alpha_{m, m-1}$ (\emph{resp.} $\beta_{1, 0}$, \dots, $\beta_{n, n-1}$) déterminent entièrement $\alpha$ (\emph{resp. }$\beta$). Pour $0 \leq i \leq m-1$, posons 
$$
\rho_{i+1,i} = (\beta_{\varphi(i), \varphi(i-1)+1}) \CompDeuxUn \alpha_{i}
$$
si $\varphi(i-1) \leq \varphi(i)-1$ et
$$
\rho_{i+1,i} = \alpha_{i}
$$
si $\varphi(i-1) = \varphi(i)$. Cela définit une transformation relative aux objets de $x$ vers $z \psi \varphi$. On pose également
$$
\tau = \beta_{n, \varphi(m)+1}
$$
si $\varphi(m) \leq n-1$
et
$$
\tau = 1_{1_{y_{n}}}
$$
si $\varphi(m) = n$. Alors, 
$$
(\psi, \beta, \delta) (\varphi, \alpha, \gamma) = (\psi \varphi, \rho, (\delta \CompDeuxZero \tau) \CompDeuxUn \gamma).
$$
Par conséquent, 
$$
D((\psi, \beta, \delta) (\varphi, \alpha, \gamma)) = (D(\psi \varphi), \rho, (\delta \CompDeuxZero \tau) \CompDeuxUn \gamma, 1_{1_{a}}).
$$
D'autre part, 
$$
\begin{aligned}
D(\psi, \beta, \delta) D(\varphi, \alpha, \gamma) &= (D(\psi), \beta, \delta, 1_{1_{a}}) (D(\varphi), \alpha, \gamma, 1_{1_{a}})  
\\
&= (D(\psi \varphi), \rho, (\delta \CompDeuxZero \tau) \CompDeuxUn \gamma, 1_{1_{a}}),
\end{aligned}
$$
la seconde égalité se déduisant de la formule générale ci-dessus de la composition des morphismes dans $\TrancheLax{\left(\DeuxIntOp{\Delta} \NerfHom \mathdeuxcat{A}\right)}{\SupHom_{\mathdeuxcat{A}}}{a}$. Cela conclut les vérifications des propriétés de fonctorialité de $D$. 

Les inclusions
$$
\begin{aligned}
{}[m] &\to [m+1]
\\
i &\mapsto i
\end{aligned}
$$
et
$$
\begin{aligned}
{}[0] &\to [m+1]
\\
0 &\mapsto m+1
\end{aligned}
$$
permettent alors de définir des transformations naturelles $\sigma : 1_{\TrancheLax{(\DeuxIntOp{\Delta} \NerfHom \mathdeuxcat{A})}{\SupHom_{\mathdeuxcat{A}}}{a}} \Rightarrow D$ et \mbox{$\mu : C_{a} \Rightarrow D$} respectivement. Vérifions-le. 

Pour tout objet $([m], x, p)$ de $\TrancheLax{\left(\DeuxIntOp{\Delta} \NerfHom \mathdeuxcat{A}\right)}{\SupHom_{\mathdeuxcat{A}}}{a}$, on définit
$$
\sigma_{([m], x, p)} : ([m], x, p) \to ([m+1], D_{p}(x), 1_{a})
$$
par
$$
\sigma_{([m], x, p)} = ([m] \to [m+1], i \mapsto i, 1_{x}, 1_{p}).
$$
Soit $(\varphi, \alpha, \gamma) : ([m], x, p) \to ([n], y, q)$ un morphisme de $\TrancheLax{(\DeuxIntOp{\Delta} \NerfHom \mathdeuxcat{A})}{\SupHom_{\mathdeuxcat{A}}}{a}$. Il s'agit de vérifier la commutativité du diagramme
$$
\xymatrix{
([m], x, p)
\ar[rrr]^{\sigma_{([m], x, p)}}
\ar[d]_{(\varphi, \alpha, \gamma)}
&&&
([m+1], D_{p}(x), 1_{a})
\ar[d]^{(D(\varphi), \alpha, \gamma, 1_{1_{a}})}
\\
([n], y, q)
\ar[rrr]_{\sigma_{([n], y, q)}}
&&&
([n+1], D_{q}(y), 1_{a})
&.
}
$$
Pour calculer les composées, on peut se référer aux formules générales ci-dessus. Elles permettent d'écrire : 
$$
\begin{aligned}
(D(\varphi), \alpha, \gamma, 1_{1_{a}}) \sigma_{([m], x, p)} &= (D(\varphi), \alpha, \gamma, 1_{1_{a}}) ([m] \to [m+1], i \mapsto i, 1_{x}, 1_{p})
\\
&= ([m] \to [n+1], i \mapsto \varphi(i), \alpha, \gamma)
\end{aligned}
$$
et
$$
\begin{aligned}
\sigma_{([n], y, q)} (\varphi, \alpha, \gamma) &= ([n] \to [n+1], j \mapsto j, 1_{y}, 1_{q}) (\varphi, \alpha, \gamma)
\\
&= ([m] \to [n+1], i \mapsto \varphi(i), \alpha, \gamma).
\end{aligned}
$$
Cela démontre la commutativité du diagramme ci-dessus, donc la naturalité de $\sigma$. 

Pour tout objet $([m], x, p)$ de $\TrancheLax{\left(\DeuxIntOp{\Delta} \NerfHom \mathdeuxcat{A}\right)}{\SupHom_{\mathdeuxcat{A}}}{a}$, on définit
$$
\mu_{([m], x, p)} : ([0], a, 1_{a}) \to ([m+1], D_{p}(x), 1_{a})
$$
par
$$
\mu_{([m], x, p)} = ([0] \to [m+1], 0 \mapsto m+1, 1_{a}, 1_{1_{a}}).
$$
Soit $(\varphi, \alpha, \gamma) : ([m], x, p) \to ([n], y, q)$ un morphisme de $\TrancheLax{(\DeuxIntOp{\Delta} \NerfHom \mathdeuxcat{A})}{\SupHom_{\mathdeuxcat{A}}}{a}$. Il s'agit de vérifier la commutativité du diagramme
$$
\xymatrix{
([0], a, 1_{a})
\ar[rrr]^{\mu_{([m], x, p)}}
\ar[d]_{(1_{[0]}, 1_{a}, 1_{1_{a}})}
&&&
([m+1], D_{p}(x), 1_{a})
\ar[d]^{(D(\varphi), \alpha, \gamma, 1_{1_{a}})}
\\
([0], a, 1_{a})
\ar[rrr]_{\mu_{([n], y, q)}}
&&&
([n+1], D_{q}(y), 1_{a})
&.
}
$$
Or, 
$$
\begin{aligned}
(D(\varphi), \alpha, \gamma, 1_{1_{a}}) \mu_{([m], x, p)} &= (D(\varphi), \alpha, \gamma, 1_{1_{a}}) ([0] \to [m+1], 0 \mapsto m+1, 1_{a}, 1_{1_{a}})
\\
&= ([0] \to [n+1], 0 \mapsto n+1, 1_{a}, 1_{1_{a}})
\end{aligned}
$$
et
$$
\begin{aligned}
\mu_{([n], y, q)} (1_{[0]}, 1_{a}, 1_{1_{a}}) &= ([0] \to [n+1], 0 \mapsto n+1, 1_{a}, 1_{1_{a}}) (1_{[0]}, 1_{a}, 1_{1_{a}})
\\
&= ([0] \to [n+1], 0 \mapsto n+1, 1_{a}, 1_{1_{a}}).
\end{aligned}
$$
Cela démontre la commutativité du diagramme ci-dessus, donc la naturalité de $\mu$. 

Il existe donc un diagramme 
$$
\xymatrix{
1_{\TrancheLax{(\DeuxIntOp{\Delta} \NerfHom \mathdeuxcat{A})}{\SupHom_{\mathdeuxcat{A}}}{a}}
\ar@{=>}[rr]^{\sigma}
&&
D
&&
C_{a}
\ar@{=>}[ll]_{\mu}
}
$$
de morphismes d'endofoncteurs de $\TrancheLax{(\DeuxIntOp{\Delta} \NerfHom \mathdeuxcat{A})}{\SupHom_{\mathdeuxcat{A}}}{a}$.

L'endofoncteur constant $C_{a}$ de $\TrancheLax{\left(\DeuxIntOp{\Delta} \NerfHom \mathdeuxcat{A}\right)}{\SupHom_{\mathdeuxcat{A}}}{a}$ est donc une équivalence faible. La catégorie $\TrancheLax{\left(\DeuxIntOp{\Delta} \NerfHom \mathdeuxcat{A}\right)}{\SupHom_{\mathdeuxcat{A}}}{a}$ est donc asphérique, en vertu du lemme \ref{EquivalenceViaPoint}. 
\end{proof}

\begin{theo}\label{EqCatLocCatDeuxCat}
L'inclusion $\Cat \hookrightarrow \DeuxCat$ et le foncteur $\DeuxIntOp{\Delta} \NerfHom : \DeuxCat \to \Cat$ induisent des équivalences de catégories
$$
\Localisation{\Cat}{(\DeuxLocFond{W} \cap \UnCell{\Cat})} \to \Localisation{\DeuxCat}{\DeuxLocFond{W}}
$$
et 
$$
\Localisation{\DeuxCat}{\DeuxLocFond{W}} \to \Localisation{\Cat}{(\DeuxLocFond{W} \cap \UnCell{\Cat})}
$$
quasi-inverses l'une de l'autre. 
\end{theo}

\begin{proof}
L'inclusion $\Cat \hookrightarrow \DeuxCat$ envoyant tautologiquement les éléments de $\DeuxLocFond{W} \cap \UnCell{\Cat}$ sur des éléments de $\DeuxLocFond{W}$, elle induit un foncteur entre les catégories localisées de l'énoncé. De plus, pour tout morphisme $u : A \to B$ de $\Cat$, le foncteur $\DeuxIntOp{\Delta} \NerfHom (u)$ n'est autre que $\Delta/\UnNerf{(u)}$, et l'on sait déjà, en vertu de la proposition \ref{UnSupAspherique}, que les flèches horizontales du diagramme commutatif
$$
\xymatrix{
\Delta/\UnNerf{A}
\ar[d]_{\Delta/\UnNerf{(u)}}
\ar[rr]^{\SupUn_{A}}
&&A
\ar[d]^{u}
\\
\Delta/\UnNerf{B}
\ar[rr]_{\SupUn_{B}}
&&B
}
$$
sont dans $\DeuxLocFond{W} \cap \UnCell{\Cat}$ puisque la classe $\DeuxLocFond{W} \cap \UnCell{\Cat}$ forme un localisateur fondamental de $\Cat$.

Pour établir le résultat, il suffit donc de montrer que le foncteur $\DeuxIntOp{\Delta} \NerfHom$ envoie les éléments de $\DeuxLocFond{W}$ sur des éléments de $\DeuxLocFond{W} \cap \UnCell{\Cat}$ et qu'il existe un isomorphisme de foncteurs entre l'identité de $\Localisation{\DeuxCat}{\DeuxLocFond{W}}$ et l'endofoncteur de cette catégorie induit par la composée du foncteur $\DeuxIntOp{\Delta} \NerfHom : \DeuxCat \to \Cat$ et de l'inclusion $\Cat \hookrightarrow \DeuxCat$. 

Il suffit donc de montrer que, pour tout morphisme $u : \mathdeuxcat{A} \to \mathdeuxcat{B}$ de $\DeuxCat$, le diagramme
$$
\xymatrix{
\DeuxIntOp{\Delta} \NerfHom \mathdeuxcat{A}
\ar[d]_{\DeuxIntOp{\Delta} \NerfHom (u)}
&&\TildeLax{\DeuxIntOp{\Delta} \NerfHom \mathdeuxcat{A}}
\ar[ll]_{\StrictCanonique{\DeuxIntOp{\Delta} \NerfHom \mathdeuxcat{A}}}
\ar[rr]^{\BarreLax{\SupHom_{\mathdeuxcat{A}}}}
\ar[d]^{\TildeLax{\DeuxIntOp{\Delta} \NerfHom (u)}}
&&\mathdeuxcat{A}
\ar[d]^{u}
\\
\DeuxIntOp{\Delta} \NerfHom \mathdeuxcat{B}
&&\TildeLax{\DeuxIntOp{\Delta} \NerfHom \mathdeuxcat{B}}
\ar[ll]^{\StrictCanonique{\DeuxIntOp{\Delta} \NerfHom \mathdeuxcat{B}}}
\ar[rr]_{\BarreLax{\SupHom_{\mathdeuxcat{B}}}}
&&\mathdeuxcat{B}
}
$$
est commutatif et que les flèches horizontales y figurant sont dans $\DeuxLocFond{W}$. Cette seconde assertion résulte directement de la proposition \ref{CouniteColaxAspherique} et de la proposition \ref{SupHomAspherique}. Pour vérifier la première, on remarque d'abord que le rectangle de gauche est commutatif par naturalité de $\TransStrictCanonique$. Pour établir la commutativité du rectangle de droite, il suffit de vérifier l'égalité 
$$
u \phantom{a} \BarreLax{\SupHom_{\mathdeuxcat{A}}} \phantom{a} \LaxCanonique{\DeuxIntOp{\Delta} \NerfHom \mathdeuxcat{A}} = \BarreLax{\SupHom_{\mathdeuxcat{B}}} \phantom{a} \TildeLax{\DeuxIntOp{\Delta} \NerfHom (u)} \phantom{a} \LaxCanonique{\DeuxIntOp{\Delta} \NerfHom \mathdeuxcat{A}}
$$
ce qui se récrit 
$$
u \phantom{a} \SupHom_{\mathdeuxcat{A}} =  \SupHom_{\mathdeuxcat{B}} \phantom{a} \DeuxIntOp{\Delta} \NerfHom (u).
$$
Cette dernière égalité se vérifie directement.
\end{proof}

\section{Localisateurs fondamentaux de $\DeuxCatLax$}\label{SectionDeuxLocFondLax}

\emph{On ne suppose plus fixé de \ClasseDeuxLocFond{}.}

\begin{df}\label{DefDeuxLocFondLax}
On dira qu'une classe $\DeuxLocFondLax{W}$ de morphismes de $\DeuxCatLax$ est un \emph{\ClasseDeuxLocFondLax{}}\index{localisateur fondamental de $\DeuxCatLax$} si les conditions suivantes sont vérifiées.
\begin{itemize}
\item[$\rm{LF1}_{lax}$] La classe $\DeuxLocFondLax{W}$ est faiblement saturée. 

\item[$\rm{LF2}_{lax}$] Si une petite \deux{}catégorie $\mathdeuxcat{A}$ admet un objet admettant un objet initial, alors le morphisme canonique $\mathdeuxcat{A} \to e$ est dans $\DeuxLocFondLax{W}$.

\item[$\rm{LF3}_{lax}$] Si 
$$
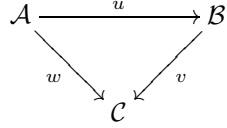
\xymatrix{
\mathdeuxcat{A} 
\ar[rr]^{u}
\ar[dr]_{w}
&&\mathdeuxcat{B}{}
\ar[dl]^{v}
\\
&\mathdeuxcat{C}
}
$$
désigne un triangle commutatif dans $\DeuxCatLax$ et si, pour tout objet $c$ de $\mathdeuxcat{C}$, le \DeuxFoncteurLax{} 
$
\DeuxFoncTrancheLax{u}{c} : \TrancheLax{\mathdeuxcat{A}}{w}{c} \to \TrancheLax{\mathdeuxcat{B}}{v}{c} 
$
est dans $\DeuxLocFondLax{W}$, alors $u$ est dans $\DeuxLocFondLax{W}$.
\end{itemize}
\end{df}

\begin{rem}\label{RemDeuxLocFondLaxInterCat}
Pour tout localisateur fondamental $\DeuxLocFondLax{W}$ de $\DeuxCatLax$, $\DeuxLocFondLax{W} \cap \UnCell{\DeuxCat}$ est un \ClasseDeuxLocFond{}, en vertu de l'implication $(ii) \Rightarrow (i)$ du théorème \ref{DeuxLocFondHuitDef}, et $\DeuxLocFondLax{W} \cap \UnCell{\Cat}$ est un \ClasseUnLocFond{}. 
\end{rem}

\begin{prop}\label{DeuxLocFondInduitLax}
Si $\DeuxLocFond{W}$ est un \ClasseDeuxLocFond{}, alors $\DeuxLocFond{W}_{lax}$ est un \ClasseDeuxLocFondLax{}{}.
\end{prop}

\begin{proof}
En vertu du lemme \ref{SatFaibleEquiLax}, la classe $\DeuxLocFond{W}_{lax}$ vérifie la condition $\rm{LF1}_{lax}$.

Montrons que $\DeuxLocFond{W}_{lax}$ vérifie $\rm{LF2}_{lax}$. Soit $\mathdeuxcat{A}$ une petite \deux{}catégorie admettant un objet admettant un objet initial. Il s'agit de montrer que le morphisme canonique $\mathdeuxcat{A} \to e$ est dans $\DeuxLocFond{W}_{lax}$. Comme c'est un \DeuxFoncteurStrict{}, il est dans $\DeuxLocFond{W}$ si et seulement s'il est dans $\DeuxLocFond{W}_{lax}$ en vertu du lemme \ref{BarreTildeW}. Comme il est dans $\DeuxLocFond{W}$ en vertu du corollaire \ref{OFAspherique3}, il est dans $\DeuxLocFond{W}_{lax}$.

La propriété $\rm{LF3}_{lax}$ résulte du théorème \ref{TheoremeALaxTrancheLax}.
\end{proof}

\begin{lemme}\label{Baba}
Pour tout localisateur fondamental $\DeuxLocFondLax{W}$ de $\DeuxCatLax$, pour toute petite \deux{}catégorie $\mathdeuxcat{A}$, le \DeuxFoncteurStrict{} $\StrictCanonique{\mathdeuxcat{A}} : \TildeLax{\mathdeuxcat{A}} \to \mathdeuxcat{A}$ et le \DeuxFoncteurLax{} $\LaxCanonique{\mathdeuxcat{A}} : \mathdeuxcat{A} \to \TildeLax{\mathdeuxcat{A}}$ sont dans $\DeuxLocFondLax{W}$.
\end{lemme}

\begin{proof}
Comme $\DeuxLocFondLax{W} \cap \UnCell{\DeuxCat}$ est un \ClasseDeuxLocFond{}, $\StrictCanonique{\mathdeuxcat{A}}$ est dans $\DeuxLocFondLax{W} \cap \UnCell{\DeuxCat}$ en vertu de la proposition \ref{CouniteColaxAspherique}. En particulier, $\StrictCanonique{\mathdeuxcat{A}}$ est dans $\DeuxLocFondLax{W}$. On en déduit que sa section $\LaxCanonique{\mathdeuxcat{A}}$ est dans $\DeuxLocFondLax{W}$ en vertu de la saturation faible de $\DeuxLocFondLax{W}$. 
\end{proof}

\begin{lemme}\label{Bibi}
Soit $\DeuxLocFondLax{W}$ un \ClasseDeuxLocFondLax{}. Un \DeuxFoncteurLax{} $u : \mathdeuxcat{A} \to \mathdeuxcat{B}$ est dans $\DeuxLocFondLax{W}$ si et seulement si $\BarreLax{u}$ l'est et si et seulement si $\TildeLax{u}$ l'est. 
\end{lemme}

\begin{proof}
Cela résulte des égalités $\BarreLax{u} \LaxCanonique{\mathdeuxcat{A}} = u$ et $\LaxCanonique{\mathdeuxcat{B}} u = \TildeLax{u} \LaxCanonique{\mathdeuxcat{A}}$, du lemme \ref{Baba} et de la saturation faible de $\DeuxLocFondLax{W}$. 
\end{proof}

\begin{df}
Pour toute classe $S$ de \DeuxFoncteursStricts{}, on notera $S_{lax}$ la classe des \DeuxFoncteursLax{} dont l'image par le foncteur de strictification de Bénabou est dans $S$. En formule :
$$
S_{lax} = \FoncBenabou^{-1}(S).
$$
\end{df}



\begin{theo}\label{IsoDeuxLocFondDeuxLocFondLax}
Les applications 
$$
\begin{aligned}
\mathcal{P} (\UnCell{\DeuxCat}) &\to \mathcal{P} (\UnCell{\DeuxCatLax})
\\
\DeuxLocFond{W} &\mapsto \DeuxLocFond{W}_{lax}
\end{aligned}
$$
et
$$
\begin{aligned}
\mathcal{P} (\UnCell{\DeuxCatLax}) &\to \mathcal{P} (\UnCell{\DeuxCat})
\\
\DeuxLocFond{W} &\mapsto \DeuxLocFond{W} \cap \UnCell{\DeuxCat}
\end{aligned}
$$
induisent des isomorphismes inverses l'un de l'autre entre la classe des localisateurs fondamentaux de $\DeuxCat$ ordonnée par l'inclusion et la classe des localisateurs fondamentaux de $\DeuxCatLax$ ordonnée par l'inclusion. 
%
\end{theo}

\begin{proof}
Ces applications respectent manifestement la relation d'inclusion. Il s'agit de vérifier que, pour tout localisateur fondamental $\DeuxLocFond{W}$ de $\DeuxCat$, on a l'égalité $\DeuxLocFond{W}_{lax} \cap \UnCell{\DeuxCat} = \DeuxLocFond{W}$ et que, pour tout localisateur fondamental $\DeuxLocFondLax{W}$ de $\DeuxCatLax$, on a l'égalité $\DeuxLocFondLax{W} = (\DeuxLocFondLax{W} \cap \UnCell{\DeuxCat})_{lax}$. 

Soit donc $\DeuxLocFond{W}$ un \ClasseDeuxLocFond{}. L'égalité $\DeuxLocFond{W}_{lax} \cap \UnCell{\DeuxCat} = \DeuxLocFond{W}$ résulte directement de la remarque \ref{EquiStricteEquiLax}. 

Soit maintenant $\DeuxLocFondLax{W}$ un \ClasseDeuxLocFondLax{}. Un \DeuxFoncteurLax{} $u$ est dans $(\DeuxLocFondLax{W} \cap \UnCell{\DeuxCat})_{lax}$ si et seulement si $\TildeLax{u}$ est dans $\DeuxLocFondLax{W} \cap \UnCell{\DeuxCat}$, donc si et seulement si $\TildeLax{u}$ est dans $\DeuxLocFondLax{W}$, donc si et seulement si $u$ est dans $\DeuxLocFondLax{W}$ (lemme \ref{Bibi}). Cela montre l'égalité $\DeuxLocFondLax{W} = (\DeuxLocFondLax{W} \cap \UnCell{\DeuxCat})_{lax}$. 
%
\end{proof}


\begin{lemme}\label{DeuxFoncUnOpWLax}
Soit $\DeuxLocFondLax{W}$ un \ClasseDeuxLocFondLax{}. Un \DeuxFoncteurLax{} $u$ est dans $\DeuxLocFondLax{W}$ si et seulement si le \DeuxFoncteurLax{} $\DeuxFoncUnOp{u}$ l'est. 
\end{lemme}

\begin{proof}
Cela résulte du fait que $\DeuxLocFondLax{W} \cap \UnCell{\DeuxCat}$ est un \ClasseDeuxLocFond{}. Plus précisément, $u$ est dans $\DeuxLocFondLax{W}$ si et seulement s'il est dans $(\DeuxLocFondLax{W} \cap \UnCell{\DeuxCat})_{lax}$, donc si et seulement si $\DeuxFoncUnOp{u}$ est dans $(\DeuxLocFondLax{W} \cap \UnCell{\DeuxCat})_{lax}$ (proposition \ref{DeuxFoncLaxUnOpW}), donc si et seulement si $\DeuxFoncUnOp{u}$ est dans $\DeuxLocFondLax{W}$.
\end{proof}

\begin{theo}
Pour toute partie $\DeuxLocFondLax{W}$ de $\UnCell{\DeuxCatLax}$, les propriétés suivantes sont équivalentes. 
\begin{itemize}
\item[(i)] La classe $\DeuxLocFondLax{W}$ est un \ClasseDeuxLocFondLax.
\item[(ii)] Les propriétés suivantes sont vérifiées.
\begin{itemize}
\item[$\rm{LF1'}_{lax}$] La classe $\DeuxLocFondLax{W}$ est faiblement saturée.

\item[$\rm{LF2'}_{lax}$] Si une petite \deux{}catégorie $\mathdeuxcat{A}$ op-admet un objet admettant un objet initial, alors le morphisme canonique $\mathdeuxcat{A} \to e$ est dans $\DeuxLocFondLax{W}$.

\item[$\rm{LF3'}_{lax}$] Si 
$$
\xymatrix{
\mathdeuxcat{A} 
\ar[rr]^{u}
\ar[dr]_{w}
&&\mathdeuxcat{B}{}
\ar[dl]^{v}
\\
&\mathdeuxcat{C}
}
$$
désigne un triangle commutatif dans $\DeuxCatLax$ et si, pour tout objet $c$ de $\mathdeuxcat{C}$, le \DeuxFoncteurLax{} 
$
\DeuxFoncOpTrancheLax{u}{c} : \OpTrancheLax{\mathdeuxcat{A}}{w}{c} \to \OpTrancheLax{\mathdeuxcat{B}}{v}{c} 
$
est dans $\DeuxLocFondLax{W}$, alors $u$ est dans $\DeuxLocFondLax{W}$.
\end{itemize}
\end{itemize}
\end{theo}

\begin{proof}
Montrons l'implication $(i) \Rightarrow (ii)$. La condition $\rm{LF1'}_{lax}$ est vérifiée par hypothèse, puisque ce n'est autre que la condition $\rm{LF1}_{lax}$. Pour vérifier la condition $\rm{LF2'}_{lax}$, considérons une petite \deux{}catégorie $\mathdeuxcat{A}$ op-admettant un objet admettant un objet initial. La \deux{}catégorie $\DeuxCatUnOp{\mathdeuxcat{A}}$ admettant un objet admettant un objet initial, le morphisme canonique $\DeuxCatUnOp{\mathdeuxcat{A}} \to \DeuxCatPonct$ est dans $\DeuxLocFondLax{W}$. En vertu du lemme \ref{DeuxFoncUnOpWLax}, le morphisme $\mathdeuxcat{A} \to \DeuxCatPonct$ est donc dans $\DeuxLocFondLax{W}$. Plaçons-nous sous les hypothèses de la condition $\rm{LF3'}_{lax}$. Pour tout objet $c$ de $\mathdeuxcat{C}$, le \DeuxFoncteurLax{} $\DeuxFoncUnOp{(\DeuxFoncTrancheLax{\DeuxFoncUnOp{u}}{c})}$ est dans $\DeuxLocFondLax{W}$. En vertu du lemme \ref{DeuxFoncUnOpWLax}, le \DeuxFoncteurLax{} $\DeuxFoncTrancheLax{\DeuxFoncUnOp{u}}{c}$ est donc dans $\DeuxLocFondLax{W}$ pour tout objet $c$ de $\mathdeuxcat{C}$. La condition $\rm{LF3}$ étant vérifiée, $\DeuxFoncUnOp{u}$ est dans $\DeuxLocFondLax{W}$. Une nouvelle invocation du lemme \ref{DeuxFoncUnOpWLax} permet de conclure. 

Pour montrer l'implication $(ii) \Rightarrow (i)$, on commence par remarquer que la classe 
$$
\DeuxFoncUnOp{\DeuxLocFond{W}} = \{ u \in \UnCell{\DeuxCatLax} \vert \DeuxFoncUnOp{u} \in \DeuxLocFondLax{W} \}
$$
est un \ClasseDeuxLocFondLax{}. La conclusion est analogue à celle de l'implication $(iii) \Rightarrow (i)$ du théorème \ref{DeuxLocFondHuitDef} en utilisant le lemme \ref{DeuxFoncUnOpWLax}.
\end{proof}

\section{Correspondances}\label{SectionCorrespondances}
%
%
%
%
%

\begin{lemme}\label{SupLaxNaturel}
Pour tout \DeuxFoncteurLax{} $u : \mathdeuxcat{A} \to \mathdeuxcat{B}$, le diagramme
$$
\xymatrix{
\Delta/\NerfLax\mathdeuxcat{A}
\ar[rr]^{\SupLaxObjet{\mathdeuxcat{A}}}
\ar[d]_{\Delta/\NerfLax{(u)}}
&&\mathdeuxcat{A}
\ar[d]^{u}
\\
\Delta/\NerfLax\mathdeuxcat{B}
\ar[rr]_{\SupLaxObjet{\mathdeuxcat{B}}}
&&\mathdeuxcat{B}
}
$$
est commutatif.
\end{lemme}

\begin{proof}
Pour tout objet $([m], x : [m] \to \mathdeuxcat{A})$ de $\Delta/\NerfLax\mathdeuxcat{A}$ ($x$ est donc un \DeuxFoncteurLax{}),
$$
\begin{aligned}
u(\SupLaxObjet{\mathdeuxcat{A}} ([m], x)) &= u(x_{m}) 
\end{aligned}
$$
et
$$
\begin{aligned}
\SupLaxObjet{\mathdeuxcat{B}} (\Delta / \NerfLax(u) ([m], x)) &= \SupLaxObjet{\mathdeuxcat{B}} ([m], ux)
\\
&= u(x_{m}).
\end{aligned}
$$

Pour tout morphisme $\varphi : [m] \to [n]$ de $([m], x)$ vers $([n], y)$ dans $\Delta / \NerfLax{\mathdeuxcat{A}}$, 
$$
\begin{aligned}
u(\SupLaxObjet{\mathdeuxcat{A}} (\varphi)) &= u(y_{n, \varphi(m)}) 
\end{aligned}
$$
et 
$$
\begin{aligned}
\SupLaxObjet{\mathdeuxcat{B}} (\Delta / \NerfLax(u) (\varphi)) &= (uy)_{n, \varphi(m)}
\\
&= u(y_{n, \varphi(m)}).
\end{aligned}
$$

Pour tout objet $([m], x : [m] \to \mathdeuxcat{A})$ de $\Delta/\NerfLax\mathdeuxcat{A}$,
$$
\begin{aligned}
(u \phantom{a} \SupLaxObjet{\mathdeuxcat{A}})_{([m], x)} &= u((\SupLaxObjet{\mathdeuxcat{A}})_{([m],x)}) u_{\SupLaxObjet{\mathdeuxcat{A}}([m],x)}
\\
&= u((x)_{m}) u_{x_{m}}  
\end{aligned} 
$$
et
$$
\begin{aligned}
(\SupLaxObjet{\mathdeuxcat{B}} \phantom{a} \Delta / \NerfLax(u)_{([m], x)} &= \SupLaxObjet{\mathdeuxcat{B}} (\Delta / \NerfLax(u))_{([m], x)}) (\SupLaxObjet{\mathdeuxcat{B}})_{\Delta / \NerfLax(u) ([m], x)}
\\
&= (\SupLaxObjet{\mathdeuxcat{B}})_{([m], ux)}
\\
&= (ux)_{m}
\\
&= u((x)_{m}) u_{x_{m}}.
\end{aligned}
$$

Soient
$$
\xymatrix{
([m], x) 
\ar[rr]^{\varphi}
&&
([n], y)
\ar[rr]^{\psi}
&&
([p], z)
}
$$
deux \un{}cellules composables dans $\Delta / \NerfLax{\mathdeuxcat{A}}$. Alors, 
$$
\begin{aligned}
(u \phantom{a} \SupLaxObjet{\mathdeuxcat{A}})_{\psi, \varphi} &= u({\SupLaxObjet{\mathdeuxcat{A}}}_{\psi, \varphi}) u_{\SupLaxObjet{\mathdeuxcat{A}}(\psi), \SupLaxObjet{\mathdeuxcat{A}}(\varphi)}
\\
&= u(z_{p, \psi(n), \psi\varphi(m)}) u_{z_{p, \psi(n)}, y_{n, \varphi(m)}}
\\
&= u(z_{p, \psi(n), \psi\varphi(m)}) u_{z_{p, \psi(n)}, z_{\psi(n), \psi\varphi(m)}}
\end{aligned} 
$$
et
$$
\begin{aligned}
(\SupLaxObjet{\mathdeuxcat{B}} \phantom{a} \Delta / \NerfLax(u))_{\psi, \varphi} &= \SupLaxObjet{\mathdeuxcat{B}}(\Delta / \NerfLax(u)_{\psi, \varphi}) (\SupLaxObjet{\mathdeuxcat{B}})_{\Delta / \NerfLax(u)(\psi), \Delta / \NerfLax(u)(\varphi)}
\\
&= (\SupLaxObjet{\mathdeuxcat{B}})_{\Delta / \NerfLax(u)(\psi), \Delta / \NerfLax(u)(\varphi)}
\\ 
&= (uz)_{p, \psi(n), \psi\varphi(m)}
\\
&= u(z_{p, \psi(n), \psi\varphi(m)}) u_{z_{p, \psi(n)}, z_{\psi(n), \psi\varphi(m)}}.
\end{aligned}
$$
Cela termine les vérifications de la commutativité du diagramme de l'énoncé. 
\end{proof}

\begin{prop}\label{SupLaxW}
Pour tout localisateur fondamental $\DeuxLocFondLax{W}$ de $\DeuxCatLax$, pour toute petite \deux{}ca\-té\-go\-rie $\mathdeuxcat{A}$, le \DeuxFoncteurLax{} $\SupLaxObjet{\mathdeuxcat{A}} : \Delta/ \NerfLax\mathdeuxcat{A} \to \mathdeuxcat{A}$ est une équi\-va\-lence faible. 
\end{prop}

\begin{proof}
En vertu des propositions \ref{InclusionNerfHomNerfCatLaxNorW}, \ref{InclusionNerfEnsNerfCatW} et \ref{InclusionNerfLaxNorNerfLaxW}, les flèches horizontales du diagramme commutatif figurant dans l'énoncé du lemme \ref{DiagrammeSups} sont des équivalences faibles. En vertu de la proposition \ref{SupHomAspherique}, $\SupHomObjet{\mathdeuxcat{A}}$ est une équivalence faible. Il en résulte, par des invocations successives de « $2$ sur $3$ », que toutes les flèches du diagramme commutatif figurant dans l'énoncé du lemme \ref{DiagrammeSups} sont des équivalences faibles. C'est donc en particulier le cas de $\SupLaxObjet{\mathdeuxcat{A}}$. 
\end{proof}

\begin{lemme}\label{Gabuzomeu}
Pour tout localisateur fondamental $\DeuxLocFondLax{W}$ de $\DeuxCatLax$, un \DeuxFoncteurLax{} $u$ est une équivalence faible si et seulement si $\Delta / \NerfLax{(u)}$ en est une. 
\end{lemme}

\begin{proof}
C'est une conséquence immédiate du lemme \ref{SupLaxNaturel} et de la proposition \ref{SupLaxW}.
\end{proof}

\begin{lemme}\label{DeuxLocFondLaxInterCat}
Pour tout localisateur fondamental $\DeuxLocFondLax{W}$ de $\DeuxCatLax$,  
$$
\DeuxLocFondLax{W} = {\NerfLax}^{-1} (i_{\Delta}^{-1}(\DeuxLocFondLax{W} \cap \UnCell{\Cat})).
$$
\end{lemme}

\begin{proof}
Un \DeuxFoncteurLax{} $u$ est dans $\DeuxLocFondLax{W}$ si et seulement si le foncteur $\Delta / \NerfLax{(u)}$ l'est (lemme \ref{Gabuzomeu}), donc si et seulement si $\Delta / \NerfLax{(u)}$ est dans $\DeuxLocFondLax{W} \cap \UnCell{\Cat}$, donc si et seulement si $u$ est dans ${\NerfLax}^{-1}i_{\Delta}^{-1}(\DeuxLocFondLax{W} \cap \UnCell{\Cat})$. 
\end{proof}

\begin{lemme}\label{Gaston}
Pour tout localisateur fondamental $\UnLocFond{W}$ de $\Cat$,
$$
{\NerfLax}^{-1} (i_{\Delta}^{-1} (\UnLocFond{W})) = ({\NerfLaxNor}^{-1} (i_{\Delta}^{-1} (\UnLocFond{W})))_{lax}.
$$
\end{lemme}

\begin{proof}
Cela résulte de la suite d'équivalences suivante, pour tout morphisme $u$ de $\DeuxCatLax$.
$$
\begin{aligned}
u \in {\NerfLax}^{-1} (i_{\Delta}^{-1} (\UnLocFond{W})) &\Longleftrightarrow \Delta / \NerfLax{(u)} \in \UnLocFond{W}
\\
&\Longleftrightarrow \Delta / \NerfLax{(\TildeLax{u})} \in \UnLocFond{W} \phantom{bla} \text{(proposition \ref{EquiDefEquiLax})}
\\
&\Longleftrightarrow \Delta / \NerfLaxNor{(\TildeLax{u})} \in \UnLocFond{W} \phantom{bla} \text{(lemme \ref{Blabla})}
\\
&\Longleftrightarrow \TildeLax{u} \in {\NerfLaxNor}^{-1} (i_{\Delta}^{-1} (\UnLocFond{W}))
\\
&\Longleftrightarrow u \in ({\NerfLaxNor}^{-1} (i_{\Delta}^{-1} (\UnLocFond{W})))_{lax}.
\end{aligned}
$$
\end{proof}

\begin{lemme}\label{Leroux}
Pour tout localisateur fondamental $\UnLocFond{W}$ de $\Cat$, la classe
$
{\NerfLax}^{-1} (i_{\Delta}^{-1} (\UnLocFond{W})) 
$
est un \ClasseDeuxLocFondLax{}.
\end{lemme}

\begin{proof}
Cela résulte du lemme \ref{Gaston}, de la remarque \ref{Levet} et de la proposition \ref{DeuxLocFondInduitLax}.
\end{proof}

%

\begin{lemme}\label{Chamfort}
Pour tout localisateur fondamental $\UnLocFond{W}$ de $\Cat$, 
$$
{\NerfLax}^{-1} (i_{\Delta}^{-1} (\UnLocFond{W})) \cap \UnCell{\Cat} =  \UnLocFond{W}.
$$
\end{lemme}

\begin{proof}
C'est une conséquence immédiate de la proposition \ref{Sade} et du fait que la restriction du nerf lax $\NerfLax$ à $\Cat$ coïncide avec le nerf $\UnNerf$. En formule : 
$$
{\NerfLax}^{-1} (i_{\Delta}^{-1} (\UnLocFond{W})) \cap \UnCell{\Cat} = {\UnNerf}^{-1} (i_{\Delta}^{-1} (\UnLocFond{W})) \cap \UnCell{\Cat} = \UnLocFond{W} \cap \UnCell{\Cat} = \UnLocFond{W}.
$$
\end{proof}

\begin{theo}\label{IsoUnLocFondDeuxLocFondLax}
Les applications
$$
\begin{aligned}
\mathcal{P}(\UnCell{\Cat}) &\to \mathcal{P}(\UnCell{\DeuxCatLax})
\\
\UnLocFond{W} &\mapsto {\NerfLax}^{-1} (i_{\Delta}^{-1} (\UnLocFond{W})) 
\end{aligned}
$$
et
$$
\begin{aligned}
\mathcal{P}(\UnCell{\DeuxCatLax}) &\to \mathcal{P}(\UnCell{\Cat})
\\
\DeuxLocFondLax{W} &\mapsto \DeuxLocFondLax{W} \cap \UnCell{\Cat}
\end{aligned}
$$
induisent des isomorphismes inverses l'un de l'autre entre la classe ordonnée par inclusion des \ClassesUnLocFond{} et la classe ordonnée par inclusion des \ClassesDeuxLocFondLax{}. 
%
\end{theo}

\begin{proof}
Ces applications respectant manifestement la relation d'inclusion, il résulte de la remarque \ref{RemDeuxLocFondLaxInterCat} et des lemmes \ref{DeuxLocFondLaxInterCat}, \ref{Leroux} et \ref{Chamfort} qu'il s'agit bien d'isomorphismes. 
\end{proof}

%

\begin{lemme}\label{Allain}
Pour tout localisateur fondamental $\DeuxLocFond{W}$ de $\DeuxCat$, 
$$
\DeuxLocFond{W} \cap \UnCell{\Cat} = \DeuxLocFondLaxInduit{\DeuxLocFond{W}} \cap \UnCell{\Cat}.
$$
\end{lemme}

\begin{proof}
En vertu du théorème \ref{IsoDeuxLocFondDeuxLocFondLax}, $ \DeuxLocFondLaxInduit{\DeuxLocFond{W}} \cap \UnCell{\DeuxCat} = \DeuxLocFond{W}$. Ainsi,
$$
\begin{aligned}
 \DeuxLocFondLaxInduit{\DeuxLocFond{W}} \cap \UnCell{\Cat} &= \DeuxLocFondLaxInduit{\DeuxLocFond{W}} \cap (\UnCell{\DeuxCat} \cap \UnCell{\Cat}) 
 \\
 &= (\DeuxLocFondLaxInduit{\DeuxLocFond{W}} \cap \UnCell{\DeuxCat}) \cap \UnCell{\Cat} 
 \\
 &= \DeuxLocFond{W} \cap \UnCell{\Cat}.
 \end{aligned}
$$
\end{proof}

\begin{theo}\label{IsoUnLocFondDeuxLocFond}
Les applications
$$
\begin{aligned}
\mathcal{P}(\UnCell{\Cat}) &\to \mathcal{P}(\UnCell{\DeuxCat})
\\
\UnLocFond{W} &\mapsto {\NerfLaxNor}^{-1} (i_{\Delta}^{-1} (\UnLocFond{W}))
\\
&\phantom{bla}(= {\NerfLax}^{-1} (i_{\Delta}^{-1} (\UnLocFond{W})) \cap \UnCell{\DeuxCat})
\end{aligned}
$$
et
$$
\begin{aligned}
\mathcal{P}(\UnCell{\DeuxCat}) &\to \mathcal{P}(\UnCell{\Cat})
\\
\DeuxLocFond{W} &\mapsto \DeuxLocFond{W} \cap \UnCell{\Cat}
\end{aligned}
$$
induisent des isomorphismes inverses l'un de l'autre entre la classe ordonnée par inclusion des \ClassesUnLocFond{} et la classe ordonnée par inclusion des \ClassesDeuxLocFond{}. 
%
\end{theo}

\begin{proof}
C'est une conséquence des théorèmes \ref{IsoDeuxLocFondDeuxLocFondLax} et \ref{IsoUnLocFondDeuxLocFondLax} et du lemme \ref{Allain}. 
\end{proof}

La notion de \ClasseDeuxLocFondLax{} est stable par intersection. On définit le \emph{localisateur fondamental minimal de $\DeuxCatLax$}\index{localisateur fondamental minimal de $\DeuxCatLax$} comme l'intersection de tous les localisateurs fondamentaux de $\DeuxCatLax$. 

\begin{theo}\label{TheoDeuxLocFondLaxMin}
Le localisateur fondamental minimal de $\DeuxCatLax$ est la classe 
$$
\DeuxLocFondLaxMin{}\index[not]{WInfini2lax@$\DeuxLocFondLaxMin{}$} = {\NerfLax}^{-1} (\EquiQuillen).
$$
\end{theo}

\begin{proof}
C'est une conséquence des théorèmes \ref{CisinskiGrothendieck} et \ref{IsoUnLocFondDeuxLocFondLax}. 
\end{proof}


La notion de \ClasseDeuxLocFond{} est stable par intersection. On définit le \emph{localisateur fondamental minimal de $\DeuxCat$}\index{localisateur fondamental minimal de $\DeuxCat$} comme l'intersection de tous les localisateurs fondamentaux de $\DeuxCat$.

\begin{theo}\label{TheoDeuxLocFondMin}
Le localisateur fondamental minimal de $\DeuxCat$ est la classe 
$$
\begin{aligned}
\DeuxLocFondMin{}\index[not]{WInfini2@$\DeuxLocFondMin{}$} &= {\NerfLaxNor}^{-1} (\EquiQuillen)
\\
&(={\NerfLax}^{-1} (\EquiQuillen) \cap \UnCell{\DeuxCat}).
\end{aligned}
$$
\end{theo}
 
\begin{proof}
C'est une conséquence des théorèmes \ref{CisinskiGrothendieck} et \ref{IsoUnLocFondDeuxLocFond}. 
\end{proof}

\begin{theo}\label{EqCatLocDeuxCatDeuxCatLaxBis}
Pour tout localisateur fondamental $\DeuxLocFondLax{W}$ de $\DeuxCatLax$, l'inclusion 
$$
I : \DeuxCat \hookrightarrow \DeuxCatLax
$$ 
et le foncteur de strictification de Bénabou 
$$
B : \DeuxCatLax \to \DeuxCat
$$ 
induisent des équivalences de catégories 
$$
\Localisation{\DeuxCat}{(\DeuxLocFondLax{W} \cap \UnCell{\DeuxCat})} \to \Localisation{\DeuxCatLax}{\DeuxLocFondLax{W}}
$$ 
et 
$$
\Localisation{\DeuxCatLax}{\DeuxLocFondLax{W}} \to \Localisation{\DeuxCat}{(\DeuxLocFondLax{W} \cap \UnCell{\DeuxCat})}
$$ 
quasi-inverses l'une de l'autre.
\end{theo}

\begin{proof}
Comme $\DeuxLocFondLax{W} \cap \UnCell{\DeuxCat}$ est un localisateur fondamental de $\DeuxCat$, on sait, en vertu du théorème \ref{EqCatLocDeuxCatDeuxCatLax}, que $I$ et $B$ induisent des équivalences de catégories 
$$
\Localisation{\DeuxCat}{(\DeuxLocFondLax{W} \cap \UnCell{\DeuxCat})} \to \Localisation{\DeuxCatLax}{(\DeuxLocFondLax{W} \cap \UnCell{\DeuxCat})_{lax}}
$$ 
et 
$$
\Localisation{\DeuxCatLax}{(\DeuxLocFondLax{W} \cap \UnCell{\DeuxCat})_{lax}} \to \Localisation{\DeuxCat}{(\DeuxLocFondLax{W} \cap \UnCell{\DeuxCat})}
$$ 
quasi-inverses l'une de l'autre. La conclusion résulte de l'égalité $(\DeuxLocFondLax{W} \cap \UnCell{\DeuxCat})_{lax} = \DeuxLocFondLax{W}$.
\end{proof}
 
\begin{theo}\label{EqCatLocCatDeuxCatLax}
Pour tout localisateur fondamental $\UnLocFond{W}$ de $\Cat$, l'inclusion 
$$
\Cat \hookrightarrow \DeuxCatLax
$$ 
et le foncteur 
$$
i_{\Delta} \NerfLax : \DeuxCatLax \to \Cat
$$ 
induisent des équivalences de catégories 
$$
\Localisation{\Cat}{\UnLocFond{W}} \to \Localisation{\DeuxCatLax}{({\NerfLax}^{-1} (i_{\Delta}^{-1} (\UnLocFond{W})))}
$$ 
et 
$$
\Localisation{\DeuxCatLax}{({\NerfLax}^{-1} (i_{\Delta}^{-1} (\UnLocFond{W})))} \to \Localisation{\Cat}{\UnLocFond{W}}
$$ 
quasi-inverses l'une de l'autre.
\end{theo}

\begin{proof}
Comme ${\NerfLaxNor}^{-1} (i_{\Delta}^{-1} (\UnLocFond{W}))$ est un localisateur fondamental de $\DeuxCat$, on sait déjà, en vertu des théorèmes \ref{EqCatLocDeuxCatDeuxCatLax} et \ref{EqCatLocCatDeuxCat}, que l'inclusion $\Cat \hookrightarrow \DeuxCatLax$ et le foncteur $\DeuxIntOp{\Delta} \NerfHom{} B : \DeuxCatLax \to \Cat$ induisent des équivalences de catégories
$$
\Localisation{\Cat}{(\NerfLaxNor^{-1} (i_{\Delta}^{-1} (\UnLocFond{W})) \cap \UnCell{\Cat})} \to \Localisation{\DeuxCatLax}{(\NerfLaxNor^{-1} (i_{\Delta}^{-1} (\UnLocFond{W})))_{lax}}
$$
et
$$
\Localisation{\DeuxCatLax}{(\NerfLaxNor^{-1} (i_{\Delta}^{-1} (\UnLocFond{W})))_{lax}} \to \Localisation{\Cat}{(\NerfLaxNor^{-1} (i_{\Delta}^{-1} (\UnLocFond{W})) \cap \UnCell{\Cat})}
$$
quasi-inverses l'une de l'autre. On a de plus les égalités 
$$
\NerfLaxNor^{-1} (i_{\Delta}^{-1} (\UnLocFond{W})) \cap \UnCell{\Cat} = \UnLocFond{W}
$$ 
et 
$$
(\NerfLaxNor^{-1} (i_{\Delta}^{-1} (\UnLocFond{W})))_{lax} = \NerfLax^{-1} (i_{\Delta}^{-1} (\UnLocFond{W})).
$$ 
La conclusion résulte donc 
%
du fait que, pour tout morphisme $u : \mathdeuxcat{A} \to \mathdeuxcat{B}$ de $\DeuxCatLax$, pour tout localisateur fondamental $\UnLocFond{W}$ de $\Cat$, il existe un diagramme commutatif\footnote{Dans lequel on a noté de la même façon les morphismes de $\Cat$ et leur image par le foncteur de localisation.} dans $\Localisation{\Cat}{\UnLocFond{W}}$
$$
\xymatrix{
\DeuxIntOp{\Delta} \NerfHom \TildeLax{\mathdeuxcat{A}}
\ar[rr]
\ar[dd]_{\DeuxIntOp{\Delta} \NerfHom (\TildeLax{u})}
&&
\DeuxIntOp{\Delta} \NerfLax{\TildeLax{\mathdeuxcat{A}}} \ar[dd]^{\DeuxIntOp{\Delta} \NerfLax (\TildeLax{u})}
&&
\DeuxIntOp{\Delta} \NerfLax \mathdeuxcat{A} = \Delta / \NerfLax \mathdeuxcat{A}
\ar[ll]_{\DeuxIntOp{\Delta} \NerfLax (\LaxCanonique{\mathdeuxcat{A}})}
\ar[dd]^{\DeuxIntOp{\Delta} \NerfLax (u) = \Delta / \NerfLax (u)}
\\
\\
\DeuxIntOp{\Delta} \NerfHom \TildeLax{\mathdeuxcat{B}}
\ar[rr]
&&
\DeuxIntOp{\Delta} \NerfLax \TildeLax{\mathdeuxcat{B}}
&&
\DeuxIntOp{\Delta} \NerfLax \mathdeuxcat{B} = \Delta / \NerfLax \mathdeuxcat{B}
\ar[ll]^{\DeuxIntOp{\Delta} \NerfLax (\LaxCanonique{\mathdeuxcat{B}})}
}
$$
dont les flèches horizontales sont des isomorphismes.
\end{proof}

\begin{theo}\label{EqCatLocCatDeuxCatLaxBis}
Pour tout localisateur fondamental $\DeuxLocFondLax{W}$ de $\DeuxCatLax$, l'inclusion 
$$
\Cat \hookrightarrow \DeuxCatLax
$$ 
et le foncteur 
$$
i_{\Delta} \NerfLax : \DeuxCatLax \to \Cat
$$ 
induisent des équivalences de catégories 
$$
\Localisation{\Cat}{(\DeuxLocFondLax{W} \cap \UnCell{\Cat})} \to \Localisation{\DeuxCatLax}{\DeuxLocFondLax{W}}
$$ 
et 
$$
\Localisation{\DeuxCatLax}{\DeuxLocFondLax{W}} \to \Localisation{\Cat}{(\DeuxLocFondLax{W} \cap \UnCell{\Cat})}
$$ 
quasi-inverses l'une de l'autre.
\end{theo}

\begin{proof}
Comme la classe $\DeuxLocFondLax{W} \cap \UnCell{\Cat}$ est un localisateur fondamental de $\Cat$, cela résulte du théorème \ref{EqCatLocCatDeuxCatLax} et de l'égalité 
$$
\NerfLax^{-1} (i_{\Delta}^{-1} (\DeuxLocFondLax{W} \cap \UnCell{\Cat})) = \DeuxLocFondLax{W}.
$$
\end{proof} 

\begin{theo}\label{EqCatLocCatDeuxCatBis}
Pour tout localisateur fondamental $\UnLocFond{W}$ de $\Cat$, l'inclusion 
$$
\Cat \hookrightarrow \DeuxCat
$$ 
et le foncteur 
$$
i_{\Delta} \NerfLaxNor : \DeuxCat \to \Cat
$$ 
induisent des équivalences de catégories 
$$
\Localisation{\Cat}{\UnLocFond{W}} \to \Localisation{\DeuxCat}{({\NerfLaxNor}^{-1} (i_{\Delta}^{-1} (\UnLocFond{W})))}
$$ 
et 
$$
\Localisation{\DeuxCat}{({\NerfLaxNor}^{-1} (i_{\Delta}^{-1} (\UnLocFond{W})))} \to \Localisation{\Cat}{\UnLocFond{W}}
$$ 
quasi-inverses l'une de l'autre.
\end{theo}

\begin{proof}
La classe $\NerfLaxNor^{-1} (i_{\Delta}^{-1} (\UnLocFond{W}))$ est un localisateur fondamental de $\DeuxCat$. Par conséquent, en vertu de l'égalité 
$$
\NerfLaxNor^{-1} (i_{\Delta}^{-1} (\UnLocFond{W})) \cap \UnCell{\Cat} = \UnLocFond{W}
$$
et du théorème \ref{EqCatLocCatDeuxCat}, l'inclusion $\Cat \hookrightarrow \DeuxCat$ et le foncteur $\DeuxIntOp{\Delta} \NerfHom : \DeuxCat \to \Cat$ induisent des équivalences de catégories quasi-inverses l'une de l'autre entre $\Localisation{\Cat}{\UnLocFond{W}}$ et $\Localisation{\DeuxCat}{(\NerfLaxNor^{-1} (i_{\Delta}^{-1} (\UnLocFond{W})))}$. La conclusion résulte donc du fait que, pour tout morphisme $u : \mathdeuxcat{A} \to \mathdeuxcat{B}$ de $\DeuxCat$, pour tout localisateur fondamental $\UnLocFond{W}$ de $\Cat$, il existe un diagramme commutatif\footnote{Dans lequel on a noté de la même façon les morphismes de $\Cat$ et leur image par le foncteur de localisation.} dans $\Localisation{\Cat}{\UnLocFond{W}}$
$$
\xymatrix{
\DeuxIntOp{\Delta} \NerfHom \mathdeuxcat{A}
\ar[rr]
\ar[dd]_{\DeuxIntOp{\Delta} \NerfHom (u)}
&&
\DeuxIntOp{\Delta} \NerfLaxNor \mathdeuxcat{A}\ar[dd]^{\DeuxIntOp{\Delta} \NerfLaxNor (u)}
\\
\\
\DeuxIntOp{\Delta} \NerfHom \mathdeuxcat{B}
\ar[rr]
&&
\DeuxIntOp{\Delta} \NerfLaxNor \mathdeuxcat{B}
}
$$
dont les flèches horizontales sont des isomorphismes.
\end{proof}

On termine cette section par quelques énoncés permettant notamment d'assurer que les isomorphismes entre localisateurs fondamentaux de $\Cat$, $\DeuxCat$ et $\DeuxCatLax$ figurant dans l'énoncé des théorèmes \ref{IsoDeuxLocFondDeuxLocFondLax}, \ref{IsoUnLocFondDeuxLocFondLax} et \ref{IsoUnLocFondDeuxLocFond} préservent la propriété d'être engendré par un \emph{ensemble} de morphismes, détail d'importance lorsqu'il s'agit de montrer l'existence de structures de catégorie de modèles sur $\Cat$ ou $\DeuxCat$ dont la classe des équivalences faibles est donnée par un localisateur fondamental. 

\begin{df}
On dira qu'un \ClasseUnLocFond{} (\emph{resp.} un \ClasseDeuxLocFond{}, \emph{resp.} un \ClasseDeuxLocFondLax{}) est \emph{engendré} par une classe $S$ de morphismes de $\Cat$ (\emph{resp.} de $\DeuxCat$, \emph{resp.} de $\DeuxCatLax$) si c'est le plus petit \ClasseUnLocFond{} (\emph{resp.} \ClasseDeuxLocFond{}, \emph{resp.} \ClasseDeuxLocFondLax{}) contenant $S$ ou, autrement dit, l'intersection de tous les \ClassesUnLocFond{} (\emph{resp.} \ClassesDeuxLocFond{}, \emph{resp.} \ClassesDeuxLocFondLax{}) contenant $S$. 
\end{df}

\begin{prop}\label{Tatata}
Si un localisateur fondamental $\DeuxLocFond{W}$ de $\DeuxCat$ est engendré par une classe $S \subset \UnCell{\DeuxCat}$, alors le localisateur fondamental $\DeuxLocFondLaxInduit{W}$ de $\DeuxCatLax$ est également engendré par $S$. 
\end{prop}

\begin{proof}
On a évidemment $S \subset \DeuxLocFondLaxInduit{W}$. Soit $\DeuxLocFond{W}'$ un localisateur fondamental de $\DeuxCatLax$ contenant $S$. En vertu du théorème \ref{IsoDeuxLocFondDeuxLocFondLax}, l'inclusion $\DeuxLocFondLaxInduit{W} \subset \DeuxLocFond{W}'$ équivaut à $\DeuxLocFondLaxInduit{W} \cap \UnCell{\DeuxCat} \subset \DeuxLocFond{W}' \cap \UnCell{\DeuxCat}$, c'est-à-dire, en vertu de ce même théorème, $\DeuxLocFond{W} \subset \DeuxLocFond{W}' \cap \UnCell{\DeuxCat}$. Cette inclusion résulte de l'hypothèse faite sur $\DeuxLocFond{W}$ et du fait que $\DeuxLocFond{W}' \cap \UnCell{\DeuxCat}$ est un \ClasseDeuxLocFond{} contenant $S$. 
\end{proof}

\begin{df}
Pour toute classe $S \subset \UnCell{\DeuxCatLax}$, on pose
$$
\widetilde{S} = B(S).
$$
Les éléments de $\widetilde{S}$ sont donc les images de ceux de $S$ par le foncteur de strictification de Bénabou. 
\end{df}

\begin{lemme}\label{Blablabla}
Soit $\DeuxLocFondLax{W}$ un \ClasseDeuxLocFondLax{}. S'il est engendré par $S \subset \UnCell{\DeuxCatLax}$, alors il est engendré par $\widetilde{S}$. 
\end{lemme}

\begin{proof}
Pour tout $u : \mathdeuxcat{A} \to \mathdeuxcat{B}$ dans $S$, on a le diagramme commutatif
$$
\xymatrix{
\TildeLax{\mathdeuxcat{A}}
\ar[r]^{\TildeLax{u}}
&\TildeLax{\mathdeuxcat{B}}
\\
\mathdeuxcat{A}
\ar[u]^{\LaxCanonique{\mathdeuxcat{A}}}
\ar[r]_{u}
&\mathdeuxcat{B}
\ar[u]_{\LaxCanonique{\mathdeuxcat{B}}}
}
$$
dont les flèches verticales sont dans $\DeuxLocFondLax{W}$. Comme $u$ l'est aussi, c'est également le cas de $\TildeLax{u}$, ce qui montre l'inclusion $\widetilde{S} \subset \DeuxLocFondLax{W}$. Soit maintenant $\DeuxLocFondLax{W}'$ un \ClasseDeuxLocFondLax{} contenant $\widetilde{S}$. La considération du même diagramme, dont les flèches verticales sont dans $\DeuxLocFondLax{W}'$, permet d'affirmer $S \subset \DeuxLocFondLax{W}'$, donc $\DeuxLocFondLax{W} \subset \DeuxLocFondLax{W}'$.
\end{proof}

\begin{prop}\label{Tetete}
Soit $\DeuxLocFondLax{W}$ un \ClasseDeuxLocFondLax{}. S'il est engendré par $S \subset \UnCell{\DeuxCatLax}$, alors le localisateur fondamental $\DeuxLocFondLax{W} \cap \UnCell{\DeuxCat}$ de $\DeuxCat$ est engendré par $\widetilde{S}$. 
\end{prop}

\begin{proof}
En vertu du lemme \ref{Blablabla}, $\DeuxLocFondLax{W}$ est engendré par $\widetilde{S}$. On a bien sûr $\widetilde{S} \subset \DeuxLocFondLax{W} \cap \UnCell{\DeuxCat}$. Soit $\DeuxLocFond{W}'$ un \ClasseDeuxLocFond{} contenant $\widetilde{S}$. Comme $\DeuxLocFond{W}' \subset \DeuxLocFondLaxInduit{\DeuxLocFond{W}'}$, l'hypothèse implique $\widetilde{S} \subset \DeuxLocFondLaxInduit{\DeuxLocFond{W}'}$, donc $\DeuxLocFond{W} \subset \DeuxLocFondLaxInduit{\DeuxLocFond{W}'}$, c'est-à-dire $(\DeuxLocFondLax{W} \cap \UnCell{\DeuxCat})_{lax} \subset \DeuxLocFondLaxInduit{\DeuxLocFond{W}'}$, donc $\DeuxLocFondLax{W} \cap \UnCell{\DeuxCat} \subset \DeuxLocFond{W}'$, ce qui permet de conclure. 
\end{proof}

\begin{rem}
On se gardera de croire que, si un localisateur fondamental $\DeuxLocFondLax{W}$ de $\DeuxCatLax$ est engendré par une classe de \DeuxFoncteursLax{} $S$, alors le localisateur fondamental $\DeuxLocFondLax{W} \cap \UnCell{\DeuxCat}$ de $\DeuxCat$ est engendré par ${S} \cap \UnCell{\DeuxCat}$. Pour un contre-exemple, on peut considérer $S = \UnCell{\DeuxCatLax} \backslash \UnCell{\DeuxCat}$, c'est-à-dire la classe des morphismes de $\DeuxCatLax$ qui ne sont pas dans $\DeuxCat$.
\end{rem}

\begin{prop}\label{Tititi}
Soit $\UnLocFond{W}$ un \ClasseUnLocFond{}. S'il est engendré par $S \subset \UnCell{\Cat}$, alors le localisateur fondamental $\NerfLax^{-1} (i_{\Delta}^{-1} (\UnLocFond{W}))$ de $\DeuxCatLax$ est également engendré par $S$. 
\end{prop}

\begin{proof}
On a bien sûr $S \subset \NerfLax^{-1} (i_{\Delta}^{-1} (\UnLocFond{W}))$. Soit de plus $\DeuxLocFondLax{W}$ un \ClasseDeuxLocFondLax{} contenant $S$. L'inclusion $\NerfLax^{-1} (i_{\Delta}^{-1} (\UnLocFond{W})) \subset \DeuxLocFondLax{W}$ équivaut à $\NerfLax^{-1} (i_{\Delta}^{-1} (\UnLocFond{W})) \cap \UnCell{\Cat} \subset \DeuxLocFondLax{W} \cap \UnCell{\Cat}$, c'est-à-dire à $\UnLocFond{W} \subset \DeuxLocFondLax{W} \cap \UnCell{\Cat}$, ce qui résulte du fait que $\DeuxLocFondLax{W} \cap \UnCell{\Cat}$ est un \ClasseUnLocFond{} contenant $S$ et de l'hypothèse faite sur $\UnLocFond{W}$. 
\end{proof}

\begin{lemme}\label{Blebleble}
Soit $\DeuxLocFondLax{W}$ un \ClasseDeuxLocFondLax{}. S'il est engendré par $S \subset \UnCell{\DeuxCatLax}$, alors il est également engendré par $i_{\Delta}(\NerfLax (S))$.
\end{lemme}

\begin{proof}
Pour tout $u : \mathdeuxcat{A} \to \mathdeuxcat{B}$ dans $S$, on a le diagramme commutatif
$$
\xymatrix{
\Delta / \NerfLax \mathdeuxcat{A}
\ar[rr]^{\Delta / \NerfLax (u)}
\ar[d]_{\SupLaxObjet{\mathdeuxcat{A}}}
&&
\Delta / \NerfLax \mathdeuxcat{B}
\ar[d]^{\SupLaxObjet{\mathdeuxcat{B}}}
\\
\mathdeuxcat{A}
\ar[rr]_{u}
&&
\mathdeuxcat{B}
}
$$
dont les flèches verticales sont dans $\DeuxLocFondLax{W}$. C'est donc également le cas de $\Delta / \NerfLax (u)$, ce qui montre $i_{\Delta}(\NerfLax (S)) \subset \DeuxLocFondLax{W}$. Étant donné un localisateur fondamental $\DeuxLocFondLax{W}'$ de $\DeuxCatLax$ contenant $i_{\Delta}(\NerfLax (S))$, la considération du même diagramme permet de conclure $S \subset \DeuxLocFondLax{W}'$, donc $\DeuxLocFondLax{W} \subset \DeuxLocFondLax{W}'$. 
\end{proof}

\begin{prop}\label{Tototo}
Soit $\DeuxLocFondLax{W}$ un \ClasseDeuxLocFondLax{}. S'il est engendré par $S \subset \UnCell{\DeuxCatLax}$, alors le localisateur fondamental $\DeuxLocFondLax{W} \cap \UnCell{\Cat}$ de $\Cat$ est engendré par $i_{\Delta}(\NerfLax (S))$. 
\end{prop}

\begin{proof}
En vertu du lemme \ref{Blebleble}, $\DeuxLocFondLax{W}$ est engendré par $i_{\Delta}(\NerfLax (S))$, donc en particulier $i_{\Delta}(\NerfLax (S)) \subset \DeuxLocFondLax{W} \cap \UnCell{\Cat}$. Soit $\UnLocFond{W}$ un \ClasseUnLocFond{} contenant $i_{\Delta}(\NerfLax (S))$. On a donc l'inclusion $S \subset \NerfLax^{-1} (i_{\Delta}^{-1} (\UnLocFond{W}))$, donc $\DeuxLocFondLax{W} \subset \NerfLax^{-1} (i_{\Delta}^{-1} (\UnLocFond{W}))$ puisque $\NerfLax^{-1} (i_{\Delta}^{-1} (\UnLocFond{W}))$ est un \ClasseDeuxLocFondLax{} et que $\DeuxLocFondLax{W}$ est le plus petit \ClasseDeuxLocFondLax{} contenant $S$. En vertu du lemme \ref{Leroux}, cela se récrit $\NerfLax^{-1} (i_{\Delta}^{-1} (\DeuxLocFondLax{W} \cap \UnCell{\Cat})) \subset \NerfLax^{-1} (i_{\Delta}^{-1} (\UnLocFond{W}))$, d'où $\DeuxLocFondLax{W} \cap \UnCell{\Cat} \subset \UnLocFond{W}$ en vertu du théorème \ref{IsoUnLocFondDeuxLocFondLax}. 
\end{proof}

\begin{prop}
Soit $\DeuxLocFond{W}$ un \ClasseDeuxLocFond{}. S'il est engendré par $S \subset \UnCell{\DeuxCat}$, alors le localisateur fondamental $\DeuxLocFond{W} \cap \UnCell{\Cat}$ de $\Cat$ est engendré par $i_{\Delta}(\NerfLaxNor (S))$. 
\end{prop}

\begin{proof}
C'est une conséquence des propositions \ref{Tatata} et \ref{Tototo} et de la remarque \ref{Nini}. 
\end{proof}

\begin{prop}
Soit $\UnLocFond{W}$ un \ClasseUnLocFond{}. S'il est engendré par $S \subset \UnCell{\Cat}$, alors le localisateur fondamental $\NerfLaxNor^{-1} (i_{\Delta}^{-1}(\UnLocFond{W}))$ de $\DeuxCat$ est engendré par $S$. 
\end{prop}

\begin{proof}
C'est une conséquence des propositions \ref{Tetete} et \ref{Tititi} et du lemme \ref{BarreTildeW}. 
\end{proof}

\section{Un Théorème A plus général}\label{SectionTheoremeAGeneral}

\emph{On suppose fixé un localisateur fondamental $\DeuxLocFond{W}$ de $\DeuxCat$.}

\begin{lemme}\label{LeoCampion}
Soient $\mathdeuxcat{A}$ une petite \deux{}catégorie, $u$ et $v$ des \DeuxFoncteursStricts{} de $\mathdeuxcat{A}$ vers $\DeuxCatDeuxCat$ et $\sigma$ une \DeuxTransformationStricte{} de $u$ vers $v$. Supposons que, pour tout objet $a$ de $\mathdeuxcat{A}$, le \DeuxFoncteurStrict{} $\sigma_{a} : u(a) \to v(a)$ soit une équivalence faible. Alors, le \DeuxFoncteurStrict{} $\DeuxInt{\mathdeuxcat{A}} \sigma : \DeuxInt{\mathdeuxcat{A}}u \to \DeuxInt{\mathdeuxcat{A}}v$ est une équivalence faible.
\end{lemme}

\begin{proof}
En vertu des lemmes \ref{KikiDeMontparnasse} et \ref{PierreDac}, de la proposition \ref{ProjectionIntegralePrefibration}, du corollaire \ref{PreadjointsW} et de la saturation faible de $\DeuxLocFond{W}$, les hypothèses impliquent que, pour tout objet $a$ de $\mathdeuxcat{A}$, le \DeuxFoncteurStrict{} $\DeuxFoncTrancheLax{(\DeuxInt{\mathdeuxcat{A}}\sigma)}{a} : \TrancheLax{(\DeuxInt{\mathdeuxcat{A}}u)}{P_{u}}{a} \to \TrancheLax{(\DeuxInt{\mathdeuxcat{A}}v)}{P_{v}}{a}$ est une équivalence faible. La proposition \ref{TheoremeATrancheLax} permet de conclure.  
\end{proof}

\begin{lemme}\label{QEquiFaible}
Pour tout \DeuxFoncteurStrict{} $w : \mathdeuxcat{A} \to \mathdeuxcat{C}$, le \DeuxFoncteurStrict{} canonique \footnote{Défini dans le paragraphe \ref{QPrefibration}.}
$$
Q_{\mathdeuxcat{A}} : \DeuxInt{\mathdeuxcat{C}} \DeuxFoncteurTranche{w} \to \mathdeuxcat{A}
$$
est une équivalence faible.
\end{lemme}

\begin{proof}
En vertu du paragraphe \ref{QPrefibration}, c'est une préfibration dont la fibre au-dessus d'un objet quelconque $a$ de $\mathdeuxcat{A}$ s'identifie à $\OpTrancheCoLax{\mathdeuxcat{C}}{}{w(a)}$, donc est asphérique. Le \DeuxFoncteurStrict{} $Q_{\mathdeuxcat{A}}$ est donc une préfibration à fibres asphériques. La conclusion résulte de la proposition \ref{PrefibrationFibresAspheriquesW}. 
\end{proof}

\begin{theo}\label{ThAStrictCoq}
Soit
$$
\xymatrix{
\mathdeuxcat{A} 
\ar[rr]^{u}
\ar[dr]_{w}
&&\mathdeuxcat{B}
\dtwocell<\omit>{<7.3>\sigma}
\ar[dl]^{v}
\\
&\mathdeuxcat{C}
&{}
}
$$
un diagramme de \DeuxFoncteursStricts{} commutatif à l'\DeuxTransformationCoLax{} $\sigma : vu \Rightarrow w$ près seulement. Si, pour tout objet $c$ de $\mathdeuxcat{C}$, le \DeuxFoncteurStrict{}\footnote{Défini dans la section \ref{SectionMorphismesInduits}.} $\DeuxFoncTrancheLaxCoq{u}{\sigma}{c} : \TrancheLax{\mathdeuxcat{A}}{w}{c} \to \TrancheLax{\mathdeuxcat{B}}{w}{c}$ est une équivalence faible, alors $u$ est une équivalence faible.
\end{theo}

\begin{proof}
Considérons le diagramme commutatif\footnote{Voir le lemme \ref{Desproges}.} de \DeuxFoncteursStricts{}
$$
\xymatrix{
\DeuxInt{\mathdeuxcat{C}} \DeuxFoncteurTranche{w}
\ar[rr]^{\DeuxInt{\mathdeuxcat{C}} \DeuxFoncteurTranche{\sigma}}
\ar[d]_{Q_{\mathdeuxcat{A}}}
&&\DeuxInt{\mathdeuxcat{C}} \DeuxFoncteurTranche{v}
\ar[d]^{Q_{\mathdeuxcat{B}}}
\\
\mathdeuxcat{A}
\ar[rr]_{u}
&&\mathdeuxcat{B}
&.
}
$$
En vertu des hypothèses et du lemme \ref{LeoCampion}, le \DeuxFoncteurStrict{} $\DeuxInt{\mathdeuxcat{C}} \DeuxFoncteurTranche{\sigma}$ est une équivalence faible. Comme $Q_{\mathdeuxcat{A}}$ et $Q_{\mathdeuxcat{B}}$ sont des équivalences faibles en vertu du lemme \ref{QEquiFaible}, la saturation faible de $\DeuxLocFond{W}$ permet de conclure.
\end{proof}

\begin{lemme}\label{CarreCommutatifThALax}
Soient
$$
\xymatrix{
\mathdeuxcat{A} 
\ar[rr]^{u}
\ar[dr]_{w}
&&\mathdeuxcat{B}
\dtwocell<\omit>{<7.3>\sigma}
\ar[dl]^{v}
\\
&\mathdeuxcat{C}
&{}
}
$$
un diagramme de \DeuxFoncteursLax{} commutatif à l'\DeuxTransformationCoLax{} $\sigma : vu \Rightarrow w$ près seulement, $c$ un objet de $\mathdeuxcat{C}$, $\DeuxFoncTrancheLaxCoq{u}{\sigma}{c} : \TrancheLax{\mathdeuxcat{A}}{w}{c} \to \TrancheLax{\mathdeuxcat{B}}{w}{c}$ le \DeuxFoncteurLax{} induit par ces données et $\DeuxFoncTrancheLaxCoq{\TildeLax{u}}{\BarreLax{\sigma}}{c}$ le \DeuxFoncteurStrict{} induit par le diagramme
$$
\xymatrix{
\TildeLax{\mathdeuxcat{A}}
\ar[rr]^{\TildeLax{u}}
\ar[dr]_{\BarreLax{w}}
&&\TildeLax{\mathdeuxcat{B}}
\dtwocell<\omit>{<7.3>\BarreLax{\sigma}}
\ar[dl]^{\BarreLax{v}}
\\
&\mathdeuxcat{C}
&,
}
$$
lequel est commutatif à l'\DeuxTransformationCoLax{} $\BarreLax{\sigma}$ de source $\BarreLax{v} \TildeLax{u} = \BarreLax{vu}$ et de but $\BarreLax{w}$ près seulement. Alors, le diagramme
$$
\xymatrix{
\TrancheLax{\TildeLax{\mathdeuxcat{A}}}{\BarreLax{w}}{c}
\ar[rr]^{\DeuxFoncTrancheLaxCoq{\TildeLax{u}}{\BarreLax{\sigma}}{c}}
&&\TrancheLax{\TildeLax{\mathdeuxcat{B}}}{\BarreLax{v}}{c}
\\
\TrancheLax{\mathdeuxcat{A}}{w}{c}
\ar[u]^{\DeuxFoncTrancheLax{\LaxCanonique{\mathdeuxcat{A}}}{c}}
\ar[rr]_{\DeuxFoncTrancheLaxCoq{u}{\sigma}{c}}
&&\TrancheLax{\mathdeuxcat{B}}{v}{c}
\ar[u]_{\DeuxFoncTrancheLax{\LaxCanonique{\mathdeuxcat{B}}}{c}}
}
$$
est commutatif. 
\end{lemme}

\begin{proof}
Pour tout objet $(a, p : w(a) \to c)$ de $\TrancheLax{\mathdeuxcat{A}}{w}{c}$, 
$$
\begin{aligned}
(\DeuxFoncTrancheLax{\LaxCanonique{\mathdeuxcat{B}}}{c}) ((\DeuxFoncTrancheLaxCoq{u}{\sigma}{c}) (a, p)) &= (\DeuxFoncTrancheLax{\LaxCanonique{\mathdeuxcat{B}}}{c}) (u(a), p \sigma_{a})
\\
&= (u(a), p \sigma_{a})
\end{aligned}
$$
et
$$
\begin{aligned}
(\DeuxFoncTrancheLaxCoq{\TildeLax{u}}{\BarreLax{\sigma}}{c}) ((\DeuxFoncTrancheLax{\LaxCanonique{\mathdeuxcat{A}}}{c}) (a,p)) &= (\DeuxFoncTrancheLaxCoq{\TildeLax{u}}{\BarreLax{\sigma}}{c}) (a,p)
\\
&= (\TildeLax{u} (a), p \BarreLax{\sigma}_{a})
\\
&= (u(a), p \sigma_{a}).
\end{aligned}
$$

Pour toute \un{}cellule $(f : a \to a', \alpha : p \Rightarrow p' u(f))$ de $(a, p)$ vers $(a', p')$ dans $\TrancheLax{\mathdeuxcat{A}}{w}{c}$, 
$$
\begin{aligned}
(\DeuxFoncTrancheLax{\LaxCanonique{\mathdeuxcat{B}}}{c}) ((\DeuxFoncTrancheLaxCoq{u}{\sigma}{c}) (f, \alpha)) &= (\DeuxFoncTrancheLax{\LaxCanonique{\mathdeuxcat{B}}}{c}) (u(f), (p' \CompDeuxZero \sigma_{f}) (\alpha \CompDeuxZero \sigma_{a}))
\\
&= (([1], u(f)),  (p' \CompDeuxZero \sigma_{f}) (\alpha \CompDeuxZero \sigma_{a}))
\end{aligned}
$$
et
$$
\begin{aligned}
(\DeuxFoncTrancheLaxCoq{\TildeLax{u}}{\BarreLax{\sigma}}{c}) ((\DeuxFoncTrancheLax{\LaxCanonique{\mathdeuxcat{A}}}{c}) (f,\alpha)) &= (\DeuxFoncTrancheLaxCoq{\TildeLax{u}}{\BarreLax{\sigma}}{c}) (\LaxCanonique{\mathdeuxcat{A}} (f), \alpha)
\\
&= (\DeuxFoncTrancheLaxCoq{\TildeLax{u}}{\BarreLax{\sigma}}{c}) (([1], f), \alpha)
\\
&= (\TildeLax{u} ([1], f), (p' \CompDeuxZero \BarreLax{\sigma}_{([1], f)}) (\alpha \CompDeuxZero \BarreLax{\sigma}_{a}))
\\
&= (([1], u(f)),  (p' \CompDeuxZero \sigma_{f}) (\alpha \CompDeuxZero \sigma_{a})).
\end{aligned}
$$

Pour toute \deux{}cellule $\beta$ de $(f, \alpha)$ vers $(f', \alpha')$ dans $\TrancheLax{\mathdeuxcat{A}}{w}{c}$, 
$$
\begin{aligned}
(\DeuxFoncTrancheLax{\LaxCanonique{\mathdeuxcat{B}}}{c}) ((\DeuxFoncTrancheLaxCoq{u}{\sigma}{c}) (\beta)) &= (\DeuxFoncTrancheLax{\LaxCanonique{\mathdeuxcat{B}}}{c}) (u(\beta)) 
\\
&= \LaxCanonique{\mathdeuxcat{B}} (u(\beta)) 
\\
&= (1_{[1]}, u(\beta))
\end{aligned}
$$
et
$$
\begin{aligned}
(\DeuxFoncTrancheLaxCoq{\TildeLax{u}}{\BarreLax{\sigma}}{c}) ((\DeuxFoncTrancheLax{\LaxCanonique{\mathdeuxcat{A}}}{c}) (\beta)) &= (\DeuxFoncTrancheLaxCoq{\TildeLax{u}}{\BarreLax{\sigma}}{c}) (\LaxCanonique{\mathdeuxcat{A}} (\beta))
\\
&= (\DeuxFoncTrancheLaxCoq{\TildeLax{u}}{\BarreLax{\sigma}}{c}) (1_{[1]}, \beta)
\\
&= \TildeLax{u} (1_{[1]}, \beta)
\\
&= (1_{[1]}, u(\beta)).
\end{aligned}
$$

Pour tout objet $(a, p : w(a) \to c)$ de $\TrancheLax{\mathdeuxcat{A}}{w}{c}$,
$$
\begin{aligned}
((\DeuxFoncTrancheLax{\LaxCanonique{\mathdeuxcat{B}}}{c}) (\DeuxFoncTrancheLaxCoq{u}{\sigma}{c}))_{(a,p)} &= (\DeuxFoncTrancheLax{\LaxCanonique{\mathdeuxcat{B}}}{c}) ((\DeuxFoncTrancheLaxCoq{u}{\sigma}{c})_{(a,p)}) (\DeuxFoncTrancheLax{\LaxCanonique{\mathdeuxcat{B}}}{c})_{(\DeuxFoncTrancheLaxCoq{u}{\sigma}{c}) (a,p)}
\\
&= (\DeuxFoncTrancheLax{\LaxCanonique{\mathdeuxcat{B}}}{c}) (u_{a}) (\DeuxFoncTrancheLax{\LaxCanonique{\mathdeuxcat{B}}}{c})_{(u(a), p \sigma_{a})}
\\
&= (\DeuxFoncTrancheLax{\LaxCanonique{\mathdeuxcat{B}}}{c}) (u_{a}) (\LaxCanonique{\mathdeuxcat{B}})_{u(a)}
\\
&= (1_{[1]}, u_{a}) ([1] \to [0], 1_{1_{u(a)}})
\\
&= ([1] \to [0], u_{a})
\end{aligned}
$$
et
$$
\begin{aligned}
((\DeuxFoncTrancheLaxCoq{\TildeLax{u}}{\BarreLax{\sigma}}{c}) (\DeuxFoncTrancheLax{\LaxCanonique{\mathdeuxcat{A}}}{c}))_{(a,p)} &= (\DeuxFoncTrancheLaxCoq{\TildeLax{u}}{\BarreLax{\sigma}}{c}) ((\DeuxFoncTrancheLax{\LaxCanonique{\mathdeuxcat{A}}}{c})_{(a,p)})
\\
&= (\DeuxFoncTrancheLaxCoq{\TildeLax{u}}{\BarreLax{\sigma}}{c}) ((\LaxCanonique{\mathdeuxcat{A}})_{a})
\\
&= (\DeuxFoncTrancheLaxCoq{\TildeLax{u}}{\BarreLax{\sigma}}{c}) ([1] \to [0], 1_{1_{a}})
\\
&= \TildeLax{u} ([1] \to [0], 1_{1_{a}})
\\
&= ([1] \to [0], u_{a}).
\end{aligned}
$$

Pour tout couple de \un{}cellules composables $(f, \alpha) : (a, p) \to (a', p')$ et $(f', \alpha') : (a', p') \to (a'', p'')$ dans $\TrancheLax{\mathdeuxcat{A}}{w}{c}$, 
$$
\begin{aligned}
((\DeuxFoncTrancheLax{\LaxCanonique{\mathdeuxcat{B}}}{c}) (\DeuxFoncTrancheLaxCoq{u}{\sigma}{c}))_{(f', \alpha'), (f, \alpha)} &= (\DeuxFoncTrancheLax{\LaxCanonique{\mathdeuxcat{B}}}{c}) ((\DeuxFoncTrancheLaxCoq{u}{\sigma}{c})_{(f', \alpha'), (f, \alpha)}) (\DeuxFoncTrancheLax{\LaxCanonique{\mathdeuxcat{B}}}{c})_{(\DeuxFoncTrancheLaxCoq{u}{\sigma}{c})(f', \alpha'), (\DeuxFoncTrancheLaxCoq{u}{\sigma}{c}) (f, \alpha)}
\\
&= (\DeuxFoncTrancheLax{\LaxCanonique{\mathdeuxcat{B}}}{c}) (u_{f', f}) (\DeuxFoncTrancheLax{\LaxCanonique{\mathdeuxcat{B}}}{c})_{(u(f'), (p'' \CompDeuxZero \sigma_{f'}) (\alpha' \CompDeuxZero \sigma_{a'})), (u(f), (p' \CompDeuxZero \sigma_{f}) (\alpha \CompDeuxZero \sigma_{a}))}
\\
&= (1_{[1]}, u_{f', f}) (\LaxCanonique{\mathdeuxcat{B}})_{u(f'), u(f)}
\\
&= (1_{[1]}, u_{f', f}) ([1] \to [2], 1_{u(f') u(f)})
\\
&= ([1] \to [2], u_{f', f})
\end{aligned}
$$
et
$$
\begin{aligned}
((\DeuxFoncTrancheLaxCoq{\TildeLax{u}}{\BarreLax{\sigma}}{c}) (\DeuxFoncTrancheLax{\LaxCanonique{\mathdeuxcat{A}}}{c}))_{(f', \alpha'), (f, \alpha)} &= (\DeuxFoncTrancheLaxCoq{\TildeLax{u}}{\BarreLax{\sigma}}{c}) ((\DeuxFoncTrancheLax{\LaxCanonique{\mathdeuxcat{A}}}{c})_{(f', \alpha'), (f, \alpha)})
\\
&= (\DeuxFoncTrancheLaxCoq{\TildeLax{u}}{\BarreLax{\sigma}}{c}) ((\LaxCanonique{\mathdeuxcat{A}})_{f', f})
\\
&= (\DeuxFoncTrancheLaxCoq{\TildeLax{u}}{\BarreLax{\sigma}}{c}) ([1] \to [2], 1_{f'f})
\\
&= \TildeLax{u}([1] \to [2], 1_{f'f})
\\
&= ([1] \to [2], u_{f', f}).
\end{aligned}
$$
En vertu de ces calculs et de la transitivité de la relation d'égalité, le lemme \ref{CarreCommutatifThALax} est démontré. 
\end{proof}

\begin{corollaire}\label{TexAvery}
Soient
$$
\xymatrix{
\mathdeuxcat{A} 
\ar[rr]^{u}
\ar[dr]_{w}
&&\mathdeuxcat{B}
\dtwocell<\omit>{<7.3>\sigma}
\ar[dl]^{v}
\\
&\mathdeuxcat{C}
&{}
}
$$
un diagramme de \DeuxFoncteursLax{} commutatif à l'\DeuxTransformationCoLax{} $\sigma : vu \Rightarrow w$ près seulement et $c$ un objet de $\mathdeuxcat{C}$. Alors, le \DeuxFoncteurLax{} $\DeuxFoncTrancheLaxCoq{u}{\sigma}{c} : \TrancheLax{\mathdeuxcat{A}}{w}{c} \to \TrancheLax{\mathdeuxcat{B}}{v}{c}$ est une équivalence faible si et seulement si le \DeuxFoncteurStrict{} $\DeuxFoncTrancheLaxCoq{\TildeLax{u}}{\BarreLax{\sigma}}{c} : \TrancheLax{\TildeLax{\mathdeuxcat{A}}}{\BarreLax{w}}{c} \to \TrancheLax{\TildeLax{\mathdeuxcat{B}}}{\BarreLax{v}}{c}$ en est une. 
\end{corollaire}

\begin{proof}
En vertu de la proposition \ref{LaxInduitEquiFaible}, les \DeuxFoncteursLax{} $\DeuxFoncTrancheLax{\LaxCanonique{\mathdeuxcat{A}}}{c}$ et $\DeuxFoncTrancheLax{\LaxCanonique{\mathdeuxcat{B}}}{c}$ sont des équivalences faibles. Le lemme \ref{CarreCommutatifThALax} et la saturation faible de la classe des équivalences faibles permettent de conclure.
\end{proof}

\begin{theo}\label{TheoremeALaxTrancheLaxCoq}
Soit
$$
\xymatrix{
\mathdeuxcat{A} 
\ar[rr]^{u}
\ar[dr]_{w}
&&\mathdeuxcat{B}
\dtwocell<\omit>{<7.3>\sigma}
\ar[dl]^{v}
\\
&\mathdeuxcat{C}
&{}
}
$$
un diagramme de \DeuxFoncteursLax{} commutatif à l'\DeuxTransformationCoLax{} $\sigma : vu \Rightarrow w$ près seulement. Supposons que, pour tout objet $c$ de $\mathdeuxcat{C}$, le \DeuxFoncteurLax{} 
$$
\DeuxFoncTrancheLaxCoq{u}{\sigma}{c} : \TrancheLax{\mathdeuxcat{A}}{w}{c} \to \TrancheLax{\mathdeuxcat{B}}{v}{c}
$$ 
soit une équivalence faible. Alors, $u$ est une équivalence faible.
\end{theo}

\begin{proof}
En vertu des hypothèses et du corollaire \ref{TexAvery}, pour tout objet $c$ de $\mathdeuxcat{C}$, le \DeuxFoncteurStrict{} $\DeuxFoncTrancheLaxCoq{\TildeLax{u}}{\BarreLax{\sigma}}{c}$ est une équivalence faible. En vertu du théorème \ref{ThAStrictCoq}, $\TildeLax{u}$ est une équivalence faible, donc $u$ est une équivalence faible.
\end{proof}

\begin{theo}\label{TheoremeALaxOpTrancheLaxCoq}
Soit
$$
\xymatrix{
\mathdeuxcat{A} 
\ar[rr]^{u}
\ar[dr]_{w}
&{}
&\mathdeuxcat{B}
\ar[dl]^{v}
\\
&\mathdeuxcat{C}
\utwocell<\omit>{\sigma}
}
$$
un diagramme de \DeuxFoncteursLax{} commutatif à la \DeuxTransformationLax{} $\sigma : w \Rightarrow vu$ près seulement. Supposons que, pour tout objet $c$ de $\mathdeuxcat{C}$, le \DeuxFoncteurLax{} 
$$
\DeuxFoncOpTrancheLaxCoq{u}{\sigma}{c} : \OpTrancheLax{\mathdeuxcat{A}}{w}{c} \to \OpTrancheLax{\mathdeuxcat{B}}{v}{c}
$$ 
soit une équivalence faible. Alors $u$ est une équivalence faible.
\end{theo}

\begin{proof}
En vertu des hypothèses, le \DeuxFoncteurLax{} $\DeuxFoncUnOp{(\DeuxFoncTrancheLaxCoq{\DeuxFoncUnOp{u}}{\DeuxTransUnOp{\sigma}}{c})}$ est une équivalence faible pour tout objet $c$ de $\mathdeuxcat{C}$. Par conséquent, en vertu de la proposition \ref{DeuxFoncLaxUnOpW}, le \DeuxFoncteurLax{} $\DeuxFoncTrancheLaxCoq{\DeuxFoncUnOp{u}}{\DeuxTransUnOp{\sigma}}{c}$ est une équivalence faible pour tout objet $c$ de $\mathdeuxcat{C}$. Le théorème \ref{TheoremeALaxTrancheLaxCoq} permet d'en déduire que $\DeuxFoncUnOp{u}$ est une équivalence faible. Une nouvelle invocation de la proposition \ref{DeuxFoncLaxUnOpW} permet de conclure.
\end{proof}

\begin{theo}\label{TheoremeAColaxTrancheCoLaxCoq}
Soit
$$
\xymatrix{
\mathdeuxcat{A} 
\ar[rr]^{u}
\ar[dr]_{w}
&&\mathdeuxcat{B}
\ar[dl]^{v}
\dtwocell<\omit>{<7.3>\sigma}
\\
&\mathdeuxcat{C}
&{}
}
$$
un diagramme de \DeuxFoncteursCoLax{} commutatif à la \DeuxTransformationLax{} $\sigma : vu \Rightarrow w$ près seulement. Supposons que, pour tout objet $c$ de $\mathdeuxcat{C}$, le \DeuxFoncteurCoLax{} 
$$
\DeuxFoncTrancheCoLaxCoq{u}{\sigma}{c} : \TrancheCoLax{\mathdeuxcat{A}}{w}{c} \to \TrancheCoLax{\mathdeuxcat{B}}{v}{c}
$$ 
soit une équivalence faible. Alors $u$ est une équivalence faible.
\end{theo}

\begin{proof}
En vertu des hypothèses, le \DeuxFoncteurCoLax{} $\DeuxFoncDeuxOp{(\DeuxFoncTrancheLaxCoq{\DeuxFoncDeuxOp{u}}{\DeuxTransDeuxOp{\sigma}}{c})}$ est une équivalence faible pour tout objet $c$ de $\mathdeuxcat{C}$. Par définition et conséquent, le \DeuxFoncteurLax{} $\DeuxFoncTrancheLaxCoq{\DeuxFoncDeuxOp{u}}{\DeuxTransDeuxOp{\sigma}}{c}$ est une équivalence faible pour tout objet $c$ de $\mathdeuxcat{C}$. Le théorème \ref{TheoremeALaxTrancheLaxCoq} permet d'en déduire que le \DeuxFoncteurLax{} $\DeuxFoncDeuxOp{u}$ est une équivalence faible, c'està-dire, par définition, que le \DeuxFoncteurCoLax{} $u$ est une équivalence faible. 
\end{proof}

\begin{theo}\label{TheoremeAColaxOpTrancheCoLaxCoq}
Soit
$$
\xymatrix{
\mathdeuxcat{A} 
\ar[rr]^{u}
\ar[dr]_{w}
&{}
&\mathdeuxcat{B}
\ar[dl]^{v}
\\
&\mathdeuxcat{C}
\utwocell<\omit>{\sigma}
}
$$
un diagramme de \DeuxFoncteursCoLax{} commutatif à l'\DeuxTransformationCoLax{} $\sigma : w \Rightarrow vu$ près seulement. Supposons que, pour tout objet $c$ de $\mathdeuxcat{C}$, le \DeuxFoncteurCoLax{} 
$$
\DeuxFoncOpTrancheCoLaxCoq{u}{\sigma}{c} : \OpTrancheCoLax{\mathdeuxcat{A}}{w}{c} \to \OpTrancheCoLax{\mathdeuxcat{B}}{v}{c}
$$ 
soit une équivalence faible. Alors $u$ est une équivalence faible.
\end{theo}

\begin{proof}
En vertu des hypothèses, le \DeuxFoncteurCoLax{} $\DeuxFoncToutOp{(\DeuxFoncTrancheLaxCoq{\DeuxFoncToutOp{u}}{\DeuxTransToutOp{\sigma}}{c})}$ est une é\-qui\-va\-lence faible pour tout objet $c$ de $\mathdeuxcat{C}$. Par conséquent, le \DeuxFoncteurLax{} $\DeuxFoncTrancheLaxCoq{\DeuxFoncToutOp{u}}{\DeuxTransToutOp{\sigma}}{c}$ est une équivalence faible pour tout objet $c$ de $\mathdeuxcat{C}$. Le \DeuxFoncteurLax{} $\DeuxFoncToutOp{u}$ est donc une équivalence faible, donc le \DeuxFoncteurCoLax{} $u$ est une équivalence faible. 
\end{proof}

\section{Critère local}\label{SectionCritereLocal}

\emph{On ne suppose plus fixé de localisateur fondamental.}

\begin{df}
Soit $\UnLocFond{W}$ un \ClasseUnLocFond{}. On dira qu'un morphisme $u : A \to B$ de $\Cat$ est \emph{$\UnLocFond{W}$-localement constant}, ou plus simplement \emph{localement constant}\index{localement constant (morphisme de $\DeuxCat$)} si, pour tout morphisme $b \to b'$ de $B$, le morphisme $A/b \to A/b'$ de $\Cat$ est une $\UnLocFond{W}$\nobreakdash-équivalence faible. 
\end{df}

\begin{theo}[Cisinski]\label{CaractTheoBCat}
Le localisateur fondamental minimal $\UnLocFondMin$ de $\Cat$ est le seul localisateur fondamental de $\Cat$ vérifiant les propriétés suivantes.
\begin{itemize}
\item[(i)] Pour tout morphisme $u : A \to B$ de $\Cat$, si $u$ est une équivalence faible, alors $\pi_{0}(u) : \pi_{0}A \to \pi_{0}B$ est une bijection.
\item[(ii)] Pour tout morphisme $u : A \to B$ de $\Cat$ localement constant, $u$ est une équivalence faible si et seulement s'il est asphérique. 
\end{itemize}
\end{theo}

\begin{proof}
C'est une conséquence de \cite[corollaire 2.3.3]{LFM}, \cite[proposition 2.3.4]{LFM} et \cite[théorème 2.3.6]{LFM}. 
\end{proof}

\begin{df}\label{DefLaxLocalementConstant}
Soit $\DeuxLocFond{W}$ un \ClasseDeuxLocFond{}. On dira qu'un morphisme $u : \mathdeuxcat{A} \to \mathdeuxcat{B}$ de $\DeuxCat$ est \emph{$\DeuxLocFond{W}$-lax-localement constant}\index{lax-localement constant (morphisme de $\DeuxCat$)}, ou plus simplement \emph{lax-localement constant} si, pour tout morphisme $b \to b'$ de $\mathdeuxcat{B}$, le morphisme $\TrancheLax{\mathdeuxcat{A}}{u}{b} \to \TrancheLax{\mathdeuxcat{A}}{u}{b'}$ de $\DeuxCat$ est une $\DeuxLocFond{W}$\nobreakdash-équivalence faible. 
\end{df}

\begin{paragr}
On rappelle qu'il existe une structure de catégorie de modèles sur $\EnsSimp$ dont les équivalences faibles sont les éléments de $\EquiQuillen$ et dont les cofibrations sont les monomorphismes. Cela permet (même si ce n'est en principe pas indispensable) de donner sens à la notion de carré homotopiquement cartésien dans $\EnsSimp$.
\end{paragr}

\begin{theo}[Cegarra]\label{ThBCegarra}
Soit $u : \mathdeuxcat{A} \to \mathdeuxcat{B}$ un morphisme de $\DeuxCat$ $\DeuxLocFondMin$\nobreakdash-lax-localement constant. Alors, pour tout objet $b$ de $\mathdeuxcat{B}$, le carré canonique
$$
\xymatrix{
\NerfLaxNor (\TrancheLax{\mathdeuxcat{A}}{u}{b})
\ar[r]
\ar[d]
&
\NerfLaxNor (\mathdeuxcat{A})
\ar[d]
\\
\NerfLaxNor (\TrancheLax{\mathdeuxcat{B}}{}{b})
\ar[r]
&
\NerfLaxNor (\mathdeuxcat{B}) 
}
$$
est homotopiquement cartésien. 
\end{theo}

\begin{proof}
C'est un énoncé dual de celui de \cite[théorème 3.2]{Cegarra}. 
\end{proof}

\begin{theo}\label{CaractTheoBDeuxCat}
Le localisateur fondamental minimal $\DeuxLocFondMin$ de $\DeuxCat$ est le seul localisateur fondamental de $\DeuxCat$ vérifiant les propriétés suivantes.
\begin{itemize}
\item[(i)] Pour tout morphisme $u : \mathdeuxcat{A} \to \mathdeuxcat{B}$ de $\DeuxCat$, si $u$ est une équivalence faible, alors $\pi_{0}(u) : \pi_{0}\mathdeuxcat{A} \to \pi_{0}\mathdeuxcat{B}$ est une bijection.
\item[(ii)] Pour tout morphisme $u : \mathdeuxcat{A} \to \mathdeuxcat{B}$ de $\DeuxCat$ lax-localement constant, $u$ est une équivalence faible si et seulement s'il est lax-asphérique. 
\end{itemize}
\end{theo}

\begin{proof}
Le localisateur fondamental $\DeuxLocFondMin$ de $\DeuxCat$ vérifie par définition la condition $(i)$ de l'énoncé. On sait déjà qu'un morphisme $\DeuxLocFondMin$\nobreakdash-lax-asphérique de $\DeuxCat$ est dans $\DeuxLocFondMin$. Réciproquement, si un morphisme $u$ de $\DeuxCat$ est $\DeuxLocFondMin$\nobreakdash-lax-localement constant et que c'est une $\DeuxLocFondMin$\nobreakdash-équivalence faible, alors, en vertu du théorème \ref{ThBCegarra}, il est $\DeuxLocFondMin$\nobreakdash-lax-asphérique. Le localisateur fondamental $\DeuxLocFondMin$ de $\DeuxCat$ vérifie donc la condition $(ii)$. 

Soit $\DeuxLocFond{W}$ un localisateur fondamental de $\DeuxCat$. S'il vérifie les conditions $(i)$ et $(ii)$ de l'énoncé, le localisateur fondamental $\DeuxLocFond{W} \cap \UnCell{\Cat}$ de $\Cat$ vérifie les conditions $(i)$ et $(ii)$ de l'énoncé du théorème \ref{CaractTheoBCat}. En vertu de ce même théorème \ref{CaractTheoBCat}, $\DeuxLocFond{W} \cap \UnCell{\Cat} = \UnLocFondMin$. On en déduit $\DeuxLocFond{W} = \DeuxLocFondMin$ en vertu du théorème \ref{IsoUnLocFondDeuxLocFond}. 
\end{proof}

\begin{rem}
Les versions duales de la définition \ref{DefLaxLocalementConstant} et du théorème \ref{ThBCegarra} permettent bien sûr de démontrer des versions duales du théorème \ref{CaractTheoBDeuxCat} de façon analogue. 
\end{rem}

\chapter*{Conclusion}

\addcontentsline{toc}{chapter}{Conclusion} \markboth{Conclusion}{}

Nous n'avons pas l'ambition de retracer l'histoire de la théorie de l'homotopie des petites catégories. Signalons simplement qu'en 1979, les auteurs de \cite{SSC} parlaient du « programme à long terme de développer une topologie algébrique des petites catégories ». Nous avons mentionné non seulement la construction fondamentale par Thomason d'une structure de catégorie de modèles sur $\Cat$ dont les équivalences faibles sont les équivalences faibles définies classiquement par le foncteur nerf, mais également la démonstration par Cisinski de la conjecture cruciale de Grothendieck de minimalité de cette classe d'équivalences faibles parmi les localisateurs fondamentaux de $\Cat$, ce résultat lui permettant d'en déduire, en utilisant celui de Thomason, pour tout localisateur fondamental accessible de $\Cat$, une structure de catégorie de modèles sur $\Cat$ dont la classe des équivalences faibles est précisément le localisateur fondamental considéré. 

Au début du travail de thèse dont ce texte présente les résultats, l'époque semblait mûre pour aborder la question du développement d'une théorie de l'homotopie des \deux{}catégories suivant le point de vue de Grothendieck : à l'importance croissante des catégories supérieures en théorie de l'homotopie s'ajoutait la publication récente d'une généralisation du Théorème A de Quillen — ingrédient de base de l'axiomatique des localisateurs fondamentaux de $\Cat$ — aux \DeuxFoncteursStricts{} \cite{BC} ainsi que d'une présumée version \deux{}catégorique du résultat de Thomason \cite{WHPT}. Les résultats concernant $\DeuxCat$ n'ont toutefois pas été démontrés dans le même ordre que les résultats analogues pour $\Cat$. En effet, même si l'article \cite{WHPT} contient des éléments intéressants, nous avons déjà mentionné l'incorrection de la démonstration de son assertion principale. L'existence de la structure de catégorie de modèles « à la Thomason » sur $\DeuxCat$ n'était donc pas encore établie lorsque nous avons dégagé la démonstration de notre résultat principal, le théorème \ref{IsoUnLocFondDeuxLocFond}, dont la propriété de minimalité de $\DeuxLocFondMin$ résulte immédiatement, en utilisant bien sûr le résultat de minimalité démontré par Cisinski pour $\UnLocFondMin$. De ce résultat de minimalité de $\DeuxLocFondMin$ découle, comme Ara l'explique dans \cite{Ara}, pour tout localisateur fondamental accessible de $\DeuxCat$, l'existence d'une structure de catégorie de modèles sur $\DeuxCat$ dont la classe des équivalences faibles est précisément le \ClasseDeuxLocFond{} considéré ; mais cela ne vaut qu'à condition d'avoir déjà démontré l'existence de la structure de catégorie de modèles « à la Thomason » pour $\DeuxLocFondMin{}$. Ce dernier résultat n'est venu que dans un second temps par rapport au théorème \ref{IsoUnLocFondDeuxLocFond} : pendant que nous achevions de construire le cadre \deux{}catégorique exposé dans le premier chapitre du présent travail, contexte dans lequel nous présentons la notion de \ClasseDeuxLocFond{}, Ara et Maltsiniotis ont résolu la question ; leur solution, présentée dans \cite{AraMaltsi}, va plus loin, puisqu'elle leur permet d'aborder la question de la structure analogue présumée sur la catégorie $n$\nobreakdash-$\Cat$ des $n$\nobreakdash-catégories strictes et des $n$\nobreakdash-foncteurs stricts. Comme nous l'avons déjà mentionné, c'est le théorème \ref{EqCatLocCatDeuxCat} de notre texte qui permet de montrer que l'adjonction de Quillen entre $\EnsSimp$ et $\DeuxCat$ construite par Ara et Maltsiniotis est en fait une équivalence de Quillen. 

Si la notion de localisateur fondamental se trouve au cœur de la « théorie de l'homotopie de Grothendieck », cette dernière contient d'autres notions. Une fois dégagés les définitions et résultats fondamentaux présentés dans notre travail, base d'une « théorie de l'homotopie de Grothendieck supérieure », on pourrait donc par exemple souhaiter développer une théorie des \emph{\deux{}catégories test}, généralisation de la notion de catégorie test dégagée par Grothendieck. Une autre direction de recherche consisterait à vouloir étudier des notions de \emph{\deux{}foncteurs lisses} et \emph{\deux{}foncteurs propres}, généralisations de celles de foncteurs lisses et foncteurs propres, notions également dues à Grothendieck et généralisant celles de catégorie préfibrée et de catégorie préopfibrée. En fait, la notion de \emph{structure d'asphéricité à droite} proposée par Maltsiniotis dans \cite{SAD}, généralisation de celle de \ClasseUnLocFond{}, permet d'envisager aussi les catégories fibrées. On pourrait donc, éventuellement pour aborder l'étude de ce que doivent être les \deux{}catégories fibrées ou les \deux{}catégories préfibrées — puisqu'il n'est pas certain que la notion de préfibration présentée dans notre travail soit celle qui mérite le mieux ce nom —, définir d'abord une notion de structure d'asphéricité à droite pour $\DeuxCat$, ce qu'il est du reste déjà possible de faire. Un autre projet de recherche envisageable consisterait à construire, délaissant le cas particulier des \deux{}catégories, une théorie des localisateurs fondamentaux de $n$\nobreakdash-$\Cat$, ce qui n'exclut évidemment pas d'envisager le développement des notions générales de $n$\nobreakdash-catégorie test et de $n$\nobreakdash-foncteur propre ou lisse... La démonstration du fait que les $n$\nobreakdash-catégories modélisent les types d'homotopie pourrait d'ailleurs s'inspirer de celle que nous avons dégagée pour $n=2$. Le présent travail indique assez l'importance que les morphismes « lax » joueront sans doute dans cette entreprise. On devrait donc également se pencher sur les analogues supérieurs du foncteur de strictification de Bénabou. 

Plus que la fierté d'avoir participé à ce programme en démontrant par une méthode inattendue un résultat fondamental dans le développement de cette « théorie de l'homotopie de Grothendieck supérieure », théorème dont quelques « sous-produits » de la démonstration vivront d'ailleurs peut-être leur existence propre — tels les généralisations présentées du Théorème A de Quillen ou les développements \deux{}catégoriques du premier chapitre —, et plus que la désillusion somme toute nécessaire et formatrice suscitée par le comportement de certaines personnes et par le fonctionnement actuel de la « communauté mathématique » et de l'Université, c'est encore le plaisir que m'a procuré la recherche mathématique que je veux retenir — et, du fait de la façon dont s'est déroulée cette thèse, il me reste à préciser que j'exprimerai mes remerciements en privé.

\newpage

\addcontentsline{toc}{chapter}{Index}

\printindex

\addcontentsline{toc}{chapter}{Notations}

\printindex[not]

\end{document}